\documentclass[11pt]{article}

\usepackage{amsfonts,amssymb, latexsym, amsmath, amsthm,mathrsfs, verbatim,geometry}
\usepackage{graphics}
\usepackage{slashed}
\usepackage{url,color}
\usepackage{enumerate}
\usepackage{esint}
\usepackage{pifont}
\usepackage{upgreek}
\usepackage{times}
\usepackage{calligra}
\usepackage{graphicx}
\usepackage{caption}
\usepackage{tikz}
\usetikzlibrary{calc}
\usetikzlibrary{calc}
\usetikzlibrary{arrows.meta}
\usetikzlibrary{hobby}
\usepackage{tkz-euclide}
\usepackage{cancel}
\usepackage{caption}
\usepackage{tikz}
\usetikzlibrary{calc}
\usetikzlibrary{arrows.meta}

\usepackage{hyperref}
\hypersetup{
    colorlinks,
    citecolor=red,
    filecolor=green,
    linkcolor=blue,
    urlcolor=black
}

\def\R{D}
\def\C{\mathbb C}
\def\N{\mathbb N}
\def\Z{\mathbb Z}

\def\D{\mathcal{D}}
\def\be{\begin{equation}}
\def\ee{\end{equation}}
\def\bea{\begin{eqnarray}}
\def\eea{\end{eqnarray}}
\def\beas{\begin{eqnarray*}}
\def\eeas{\end{eqnarray*}}

\def\l{\lambda}
\def\pa{\partial }

\def\l{\lambda}

\def\lv{\left\vert}
\def\rv{\right\vert}

\def\bk{\bar{\kappa}}

\def\w{{\bf w}}
\def\d{{\bf d}}
\def\ch{{\boldsymbol \chi}}

\def\ttau{\tilde{\tau}}

\def\e{\epsilon}
\def\tr{\tilde{r}}

\def\tG{\tilde G}
\def\bcr{\begin{color}{red}}
\def\bcb{\begin{color}{blue}}
\def\bcg{\begin{color}{green}}
\def\bcv{\begin{color}{violet}}
\def\ec{\end{color}}

\def\k{\varepsilon}
\def\e{\eta}

\def\F{J}

\def\xc{x_{\text{crit}}}
\def\xmin{x_{\text{min}}}
\def\xmax{x_{\text{max}}}

\def\wl{{\bf w}}

\def\rl{\tilde{r}}
\def\DS{\Pi}

\def\Rl{{\bf d}}
\def\RY{d}
\def\vY{w}
\def\QY{Q}
\def\yms{Y^{\text{ms}}}

\def\grlp{g_{\text{RLP},\k}}
\def\MRLP{\mathcal M_{\text{RLP},\k}}
\def\DRLP{\mathcal D_{\text{RLP},\k}}
\def\DRLPtilde{\tilde{\mathcal D}_{\text{RLP},\k}}
\def\MS{\mathcal{M}\mathcal{S}_\k}
\def\dy{\delta Y}

\def\pau{\partial_{p}}
\def\pav{\partial_{q}}
\def\u{p}
\def\v{q}
\def\Om{\Omega}

\def\tt{\tilde{\tau}}

\def\Ysp{Y^{\text{sp}}}
\def\xf{x_F}
\def\dx{\delta x}
\def\RH{\mathcal{R}}
\def\WH{\mathcal{W}}
\def\xs{x_\ast}
\def\dz{\delta z}
\def\uC{\underline{\mathcal C}}
\def\C{\mathcal C}

\def\po{\pa W_-}
\def\pr{\pa D_-}

\def\X{\mathcal X_{>\frac13}}
\def\Y{\mathcal X_{\frac13}}
\def\Z{\mathcal X_{<\frac13}}

\newcommand{\vertiii}[1]{{\left\vert\kern-0.25ex\left\vert\kern-0.25ex\left\vert #1 
    \right\vert\kern-0.25ex\right\vert\kern-0.25ex\right\vert}}

\newcommand{\prfe}{\hspace*{\fill} $\Box$

\smallskip \noindent}

\newtheorem{theorem}{Theorem}[section]
\newtheorem{definition}[theorem]{Definition}
\newtheorem{proposition}[theorem]{Proposition}

\newtheorem{corollary}[theorem]{Corollary}
\newtheorem{lemma}[theorem]{Lemma}
\newtheorem{remark}[theorem]{Remark}

\title{Naked singularities in the Einstein-Euler system}

\author{Yan Guo\thanks{Division of Applied Mathematics, Brown University, Providence, RI 02912, USA, Email: Yan\_Guo@brown.edu.}, \ Mahir Hadzic\thanks{Department of Mathematics, University College London, London UK, Email: m.hadzic@ucl.ac.uk.}, \ and Juhi Jang\thanks{Department of Mathematics, University of Southern California, Los Angeles, CA 90089, USA, and Korea Institute for Advanced Study, Seoul, Korea.  Email: juhijang@usc.edu.}}

\date{}

\begin{document}

\maketitle

\abstract{In 1990, based on numerical and formal asymptotic analysis, Ori and Piran predicted the existence of self-similar spacetimes, called  relativistic Larson-Penston solutions, that 
can be suitably flattened to obtain examples of spacetimes that dynamically form naked singularities from smooth initial data, and solve the radially symmetric Einstein-Euler system. Despite its importance, a rigorous proof of the existence of such spacetimes has remained elusive, in part due to the complications associated with the analysis across the so-called sonic hypersurface. We provide a rigorous mathematical proof.

Our strategy is based on a delicate study of nonlinear invariances associated with the underlying non-autonomous dynamical system to which the problem reduces after a self-similar reduction.
Key technical ingredients are a monotonicity lemma tailored to the problem, an ad hoc shooting method developed to construct a solution connecting the sonic hypersurface to the so-called Friedmann solution,
and a nonlinear argument to construct the maximal analytic extension of the solution. Finally, we reformulate the problem in double-null gauge to truncate the self-similar profile
and thus obtain an asymptotically flat spacetime with an isolated naked singularity.}

\tableofcontents

\section{Introduction}

We study the Einstein-Euler system which couples the Einstein field equations to the Euler equations of fluid mechanics. 
The unknowns are the 4-dimensional Lorentzian spacetime $(\mathcal M, g)$,  the fluid pressure $p$, the mass-density $\rho$, and the 4-velocity $u^\alpha$.
In an arbitrary coordinate system, the Einstein-Euler equations read
\begin{align}
\text{Ric}_{\alpha\beta}-\frac{1}{2}\mathcal Rg_{\alpha\beta}&=T_{\alpha\beta},  \ \ (\alpha,\,\beta=\,0,1,2,3), \label{E:EINSTEIN}\\
\nabla_{\alpha}T^{\alpha\beta}&=0,  \ \  (\beta=0,1,2,3),\label{E:BIANCHI}\\
g_{\alpha \beta} u^{\alpha} u^{\beta}&=-1,\,\label{E:NORMALISATION}
\end{align}
%\begin{align}\label{E:EINSTEIN}
%G^{\mu\nu} = 8\pi T^{\mu\nu}, \ \ \mu,\nu=0,1,2,3,
%\end{align}
where $\text{Ric}_{\alpha\beta}$ is the Ricci curvature tensor, $\mathcal R$ the scalar curvature of $g_{\alpha\beta}$, and $T_{\alpha\beta}$ is the energy momentum tensor given by the formula
\begin{align}\label{E:EMTENSOR}
T_{\alpha\beta} = (\rho+p)u_{\alpha}u_{\beta} + p g_{\alpha\beta}, \ \ ( \alpha, \, \beta=0,1,2,3).
\end{align} 
To close the system,
we assume the linear equation of state
\be\label{E:EOS}
p = \k \rho,
\ee
where $0<\k<1$ corresponds to the square of the speed of sound.

The system~\eqref{E:EINSTEIN}--\eqref{E:EOS} is a fundamental model of a self-gravitating relativistic gas.
We are interested in the existence of  self-similar solutions to~\eqref{E:EINSTEIN}--\eqref{E:EOS} under
the assumption of radial symmetry. This amounts to the existence of a homothetic Killing vector field $\xi$
with the property
\begin{align}
\mathcal L_\xi g = 2 g,
\end{align}
where the left-hand side is the Lie derivative of the metric $g$. The presence of such a vector field
induces a scaling symmetry, which allows us to look for  self-similar solutions to~\eqref{E:EINSTEIN}--\eqref{E:EOS}.
Study of self-similar solutions to Einstein-matter systems has a rich history in the physics literature. They in
particular provide a way of constructing spacetimes with  so-called naked singularities, a notion intimately
tied to the validity of the weak cosmic censorship of Penrose~\cite{Penrose1969}, see the discussion in~\cite{Ch1999a,RoSR2019,DaLu2017}. Naked singularities intuitively 
correspond to spacetime singularities that are ``visible" to far away observers, which informally means that there exists 
a future outgoing null-geodesic ``emanating" from the singularity and reaching the asymptotically flat region of the spacetime. We adopt here a precise mathematical definition
from the work of Rodnianski and Shlapentokh-Rothman~\cite[Definition 1.1]{RoSR2019}, which in turn is related to a formulation of weak cosmic censorship by Christodoulou~\cite{Ch1999a}. 

In the absence of pressure ($\k=0$ in~\eqref{E:EOS}) and under the assumption of radial symmetry, the problem simplifies considerably.
The corresponding family of solutions was studied by Lema\^itre~\cite{Lemaitre1933} and Tolman~\cite{To1934} in early 1930s (see also~\cite{Bondi1947}). 
In their seminal work from 1939, Oppenheimer and Snyder~\cite{OpSn1939} studied the causal structure
of a subclass of Lema\^itre-Tolman solutions with space-homogeneous densities, thus exhibiting 
the first example of a dynamically forming  (what later became known as) black hole. However,
in 1984 Christodoulou~\cite{Ch1984} showed that, within the larger class of Lema\^itre-Tolman solutions with
space-inhomogeneous densities, black holes
are exceptional and instead naked singularities form generically.

Of course, in the context of astrophysics, one expects the role of pressure
to be very important in the process of gravitational collapse for relativistic gases. In the late stages
of collapse, the core region is expected to be very dense and the linear equation of state~\eqref{E:EOS}
is commonly used in such a setting, as it is compatible with the requirement that the speed of sound is 
smaller than the speed of light, $\sqrt\k<1$. In their pioneering works, Ori and Piran~\cite{OP1987,OP1988,OP1990}
found numerically self-similar solutions to~\eqref{E:EINSTEIN}--\eqref{E:EOS}, which are
the relativistic analogues of the Larson-Penston self-similar collapsing solutions to the isothermal Euler-Poisson system, see~\cite{Larson1969,Penston1969,GHJ2021}.
Through both numerical and asymptotic analysis methods Ori and Piran investigated the causal structure of such relativistic Larson-Penston
solutions, ascertaining the existence of spacetimes with naked singularities when $\k$ is smaller than a certain value. 
Our main goal  is to justify the findings of Ori and Piran on  rigorous
mathematical grounds.

Broadly speaking, this manuscript consists of two parts.
In the first part, which constitutes the bulk of our work, we 
construct a self-similar solution of the Einstein-Euler system (Sections~\ref{S:FORMULATION}--\ref{S:RLP}), assuming that $\k$ -- the square of the speed 
of sound -- is sufficiently small.

%%%%%%%%%%%%%%%%%%%%%%
%%%%%%%%%%%%%%%%%%%%%%

\begin{theorem}[Existence of the relativistic Larson-Penston spacetimes]\label{T:MAINTHEOREM2}
For any sufficiently small $0<\k\ll1$ there exists a radially symmetric real-analytic self-similar solution to the Einstein-Euler system
with a curvature singularity at the scaling origin and an outgoing null-geodesic emanating from it all the way to infinity. 
The resulting spacetime is called the relativistic Larson-Penston (RLP) spacetime.
\end{theorem}

%%%%%%%%%%%%%%%%%%%%%%
%%%%%%%%%%%%%%%%%%%%%%

%The ODE system that we study is either a $2\times2$ or a $3\times3$ non-autonomous system, depending on whether 
%we work with the Schwarzschild or the comoving self-similar variables, respectively.
%The problem is challenging and we are aware of no general ODE theory for global existence 
%%and existence of heteroclinic orbits, 
%%particularly 
%%in the presence of singular sonic points. Therefore,
%our proofs are all based on continuity arguments, where we extract many
%delicate invariant properties of the nonlinear flow, specific to the ODE system at hand. In particular, 
%a discovery of the crucial monotonicity lemma -- Lemma~\ref{L:JUHILEMMA} -- enables us to apply 
%an ad hoc shooting method, which was developed for the limiting Larson-Penston solution in the nonrelativistic context.
%boundary conditions at $x=0$ and $x=\infty$, which leads to different $\k$-dependent
%leading order asymptotics. 

It is not hard to see that the self-similar solution constructed in 
Theorem~\ref{T:MAINTHEOREM2} is not asymptotically flat.
In the second step (Section~\ref{S:DOUBLENULL}), using PDE techniques, we flatten
the self-similar RLP-profile in a region away from the singularity and thus obtain an asymptotically flat solution
with a naked singularity. Thus, our main theorem states that in the presence of pressure there do exist examples of naked singularities which form 
from smooth data.

%%%%%%%%%%%%%%%%%%%%%%
%%%%%%%%%%%%%%%%%%%%%%

\begin{theorem}[Existence of naked singularities]\label{T:MAINTHEOREM}
For sufficiently small $0<\k\ll1$ there exist radially symmetric asymptotically flat solutions to the Einstein-Euler system that form
a naked singularity in the sense of~\cite[Definition 1.1]{RoSR2019}.
\end{theorem}

%%%%%%%%%%%%%%%%%%%%%%
%%%%%%%%%%%%%%%%%%%%%%

%It is a challenging to study the ODE system, which is non-automous. We
%are aware of no general ODE theory for global existence and existence of hetro-
%clinical orbits, partiularly in the presence of a singular sonic points. Therefore,
%our mathemtical proofs are all based on continuity arguments to extract many
%delicate properties for solutions from particular structures of ODE. In particu-
%lar, a discovery the crucial monoticity lemma enables us to apply the shooting
%method developed for the limiting LP solution in the relativstic context.
%\ref{S:RLP}.
% of Theorems~\ref{T:MAINTHEOREM2} and~\ref{T:MAINTHEOREM}.
%We also introduce some important notation and concepts, necessary for the analysis in t

Other than the dust-Einstein model mentioned above, we are aware of two other rigorous results on the existence of naked singularities.
In 1994 Christodoulou~\cite{Ch1994} provided a rigorous proof of the existence of radially symmetric solutions to the Einstein-scalar field
system, which contain naked singularities (see also~\cite{Ch1999b} for the proof of their instability). 
%The strategy of~\cite{Ch1994} is to construct a self-similar solution, which 
%is shown to contain an outgoing null geodesic emanating from the centre of scaling symmetry and ``reaching" the infinity.
%Since the self-similar spacetime is not asymptotically flat, the solution is then suitably truncated in a region ``away" from
%the scaling origin. 
Very recently, Rodnianski and Shlapentokh-Rothman~\cite{RoSR2019} proved the existence of solutions to the
Einstein-vacuum equations which contain naked singularities and are (necessarily) not radially symmetric.

In the physics literature much attention has been given to self-similar solutions and naked singularities for the Einstein-Euler system, see for example~\cite{CaCo1999,GuGa2007}.
A self-similar reduction of the problem was first given in~\cite{StShGu1965}. As explained above, a detailed analysis of the
resulting equations, including the discussion of naked singularities, was given in~\cite{OP1987, OP1988,OP1990}.
Subsequent to~\cite{OP1990}, a further analysis of the causal structure, including the nonradial null-geodesics was presented in~\cite{JoDw1992}, see also~\cite{CaGu2003}.
There exist various approaches to the existence of solutions to the self-similar problem, most of them rely on numerics~\cite{OP1990,CaCoGoNiUg2000,Ha1998}. A dynamical systems approach
with a discussion of some qualitative properties of the solutions was developed in~\cite{GoNiUg1998,CaCoGoNiUg2000}. 
Numerical investigation of the stability of the RLP-spacetimes can be found in~\cite{Ha1998,HaMa2001}. Self-similar relativistic perfect fluids play an important role in 
the study of the so-called critical phenomena - we refer to reviews~\cite{GuGa2007,NeCh}.

The proof of Theorem~\ref{T:MAINTHEOREM2} relies on a careful study of the nonlinear invariances of the finite-dimensional
non-autonomous dynamical system obtained through the self-similar reduction.  The solutions we construct are real-analytic in a suitable choice of coordinates. 
A special role is played by the so-called sonic line (sonic point),
the boundary of the backward sound cone emanating from the scaling origin $\mathcal O$. 
Many difficulties in the proof of Theorem~\ref{T:MAINTHEOREM2}
originate from possible singularities across this line, which together with the requirement of smoothness, puts severe limitations on the possible space of smooth self-similar
solutions. We are aware of no general ODE theory for global existence 
%and existence of heteroclinic orbits, 
%particularly 
in the presence of singular sonic points. Therefore,
our proofs are all based on continuity arguments, where we extract many
delicate invariant properties of the nonlinear flow, specific to the ODE system at hand. In particular, 
the discovery of a crucial monotonicity lemma enables us to apply 
an ad hoc shooting method, which was developed for the limiting Larson-Penston solution in the non-relativistic context by the authors.
%Another characteristic curve in our analysis is the boundary of the backward light cone emanating from $\mathcal O$. 
%Unlike in~\cite{Ch1994}, the regularity of solutions
%is not affected as this line is crossed. 
From the point of view of fluid mechanics, the singularity at $\mathcal O$ is an imploding one, as the energy density blows
up on approach to $\mathcal O$. It is in particular not a shock singularity.

The asymptotic flattening in Theorem~\ref{T:MAINTHEOREM} requires solving a suitable characteristic problem for the Einstein-Euler 
system formulated in the double-null gauge. We do this in a semi-infinite characteristic rectangular domain 
wherefrom the resulting solution can be glued smoothly to the exact self-similar solution in the region 
around the singularity $\mathcal O$.  The proof of Theorem~\ref{T:MAINTHEOREM} is given in Section~\ref{SS:NAKED}.

Due to the complexity of our analysis, in Section~\ref{S:METHODS} we give an extensive overview of our methods
and key ideas behind the detailed proofs in Sections~\ref{S:FORMULATION}--\ref{S:DOUBLENULL}. 

\bigskip 

{\bf Acknowledgments.}
Y. Guo's research is supported in part by NSF DMS-grant 2106650.
M. Hadzic's research is supported by the EPSRC Early Career Fellowship EP/S02218X/1.
J. Jang's research is supported by the NSF DMS-grant 2009458
and the Simons Fellowship (grant number 616364).

%%%%%%%%%%%%%%%%%%%%%%%%%%%%
%%%%%%%%%%%%%%%%%%%%%%%%%%%%

\section{Methodology and outline}\label{S:METHODS}

%%%%%%%%%%%%%%%%%%%%%%%%%%%
%%%%%%%%%%%%%%%%%%%%%%%%%%%

%%%%%%%%%%%%%%%%%%%%%%%%%%%
%%%%%%%%%%%%%%%%%%%%%%%%%%%

\subsection{Formulation of the problem (Section~\ref{S:FORMULATION})}

%%%%%%%%%%%%%%%%%%%%%%%%%%%
%%%%%%%%%%%%%%%%%%%%%%%%%%%

Following~\cite{OP1990} it is convenient 
to work with the comoving coordinates
%We shall assume that the spacetime is spherically symmetric and the metric is given in the coordinates
\be\label{E:METRIC}
g = -e^{2\mu(\tau, R)} d\tau^2 + e^{2\l(\tau,R)}dR^2 + r^2(\tau,R) \,\gamma,
\ee
where $\gamma = \gamma_{AB}dx^A\,dx^B$ is the standard metric on $\mathbb S^2$, $x^A$, $A=2,3$ are local coordinates on $\mathbb S^2$, 
%(we follow here the notation from~\cite{DaRe2005}).
and $r$ is the areal radius. 
The vector field $\pa_\tau$ is chosen in such a way that the four velocity $u^\nu$ is parallel to $\pa_\tau$. The normalisation condition~\eqref{E:NORMALISATION} then implies
\begin{align}
u = e^{-\mu}\pa_\tau.
\end{align}
%We first formulate the spherically symmetric Einstein-Euler equations in Lagrangian coordinates following the (beautiful) presentation in~\cite{EhKi1993}. 
%We  introduce the comoving coordinates 
%\[
%\upeta : I\times \mathfrak M \to \mathcal M,
%\]
%where $\mathfrak M$ is a 3-dimensional Riemmanian manifold, also referred to as the material manifold.
%% $R_b$ the initial radius of the star, and
%We let $\upeta$ solve
%\begin{align*}
%\pa_\tau \eta^\nu & = u^\nu\circ \upeta, \ \ \nu=01,2,3, \\
%\eta^\nu (0, R) & = (0,R).
%\end{align*}
The coordinate $R$ acts as a particle label.
The coordinates $(\tau, R)$ are then uniquely determined by fixing the remaining gauge freedoms in the problem, the value of $\mu(\tau,R)\big|_{R=0}$ and
by setting
%\begin{align}
%\mu(\tau, 0) & = c \label{E:FIX1}\\
$r(-1, R)  = R$, which states that on the hypersurface $\tau=-1$ the comoving %Lagrangian 
label $R$ coincides with the areal radius.

Introduce the radial velocity
\be\label{E:VDEF0}
\mathcal V : = e^{-\mu} \pa_\tau r,
\ee
the Hawking (also known as Misner-Sharp) mass
\begin{align}\label{E:HAWKINGCOMOVING}
m(\tau,R):= 4\pi \int_0^{r(\tau,R)} \rho s^2 \,ds = 4\pi \int_0^R \rho r^2 \pa_Rr\,d\bar R,
\end{align}
and the mean density
\be\label{E:GDEF}
G(\tau,R): = \frac{m(\tau,R)}{\frac{4\pi}{3}r(\tau,R)^3} = \frac3 {r(\tau,R)^3} \int_0^{R} \rho(\tau, \bar R) r(\tau,\bar R)^2 \pa_Rr(\tau,\bar R)\,d\bar R . 
\ee

%%%%%%%%%%%%%%%%%%%%%%%%%
%%%%%%%%%%%%%%%%%%%%%%%%%

Recalling the %isothermal 
equation of state~\eqref{E:EOS}, the spherically symmetric Einstein-Euler system in comoving coordinates reads (see \cite{MiSh1964,EhKi1993})
%. Plugging~\eqref{E:EOS} into~\eqref{E:RHOEQN}--\eqref{E:LAMBDACONSTRAINT}, we obtain the isothermal Einstein-Euler system:
%The Euler-Einstein system reads
\begin{align}
\pa_\tau\rho + (1+\k)\rho\left(\frac{\pa_R\mathcal V}{\pa_Rr}+2\frac {\mathcal V}r\right)e^\mu & =0,  \label{E:RHOEQN1}\\
\pa_\tau\l & = e^\mu \frac{\pa_R{\mathcal V}}{\pa_Rr},  \label{E:LAMBDAEQN1}\\
e^{-\mu}\pa_\tau \mathcal V + \frac \k{1+\k}\frac{\pa_Rr e^{-2\l}}{\rho} \pa_R\rho + 4\pi r \left(\frac13 G+\k\rho\right) & =0,  \label{E:VEQN1}\\
(\pa_Rr)^2 e^{-2\l} & = 1+ \mathcal V^2 - \frac{8\pi}{3} G r^2, \label{E:LAMBDACONSTRAINT1}
\end{align}
where we recall~\eqref{E:GDEF}.
%\begin{align}
%G(\tau,R)=\frac3 {r(\tau,R)^3} \int_0^{R}
%\end{align}
The well-known Tolman-Oppenheimer-Volkov relation reads $(\rho+p) \pa_R\mu + \pa_Rp=0$, which after plugging in~\eqref{E:EOS} further gives the relation
\begin{align}\label{E:MUFORMULA0}
\pa_R\mu  = - \frac{\k}{1+\k}\frac{\pa_R\rho}{\rho}.
\end{align} 

%%%%%%%%%%%%%%%%%%%%%%%%%
%%%%%%%%%%%%%%%%%%%%%%%%%

\subsubsection{Comoving self-similar formulation}\label{SS:SSCOMOVING}

%%%%%%%%%%%%%%%%%%%%%%%%%
%%%%%%%%%%%%%%%%%%%%%%%%%

It is straightforward to check that the system~\eqref{E:RHOEQN1}--\eqref{E:LAMBDACONSTRAINT1} is invariant under the scaling transformation
\begin{align}
\rho\mapsto a^{-2}\rho(s,y), \ \ r\mapsto a r(s,y), \ \ \mathcal V \mapsto \mathcal V(s,y), \ \ \lambda\mapsto \lambda(s,y), \ \ \mu\mapsto \mu(s,y), \label{E:T1} 
\end{align}
where the comoving ``time" $\tau$ and the particle label $R$ scale according to
\begin{align}
 s = \frac{\tau}{a}, \ \ y = \frac{R}{a}, \ \ a>0. \label{E:T2}
\end{align}
Motivated by the scaling invariance~\eqref{E:T1}--\eqref{E:T2}, we look for self-similar spacetimes of the form
%We therefore look for a self-similar solution 
%using the following change of variables:
\begin{align}
\rho(\tau,R) & = \frac{1}{2\pi \tau^2} \Sigma(y), \label{E:SS1}\\
r(\tau, R) & = -\sqrt \k \tau \tr(y) ,\label{E:AREARADIUSSS}\\
\mathcal V(\tau,R) & = \sqrt \k  V(y) ,\label{E:VELOCITYSSCHANGE}\\
\lambda(\tau,R) & = \lambda(y), \\
\mu(\tau,R) & = \mu(y),\\
G(\tau, R) & = \frac1{4\pi\tau^2} \tG(y) = \frac{3}{2\pi \tau^2\tr^3}\int_0^y \Sigma(\tilde y) \tr^2 \tr' \,d\tilde y, \label{E:SS6}
\end{align}
where
\be\label{E:SSCDEF}
y= \frac{R}{-\sqrt \k\tau}.
\ee
Associated with the comoving self-similar coordinates are the two fundamental unknowns:
\begin{align}
\d &: = \Sigma^{\frac{1-\k}{1+\k}},  \label{E:LITTLEBOLDDDEF}\\
\w &: = (1+\k) \frac{e^\mu V + \tilde r }{\tilde r}  - \k . \label{E:LITTLEWDEF}
\end{align}
For future use it is convenient to sometimes consider the quantity
\begin{align}\label{E:BOLDFACECHI}
\ch(y): = \frac{\tr(y)}{y}, 
\end{align}
instead of $\tr$. Quantity
$\d$ corresponds to the self-similar number density, while $\w$ is referred to as the relative velocity.
It is shown in Proposition~\ref{P:COMOVINGSS} that the radial Einstein-Euler system under the self-similar
ansatz above reduces to the following system
of ODE:
\begin{align}
\d' &= -  \frac{ 2(1-\k) \d (\d-\w)}{(1+\k)y(e^{2\mu - 2\lambda} y^{-2} -1)} ,\label{E:DSSEQNBOLD}\\
\w' &= \frac{(\w+\k)(1 -3\w)}{(1+\k)y} + \frac{2\w (\d-\w)}{y(e^{2\mu - 2\lambda} y^{-2} -1)}.\label{E:WSSEQNBOLD}
\end{align}

This formulation of the self-similar problem highlights the danger from possible singularities
associated with the vanishing of the denominators on the right-hand side of~\eqref{E:DSSEQNBOLD}--\eqref{E:WSSEQNBOLD}.
Such points play a distinguished role in our analysis, and as we shall show shortly, are unavoidable in the study 
of physically interesting self-similar solutions.

\begin{definition}[The sonic point]\label{D:SPDEF}
For any smooth solution to~\eqref{E:DSSEQNBOLD}--\eqref{E:WSSEQNBOLD} we refer to a point $y_\ast\in(0,\infty)$
satisfying
\begin{align}
y_\ast^2 = e^{2\mu(y_\ast)-2\l(y_\ast)}
\end{align}
as the {\em sonic point}.
\end{definition}

%%%%%%%%%%%%%%%%%%%%%%%%%%%
%%%%%%%%%%%%%%%%%%%%%%%%%%%

\subsubsection{Schwarzschild self-similar formulation}\label{SS:SCHWARZSCHILD}

The comoving formulation~\eqref{E:DSSEQNBOLD}--\eqref{E:WSSEQNBOLD} as written 
does not form a closed system of ODE. To do so, we must express the metric coefficients $\mu,\lambda$ 
as functions of $\d,\w$, which can be done at the expense of working with $\tr$ (or equivalently $\ch$) as a further unknown. To avoid this, 
it is possible to introduce the so-called Schwarzschild self-similar coordinate:
%\footnote{In the language of the classical Newtonian 
%mechanics it can be thought of as the ``Eulerian" variable.}
\be\label{E:SSSDEF}
x:= \tilde r (y)
\ee
so that 
\be\label{E:TILDERPRIME}
\frac{dx}{dy} = \tilde r ' = \frac{x (\w+\k)}{y(1+\k)}.
\ee
In this coordinate system the problem takes on a form analogous to the Eulerian formulation of the self-similar Euler-Poisson system from~\cite{GHJ2021}.
It is shown in Lemma~\ref{L:SCHW} that the new unknowns
\begin{align}\label{E:RWDEF}
\R(x):=\d(y), \ \ W(x):=\w (y),
\end{align}
solve the system 
\begin{align}
\R'(x) & = -  \frac{ 2x(1-\k) \R (W + \k) (\R-W) }{B},  \label{E:RODE} \\
W'(x) & = \frac{(1 -3W )}{x} + \frac{2x(1+\k) W (W + \k) (\R-W)}{B}, \label{E:WODE}
\end{align}
where 
\begin{align}\label{E:BDEF0}
B=B[x;\R,W] : =   \R^{-\e} -\left[ ( W + \k)^2 - \k (W - 1)^2 + 4\k \R W \right] x^2,
\end{align}
and
\begin{align}\label{E:ETADEF}
\eta:=\frac{2\k}{1-\k}.
\end{align}
The Greek letter $\e$ will always be used to mean~\eqref{E:ETADEF} in the rest of the paper.
In this formulation, sonic points correspond to zeroes of $B=B[x;\R,W]$, i.e. if $y_\ast$ is a sonic point in the sense of Definition~\ref{D:SPDEF},
then $\xs:=\tr(y_\ast)$ is a zero of the denominator $B$.
%%%%%%%%%%%%%%%%%%%%%%%%%%%%
%%%%%%%%%%%%%%%%%%%%%%%%%%%%

\subsubsection{Friedmann, far-field, and the necessity of the sonic point}
%%%%%%%%%%%%%%%%%%%%%%%%%%%%
%%%%%%%%%%%%%%%%%%%%%%%%%%%%
There are two exact solutions to~\eqref{E:RODE}--\eqref{E:WODE}.
The  far-field solution
\begin{align}
%\R(x)=
\R_f(x) = 
%(1-\k)^{-2\frac{1-\k}{1+\k}} x^{-2\frac{1-\k}{1+\k}} = 
(1-\k)^{-\frac2{1+\e}} x^{-\frac2{1+\e}}, \ \
%W(x)=
W_f(x)  = 1, \label{E:FARFIELDINTRO}
\end{align}
features a density $\R$ that blows up at $x=0$ and decays to $0$ as $x\to\infty$. On the other hand, the Friedmann
solution
\begin{align}
%\R(x)=
\R_F(x)  = \frac13, \ \
%W(x)=
W_F(x)  = \frac13, \label{E:FRIEDMANNINTRO}
\end{align}
is bounded at $x=0$, but the density does not decay as $x\to\infty$. 
Our goal is to construct a smooth solution to~\eqref{E:RODE}-\eqref{E:WODE}
which qualitatively behaves like the far-field solution as $x\to\infty$ and like the Friedmann solution as $x\to0^+$. We can 
therefore think of it as a heteroclinic orbit for the dynamical system~\eqref{E:RODE}-\eqref{E:WODE}.
It is then easy to see that any such solution has the property $\lim_{x\to\infty}B=-\infty$ and $\lim_{x\to0^+}B(x)>0$. By
the intermediate value theorem there must exist a point where $B$ vanishes, i.e. a sonic point.

It is important to understand the formal Newtonian limit, which  is obtained by letting $\k =0$ in~\eqref{E:RODE}--\eqref{E:WODE}. This yields the
system
\begin{align}
\tilde\R'(x) & = -  \frac{ 2x \tilde\R \tilde W (\tilde\R-\tilde W) }{1-x^2\tilde W^2} , \label{E:RODEN} \\
\tilde W'(x) & = \frac{(1 -3\tilde W )}{x} + \frac{2x \tilde W^2 (\tilde \R-\tilde W)}{1-x^2 \tilde W^2}, \label{E:WODEN}
\end{align}
which is precisely the self-similar formulation of the isothermal Euler-Poisson system. In~\cite{GHJ2021}
we showed that there exists a solution to~\eqref{E:RODEN}--\eqref{E:WODEN} satisfying  %such that 
$\tilde W(0)=\frac13$, $\lim_{x\to\infty}\tilde W(x)=1$,
$\tilde\R(0)>\frac13$, and $\tilde\R\asymp_{x\to\infty}x^{-2}$.\footnote{For two non-vanishing functions $x\mapsto A(x)$, $x\mapsto B(x)$ we write $A\asymp_{x\to \bar x} B$ to mean 
that there exist $\k$-independent constants $C_1, C_2>0$ such that 
\[
C_1 \le \liminf_{x\to \bar x} \frac{A(x)}{B(x)} \le \limsup_{x\to\bar x} \frac{A(x)}{B(x)} \le C_2.
\]
} 
The behaviour of the relativistic solutions when $0<\k\ll1$ in the region $x\in[0,\infty)$
is modelled on this solution, called the Larson-Penston (LP) solution. 

The system~\eqref{E:RODE}--\eqref{E:BDEF0} is a non-autonomous $2\times2$ system of ODE which is at
the heart of the proof of Theorem~\ref{T:MAINTHEOREM2} and is used to show the existence of the RLP-solution in the region $x\in[0,\infty)$.
As $x\to\infty$, we are forced to switch back to a version of the comoving variables in order to extend the solution beyond $x=\infty$ in a unique way, see the 
discussion in Section~\ref{SS:MAEINTRO}. A version of the comoving formulation~\eqref{E:DSSEQNBOLD}--\eqref{E:WSSEQNBOLD} plays a crucial role
in that extension. In our analysis of the radial null-geodesics (Section~\ref{S:RLP}) and the nonradial ones (Appendix~\ref{A:NNG}), we often switch between 
different choices of coordinates to facilitate our calculations.

Even though the smallness of $\k$ is crucial for the validity of our estimates, 
we emphasise that we do not use perturbation theory to construct the relativistic
solution, by for example perturbing away from the Newtonian one. Such an argument 
is a priori challenging due to the singular nature of the sonic point, as well as complications
arising from boundary conditions.

We mention that self-similar imploding flows for the compressible Euler system, featuring a sonic point,
were constructed recently in the pioneering work of Merle, Rapha\"el, Rodnianski, and Szeftel~\cite{Merle19} - here the associated 
$2\times2$ dynamical system is autonomous. In the context of the Euler-Poisson system with polytropic gas law, self-similar collapsing
solutions featuring a sonic point were recently constructed in~\cite{GHJS2021}.
%At the heart of the proof 
%of Theorem~\ref{T:MAINTHEOREM2} is the direct analysis of the ODE system~\eqref{E:RODE}--\eqref{E:BDEF0}.
%Instead, for any given value of $\k>0$,  we analyse the
%dynamic invariances intrinsic to the relativistic flow and use the key technical tool - a new type of monotonicity lemma - Lemma~\ref{L:JUHILEMMA}. 

%%%%%%%%%%%%%%%%%%%%%%%%%%%%
%%%%%%%%%%%%%%%%%%%%%%%%%%%%

%\subsubsection{Formal Newtonian limit}

%%%%%%%%%%%%%%%%%%%%%%%%%%%%
%%%%%%%%%%%%%%%%%%%%%%%%%%%%

%The formal Newtonian limit is obtained by letting $\k =0$ in~\eqref{E:RODE}--\eqref{E:WODE}, which yields the
%system
%\begin{align}
%\tilde\R'(x) & = -  \frac{ 2x \tilde\R \tilde W (\tilde\R-\tilde W) }{1-x^2\tilde W^2}  \label{E:RODEN} \\
%\tilde W'(x) & = \frac{(1 -3\tilde W )}{x} + \frac{2x \tilde W^2 (\tilde \R-\tilde W)}{1-x^2 \tilde W^2}, \label{E:WODEN}
%\end{align}
%which is precisely the self-similar formulation of the isothermal Euler-Poisson system. In~\cite{GHJ2021}
%we showed that there exists a self-similar solution to~\eqref{E:RODEN}--\eqref{E:WODEN}, such that $\tilde W(0)=\frac13$, $\lim_{x\to\infty}\tilde W(x)=1$,
%$\tilde\R(0)>0$, and $\tilde\R\asymp_{x\to\infty}x^{-2}$. 

%%%%%%%%%%%%%%%%%%%%%%%%%%%%
%%%%%%%%%%%%%%%%%%%%%%%%%%%%

\subsection{The sonic point analysis (Section~\ref{S:SONIC})}\label{SS:SONICINTRO}

%%%%%%%%%%%%%%%%%%%%%%%%%%%%
%%%%%%%%%%%%%%%%%%%%%%%%%%%%

The solution we are trying to construct is on one hand assumed to be smooth, but 
it also must feature an a priori unknown sonic point, which we name $x_\ast$. This
is a singular point for the dynamical system and the assumption of smoothness
therefore imposes a hierarchy of constraints on the Taylor coefficients of a solution
in a neighbourhood of $x_\ast$. More precisely, 
we look for solutions $(\R ,W)$ to \eqref{E:RODE}-\eqref{E:WODE} of the form
\be\label{Taylor}
\R = \sum_{N=0}^\infty \R_N (x-\xs)^N, \quad W = \sum_{N=0}^\infty W_N (x-\xs)^N. 
\ee
It is clear that the first constraint reads $\R_0=W_0$, as the numerator in~\eqref{E:RODE} must vanish at $\xs$ for the solution to be smooth. Together with the condition 
$B(\xs)=0$, we can show that for any $\k>0$ sufficiently small, $\R_0=W_0$ is a function of $\xs$, which converges to $\frac1{\xs}$ as $\k\to0$ in accordance with the limiting Newtonian 
problem~\eqref{E:RODEN}--\eqref{E:WODEN}. 
Our goal is to express $\R_N, W_N$ recursively in terms of $\R_0,\dots, \R_{N-1}$, $W_0,\dots, W_{N-1}$ and thus obtain a hierarchy of  algebraic relations 
that allows to compute the Taylor coefficients up to an arbitrary order. 

%Already at the level of the Newtonian problem, an intriguing dichotomy emerges. 
However, an intriguing dichotomy emerges. At the next order, one obtains a cubic equation for $W_1$, see Lemma~\ref{L:CUBIC}. One of the solutions
is a ``ghost" solution and therefore unphysical, while the remaining two roots, when real, are both viable candidates for $W_1$, given 
as a function of $\R_0$ (and therefore $\xs$). This is related to an analogous dichotomy in the Newtonian case - where one choice of the root leads to
so-called Larson-Penston-type (LP type) solutions, while the other choice of the root leads to Hunter-type solutions. Based on this, for any 
$0<\k\ll1$ sufficiently small, we select $W_1=W_1(\k)$ to correspond to the choice of the branch converging to the LP-type coefficient as $\k\to0$.
This is a de facto selection principle which allows us to introduce formal Taylor expansions of the relativistic Larson-Penston-type (RLP type), see Definition~\ref{D:RLPTYPEDEF}. 
Upon fixing the choice of $W_1$ (and thereby $\R_1$), all the higher-order coefficients $(\R_N,W_N)$, $N\ge2$, are then uniquely determined 
through a recursive relation, see Section~\ref{SS:HIGHER}. We mention that $\R_1$ and $W_1$ cease to exist as real numbers before $\xs$ reaches 2 from above, which led us to define $\xc(\k)$, see Lemma \ref{L:XMINDEF}.  
This should be contrasted to the Newtonian problem where $\R_1$ and $W_1$ are real-valued as $\xs$ passes below $2$\footnote{The LP- and Hunter-type solutions coincide at $\xs=2$.}. The existence of  
a forbidden range for $\xs$ is related to the band structure in the space of all smooth solutions, see~\cite{OP1990,GoNiUg1998}.

Guided by the intuition developed in the construction of the (non-relativistic) Larson-Penston
solution~\cite{GHJ2021}, our next goal is to identify the so-called sonic window - a closed
interval $[\xmin,\xmax]$ within which we will find a sonic point for the global solution of the
ODE system on $[0,\infty)$. %In the nonrelativistic case ($\k=0$), such a window is given by the interval $[2,3]$.
In fact, by Lemma \ref{L:PREPX} and Lemma \ref{L:XYNONEMPTY}, the set $W<\frac13$ and the set $W>\frac{1}{2-2\e}=\frac12+O(\k)$ are invariant under the flow to the left of the sonic point. Motivated by this, 
we choose $\xmax=\xmax(\k)<3$ so that the zero order coefficient $W_0$ coincides with the Friedmann solution \eqref{E:FRIEDMANNINTRO}: $W_0|_{\xs = \xmax}= \frac13$ (see \eqref{E:XMAXDEF}) and fix $\xmin=2+\delta_0>\xc$ for $\delta_0>0$ sufficiently small but independent of $\k$, so that $W(x;\xmin)> \frac{1}{2-2\e}$  for some $x<\xmin$ (see \eqref{E:XMINDEF}). %The choice of these end points $\xmin,\xmax$  is intimately tied to the invariances of the flow to the left of the sonic point.

The main result of Section~\ref{S:SONIC} is Theorem~\ref{T:SONICLWP}, which states that there exists an $0<\k_0\ll1$ sufficiently small such that for all $0<\k\le\k_0$ and for any choice of 
$\xs$ in the sonic window, there in fact exists a local-in-$x$ real analytic solution around $x=\xs$. The proof of this theorem relies
on a delicate combinatorial argument, where enumeration of indices and $N$-dependent growth bounds for the coefficients $(\R_N,W_N)$ 
are moved to Appendix~\ref{A:SONIC}.

Having fixed the sonic window $[\xmin,\xmax]\subset [2,3]$, our strategy is to determine
what values of $\xs\in[\xmin,\xmax]$ allow for RLP-type solutions that exist on the whole real
line. We approach this problem by splitting it into two subquestions. We identify those $\xs$ which 
give global solutions to the left, i.e. all the way from $x=x_\ast$ to $x=0$, and separately to the right, i.e. on $[\xs,\infty)$.

%%%%%%%%%%%%%%%%%%%%%%%%%%%%
%%%%%%%%%%%%%%%%%%%%%%%%%%%%

\subsection{The Friedmann connection (Section~\ref{S:FRIEDMANN})}

%%%%%%%%%%%%%%%%%%%%%%%%%%%%
%%%%%%%%%%%%%%%%%%%%%%%%%%%%

The main goal of Section~\ref{S:FRIEDMANN} is to identify a value $\bar x_\ast\in[\xmin,\xmax]$, so that the associated local solution
$(\R(\cdot;\bar x_\ast), W(\cdot;\bar x_\ast))$ exists on $[0,\bar{x}_\ast]$ and is real analytic everywhere. 

%\subsubsection{The monotonicity lemma}

From the technical point of view the main obstruction to our analysis is 
the possibility that the flow features more than one sonic point.   
By the precise analysis around $x=\xs$ in Section~\ref{S:SONIC} we know that $B$ is strictly positive/negative locally around $\xs$ to the left/right respectively (and vanishes at $x=\xs$). Our strategy 
is to propagate these signs dynamically. To do so we develop a technical tool, referred to as the monotonicity lemma, even though it is an exact identity, see Lemma~\ref{L:JUHILEMMA}.
It is a first order differential equation for a quantity $f(x) = J[x;D]-xD$ with a source term depending on the solution, but with good sign properties
in the regime we are interested in, hence -- monotonicity lemma. Here $J[x;D]-xW$ is a factor of $B$ in \eqref{E:BDEF0}: $B=(1-\k)(J-xW)(J+2\e(1+D)x+ xW)$  %$f(x)$ 
(see Lemmas~\ref{L:BALGEBRA}--\ref{L:JUHILEMMA}).  Roughly speaking, the function $f$ allows us to relate the sign of $B$ to the sign of the difference $(D-W)$ in a precise dynamic way,
so that we eventually show that 
away from the sonic point  $D>W$ and $B>0$ to the left, while $D<W$ and $B<0$ to the right of the sonic point for the relativistic Larson-Penston solution.

To construct the solution to the left of the sonic point
we develop a shooting-type method, which we refer to as {\em shooting toward Friedmann}. Namely, the requirement of smoothness
at $x=0$ is easily seen to imply $W(0)=\frac13$, which precisely agrees with the value of $W_F$, see~\eqref{E:FRIEDMANNINTRO}.

The key idea is to separate the sonic window $[\xmin,\xmax]$ into the sets of sonic points $\xs$ that launch solutions $W(\cdot;\xs)$ which either stay above the Friedmann value $W_F=\frac13$ 
on its maximum interval of existence $(s(\xs),\xs]$ or cross it, see Figure~\ref{F:SHOOTING}. This motivates the following definition.

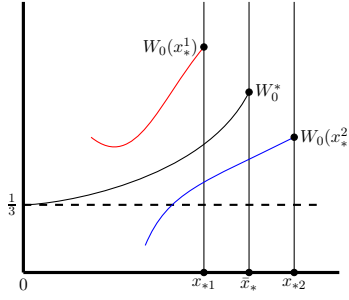
\begin{figure}
\begin{center}
\begin{tikzpicture}
\begin{scope}[scale=0.6, transform shape]

\coordinate[label=below:$0$] (A) at (0,0){};
\coordinate[label=below:$x_{\ast1}$] (B) at (4,0){};
\coordinate[label=below:$\bar x_\ast$] (C) at (5,0){};
\coordinate[label=below:$x_{\ast2}$] (D) at (6,0){};
\coordinate[label = left:$\frac13$] (E) at (0,1.5){};

% add black circle at B
\draw[fill=black] (B) circle (2pt);

% add black circle at C
\draw[fill=black] (C) circle (2pt);

% add black circle at D
\draw[fill=black] (D) circle (2pt);

% draw horizontal axis
\draw[very thick] (A)--(7,0);

%draw vertical axis
\draw[very thick] (A)--(0,6);

%draw 1/3 line
\draw[dashed, thick] (E)--(6.5,1.5);

% draw vertical axis at \xs1
\draw (B) -- (4,6){};

% draw vertical axis at \bar x_\ast
\draw (C) -- (5,6){};

% draw vertical axis at \xs2
\draw (D) -- (6,6){};

% shooting point for \xs2
\coordinate[label=right:$W_0(\xs^2)$] (F) at (6,3){};

% draw W curve for \xs2
\draw[blue] (F) .. controls +(-2,-1) and +(0.5,1.2) .. (2.7,0.6);

% add black circle at F
\draw[fill=black] (F) circle (2pt);

% shooting point for \bar x_\ast
\coordinate[label=right:$W_0^\ast$] (G) at (5,4){};

% draw W curve for \bar x_\ast
\draw (G) .. controls +(-1,-2) and +(.7,0) .. (E);

% add black circle at G
\draw[fill=black] (G) circle (2pt);

% shooting point for \xs1
\coordinate[label=left:$W_0(\xs^1)$] (H) at (4,5){};

% draw W curve for \xs1
\draw[red] (H) .. controls +(-1,-1.2) and +(1,-0.8) .. (1.5,3);

% add black circle at H
\draw[fill=black] (H) circle (2pt);

\end{scope}
\end{tikzpicture}
 \caption{Schematic depiction of the shooting argument. Here $x_{\ast1}\in\X$, $x_{\ast2}\in\Y$. The critical point $\bar x_\ast$ is obtained by sliding to the left
 in $\Y$ until we reach the boundary of its first connected component $X$.}\label{F:SHOOTING}
\end{center}
\end{figure}

\begin{definition}[$\X $, $\Y $, $\Z$, and $X $]\label{D:XYZ}
Let $\k_0>0$ be a small constant introduced in Section~\ref{SS:SONICINTRO} (see also Theorem~\ref{T:SONICLWP}). For any $\k\in(0,\k_0]$ and $x_\ast\in[\xmin,\xmax]$ we consider the associated RLP-type solution $(\R(\cdot;x_\ast), W(\cdot;x_\ast))$.
We introduce the sets
\begin{align}\label{E:XDEF}
\X & : = \left\{x_\ast\in[\xmin,\xmax]\,\Big| \inf_{x\in(s(x_\ast),x_\ast)}W(x;x_\ast)>\frac13 \right\}, \\
\label{E:YDEF13}
\Y &: = \left\{x_\ast\in[\xmin,\xmax]\,\big| \   \exists  \ x\in (s(x_\ast), x_\ast) \ \text{such that } W(x;x_\ast)=\frac13\right\}, \\
\label{E:ZDEF}
\Z &: = \left\{x_\ast\in[\xmin,\xmax]\,\Big| W(x;x_\ast)>\frac13 \ \text{ for all } \ x\in(s(x_\ast),x_\ast) \ \text{ and } \ \inf_{x\in(s(x_\ast),x_\ast)}W(x;x_\ast)\le\frac13\right\}.
\end{align}
Finally, we introduce the \underline{fundamental set} $X\subset \Y$ given by
\begin{align}
X: = \left\{x_\ast\in[\xmin,\xmax]\,\big| \ \tilde x_\ast \in \Y \ \ \text{ for all } \ \ \tilde x_\ast\in [x_\ast,\xmax] \right\}. 
\end{align}
\end{definition}

The basic observation is that solutions that correspond to the set $\Y$ have the property that once
they take on value $\frac13$, they never go back up above it. This is a nonlinear invariance of the flow, which 
guides our shooting argument idea.  It is possible to show that both sets $\X$ and $\Y$ are non-empty.
In Lemma~\ref{L:XYNONEMPTY} we show that $\xmin\in\X$ and that for some $\kappa>0$, $(\xmax-\kappa,\xmax]\subset\Y$.
Intuitively, we then slide down $\xs$ starting from $\xmax$, until we reach the first value of $\xs$ that does not belong to $\Y$, i.e.
we let
\begin{align}\label{critical_x_INTRO}
\bar x_\ast : = \inf_{x_\ast\in X } x_\ast,
\end{align}
see Figure~\ref{F:SHOOTING}.
This is the candidate for the value of $\xs$ which gives a real-analytic solution on $[0,\bar x_\ast]$. 
Using the nonlinear invariances of the flow and its continuity properties, in Proposition~\ref{P:GLOBALFRIEDMAN},
we show that the solution $(\R(\cdot;\bar x_\ast),W(\cdot,\bar x_\ast))$ exists on the semi-open interval $(0,\bar x_\ast]$.

To show that the solution is indeed analytic all the way to $x=0$, $W(0;\bar x_\ast)=\frac13$, and $\R(0;\bar x_\ast)>\frac13$, we 
adapt the strategy developed for the classical LP solution in~\cite{GHJ2021}. Using the method
of {\em upper} and {\em lower} solutions (Definition~\ref{D:UPPERLOWER}), we show that there 
exists a choice of $D_0>\frac13$ such that the solution $(\R(\cdot;\bar x_\ast),W(\cdot,\bar x_\ast))$ coincides with 
a unique real analytic solution to~\eqref{E:RODE}--\eqref{E:WODE}, with data $D(0)=D_0$, $W(0)=\frac13$, solving from
$x=0$ to the right. Detailed account of this strategy is contained in Sections~\ref{SS:TOTHERIGHT}--\ref{SS:UL}.
Finally, combining the above results we can prove the central statement of Section~\ref{S:FRIEDMANN}:

%%%%%%%%%%%%%%%%%%%%%%%%%%%%
%%%%%%%%%%%%%%%%%%%%%%%%%%%%

\begin{theorem}\label{T:FRIEDMANN}
There exists an $\k_0>0$ sufficiently small, such that for any $0<\k\le\k_0$, the solution of RLP-type to~\eqref{E:RODE}--\eqref{E:WODE} launched at $\bar x_\ast$ (defined by~\eqref{critical_x_INTRO})
extends (to the left) to the closed interval $[0,\bar x_\ast]$, is real analytic, and satisfies $W(0;\bar x_\ast)=\frac13$, $\R(0;\bar x_\ast)>\frac13$.
\end{theorem}

This theorem is formally proved at the very end of Section~\ref{SS:UL}.
%%%%%%%%%%%%%%%%%%%%%%%%%%%%
%%%%%%%%%%%%%%%%%%%%%%%%%%%%

\subsection{The far-field connection (Section~\ref{S:FARFIELD})} \label{SS:FFINTRO}

%%%%%%%%%%%%%%%%%%%%%%%%%%%%
%%%%%%%%%%%%%%%%%%%%%%%%%%%%

By contrast to establishing the existence of the Friedmann connection in Section~\ref{S:FRIEDMANN}, 
we show that for any choice of $\xs$ in our sonic window $[\xmin,\xmax]$, there exists a global solution to the right, 
i.e. we prove the following theorem:

%%%%%%%%%%%%%%%%%%%

\begin{theorem}\label{T:GLOBALRIGHT}
Let $x_\ast\in[\xmin,\xmax]$. There exists an $0<\k_0\ll1$ such that the unique RLP-type solution $(W,\R)$ exists globally to the right for all $\k\in(0,\k_0]$.
\end{theorem}

%%%%%%%%%%%%%%%%%%%%

The proof relies on a careful study of nonlinear invariances of the flow, and once again the monotonicity properties
encoded in Lemma~\ref{L:JUHILEMMA} play a critical role in our proof. This is highlighted in Lemma~\ref{L:FLOWERBOUND}, 
which allows us to propagate the negativity of the sonic denominator $B$ to the right of the sonic point. 
Importantly, we may now let $\xs = \bar x_\ast$
defined in~\eqref{critical_x_INTRO} to obtain a real-analytic RLP-type solution defined globally on $[0,\infty)$. As it turns out, 
the obtained spacetime is not maximally extended, and to address this issue we need to understand the
asymptotic behaviour of our solutions as $x\to\infty$.

%We then study the asymptotic behaviour of
%the solutions as $x\to\infty$. It turns out that these results are very important to 
%understand the meaning of the unique extension of the solution beyond $x=\infty$ in Section~\ref{S:MAE}.
In Lemma~\ref{L:ROUGHASYMPTOTICS} we show that solutions from Theorem~\ref{T:GLOBALRIGHT}
honour the asymptotic behaviour
\be\label{E:AINTRO1}
\lim_{x\to\infty} W(x;\xs)=1, \ \ \lim_{x\to\infty} \left[D(x;\xs)x^{\frac2{1+\e}}\right]>0;
\ee
hence the name  far-field connection, see~\eqref{E:FARFIELDINTRO}. This is however not
enough for the purposes of extending the solution beyond $x=\infty$, as we also need sharp asymptotic behaviour of
the relative velocity $W$. Working with the nonlinear flow~\eqref{E:RODE}--\eqref{E:WODE}, in Proposition~\ref{P:PRECISEW} we show that 
%there exists a $\vY<0$ such that 
the leading order behaviour of $W$ is given by the relation 
\be\label{E:AINTRO2}
1-W \asymp_{x\to\infty} x^{-\frac1{1+\e}}.
\ee

%%%%%%%%%%%%%%%%%%%%%%%%%%%%
%%%%%%%%%%%%%%%%%%%%%%%%%%%%

\subsection{Maximal analytic extension (Section~\ref{S:MAE})} \label{SS:MAEINTRO}

%%%%%%%%%%%%%%%%%%%%%%%%%%%%
%%%%%%%%%%%%%%%%%%%%%%%%%%%%

Asymptotic relations~\eqref{E:AINTRO1}--\eqref{E:AINTRO2} suggest that our unknowns are 
asymptotically ``regular" only if thought of as functions of $x^{-\frac1{1+\e}}$. In fact, it turns out
to be more convenient to interpret this in the original comoving self-similar variable $y$, see~\eqref{E:SSCDEF}. 
Due to~\eqref{E:AINTRO1} and~\eqref{E:SSSDEF}--\eqref{E:TILDERPRIME}, it is easy to see that asymptotically
\be\label{E:XEQUALY}
x\asymp_{x\to\infty} y.
\ee
Moreover, by~\eqref{E:AINTRO1} and~\eqref{E:RWDEF}, we have $\d(y)\asymp_{y\to\infty}y^{-\frac{2}{1+\e}}$. 
Furthermore by~\eqref{E:MUFORMULASS2} we have the relation
\be\label{E:GTT}
e^{2\mu}= \frac1{(1+\k)^2}  \Sigma^{-\frac{\e}{1+\e}} = \frac1{(1+\k)^2} \d^{-\e}.
\ee
As a consequence of~\eqref{E:AINTRO1}
we then conclude that 
\begin{align}\label{E:MUAS}
e^{2\mu}\asymp_{y\to\infty}  y^{\frac{4\k}{1+\k}},
\end{align}
which implies that the metric $g$ given by~\eqref{E:METRIC}
becomes singular as $y\to\infty$ (or equivalently $x\to\infty$). 
In Section~\ref{S:MAE} we show that this is merely a coordinate singularity, and the spacetime extends smoothly (in fact analytically in a suitable choice of coordinates) across the surface
$\{(\tau\equiv 0,R)\, \big|\, R>0\}$. 

Motivated by the above considerations, we switch to an adapted comoving chart $(\tilde\tau,R)$,
%To smoothly extend the spacetime across the hypersurface $y=\infty$, 
defined through 
\be\label{E:TILDETAUDEF}
\tilde\tau(\tau,R) =- \k^{-\frac{\k}{1+\k}} R^{\frac{2\k}{1+\k}}(-\tau)^{\frac{1-\k}{1+\k}}, \ \ \tau<0, \  R>0.
\ee 
%\[
%y^{\frac{4\k}{1+\k}} d\tau^2 = d\tilde\tau^2,
%\]
%which is easily solved to give
%\begin{align}
%\tilde\tau = R^{\frac{2\k}{1+\k}} \tau^{\frac{1-\k}{1+\k}}.
%\end{align}
We introduce the self-similar variable 
\be\label{E:YDEF}
Y: = \frac{-\sqrt \k\tilde\tau}{R},
\ee
which is then easily checked to be equivalent 
to the change of variables
\begin{align}\label{E:YLITTLEY}
Y 
%\frac1{y^\frac{1}{1+\e}} 
= y^{-\frac{1}{1+\e}}, \ \ y = Y^{-1-\e},
\end{align}
where we recall $\e= \e(\k)=\frac{2\k}{1-\k}$. To formulate the extension problem,
it is natural to define the new variables
\begin{align}\label{E:YVARDEF}
%r(Y) = \rl(y), 
\chi(Y): = \ch(y), \ \ \RY(Y) : =  \d(y),  \ \ w(Y) : = \wl (y),
\end{align} 
where we recall the fundamental variables $\d,\w,\ch$ from~\eqref{E:LITTLEBOLDDDEF}--\eqref{E:BOLDFACECHI}.
Note that by~\eqref{E:BOLDFACECHI} and~\eqref{E:YLITTLEY} we have $\chi(Y)=Y^{1+\e}  \tr(y)$.
It is shown in Lemma~\ref{L:ADAPTED} that the original system~\eqref{E:DSSEQNBOLD}--\eqref{E:WSSEQNBOLD}
in the new variables reads
\begin{align}
\RY' &=    \frac{2\chi^2}{Y} \frac{ \RY (\vY+\k)^2(\RY-\vY)}{\mathcal C}, \label{E:LDEQN2}\\
\vY' &= -\frac{(\vY+\k)(1-3\vY)}{(1-\k)Y} -\frac{2(1+\k)\chi^2}{(1-\k)Y} \frac{\vY(\vY+\k)^2(\RY-\vY)}{\mathcal C},  \label{E:LWEQN2} \\
\chi' & = \frac{1-\vY}{(1-\k)Y} \chi, \label{E:CHIEQN}
\end{align}
where
\begin{align}\label{E:CDEF}
\mathcal C : = \left(\RY Y^{-2}\right)^{-\eta}Y^2 - \chi^2 \left[(\vY + \k)^2-\k(\vY-1)^2 + 4\k \vY \RY \right].
\end{align}

The first main result of Section~\ref{S:MAE} is Theorem~\ref{T:LE}, where 
we prove the local existence of a real analytic solution in an open neighbourhood of $Y=0$, which provides the local 
extension of the solution from $Y=0^+$ to $Y=0^-$.  Initial
data at $Y=0$ are read off from the asymptotic behaviour~\eqref{E:AINTRO1}--\eqref{E:AINTRO2}, 
see Remark~\ref{R:TAYLORR}. The most important result of the section
is the maximal extension theorem of the solution to the negative $Y$-s:

%%%%%%%%%%%%%%%%%%%%%%%%%%%%
%%%%%%%%%%%%%%%%%%%%%%%%%%%%

\begin{theorem}[Maximal extension]\label{T:GE}
There exists an $0<\k_0\ll1$ sufficiently small such that for any $\k\in(0,\k_0]$  there exists a $\yms<0$ such that the unique solution to the initial value problem~\eqref{E:LDEQN2}--\eqref{E:CHIEQN} exists on the interval $(\yms,0]$,
and
\begin{align}
\lim_{Y\to (\yms)^-}\vY(Y) = \lim_{Y\to (\yms)^-}\RY(Y) & = \infty, \\
\chi(Y)& >0, \ \ Y\in(\yms,0], \\
\lim_{Y\to (\yms)^-}\chi(Y) & = 0.
\end{align}
\end{theorem}

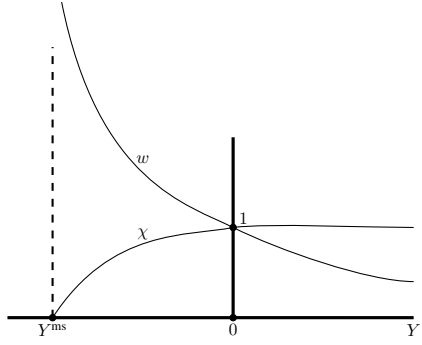
\begin{figure}
\begin{center}
\begin{tikzpicture}
\begin{scope}[scale=0.6, transform shape]

\coordinate [label=below:$0$] (A) at (-3,7){};
\coordinate [label=below:$\yms$] (B) at (-7,7){};
\coordinate [] (C) at (-7,11){};
\coordinate [label=below:$Y$] (D) at (1,7){};
\coordinate [label=above:$\chi$] (E) at (-5,8.6){};
\coordinate [label=right:$1$] (F) at (-3,9.2){};
\coordinate [label=above:$\vY$] (F) at (-5,10.3){};

%\coordinate [label=right:$\mathcal N$] (F) at (-1.3,5.5){};
%\coordinate [label=left:$\ttau$] (G) at (-3,10.5){};
%\coordinate [label=left:$\frac1{g(0)}$] (D) at (-3,11){};
%\coordinate  (L) at (0.2,12.5){};
%\coordinate  (M) at (1,13.9){};
%\tkzCircumCenter(D,L,M)\tkzGetPoint{O}
%\draw[axis] (-3,7)  -- (1,7) node(xline)[right] {$k$};
\draw[very thick] (-3,7)--(-8,7);
\draw[very thick] (A)--(1,7);
\draw[very thick] (A)--(-3,11);
\draw[dashed, thick] (B) -- (-7,13);
\draw (-7,7) .. controls +(1.3,2) and +(-1,-0.2) .. (-3,9);
\draw (-3,9) .. controls +(0.8,0.1) and +(-1,0) .. (1,9);
\draw (-6.8,14) .. controls +(0.8,-4) and +(-1,0.5) .. (-3,9);
\draw (-3,9) .. controls +(1,-0.5) and +(-1,0) .. (1,7.8);

%\draw[thick] (A)--(-1,10);
%\draw[thick] (A)--(0,9);
%\draw[thick] (A)--(0.5,8);
%\draw[thick] (A)--(0,4);
\draw[fill=black] (-3,7) circle (2pt);
\draw[fill=black] (-7,7) circle (2pt);
\draw[fill=black] (-3,9) circle (2pt);

%\draw [fill=gray!20, very thick] (D) -- (A) -- (B) -- (M);
%\tkzDrawArc[fill=white, dashed, very thick](O,D)(M)
%\coordinate [label=below: $\Xi$] (B) at (-1,9.2) {};
%\node at (-1,12) {$\Gamma$};
\end{scope}
\end{tikzpicture}
 \caption{Schematic depiction of the behaviour of $\chi(Y)$ and $\vY(Y)$ in the maximal extension. As $Y$ approaches $\yms$
 from the right, $\chi$ approaches $0$ and $\vY$ blows up to $\infty$.}
\end{center}
\end{figure}

We see from the statement of the theorem that the maximal extension is characterised by the simultaneous blow-up of $\vY, \RY,$ and $\frac1{\chi}$ at the terminal point $\yms$.
By~\eqref{E:YDEF}, in the adapted comoving chart, the point $\yms$ coincides with 
the hypersurface
\[
\MS:=\left\{(\ttau,R)\,\big| \ttau = \frac1{\sqrt\k} |\yms| R\right\} \setminus \{(0,0)\} , 
\]
which we refer to as the {\em massive singularity}, following the terminology in~\cite{OP1990}.
The proof of Theorem~\ref{T:GE} relies on a careful understanding of nonlinear invariances associated with the dynamical system~\eqref{E:LDEQN2}--\eqref{E:CHIEQN} and the key 
dynamic ``sandwich" bound  
\[
1\lesssim \frac{\RY}{\vY} <1, \ \text{ for } Y<Y_0<0,
\]
where $Y_0<0$ is small. This is shown in Lemma~\ref{L:SIGMAVBOUND}.  The blow-up proof finally follows from a Ricatti-type ordinary differential inequality for the 
relative velocity $\vY$. 

In Section~\ref{SS:MASSIVE} we compute the sharp asymptotics of $\vY, \RY,$ and $\chi$ on approach to the massive singularity $\yms$. This result
is stated in Proposition~\ref{P:SINGBEHAVIOUR}, which is later crucially used in the study of the causal structure of such a maximally extended solution,
where it is in particular shown that the spacetime curvature blows up on approach to the massive singularity.

\begin{remark}\label{R:MSLIMIT}
The maximal self-similar extension makes sense in the Newtonian limit $\k\to0$, which is one of the key observations of Ori and Piran~\cite{OP1990}.
Our proof of Theorem~\ref{T:GE} easily extends to the simpler case $\k=0$, which in particular shows that the LP-solutions
constructed in~\cite{GHJ2021} have a natural maximal extension in the Lagrangian (comoving) coordinates.
\end{remark}

%%%%%%%%%%%%%%%%%%%%%%%%%%%%
%%%%%%%%%%%%%%%%%%%%%%%%%%%%

\subsection{The RLP-spacetime and its causal structure (Section~\ref{S:RLP})} \label{SS:RLPINTRO}

%%%%%%%%%%%%%%%%%%%%%%%%%%%%
%%%%%%%%%%%%%%%%%%%%%%%%%%%%

As a consequence of Theorems~\ref{T:FRIEDMANN},~\ref{T:GLOBALRIGHT}, and~\ref{T:GE}, 
we can now formally introduce the exactly self-similar solution of the Einstein-Euler system by patching the solutions in the subsonic region $x\in[0,\bar x_\ast]$, the supersonic region
$x\in [\bar x_\ast,\infty)$, and the extended region $Y\in(\yms,0]$.
Following Ori and Piran~\cite{OP1990}, we call this spacetime
the relativistic Larson-Penston (RLP) solution\footnote{The nomenclature is motivated by their Newtonian analogue 
discovered by Larson~\cite{Larson1969} and Penston~\cite{Penston1969} in 1969, see~\cite{GHJ2021} for the rigorous proof of existence of the Newtonian solution.}.

%%%%%%%%%%%%%%%%%%%%%%%%%%%%%
%%%%%%%%%%%%%%%%%%%%%%%%%%%%%

\begin{definition}[$\grlp$-metric]\label{D:RLPST}
%\begin{enumerate}
%\item 
We refer to the $1$-parameter family of spherically symmetric self-similar spacetimes
$(\MRLP, %\mathcal M_\k,
\grlp)$
constructed above
%\begin{align}
%\mathcal M_\k &= \left\{(\tilde \tau, R) \, \Big| Y\in (\yms,\infty)\right\} \setminus \{(0,0)\} \\
%\grlp & = -e^{2\mu_\k(Y)} d\tilde\tau^2 + e^{2\lambda_\k(Y)}dR^2 + r_\k(Y)^2 \ d\Omega, \ \ Y: = -\frac{\sqrt\k\tilde\tau}{R}. \label{E:METRICRLP}
%\end{align}
as the \underline{relativistic Larson-Penston spacetimes}.
%In $(\tau,R)$ coordinates, the metric takes the form
%\begin{align}
%\grlp & = -e^{2\mu_\k} d\tau^2 + e^{2\lambda_\k}dR^2 +  r_\k^2 \ \gamma,  \label{E:METRICRLP1}
%\end{align}
%where the metric coefficients are defined on the connected component of the $(\tau,R)$ coordinate plane given by 
%\begin{align}
%\DRLP : = \left\{(\tau, R) \in (-\infty,0)\times (0,\infty)\right\}.
%\end{align}
%Here 
%\begin{align}
%r_\k(\tau,R) = - \sqrt\k \tau  \tr_\k(y), \ \ \mu_\k(\tau,R) = \mu_\k(y), \ \ \lambda_\k(\tau,R) = \lambda_\k(y),
%\end{align}
%where
%$
%y = \frac{R}{-\sqrt\k \tau}.
%$
In the adapted comoving coordinates $(\tilde\tau,R)$ the metric takes the form
\begin{align}
%\mathcal M_\k &= \left\{(\tilde \tau, R) \, \Big| Y\in (\yms,\infty)\right\} \setminus \{(0,0)\} \\
\grlp & = - e^{2\tilde\mu}\,d\tilde\tau^2 -\frac{4\sqrt\k}{1+\k} Y e^{2\tilde\mu}\,d\tt\,dR + \left(e^{2\tilde\l}-\frac{4\k}{(1+\k)^2} Y^2 e^{2\tilde\mu} \right)\,dR^2 + r^2\ \gamma, %\Omega^2, 
\label{E:METRICRLP2}
\end{align}
where the metric coefficients are defined on the connected component of the $(\tilde\tau,R)$ coordinate plane given by 
\begin{align}\label{E:MATHCALDTILDE}
\DRLPtilde : = \left\{(\tilde\tau, R)\,\big| \, R>0, \ \   Y\in (\yms,\infty)\right\}.
\end{align}
Here 
\begin{align}
r(\ttau,R) =  \chi(Y) R, \ \ \tilde\mu(\ttau,R) = \tilde\mu(Y), \ \ \tilde\lambda(\ttau,R) = \tilde\lambda(Y),
\end{align}
where
$
Y = -\frac{\sqrt\k\tilde\tau}{R},
$
\begin{align}\label{E:MUTILDE}
e^{2\tilde\mu(Y)}  =   \frac{(1+\k)^2}{(1-\k)^2} Y^{2\e} e^{2\mu(y)}, \ \ y>0,
\end{align}
and
\begin{align}\label{E:LAMBDATILDE}
e^{2\tilde\l(Y)} = e^{2\l(y)}, \ \ y>0.
\end{align}
Metric coefficients $\tilde\mu(Y),\tilde\l(Y)$ for $Y\le0$ are then defined by expressing them as appropriate functions of $\RY,\vY,\chi$ and extending to $Y\le0$, see Proposition~\ref{P:SINGBEHAVIOUR}.
%The metric coefficients are defined on the connected component of the $(\tau,R)$ coordinate plane given by 
%\begin{align}
%\DRLP : = \left\{(\tau, R) \, \Big| \, y\in (-\infty,\yms)\cup (0,\infty)\right\},
%\end{align}
%where
%\begin{align}
%\yms: = \frac1{Y_\k}.
%\end{align}
%
%\item
%The hypersurface $\mathcal M\mathcal S_\k$ defined through
%\begin{align*}
%\MS &: = \left\{ (\tilde\tau, R) \in\partial \DRLPtilde \, \Big| \,  Y = \yms \right\}  \\
%& = \left\{ (\tilde\tau, R) \, \Big| \, R = \frac{\sqrt\k}{|\yms|} \tilde\tau\right\} \setminus \{(0,0)\}.
%\end{align*}
%is called the \underline{massive singularity}. 
%\end{enumerate}
\end{definition}

{\em The RLP-spacetime in the original comoving coordinates.}
In the original comoving  coordinates $(\tau,R)$, the metric takes the form
\begin{align}
\grlp & = -e^{2\mu} d\tau^2 + e^{2\lambda}dR^2 +  r^2 \ \gamma.  \label{E:METRICRLP1}
\end{align}
It is clear from~\eqref{E:MUAS} that the metric~\eqref{E:METRICRLP1}
%\begin{align}
%&g = -e^{2\mu(\tau, R)} \, d\tau^2 + e^{2\l(\tau,R)}\, dR^2 + r^2(\tau,R) \, \gamma, \label{E:METRICAPP}
%%& g_{\tau\tau} =- e^{2\mu}, \ \ g_{RR}= e^{2\l}, \ \ g_{\theta\theta} = r^2, \ \ g_{\phi\phi} = r^2\sin^2\theta. 
%\end{align}
becomes singular across the surface $\tau=0$ (equivalently $y=\infty$). Nevertheless, away from this surface (i.e. when $\tau>0$ and $\tau<0$) we can keep using comoving coordinates, whereby 
we define
\begin{align}
\tau =  |Y|^\e \ttau, \ \ \ttau>0,
\end{align}
where we recall~\eqref{E:YDEF}.
Therefore the metric coefficients are defined on the union of the connected components of the $(\tau,R)$ coordinate plane given by 
\begin{align}
\DRLP : = \left\{(\tau, R) \in (-\infty,0)\times (0,\infty)\right\} \cup \left\{(\tau, R) \in (0,\infty)\times (0,\infty)\,\big| \, y\in(-\infty, y^{\text{ms}})\right\},
\end{align}
where 
\[
y^{\text{ms}} : = - |\yms|^{-(1+\e)}.
\]
Here 
\begin{align*}
r(\tau,R) = - \sqrt\k \tau  \tr(y), \ \ \mu(\tau,R) = \mu(y), \ \ \lambda(\tau,R) = \lambda(y), \ \ \text{ for $\tau<0$},
\end{align*}
and
\[
r = - \sqrt\k \tau \tr (y), \ \ \tr(y) = \ch(y) y, \ \ y = \frac{R}{-\sqrt\k \tau} =- |Y|^{-1-\e}, \ \ \text{ for $\tau>0$},
\]
where $\ch(y)=\chi(Y)$, $y<0$.

\begin{remark}\label{R:TILDRREG} 
It is of interest to understand the leading order asymptotic behaviour of the 
radius $\tilde r$ as a function of the comoving self-similar variable $y$. It
follows from~\eqref{E:RWDEF} and the boundary condition $W(0)=\frac13$
that $\lim_{y\to0^+}\w(y) = \frac13
%\frac{1+3\k}{3(1+\k)}
$. Therefore, using~\eqref{E:TILDERPRIME}
it is easy to see that the leading order behaviour of $\tilde r(y)$ at $y=0$ is of the form
\begin{align}\label{E:TILDERREG}
\tilde r(y) = \tilde r_0 y^{\frac{1+3\k}{3(1+\k)}} + o_{y\to0^+}(y^{\frac{1+3\k}{3(1+\k)}}), \ \ \tr'(y) \asymp_{y\to 0^+} y^{-\frac{2}{3(1+\k)}},
\end{align} 
for some $\tilde r_0>0$. It follows in particular that $y\mapsto \tilde r(y)$ is only $C^{0,\frac{1+3\k}{3(1+\k)}}$ 
at $y=0$. 
The self-similar reduction of the constraint equation~\eqref{E:LAMBDACONSTRAINT1} reads $\tr'^2 e^{-2\l} = 1+\k V^2 - \frac{2}{3} \k \tG \tr^2$ (see~\eqref{E:LAMBDACONSTRAINTSS}).
Using this and~\eqref{E:TILDERREG}, it then follows 
\be\label{E:LAMBDABADATZERO}
e^{\l(y)}\asymp_{y\to0^+} \tr'(y) \to_{y\to 0^+} \infty, 
\ee 
which shows that the metric $g$~\eqref{E:METRIC} is singular at $y=0$.

This singularity is not geometric, but instead caused by the Friedmann-like behaviour of the relative velocity $W$ at $y=0$, see~\eqref{E:FRIEDMANNINTRO}.
It captures an important difference between the comoving and the Schwarzschild coordinates at the centre of symmetry
$\{(\tau, R)\, \big| \, \tau<0, R=0\}$. It can be checked that the space-time is regular at the centre of symmetry by switching to the $(\tau,r)$ coordinate.
The same phenomenon occurs in the Newtonian setting, where it can be shown that the map $\ch = \frac{\tr}{y}$ associated with the LP-solution is exactly $C^{0,\frac13}$ and therefore $\ch$
is not smooth at the labelling origin $y=0$.
\end{remark}

%\begin{remark}

\begin{remark}
Note that the trace of $T_{\mu\nu}$ is easily evaluated
\[
g^{\mu\nu}T_{\mu\nu} = (\rho+p)g^{\mu\nu}u_\mu u_\nu + 4 p = - \rho+3p = -(1-3\k)\rho.
\]
On the other hand, the trace of the left-hand side of~\eqref{E:EINSTEIN} is exactly $-\mathcal R$, and therefore
the Ricci scalar of
any classical solution of~\eqref{E:EINSTEIN}--\eqref{E:EOS} satisfies the relation
\begin{align}\label{E:CURVATUREDENSITY}
\mathcal R = (1-3\k)\rho.
\end{align} 
This relation also implies that the blow-up of the Ricci scalar is equivalent to the blow up of the mass-energy density when $\k\neq\frac13$.
\end{remark}

%{\em Singular boundaries.}
%The massive singularity $\MS$ is in the causal future of the Cauchy horizon and therefore has no
%canonical physical meaning, but it is nevertheless useful for our analysis, see the discussion in Section~\ref{SSS:OUTGOING}.

%%%%%%%%%%%%%%%%%%%%%%%
%%%%%%%%%%%%%%%%%%%%%%%

\subsubsection{The outgoing null-geodesic}\label{SSS:OUTGOING}

%%%%%%%%%%%%%%%%%%%%%%%
%%%%%%%%%%%%%%%%%%%%%%%

The maximally extended RLP spacetime constructed above has two singular boundary components - the scaling origin $\mathcal O$ and the massive singularity $\MS$ introduced in Section~\ref{SS:MAEINTRO}, see
Figure~\ref{F:CAUSAL}.
They are very different in nature, as the density (and the curvature components) blow up at two distinct rates. 

The main result of Section~\ref{S:RLP} states that there exist 
an outgoing radial null-geodesic (RNG) emanating from the scaling 
origin $\mathcal O$ and reaching infinity. Following
Ori and Piran we look for a so-called simple
RNG:

\begin{definition}[Simple radial null-geodesics]\label{D:SIMPLERNG}
An RNG of the form
\begin{align}\label{E:SIMPLERNGDEF}
R(\tilde\tau) = \sigma \tilde\tau, \ \ \sigma\in\mathbb R\setminus\{0\},
\end{align}
is called a \underline{simple radial null-geodesic} (simple RNG).
\end{definition}

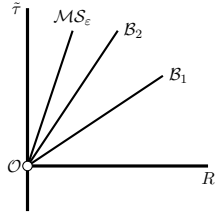
\begin{figure}

\begin{center}
\begin{tikzpicture}
\begin{scope}[scale=0.6, transform shape]

%scaling origin
\coordinate [label=left:$\mathcal O$] (A) at (-3,7){};

% horizontal axis coordinate
\coordinate [label=below:$R$] (B) at (1,7){};

% Massive singularity
\coordinate [label=above:$\MS$] (C) at (-2,10){};

% \mathcal B_2 curve
\coordinate [label=right:$\mathcal B_2$] (D) at (-1,10){};

% \mathcal B_1 curve
\coordinate [label=right:$\mathcal B_1$] (E) at (0,9){};

% backward null-curve
%\coordinate [label=right:$\mathcal N$] (F) at (-1.3,5.5){};

% vertical axis coordinate
\coordinate [label=left:$\ttau$] (G) at (-3,10.5){};

%\tkzCircumCenter(D,L,M)\tkzGetPoint{O}

\draw[very thick] (A)--(1,7);

% lower vertical axis
\draw[very thick] (A)--(-3,6);

% Vertical axis
\draw[very thick] (A)--(-3,10.5);

% Connection to massive singularity
\draw[thick] (A)--(-2,10);

% \mathcal B_2 curve
\draw[thick] (A)--(-1,10);

% \mathcal B_1 curve
\draw[thick] (A)--(0,9);

% backward null-curve
%\draw[thick] (A)--(-0.2,4.2);

% scaling origin
\draw[fill=white] (-3,7) circle (3pt);
%\draw [fill=gray!20, very thick] (D) -- (A) -- (B) -- (M);
%\tkzDrawArc[fill=white, dashed, very thick](O,D)(M)
%\coordinate [label=below: $\Xi$] (B) at (-1,9.2) {};
%\node at (-1,12) {$\Gamma$};
\end{scope}
\end{tikzpicture}
 \caption{Schematic depiction of the outgoing null-geodesics $\mathcal B_1$, $\mathcal B_2$, and the massive singularity $\MS$.
% and the boundary $\mathcal N$ of the backward light cone ``emanating" from the scaling origin $\mathcal O$.
}
    \label{F:CAUSAL}
\end{center}
\end{figure}

Then the key result we prove is the following theorem.

\begin{theorem}[Existence of global outgoing simple RNG-s]\label{T:NS}
There exists an $0<\k_0\ll1$ sufficiently small so that for any $\k\in(0,\k_0]$ there exist at least two and at 
most finitely many outgoing simple RNG-s emanating out of the singularity $(0,0)$. In other words, there exist 
\begin{align}\label{E:SRNGS}
\yms < Y_n <\dots Y_1 <0, \ \ n\ge 2,
\end{align}
so that the associated simple RNG-s are given by 
\begin{align}
\mathcal B_i : = \{(\tt, R)\in\MRLP\,\big| \, \frac{-\sqrt\k \tt}{R} = Y_{i}\}, \ \ i=1,\dots n.
\end{align}
%$\mathcal B_i$, $i=1,2,\dots n$, correspond to curves $\{\frac{R}{\tilde \tau} = \frac{\sqrt\k}{|Y^i_\k|}\}$, $i=1,2$.
\end{theorem}

The proof relies on a beautiful idea of Ori and Piran~\cite{OP1990}, which we make rigorous. Namely, one can show that the slopes
of outgoing simple RNG-s must correspond to roots of a certain real-analytic function, see Lemma~\ref{L:RNGS}. 
Using the sharp asymptotic behaviour of the metric coefficients, brought about through our analysis in Sections~\ref{S:FARFIELD}--\ref{S:MAE}, 
we can prove that this function converges to negative values at $Y=0$ and $Y=\yms$. On the other hand, by the 
local existence theory for the  ODE-system~\eqref{E:LDEQN2}--\eqref{E:CHIEQN}, we can also ascertain the function in question peaks above $0$
for a $Y_0\in(\yms,0)$. Therefore, by the intermediate value theorem, we conclude the proof of Theorem~\ref{T:NS}, see Section~\ref{S:RLP}, immediately after 
Proposition~\ref{P:UNIF}.

%\begin{remark}\label{R:NS}
Informally, the null-hypersurface $\mathcal B_1$ is  the ``first" outgoing null-curve emanating from the singular scaling origin
and reaching the infinity. It is easy to see that $r$ grows to $+\infty$ along $\mathcal B_1$.
Since the spacetime is not asymptotically  flat, we perform a suitable truncation in Section~\ref{S:DOUBLENULL} in order to interpret
$\mathcal O$ as a naked singularity.

The maximal extension we construct is unique only if we insist on it being self-similar, otherwise there could exist other extensions in the 
causal future of $\mathcal O$. Nevertheless, in our analysis the role of the massive singularity $\MS$ is important, as we use the 
sharp blow-up asymptotics of our unknowns to run the intermediate value theorem-argument above. Conceptually, $\MS$ has a natural Newtonian
limit as $\k\to0$, see Remark~\ref{R:MSLIMIT}, which makes it a useful object for our analysis.  
% The part of the spacetime ``below" $\mathcal B_1$ is characterised by the condition 
%$G_+(Y)>0$.
%\end{remark}

In the remainder of Section~\ref{S:RLP} we give a detailed account of radial null-geodesics, showing in particular that there is a
unique ingoing null-geodesic $\mathcal N$ emanating from the scaling origin to the past, see Figure~\ref{F:INGOING}. This is the boundary 
of the backward light cone ``emanating" from the scaling origin. Following the terminology in~\cite{Ch1994}, it splits the spacetime into the {\em exterior} region (in the future of $\mathcal N$) and the {\em interior} region 
(in the past of $\mathcal N$), see Definition~\ref{D:EXTERIOR} and Figure~\ref{F:EXTERIOR}. Moreover, the sonic line  is contained strictly in the interior region.
The complete analysis of nonradial null-geodesics is given in Appendix~\ref{A:NNG}.

%%%%%%%%%%%%%%%%%%%%%%%%
%%%%%%%%%%%%%%%%%%%%%%%%

\begin{figure}
\begin{center}
\begin{tikzpicture}
%[domain=0:5, scale = 0.6]
\begin{scope}[scale=0.6, transform shape]

%scaling origin
\coordinate [label=left:$\mathcal O$] (A) at (-3,7){};

% horizontal axis coordinate
\coordinate [label=below:$R$] (B) at (1,7){};

% upper vertical axis coordinate
\coordinate [label=left:$\ttau$] (G) at (-3,8){};

% backward null-curve
\coordinate [label=right:$\mathcal N$] (F) at (-1.3,5.5){};

% sonic line
\coordinate [label=below:sonic line] (H) at (-2,4.1){};

%\tkzCircumCenter(D,L,M)\tkzGetPoint{O}

% horizontal axis
\draw[very thick] (A)--(1,7);

% lower vertical axis
\draw[very thick] (A)--(-3,4);

% upper vertical axis
\draw[very thick] (A)--(-3,8);

% backward null-curve
\draw[thick] (A)--(-0.2,4.2);

% sonic line
\draw[thick] (A)--(-2.2,4.1);

% scaling origin
\draw[fill=white] (-3,7) circle (3pt);

\end{scope}
 \end{tikzpicture}
   \end{center}
    \caption{Schematic depiction of the ingoing null-curve $\mathcal N$ and the sonic line. The spacetime is smooth across both
 surfaces.
}
\label{F:INGOING}
\end{figure}
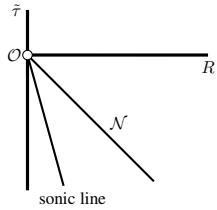

%%%%%%%%%%%%%%%%%%%%%%%%
%%%%%%%%%%%%%%%%%%%%%%%%

%\end{remark}

%%%%%%%%%%%%%%%%%%%%%%%%%%%%%%
%%%%%%%%%%%%%%%%%%%%%%%%%%%%%%

%%%%%%%%%%%%%%%%%%%%%%%%%%%%%%
%%%%%%%%%%%%%%%%%%%%%%%%%%%%%%

%%%%%%%%%%%%%%%%%%%%%%%%%%%%
%%%%%%%%%%%%%%%%%%%%%%%%%%%%

\subsection{Double null gauge, asymptotic flattening, and naked singularities (Section~\ref{S:DOUBLENULL})}
\label{SS:DOUBLENULLINTRO}

%%%%%%%%%%%%%%%%%%%%%%%%%%%%
%%%%%%%%%%%%%%%%%%%%%%%%%%%%

The final step in the proof of Theorem~\ref{T:MAINTHEOREM}
is to truncate the profile away from the scaling origin $\mathcal O$ and
glue it to the already constructed self-similar solution.
To do that we set up this problem %question 
in the double-null gauge:
\begin{align}\label{E:DNG}
g = - \Omega^2 \, d\u\,d\v + r^2 \gamma,
\end{align}
where $\u=\text{const.}$ corresponds to outgoing null-surfaces and 
$\v=\text{const.}$ to the ingoing null-surfaces. 
A similar procedure for the scalar field model was implemented in~\cite{Ch1994},
however due to the complications associated with the Euler evolution, a 
mere cut-off argument connecting the ``inner region" to pure vacuum as we 
approach null-infinity is hard to do. Instead, we carefully design function
spaces that capture the asymptotic decay of the fluid density toward 
future null-infinity in a way that is both consistent with the asymptotic flatness, and can 
be propagated dynamically.

%%%%%%%%%%%%%%%%%%%
%%%%%%%%%%%%%%%%%%%

\subsubsection{Formulation of the problem in double-null gauge}\label{SS:DNFORM}

%%%%%%%%%%%%%%%%%%%
%%%%%%%%%%%%%%%%%%%

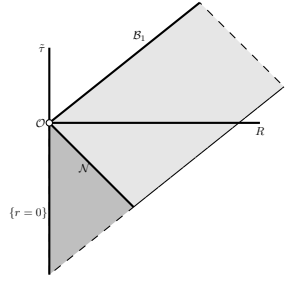
\begin{figure}
\begin{center}
\begin{tikzpicture}
%[domain=0:5, scale = 0.6]
\begin{scope}[scale=0.4, transform shape]

%scaling origin
\coordinate [label=left:$\mathcal O$] (A) at (-3,7){};

% horizontal axis coordinate
\coordinate [label=below:$R$] (B) at (4,7){};

% upper vertical axis coordinate
\coordinate [label=left:$\ttau$] (G) at (-3,9.5){};

% denote \mathcal B_1 curve
\coordinate [label=above:$\mathcal B_1$] (E) at (0,9.6){};

% denote backward null-curve
\coordinate [label=left:$\mathcal N$] (F) at (-1.5,5.5){};

% Tip of the backward null curve
\coordinate (I) at (-0.2,4.2);

% Tip of \mathcal B_1 curve
\coordinate (J) at (2,11);

% The missing vertex of the rhomboid
\coordinate (K) at (4.8,8.2);

% Bottom tip of the interior region
\coordinate (L) at (-3,1.95);

\node at (-3.7,4) {$\{r=0\}$};

%\tkzCircumCenter(D,L,M)\tkzGetPoint{O}

% Draw the exterior region
\draw [fill=gray!20, dashed] (A) -- (I) -- (K) -- (J);

% Draw the interior region
\draw [fill=gray!50, dashed] (A) -- (L) -- (I);

% backward null-curve
\coordinate [label=left:$\mathcal N$] (F) at (-1.5,5.5){};

% Draw the thick connection between I and K
\draw[dashed] (I)--(K);

% horizontal axis
\draw[thick] (A)--(4,7);

% lower vertical axis
\draw[thick] (A)--(L);

% upper vertical axis
\draw[thick] (A)--(-3,9.5);

% draw \mathcal B_1 curve
\draw[thick] (A)--(J);

% backward null-curve
\draw[thick] (A)--(-0.2,4.2);

% scaling origin
\draw[fill=white] (-3,7) circle (3pt);

\end{scope}
 \end{tikzpicture}
   \end{center}
    \caption{Schematic depiction of the interior (dark grey) and the exterior (light grey) region, see Definition~\ref{D:EXTERIOR}.}
\label{F:EXTERIOR}
\end{figure}

In addition to the unknowns 
associated with the fluid, %gas, 
the metric unknowns are the conformal factor $\Om=\Om(\u,\v)$
and the areal radius $r=r(\u,\v)$. 
%The unknowns are the conformal factor $\Omega=\Om(\u,\v)$ and the area radius $r(\u,\v)$.
Clearly,
\begin{align}\label{E:METRICNULL}
g_{\u\u}=g_{\v\v}=0, \ \ g_{\u\v} = -\frac12\Omega^2, \ \ g^{\u\v} = - 2\Om^{-2}.
\end{align}
%\begin{remark}[Notational conventions]
%We never sum over the repeated indices $\u$ and $\v$ and we always sum over the repeated ``upstairs" and ``downstairs" indices
%$A,B\in\{2,3\}$.
%\end{remark}
Let 
%\begin{align}
$u = (u^{\u}, u^{\v},0,0)$
%\end{align}
be the components of the velocity 4-vector $u$ in the frame
%\begin{align}
$\{\pau, \  \pav, \ E_{2}, \  E_3\}$, 
%\end{align}
where 
%$\{L,\bar L\}$ are the generators of the null-cones, and 
$\{E_A\}$, $A=2,3$, describe the generators of some local coordinates on $\mathbb S^2$.
We have $g(\pau,E_A) = g(\pav, E_A) = 0$, $A=2,3$.
%\begin{align}
%g(L,L)=g(\bar L,\bar L) = 0, \ \ g(L,\bar L) =  - 2\Om^{-2}, 
%\ \ g(L,E_A) = g(\bar L, E_A) = 0, \ \ A=1,2. 
%\end{align}
The normalisation condition~\eqref{E:NORMALISATION} in the double-null gauge reads
\begin{align}
-1 = g^{\mu\nu}u_{\mu}u_{\nu} = g^{\u\v} u_\u u_\v = -4\Om^{-2} u_\u u_\v,
\end{align}
which therefore equivalently reads
\begin{align}\label{E:NORMALISATIONNULL}
u_\u u_\v = \frac14 \Om^2, \ \ u^\u u^\v = \Om^{-2}.
\end{align}

%%%%%%%%%%%%%%%%%%%%%%%%%
%%%%%%%%%%%%%%%%%%%%%%%%%

%\begin{remark}[Notational conventions]
%We never sum over the repeated indices $\u$ and $\v$ and we always sum over the repeated ``upstairs" and ``downstairs" indices
%$A,B\in\{2,3\}$.
%\end{remark}

%%%%%%%%%%%%%%%%%%%%%%%%%
%%%%%%%%%%%%%%%%%%%%%%%%%

%The components of the energy momentum tensor read
%\begin{align}
%& T^{\u\u} = (1+\k)\rho (u^\u)^2, \ \ T^{\v\v} = (1+\k)\rho (u^\v)^2,  T^{\u\v} = %(1+\k) \rho u^\u u^\v - 2 \k \Om^{-2} \rho =  
%(1-\k)\rho \Om^{-2},  \label{E:TPP}\\ 
%%& T^{\u\v} = %(1+\k) \rho u^\u u^\v - 2 \k \Om^{-2} \rho =  
%%(1-\k)\rho \Om^{-2}  \label{E:TUV}\\
%& T^{AB} %= (1+\k)\rho \underbrace{u^Au^B}_{=0} + \k \rho g^{AB} 
%= \k \rho r^{-2} \gamma^{AB},   \ \
% T^{\u A} = T^{\v A} = 0, \ \ A,B = 2,3. \label{E:TAB} 
% %\label{E:TAB2}
%\end{align}

\begin{lemma}[Einstein-Euler system in double-null gauge]\label{L:DOUBLENULL}
In the double-null gauge~\eqref{E:DNG}, the 
Einstein field equations take the form
\begin{align}
\pau \pav r & = - \frac{\Om^2}{4r} - \frac{1}{r} \pau r \pav r  + \pi r\Om^4 T^{\u \v}, \label{E:RWAVE1} \\
\pau \pav \log \Om & = - (1+\e)\pi\Om^4 T^{\u\v} +\frac{\Om^2}{4r^2} + \frac1{r^2} \pau r \pav r, \label{E:OMEGAWAVE1} \\
\pav \left(\Om^{-2}\pav r\right) & = - \pi r \Om^2 T^{\u\u}, \label{E:CONSTRAINTV1}\\
\pau\left(\Om^{-2}\pau r\right) & = - \pi r \Om^2 T^{\v\v}. \label{E:CONSTRAINTU1}
\end{align}
Moreover, the components of the energy-momentum tensor satisfy
%The system~\eqref{E:RWAVE1}--\eqref{E:CONSTRAINTU1} is coupled to the fluid evolution equations
\begin{align}
\pa_\u (\Om^4r^2T^{\u\u}) + \frac{\Om^2}{r^{2\e}} \pa_\v ( \Om^2r^{2+2\e} T^{\u\v}  ) &=0, 
 \label{E:CONTNULL13}\\
\pa_\u ( \Om^2r^{2+2\e} T^{\u\v}  )  + \frac{r^{2\e}}{\Om^2}\pa_\v (\Om^4r^2T^{\v\v}) & = 0,
\label{E:MOMENTUMNULL13}
\end{align}
%
%\begin{align}
%\left(\pa_\u + \pau \log\left(\Om^4r^2\right)\right)T^{\u\u} + \left(\pa_\v + \pav \log\left(\Om^2r^{2+2\e}\right)\right)T^{\u\v}   & = 0,   \label{E:CONTNULL1}\\
%\left(\pau + \pau \log\left(\Om^2r^{2+2\e}\right)\right)T^{\u\v} + \left(\pa_\v + \pav \log\left(\Om^4r^2\right)\right)T^{\v\v} & = 0,
%\label{E:MOMENTUMNULL1}
%\end{align}
where the energy-momentum tensor is given by the formulas
\begin{align}
& T^{\u\u} = (1+\k)\rho (u^\u)^2, \ \ T^{\v\v} = (1+\k)\rho (u^\v)^2, \ \ T^{\u\v} = %(1+\k) \rho u^\u u^\v - 2 \k \Om^{-2} \rho =  
(1-\k)\rho \Om^{-2},  \label{E:TPP}\\ 
%& T^{\u\v} = %(1+\k) \rho u^\u u^\v - 2 \k \Om^{-2} \rho =  
%(1-\k)\rho \Om^{-2}  \label{E:TUV}\\
& T^{AB} %= (1+\k)\rho \underbrace{u^Au^B}_{=0} + \k \rho g^{AB} 
= \k \rho r^{-2} \gamma^{AB},   \ \
 T^{\u A} = T^{\v A} = 0, \ \ A,B = 2,3. \label{E:TAB} 
 %\label{E:TAB2}
\end{align}
Moreover, its components are related through the algebraic relation
\begin{align}
T^{\u\u} T^{\v\v}
= (1+\e)^2(T^{\u\v})^2 . \label{E:TCONSTRAINT}
\end{align}
\end{lemma}

\begin{proof}
The proof is a straightforward calculation and is given in Appendix~\ref{A:DOUBLENULL}.
\end{proof}

%%%%%%%%%%%%%%%%%%%%%
%%%%%%%%%%%%%%%%%%%%%

\subsubsection{The characteristic Cauchy problem and asymptotic flattening}

%%%%%%%%%%%%%%%%%%%%%
%%%%%%%%%%%%%%%%%%%

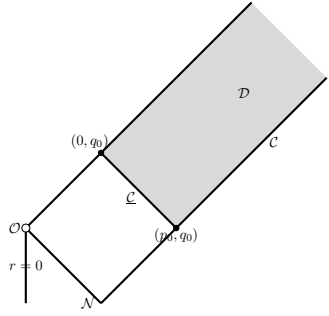
\begin{figure}

\begin{center}
\begin{tikzpicture}
%[domain=0:5, scale = 0.6]
\begin{scope}[scale=0.5, transform shape]

%scaling origin
\coordinate [label=left:$\mathcal O$] (A) at (-2,-2){};

% Tip of the backward null curve
\coordinate (I) at (0,-4);

% Tip of \mathcal B_1 curve
\coordinate (J) at (4,4);

% The missing vertex of the rhomboid
%\coordinate (K) at (4.8,8.2);

% Bottom tip of the interior region
\coordinate (L) at (-2,-4);

% p=0, q=q_0
\coordinate (P) at (0,0);

% p = p_0, q=q_0
\coordinate (Q) at (2,-2){};

\coordinate (S) at (6,2){};

%label p_0, q_0
\node at (2,-2.2) {$(\u_0,\v_0)$};

% Label for the lower vertical axis
%\coordinate [label=left:] (M) at (-3,3.5){};

\node at (-2,-3) {$r=0$};

% label \uC
\node at (0.8,-1.2) {$\uC$};

% label \C
\node at (4.6,0.3) {$\C$};

% label p=0, q=q_0
\node at (-0.3,0.3) {$(0,q_0)$};

% colour region \D gray
\draw [fill=gray!30, dashed] (J) -- (P) -- (Q) -- (S);

% denote backward null-curve
\coordinate [label=left:$\mathcal N$] (F) at (0,-4){};

% lower vertical axis
\draw[thick] (A)--(L);

% draw \mathcal B_1 curve
\draw[thick] (A)--(J);

% backward null-curve
\draw[thick] (A)--(I);

% bdry \uC
\draw[thick] (P)--(Q);

% scaling origin
\draw[fill=white] (-2,-2) circle (3pt);

% p_0 q_0
\draw[fill=black] (2,-2) circle (2pt);

% p=0, q=q_0 
\draw[fill=black] (0,0) circle (2pt);

% connect tip of N to the future
\draw[thick] (I)--(S);

% label region \D
\node at (3.8,1.6) {$\D$};

\end{scope}
 \end{tikzpicture}
   \end{center}
    \caption{The grey shaded area is the region $\D$, where the truncation of the self-similar profile takes place. Data are prescribed on the two characteristic surfaces $\C$ and $\uC$.}    
     \label{F:CUTOFF}
\end{figure}

The idea is to choose a point $(\u_0,\v_0)$ in the exterior region (recall Figure~\ref{F:EXTERIOR}) and solve the Einstein-Euler system
in an infinite semi-rectangular domain (in the $(\u,\v)$-plane) $\D$ depicted in Figure~\ref{F:CUTOFF}. We normalise the choice
of the double-null coordinates by making the outgoing curve $\mathcal B_1$ in the RLP-spacetime (see Figures~\ref{F:INGOING}--\ref{F:EXTERIOR})
correspond to the $\{\u=0\}$ level set, and the ingoing curve $\mathcal N$ (see Figure~\ref{F:EXTERIOR}) to the $\{\v=0\}$ level set.
We have the freedom to prescribe the data along 
the ingoing boundary $\uC = \{(\u,\v)\,\big| \, \u\in[\u_0,0], \v=\v_0\}$ and the outgoing boundary $\C = \{(\u,\v)\,\big| \, \u=\u_0, \v\ge\v_0\}$. On $\uC$ we demand that data be given by the restriction of
the self-similar RLP solution to $\uC$, and on the outgoing piece we make the data exactly self-similar on a subinterval $\v\in[\v_0,\v_0+A_0]$ for some $A_0>0$. On the remaining part of the outgoing boundary $\v\in[\v_0+A_0,\infty)$,
we prescribe asymptotically flat data.

The key result of Section~\ref{S:DOUBLENULL} is Theorem~\ref{thm:LWP}, which states that the above described PDE
is well-posed on $\D$, if we choose $|p_0|=\delta$ sufficiently small. We are not aware of such a well-posedness result for the
%{\em isothermal} 
Einstein-Euler system in the double-null gauge in the literature, and we therefore carefully develop the necessary theory.
Precise statements are provided in Section~\ref{SS:LWPSTATEMENT}.
The idea is standard and relies on the method of characteristics. However, to make it work we rely on an effective ``diagonalisation"
of the Euler equations~\eqref{E:CONTNULL13}--\eqref{E:MOMENTUMNULL13}, which replaces these equations by two transport
equations for the new unknowns $f^+$ and $f^-$, see Lemma~\ref{L:EULERREFORM}. This change of variables 
highlights the role of the acoustic cone and allows us to track the acoustic domain of dependence by following the characteristics, see Figure~\ref{F:SOUNDVERSUSLIGHT}.
The analysis of the fluid characteristics is presented in Section~\ref{SS:CHARACTERISTICS}. Various a priori bounds are given in Section~\ref{SS:APRIORIPDE}. 
In Section~\ref{SS:LWPPDE} we finally introduce an iteration procedure and prove Theorem~\ref{thm:LWP}.

\begin{figure}

\begin{center}
\begin{tikzpicture}
%[domain=0:5, scale = 0.6]
\begin{scope}[scale=0.6, transform shape]

%scaling origin
\coordinate [label=left:$\mathcal O$] (A) at (-2,-2){};

% Tip of the backward null curve
\coordinate (I) at (0,-4);

% Tip of \mathcal B_1 curve
\coordinate (J) at (4,4);

% The missing vertex of the rhomboid
%\coordinate (K) at (4.8,8.2);

% Bottom tip of the interior region
\coordinate (L) at (-2,-4);

% p=0, q=q_0
\coordinate (P) at (0,0);

% p = p_0, q=q_0
\coordinate (Q) at (2,-2){};

\coordinate (S) at (6,2){};

% p=0, q=q_0+A
\coordinate (T) at (1,1){};

% p=p_0, q=q_0+A
\coordinate (U) at (3,-1){};

%label p_0, q_0
\node at (2,-2.2) {$(\u_0,\v_0)$};

%label p_0, q_0+A
\node at (4,-1.1) {$(\u_0,\v_0+A_0)$};

% Label for the lower vertical axis
%\coordinate [label=left:] (M) at (-3,3.5){};

\node at (-2,-3) {$r=0$};

% label \uC
\node at (0.8,-1.2) {$\uC$};

% label \C
\node at (4.6,0.3) {$\C$};

% label p=0, q=q_0
\node at (-0.3,0.3) {$(0,\v_0)$};

% colour region \D gray
\draw [fill=gray!30, dashed] (J) -- (P) -- (Q) -- (S);

% colour region \D_A light gray
\draw [fill=gray!10, dashed] (T) -- (P) -- (Q) -- (U);

% denote backward null-curve
\coordinate [label=left:$\mathcal N$] (F) at (-0.5,-3.5){};

% lower vertical axis
\draw[thick] (A)--(L);

% draw \mathcal B_1 curve
\draw[thick] (A)--(J);

% backward null-curve
\draw[thick] (A)--(I);

% bdry \uC
\draw[thick] (P)--(Q);

% scaling origin
\draw[fill=white] (-2,-2) circle (3pt);

% p_0 q_0
\draw[fill=black] (2,-2) circle (2pt);

% p=0, q=q_0 
\draw[fill=black] (0,0) circle (2pt);

% p=0, q=q_0 
\draw[fill=black] (3,-1) circle (2pt);

% connect tip of N to the future
\draw[thick] (I)--(S);

% connect (p_0,q_0+A) to (0,q_0+A) with a dashed line
\draw[dashed] (U)--(T);

% label region \D
\node at (3.5,1.3) {$\D\setminus \D_{A_0}$};

% label region \D_A
\node at (1.8,-0.6) {$\D_{A_0}$};

\end{scope}
 \end{tikzpicture}
   \end{center}
    \caption{The light grey shaded area is the region $\D_{A_0}$, where the truncated solution from Theorem~\ref{thm:LWP} coincides with the exact self-similar solution.}  
\label{F:CUTOFF2}
\end{figure}
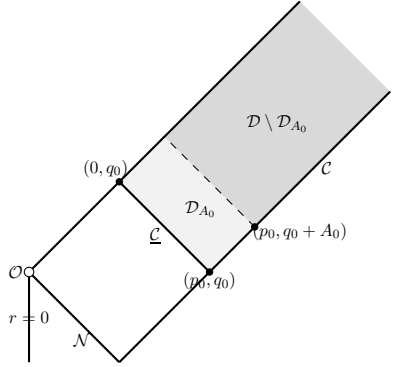

\subsubsection{Formation of naked singularities: proof of Theorem~\ref{T:MAINTHEOREM}}\label{SS:NSINTRO}

Since the data on $\uC$ and the portion of $\C$ with $\v_0\le \v \le \v_0+A_0$
agrees with the RLP-solution, the solution must, by the finite-speed of propagation, 
coincide with the exact RLP solution in the region $\D_{A_0}$ depicted in Figure~\ref{F:CUTOFF2}.
By uniqueness the solution extends smoothly (in fact analytically) to the exact RLP-solution
in the past of $\uC$, all the way to the regular centre $\{r=R=0\}$.

The future boundary of the maximal development is precisely the surface $\{\u=0\}$, the Cauchy horizon for the new spacetime.
The boundary of the null-cone corresponding to $\{\u=\u_0\}$ is complete, and we show in Section~\ref{SS:NAKED}
that for any sequence of points $(\u_0,\v_n)$ with $\v_n$ approaching infinity, the affine length of maximal future-oriented ingoing geodesics launched 
from $(\u_0,\v_n)$ and normalised so that the tangent vector corresponds to $\pau$, is bounded by a constant, uniformly-in-$n$. This shows that the
future null-infinity is incomplete in the sense of~\cite[Definition 1.1]{RoSR2019}, thus completing the proof of Theorem~\ref{T:MAINTHEOREM}, see the Penrose diagram, 
Figure~\ref{F:PENROSE}.

%%%%%%%%%%%%%%%%%%%%
%%%%%%%%%%%%%%%%%%%%

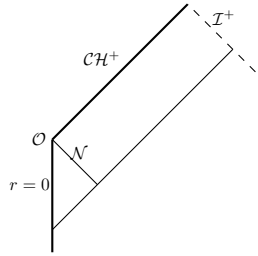
\begin{figure}

\begin{center}
\begin{tikzpicture}
%[domain=0:5, scale = 0.6]
\begin{scope}[scale=0.6, transform shape]

%bottom left point 
\coordinate [] (A) at (0,0){};

\coordinate [] (A') at (0,-0.5);

% scaling origin
\coordinate[label=left:$\mathcal O$] (O) at (0,2);

% bottom right corrner of null infinity
\coordinate (B) at (4,4);

% bottom right corner of null infinity 2
\coordinate (B') at (4.5,3.5);

% upper corner of null-infinity
\coordinate (C) at (3,5);

% lower end of N
\coordinate (E) at (1,1);

\draw [thick] (O) -- (A');

\draw (A) -- (B);

\draw[dashed] (B') -- (C);

\draw [thick] (O) -- (C);

\draw (O)--(E);

% label $\N$
\node at (0.6,1.7) {$\mathcal N$};

% label I+
\node at (3.8,4.7) {$\mathcal I^+$};

% label Cauchy horizon
\node at (1.1, 3.8) {$\mathcal C\mathcal H^+$};

% \label r=0
\node at (-0.5,1) {$r=0$};

%% colour region \D gray
%\draw [fill=gray!30, dashed] (J) -- (P) -- (Q) -- (S);
%
%% colour region \D_A light gray
%\draw [fill=gray!10, dashed] (T) -- (P) -- (Q) -- (U);

%% scaling origin
%\draw[fill=white] (-2,-2) circle (3pt);
%
%% p_0 q_0
%\draw[fill=black] (2,-2) circle (2pt);
%
%% p=0, q=q_0 
%\draw[fill=black] (0,0) circle (2pt);
%
%% p=0, q=q_0 
%\draw[fill=black] (3,-1) circle (2pt);
%
%
%% connect tip of N to the future
%\draw[thick] (I)--(S);
%
%% connect (p_0,q_0+A) to (0,q_0+A) with a dashed line
%\draw[dashed] (U)--(T);
%
%% label region \D
%\node at (3.5,1.3) {$\D\setminus \D_{A_0}$};
%
%% label region \D_A
%\node at (1.8,-0.6) {$\D_{A_0}$};

\end{scope}
 \end{tikzpicture}
   \end{center}
    \caption{Penrose diagram of our spacetime with an incomplete future null-infinity.}  
\label{F:PENROSE}
\end{figure}

%\bigskip 
%
%{\bf Acknowledgments.}
%Y. Guo's research is supported in part by NSF DMS-grant 2106650.
%M. Hadzic's research is supported by the EPSRC Early Career Fellowship EP/S02218X/1.
%J. Jang's research is supported by the NSF DMS-grant 2009458
%and the Simons Fellowship (grant number 616364). 

%%%%%%%%%%%%%%%%%%%%%%%%%%%%
%%%%%%%%%%%%%%%%%%%%%%%%%%%%
%%%%%%%%%%%%%%%%%%%%%%%%%%%%

\section{Radially symmetric Einstein-Euler system in comoving coordinates}\label{S:FORMULATION}

\subsection{Self-similar comoving coordinates}

%%%%%%%%%%%%%%%%%%%%%%%%%%%%%%%
%%%%%%%%%%%%%%%%%%%%%%%%%%%%%%%

%It is straightforward to check that the system~\eqref{E:RHOEQN1}--\eqref{E:LAMBDACONSTRAINT1} is invariant under the scaling transformation
%\begin{align}
%\rho\mapsto a^{-2}\rho(s,y), \ \ r\mapsto a r(s,y), \ \ \mathcal V \mapsto \mathcal V(s,y), \ \ \lambda\mapsto \lambda(s,y), \ \ \mu\mapsto \mu(s,y), \label{E:T1} 
%\end{align}
%where the comoving ``time" $\tau$ and the particle label $R$ scale according to
%\begin{align}
% s = \frac{\tau}{a}, \ \ y = \frac{R}{a}, \ \ a\neq0. \label{E:T2}
%\end{align}
%Motivated by the scaling invariance~\eqref{E:T1}--\eqref{E:T2}, we look to construct a self-similar spacetime of the form
%%We therefore look for a self-similar solution 
%%using the following change of variables:
%\begin{align}
%\rho(\tau,R) & = \frac{1}{2\pi \tau^2} \Sigma(y) \label{E:SS1}\\
%r(\tau, R) & = -\sqrt \k \tau \tr(y) \label{E:AREARADIUSSS}\\
%\mathcal V(\tau,R) & = \sqrt \k  V(y) \label{E:VELOCITYSSCHANGE}\\
%\lambda(\tau,R) & = \lambda(y) \\
%\mu(\tau,R) & = \mu(y)\\
%G(\tau, R) & = \frac1{4\pi\tau^2} \tG(y) = \frac{3}{2\pi \tau^2\tr^3}\int_0^y \Sigma(\tilde y) \tr^2 \tr' \,d\tilde y, \label{E:SS6}
%\end{align}
%where
%\[
%y= \frac{R}{-\sqrt \k\tau}.
%\]
We plug in~\eqref{E:SS1}--\eqref{E:SS6} into~\eqref{E:RHOEQN1}--\eqref{E:LAMBDACONSTRAINT1} and obtain the self-similar formulation of the %isothermal 
Einstein-Euler system:
\begin{align}
2\Sigma + y\Sigma'(y) +(1+\k)\Sigma\left(\frac{V'}{\tr'}+\frac{2V}{\tr}\right) e^{\mu} & = 0, \label{E:RHOEQNSS}\\
y \l'(y) = e^{\mu}\frac{V'}{\tr'}, \label{E:LAMBDAEQNSS}\\
e^{-\mu} y V'(y) + \frac{1}{1+\k} \frac{\tr'(y) e^{-2\l(y)}}{\Sigma(y)}\Sigma'(y) + \tr\left(\frac13 \tG + 2\k \Sigma\right) & = 0 ,\label{E:VEQNSS}\\
\tr'^2 e^{-2\l} & = 1+\k V^2 - \frac{2}{3} \k \tG \tr^2. \label{E:LAMBDACONSTRAINTSS}
\end{align}

\begin{remark}
Recall that the radial velocity $\mathcal V$ satisfies $\mathcal V=e^{-\mu}\partial_\tau r$ and therefore
\begin{align}
V = e^{-\mu} (-\rl + y\rl'(y)). \label{E:RADIALVELSS}
\end{align}
From~\eqref{E:RHOEQNSS},~\eqref{E:LAMBDAEQNSS}, %\eqref{E:VEQNSS}, 
and~\eqref{E:RADIALVELSS} we obtain %conclude
\begin{align}
0 & = 2 + y\frac{\Sigma'(y)}{\Sigma(y)} +(1+\k)\left(y \l'(y)+2 \left(y \frac{\tr'(y)}{\tr(y)}-1\right)\right) \notag 
%\\& = -2\k + y \left(\frac{D'(y)}{D(y)} + (1+\k) \left(\l'(y)+2\frac{\tr'(y)}{\tr(y)} \right)\right) \notag 
\end{align}
and therefore
\begin{align}
\l'(y) = -\frac1{1+\k}\left(\frac{\Sigma'}{\Sigma}+\frac{2}{y}\right) -2 \left(\frac{\tr'}{\tr}-\frac1y\right). \label{E:LAMBDASSNEW}
\end{align}
For smooth solutions with strictly positive density $\Sigma$, an explicit integration of the identity~\eqref{E:LAMBDASSNEW} gives the formula
\begin{align}\label{E:LAMBDAFORMULA}
e^{\lambda} = \alpha \Sigma^{-\frac1{1+\k}} y^{\frac{2\k}{1+\k}}\tr^{-2},
\end{align}
for some constant $\alpha>0$.
\end{remark}
%Note that~\eqref{E:LAMBDASSNEW} corresponds to (2.20d) in Ori-Piran's work.

\begin{remark}[Gauge normalisation]
From~\eqref{E:MUFORMULA0} we also obtain
\begin{align}\label{E:MUFORMULASS}
\mu'(y) = - \frac{\k}{1+\k} \frac{\Sigma'(y)}{\Sigma(y)}.
\end{align}
%Note that~\eqref{E:MUFORMULASS} corresponds to (2.20c) in Ori-Piran's work.
%{\em Gauge normalisation.}
For any given $0<\k<1$ we finally fix the remaining freedom in the problem by
setting 
\begin{align}\label{E:GAUGENORM}
c = e^{2\mu(0)}\Sigma(0)^\frac{2\k}{\k+1} = \frac1{(1+\k)^2}.
\end{align} 
Upon integrating~\eqref{E:MUFORMULASS} we obtain the identity
\begin{align}\label{E:MUFORMULASS2}
e^{\mu(y)} = \frac1{1+\k} \ \Sigma(y)^{-\frac{\k}{1+\k}}.
\end{align}

\end{remark}

%\subsection{Number density formulation}

%We can alternatively rephrase the problem in terms of the selfsimilar number density $N$ defined by
%\begin{align}
%N = D^{\frac1{1+\k}}.
%\end{align}
%
%Then the self-similar Euler-Einstein~\eqref{E:RHOEQNSS}--\eqref{E:LAMBDACONSTRAINTSS} takes the form
%\begin{align}
%- \frac{2\k}{(1+\k)y} + \frac{N'}{N} + \l' + 2\frac{\tr'}{\tr} & = 0 \label{E:NEQNSS}\\
% y \l'(y) = e^{\mu}\frac{V'}{\tr'} \\
%e^{-\mu} y V'(y) +  \frac{\tr'(y) e^{-2\l(y)}}{N(y)}N'(y) + \tr\left(\frac13 \tG + 2\k N^{1+\k}\right) & = 0 \label{E:VEQNSS2}\\
%\tr'^2 e^{-2\l} & = 1+\k V^2 - \frac{2}{3} \k \tG \tr^2, \label{E:LAMBDACONSTRAINTSS2}
%\end{align}
%where 
%\[
%\tG(y) = \frac{24\pi}{\tr^3}\int_0^R N^{1+\k}\tr^2 \tr'\,dy.
%\]
%Equation~\eqref{E:NEQNSS} can be integrated and gives
%\begin{align}\label{E:NEQNSS2}
%N(y)e^{\l(y)}\tr^2(y) = \alpha y^{\frac{2\k}{1+\k}}
%\end{align}
%for some $\alpha>0$.

%\begin{remark}
%In the Newtonian limit $\k\to0$ we see from~\eqref{E:LAMBDACONSTRAINTSS2} that $e^\l \to \tr'$ and we therefore recover the Newtonian relation~\eqref{E:DFORMULA}.
%\end{remark}

%We introduce the fundamental unknowns in the comoving self-similar coordinates:
%\begin{align}
%\d &: = \Sigma^{\frac{1-\k}{1+\k}},  \label{E:LITTLEBOLDDDEF}\\
%\w &: = (1+\k) \frac{e^\mu V + \tilde r }{\tilde r}  - \k . \label{E:LITTLEWDEF}
%\end{align}

Before we reduce the self-similar formulation~\eqref{E:RHOEQNSS}--\eqref{E:LAMBDACONSTRAINTSS} to a suitable $2\times2$ nonautonomous ODE system, we first prove
an auxiliary lemma.

%%%%%%%%%%%%%%%%%%%%%%%%%%
%%%%%%%%%%%%%%%%%%%%%%%%%%

\begin{lemma}\label{L:NONLOCALISLOCAL}
Let $(\Sigma,V,\l,\mu)$ be a smooth solution to~\eqref{E:RHOEQNSS}--\eqref{E:LAMBDACONSTRAINTSS}.
\begin{enumerate}
\item[{\em (a)}]
The following important relationship holds:
\begin{align}\label{E:WTR}
%w =
\frac{\w+\k}{1+\k} =  \frac{y\tr'}{\tr}, 
\end{align}
with $\w$  defined by~\eqref{E:LITTLEWDEF}. 
\item[{\em (b)}]
(Local expression for the mean density)
The selfsimilar mean density $\tilde G$ defined by
\be\label{E:TILDEGSS}
\tilde G(y) = \frac{6}{\tilde r^3}\int_0^y \Sigma \tilde r^2 \tilde r' d\tilde y 
\ee
satisfies the relation
\be\label{E:TILDEGLOCAL}
\tilde G = 6 \Sigma \w. 
\ee
\item[{\em (c)}]
The expression $K$ defined by
\begin{align}\label{H1}
K:=-\frac \w{1+\k}+e^{2\mu}\left(\frac16\frac{(1+\k)\tG}{\w+\k}+\frac{\k(1+\k) \Sigma}{\w+\k}\right).
\end{align}
satisfies the relation
\be
K =  \frac1{1+\k}\left(\d - \w\right),
\ee
with $\d$ defined by~\eqref{E:LITTLEBOLDDDEF}.
\item[{\em (d)}]
The metric coefficient $e^{2\l}$ satisfies the formula
 \begin{align}\label{E:LAMBDAFLUID}
e^{2\l} = \frac{\tr^2(\w+\k)^2}{(1+\k)^2 y^2\left[1 + \k \Sigma^{\frac{2\k}{1+\k}}\tr^2(\w-1)^2-4\k \Sigma\w \tr^2\right]}. 
\end{align}
%where
%\begin{align}\label{E:LITTLECDEF}
%c=e^{2\mu(0)}D(0)^\frac{2\k}{\k+1}.
%\end{align}
\end{enumerate}
\end{lemma}

%%%%%%%%%%%%%%%%%%%%%%%%%%%%%%%

\begin{proof}
{\em Proof of part {\em (a)}.}
This is a trivial consequence of ~\eqref{E:RADIALVELSS} and~\eqref{E:LITTLEWDEF}.

\noindent
{\em Proof of part {\em (b)}.}
%Next, we try to find a local expression of $\tilde G$, Newtonian analogue of (1.42) in Mahir's note. 
After multiplying~\eqref{E:RHOEQNSS} by $\tilde r^2 \tilde r'$ we obtain
\begin{align}
0&= 2\Sigma \tilde r^2 \tilde r' + \underbrace{y \tilde r' }_{= \tilde r+ Ve^\mu }\tilde r^2 \Sigma'  + (1+\k) \Sigma (\tilde r^2 V' + 2 \tilde r \tilde r' V) e^\mu \\
&=  2\Sigma \tilde r^2 \tilde r' + \Sigma' \tilde r^3 + e^\mu \Sigma' \tilde r^2 V + (1+\k) e^\mu \Sigma (\tilde r^2 V)' \\
&= 2\Sigma \tilde r^2 \tilde r' + \Sigma' \tilde r^3  + (1+\k) (e^\mu \Sigma \tilde r^2 V) ' 
\end{align}
where in the second line we used $y \tilde r' = \tilde r+ Ve^\mu$ (which follows from~\eqref{E:RADIALVELSS})
and in the last line the identity
\[
(\Sigma e^\mu)' = \Sigma' e^\mu + \Sigma e^\mu ( - \frac{\k}{1+\k} \frac{\Sigma'}{\Sigma}) = \frac{1}{1+\k} \Sigma' e^\mu,
\]
which uses the field equation~\eqref{E:MUFORMULASS}.
After integrating over the interval $[0,y]$ and integrating-by-parts, we obtain 
\be
\int_0^y \Sigma \tilde r^2 \tilde r' ds = \Sigma \tilde r^2 (\tilde r + (1+\k) e^\mu V) = \Sigma \tilde r^3 \w
\ee
where we have used~\eqref{E:LITTLEWDEF}. Hence dividing by $\rl^3$ we obtain~\eqref{E:TILDEGLOCAL}. 

\noindent
{\em Proof of part {\em (c)}.}
From~\eqref{H1} and~\eqref{E:TILDEGLOCAL} we immediately have 
\be
K = (1+\k)e^{2\mu} \Sigma \frac{\w}{1+\k} =\frac1{1+\k} \left(\Sigma^{\frac{1-\k}{1+\k}} - \w \right),
\ee
where we have used~\eqref{E:MUFORMULASS2}.
%\be\label{E:METRICRELATION}
%e^{2\mu}= \frac1{(1+\k)^2}\Sigma^{-\frac{2\k}{1+\k}},
%\ee
%which follows from~\eqref{E:MUFORMULASS2} and the gauge choice~\eqref{E:GAUGENORM}.

\noindent
{\em Proof of part {\em (d)}.} The formula is a simple consequence of~\eqref{E:LAMBDACONSTRAINTSS},~\eqref{E:RADIALVELSS}, and~\eqref{E:TILDEGLOCAL}.
%Here $c=e^{2\mu(0)}D(0)^\frac{2k}{k+1}$. 
\end{proof}

%%%%%%%%%%%%%%%%%%%%%%%%%%%%
%%%%%%%%%%%%%%%%%%%%%%%%%%%%

%From~\eqref{E:RADIALVELSS} we deduce an important relationship
%\begin{align}\label{E:WTR}
%%w =
%\frac{\w+\k}{1+\k} =  \frac{y\tr'}{\tr}.
%\end{align}

\begin{proposition}[The ODE system in comoving self-similar coordinates]\label{P:COMOVINGSS}
Let $(\Sigma,V,\l,\mu)$ be a smooth solution to~\eqref{E:RHOEQNSS}--\eqref{E:LAMBDACONSTRAINTSS}.
Then the pair $(\d,\w)$ solves the system~\eqref{E:DSSEQNBOLD}--\eqref{E:WSSEQNBOLD},
%\begin{align}
%\d' &= -  \frac{ 2(1-\k) \d (\d-\w)}{(1+\k)y(e^{2\mu - 2\lambda} y^{-2} -1)} \label{E:DSSEQNBOLD}\\
%\w' &= \frac{(\w+\k)(1 -3\w)}{(1+\k)y} + \frac{2\w (\d-\w)}{y(e^{2\mu - 2\lambda} y^{-2} -1)},\label{E:WSSEQNBOLD}
%\end{align}
where $\d,\w$ are defined in~\eqref{E:LITTLEBOLDDDEF}--\eqref{E:LITTLEWDEF}.
\end{proposition}

%%%%%%%%%%%%%%%%%%%%%%%%%%%%%%

\begin{proof}
A routine calculation starting from~\eqref{E:RHOEQNSS} and~\eqref{E:VEQNSS} (these can be thought of as the continuity and the momentum equation respectively) gives:
\begin{align}
&\left(y+\k\frac y{\w+\k}(\w-1) \right) \Sigma' + \frac{(1+\k)\Sigma y}{\w+\k}\w' - (1+3\k)\Sigma+ 3 (\w+\k)\Sigma  = 0, \label{E:NEQN2}\\
&e^{-2\mu}y\Sigma \w' + \frac1{1+\k}\left(\k e^{-2\mu} y (\w-1)+\frac{(\w+\k)e^{-2\l}}{y}\right)\Sigma'+e^{-2\mu}\Sigma\frac{(\w+\k)(\w-1)}{1+\k} + (1+\k)\left(\frac13 \Sigma\tG + 2\k  \Sigma^2\right)  = 0.\label{E:WEQN2}
\end{align}
We may rewrite~\eqref{E:NEQN2}--\eqref{E:WEQN2} in the form
\begin{align}
\Sigma' & = \frac{2y(1+\k)\Sigma}{y^2-e^{2\mu-2\l}}\left(-\frac \w{1+\k}+e^{2\mu}\left(\frac16\frac{(1+\k)\tG}{\w+\k}+\frac{\k(1+\k) \Sigma}{\w+\k}\right)\right)  ,\label{E:N1}\\
\frac{\w'}{1+\k} & = \frac{(\w+\k)\left(1-3\w\right)}{(1+\k)^2y} - \frac{2y \w}{y^2 - e^{2\mu-2\l}}
\left(-\frac \w{1+\k}+e^{2\mu}\left(\frac16\frac{(1+\k)\tG}{\w+\k}+\frac{\k(1+\k) \Sigma}{\w+\k}\right)\right). \label{E:w1}
%\Big((2k+2)w^2 - 4k w +\frac{2k^2}{1+k} \notag \\
%& \ \ \ \ - e^{2\mu}\left(\frac13 \tG +2 k N^{1+k}\right)\left(1+ k \frac{w-1}{w}\right)\Big).
\end{align}
With~\eqref{H1} in mind,
%\begin{align}\label{H1}
%K:=-\frac \w{1+\k}+e^{2\mu}\left(\frac16\frac{(1+\k)\tG}{\w+\k}+\frac{\k(1+\k) \Sigma}{\w+\k}\right).
%\end{align}
equations~\eqref{E:N1}--\eqref{E:w1} take the form
%\begin{align}
%N' &= - \frac{2NH}{y(y^{-2}e^{2\mu-2\l}-1)} \\
%w' & = \frac{w\left(\frac{1+3\k}{1+\k} - 3 w\right)}{y} +\frac{2\left((1+\k)w-\k\right)H}{y(1-y^{-2}e^{2\mu-2\l})}
%\end{align}
\begin{align}
\Sigma' &= -  \frac{ 2(1+\k) \Sigma K }{y(e^{2\mu - 2\lambda} y^{-2} -1)} ,\label{E:DSSEQN}\\
\w' &= \frac{(\w+\k)(1 -3\w)}{(1+\k)y} + \frac{2(1+\k)\w K}{y(e^{2\mu - 2\lambda} y^{-2} -1)},\label{E:WSSEQN}
\end{align}
where $K$ is given by~\eqref{H1} and $\tilde G$ by~\eqref{E:TILDEGSS}. Equations~\eqref{E:DSSEQNBOLD}--\eqref{E:WSSEQNBOLD} now
follow from~\eqref{E:LITTLEBOLDDDEF} and parts (b) and (c) of Lemma~\ref{L:NONLOCALISLOCAL}.
\end{proof}

\subsection{Selfsimilar Schwarzschild coordinates}\label{SS:SCHW}

%%%%%%%%%%%%%%%%%%%%%%%%%%%%%%%
%%%%%%%%%%%%%%%%%%%%%%%%%%%%%%%

We recall the self-similar Schwarzschild coordinate introduced in~\eqref{E:SSSDEF}.
%In order to study the existence of the self-similar solution it will 
%be useful to resort to a version of self-similar Schwarzschild coordinates.
%To that end, we introduce
%\[
%x:= \tilde r (y)
%\]
%so that 
%\be\label{E:TILDERPRIME}
%\frac{dx}{dy} = \tilde r ' = \frac{x (\w+\k)}{y(1+\k)}.
%\ee

%%%%%%%%%%%%%%%%%%%%%%%%%%%%
%%%%%%%%%%%%%%%%%%%%%%%%%%%%

\begin{lemma}[Self-similar Schwarzschild formulation]\label{L:SCHW}
Let $(\Sigma,\w)$ be a smooth solution to~\eqref{E:DSSEQN}--\eqref{E:WSSEQN}. Then the variables $(\R,W)$ defined by~\eqref{E:RWDEF}
%\begin{align}\label{E:RWDEF}
%\R(x):=\d(y)^{\frac{1-\k}{1+\k}}, \ \ W(x):=\w (y),
%\end{align}
solve the system~\eqref{E:RODE}--\eqref{E:WODE}, with $B$ given by~\eqref{E:BDEF0} and $\e=\e(\k)=\frac{2\k}{1-\k}$.
\end{lemma}

%%%%%%%%%%%%%%%%%%%%%%%%%%%%

\begin{proof}
With the above notation and Lemma~\ref{L:NONLOCALISLOCAL} we have $K= \frac1{1+\k}\left(\R-W\right)$.
%Let 
%\be
%\tilde \R(x)= D(y), \quad \tilde W(x)= w (y)
%\ee
It is straightforward to see that the system~\eqref{E:DSSEQN}--\eqref{E:WSSEQN} transforms into
\begin{align}
\R' & = - \frac{2(1-\k)x\R(W+\k)(\R-W)}{x^2(W+\k)^2(e^{2\mu - 2\lambda} y^{-2} -1)} ,\label{E:TILDEREQN}\\
W' & = \frac{1-3W}{x} + \frac{2(1+\k)x W(W+\k)(\R-W)}{x^2(W+\k)^2(e^{2\mu - 2\lambda} y^{-2} -1)}. \label{E:TILDEWEQN}
\end{align}
%\begin{align}
%\frac{d\tilde \R}{dx} &= -  \frac{ 2(1+\k) x \tilde W\tilde \R  }{(\tilde W^2 x^2e^{2\mu - 2\lambda} y^{-2} -\tilde W^2x^2)} H\label{E:TILDEREQN}\\
%\frac{d\tilde W}{dx}&= \frac{(\frac{1+3\k}{1+\k} -3\tilde W )}{x} + \frac{2x\tilde W(\tilde W+\k(\tilde W-1))}{(\tilde W^2 x^2e^{2\mu - 2\lambda} y^{-2} - \tilde W^2x^2)}H\label{E:TILDEWEQN}
%\end{align}
%where from Lemma~\ref{L:NONLOCALISLOCAL}
%\be\label{E:TILDEH}
%H = \frac1{1+\k} \tilde R^{\frac{1-\k}{1+\k}} - \tilde W + \frac{\k}{\k+1}.
%\ee
From~\eqref{E:TILDERPRIME}, the constraint equation~\eqref{E:LAMBDACONSTRAINTSS}, and~\eqref{E:LITTLEWDEF}
\begin{align}
\frac{(W+\k)^2}{(1+\k)^2} x^2e^{2\mu - 2\lambda} y^{-2}&= \tilde r'^2 e^{2\mu - 2\lambda} = e^{2\mu}(1+ \k V^2 - \frac{2}{3}\k \tilde G \tilde r^2)\notag\\
&=e^{2\mu} \left( 1 +\k e^{-2\mu}x^2 \frac{(W-1)^2}{(1+\k)^2} - 4 \k \Sigma W x^2 \right),\label{E:SONICKEY1}
%&\rightarrow 1 \quad\text{as}\quad \k \rightarrow 0  \quad\text{(formally)}
\end{align}
where we have slightly abused notation by letting $\Sigma(x)=\Sigma(y)$.
Therefore
\be\label{E:SONICDEN}
\frac{(W+\k)^2}{(1+\k)^2} x^2e^{2\mu - 2\lambda} y^{-2} -\frac{(W+\k)^2}{(1+\k)^2} x^2 = 
\frac{\Sigma^{-\frac{2\k}{1+\k}} }{(1+\k)^2}- \left[ \frac{(W+\k)^2}{(1+\k)^2} - \k \frac{(W-1)^2}{(1+\k)^2} + \frac{4\k}{(1+\k)^2} \Sigma^{\frac{1-\k}{1+\k}}W\right] x^2,
\ee
where we have used~\eqref{E:MUFORMULASS2}. Plugging this back into~\eqref{E:TILDEREQN}--\eqref{E:TILDEREQN}, the claim follows.
\end{proof}

%%%%%%%%%%%%%%%%%%%%%%%%%%%%%%
%%%%%%%%%%%%%%%%%%%%%%%%%%%%%%

\begin{lemma}[Algebraic structure of the sonic denominator $B$]\label{L:BALGEBRA}
Consider the denominator $B[x;\R,W]$ defined in~\eqref{E:BDEF0}. We may factorise $B$ in the form
\begin{align}\label{E:BDEF}
B[x;\R,W] = (1-\k)  (\F[x;\R] -xW(x)) (H[x;\R] + xW(x)),
\end{align}
where 
\begin{align}
\F[x;\R]=\F & : = -\frac{2\k}{1-\k}(1 +\R)x + \sqrt{ \frac{4\k^2}{(1-\k)^2} (1 +\R)^2 x^2+  \k x^2 
+ \frac{\R^{-\e}}{1-\k}} , \label{E:FDEF} \\
H[x;\R]=H& : =  \F[x;\R] +\frac{4\k}{1-\k}(1+\R)x. \label{E:HDEF}
\end{align}
Moreover,
\be\label{FH}
(1-\k)\F H =  \R^{-\e} + \k(1-\k) x^2, 
\ee
and
\begin{align}\label{E:FPRIME}
\F ' = \frac{ -2\k \F  (1+\R) + \k(1-\k)x - (2\k x \F  + \frac{\k}{1-\k} \R^{-\e -1}) \R' }{(1-\k) \F  + 2\k x (1+\R)}. 
\end{align}
\end{lemma}

%%%%%%%%%%%%%%%%%%%%%%%%%%%%%%

\begin{proof}
From~\eqref{E:BDEF0} it is clear that we may view $B[x;\R,W]$ as a quadratic polynomial in $W$:
\[
W^2 + \frac{4\k}{1-\k} (1+\R) W - \k - \frac{ x^{-2} \R^{-\e}}{1-\k}=0. 
\]
 Solving for $W$, we
obtain two roots, which when multiplied by $x$ give
\be
xW_\pm= \frac{-2\k x(1 +\R) \pm \sqrt{ 4\k^2 x^2(1 +\R)^2 + (1-\k) \left( \k(1-\k) + \R^{-\e} \right) } }{1-\k}. 
\ee
We now observe from~\eqref{E:FDEF}--\eqref{E:HDEF} that  the positive solution corresponds to $\F [x;\R]$ and the negative one to $- H[x;\R]$. This in turn
immediately gives~\eqref{E:BDEF}. Property~\eqref{FH} is obvious from~\eqref{E:FDEF}--\eqref{E:HDEF}. Finally, to show~\eqref{E:FPRIME} observe that
$\F$ solves
\[
(1-\k) \F ^2 + 4\k x (1+\R) \F  - \k(1-\k) x^2- \R^{-\e} =0.
\]
We differentiate the above equality, regroup terms and obtain~\eqref{E:FPRIME}. 
\end{proof}

%%%%%%%%%%%%%%%%%%%%%%%%%%%%%%%%
%%%%%%%%%%%%%%%%%%%%%%%%%%%%%%%%
%%%%%%%%%%%%%%%%%%%%%%%%%%%%%%%%

\subsection{The monotonicity lemma}

%%%%%%%%%%%%%%%%%%%%%%%%%%%%%%%%
%%%%%%%%%%%%%%%%%%%%%%%%%%%%%%%%
%%%%%%%%%%%%%%%%%%%%%%%%%%%%%%%%

Controlling the sonic denominator $B$ from below will be one of the central technical challenges in our analysis. 
From Lemma~\ref{L:BALGEBRA} it is clear that this can be accomplished by tracking the quantity $\F [x;\R]-xW(x)$. 
A related quantity, of fundamental importance in our analysis is given by 
\be\label{E:LITTLEFDEF}
f(x) : = \F [x;\R]- x\R.
\ee
The goal of the next lemma is to derive a first order ODE satisfied by $f$, assuming that we have a smooth solution to the system~\eqref{E:RODE}--\eqref{E:WODE}.
This lemma will play a central role in our analysis.

%%%%%%%%%%%%%%%%%%%%%%%%%%%
%%%%%%%%%%%%%%%%%%%%%%%%%%%

\begin{lemma}\label{L:JUHILEMMA}
Let $(\R,W)$ be a smooth solution to the self-similar Einstein-Euler system~\eqref{E:RODE}--\eqref{E:WODE} on some interval $I\subset(0,\infty)$, and let $f$ be given by~\eqref{E:LITTLEFDEF}.
Then, the function $f$ satisfies the ODE
\begin{align}\label{E:LITTLEFDYNAMICS}
f'(x) + a[x;\R,W] f(x) = b[x;\R,W],  \ \ x\in I,
\end{align}
where
\[
 a[x;\R,W]=a_1[x;\R,W] + a_2[x;\R,W], \ \ \   b[x;\R,W]=b_1[x;\R,W] + \k b_2[x;\R,W]
 \]
 and
\begin{align}
a_1[x;\R,W] & =  \left(\frac{2\k x \F  + \frac{\k}{1-\k} \R^{-\e-1}}{(1-\k)\F +2\k x\left(1+\R\right)}+x\right)\frac{2(1-\k)\R\left(W+\k\right)}{B} \notag \\
& \ \ \ \  -2\k \left( \frac1{1-\k}\frac{ \R^{-\e}}{x}
+ (1-\k)x +  (\R+\k)f \right) Z^{-1};  \label{E:AONEDEF} \\
a_2[x;\R,W] & =   2\k\left[\left(\F -xW\right)\left(\R-1\right)  + 2f + 4\k x + x\R\left(5+\k\right) \right] Z^{-1};\label{E:ATWODEF}\\
b_1[x;\R,W]  & = \frac{\R}{H+xW}(xW-\F ) + \k(xW-\F )Z^{-1}\left(x\left(2\R^2-2\R+1-\k\right)+\frac2{(1-\k)x}\R^{-\e}\right); \label{E:BONEDEF}\\
b_2[x,\R,W] & = - 2 x^2\Big\{\R\left[\R^2+(3+\k)\R -(1-\k)\right] \notag \\
& \ \ \ \ +\k \left[\frac4{1-\k}\R^3+\frac{2(1+\k)}{1-\k}\R(1+\R) + \frac{5-\k}{1-\k}\R^2+3\R - 2+\k\right]\Big\}Z^{-1}; \label{E:BTWODEF}\\
Z[x;\R,W] & = \left[(1-\k)\F  + 2\k x \left(1+\R\right)\right]\left[H+xW\right].  \label{E:ZDEF0}
\end{align}
\end{lemma}

%%%%%%%%%%%%%%%%%%%%%%%%%%%

\begin{proof}
Since $f'= \F ' - x\R'-\R$, the goal is to find the desirable form \eqref{E:LITTLEFDYNAMICS} by using the dynamics of $\F $ and $\R$ \eqref{E:FPRIME} and \eqref{E:RODE}. The factorisation of the denominator $B$ in terms of $\F $ and $H$ given in \eqref{E:BDEF} will be importantly used in the derivation. 

Using \eqref{E:BDEF} and $x\R-xW= x\R -\F  +\F -xW$, we first rewrite $\R'$ as 
\[
\R' = - \frac{2x(1-\k) \R (W+ \k) (\R-W)}{B} =  \frac{2(1-\k) \R (W+ \k) }{B} f  - \frac{2 \R (W+ \k)}{H + xW} . 
\]
Using further \eqref{E:FPRIME}, it leads to  
\begin{align}
f'&=-\Big(  \frac{   2\k x \F  + \frac{\k}{1-\k} \R^{-\e -1} }{(1-\k) \F  + 2\k x (1+\R) } +x  \Big)\Big(  \frac{2(1-\k) \R (W+ \k) }{B} \Big)  f\label{1.54}  \\
&\quad-  \left( \frac{  2\k \F  (1+\R) - \k (1-\k)x - (2\k x \F  + \frac{\k}{1-\k} \R^{-\e -1}) \frac{2 \R (W+ \k)}{H + xW} }{(1-\k) \F  + 2\k x (1+\R) } \right) \label{1.55}\\
&\quad - \left( -x \frac{2 \R (W+\k)}{H + xW}  + \R   \right). \label{1.56}
\end{align}
\eqref{1.54} is the form of $-a f$, and it corresponds to the first term of $a_1$ in \eqref{E:AONEDEF}. 

We next examine \eqref{1.56}. Using  \eqref{E:HDEF}, we see that 
\begin{align}
- \,\eqref{1.56} &= \frac{\R}{H+xW}\left( -2x (W + \k) + H + xW\right) \notag\\
&=  \frac{\R}{H+xW}\left( -2x (W +\k) + \F  + \frac{4\k}{1-\k} (1 + \R) x + xW  \right)\notag\\
&=\frac{\R}{H+xW}\left( \F  -  xW\right) + \frac{\R}{H+xW}\left( 2\k \frac{1+\k}{1-\k}x + \frac{4\k}{1-\k} \R x  \right)\notag \\
&=: I_1 + I_2. \label{I12}
\end{align}
Then the first term $I_1$ corresponds to the first term of $b_1$ in \eqref{E:BONEDEF}. The second term $I_2$ will be combined together with the second line \eqref{1.55}. 

Let us rewrite $I_2-\eqref{1.55}$ as 
\begin{align}
I_2-\eqref{1.55} = \frac{I_3}{ Z }
\end{align}
where $Z$ is given in \eqref{E:ZDEF0} and the numerator $I_3$ reads as 
\begin{align}
I_3& =  \R  ( (1-\k) \F  + 2\k x (1+\R ))( 2\k\tfrac{1+\k}{1-\k}x + \tfrac{4\k}{1-\k} \R x  ) +2\k \F  (1+\R )  (H+xW) \label{1.60}\\
& - 2 (W + \k)(2\k x \R \F  + \tfrac{\k}{1-\k}  \R ^{-\e}) -\k(1-\k)x (H+xW) \label{1.61}\\
&=\R  ( (1-\k) \F  + 2\k x (1+\R ))(2\k\tfrac{1+\k}{1-\k}x + \tfrac{4\k}{1-\k} \R x ) +2\k \F  (1+\R )  (H+\F ) \label{pos}\\
& - 2 (\tfrac{\F }{x} +\k)(2\k x \R \F  + \tfrac{\k}{1-\k}  \R ^{-\e}) - \k(1-\k) x (H+\F )\label{neg}\\
&  - 2 (W- \frac{\F }{x})(2\k x \R \F  + \tfrac{\k}{1-\k}  \R ^{-\e}) - \k(1-\k)x (xW-\F ) + 2\k \F  (1+\R ) (xW-\F ) . \label{1.65}
\end{align}
We first observe that \eqref{1.65} can be written into the form of $a f - b$:  %where $b_1\geq 0$: 
\begin{align}
\eqref{1.65}&=2\k (\F -xW)  ( \R -1 ) f   \\
& \quad + \k (\F -xW)  \left( x\left( 2\R ^2 - 2\R  + 1-\k \right) +  \tfrac{2}{1-\k}  \tfrac{\R ^{-\e}}{x}   \right)
\end{align}
where the first line corresponds to the first term of $a_2$ in \eqref{E:ATWODEF} and the second line corresponds to the second term of $b_1$ in \eqref{E:BONEDEF}.  

We next examine \eqref{pos} and \eqref{neg}. Using \eqref{E:HDEF} to replace $H$, and \eqref{FH} to replace  $\R  ^{-\e}$, and writing $\F = f + x\R$ to replace $\F $, we arrive at  
\begin{align}
\eqref{pos} + &\eqref{neg} =\tilde a f + \\
&+\R  ( (1-\k) x\R  + 2\k x (1+\R ))(2\k\tfrac{1+\k}{1-\k}x + \tfrac{4\k}{1-\k} \R x  )\label{pos1} \\
&+ 4\k x^2\R ^2 (1+\R ) + \tfrac{8\k^2}{1-\k} x^2 \R  (1+\R )^2 \\
&+2\k^2 x^2(\R  +\k ) \\
&- 2 (\R  +\k) (2\k x^2 \R ^2 + \k x^2\R ^2 + \tfrac{4\k^2}{1-\k}(1 +\R )x^2\R  ) \\
&-\k(1-\k)x (2x\R  + \tfrac{4\k}{1-\k}(1+\R )x ), \label{neg1} 
\end{align}
where  
\begin{align}
\tilde a&= \R  (1-\k) (2\k\tfrac{1+\k}{1-\k}x + \tfrac{4\k}{1-\k} \R x  ) \\
&\quad+ 2\k (1+\R ) \left[ 2 f + 2x \R  + (2x\R  + \tfrac{4\k}{1-\k} (1+\R ) x)   \right] \\
&\quad- 2 \tfrac{1}{x} \left[2\k x\R  f + 2\k x^2 \R ^2+ \tfrac{\k}{1-\k} \R ^{-\e} \right] \\
&\quad-2(\R +\k) \k \left[ 2x\R  + f +(2x\R  + \tfrac{4\k}{1-\k} (1+\R ) x)     \right]\\
&\quad-2\k (1-\k) x \\
&= -2\k \left(   \tfrac{1}{(1-\k)x} \R ^{-\e} + (1-\k) x + (\R  + \k)f \right)  \label{a11}\\
&\quad+ 2\k \left[ 2 f+ x\R  (5+\k) + 4\k x \right] .  \label{a22}
\end{align}
Now \eqref{a11} corresponds to the second term of $a_1$ in \eqref{E:AONEDEF} and \eqref{a22} corresponds to the last three terms of $a_2$ in \eqref{E:ATWODEF}. 
It remains to check the formula for $b_2$. To this end, we now group \eqref{pos1}--\eqref{neg1} into $\k$ term and $\k^2$ terms: 
\begin{align}
&\eqref{pos1}+\cdots+\eqref{neg1}\\
&= \Big[ 2\k(1+\k) x^2 \R ^2 + 4\k \R ^3 x^2 + 4\k^2\tfrac{1+\k}{1-\k} x^2 \R  (1+ \R ) + \tfrac{8\k^2}{1-\k} x^2 \R ^2 (1+\R ) \\
&\quad+ 4\k x^2 \R ^2 (1 + \R )  + \tfrac{8\k^2}{1-\k}x^2 \R  (1 + \R )^2 + 2\k^2 x^2 (\R + \k) \Big]\\
&- \Big[ 6\k x^2 \R ^3 +6\k^2 x^2 \R ^2 + \tfrac{8\k^2}{1-\k}x^2 \R ^2 (1+\R )+ \tfrac{8\k^3}{1-\k}x^2 \R  (1+\R ) \\
&\quad +2\k (1-\k) x^2 \R  +4\k^2 x^2 (1+\R ) \Big] \\
&= \big[8\k x^2 \R ^3 + 2\k x^2\R ^2 ( 3+\k)\big]  -\big[6\k x^2 \R ^3 + 2\k(1-\k) x^2 \R  \big] \notag\\
&+ \big[ 4\k^2\tfrac{1+\k}{1-\k} x^2 \R  (1+ \R ) + \tfrac{8\k^2}{1-\k}x^2 \R ^3 + \tfrac{16 \k^2}{1-\k} x^2 \R ^2 + 2\k^2 (\tfrac{4}{1-\k} +1)x^2\R  +2\k^3 x^2  \big]\notag \\
&- \big[ (6\k^2 + \tfrac{8\k^3}{1-\k} )x^2 \R ^2 + (\tfrac{8\k^3}{1-\k} + 4\k^2) x^2 \R  + 4\k^2 x^2  \big]\notag \\
&= 2\k x^2\R  \left[ \R ^2 + (3+\k) \R  - (1-\k) \right]  \\
& + 2\k^2x^2\left[ \tfrac{4}{1-\k} \R ^3+  \tfrac{2(1+\k)}{1-\k} \R (1+\R )  + \tfrac{5-\k}{1-\k} \R ^2 + 3 \R - 2+\k \right]. 
\end{align}
This completes the proof. 
\end{proof}

%%%%%%%%%%%%%%%%%%%%%%%%%%%%%%%%%%%
%%%%%%%%%%%%%%%%%%%%%%%%%%%%%%%%%%%

\begin{corollary}\label{C:FFORMULA}
Let $(\R ,W)$ be a smooth solution to the self-similar Einstein-Euler system~\eqref{E:RODE}--\eqref{E:WODE} on some interval $I\subset(0,\infty)$. 
%\begin{align}
%xW(x) & >F[R],   \ \ x\in(x_\ast,X). \label{E:SONICASSNEW}
%\end{align}
Then for any $x_1<x$, $x_1,x\in I$, we have the formula
\begin{align}
f(x) = & f(x_1) e^{-\int_{x_1}^x a[z;\R ,W]\,dz} \notag \\
& +e^{-\int_{x_1}^x a[z;\R ,W]\,dz}  \int_{x_1}^x \left(b_1[z;\R ,W] + \k b_2[z;\R ,W]\right) e^{\int_{x_1}^z a[s;\R ,W]\,ds}\,dz. 
\label{E:LITTLEFFORMULA}
\end{align}
\end{corollary}

\begin{proof}
The proof follows by applying the integrating factor method to~\eqref{E:LITTLEFDYNAMICS}.
\end{proof}

%%%%%%%%%%%%%%%%%%%%%%%%%%%
%%%%%%%%%%%%%%%%%%%%%%%%%%%

\begin{lemma}[Sign properties of $b$]\label{L:LITTLEBPOS}
Assume that  $\F -xW>0$ and $B[x;\R ,W]>0$. Then there exists an $\k_0>0$ sufficiently small such that for all $0<\k\le\k_0$ the following statements hold:  
\begin{enumerate}
\item[{\em (a)}] 
For any $\R >0$
\begin{align}
b_1[x;\R ,W]<0.
\end{align}
\item[{\em (b)}]
Furthermore, 
\begin{align}
b_2[x;\R ,W] <0 \ \ \text{ for all } \ \R >\frac13.
\end{align}
\end{enumerate}
\end{lemma}

%%%%%%%%%%%%%%%%%%%%%%%%%%%

\begin{proof}
The assumptions of the lemma and the decomposition~\eqref{E:BDEF} imply that $Z>0$ and $H+xW>0$.

\noindent
{\em Proof of part (a).}
The negativity of $b_1$ is obvious from~\eqref{E:BONEDEF} and the obvious bound
$2\R ^2-2\R +1-\k>0$ for $\k$ sufficiently small.

\noindent
{\em Proof of part (b).} Let $\varphi_0$ be the larger of the two roots of the quadratic polynomial $\R \mapsto \R ^2+(3+\k)\R  -(1-\k)$, which is given by
\[
\varphi_0:=\frac{-(3+\k) +\sqrt{(3+\k)^2+ 4(1-\k)}}{2}.
\]
It is easily checked that there exists an $\k_0>0$ such that $0<\varphi_0<\frac13$ for all $0<\k\le\k_0$ and in particular
\[\R \left[\R ^2+(3+\k)\R  -(1-\k)\right] >0\]
for $\R >\frac13 (>\varphi_0)$. On the other hand, 
\[
\frac4{1-\k}\R ^3+\frac{2(1+\k)}{1-\k}\R (1+\R ) + \frac{5-\k}{1-\k}\R ^2+3\R  - 2+\k> 4 \R ^3 + 6\R ^2 +5\R -2>0 
\]
for $\R >\frac13$ and the claim follows from~\eqref{E:BTWODEF}.
\end{proof}

%%%%%%%%%%%%%%%%%%%%%%%%%%%
%%%%%%%%%%%%%%%%%%%%%%%%%%%

%%%%%%%%%%%%%%%%%%%%%%%%%%%
%%%%%%%%%%%%%%%%%%%%%%%%%%%
%%%%%%%%%%%%%%%%%%%%%%%%%%%

\section{The sonic point analysis}\label{S:SONIC}

%%%%%%%%%%%%%%%%%%%%%%%%%%%
%%%%%%%%%%%%%%%%%%%%%%%%%%%
%%%%%%%%%%%%%%%%%%%%%%%%%%%

%We look for solutions $\R ,W$ to \eqref{E:RODE}-\eqref{E:WODE} of the form
%\be\label{Taylor}
%\R = \sum_{N=0}^\infty \R_N (x-\xs)^N, \quad W = \sum_{N=0}^\infty W_N (x-\xs)^N
%\ee
%where  $\R_0$, $W_0$ are given in \eqref{W0}. 
%Our goal is to express $\R_N, W_N$ recursively in terms of $\R_0,\dots, \R_{N-1}$, $W_0,\dots, W_{N-1}$ and thus obtain a hierarchy of  algebraic relations 
%that allows to compute the Taylor coefficients up to an arbitrary order. 
It turns out that for purposes of homogeneity, it is more convenient to work with rescaled unknowns where the sonic point is 
pulled-back to a fixed value $1$. Namely, we introduce the change
of variables: 
\be
z= \frac{x}{x_\ast}, \quad \WH(z)=W(x), \quad \RH(z) =  \R(x). 
\ee
so that the sonic point $x_\ast$ is mapped to $z=1$. 
It is then easily checked from~\eqref{E:RODE}--\eqref{E:WODE} that $(\RH,\WH)$ solves
\begin{align}
\frac{d  \RH}{dz} &= -  \frac{ 2x_\ast^2 z(1-\k)  \RH ( \WH +\k) ( \RH- \WH) }{B} ,\label{Eq:R} \\
\frac{d \WH}{dz}&= \frac{(1-3\WH )}{z} + \frac{2x_\ast^2 z(1+\k)  \WH ( \WH + \k)
( \RH- \WH)}{B}, \label{Eq:W}
\end{align}
where 
\be
B =  \RH^{-\e} -\left[ ( \WH + \k)^2 - \k ( \WH - 1)^2 + 4\k  \RH \WH \right] x_\ast^2z^2. 
\ee

We introduce
\be
\dz := z-1. 
\ee
we look for solutions $\RH,\WH$ to~\eqref{Eq:R}-\eqref{Eq:W} of the form
\be\label{E:TAYLORZ}
\RH = \sum_{N=0}^\infty \RH_N (\delta z)^N, \quad \WH = \sum_{N=0}^\infty \WH_N (\delta z)^N.
\ee
We observe that there is a simple relation between the formal Taylor coefficients of $(\R,W)$ and $(\RH,\WH)$:
\begin{align}\label{E:SONICTRANS}
\RH_j = \R_j \xs^j, \ \ \WH_j = W_j \xs^j, \ \ j=0,1,2,\dots.
\end{align}

%%%%%%%%%%%%%%%%%%%%%%%%%%%
%%%%%%%%%%%%%%%%%%%%%%%%%%%

\subsection{Sonic point conditions}

%%%%%%%%%%%%%%%%%%%%%%%%%%%
%%%%%%%%%%%%%%%%%%%%%%%%%%%

\begin{lemma}[Sonic conditions]\label{R0W0} %Let $x_\ast\in [\frac32, \frac52]$ be fixed. 
There exists a small $\k_0>0$ such that for all $|\k|\leq \k_0$ and all $x_\ast\in [\frac32, \frac72]$ there exists a continuously differentiable curve 
$[-\k_0,\k_0]\ni\k \mapsto (\RH_0(\k),\WH_0(\k))=(\RH_0(\k;x_\ast),\WH_0(\k;x_\ast))$ such that 
\begin{enumerate}
\item  $\RH_0(\k)=\WH_0(\k)>0$. 
\item $B[\RH_0(\k),\WH_0(\k)]=0$. 
\item $\RH_0(0)=\WH_0(0) = \frac{1}{x_\ast}$. 
\item $-\infty<\pa_\k\RH_0(0)=\pa_\k\WH_0(0)<0$. 
 \end{enumerate}
\end{lemma}

%%%%%%%%%%%%%%%%%%%%%%%%%%%

\begin{proof}
Fix an $x_\ast\in[\frac32,\frac72]$. Consider a small neighbourhood of  $(\k,\WH)=(0,\frac{1}{x_\ast})$, open rectangle $(\k,\WH)\in (-l,l)\times (w_1,w_2) $, and a continuously differentiable function 
$h:(-l,l)\times (w_1,w_2) \rightarrow \mathbb R$ defined by 
\be\label{SC}
h(\k,\WH):= \left[ (1+3\k) \WH^2 + 4\k \WH +\k(\k-1)\right] \WH^{\e} - \frac{1}{x_\ast^2},
\ee
where we recall $\e=\e(\k)=\frac{2\k}{1-\k}$.
Then the sonic point conditions $\RH=\WH$ and $B=0$ with $z=1$ reduce to $h(\k,\WH)=0$ and moreover we have $h(0,\frac{1}{x_\ast})=0$. Clearly $h$ is continuously differentiable in all arguments. Observe that 
\be\label{dfdW}
\frac{\partial h}{\partial \WH} = \left[ 2(1+3\k) \WH^2 + 4\k \WH  +\e\left((1+3\k) \WH^2 + 4\k \WH + \k(\k-1)  \right)  \right] \WH^{\e-1}
\ee
from which we have
\be
\frac{\partial h}{\partial \WH} \Big|_{(\k,\WH)=(0,\frac{1}{x_\ast})} = \frac{2}{x_\ast} >0 . 
\ee
Therefore, by the implicit function theorem, we deduce that there exists an open interval $(-l_0,l_0)$ of $\k=0$ and a unique continuously differential function $g:(-l_0,l_0) \rightarrow (w_1,w_2)$ such that $g(0)=\frac{1}{x_\ast}$ 
and $h(\k, g(\k)) =0$ for all $\k\in (-l_0,l_0)$. Moreover, we have 
\be
\frac{\partial g}{\partial \k} = - \frac{\frac{\partial h}{\partial \k}}{\frac{\partial h}{\partial \WH}}, \quad \k\in (-l_0,l_0)
\ee
where 
\be\label{dfdk}
\frac{\partial h}{\partial \k}= \left[ 3 \WH^2 + 4\WH + 2\k-1+ \left( (1+3\k) \WH^2 + 4\k \WH +\k(\k-1)\right) \frac{2\ln \WH}{(1-\k)^2} \right] \WH^{\e}.
\ee
When $\k=0$, 
\be
\begin{split}
\frac{\partial g}{\partial \k}(0) = - \frac{\frac{\partial h}{\partial \k}|_{(0,\frac{1}{x_\ast})}}{\frac{\partial h}{\partial \WH}|_{(0,\frac{1}{x_\ast})}}
&=- \frac{\frac{3}{x_\ast^2} + \frac{4}{x_\ast} -1 -\frac{2\ln x_\ast}{x_\ast^2}}{\frac{2}{x_\ast}} =  \frac{x_\ast^2 - 4x_\ast+2\log x_\ast-3}{2x_\ast}. 
%&<0 \quad \text{for } \ \ x_\ast\in (2^-,3^+)
\end{split}
\ee
 The derivative of the map $x_\ast\to x_\ast^2 - 4x_\ast+2\log x_\ast-3$ is $2\frac{(\xs-1)^2}{\xs}$ and the function is therefore strictly increasing for $x_\ast\neq1$. 
It is easy to check that the value at $x_\ast=\frac72$ is negative,
and therefore there exists a constant $\kappa>0$ such that $\frac{\partial g}{\partial \k}(0)<-\kappa$ for all $x_\ast\in[\frac32,\frac72]$.  
In particular, there exists a $0<\k_0\ll1$ sufficiently small  and a constant $c_\ast>0$ such that
\[
\frac1{x_\ast}-c_\ast \k < g(\k;x_\ast)< \frac{1}{x_\ast}, \ \ \k\in(0,\k_0], \ \ x_\ast\in[\frac32,\frac72].
\]
We let
\be\label{W0}
 \RH_0:= \RH(\k)=g(\k), \quad \WH_0:= \WH_0(\k) = g(\k). 
\ee
\end{proof}

%%%%%%%%%%%%%%%%%%%%%%%%%%%
%%%%%%%%%%%%%%%%%%%%%%%%%%%

\begin{remark}[The map $\xs\mapsto \WH_0(\k;\xs)$ is decreasing]\label{R:XSDECREASING}
In order to examine the behaviour of $\WH_0(\k; x_\ast)=g(\k;x_\ast)$ as a function of $x_\ast$ for any fixed $\k$, we rewrite the relation $h(\k,\WH_0(\k))=0$ in the form
\[
h(\k,g(\k;x_\ast);x_\ast)=0.
\]
Upon taking the $\frac{\pa}{\pa \xs}$ derivative of the above, we easily see that $\pa_{\xs}g(\k;\xs)<0$. In fact,  we have $\pa_{\xs}g = - \frac{\pa_{\xs} h}{\pa_{\WH} h}$ where $\pa_{\xs} h = \frac{2}{x_\ast^3}>0$ from \eqref{SC},   %at $\k=0$
 and $\pa_{\WH} h>0$ is given in \eqref{dfdW}.
\end{remark}

%%%%%%%%%%%%%%%%%%%%%%%%%%%
%%%%%%%%%%%%%%%%%%%%%%%%%%%

%so that $W_0$, $\RH_0$ satisfy the sonic conditions $W_0=\RH_0$ and \eqref{SC} $f(k,W_0)=0$. 

%\subsection{Taylor expansion}

%Let 
%\be
%\delta z := z-1
%\ee
%we look for solutions $R,W$ to \eqref{Eq:R}-\eqref{Eq:W} of the form
%\be\label{Taylor}
%R = \sum_{N=0}^\infty R_N (\delta z)^N, \quad W = \sum_{N=0}^\infty W_N (\delta z)^N
%\ee
%where  $R_0$, $W_0$ are given in \eqref{W0}. The first goal is to compute a recursive relation of $R_N, W_N$ in terms of $R_0,\dots, R_{N-1}$, $W_0,\dots, W_{N-1}$. 

For a given function $f$, we write $(f)_M$, $M\in\mathbb N$, to denote the $M$-th Taylor coefficient in the expansion of $f$ around the sonic point $z = 1$. In particular, 
\begin{align*}
(\RH ( \WH + \k	) ( \RH- \WH))_M &= \sum_{l+m+n=M} \RH_l ( \WH_m + \k\delta_{m}^{0}) ( \RH_n- \WH_n), \\
(\WH ( \WH + \k)( \RH- \WH))_M &=  \sum_{l+m+n=M} \WH_l ( \WH_m + \k\delta_{m}^{0}) ( \RH_n- \WH_n),\\
(\WH^2)_M &= \sum_{l+m=M} \WH_l \WH_m, \\
(\RH\WH)_M &= \sum_{l+m=M} \RH_l \WH_m.
\end{align*}
We set $(f)_M=0$ for $M<0$.

{\em Formula of Faa Di Bruno.} Given two functions $f,g$ with formula power series expansions
\be
f(x)=\sum_{n=0}^\infty {f_n} x^n, \quad g(x)=\sum_{n=1}^\infty {g_n} x^n, 
\ee
we can compute the formal Taylor series expansion of the composition $h=f \circ g$ via 
\be\label{E:FAAH}
h(x) =\sum_{n=0}^\infty {h_n} x^n
\ee
where 
\be
h_n = \sum_{m=1}^n \sum_{\pi(n,m)} \frac{m !}{\lambda_1 !\dots \lambda_n !} f_m \left({g_1} \right)^{\lambda_1}\dots
\left({g_n}\right)^{\lambda_n}, \quad h_0= f_0
\ee
and 
\be\label{pi}
\pi(n,m)= \left\{(\lambda_1, \dots,\lambda_n): \lambda_i \in \mathbb Z_{\geq 0}, \sum_{i=1}^n \lambda_i =m, \sum_{i=1}^n i \lambda_i =n \right\}. 
\ee
An element of $\pi(n,m)$ encodes the partitions of the first $n$ numbers into $\lambda_i$ classes of cardinality $i$ 
for $i\in \{1,\dots,m\}$. Observe that by necessity 
\[
\lambda_j=0 \ \text{ for } \  n-m+2\leq  j\leq n. 
\]
{To see this, suppose  $\lambda_j  = p\geq 1$ for some $n-m+2\leq  j\leq n$. Then $ m - p=\sum_{i\neq j} \lambda_i \leq \sum_{i\neq j} i\lambda_i =  n-j p \leq n- (n-m+2) p$, which leads to $(n-m+1)p \leq n-m$. But this is impossible if $p\geq 1$.}

Now 
\be
\begin{split}
\RH^{-\e} &= \RH_0^{-\e} \left( 1+ \sum_{i=1}^\infty \frac{\RH_i}{\RH_0} (\delta z)^i  \right)^{-\e}= \RH_0^{-\e}+ \sum_{j=1}^\infty (\RH^{-\e}  )_j (\delta z)^j 
\end{split}
\ee
where
\be\label{faaR}
(\RH^{-\e}  )_j =\RH_0^{-\e} \sum_{m=1}^j \frac{1}{\RH_0^m}\sum_{\pi(j,m)} (- \e)_m \frac{1}{\lambda_1 ! \dots \lambda_j !} {\RH_1}^{\lambda_1}\dots 
{\RH_j}^{\lambda_j}, \quad j \geq 1. 
\ee 
Here $(-\e)_m = (-\e)(-\e-1) \cdots (-\e-m+1)$.  
Then we may write $B$ as 
\begin{align}
B&= \RH^{-\e} -\left[ (1-\k)\WH^2  + 4\k \WH+ 4\k  \RH \WH+ \k^2-\k \right] x_\ast^2z^2\\
&=:  \RH^{-\e} - x_\ast^2 Hz^2\\
&= \sum_{l=0}^\infty (\RH^{-\e})_l  (\delta z)^l - x_\ast^2 
 \sum_{l=0}^\infty H_l (\delta z)^l(1+ 2 \delta z + (\delta z)^2) \label{B_exp}
%&= \RH_0^{-\e}+ \sum_{j=1}^\infty S_j (\delta z)^j \\
%&- x_\ast^2\left[ \sum_{j=0}^\infty \left((1-k)(\WH^2)_j  + \frac{4k}{k+1} \WH_j+ 4k ( \RH \WH)_j\right)(\delta z)^j+ \frac{k^2-k}{(k+1)^2} \right] (1+2\delta z+(\delta z)^2)
\end{align}
so that
\be\label{Hl}
H_l = (1-\k)(\WH^2)_l  + 4\k \WH_l+ 4\k ( \RH \WH)_l+ (\k^2-\k)\delta^0_l . 
\ee 

\begin{lemma} For any $N\geq 0$, the following formulas hold: 
\be
\begin{split}\label{RN+1}
&\sum_{l+m=N} (m+1) \RH_{m+1} (\RH^{-\e})_l  \\
&- x_\ast^2 \left( \sum_{l+m=N} (m+1) \RH_{m+1} H_l + 2\sum_{l+m=N-1} (m+1) \RH_{m+1}H_l +\sum_{l+m=N-2} (m+1) \RH_{m+1}H_l  \right)\\
&+ 2x_\ast^2(1-\k)\left( ( \RH ( \WH + \k) ( \RH- \WH) )_N +( \RH ( \WH + \k) ( \RH- \WH) )_{N-1}  \right) =0
\end{split}
\ee
and 
\be\label{WN+1}
\begin{split}
&\sum_{l+m=N} (m+1) \WH_{m+1} (\RH^{-\e})_l  \\
&- x_\ast^2 \left( \sum_{l+m=N} (m+1) \WH_{m+1} H_l + 2\sum_{l+m=N-1} (m+1) \WH_{m+1}H_l +\sum_{l+m=N-2} (m+1) \WH_{m+1}H_l  \right)\\
& - \sum_{l+m=N} (\RH^{-\e})_l  (-1)^m + 3\sum_{l+m+n=N} \WH_n (\RH^{-\e})_l  (-1)^m    \\
& + \xs^2 \left(  \sum_{l+m=N} H_l (-1)^m +2 \sum_{l+m=N-1} H_l (-1)^m  + \sum_{l+m=N-2}H_l (-1)^m  \right)  \\
& -3x_\ast^2 \left(  \sum_{l+m+n=N} \WH_n H_l (-1)^m +2 \sum_{l+m+n=N-1} \WH_nH_l (-1)^m  + \sum_{l+m+n=N-2}\WH_n H_l (-1)^m  \right) \\ 
&- 2x_\ast^2(1+\k)\left( ( \WH ( \WH + \k ) ( \RH- \WH) )_N +( \WH ( \WH + \k) ( \RH- \WH) )_{N-1}  \right) =0
\end{split}
\ee
where we recall $(f)_M=0$ for $M<0$. 
\end{lemma}

\begin{proof} {\it Proof of \eqref{RN+1}. } We now plug \eqref{E:TAYLORZ} %\eqref{Taylor} 
into 
\begin{align}
B \RH' &+ 2x_\ast^2(1-\k)   (1+\dz) \RH ( \WH + \k) ( \RH- \WH) =0, \label{Eq:R1}\\
B \WH'&- \frac{(1 -3 \WH ) B}{1+\delta z} - 2x_\ast^2 (1+\k)  (1+\delta z) \WH ( \WH + \k)
( \RH- \WH)=0. \label{Eq:W1}
\end{align}
\eqref{Eq:R1} reads as 
\begin{align*}
0&= \left[\sum_{l=0}^\infty (\RH^{-\e})_l  (\delta z)^l - x_\ast^2
 \sum_{l=0}^\infty H_l (\delta z)^l(1+ 2 \delta z + (\delta z)^2)\right]\left[ \sum_{m=0}^\infty (m+1) \RH_{m+1} (\delta z)^m \right] \\
& + 2x_\ast^2(1-\k) \sum_{l=0}^\infty ( \RH ( \WH + \k ) ( \RH- \WH) )_l (\delta z)^l (1+\delta z) \\
&= \sum_{N=0}^\infty \sum_{l+m=N} (m+1) \RH_{m+1} (\RH^{-\e})_l (\delta z)^N \\
& - x_\ast^2 \sum_{N=0}^\infty\left( \sum_{l+m=N} (m+1) \RH_{m+1} H_l + 2\sum_{l+m=N-1} (m+1) \RH_{m+1}H_l +\sum_{l+m=N-2} (m+1) \RH_{m+1}H_l  \right) (\delta z)^N\\
&+ 2x_\ast^2(1-\k) \sum_{N=0}^\infty \left( ( \RH ( \WH + \k ) ( \RH- \WH) )_N +( \RH ( \WH + \k ) ( \RH- \WH) )_{N-1}  \right) (\delta z)^N . 
\end{align*}
Comparing the coefficients, we obtain \eqref{RN+1}. 

\

%\begin{align}
%&\sum_{l+m=N} (m+1) R_{m+1} (R^{-\e})_l  \\
%&- x_\ast^2 \left( \sum_{l+m=N} (m+1) \RH_{m+1} H_l + 2\sum_{l+m=N-1} (m+1) \RH_{m+1}H_l +\sum_{l+m=N-2} (m+1) \RH_{m+1}H_l  \right)\\
%&+ 2x_\ast^2(1-k)\left( ( \RH ( W + \tfrac{k}{k+1}) ( \RH- W) )_N +( R ( W + \tfrac{k}{k+1}) ( \RH - W) )_{N-1}  \right) =0
%\end{align}

\noindent{\it Proof of \eqref{WN+1}. } 
Since $\frac{1}{1+\delta z} = \sum_{m=0}^\infty (-1)^m (\delta z)^m$, we can expand $\frac{(1 -3 \WH ) B}{1+\delta z}$ as 
\begin{align*}
&\frac{(1 -3 \WH ) B}{1+\delta z} \\
&= \left(1 -3\sum_{n=0}^\infty \WH_n (\delta z)^n \right)
\left( \sum_{l=0}^\infty (\RH^{-\e})_l  (\delta z)^l - x_\ast^2
 \sum_{l=0}^\infty H_l(\delta z)^l (1+ 2 \delta z + (\delta z)^2) \right)   \sum_{m=0}^\infty (-1)^m (\delta z)^m\\
 &= \sum_{N=0}^\infty \sum_{l+m=N} (\RH^{-\e})_l  (-1)^m (\delta z)^N - 3 \sum_{N=0}^\infty \sum_{l+m+n=N} \WH_n (\RH^{-\e})_l  (-1)^m (\delta z)^N \\
 &- \xs^2  \sum_{N=0}^\infty\left(  \sum_{l+m=N} H_l (-1)^m +2 \sum_{l+m=N-1} H_l (-1)^m  + \sum_{l+m=N-2}H_l (-1)^m  \right)(\delta z)^N\\
 &+3x_\ast^2  \sum_{N=0}^\infty\left(  \sum_{l+m+n=N} \WH_n H_l (-1)^m +2 \sum_{l+m+n=N-1} \WH_nH_l (-1)^m  + \sum_{l+m+n=N-2}\WH_n H_l (-1)^m  \right)(\delta z)^N. 
\end{align*}
We plug  \eqref{Taylor} into \eqref{Eq:W1} 
\begin{align*}
0&= \left[\sum_{l=0}^\infty (\RH^{-\e})_l  (\delta z)^l - x_\ast^2
 \sum_{l=0}^\infty H_l (\delta z)^l(1+ 2 \delta z + (\delta z)^2)\right]\left[ \sum_{m=0}^\infty (m+1) \WH_{m+1} (\delta z)^m \right] \\
 &- \sum_{N=0}^\infty \left( \frac{(1 -3 \WH ) B}{1+\delta z} \right)_N  (\delta z)^N  - 2x_\ast^2(1+\k) \sum_{l=0}^\infty ( \WH ( \WH + \k ) ( \RH- \WH) )_l (\delta z)^l (1+\delta z) \\
&= \sum_{N=0}^\infty \sum_{l+m=N} (m+1) \WH_{m+1} (\RH^{-\e})_l (\delta z)^N \\
& - x_\ast^2 \sum_{N=0}^\infty\left( \sum_{l+m=N} (m+1) \WH_{m+1} H_l + 2\sum_{l+m=N-1} (m+1) \WH_{m+1}H_l +\sum_{l+m=N-2} (m+1) \WH_{m+1}H_l  \right) (\delta z)^N\\
&- \sum_{N=0}^\infty \left( \frac{(1 -3 \WH ) B}{1+\delta z} \right)_N  (\delta z)^N  \\
&- 2x_\ast^2(1+\k) \sum_{N=0}^\infty \left( ( \WH ( \WH + \k ) ( \RH- \WH) )_N +( \WH ( \WH + \k ) ( \RH- \WH) )_{N-1}  \right) (\delta z)^N 
\end{align*}
which leads to \eqref{WN+1}.  
\end{proof}

\begin{remark}\label{remark1} $(\RH_0,\WH_0)$ obtained in Lemma \ref{R0W0} satisfy the sonic conditions: 
\be
\RH_0=\WH_0, \quad \RH_0^{-\e}=x_\ast^2 H_0 \label{sonic}
\ee
%$R_0=W_0$ and $R_0^{-\e}=x_\ast^2 H_0$, 
and it is easy to verify that  for such choice of $(\RH_0,\WH_0)$, the relations $\eqref{RN+1}_{N=0}$ and $\eqref{WN+1}_{N=0}$ trivially hold. Moreover, with \eqref{sonic}, there are no $(\RH_{N+1}, \WH_{N+1})$ in  $\eqref{RN+1}$ and $\eqref{WN+1}$. 
\end{remark}

%%%%%%%%%%%%%%%%%%%%%%%%%%%

%%%%%%%%%%%%%%%%%%%%%%%%%%%%%
%%%%%%%%%%%%%%%%%%%%%%%%%%%%%

\begin{remark}\label{R:CONSISTENCYSONIC}
To determine $R_1, \WH_1$, we record  $\eqref{RN+1}_{N=1}$ and $\eqref{WN+1}_{N=1}$: 
\be\label{R1ZERO}
\begin{split}
(\RH^{-	\e })_1 \RH_1 &- x_\ast^2 H_1 \RH_1 - 2x_\ast^2 H_0 \RH_1+ 2x_\ast^2 (1-\k) \RH_0(\WH_0 + \k)(\RH_1 -\WH_1) =0 
\end{split}
\ee
and
\be\label{W1ZERO}
\begin{split}
&(\RH^{-\e  })_1 \WH_1 - x_\ast^2 (\WH_1 H_1+ 2\WH_1 H_0) 
- [ (\RH^{-\e  })_1- \RH_0^{-\e  }] \\
&+ 3 [ \WH_1\RH_0^{-\e  } + \WH_0 (\RH^{-\e  })_1 - \WH_0 \RH_0^{-\e  }] + x_\ast^2 [H_1+H_0] \\
&- 3x_\ast^2 [\WH_1 H_0 + \WH_0 H_1 + \WH_0 H_0]
- 2x_\ast^2 (1+\k) \WH_0(\WH_0 + \k)(\RH_1 -\WH_1) =0, 
\end{split}
\ee where we have used the sonic conditions \eqref{sonic}. Recalling from \eqref{faaR} and \eqref{Hl}
\begin{align}
(\RH^{-\e  })_1 &= -\e\RH_0^{-\e  -1} \RH_1,\label{R11}\\
H_1 &= 2\left[ (1+\k)\WH_0 + 2\k\right] \WH_1 + 4\k \WH_0 \RH_1, \label{R11.1}
\end{align} 
we see that for general nonzero $\k$, \eqref{R1ZERO} is linear in $\WH_1$ and quadratic in $\RH_1$ and \eqref{W1ZERO} is linear in $\RH_1$ and quadratic in $\WH_1$. On the other hand, when $\k=0$, since $(\RH^{-\e  })_1 =0$, $H_1= 2\WH_0\WH_1$, $H_0=\WH_0^2$,  \eqref{R1ZERO}  becomes  
\[
\begin{split}
-2x_\ast^2 \WH_0\WH_1 \RH_1 - 2x_\ast^2\WH_0^2 \RH_1 + 2x_\ast^2 \RH_0\WH_0 (\RH_1-\WH_1)=0, 
\end{split}
\]
and \eqref{W1ZERO} becomes 
\[
\begin{split}
-x_\ast^2(2\WH_0\WH_1^2 + 2 \WH_1\WH_0^2) + 1 + 3[\WH_1-\WH_0] + x_\ast^2[2\WH_0\WH_1+\WH_0^2] \\
- 3x_\ast^2 [\WH_1\WH_0^2+2\WH_0^2\WH_1 + \WH_0^3] - 2x_\ast^2\WH_0^2 (\RH_1-\WH_1) =0. 
\end{split}
\]
Thus by using $\RH_0\big|_{\k=0}=\WH_0\big|_{\k=0}=\frac{1}{x_\ast}$, we see that 
\eqref{R1ZERO} and \eqref{W1ZERO} are reduced to 
\begin{align}
-2\WH_1 \left( x_\ast {\RH_1} +1 \right) =0, \label{R10} \\
-2 \left( x_\ast \WH_1^2 - x_\ast \WH_1 + 3\WH_1 + \RH_1 + \frac{3}{x_\ast} -1\right) =0, \label{W10}
\end{align}which are the the sonic point conditions satisfied in the Newtonian limit, see~\cite{GHJ2021}.
In general, there are two pairs of solutions to \eqref{R10}-\eqref{W10}, one is of Larson-Penston type given by $(\RH_1,\WH_1)=(-\frac{1}{x_\ast}, 1-\frac{2}{x_\ast})$  and the other one is of Hunter type given by $(\RH_1,\WH_1)=(1-\frac{3}{x_\ast},0)$. 
\end{remark}

In the following, we will show that there exists a continuously differentiable curve $(\RH_1(\k),\WH_1(\k))$ satisfying 
\eqref{R1ZERO}--\eqref{W1ZERO}, which at $\k=0$ agrees precisely with the values of $(\RH_1,\WH_1)$ 
associated with the (Newtonian) Larson-Penston solution, namely $(\RH_1(0),\WH_1(0))=(-\frac{1}{x_\ast}, 1-\frac{2}{x_\ast}) $ for $x_\ast>2$.

\begin{lemma}[RLP conditions]\label{R1W1} Let $x_\ast\in (2, \frac72)$ be fixed and let $(\RH_0,\WH_0)$ be as obtained in Lemma \ref{R0W0}. 
Then there exists an $\k_1>0$ such that there exists a continuously differentiable curve $(-\k_1,\k_1)\ni \k\mapsto (\RH_1(\k)$, $\WH_1(\k))$  such that 
\begin{enumerate}
\item  The relations $\eqref{RN+1}_{N=1}$ and $\eqref{WN+1}_{N=1}$ hold.   %, and
\item $\RH_1(0)= -\frac{1}{x_\ast}$ and $\WH_1(0)= 1-\frac{2}{x_\ast}$. 
%\item $-\infty<\RH_1'(0)$ $\WH_1'(0)<0$
 \end{enumerate}
\end{lemma}

\begin{proof} As in Lemma \ref{R0W0} we will use the implicit function theorem. To this end, consider a small neighborhood $(-l,l)\times (r_1,r_2)\times (w_1,w_2)$ of $(0, -\frac{1}{x_\ast}, 1-\frac{2}{x_\ast})$ and introduce $\mathcal F:(-l,l)\times (r_1,r_2)\times (w_1,w_2)\rightarrow \mathbb R^2$ as 
\be
\mathcal F(\k , \RH_1,\WH_1) = \left( \mathcal F_1 (\k , \RH_1,\WH_1),\mathcal F_2(\k , \RH_1,\WH_1)  \right)^T
\ee where 
\be
\mathcal F_1(\k , \RH_1,\WH_1)=\text{ LHS of }\eqref{R1ZERO} \ \text{ and } \  \mathcal F_2(\k , \RH_1,\WH_1)=\text{ LHS of }\eqref{W1ZERO}
\ee so that  \eqref{R1ZERO} and \eqref{W1ZERO} are equivalent to $\mathcal F(\k , \RH_1,\WH_1) =0$. 
%\begin{align}
%\mathcal F_1(\k , R_1,\WH_1)& = 2x_\ast^2 \left[ ((1+\k )\WH_0 + \tfrac{2\k }{1+\k }) R_1 + (1-k)\WH_0 (\WH_0+\tfrac{k}{1+k}) \right] \WH_1  \\
%&+ 2x_\ast^2 k\left[ \frac{(3+k)\WH_0^2 + \frac{4k}{1+k} \WH_0 + \frac{k^2-k}{(1+k)^2}}{(1-k) \WH_0} R_1^2 + (4\WH_0^2 + \tfrac{3+k}{1+k}\WH_0 + \tfrac{k-1}{(1+k)^2} ) R_1 \right]\notag \\\mathcal F_2(k, R_1,\WH_1) &= 
%\end{align}
It is clear that $\mathcal F$ is continuously differentiable in all arguments and  $\mathcal F (0, -\frac{1}{x_\ast}, 1-\frac{2}{x_\ast})=0 $. We next compute the Jacobi matrix  $\frac{\partial \mathcal F}{\partial [\RH_1,\WH_1]}$ 
\be
\begin{split}
\frac{\partial \mathcal F}{\partial [\RH_1,\WH_1]} & = 
\left( 
\begin{array}{cc} \frac{\partial \mathcal F_1}{\partial \RH_1} &  \frac{\partial \mathcal F_1}{\partial \WH_1} \\
\frac{\partial \mathcal F_2}{\partial \RH_1} &  \frac{\partial \mathcal F_2}{\partial \WH_1}
\end{array}
\right)%\Big|_{k=0}  \\
\end{split}
\ee where 
\begin{align*}
\frac{\partial \mathcal F_1}{\partial \RH_1}& = -2\e \RH_0^{-\e-1} \RH_1 - x_\ast^2 H_1 - x_\ast^2 \RH_1 \tfrac{\partial H_1}{\partial \RH_1} - 2x_\ast^2 H_0 + 2x_\ast^2 (1-\k) \RH_0 (\WH_0 +\k) , \\
\frac{\partial \mathcal F_1}{\partial \WH_1} &= - x_\ast^2 \RH_1 \tfrac{\partial H_1}{\partial \WH_1} -  2x_\ast^2 (1-\k) \RH_0 (\WH_0 +\k) , \\
\frac{\partial \mathcal F_2}{\partial \RH_1}& =- \e \RH_0^{-\e-1} \WH_1 - x_\ast^2 \WH_1 \tfrac{\partial H_1}{\partial \RH_1} + \e \RH_0^{-\e-1} -  3\e \WH_0\RH_0^{-\e-1} \\
&\quad+ \xs^2\tfrac{\partial H_1}{\partial \RH_1} - 3x_\ast^2 \WH_0 \tfrac{\partial H_1}{\partial \RH_1} - 2x_\ast^2 (1+\k) \WH_0 (\WH_0 +\k), \\
\frac{\partial \mathcal F_2}{\partial \WH_1}& = - \e \RH_0^{-\e-1} \RH_1 - x_\ast^2 (H_1+2H_0)- x_\ast^2 \WH_1 \tfrac{\partial H_1}{\partial \WH_1} + 3 \RH_0^{-\e}  \\
&\quad+ x_\ast^2\tfrac{\partial H_1}{\partial \WH_1}- 3x_\ast^2 H_0 
- 3x_\ast^2 \WH_0 \tfrac{\partial H_1}{\partial \WH_1} + 2x_\ast^2 (1+\k) \WH_0 (\WH_0 +\k), 
\end{align*}
{and 
\[
\tfrac{\partial H_1}{\partial \RH_1} = 4\k\WH_0, \quad \tfrac{\partial H_1}{\partial \WH_1} = 2[(1+\k)\WH_0 + 2\k]. 
\]
Using $H_1(0)= 2\WH_0(0)\WH_1(0)$, $\tfrac{\partial H_1}{\partial \RH_1} (0)=0$, $\tfrac{\partial H_1}{\partial \WH_1}(0)= 2\WH_0(0)$,  $H_0(0)=\WH_0^2(0)$, $\WH_0(0)=\RH_0(0) =\frac{1}{x_\ast}$, $\WH_1(0)=\frac{x_\ast-2}{x_\ast}$ and $\RH_1(0)=-\frac{1}{x_\ast}$, we evaluate the above at $(\k,\RH_1,\WH_1)= (0,-\frac{1}{x_\ast}, 1-\frac{2}{x_\ast})$
\[
\begin{split}
\frac{\partial \mathcal F_1}{\partial \RH_1}\Big|_{(k,\RH_1,\WH_1)= (0,-\frac{1}{x_\ast}, 1-\frac{2}{x_\ast})}&= - x_\ast^2H_1(0) - 2x_\ast^2 H_0 (0) + 2x_\ast^2 \RH_0(0) \WH_{0}(0) = - 2(x_\ast -2) , \\
\frac{\partial \mathcal F_1}{\partial \WH_1}\Big|_{(k,\RH_1,\WH_1)= (0,-\frac{1}{x_\ast}, 1-\frac{2}{x_\ast})}&=- 2x_\ast^2 \RH_1(0) \WH_0(0) - 2x_\ast^2 \RH_0(0) \WH_0(0)= 0,  \\
\frac{\partial \mathcal F_2}{\partial \RH_1}\Big|_{(k,\RH_1,\WH_1)= (0,-\frac{1}{x_\ast}, 1-\frac{2}{x_\ast})}&=- 2x_\ast^2 \WH_0^2(0)= -2, \\
\frac{\partial \mathcal F_2}{\partial \WH_1}\Big|_{(k,\RH_1,\WH_1)= (0,-\frac{1}{x_\ast}, 1-\frac{2}{x_\ast})}&=- x_\ast^2 (H_1(0) +2H_0(0)) - 2x_\ast^2 \WH_1(0)\WH_0(0) + 3 \\
&\quad+ 2x_\ast^2 \WH_0(0) - 3x_\ast^2 H_0(0) - 6x_\ast^2 \WH_0^2(0)+ 2x_\ast^2 \WH_0^2(0) \\
&= - 4x_\ast^2 \WH_0(0)\WH_1(0) + 2 x_\ast - 6 = - 4 (x_\ast-2) + 2 x_\ast - 6 \\
&= -2 (x_\ast -1). 
\end{split}
\] }
The Jacobian at $(\k,\RH_1,\WH_1)= (0,-\frac{1}{x_\ast}, 1-\frac{2}{x_\ast})$ is 
\be
\begin{split}
\left| \frac{\partial \mathcal F}{\partial [\RH_1,\WH_1]} \right|_{(\k,\RH_1,\WH_1)= (0,-\frac{1}{x_\ast}, 1-\frac{2}{x_\ast})}
&= 4 (x_\ast -2) (x_\ast -1) >0 \quad \text{if } \ x_\ast >2
\end{split}
\ee
and thus, $\frac{\partial \mathcal F}{\partial [\RH_1,\WH_1]} $ is invertible in sufficiently small neighborhood of $(0,-\frac{1}{x_\ast}, 1-\frac{2}{x_\ast})$ for any fixed $x_\ast>2$. Now by the implicit function theorem, we deduce that there exists an open interval $(-l_0,l_0)$ of $\k=0$ and unique continuously differential functions $g_1:(-l_0,l_0) \rightarrow (r_1,r_2)$, $g_2:(-l_0,l_0) \rightarrow (w_1,w_2)$ such that $g_1(0)=-\frac{1}{x_\ast}$, $g_2(0)= 1-\frac{2}{x_\ast}$ and $\mathcal F(\k, g_1(\k), g_2(\k)) =0$ for all $\k\in (-l_0,l_0)$.
\end{proof}

It is evident that $\RH_1$ and $\WH_1$ become degenerate at $x_\ast=2$. In particular, when $\k=0$, the corresponding $\RH_1$ and $\WH_1$ for the LP-type and the Hunter-type solutions coincide precisely at $x_\ast=2$, while they are well-defined for further values of $x_\ast$ below 2, see  \cite{GHJ2021}. Interestingly this feature does not persist for $\k>0$.  In fact, we will show that  $\RH_1$ and $\WH_1$ cease to exist as real numbers before $x_\ast$ reaches $2$ from 
above. %This is related to the ``band structure" in the space of all smooth solutions pointed out in~\cite{OP1990}.  
To this end, we first derive the algebraic equation satisfied by $\WH_1$. 

%%%%%%%%%%%%%%%%%%%%%%%%%%%%
%%%%%%%%%%%%%%%%%%%%%%%%%%%%

\begin{lemma}[The cubic equation for $\WH_1$] \label{L:CUBIC}
The Taylor coefficient $\WH_1$ is a root of the cubic polynomial
\begin{align}\label{W_eq}
P(X) := X^3 + \frac{a_2}{a_3} X^2+ \frac{a_1}{a_3} X + \frac{a_0}{a_3},
\end{align}
where $a_j$, $j=0,1,2,3$ are real-valued continuous functions of $\WH_0$ given in~\eqref{E:ATHREEDEF}--\eqref{E:AZERODEF} below. Moreover,
there exists a constant $c>0$ and $0<\k_0\ll1$ such that for all $0<\k\ll\k_0$ and any 
$\frac13<\WH_0<1$, we have $c<a_3<\frac1c$, $c<\frac{a_0}{\k(3\WH_0-1)}<\frac1c$. 
In particular any root of $X\mapsto P(X)$ is not zero for $0<\k\ll\k_0$. 
\end{lemma}

%%%%%%%%%%%%%%%%%%%%%%%%%%%%

\begin{proof}
%\subsection{$\RH_1$ and $\WH_1$}
%
%Recall 
%\be
%R_0=\WH_0, \quad R_0^{-\e}=x_\ast^2 H_0 \label{sonic}
%\ee
%and 
%\begin{align}
%H_0 &= (1+3k)\WH_0^2   + \frac{4k}{k+1} \WH_0+ \frac{k^2-k}{(k+1)^2}\label{H00}\\
%H_1 &= 2\left[ (1+k)\WH_0 + \tfrac{2k}{1+k}\right] \WH_1 + 4k \WH_0 \RH_1\label{H11}\\
%(R^{-\e })_1 &= - \e R_0^{-\e -1} \RH_1= - \e x_\ast^2 \tfrac{H_0}{\WH_0} \RH_1 \label{R11}
%\end{align}
%
%\subsubsection{Derivation of $\WH_1$ equation}

%$\RH_1$ and $\WH_1$ satisfy 
%\be\label{R1}
%\begin{split}
%(R^{-\e })_1 \RH_1 &- x_\ast^2 H_1 \RH_1 - 2x_\ast^2 H_0 \RH_1+ 2x_\ast^2 (1-k) R_0(\WH_0 + \tfrac{k}{1+k})(\RH_1 -\WH_1) =0 
%\end{split}
%\ee
%and
%\be\label{W1}
%\begin{split}
%&(R^{-\e })_1 \WH_1 - x_\ast^2 (\WH_1 H_1+ 2\WH_1 H_0) 
%-\tfrac{1}{1+k} [ (R^{-\e })_1- R_0^{-\e }] \\
%&+ 3 [ \WH_1R_0^{-\e } + \WH_0 (R^{-\e })_1 - \WH_0 R_0^{-\e }] + \tfrac{x_\ast^2}{1+k} [H_1+H_0] \\
%&- 3x_\ast^2 [\WH_1 H_0 + \WH_0 H_1 + \WH_0 H_0]
%- 2x_\ast^2 (1+k) \WH_0(\WH_0 + \tfrac{k}{1+k})(\RH_1 -\WH_1) =0 
%\end{split}
%\ee 
Note that \eqref{W1ZERO} can be rearranged as 
\be\label{R22}
\begin{split}
-x_\ast^2H_1 [\WH_1+ 3\WH_0 -1] - x_\ast^2 H_0 [\e\tfrac{\RH_1}{\WH_0} +2]  [\WH_1+ 3\WH_0 -1] \\
- 2x_\ast^2 (1+\k) \WH_0(\WH_0 + \k)(\RH_1 -\WH_1) =0 . 
\end{split}
\ee 
Rewrite \eqref{R1ZERO} as 
\be\label{W22}
\begin{split}
& \left[{(1-\k) \WH_0(\WH_0+\k) -H_0 - [ \tfrac{\e}{2} \tfrac{H_0}{\WH_0} + 2\k\WH_0]\RH_1}\right]\RH_1\\
&\quad-\left[ {(1-\k)\WH_0(\WH_0+\k) +[(1+\k)\WH_0+2\k]\RH_1}
\right]\WH_1= 0, 
\end{split}
\ee
%\be
%\begin{split}
%\WH_1= \RH_1 \frac{(1-k) \WH_0(\WH_0+\frac{k}{1+k}) -H_0 - [\frac{k}{1-k}\frac{H_0}{\WH_0} + 2k\WH_0]\RH_1}{(1-k)\WH_0(\WH_0+\frac{k}{1+k}) +[(1+k)\WH_0+\frac{2k}{1+k}]\RH_1}
%\end{split}
%\ee
and \eqref{R22} as 
\be\label{R33}
\begin{split}
\RH_1= -  \frac{ [(1+\k)\WH_0+2\k ]\WH_1^2 + [(3+5\k)\WH_0^2- (1-8\k+\k^2)\WH_0- (3\k-\k^2)] \WH_1 + (3\WH_0-1) H_0}
{(1+\k)\WH_0(\WH_0+\k) +[\frac{\e}{2}\frac{H_0}{\WH_0} + 2\k \WH_0] (3\WH_0-1) +[\frac{\e}{2}\frac{H_0}{\WH_0} + 2\k \WH_0]\WH_1}. 
\end{split}
\ee
We would like to derive an algebraic equation for $\WH_1$. To this end, write \eqref{R33} and \eqref{W22} as 
\be\label{E:RONEEXPRESSION}
\RH_1 = - \frac{d\WH_1^2 + e\WH_1 + cH_0}{h+ b\WH_1},
\ee
where
\begin{align}
&d:=(1+\k)\WH_0+2\k, \ \ e:=(3+5\k)\WH_0^2- (1-8\k+\k^2)\WH_0- (3\k-\k^2), \ \ c:=3\WH_0-1, \label{E:ABCDEF1}\\
&h:= (1+\k)\WH_0(\WH_0+\k) +[\frac{\e}{2}\frac{H_0}{\WH_0} + 2\k \WH_0] (3\WH_0-1) , \ \  b:= \frac{\e}{2}\frac{H_0}{\WH_0} + 2\k \WH_0.
\end{align}
We observe that we may write
$h=(1+\k)a+ bc$, where
\begin{align}
a:=\WH_0(\WH_0+\k). \label{E:ABCDEF3}
\end{align} 
If we let
\be\label{E:VERYLITTLEFDEF}
f:=(1-\k)a -H_0,
\ee 
then~\eqref{W22} can be rewritten in the form
\be\label{E:ABSTRACTFORM}
[f- b\RH_1]\RH_1 - [(1-\k)a + d \RH_1]\WH_1=0,
\ee
with $f,.b,a, d$ as above.
Now replace $\RH_1$  in \eqref{E:ABSTRACTFORM} using the relation~ \eqref{E:RONEEXPRESSION}: 
\begin{align*}
 -\left[ f+ b \frac{d\WH_1^2 + e\WH_1 + cH_0}{h+ b\WH_1} \right]\frac{d\WH_1^2 + e\WH_1 + cH_0}{h+ b\WH_1} 
- \left[(1-\k)a - d \frac{d\WH_1^2 + e\WH_1 + cH_0}{h+ b\WH_1} \right] \WH_1=0. 
\end{align*}
Multiply $(h+ b\WH_1)^2$: 
\begin{align*}
& - \left[ f(h+ b\WH_1)+ b (d\WH_1^2 + e\WH_1 + cH_0) \right](d\WH_1^2 + e\WH_1 + cH_0) \\
&- \left[(1-\k)a (h+ b\WH_1)- d(d\WH_1^2 + e\WH_1 + cH_0)  \right](h+ b\WH_1) \WH_1=0. 
\end{align*}
Note that the highest order term $\WH_1^4$ is not present. Rearrange similar terms to conclude the identity
\be\label{E:WCUBIC}
a_3 \WH_1^3 + a_2 \WH_1^2 + a_1\WH_1 + a_0 = 0, 
\ee
where
\begin{align}
a_3&:=  d ( d h - be )- b (fd +(1-\k) ab), \label{E:ATHREEDEF}\\
a_2&:=-\Big[bcdH_0 -e  ( d h - be ) +h (fd+(1-\k)ab) + b ( fe + (1-\k) a h)  \Big],\\
a_1&:=-\Big[ bceH_0 -cH_0   ( d h - be ) + bcf H_0  
+  h  ( fe + (1-\k) a h ) \Big],\\
a_0&:=- cH_0 ( bcH_0 + fh). \label{E:AZERODEF}
\end{align}
%We claim that  $a_0>0$ if $\tfrac{1}{3(1+k)}< \WH_0 <1 \text{ and } 0<k\ll 1$. 
%Since clearly $a_0=- cH_0 ( bcH_0 + fh)$
% \begin{align}
%a_0= - bc^2H_0^2 - cfh H_0 = - cH_0 ( bcH_0 + fh)
%\end{align}
By~\eqref{E:AZERODEF}, the sign of $a_0$ is the same as the sign of $bcH_0 + fh$.
We now use
\begin{align*}
 &bcH_0 + fh = \left[\frac{\e}{2}\frac{H_0}{\WH_0} + 2\k \WH_0 \right]\left[3\WH_0-1 \right]H_0\\
 & + \left[(1-\k) \WH_0(\WH_0+\k) -H_0 \right]\left[(1+\k)\WH_0(\WH_0+\k) +[\tfrac{\e}{2}\tfrac{H_0}{\WH_0} + 2\k \WH_0] (3\WH_0-1) \right]\\
 &=  \left[(1-\k) \WH_0(\WH_0+\k) -H_0 \right] (1+\k)\WH_0(\WH_0+\k) \\
 &\quad+ (1-\k) \WH_0(\WH_0+\k) [\tfrac{\e}{2}\tfrac{H_0}{\WH_0} + 2\k \WH_0] (3\WH_0-1)\\
 &= \WH_0(\WH_0+\k)\underbrace{ \Big\{(1+\k)\underbrace{\left[(1-\k) \WH_0(\WH_0+\k) -H_0 \right]}_{(I)}  + \k\underbrace{[\tfrac{H_0}{\WH_0} + 2(1-\k)\WH_0]}_{(II)} (3\WH_0-1) \Big\}}_{(III)},
\end{align*}
where we have used $\frac{\e}{2}=\frac{\k}{1-\k}$ in the last line.
Now 
\begin{align*}
(I)&= (1-\k) \WH_0(\WH_0+\k) - (1+3\k)\WH_0^2 - 4\k \WH_0 +\k-\k^2\\
&=\k \left[ -4 \WH_0^2 - (3+\k) \WH_0 +1-\k\right]
\end{align*}
and 
\begin{align*}
(II)&= (1+3\k)\WH_0 + 4\k - \tfrac{\k-\k^2}{\WH_0} + 2(1-\k) \WH_0\\
&= (3+\k)\WH_0 + 4\k - \tfrac{\k-\k^2}{\WH_0}.  
\end{align*}
Hence 
\begin{align*}
(III)&=\k \Big\{ -4(1+\k)\WH_0^2 - (3+\k)(1+\k)\WH_0 + 1-\k^2 + [  (3+\k)\WH_0+ 4\k - \tfrac{\k-\k^2}{\WH_0} ](3\WH_0-1)   \Big\}\\
&= \k \Big\{ (5-\k)\WH_0^2 - (6-7\k+\k^2) \WH_0 + (1-7\k+2\k^2) + \tfrac{\k-\k^2}{\WH_0}\Big\}\\
&= \k \Big\{ 5\WH_0^2 - 6\WH_0 + 1+ O(\k) \Big\}.  \label{E:AZEROASYMP}
\end{align*}
Since $5\WH_0^2 - 6\WH_0 + 1= (5\WH_0-1) (\WH_0-1) <0$ for $\frac15 <\WH_0<1$, 
\begin{align}
a_0 &= - cH_0 ( bcH_0 + fh) \notag \\
&=  - \k cH_0 \WH_0(\WH_0+\k)\Big\{ 5\WH_0^2 - 6\WH_0 + 1+ O(\k) \Big\} \notag\\
&>0  \ \text{ for } \ \frac13< \WH_0 <1 \text{ and } 0<\k\ll 1. 
\end{align}
%In particular, this strict sign condition for $a_0$ shows that $W$ can't vanish if $k>0$ (cf. \eqref{W_eq}). 
From~\eqref{E:ABCDEF1}--\eqref{E:ABCDEF3} it is easy to see that 
\begin{align}
a&= \WH_0^2 + O(\k), \ \ b= O(\k), \ \ c= 3\WH_0-1, \ \ d= \WH_0 +O(\k), \\
e&= 3\WH_0^2-\WH_0 +O(\k), \ \  f= (1-\k) \WH_0(\WH_0+\k) -H_0= O(\k) , \\
h&= (1+\k)\WH_0(\WH_0+\k) +[\tfrac{\k}{1-\k}\tfrac{H_0}{\WH_0} + 2\k \WH_0] (3\WH_0-1)= \WH_0^2 +O(\k).
\end{align}
This then implies
\begin{align}\label{E:AIASYMP}
a_3=  \WH_0^4 +O(\k), \ \ a_2=\WH_0^4(3\WH_0-1) +O(\k), \ \ a_1=   \WH_0^5(2\WH_0-1)+O(\k), \ \  \frac{a_0}{\k(3\WH_0-1)} \sim 1%a_0 = O(\k), 
\end{align}
which shows the uniform positivity and boundedness of $a_3$ for sufficiently small $\k$. Upon dividing~\eqref{E:WCUBIC} by $a_3$ we finally conclude the proof of the lemma. 
\end{proof}

%%%%%%%%%%%%%%%%%%%%%%%%%%%%%%
%%%%%%%%%%%%%%%%%%%%%%%%%%%%%%

%Therefore, we obtain the cubic equation for $\WH_1$: 
%\begin{align}\label{W_eq}
%a_3\WH_1^3 + a_2 \WH_1^2+ a_1 \WH_1 + a_0 =0
%\end{align}

\begin{remark}
%Note that
%\begin{align}
%a&= \WH_0(\WH_0+\tfrac{k}{1+k})= \WH_0^2 + O(k)\\
%b&=\tfrac{k}{1-k}\tfrac{H_0}{\WH_0} + 2k\WH_0= O(k)\\
%c&= 3\WH_0-\tfrac{1}{1+k}\\
%d&= (1+k)\WH_0+\tfrac{2k}{1+k}= \WH_0 +O(k)\\
%e&=  (3+5k)\WH_0^2- \tfrac{1-8k+k^2}{1+k}\WH_0- \tfrac{3k-k^2}{(1+k)^2}= 3\WH_0^2-\WH_0 +O(k)\\
%f&= (1-k) \WH_0(\WH_0+\tfrac{k}{1+k}) -H_0= O(k)\\
%h&= (1+k)\WH_0(\WH_0+\tfrac{k}{1+k}) +[\tfrac{k}{1-k}\tfrac{H_0}{\WH_0} + 2k\WH_0] (3\WH_0-\tfrac{1}{1+k})= \WH_0^2 +O(k)
%\end{align}
%and that
%\begin{align}
%a_3=  \WH_0^4 +O(k) \\
%a_2=\WH_0^4(3\WH_0-1) +O(k)\\
%a_1=   \WH_0^5(2\WH_0-1)+O(k)\\
%a_0 = O(k)
%\end{align}
In the formal Newtonian limit $\k=0$, the cubic equation $P(X)=0$
%In particular, when $k=0$, we recover the EP problem: 
reduces to
\be
X^3 + (3\WH_0-1) X^2 + (2\WH_0-1)\WH_0 X = X (X +\WH_0)(X+ 2\WH_0 -1 ) =0.
\ee
The root $X=0$ corresponds to the Hunter-type, $X=1-2\WH_0$ is of the LP-type, and $X=-\WH_0$ is the Newtonian ghost solution, see~\cite{GHJ2021}. 
\end{remark}

%\begin{remark}\label{R:AZERO}
%A simple consequence of~\eqref{E:AZEROASYMP} is the existence of constants $\alpha_1,\alpha_2>0$ such that or all $\k\in(0,\k_0]$. 
%\end{remark}
%\subsubsection{On $a_0$}

%%%%%%%%%%%%%%%%%%%%%%%%%%%%%%%
%%%%%%%%%%%%%%%%%%%%%%%%%%%%%%%
%%%%%%%%%%%%%%%%%%%%%%%%%%%%%%%

\subsection{Relativistic Larson-Penston-type solutions}

%%%%%%%%%%%%%%%%%%%%%%%%%%%%%%%
%%%%%%%%%%%%%%%%%%%%%%%%%%%%%%%
%%%%%%%%%%%%%%%%%%%%%%%%%%%%%%%

By Lemma~\ref{L:CUBIC} we know that there are in general three complex roots of~\eqref{W_eq} giving possible values for $\WH_1$.
One of those roots is a ``ghost root" and is discarded as unphysical solution of~\eqref{W_eq}. More precisely, in the Newtonian limit $\k\to0$ 
such a spurious root corresponds to the value $\WH_1(0)=-\frac{1}{x_\ast}$. Our goal is to first mod out such a solution by using the implicit function theorem to
construct a curve $\k\to\WH^{\text{gh}}_1(\k)$ of spurious solutions agreeing with $-\frac{1}{x_\ast}$ when $\k=0$.

%%%%%%%%%%%%%%%%%%%%%%%%%%%%%%%
%%%%%%%%%%%%%%%%%%%%%%%%%%%%%%%

\begin{lemma}[Ghost root]
The cubic polynomial $P$ introduced in Lemma~\ref{L:CUBIC} can be factorised in the form
\begin{align}\label{E:PFACTOR}
P(X) = (X-\WH_1^{\text{gh}}(\k)) Q(X),
\end{align}
where $\WH_1^{\text{gh}}(\k)$ is a real valued function of $\WH_0(\k)$ and
$Q$ is a quadratic polynomial given by
\begin{align}\label{E:QPOLDEF}
Q(X) = X^2 + \left[ \frac{a_2}{a_3} +\WH_1^{\text{gh}} \right] X -  \frac{a_0}{a_3 \WH_1^{\text{gh}}} .
\end{align}
%Moreover, the roots of $Q$ are real for all $\k\in(0,\k_0]$ and $x_\ast\in[\xmin,\xmax]$.
\end{lemma}

%%%%%%%%%%%%%%%%%%%%%%%%%%%%%%%
%%%%%%%%%%%%%%%%%%%%%%%%%%%%%%%

\begin{proof}
%We first construct a ghost root $\WH_1^{\text{gh}}=\WH_1^{\text{gh}}(k)$ satisfying \eqref{W_eq} with $\WH_1^{\text{gh}}(0)= -\frac{1}{x_\ast}$. This can be easily done by the IFT. 
%Let 
%\be
%f_{\text{gh}}(\k, W)=W^3 + \frac{a_2}{a_3} W^2+ \frac{a_1}{a_3} W + \frac{a_0}{a_3}
%\ee
We differentiate~\eqref{W_eq} to obtain
%Then $f_{\text{gh}}$ is smooth and since 
\be
\begin{split}
\frac{\pa P}{\partial X}\Big|_{(\k,X)=(0,-\frac{1}{x_\ast})}&=  \left(3X^2 + 2\frac{a_2}{a_3} X+ \frac{a_1}{a_3}\right) \Big|_{(\k,X)=(0,-\frac{1}{x_\ast})} \\
&=  3\frac{1}{x_\ast^2} - 2\frac{3-x_\ast}{x_\ast}\frac{1}{x_\ast} + \frac{2-x_\ast}{x_\ast^2}= \frac{x_\ast-1}{x_\ast^2} >0.
\end{split}
\ee
The  implicit function theorem implies the existence of a unique $\k$-parametrised curve 
\be\label{E:GHOSTCURVE}
\WH_1^{\text{gh}}(\k)=  -\frac{1}{x_\ast} + O(\k)
\ee 
satisfying $P(\WH_1^{\text{gh}}(\k))=0$. 
Identity~\eqref{E:PFACTOR} can now be checked directly, keeping in mind~\eqref{E:QPOLDEF}.

The discriminant of $Q$ is given by 
\[
\Delta_Q= ( \frac{a_2}{a_3} +\WH_1^{\text{gh}} )^2 +  \frac{4a_0}{a_3 \WH_1^{\text{gh}}}.
\]
\end{proof}
%We may rewrite \eqref{W_eq} as
%\be
%0=\WH_1^3 + \frac{a_2}{a_3} \WH_1^2+ \frac{a_1}{a_3} \WH_1 + \frac{a_0}{a_3} = (\WH_1- \WH_1^{\text{gh}}) \left(\WH_1^2 + \left[ \frac{a_2}{a_3} +\WH_1^{\text{gh}} \right] \WH_1 -  \frac{a_0}{a_3 \WH_1^{\text{gh}}} \right)
%\ee

%%%%%%%%%%%%%%%%%%%%%%
%%%%%%%%%%%%%%%%%%%%%%

%\begin{definition}[Solutions of RLP-type]\label{D:RLPTYPEDEF}
%Let $0<\k_0\ll1$ be given by Lemma~\ref{L:CUBIC} and let $x_\ast\in[\xmin,\xmax]$. We say that the 
%sequence $(\RH_N,\WH_N)_{N\in\mathbb N}$ is of \underline{relativistic Larson-Penston (RLP)-type} if for all $\k\in(0,\k_0]$
%\begin{align}
%\RH_0 &= \WH_0,
%\end{align}
%the coefficients $(\RH_N,\WH_N)_{N\in\mathbb N}$ satisfy the recursive relations~\eqref{RN+1}--\eqref{WN+1},
%$\WH_1$ is the root of the quadratic polynomial $Q$~\eqref{E:QPOLDEF} given by
%\begin{align}\label{E:WONERLP}
%\WH_1 =\WH_1^{\text{{\em RLP}}} = \frac{- \left[ \frac{a_2}{a_3} +\WH_1^{\text{gh}} \right]  + \sqrt{( \frac{a_2}{a_3} +\WH_1^{\text{gh}} )^2 +  \frac{4a_0}{a_3 \WH_1^{\text{gh}}}} }{2},
%\end{align}
%and $\RH_1$ is given as a function of $\WH_0$ and $\WH_1$ via~\eqref{R33}.
%\end{definition}
%%%%%%%%%%%%%%%%%%%%%%%
%%%%%%%%%%%%%%%%%%%%%%%

%%%%%%%%%%%%%%%%%%%%%%
%%%%%%%%%%%%%%%%%%%%%%

\begin{lemma}[Definition of $\xc(\k)$]\label{L:XMINDEF}
There exists a continuous curve $(0,\k_0]\ni \k\mapsto\xc(\k)\in(\frac12,\frac72)$ such that 
\begin{align}
\Delta_Q(\xc(\k)) & =0, \label{E:XMIN1}\\
\Delta_Q(x_\ast)& >0, \ \ x_\ast\in(\xc(\k),\xmax(\k)], \label{E:XMIN2}\\
\xc(\k) & = 2+ O(\sqrt\k)>2. \label{E:XMIN3}
\end{align}
Moreover, for any $x_\ast\in(\xc(\k),\frac72]$ there exists a constant $C_\ast=C_\ast(x_\ast)>0$ such that for all $\k\in(0,\k_0]$
\begin{align}\label{E:XMINRATE1}
\lv \WH_1^{\text{{\em RLP}}} (\k; x_\ast) - \frac{x_\ast-2}{x_\ast}  \rv\le C_\ast \k.
\end{align}
When $\xs=\xc(\k)$ the rate of convergence changes and there exists a constant $\bar C>0$ such that 
\begin{align}\label{E:XMINRATE2}
\lv \WH_1^{\text{{\em RLP}}} (\k; \xc(\k)) - \frac{\xc(\k)-2}{\xc(\k)}  \rv\le \bar C \sqrt{\k}.
\end{align}
\end{lemma}

%%%%%%%%%%%%%%%%%%%%%%

\begin{proof}
{\em Step 1. Existence of $\xc$.}
From the asymptotic formulas~\eqref{E:AIASYMP} and~\eqref{E:GHOSTCURVE} it is clear that for all $\k\in(0,\k_0]$ we have
\be\label{E:GHOSTASYMP}
- \left[ \frac{a_2}{a_3} +\WH_1^{\text{gh}} \right] = \frac{x_\ast-2}{x_\ast}+O(\k), 
\ee
and therefore 
\[
\Delta_Q(\k,x_\ast) = \left(\frac{x_\ast-2}{x_\ast}+O(\k)\right)^2 +  \frac{4a_0}{a_3 \WH_1^{\text{gh}}} .
\]
However, by Lemma~\ref{L:CUBIC} we know that %$0<-\frac{\frac{4a_0}{a_3 \WH_1^{\text{gh}}}}{\k (3\WH_0-1)}=O(1)$ 
$ \frac{4a_0}{a_3 \WH_1^{\text{gh}}} \le - \alpha(3\WH_0-1) \k$ for some positive constant $\alpha>0$. Choosing $x_\ast<2$ such that $2-\xs=O(\k)$, we see that 
$\Delta_Q(\k,x_\ast)<0$. On the other hand, for any $\xs\ge2+\delta$ for some small, but fixed $\delta>0$, we see that $\Delta_Q(\k,\xs)>0$.
Therefore, by the intermediate value property, there exists an $\xs\in (2-O(\k),2+\delta)$ such that $\Delta_Q(\xs)=0$.
Let $\xc$ be the largest such $\xs$ - it is clear from the construction that the properties~\eqref{E:XMIN1}--\eqref{E:XMIN2} are satisfied.

\noindent
{\em Step 2. Asymptotic behaviour of $\xc(\k)$.}
From~\eqref{E:GHOSTASYMP}, the identity
\be
- \left[ \frac{a_2}{a_3} +\WH_1^{\text{gh}} \right] \Big\vert_{\xs=\xc(\k)} =\sqrt{ -  \frac{4a_0}{a_3 \WH_1^{\text{gh}}}  }\Big\vert_{\xs=\xc(\k)},
\ee
and Lemma~\ref{L:CUBIC} we easily conclude~\eqref{E:XMIN3}.

\noindent
{\em Step 3. Continuity properties of the map $\k\mapsto\WH_1^{\text{{\em RLP}}}(\k,\xs)$.}
Fix an $\xs\in(\xc,\xmax]$. Then 
\begin{align}
\WH_1^{\text{RLP}} (\k; x_\ast) - \frac{x_\ast-2}{x_\ast} &= \frac{- \left[ \frac{a_2}{a_3} +\WH_1^{\text{gh}} \right]  
+ \sqrt{( \frac{a_2}{a_3} +\WH_1^{\text{gh}} )^2 +  \frac{4a_0}{a_3 \WH_1^{\text{gh}}}} }{2} - \frac{x_\ast-2}{x_\ast} \notag \\
&=\frac{ - \left[ \frac{a_2}{a_3} +\WH_1^{\text{gh}} \right] - \frac{x_\ast-2}{x_\ast}}{2} + \frac{\sqrt{( \frac{a_2}{a_3} +\WH_1^{\text{gh}} )^2 +  \frac{4a_0}{a_3 \WH_1^{\text{gh}}}} -  \frac{x_\ast-2}{x_\ast} }{2} \notag \\
&=: (i) + (ii). 
\end{align}
By~\eqref{E:GHOSTASYMP} we have $(i)=O(\k)$. For $(ii)$, 
\begin{align}
(ii)= \frac{( \frac{a_2}{a_3} +\WH_1^{\text{gh}} )^2 - (\frac{x_\ast-2}{x_\ast})^2 
+ \frac{4a_0}{a_3 \WH_1^{\text{gh}} }}{ 2 \left[  \sqrt{( \frac{a_2}{a_3} +\WH_1^{\text{gh}} )^2 +  \frac{4a_0}{a_3 \WH_1^{\text{gh}}}} +  \frac{x_\ast-2}{x_\ast}\right] } 
= \frac{O(\k)}{ \sqrt{(x_\ast-2)^2 + O(\k) } +(x_\ast-2) } \ \ \text{ as $\k \to 0^+$},
\end{align}
by~\eqref{E:AIASYMP} and~\eqref{E:GHOSTASYMP}.
Therefore, the bound~\eqref{E:XMINRATE1} follows.
% there exists a constant $C_\ast=C_\ast(\xs)$ such that $\lv \WH_1^{\text{RLP}} (\k; x_\ast) - \frac{x_\ast-2}{x_\ast}  \rv\le C_\ast \k$ for all $\k\in(0,\k_0]$.

If we now let $\xs=\xc(\k)$, since by definition $\Delta_Q(\xc(\k))=0$, we have
\begin{align}
\WH_1^{\text{RLP}} (\k; \xc(\k)) - \frac{\xc(\k)-2}{\xc(\k)}  & = -\frac{ \frac{a_2}{a_3} +\WH_1^{\text{gh}}}{2} - \frac{\xc(\k)-2}{\xc(\k)} \notag \\
& = - \sqrt{-\frac{a_0}{a_3\WH_1^{\text{gh}}}}- \frac{\xc(\k)-2}{\xc(\k)} \notag\\
& = O_{\k\to0^+}(\sqrt \k),
\end{align}
where we have used~\eqref{E:XMIN3} and Lemma~\ref{L:CUBIC}. This proves~\eqref{E:XMINRATE2}.
\end{proof}

%%%%%%%%%%%%%%%%%%%%%%
%%%%%%%%%%%%%%%%%%%%%%

Of special importance in our analysis is the Friedmann solution, which has the property that for any $\k>0$ $\WH_0=\RH_0=\frac13$. By Lemma~\ref{R0W0}
there exists a continuously differentiable curve $\k\to\xmax(\k)$ defined through the property
\be\label{E:XMAXDEF}
\WH_0(\k;\xmax(\k)) = \frac13, \ \ \k\in[0,\k_0].
\ee
Since $\WH_0(0;3)=\frac13$, we conclude from Lemma~\ref{R0W0} that 
\be\label{E:XMAXPROP}
\xmax(\k)<3 \ \ \text{ and } \ 3-\xmax(\k)=O(\k), \ \   \k\in[0,\k_0].
\ee

%%%%%%%%%%%%%%%%%%%%%%%%%
%%%%%%%%%%%%%%%%%%%%%%%%%

With above preparations in place, we are now ready to define what we mean by a solution of the {\em relativistic Larson-Penston type}.

\begin{definition}[Solutions of RLP-type]\label{D:RLPTYPEDEF}
Let $0<\k_0\ll1$ be given by Lemma~\ref{L:CUBIC} and let  $x_\ast\in[\xc,\xmax]$. We say that the 
sequence $(\RH_N,\WH_N)_{N\in\mathbb N}$ is of \underline{relativistic Larson-Penston (RLP)-type} if for all $\k\in(0,\k_0]$
\begin{align}
\RH_0 &= \WH_0,
\end{align}
the coefficients $(\RH_N,\WH_N)_{N\in\mathbb N}$ satisfy the recursive relations~\eqref{RN+1}--\eqref{WN+1},
$\WH_1$ is the root of the quadratic polynomial $Q$~\eqref{E:QPOLDEF} given by
\begin{align}\label{E:WONERLP}
\WH_1 =\WH_1^{\text{{\em RLP}}} = \frac{- \left[ \frac{a_2}{a_3} +\WH_1^{\text{gh}} \right]  + \sqrt{( \frac{a_2}{a_3} +\WH_1^{\text{gh}} )^2 +  \frac{4a_0}{a_3 \WH_1^{\text{gh}}}} }{2},
\end{align}
and $\RH_1$ is given as a function of $\WH_0$ and $\WH_1$ via~\eqref{R33}.
\end{definition}

%%%%%%%%%%%%%%%%%%%%%%
%%%%%%%%%%%%%%%%%%%%%%

%%%%%%%%%%%%%%%%%%%%%%%%%
%%%%%%%%%%%%%%%%%%%%%%%%%

\begin{remark}
For all $\k\in(0,\k_0]$ it is clear from the proof of the lemma that the following  bound holds in the %sonic 
interval $\xs\in[\xc(\k),\xmax(\k)]$: 
\begin{align}
\WH_1^{\text{RLP}} (\k; x_\ast) - \frac{x_\ast-2}{x_\ast}  = \frac{O(\k)}{ \sqrt{(x_\ast-2)^2 + O(\k) } +(x_\ast-2) }.
\end{align}
\end{remark}

%%%%%%%%%%%%%%%%%%%%%%%%%
%%%%%%%%%%%%%%%%%%%%%%%%%
%%%%%%%%%%%%%%%%%%%%%%%%%

\subsection{High-order Taylor coefficients}\label{SS:HIGHER}

%%%%%%%%%%%%%%%%%%%%%%%%%
%%%%%%%%%%%%%%%%%%%%%%%%%
%%%%%%%%%%%%%%%%%%%%%%%%%

We now consider $\eqref{RN+1}_{N\geq2}$ and $\eqref{WN+1}_{N\geq 2}$. 

\begin{lemma}\label{Lem:N} Let $N\geq 2$. Then the following holds 
\be\label{E:MATRIXEQN}
\mathcal A_N (\WH_0, \WH_1, \RH_1)  \left(\begin{array}{c}
\RH_N \\
\WH_N 
\end{array}
 \right) = 
  \left(\begin{array}{c}
\mathcal S_N \\
\mathcal V_N
\end{array}
 \right) . 
\ee Here 
\be
\mathcal A_N (\WH_0, \WH_1, \RH_1)  = 
\left(\begin{array}{cc}
A_{11} \  & A_{12}  \\
A_{21} \  & A_{22}  
\end{array}
 \right) 
\ee where 
\begin{align*}
A_{11}&=-2x_\ast^2 H_0 N  + 2x_\ast^2 (1-\k) \RH_0 (\WH_0+\k) -x_\ast^2 H_1 N - \e \RH_0^{-\e-1}\RH_1(N+1) - 4\k x_\ast^2 \WH_0\RH_1, \\
A_{12}&= -x_\ast^2 \RH_1 (2(1-\k) \WH_0 + 4\k + 4\k\RH_0) - 2x_\ast^2 (1-\k) \RH_0 (\WH_0+\k), \\
A_{21}&= - 2x_\ast^2 (1+\k) \WH_0 (\WH_0+\k) - (\e \RH_0^{-\e-1} +4\k \WH_0x_\ast^2 ) (\WH_1 +3\WH_0 -1), \\
%&\quad - 4k\WH_0x_\ast^2(\WH_1 +3\WH_0 -\tfrac{1}{1+k})\\
A_{22}&=  - 2x_\ast^2 H_0 N - 3x_\ast^2 H_0 + 3\RH_0^{-\e} + 2x_\ast^2 (1+\k) \WH_0 (\WH_0+\k)\\
&\quad  -x_\ast^2 (  \WH_1  + 3\WH_0-1)  (2(1-\k) \WH_0 + 4\k + 4\k \RH_0) -x_\ast^2 H_1N - 
 \e \RH_0^{-\e-1} \RH_1N , 
\end{align*}
and 
\begin{align*}
\mathcal S_N= \mathcal S_N[\RH_0,\WH_0; R_1,\WH_1;\cdots, \RH_{N-1}, \WH_{N-1}], \\
\mathcal V_N= \mathcal V_N[\RH_0,\WH_0; R_1,\WH_1;\cdots, \RH_{N-1}, \WH_{N-1}], 
\end{align*} are given in \eqref{SN} and \eqref{VN}.
\end{lemma}

\begin{proof}  We first observe that there are no terms involving  $(\RH_{N+1}, \WH_{N+1})$ in \eqref{RN+1} and \eqref{WN+1} due to the sonic conditions \eqref{sonic} (cf.  Remark \ref{remark1} and see the cancelation below). To prove the lemma, we isolate all the coefficients in  $ (\RH^{-\e})_N$ and $H_N$ that contain  the top order coefficients $\RH_{N}, \WH_{N}$. For  $ (\RH^{-\e})_N$, from \eqref{faaR} we have 
\be\label{faaR1}
\begin{split}
 (\RH^{-\e})_N&= -\e \RH_0^{-\e-1}\RH_N \\
&\quad + \RH_0^{-\e} \sum_{m=2}^N \frac{1}{\RH_0^m}\sum_{\pi(N,m)} (- \e)_m \frac{1}{\lambda_1 ! \dots \lambda_N !} {\RH_1}^{\lambda_1}\dots 
{\RH_N}^{\lambda_N}
\end{split}
\ee
where we recall  $\lambda_j=0$ for $N-m+2\le j\le N$, in particular $\lambda_N=0$ and therefore there are no terms involving $\RH_N$ in the summation term.  
For $H_N$, from \eqref{Hl}, we have 
\be\label{HN1}
\begin{split}
H_N&= 2(1-\k)\WH_0 \WH_N + 4\k \WH_N + 4\k (\RH_0 \WH_N+ \RH_N \WH_0)\\
&\quad + (1-\k) \sum_{\substack{\ell +m =N\\ 1\leq m\leq N-1}} \WH_\ell \WH_m+  \sum_{\substack{\ell +m =N\\ 1\leq m\leq N-1}} 4\k \RH_\ell  \WH_m . 
\end{split}
\ee
Using \eqref{faaR1} and \eqref{HN1}, we can isolate all the coefficients in \eqref{RN+1} that contain contributions from $(R_{N}, W_{N})$ as follows: 
\be
\begin{split}
0=&%\sum_{l+m=N} (m+1) R_{m+1} (\RH^{-\e})_l  \\
\cancel{(N+1)\RH_{N+1} \RH_0^{-\e} }+ N \RH_N  (\RH^{-\e})_1 + \RH_1 (\RH^{-\e})_N \\
%&- x_\ast^2 \left( \sum_{l+m=N} (m+1) R_{m+1} H_l + 2\sum_{l+m=N-1} (m+1) R_{m+1}H_l +\sum_{l+m=N-2} (m+1) R_{m+1}H_l  \right)\\
& - x_\ast^2 \left( \cancel{(N+1) \RH_{N+1} H_0} + N \RH_N H_1+ \RH_1H_N + 2N \RH_N H_0 \right) \\
%&+ 2x_\ast^2(1-k)\left( ( R ( W + \tfrac{k}{k+1}) ( R- W) )_N +( R ( W + \tfrac{k}{k+1}) ( R- W) )_{N-1}  \right) =0 
& + 2x_\ast^2(1-\k) \RH_0 ( \WH_0 + \k) ( \RH_N- \WH_N) - \tilde{\mathcal S}_N \\
=&\Big[ -\e \RH_0^{-\e-1}\RH_1(N+1) -x_\ast^2 NH_1 - 4\k x_\ast^2 \WH_0 \RH_1 - 2x_\ast^2 N H_0 
\\
&\quad +2x_\ast^2(1-\k) \RH_0 ( \WH_0 + \k)  \Big] \RH_N \\
+& \Big[-x_\ast^2\RH_1 ( 2(1-\k) \WH_0+4\k + 4\k \RH_0 ) -2x_\ast^2(1-\k) \RH_0 ( \WH_0 + \k) \Big] \WH_N - \mathcal S_N   
\end{split}
\ee where
\be\label{SN}
\begin{split}
 \mathcal S_N =  -  \RH_1\RH_0^{-\e} \sum_{m=2}^N \frac{1}{\RH_0^m}\sum_{\pi(N,m)} (- \e)_m \frac{1}{\lambda_1 ! \dots \lambda_N !} {\RH_1}^{\lambda_1}\dots 
{\RH_N}^{\lambda_N} \\
+x_\ast^2 \RH_1 \Big[   (1-\k) \sum_{\substack{\ell +m =N\\ 1\leq m\leq N-1}} \WH_\ell \WH_m+  \sum_{\substack{\ell+m =N\\ 1\leq m\leq N-1}} 4\k \RH_\ell \WH_m\Big]
+\tilde{\mathcal S}_N, 
 \end{split}
\ee and 
\be\label{SN0}
\begin{split}
&\tilde{\mathcal S}_N = - \sum_{\substack{\ell +m=N \\ 1\leq m\leq N-2}} (m+1) \RH_{m+1} (\RH^{-\e})_\ell  \\
&+x_\ast^2 \left( \sum_{\substack{\ell+m=N\\ 1\leq m\leq  N-2}} (m+1) \RH_{m+1} H_\ell + 2\sum_{\substack{\ell+m=N-1 \\ m\leq N-2}} (m+1) \RH_{m+1}H_\ell +\sum_{\ell+m=N-2} (m+1) \RH_{m+1}H_\ell  \right) \\
&- 2x_\ast^2(1-\k)\left(  \sum_{\substack{\ell+m+n=N\\ 1\leq n\leq N-1}} \RH_\ell ( \WH + \k)_m ( \RH- \WH)_n  +( \RH ( \WH + \k) ( \RH- \WH) )_{N-1}  \right).
\end{split}
\ee 
Here we recall the definitions of $(\RH^{-\e})_\ell$ in \eqref{faaR} and $H_1$ in \eqref{R11.1} as well as the sonic conditions in \eqref{sonic}. Note that we have also used~\eqref{R11} above. 

Following the same procedure, we now  isolate all the coefficients in \eqref{WN+1} that contain contributions from $(\RH_{N}, \WH_{N})$. 
\be
\begin{split}
%&\sum_{l+m=N} (m+1) W_{m+1} (\RH^{-\e})_l  \\
0=& \cancel{(N+1) \WH_{N+1} \RH_0^{-\e}} + N\WH_{N} (\RH^{-\e})_1 + \WH_1 (\RH^{-\e})_N \\
%&- x_\ast^2 \left( \sum_{l+m=N} (m+1) W_{m+1} H_l + 2\sum_{l+m=N-1} (m+1) W_{m+1}H_l +\sum_{l+m=N-2} (m+1) W_{m+1}H_l  \right)\\
&-  x_\ast^2 \left( \cancel{(N+1) \WH_{N+1} H_0} + N\WH_N H_1 + \WH_1 H_N  + 2 N \WH_N H_0 \right)\\ 
%& - \frac{1}{1+k}\sum_{l+m=N} (\RH^{-\e})_l  (-1)^m + 3\sum_{l+m+n=N} W_n (\RH^{-\e})_l  (-1)^m    \\
& - (\RH^{-\e})_N  + 3(  \WH_N (\RH^{-\e})_0 +  \WH_0 (\RH^{-\e})_N)\\ 
%& + \frac{x_\ast^2}{1+k} \left(  \sum_{l+m=N} H_l (-1)^m +2 \sum_{l+m=N-1} H_l (-1)^m  + \sum_{l+m=N-2}H_l (-1)^m  \right)  \\
& + \xs^2  H_N  - 3x_\ast^2 (\WH_N H_0 +\WH_0H_N) \\
%& -3x_\ast^2 \left(  \sum_{l+m+n=N} W_n H_l (-1)^m +2 \sum_{l+m+n=N-1} W_nH_l (-1)^m  + \sum_{l+m+n=N-2}W_n H_l (-1)^m  \right) \\ 
%&- 2x_\ast^2(1+k)\left( ( W ( W + \tfrac{k}{k+1}) ( R- W) )_N +( W ( W + \tfrac{k}{k+1}) ( R- W) )_{N-1}  \right) =0 
& - 2x_\ast^2(1+\k)  \WH_0 ( \WH_0 + \k) ( \RH_N- \WH_N) -\tilde{\mathcal V}_N \\
=& \Big[-\e \RH_0^{-\e-1} (\WH_1-1 + 3\WH_0)  - x_\ast^2 (  \WH_1 -1+ 3\WH_0)4\k\WH_0 \\
&\quad- 2x_\ast^2 (1+\k) \WH_0 (\WH_0+\k)\Big] \RH_N \\
+& \Big[-\e \RH_0^{-\e-1} \RH_1N  -x_\ast^2 H_1N - 2x_\ast^2 H_0 N + 3\RH_0^{-\e} - 3x_\ast^2 H_0 \\
&\quad - x_\ast^2 (  \WH_1 -1 + 3\WH_0) (2(1-\k)\WH_0 +4\k + 4\k\RH_0)
\\
&\quad + 2x_\ast^2 (1+\k) \WH_0 (\WH_0+\k) \Big] \WH_N -  \mathcal V_N
\end{split}
\ee
where
\be\label{VN}
\begin{split}
 \mathcal V_N =  -  (\WH_1-1 +3\WH_0)\RH_0^{-\e} \sum_{m=2}^N \frac{1}{\RH_0^m}\sum_{\pi(N,m)} (- \e)_m \frac{1}{\lambda_1 ! \dots \lambda_N !} {\RH_1}^{\lambda_1}\dots 
{\RH_N}^{\lambda_N} \\
+x_\ast^2  (  \WH_1 -1 + 3\WH_0)  \Big[   (1-\k) \sum_{\substack{\ell+m =N\\ 1\leq m\leq N-1}} \WH_\ell \WH_m+  \sum_{\substack{\ell+m =N\\ 1\leq m\leq N-1}} 4\k \RH_\ell \WH_m\Big]
+\tilde{\mathcal V}_N
 \end{split}
\ee 
and 
\be\label{VN0}
\begin{split}
&\tilde{\mathcal V}_N = - \sum_{\substack{\ell+m=N \\ 1\leq m\leq N-2}} (m+1) \WH_{m+1} (\RH^{-\e})_\ell \\
&+x_\ast^2 \left( \sum_{\substack{\ell+m=N\\ 1\leq m\leq  N-2}} (m+1) \WH_{m+1} H_\ell + 2\sum_{\substack{\ell+m=N-1 \\ m\leq N-2}} (m+1) \WH_{m+1}H_\ell +\sum_{\ell+m=N-2} (m+1) \WH_{m+1}H_\ell  \right) \\
&+ \sum_{\substack{\ell+m=N \\ 1\leq m\leq N}} (\RH^{-\e})_\ell (-1)^m - 3\sum_{\substack{\ell+m+n=N \\ 1\leq m\leq N}} \WH_n (\RH^{-\e})_\ell  (-1)^m \\
&-  \xs^2 \left(  \sum_{\substack{\ell+m=N \\ 1\leq m\leq N}} H_\ell (-1)^m +2 \sum_{\ell+m=N-1} H_\ell (-1)^m  + \sum_{\ell+m=N-2}H_\ell (-1)^m  \right)\\
& + 3x_\ast^2 \left( \sum_{\substack{\ell+n=N\\1\leq  n\leq N-1}} \WH_n H_\ell  + \sum_{\substack{\ell+m+n=N\\1\leq  m\leq N}} \WH_n H_\ell (-1)^m +2 \sum_{\ell+m+n=N-1} W_nH_\ell (-1)^m  + \sum_{\ell+m+n=N-2}W_n H_\ell (-1)^m  \right)\\
&+ 2x_\ast^2(1+\k)\left( \sum_{\substack{\ell+m+n=N\\ 1\leq n\leq N-1}} \WH_\ell ( \WH + \k)_m ( \RH- \WH)_n +( \WH ( \WH + \k) ( \RH- \WH) )_{N-1}  \right). 
\end{split}
\ee 
This completes the proof. 
\end{proof}

%%%%%%%%%%%%%%%%%%%%%%%%%%
%%%%%%%%%%%%%%%%%%%%%%%%%%
Let $\kappa>0$ be a sufficiently small number independent of $\k$ to be fixed later in Section \ref{Sec:XMIN}. 

%%%%%%%%%%%%%%%%%%%%%%%%%%
%%%%%%%%%%%%%%%%%%%%%%%%%%

\begin{lemma}\label{lem:detA} Let $x_\ast\in[\xc+\kappa,\xmax]$ be given.  Then the components $A_{ij}$ of the matrix $\mathcal A_N$ satisfy the following 
\begin{align}
A_{11}&= -2 N\left( x_\ast -1 + O(\k)\right) + 2 +O(\k), \\
 A_{12} &=O(\k), \\
 A_{21} &= -2+O(\k), \\
  A_{22}&= -2N\left( x_\ast -1+O(\k) \right)  +O(\k). 
\end{align}
In particular, the matrix $\mathcal A_N$ is invertible for all sufficiently small $\k$ and $\det \mathcal A_N = O(N^2)$. 
Moreover, the Taylor coefficients $(\RH_N,\WH_N)$, $N\ge2$, satisfy the recursive relationship
\begin{align}
\RH_N = \frac{A_{22}}{\det A_N} \mathcal S_N - \frac{A_{12}}{\det A_N } \mathcal V_N,\label{recR}  \\
\WH_N = \frac{ A_{11}}{\det A_N} \mathcal V_N- \frac{A_{21}}{\det A_N} \mathcal S_N, \label{recW}
\end{align} 
where the source terms $\mathcal V_N,\mathcal S_N$, $N\ge2$, are given by Lemma~\ref{Lem:N}.
\end{lemma}

%%%%%%%%%%%%%%%%%%%%%%%%%%

\begin{proof} We rearrange terms in $A_{ij}$ of Lemma \ref{Lem:N} as 
\begin{align*}
A_{11}&=-2N \left( x_\ast^2 H_0 + \tfrac{1}{2} x_\ast^2 H_1 +  \frac\eta2\RH_0^{-\e-1}\RH_1\right)  \\
&\quad+ 2x_\ast^2  \RH_0\WH_0 + 2\k x_\ast^2 \RH_0  (1-\k-\WH_0)   - \e \RH_0^{-\e-1}\RH_1 - 4\k x_\ast^2 \WH_0\RH_1,\\
A_{12}&= - 2x_\ast^2 \RH_1 \WH_0 - 2x_\ast^2 \RH_0 \WH_0  -  2\k x_\ast^2 \RH_1 ( 2 + 2\RH_0- \WH_0) - 2\k x_\ast^2 \RH_0 (1-\k-\WH_0),\\
A_{21}&= -  2x_\ast^2\WH_0^2 -   2\k x_\ast^2  \WH_0 (\WH_0+1+\k) - 2\k (\tfrac{1}{1-\k}\RH_0^{-\e-1} +2 \WH_0x_\ast^2 ) (\WH_1 +3\WH_0 -1),\\
%&\quad - 4k\WH_0x_\ast^2(\WH_1 +3\WH_0 -\tfrac{1}{1+k})\\
A_{22}&=-2N \left(x_\ast^2 H_0 + \tfrac{1}{2} x_\ast^2 H_1 +  \frac\eta2\RH_0^{-\e-1}\RH_1 \right) \\
&\quad  - 3x_\ast^2 H_0 + 3\RH_0^{-\e} + 2x_\ast^2 \WH_0^2  
%- 2x_\ast^2 \WH_0 (\WH_1+ 3\WH_0)  
-2x_\ast^2\WH_0 (  \WH_1  + 3\WH_0-1)  \\
&\quad + 2\k x_\ast^2  \WH_0 (\WH_0+1+\k)  -2\k x_\ast^2 (  \WH_1  + 3\WH_0-1)  (2  +2\RH_0- \WH_0)  . 
\end{align*}
From~\eqref{sonic}, equations~\eqref{R11}--\eqref{R11.1} and Lemmas~\ref{R0W0},~\ref{R1W1}, we have 
the relations
\[
H_0 = \frac1{\xs^2} + O(\k), \ \ H_1= 2\WH_0(0)\WH_1(0) + O(\k) = \frac{2}{\xs}\left(1-\frac2{\xs}\right) + O(\k),
\]
and therefore
\begin{align*}
A_{11}&= -2 N\left( 1+ \tfrac{\WH_1(0)}{\WH_0(0)} + O(\k)\right) + 2 +O(\k), \\
 A_{12} &= -2\tfrac{\RH_1(0)}{\WH_0(0)} - 2+O(\k) ,\\
 A_{21} &= -2+O(\k), \\
  A_{22}&= -2N\left( 1+ \tfrac{\WH_1(0)}{\WH_0(0)}+O(\k) \right) -4 +2\tfrac{1}{\WH_0(0)}-2\tfrac{\WH_1(0)}{\WH_0(0)} +O(\k). 
\end{align*}
Since $\RH_0(0)=\WH_0(0)=\frac{1}{x_\ast}$, $\frac{\WH_1(0)}{\WH_0(0)} = x_\ast -2$ and $\RH_1(0)=-\frac{1}{x_\ast}$, the claimed behavior of $A_{ij}$ follows. 

Since $N\geq 2$ and  $x_\ast\geq \xc +\kappa >2$, the determinant of $\mathcal A_N$ has a  lower bound 
\begin{align*}
\tfrac14\det \mathcal A_N &=  \left(N(x_\ast-1+O(\k)) +O(\k) \right) \left( N (x_\ast-1+O(\k)) -1 +O(\k)\right) +O(\k) \\
&\geq  \left( N(x_\ast-1-C\k_0))-C\k_0 \right)\left(N (x_\ast-1-C\k_0) -1-C\k_0\right) - C\k_0
\end{align*}
for some universal constant $C>0$. For a sufficiently small $\k_0>0$ we have $ N(x_\ast-1-C\k_0))-C\k_0 \geq \frac32$ for $N\geq 2$ and  $x_\ast\geq \xc+\kappa>2$. 
We see that $\det \mathcal A_N >0$ and hence $\mathcal A_N$ is invertible for all $0<\k\le\k_0$ with $\k_0$ chosen sufficiently small. 
It immediately follows that  $\det \mathcal A_N = O(N^2)$. 
Since $\mathcal A_N$ is invertible, relations~\eqref{recR}--\eqref{recW} follow by multiplying~\eqref{E:MATRIXEQN} by $\mathcal A_N^{-1}$ from the left.
\end{proof}

%%%%%%%%%%%%%%%%%%%%%%%%%%
%%%%%%%%%%%%%%%%%%%%%%%%%%

\begin{remark}
A simple consequence of the previous lemma is the existence of a universal constant $\beta_0>0$ such that for any  
 $x_\ast\in[\xc+\kappa,\xmax]$  and sufficiently small $0<\k\le\k_0$ 
the following bounds hold: 
\begin{align}
| \RH_N|\leq  \frac{\beta_0}{N} \left(  |\mathcal S_N| + \frac{\k}{N} |\mathcal V_N| \right),\label{recR1} \\
|\WH_N| \leq \frac{\beta_0}{N} \left(  |\mathcal V_N| + \frac{1}{N} |\mathcal S_N| \right). \label{recW1}
\end{align} 
\end{remark}

%%%%%%%%%%%%%%%%%%%%%%%%%
%%%%%%%%%%%%%%%%%%%%%%%%%

\begin{remark} It is a routine to check 
\begin{align}
\RH_2& = \frac{-x_\ast^2 + 6x_\ast -7}{ 2x_\ast(2x_\ast -3)} + O({\k}),\label{coeff_R2} \\
\WH_2&= \frac{-5x_\ast^2 + 19 x_\ast - 17}{ 2x_\ast (2x_\ast -3)} + O({\k}), \label{coeff_W2}
\end{align}
for any $x_\ast\in [x_{\text{cr}}+\kappa, x_{\max}]$. 
\end{remark}

%\begin{remark} When $k=0$, $H_0=\WH_0^2$, $H_1= 2\WH_0\WH_1$ where $\RH_1=-\frac{1}{x_\ast}$, $\WH_1=1-\frac{2}{x_\ast}$ and $\WH_0=\RH_0=\frac{1}{x_\ast^2}$ and one can check that the coefficients of $\mathcal A_N$ in Lemma \ref{Lem:N} are written as  
%\begin{align*}
%A_{11}= -2N + 2 - 2N \tfrac{\WH_1}{\WH_0}, \qquad \ \  A_{12} = -\tfrac{2\RH_1}{\WH_0} - 2, \\
 %A_{21} = -2, \quad A_{22}= -2N-4 +\tfrac{2}{\WH_0} - (2N+2)\tfrac{\WH_1}{\WH_0}
%\end{align*} Therefore, we recover Lemma 2.3 of our LP paper. 
%\end{remark}
%By Lemma \ref{Lem:N} and \ref{lem:detA}, for $x_\ast\in(\xmin,\xmax)$, we have the recursive relations
%\begin{align}
%\RH_N = \frac{A_{22}}{\det A_N} \mathcal S_N - \frac{A_{12}}{\det A_N } \mathcal V_N\label{recR}  \\
%\WH_N = \frac{ A_{11}}{\det A_N} \mathcal V_N- \frac{A_{21}}{\det A_N} \mathcal S_N \label{recW}
%\end{align} 

%%%%%%%%%%%%%%%%%%%%%%%%%%
%%%%%%%%%%%%%%%%%%%%%%%%%%
\begin{remark} In Lemma \ref{lem:detA}, the lower bound $\xc+\kappa$ for the $x_\ast$-interval  has been chosen for convenience  to ensure $O(\k)$ disturbance of the coefficients to the corresponding ones to LP type solutions. It can be relaxed to $\xc$ by replacing   $O(\k)$  by  $O(\sqrt\k)$. See the change of the distance of $\WH_1$ near $\xc$ in Lemma \ref{L:XMINDEF}. 
\end{remark}

%%%%%%%%%%%%%%%%%%%%%%%%%%
%%%%%%%%%%%%%%%%%%%%%%%%%%

%%%%%%%%%%%%%%%%%%%%%%%%%
%%%%%%%%%%%%%%%%%%%%%%%%%

\subsection{Series convergence and local existence around a sonic point}

%%%%%%%%%%%%%%%%%%%%%%%%%
%%%%%%%%%%%%%%%%%%%%%%%%%

\begin{theorem}[Local existence around the sonic point]\label{T:SONICLWP} 
There exist an $\k_0>0$ and $r>0$ sufficiently small such that for all $0<\k\le\k_0$ and all  $\xs \in [\xc(\k)+\kappa,\xmax(\k)]$  the 
sequence $\{\RH_N, \WH_N\}_{N\in \mathbb Z_{\geq 0}}$ of RLP-type (see Definition~\ref{D:RLPTYPEDEF}) has the following property:
%) be the coefficients given in Lemma  \ref{R0W0}, Lemma \ref{R1W1} and \eqref{recR} and \eqref{recW}. Then there exist an $0<r<1$ and $k_\ast>0$ such that 
%for all $0<k<k_\ast$ and any  $x_\ast \in (\xmin,\xmax)$ and  
the formal power series 
\be\label{infinites}
\RH (z):= \sum_{N=0}^\infty \RH_N (\delta z)^N, \quad \WH(z) := \sum_{N=0}^\infty \WH_N (\delta z)^N
\ee
converge for all $z\in(1-r,1+r)$ and functions $z\mapsto \RH(z)$ and $z\mapsto \WH(z)$ are real analytic inside $|z-1|<r$. We can differentiate the infinite sums term by term, the pair $z\mapsto (\RH(z), \WH(z))$ 
solves~\eqref{Eq:R}--\eqref{Eq:W} for $|z-1|<r$, and $R(z)$ is strictly positive for $|z-1|<r$.  
\end{theorem}

%%%%%%%%%%%%%%%%%%%%%%%%%

\begin{proof} The proof is analogous to Theorem 2.10 of \cite{GHJ2021}. By Lemma \ref{lem:InductionFinal}, using $\alpha\in(1,2)$, there exists $C>1$ such that 
\[
\sum_{N=1}^\infty |\RH_N||\delta z|^N + \sum_{N=1}^\infty |\WH_N||\delta z|^N \leq 2 \sum_{N=1}^\infty \frac{|C\delta z|^N}{CN^3} <\infty 
\] when $|\delta z|<\frac{1}{C}=: r$. The claim follows by the comparison test. The real analyticity and differentiability statements are clear. Recalling \eqref{B_exp}, we may rewrite $B$ as 
\be
\begin{split}
B&=\sum_{N=0}^\infty (\RH^{-\e })_N  (\delta z)^N - x_\ast^2
 \sum_{N=0}^\infty H_N (\delta z)^N(1+ 2 \delta z + (\delta z)^2) \\
 &= \left[ (\RH^{-\e })_1 - x_\ast^2 (H_1 + 2 H_0) \right]\delta z + \sum_{N=2}^\infty \left [ (\RH^{-\e })_N -  x_\ast^2 (H_N + 2H_{N-1} + H_{N-2}) \right] (\delta z)^N.
\end{split}
\ee 
By~\eqref{R11}--\eqref{R11.1}
\be
\begin{split}
&(\RH^{-\e })_1 - x_\ast^2 (H_1 + 2 H_0) \\
&= -\e  \RH_0^{-\e-1} \RH_1 - x_\ast^2 \Big( 2\left[ (1+\k)\WH_0 + 2\k \right] \WH_1 + 4\k \WH_0 \RH_1 \\
&\qquad \qquad\qquad \qquad \qquad \quad  +2\left[ (1-\k)\WH_0^2 +  4\k \WH_0+ 4\k \RH_0\WH_0+\k^2-\k \right] \Big) \\
&= - 2x_\ast^2 \WH_0 (\WH_1+ \WH_0) -2\k \mathcal P \neq 0,%- 2k \Big(  \tfrac{1}{1-k} \RH_0^{-\e -1} \RH_1 +x_\ast^2 ( \WH_0\WH_1+\tfrac{2}{1+k}\WH_1 + 2\WH_0\RH_1\\
%& \qquad \qquad\qquad \qquad\qquad \qquad-\WH_0^2 +\tfrac{4}{1+k} \WH_0 + 4\RH_0\WH_0 +\tfrac{k-1}{(1+k)^2} ) \Big)
\end{split}
\ee 
for all sufficiently small $\k$, where 
\[
\mathcal P =  \tfrac{1}{1-\k} \RH_0^{-\e -1} \RH_1 +x_\ast^2 \left( \WH_0\WH_1+2\WH_1 + 2\WH_0\RH_1 +4 \WH_0 + 3\WH_0^2 +\k-1 \right) . 
\]
Therefore, it is now easy to see that for all sufficiently small $\k$ and for $r>0$ sufficiently small, the function $B \neq 0$  for all $|z-1|<r$ and $z\neq 1$. As a consequence, $\RH(z)$ and $\WH(z)$ are indeed the solutions as can be seen by plugging the infinite series \eqref{infinites} into \eqref{Eq:R1} and \eqref{Eq:W1}. 
\end{proof}

%%%%%%%%%%%%%%%%%%%%%%%%%
%%%%%%%%%%%%%%%%%%%%%%%%%
%%%%%%%%%%%%%%%%%%%%%%%%%
%%%%%%%%%%%%%%%%%%%%%%%%%

\begin{lemma}\label{Lem:analyticXast}  
There exist an $\k_0>0$ and $r>0$ sufficiently small such that for any $\alpha\in(1,2)$ and for all $\xs \in [\xc(\k)+\kappa,\xmax(\k)]$ 
%and all $0<\k\le\k_0$ and 
%all $\xs \in (\xmin(\k),\xmax(\k))$ the 
the sequence $\{\RH_N, \WH_N\}_{N\in \mathbb N_{\geq 0}}$ of RLP-type (see Definition~\ref{D:RLPTYPEDEF}) has the following property:
%Let $\{\RH_N, \WH_N\}_{N\in \mathbb N_{\geq 0}}$ be the coefficients given in Lemma  \ref{R0W0}, Lemma \ref{R1W1} and \eqref{recR} and \eqref{recW}. For any  $x_\ast\in(\xmin,\xmax)$ and for $\alpha\in(1,2)$, there exist a constant $C=C(x_\ast, \alpha)>0$ and $k_\ast>0$ such that for all $0<k<k_\ast$,  
There exists a constant $C=C(\xs,\alpha)>0$ such that for all $0<\k\le\k_0$ the bounds
$|\pa_{x_\ast}\RH_0|$, $|\pa_{x_\ast}\WH_0|$, $|\pa_{x_\ast}\RH_1|$, $|\pa_{x_\ast}\WH_1|\leq C$ hold and 
\begin{align}
|\pa_{x_\ast}\RH_N|\leq \frac{C^{N-\alpha}}{N^3}, \ \ N\ge2, \label{ind1}\\
|\pa_{x_\ast}\WH_N|\leq \frac{C^{N-\alpha}}{N^3}, \ \ N\ge2. \label{ind2}
\end{align} 
In particular, the formal power series 
\[
 \sum_{N=0}^\infty \pa_{x_\ast} \RH_N (\delta z)^N, \quad \sum_{N=0}^\infty  \pa_{x_\ast} \WH_N (\delta z)^N
\]
converge for all $z$ satisfying $|z-1|<r$. Moreover, the function $x_\ast \in (\xc+\kappa,\xmax)\rightarrow ( \RH(z; x_\ast), \WH (z;x_\ast) )$ is $C^1$ and the derivatives $\pa_{x_\ast} \RH$ and $\pa_{x_\ast} \WH$ are given by the infinite series above. 
\end{lemma}

%%%%%%%%%%%%%%%%%%%%%%%%%

\begin{proof} From Lemma \ref{R0W0} and Lemma \ref{R1W1}, $\pa_{x_\ast} \RH_0(0) =\pa_{x_\ast} \WH_0(0)= -\frac{1}{x_\ast^2}$ and $\pa_{x_\ast}\RH_1(0)= \frac{1}{x_\ast^2}$, $\pa_{x_\ast}\WH_1(0)= \frac{2}{x_\ast^2}$, 
and it is clear that $|\pa_{x_\ast}\RH_0|$, $|\pa_{x_\ast}\WH_0|$, $|\pa_{x_\ast}\RH_1|$, $|\pa_{x_\ast}\WH_1|\leq C$ for  $x_\ast\in[\xc+\kappa,\xmax]$ and for all sufficiently small $\k$. For $N\geq 2$, $\pa_{x_\ast} \RH_N$ and $\pa_{x_\ast} \WH_N$ are recursively given by differentiating the expression in \eqref{recR} and \eqref{recW}: 
\begin{align}
\pa_{x_\ast}  \RH_N =\pa_{x_\ast} \left(  \frac{A_{22}}{\det A_N}\right)  \mathcal S_N+ \frac{A_{22}}{\det A_N} \pa_{x_\ast}\mathcal S_N - \pa_{x_\ast}\left( \frac{A_{12}}{\det A_N }\right) \mathcal V_N- \frac{A_{12}}{\det A_N } \pa_{x_\ast}\mathcal V_N, \\
\pa_{x_\ast}  \WH_N =\pa_{x_\ast} \left(  \frac{ A_{11}}{\det A_N} \right) \mathcal V_N + \frac{ A_{11}}{\det A_N} \pa_{x_\ast}\mathcal V_N -\pa_{x_\ast} \left( \frac{A_{21}}{\det A_N}\right) \mathcal S_N - \frac{A_{21}}{\det A_N} \pa_{x_\ast}\mathcal S_N . 
\end{align}
When $N=2$, the claim immediately follows. For $N\geq 3$, we will apply the same induction argument used for $\RH_N, \WH_N$ bounds. To this end, we first observe that from Lemma \ref{Lem:N} and Lemma \ref{lem:detA} 
\begin{align}
&\pa_{x_\ast} A_{11} = - 2N (1+O(\k)) + O(\k), \quad \pa_{x_\ast} A_{12} = O(\k), \\
& \pa_{x_\ast} A_{21} =O(\k), \quad \pa_{x_\ast} A_{22} = - 2N (1+O(\k)) + O(\k), \\
&\pa_{x_\ast} \det \mathcal A_N = 4N(1+O(\k))( 2N (( x_\ast-1) +O(\k)) -1) + O(\k) , 
\end{align} leading to 
\[
\left| \pa_{x_\ast} \left(  \frac{A_{22}}{\det A_N}\right) \right|, \  \left| \pa_{x_\ast} \left(  \frac{A_{11}}{\det A_N}\right) \right| \lesssim \frac{1}{N},\ \  \left| \pa_{x_\ast} \left(  \frac{A_{12}}{\det A_N}\right) \right|\lesssim \frac{\k}{N^2}, \  \left| \pa_{x_\ast} \left(  \frac{A_{21}}{\det A_N}\right) \right| \lesssim \frac{1}{N^2}. 
\]
Hence using Lemmas~\ref{lem:SVbound} and~\ref{lem:InductionFinal}, 
\be
\begin{split}
\left| \pa_{x_\ast}  \RH_N \right|&\lesssim \frac{1}{N} |\mathcal S_N|+\frac{1}{N} |\pa_{x_\ast}\mathcal S_N| +\frac{\k}{N^2} |\mathcal V_N| +\frac{\k}{N^2} |\pa_{x_\ast}\mathcal V_N| \\
&\lesssim  \frac{ C^{N-\alpha}}{N^3}  +\frac{1}{N} |\pa_{x_\ast}\mathcal S_N| +\frac{\k}{N^2} |\pa_{x_\ast}\mathcal V_N|, \\
\left|\pa_{x_\ast}  \WH_N  \right| & \lesssim \frac{1}{N}| \mathcal V_N |+\frac{1}{N} |\pa_{x_\ast}\mathcal V_N  | +\frac{1}{N^2} |\mathcal S_N| +\frac{1}{N^2} |\pa_{x_\ast}\mathcal S_N | \\
& \lesssim  \frac{ C^{N-\alpha}}{N^3} +\frac{1}{N} |\pa_{x_\ast}\mathcal V_N  |  + \frac{1}{N^2} |\pa_{x_\ast}\mathcal S_N |. 
\end{split}
\ee

We now recall that $\mathcal S_N$ and $\mathcal V_N$ consist of sum and product of polynomials in $\RH_0,\WH_0,\dots,\RH_{N-1},\WH_{N-1}$ and power functions of $\RH_0$. When we differentiate with respect to $x_\ast$, at most one term indexed by $\RH_i$ or $\WH_i$, $0\leq i\leq N-1$ is differentiated. In particular, the same combinatorial structure in the problem is maintained and the same inductive proof relying on the already established bounds \eqref{E:inductionR} and \eqref{E:inductionW} gives \eqref{ind1} and \eqref{ind2}. The remaining conclusions now follow easily. 
\end{proof}

%%%%%%%%%%%%%%%%%%%%%%%%%%%
%%%%%%%%%%%%%%%%%%%%%%%%%%%

\subsection{The sonic window and $\xmin$}\label{Sec:XMIN}

%%%%%%%%%%%%%%%%%%%%%%%%%%%
%%%%%%%%%%%%%%%%%%%%%%%%%%%

The goal of this subsection is to define $\xmin$ and to define the sonic window, which serves as the basic interval in our shooting method in the next section. We begin with the following lemma.

\begin{lemma}\label{L:XMINN} Consider the RLP type solution constructed in Theorem \ref{T:SONICLWP}. There exist a small constant $\delta_0>0$ independent of $\k$ and $0<z_0=z_0(\delta_0) <1$ such that for $x_\ast= 2+\delta_0 > \xc$ 
$$\WH (z_0) >\frac{1}{2-2\e} $$
for all sufficiently small $\k>0$. Here we recall $\e= \frac{2\k}{1-\k}$. 
\end{lemma}
\begin{proof} For any $x_\ast\in (\xc, x_{\max})$, we Taylor-expand $\WH$ in $z$ around $z=1$:  
\be\label{TaylorExpW}
\WH(z) = \WH (1) + \WH'(1) (z-1) + \frac{\WH''(\tilde z)}{2} (z-1)^2
\ee
for some $\tilde z \in [z,1]$. We note that $ \WH (1)= \WH_0$ and $\WH'(1)= \WH_1$. Let $\delta>0$ be a small constant independent of $\k$ to be fixed and set $x_\ast= 2+\delta$. Then  we have 
\be
\WH_0 = \frac{1}{2+\delta} + O(\k), \quad \WH_1= \frac{\delta}{2+\delta } + O(\k),
\ee
see the end of Remark~\ref{R:CONSISTENCYSONIC}. 
On the other hand, from \eqref{coeff_W2} we have 
\be\label{coeff_W''}
\WH''(1)= 2 \WH_2 = \frac{1-\delta - 5\delta^2}{(2+\delta) (1+2\delta)} + O({\k}).  %= \frac12 + O(\kappa) + O(\k)
\ee
Now let 
\[
%z_0=
z_0(\delta,\k) = \min \{ z_1 : \WH''(z) \ge \frac{1}{4}  \ \text{for all}  \ z\in [z_1,1] \}. 
\]
Observe that $z_0<1$ by~\eqref{coeff_W''} for $\k,\delta$ sufficiently small.  
We now claim that there exists a small enough $\delta_0>0$ such that $z_0(\delta_0,\k)\le 1-\delta_0^{1/4}$  for all sufficiently small $\k$. Suppose not. Then for all $\delta>0$ and for some $\k_0>0$,  $z_0(\delta,\k_0) > 1-\delta^{1/4}$. Thus there exists $1-\delta^{1/4}<z_1=z_1(\delta,\k_0)<1$ so that $\WH''(z_1) <\frac14$, but this is impossible because of \eqref{coeff_W''} and the continuity of $\WH''$. 
Now the Taylor expansion \eqref{TaylorExpW} at $z=1-\delta_0^{1/4}$ gives rise to 
\be
\begin{split}
\WH(1-\delta_0^{1/4}) &=   \frac{1}{2+\delta_0} - \frac{\delta_0}{2+\delta_0} \delta_0^{1/4} 
+ \frac{\WH''(\tilde z)}{2} \delta_0^{1/2} + O(\k)\\
&\ge  \frac{1}{2+\delta_0} - \frac{\delta_0}{2+\delta_0} \delta_0^{1/4} 
+ \frac{1}{8} \delta_0^{1/2} + O(\k)
\end{split}
\ee
from which we deduce $\WH(1-\delta_0^{1/4}) > \frac{1}{2-2\e}$ for all sufficiently small $\k$.
\end{proof}

We now define 
\be\label{E:XMINDEF}
\xmin := 2+\delta_0
\ee
where $\delta_0$ is given in Lemma \ref{L:XMINN}. We observe that by construction $\xmin$ is independent of $\k$. We are ready to introduce the sonic window:

\begin{definition}[The sonic window]\label{D:SONICINTERVAL}
For any $0\le \k\le \k_0$ we refer to the interval $[\xmin,\xmax(\k)]$ as the \underline{sonic window}. We often drop the $\k$-dependence when the $\k$ is fixed.
\end{definition}

%%%%%%%%%%%%%%%%%%%%%%
%%%%%%%%%%%%%%%%%%%%%%

\begin{remark}
Observe that by construction, the sonic window $[\xmin,\xmax(\k)]$ is a strict subset of  the interval $[2,3]$, while the interval $[\xc(\k), \xmax(\k)]$ coincides with the interval $[2,3]$ when $\k=0$. The latter is precisely the range of possible sonic points within
which we found the Newtonian Larson-Penston solution in~\cite{GHJ2021} and our new sonic window $[\xmin,\xmax(\k)]$ shows that the lower bound $x_\ast=2$ can be improved to $2+\delta_0$ even for the Newtonian problem. 
% coincides with the interval $[2,3]$ when $\k=0$. This is precisely the range of possible sonic points within
%which we found the Newtonian Larson-Penston solution in~\cite{GHJ2021}.
\end{remark}
%%%%%%%%%%%%%%%%%%%%%%
%%%%%%%%%%%%%%%%%%%%%%

%%%%%%%%%%%%%%%%%%%%%%
%%%%%%%%%%%%%%%%%%%%%%

For future use in Sections~\ref{S:FRIEDMANN} and~\ref{S:FARFIELD}, we analyse the behaviour of $\F -xW$ and $\F -x\R $ near the sonic point $x_\ast$. 

\begin{lemma}[Initialisation]\label{L:INITIAL}
Let $\k\in(0,\k_0]$, where $\k_0>0$ is a sufficiently small constant given by Theorem~\ref{T:SONICLWP}.
There exist a $\delta>0$ and $c_0>0$ such that $c_0\k_0<\delta$ and for any $x_\ast\in[\xmin,\xmax]$, the unique local RLP-type solution associated to $x_\ast$ given by Theorem~\ref{T:SONICLWP},
satisfies the bounds
\begin{enumerate}
\item[{\em (a)}]
\begin{align}
\F [x;\R ]&<xW, \ \  x\in(\xs,\xs+\delta),\label{initial1} \\
\F [x;\R ]&>xW, \ \  x\in(\xs-\delta,\xs). \label{initial2}
\end{align}
\item[{\em (b)}]
\begin{align}
\F [x;\R ]&>x\R , \ \  x\in(\xs-2\delta, \xs-c_0\k_0),\label{initial_fleft}\\
\F [x;\R ]&>x\R , \ \  x\in(\xs+ c_0\k_0, \xs+2\delta).\label{initial_fright}
\end{align}
\item[{\em (c)}]
Moreover, the following bound holds
\begin{align}\label{E:INAUX1}
\left[x(1-W(x))\right]\Big|_{x=x_\ast+\delta}\ge \frac12, \ \ \text{ for all} \ \k\in(0,\k_0].
\end{align}
\end{enumerate}
\end{lemma}

%%%%%%%%%%%%%%%%%%%%%%%

\begin{proof} 
{\em Proof of part (a).}
Let $g(x):= xW - \F $. Since $g$ is a smooth function of $\R $ and $W$, by Theorem~\ref{T:SONICLWP}, $g$ is smooth near the sonic point. Note that $g(x_\ast)=0$ from Lemma \ref{R0W0}. Since $g'= W + x W' - \F '$, using Lemma \ref{R0W0}, Lemma \ref{R1W1} and \eqref{E:FPRIME}, we deduce that 
\[
g'(x_\ast) = 1- \frac{1}{x_\ast} + O(\k) >\frac13 \ \text{for all} \  x_\ast\in[\xmin,\xmax]
\]
for all sufficiently small $\k>0$. Therefore, $g$ is locally strictly increasing and \eqref{initial1} and \eqref{initial2} hold for some $\delta>0$. 

\noindent
{\em Proof of part (b).}
Since $\F -x\R =0$ at $x=x_\ast$, we use this and the formula~\eqref{E:FPRIME} to conclude
\begin{align}
(\F ' - (x\R )')\Big|_{x=\xs} = \k  \frac{-2 (1+\R _0)\xs \R _0+(1-\k)\xs -\left(2\xs^2 \R _0 \R _1+\frac1{1-\k}\R _0^{-\e-1}\R _1\right)}{(1-\k)\xs \R _0 + 2\k\xs (1+\R _0)} - \xs \R _1- \R _0,
\end{align}
where we recall~\eqref{Taylor}. 
By Remark~\ref{R:CONSISTENCYSONIC} and~\eqref{E:SONICTRANS}, we have $\R _0=W_0=\frac1{\xs}+O(\k)$ and $\R _1=-\frac1{\xs^2}+O(\k)$. Plugging this into the above
expression, we conclude that
\[
f'(\xs)=(\F ' - (x\R )')\Big|_{x=\xs} = O(\k).
\] 
%Recall $f= F-xR$. Since $f$ is a smooth function of $R$ and $W$, by Theorem~\ref{T:SONICLWP}, $f$ is smooth near the sonic point. To prove \eqref{initial_fleft} and \eqref{initial_fright}, we will show the sign of $f$ near the sonic point by Taylor expansion up to the 2nd order. To this end, we first compute $f(x_\ast)$, $f'(x_\ast)$ and $f''(x_\ast)$. 
%
%Note  that $f(x_\ast)=0$ due to the sonic conditions in Lemma  \ref{R0W0}. And by Lemma \ref{L:JUHILEMMA},  we have $f'(x) = - a f(x) + b$ and hence using $f(x_\ast)=0$ and $xW-F|_{x=x_\ast}=0$, we obtain $$f'(x_\ast)= \k b_2[x_\ast; R_0,W_0]$$ where we recall $R_0= R(x_\ast)$ and $W_0=W(x_\ast)$. Since $\frac{1}{3}\le  R_0 \le \frac{1}{2}+ O(\k)$ for $x\in [\xmin,\xmax]$, by recalling the expression of $b_2$ in \eqref{E:BTWODEF}, we obtain 
%\be\label{b_2bound}
% -c_1<b_2[x_\ast; R_0,W_0] <0
% \ee 
 %for all sufficiently small $\k$ and some universal constant $c_1>0$.  
Similarly, $f'' = \F '' - 2 \R ' - x \R ''$ and thus using $\R ''(x_\ast) = \frac{2}{x_\ast^2} \RH_2=  \frac{2}{x_\ast^2} \frac{-x_\ast^2 + 6x_\ast -7}{2x_\ast (2x_\ast -3)} +O(\k)$ (see~\eqref{coeff_R2}) and $\F ''(x_\ast) = O(\k)$, we have 
\[
f''(x_\ast) = \frac{2}{x_\ast^2}  -  x_\ast\frac{2}{x_\ast^2} \left( \frac{-x_\ast^2 + 6x_\ast -7}{2x_\ast (2x_\ast -3)}\right)  + O(\k)= \frac{x_\ast^2 - 2x_\ast +1 }{x_\ast^2(2x_\ast -3 )}+ O(\k)
\]
and hence $f''(x_\ast)>\frac19$ for all $x_\ast\in [\xmin,\xmax]$. Since $f''$ is uniformly continuous, there exists a $\delta>0$ such that 
\be\label{f''bound}
f''(x)> \frac{1}{18} \ \text{for} \ x\in(x_\ast-\delta,x_\ast+\delta)
\ee
 for all $x_\ast\in [\xmin,\xmax]$. 

We now Taylor-expand $f$ at $x=x_\ast$ to obtain 
\[
f(x) =   O(\k) (x-x_\ast) + \frac{f''(\tilde x)}{2} (x-x_\ast)^2, \quad x\in (x_\ast - \delta,   x_\ast +\delta)
\]
for some $\tilde x$ between $x_\ast$ and $x$.
% If $x<x_\ast$, since $b_2<0$ and $f''(\tilde x)>0$, we immediately see that $f>0$, which proves  \eqref{initial_fleft}. Now let $x>x_\ast$. Then using \eqref{b_2bound} and \eqref{f''bound}, we obtain 
For $x>\xs$ and for some $c_1>0$ we have
\[
f(x) > - c_1 \k (x-x_\ast) + \frac{1}{36} (x-x_\ast)^2 =\frac{1}{36} (x-x_\ast) \left( x-x_\ast - 36 c_1\k \right) . 
\]
Therefore we deduce \eqref{initial_fright} with $c_0=36c_1$. An analogous argument gives~\eqref{initial_fleft}.

\noindent
{\em Proof of part (c).}
Bound~\eqref{E:INAUX1} follows trivially from Theorem~\ref{T:SONICLWP} and the asymptotic behaviour as $\k\to0$ in Lemma~\ref{R0W0}.
\end{proof}

%%%%%%%%%%%%%%%%%%%%%%%%%
%%%%%%%%%%%%%%%%%%%%%%%%%
%%%%%%%%%%%%%%%%%%%%%%%%%

\subsection{Singularity at the origin $x=0$} 

%%%%%%%%%%%%%%%%%%%%%%%%%
%%%%%%%%%%%%%%%%%%%%%%%%%
%%%%%%%%%%%%%%%%%%%%%%%%%

By analogy to the previous section, we Taylor-expand the solution at the origin $z = 0$ in
order to prove a local existence theorem starting from the origin to the right. An immediate consistency condition
follows from the presence of $\frac{1 -3W}{z}$ in \eqref{Eq:W}: $\WH(0) =\frac{1}{3}$ and $\RH(0)= \tilde \RH_0>0$ is a free parameter.   
Denote the solution from the left by $\RH_-$ and $\WH_-$ and assume that locally around $z=0$ 
\be\label{series0}
\RH_-(z;\tilde \RH_0) := \sum_{N=0}^\infty\tilde \RH_N z^N, \quad \WH_-(z;\tilde \RH_0) := \sum_{N=0}^\infty\tilde \WH_N z^N
\ee 
where $\tilde \RH_0>0$ is a free parameter and $\tilde \WH_0 = \frac{1}{3}$. 
The following theorem asserts that the formal power series converge and hence $(\RH_-,\WH_-)$ are real analytic in a small neighbourhood of $z=0$. 

%%%%%%%%%%%%%%%%%%%%%%%%%
%%%%%%%%%%%%%%%%%%%%%%%%%

\begin{theorem} \label{T:SONICLWPORIGIN}
Let $\tilde \RH_0>0$ be given. There exists an $0<\tilde r<1$ such that the formal power series \eqref{series0} converge for all $z\in [0,\tilde r)$. In particular, $\RH_-$ and $\WH_-$ are real analytic on $[0,\tilde r)$. We can differentiate the infinite sums term by term and the functions $\RH_-(z,\tilde \RH_0)$ and $\WH_-(z,\tilde \RH_0)$ solve \eqref{Eq:R} and \eqref{Eq:W} with the initial conditions  $\RH_-(0;\tilde \RH_0)= \tilde \RH_0$,  $\WH_-(0;\tilde \RH_0) = \frac{1}{3}$. 
\end{theorem}

%%%%%%%%%%%%%%%%%%%%%%%%%

\begin{proof}
%Our starting point is 
%\begin{align}
%B R' &+ 2x_\ast^2(1-k)  z R ( W + \tfrac{k}{k+1}) ( R- W) =0 \label{Eq:R5}\\
%B W'&- \frac{(\frac{1}{1+k} -3 W )}{z} B - 2x_\ast^2 (1+k)  z W ( W + \frac{k}{k+1})
%( R- W)=0\label{Eq:W5}
%\end{align} 
Around the origin $z=0$ we write out the formal expansion of $B$: 
\be
B=  \RH^{-\e} -x_\ast^2Hz^2 = \sum_{j=0}^\infty  (\RH^{-\e})_j z^j - x_\ast^2 \sum_{j=0}^\infty H_j z^{j+2} . 
\ee 
By the Faa di Bruno formula~\eqref{E:FAAH}--\eqref{pi}
\be
 (\RH^{-\e}  )_j =\tilde \RH_0^{-\e} \sum_{m=1}^j \frac{1}{\tilde \RH_0^m}\sum_{\pi(j,m)} (- \e )_m \frac{1}{\lambda_1 ! \dots \lambda_j !} {\tilde \RH_1}^{\lambda_1}\dots 
{\tilde \RH_j}^{\lambda_j}, \quad j \geq 1
\ee
and $ (\RH^{-\e}  )_0 = \tilde \RH_0^{-\e}$. Plugging \eqref{series0} into \eqref{Eq:R}, we obtain the formal relation
\be
\begin{split}
0&=\left( \sum_{\ell=0}^\infty  (\RH^{-\e})_\ell z^\ell - x_\ast^2 \sum_{\ell=0}^\infty H_\ell z^{\ell+2} \right) \left(  \sum_{m=0}^\infty (m+1) \tilde \RH_{m+1} z^{m} \right) \\
&\quad+ 2x_\ast^2(1-\k) \sum_{\ell=0}^\infty (\RH ( \WH +\k) ( \RH- \WH))_\ell z^{\ell+1} \\
&=\sum_{N=0}^\infty \sum_{\ell+m=N}(m+1) \tilde \RH_{m+1}  (\RH^{-\e})_\ell z^N -x_\ast^2 \sum_{N=0}^\infty \sum_{\ell+m=N-2}(m+1) \tilde \RH_{m+1} H_\ell z^N \\
&\quad + 2x_\ast^2(1-\k) \sum_{N=0}^\infty (\RH ( \WH + \k) ( \RH- \WH))_{N-1} z^{N}. 
\end{split}
\ee
Comparing the coefficients, we obtain
\be
\begin{split}\label{RN+10}
\sum_{\ell+m=N} (m+1) \tilde \RH_{m+1} (\RH^{-\e})_\ell  - x_\ast^2 \sum_{\ell+m=N-2} (m+1) \tilde \RH_{m+1}H_\ell  \\
\quad+ 2x_\ast^2(1-\k) ( \RH ( \WH +\k) ( \RH- \WH) )_{N-1}  =0. 
\end{split}
\ee
Similarly, plugging \eqref{series0} into \eqref{Eq:W}, we obtain 
\be
\begin{split}
0&=\left( \sum_{\ell=0}^\infty  (R^{-\e})_\ell z^\ell - x_\ast^2 \sum_{\ell=0}^\infty H_\ell z^{\ell+2} \right) \left(  \sum_{m=0}^\infty (m+1) \tilde \WH_{m+1} z^{m} + 3\sum_{m=0}^\infty \tilde \WH_{m+1} z^m \right) \\
&\quad- 2x_\ast^2(1+\k) \sum_{\ell=0}^\infty (\WH ( \WH +\k) ( \RH- \WH))_\ell z^{\ell+1} \\
&=\sum_{N=0}^\infty \sum_{\ell+m=N}(m+4) \tilde \WH_{m+1}  (\RH^{-\e})_\ell z^N -x_\ast^2 \sum_{N=0}^\infty \sum_{\ell+m=N-2}(m+4) \tilde \WH_{m+1} H_\ell z^N \\
&\quad -2x_\ast^2(1+\k) \sum_{N=0}^\infty (\WH ( \WH + \k) ( \RH- \WH))_{N-1} z^{N}.
\end{split}
\ee 
Comparing the coefficients, we obtain 
\be\label{WN+10}
\begin{split}
\sum_{\ell+m=N} (m+4) \tilde \WH_{m+1} (\RH^{-\e})_\ell  - x_\ast^2 \sum_{\ell+m=N-2} (m+4) \tilde \WH_{m+1}H_\ell  \\
\quad- 2x_\ast^2(1+\k) ( \RH ( \WH + \k) ( \RH- \WH) )_{N-1}  =0. 
\end{split}
\ee
Identities \eqref{RN+10} and \eqref{WN+10} give the recursive relationships 
\begin{align}
\tilde \RH_{N+1} = \frac{1}{N+1} \tilde{\mathcal S}_{N+1}, \quad N\geq 0, \label{tildeR}\\
\tilde \WH_{N+1} = \frac{1}{N+4} \tilde{\mathcal V}_{N+1}, \quad N\geq 0, \label{tildeW}
\end{align} where
\begin{align}
\tilde{\mathcal S}_{N+1}&=\tilde \RH_0^{\e} \Big[ - \sum_{\substack{\ell+m=N\\ m\le N-1}} (m+1) \tilde \RH_{m+1} (\RH^{-\e})_\ell  +x_\ast^2 \sum_{\ell+m=N-2} (m+1) \tilde \RH_{m+1}H_\ell \label{tildeS}\\
&\qquad \qquad  -  2x_\ast^2(1-\k) ( \RH ( \WH + \k) ( \RH- \WH) )_{N-1}\Big] ,\notag\\
\tilde{\mathcal V}_{N+1}&=\tilde \RH_0^{\e} \Big[ - \sum_{\substack{\ell+m=N\\ m\le N-1}} (m+4) \tilde \WH_{m+1} (\RH^{-\e})_\ell  +x_\ast^2 \sum_{\ell+m=N-2} (m+4) \tilde \WH_{m+1}H_\ell  \label{tildeV}\\
&\qquad \qquad  +  2x_\ast^2(1+\k) ( \WH ( \WH + \k) ( \RH- \WH) )_{N-1}\Big],\notag
\end{align} where $\tilde{\mathcal S}_{N+1}$ and $\tilde{\mathcal V}_{N+1}$ depend only on $(\tilde \RH_i, \tilde \WH_i)$ for $0\le i\le N$. 
The rest of the proof is now entirely analogous to the proof of Theorem \ref{T:SONICLWP} and we omit  the
details. 
\end{proof}

%%%%%%%%%%%%%%%%%%%%%%%%%
%%%%%%%%%%%%%%%%%%%%%%%%%

%\begin{remark} At the sonic point, the fraction involving $B$ in RHS contributes to the leading order - the index is shifted by one, and the fraction involves both $\RH$ and $\WH$ nonlinearly. So even solving for $N$-th coefficients involves inverting the system. On the other hand, at the origin, that fraction is lower order (and contains one $z$ factor), while $(1-3\WH)/z$ term will contribute to the leading order but only to $\WH$ equation linearly. So the effective equations for $N$-th coefficients are derived easily without solving the system. 
%\end{remark}

%%%%%%%%%%%%%%%%%%%%%%%%%
%%%%%%%%%%%%%%%%%%%%%%%%%

%\begin{remark}
%Letting $N=0$ in \eqref{tildeS} and \eqref{tildeV} we see that $\tilde{\mathcal S}_1=0$ and $\tilde{\mathcal V}_1=0$ and therefore  \eqref{tildeR} and \eqref{tildeW} we see that $\tilde \RH_1=\tilde \WH_1=0$. Letting $N=1$ in  \eqref{tildeR} and \eqref{tildeW} we obtain 
%\begin{align}
%\tilde \RH_2 &= - x_\ast^2 \frac{(1-\k) (1+3\k)}{3(1+\k)}\tilde \RH_0^{\e+1} ( \tilde \RH_0 - \frac{1}{3(1+\k)} )  \\
%\tilde \WH_2 &=x_\ast^2 \frac{2(1+3\k)}{45(1+\k)}\tilde \RH_0^{\e} ( \tilde \RH_0 - \frac{1}{3(1+\k)} )
%\end{align}
%\end{remark}
%
%%%%%%%%%%%%%%%%%%%%%%%%%
%%%%%%%%%%%%%%%%%%%%%%%%%

\begin{remark} \label{R:CONEPROP}
We may repeat the same procedure as in Lemma \ref{Lem:analyticXast} to deduce that  $\partial_{\tilde \RH_0} \RH_-(z; \tilde \RH_0)$ and  $\partial_{\tilde \RH_0} \WH_-(z; \tilde \RH_0)$ have the convergent power series near the origin and the function $\tilde \RH_0\in (0,\infty) \rightarrow (\RH_-(z;\tilde \RH_0), \WH_-(z;\tilde \RH_0))$ is $C^1$. And the derivatives $\partial_{\tilde \RH_0} \RH_-$ and $\partial_{\tilde \RH_0}\WH_-$ satisfy the system of ODEs obtained by the differentiating \eqref{Eq:R} and \eqref{Eq:W} with initial conditions  $\partial_{\tilde \RH_0} \RH_-(0; \tilde \RH_0)=1$ and $\partial_{\tilde \RH_0}\WH_-(0;\tilde \RH_0 )=0$. 
\end{remark}

%%%%%%%%%%%%%%%%%%%%%%%%%%%
%%%%%%%%%%%%%%%%%%%%%%%%%%%
%%%%%%%%%%%%%%%%%%%%%%%%%%%
%%%%%%%%%%%%%%%%%%%%%%%%%%%
%%%%%%%%%%%%%%%%%%%%%%%%%%%

% The following section describes the analysis to the left of the sonic point.

\section{The Friedman connection}\label{S:FRIEDMANN}

%%%%%%%%%%%%%%%%%%%%%%%%%%%
%%%%%%%%%%%%%%%%%%%%%%%%%%%
%%%%%%%%%%%%%%%%%%%%%%%%%%%

\subsection{Sonic time, a priori bounds, and the key continuity properties of the flow}
We denote $x_\ast r$ by $\tilde r$, where $r$ is the analyticity radius given by Theorem~\ref{T:SONICLWP}.   

%%%%%%%%%%%%%%%%%%%%%%%%%%%
%%%%%%%%%%%%%%%%%%%%%%%%%%%

\begin{definition}[Sonic time]\label{D:SONICTIME}
For any $x_\ast\in[\xmin,\xmax]$ consider the unique local solution on the interval $[x_\ast-\tilde r,x_\ast+\tilde r]$ %$[x_\ast-r,x_\ast+r]$ 
given by Theorem~\ref{T:SONICLWP}. The {\em sonic time}
$s(x_\ast)$ is given by
\begin{align}
s(x_\ast) : = &\inf_{x\in(0,x_\ast]} \Big\{ \text{$(\R(\cdot,x_\ast), W(\cdot;x_\ast))$ is a solution to~\eqref{E:RODE}--\eqref{E:WODE} on $[x,x_\ast)$ and } \notag\\
&   \ \  \ B(x';\R,W)>0 \text{ for all } \ x'\in[x,x_\ast)\Big\},
\end{align}
where we recall the definition~\eqref{E:BDEF} of $B$.
\end{definition}

%%%%%%%%%%%%%%%%%%%%%%%%%%%
%%%%%%%%%%%%%%%%%%%%%%%%%%%

%%%%%%%%%%%%%%%%%%%%%%%%%%%
%%%%%%%%%%%%%%%%%%%%%%%%%%%

\begin{lemma}\label{L:AUXPRELIM}
%There exists an $\k_0>0$ sufficiently small such that for any $\k\in(0,\k_0]$ the unique local solution on the interval $[x_\ast-\tilde r,x_\ast+\tilde r]$ associated with sonic point $x_\ast\in[\xmin,\xmax]$ 
%satisfies
Let $\tilde r>0$ be as above. Then
\be\label{E:AUXPRELIM}
\frac13<\frac{J[x;\frac13]}{x} \ \ \text{ for all } \ x\in[0,\xmax-\tilde r].
\ee
\end{lemma}

%%%%%%%%%%%%%%%%%%%%%%%%%%

\begin{proof}
By~\eqref{E:FDEF}, bound~\eqref{E:AUXPRELIM} is equivalent to showing
\begin{align}
\frac19+\frac23\e(1+\frac13) < \k + \frac{(\frac13)^{-\e}}{(1-\k)x^2}. 
\end{align}
Since $x\le\xmax-\tilde r=3+O(\k)-\tilde r$ the right-hand side above is larger than $\k+\frac{(\frac13)^{-\e}}{(1-\k)(3+O(\k)-r)^2}$, which converges to $\frac1{(3-r)^2}$ as $\k\to0^+$.
The left-hand side on the other hand converges to $\frac19$ and thus the claim follows.
\end{proof}

%%%%%%%%%%%%%%%%%%%%%%%%%%%
%%%%%%%%%%%%%%%%%%%%%%%%%%%

\begin{lemma}\label{L:PRELIM1}
For any $x_\ast\in[\xmin,\xmax]$ consider the unique local solution on the interval $[x_\ast-\tilde r,x_\ast+\tilde r]$  given by Theorem~\ref{T:SONICLWP}. Then for any $x\in(s(x_\ast),x_\ast]$ the 
following bounds hold:
\begin{align}
W(x) & < \frac{\F (x)}{x},\label{E:WBOUNDL}\\
|W(x)| & <  \frac{\F (x)}{x} + 2\e(1+\R), \label{E:WBOUNDL2} \\
0&<\R(x),
\end{align}
where $\F $ is defined in~\eqref{E:FDEF}. Moreover, for any $x\in (s(x_\ast),x_\ast-\delta]$ such that  $\R(x)\ge \frac{1}{3}$ we have the upper bound
\begin{align}
\R(x) &< %\begin{cases} 
\frac{\F (x)}{x},  \label{E:RBOUNDL} %& \text{ if } \ R\ge\frac13, \\ \frac13  & \text{ if } \ %R<\frac13 \end{cases} 
\end{align}
%for any $x\in(s(x_\ast),x_\ast -\delta]$ 
where $0<\delta\ll 1$ is an $\k$-independent constant from Lemma \ref{L:INITIAL}. 
%Furthermore, for any $\mathring x\in(s(x_\ast),x_\ast)$ there exists a constant $M=M(x_\ast,\mathring x)$ such that 
%\begin{align}\label{E:MBOUND}
%|W(x)|<M, \ \ R(x)<M.
%\end{align}
\end{lemma}

%%%%%%%%%%%%%%%%%%%%%%%%%%%

\begin{proof}
Let $x_\ast\in[\xmin,\xmax]$ and let $\mathring x\in(s(x_\ast),x_\ast)$. By Definition~\ref{D:SONICTIME} there exists a $\kappa>0$ such that 
$B(x)>\kappa$ for all $x\in[\mathring x,x_\ast-\tilde r)$, which according to~\eqref{E:BDEF} is equivalent to the bound
\[
\left(\F -xW\right)\left(\F  + 2\e x(1+\R) + xW\right)>\frac{\kappa}{1-\k}. 
\]
If $W(x)>0$ then from the strict positivity of $\F $ and the above bound we immediately have $W(x)<\frac{\F (x)}{x}$. If on the other hand $W(x)\le0$ then $\F -xW>0$ and therefore from 
the above bound again
$
|xW| = - xW < \F  +  2\e x(1+\R).
$
The two bounds together imply 
\be
|W(x)| \le \frac{\F (x)}{x} + 2\e(1+\R), \ \ x\in(s(x_\ast), x_\ast),
\ee
which shows~\eqref{E:WBOUNDL2}.
The strict positivity of $\R$ on $(s(x_\ast),x_\ast]$ follows by rewriting the equation~\eqref{E:RODE}
in the form
\begin{align}\label{E:LOGRFORMULA}
\frac{d}{dx}\left(\log \R\right) & = -  \frac{ 2x(1-\k) (W + \k) (\R-W) }{B}.
\end{align}
Finally, to prove~\eqref{E:RBOUNDL}, we observe that it suffices to show $f>0$ where we recall the formula $f=\F -x\R$.  % Note that $f(x_\ast -\delta) >0$ by Lemma \ref{L:INITIAL} and $b[x;\R]<0$ by Lemma \ref{L:LITTLEBPOS} (here we use~\eqref{E:WBOUNDL} and $\R\ge \frac13$). 
If $D(x)=\frac13$, by Lemma~\ref{L:AUXPRELIM}, we are done. If $D(x)>\frac13$, we consider two cases. First suppose $D>\frac13$ on $[x,x_\ast-\delta]$. Then $b<0$ %$b[x;\R]<0$ 
on $[x,x_\ast-\delta]$ by Lemma \ref{L:LITTLEBPOS} and \eqref{E:WBOUNDL}, and $f(x_\ast -\delta) >0$ by Lemma \ref{L:INITIAL}. Hence, by using  Corollary~\ref{C:FFORMULA}, we have 
\[
f(x) > f(x_\ast -\delta) e^{\int_{x}^{x_\ast -\delta} a[z;\R,W] \,dz} >0. 
\]
If $D\ngtr \frac13$ on $[x,x_\ast-\delta]$, 
there should exist $x_1\in (x,x_\ast-\delta]$ such that $D>\frac13$ on $[x,x_1)$ and $D(x_1)=\frac13$. Note that $b<0$ on $[x,x_1)$ and $f(x_1)>0$ by Lemma~\ref{L:AUXPRELIM}. %Hence, 
By using Corollary~\ref{C:FFORMULA} again, 
%we obtain for any $x\in (s(x_\ast),x_\ast-\delta)$, 
%\be\label{E:FXONEXTWO}
%f(x) > f(x_\ast -\delta) e^{\int_{x}^{x_\ast -\delta} a[z;\R,W] \,dz} >0
%\ee
% we use Corollary~\ref{C:FFORMULA} and observe that for any $x\in(s(x_\ast),x_\ast-r)$ we have $b[x;R]<0$ whenever $R\ge\frac13$ (here we use~\eqref{E:WBOUNDL}).
%on any interval $[x_1,x_2]\subset (s(x_\ast),x_\ast-\tilde r)$ where $b$ is positive and $f(x_2)>0$, we  %haveWe 
we obtain
\begin{align}\label{E:FXONEXTWO}
f(x)> f(x_1) e^{\int_{x}^{x_1} a[z;\R,W]\,dz} >0, 
\end{align}
%By Lemma~\ref{L:AUXPRELIM} at any point $x_2$ where $\R(x_2)=\frac13$ we have $f(x_2)>0$ 
which proves the claim. %and the claim therefore follows from~\eqref{E:FXONEXTWO} and the formula $f = F-xR$.
%\ec
\end{proof}

%%%%%%%%%%%%%%%%%%%%%%%%%%%%%%%%%%
%%%%%%%%%%%%%%%%%%%%%%%%%%%%%%%%%%

\begin{lemma} \label{L:APRIORIMINUS}
Let $\k\in(0,\k_0]$, where $\k_0>0$ is a small constant given by Theorem~\ref{T:SONICLWP}.
For any $x_\ast\in[\xmin,\xmax]$ consider the unique local RLP-type solution associated to $x_\ast$ given by Theorem~\ref{T:SONICLWP}.  If  $\R(x)\geq \frac13$ for some  $x\in(s(x_\ast),x_\ast-\delta)$, then 
\be\label{E:FtoRho}
\R(x) < \frac{1}{x^{1-\k}}.
\ee
\end{lemma}

%%%%%%%%%%%%%%%%%%%%%%%%%%%%%%%%%%

\begin{proof}
Since $\R(x)\ge\frac13$, by Lemma~\ref{L:PRELIM1}, we have $\frac{\R}{\F [x;\R]} <\frac{1}{x}$.
Using the definition of $\F [\cdot;\R]$~\eqref{E:FDEF} it is easy to see that the inequality  $\frac{\R}{\F [x;\R]} <\frac{1}{x}$ is equivalent to
%If $\frac{\rho}{F[\rho]} <\frac{1}{x_\ast z}$, then from \eqref{Frho}, 
\begin{align*}
%&\frac{\rho}{- \frac{2k}{1-k} (1+\rho) x_\ast z + \sqrt{  \left(\frac{2k}{1-k}\right)^2 (1+\rho)^2 x_\ast^2 z^2 + kx_\ast^2 z^2 + \frac{\rho^{-\frac{2k}{1-k}}}{1-k}}} < \frac{1}{x_\ast z} \\
%\Longleftrightarrow  \ \ &
%\frac{\frac{2k}{1-k} (1+\rho) x_\ast z + \sqrt{  \left(\frac{2k}{1-k}\right)^2 (1+\rho)^2 x_\ast^2 z^2 + kx_\ast^2 z^2 + \frac{\rho^{-\frac{2k}{1-k}}}{1-k}}}{{   kx_\ast^2 z^2 + \frac{\rho^{-\frac{2k}{1-k}}}{1-k}}} < \frac{1}{\rho x_\ast z} \\
%\Longleftrightarrow  \ \ & 
\sqrt{ \e^2 (1+\R)^2 x^2 + \k x^2 + \tfrac{\R^{-\e}}{1-\k}}  < \frac{ \k x^2 + \frac{\R^{-\e}}{1-\k} }{\R x} - \e (1+\R) x. 
\end{align*}
This in turn implies
\begin{align*}
&  \k x^2 + \tfrac{\R^{-\e}}{1-\k } <  \frac{( \k x^2 + \frac{\R^{-\e}}{1-\k })^2 }{\R^2 x^2} -  2\e \tfrac{ (1+\R)}{\R} ( \k x^2 + \tfrac{\R^{-\e}}{1-\k })\\
\Longleftrightarrow  \ \ & 1< \frac{( \k x^2 + \frac{\R^{-\e}}{1-\k }) }{\R^2 x^2} - 2\e \tfrac{ (1+\R)}{\R} \\
\Longleftrightarrow  \ \ & \left[ (1+3\k )\R^2 + 4\k  \R -\k (1-\k ) \right] \R^{\e} < \frac{1}{x^2}. 
\end{align*} 
Now we note that 
\[
(1+3\k )\R^2 + 4\k  \R -\k (1-\k ) = \R^2 + \k  \left[ 3\R^2 + 4\R - (1-\k )\right] \geq  \R^2 
\]
where we have used $\R\geq \frac13$. This implies~\eqref{E:FtoRho}.
\end{proof}

%%%%%%%%%%%%%%%%%%%%%%%%%%%%%%%%%%
%%%%%%%%%%%%%%%%%%%%%%%%%%%%%%%%%%

%%%%%%%%%%%%%%%%%%%%%%%%%%%
%%%%%%%%%%%%%%%%%%%%%%%%%%%

\begin{remark}
It is a priori possible that the solution blows-up at a point at which $B$ remains strictly positive, for example through blow-up of $W$. 
It is trivial to see that this cannot happen in the Newtonian setting, but in the relativistic case it requires a careful argument, which is given in the next lemma.
\end{remark}

%%%%%%%%%%%%%%%%%%%%%%%%%%%
%%%%%%%%%%%%%%%%%%%%%%%%%%%

\begin{lemma}[No blow up before the sonic point]\label{L:OBSTRUCTION}
For any $x_\ast\in[\xmin,\xmax]$ consider the unique RLP-type solution on the interval $(s(x_\ast),x_\ast]$. 
If $s(x_\ast)>0$, then
\begin{align}
\liminf_{x\to s(x_\ast)^+}B(x)=0.
\end{align}
\end{lemma}

%%%%%%%%%%%%%%%%%%%%%%%%%%%

\begin{proof} 
Assume the opposite. In that case there exists a constant  $\kappa>0$ such that $B(x)\ge\kappa$ for all $x\in(s(x_\ast),x_\ast-\tilde r]$.  Our goal is to show that $|\R(x)|+|W(x)|<\infty$ on $(s(x_\ast), x_\ast]$, which would lead to the contradiction. 

\noindent{\it Step 1. Boundedness of $\R$.} If $\R(x)\ge \frac13$ the bound~\eqref{E:RBOUNDL} gives $\R+\e(1+\R)<\sqrt{\e^2 (1+\R)^2+\k +\frac{\R^{-\e}}{(1-\k)x^2}}$,
which upon taking a square and using $2\e \R(1+\R)>0$ leads to
\[
\R^2 < \k + \frac{\R^{-\e}}{(1-\k)x^2}< \k + \frac{\R^{-\e}}{(1-\k)s(x_\ast)^2}.
\] 
Since $0<\e\ll1$, this gives a uniform upper bound on $\R$ on $(s(x_\ast),x_\ast-\tilde r]$
\be\label{RboundM}
\R(x)\le M, \ \ x\in(s(x_\ast), x_\ast].
\ee
%To prove the upper bound on $|W|$, assume that there exists a sequence $\{x_n\}_{n\in\mathbb N}\in[\mathring x,x_\ast)$ such that $W(x_n)>n$ for all $n\in\mathbb N$. By~\eqref{E:WBOUNDL2} this implies that $\limsup_{n\to\infty}F(x_n)=\infty$
%and since $R$ is bounded from above, it follows from~\eqref{E:FDEF} that $\liminf_{n\to\infty}R(x_n)=0$. On the other hand, 

\noindent{\it Step 2. Boundedness of $|W|$.} 
It follows from~\eqref{E:RODE}--\eqref{E:WODE} that there exists a sufficiently large value $N=N(s(x_\ast),x_\ast)>0$ such that
$W'<0$, $\R'>0$ if $W>N$ and $W'>0$, $\R'>0$ if $W<-N$, where we use the already shown upper bound on $\R$. In both cases, the two regions are dynamically trapped and we denote the union of the two regions by $I_N$.
For any $x\in I_N$ we multiply~\eqref{E:RODE} by $\frac{1+\k}{(1-\k)\R}$,~\eqref{E:WODE} by $\frac1W$, and sum them to obtain
\begin{align}\label{E:LOGID0}
\left(\log\left(\R^{2(1+\e)}W^2\right)\right)' = \frac{2}x\left(\frac1W-3\right) <0 \ \ \text{ for } \ |W|>N,
\end{align}
with $N$ sufficiently large.
In particular, for any $s(x_\ast)<x_1<x_2<x_\ast$, where $x_1,x_2$ both belong to the invariant region $I_N$ above, we obtain
\begin{align}\label{E:RLOWER}
\R(x_1)^{2(1+\e)}W(x_1)^2> \R(x_2)^{2(1+\e)}W(x_2)^2
\end{align}
On the other hand, since $B\ge\kappa $, $|W|>N$, and $\R<M$  there exists a universal constant constant $C=C(s(x_\ast),x_\ast)$ such that for a sufficiently large
choice of  $N$, rom \eqref{E:BDEF0}  we have
\be\label{E:RWINV}
\R^{-\e}>C W^2, \ \ x\in I_N.
\ee
We apply this to~\eqref{E:RLOWER} to conclude that 
\[
\frac1C \R(x_1)^{2+\e} > \R(x_2)^{2(1+\e)}W(x_2)^2, \ \ \text{ for any } \ x_1\in(s(x_\ast),x_2).
\]
This gives a  lower bound for $\R$ and therefore an upper bound for $|W|$ via~\eqref{E:RWINV} in the region $I_N$. In particular $\limsup_{x\to s(x_\ast)^+}|W(x)|<\infty$ and the claim follows.
%However, for any $\gamma>0$ using~\eqref{E:RODE} we may compute
%\begin{align}\label{E:RPRIMEGAMMA}
%\frac{R'x}{R} - \gamma = \frac{-\gamma B - 2(1-\k)x^2(W+\k)(R-W)}{B}.
%\end{align}
%We focus on the numerator of the right-hand side above. By~\eqref{E:BDEF} we have
%\begin{align}
%& -\gamma B - 2(1-\k)x^2(W+\k)(R-W) \notag \\
%& =- \gamma R^{-\e} +\gamma \left[ ( W + \k)^2 - \k (W - 1)^2 + 4\k R W \right] x^2- 2(1-\k)x^2(W+\k)(R-W) \notag \\
%& = - \gamma R^{-\e} + (1-\k)(\gamma-2) x^2 W^2
%\end{align}
\end{proof}

%%%%%%%%%%%%%%%%%%%%%%%%%%%
%%%%%%%%%%%%%%%%%%%%%%%%%%%

Essentially a consequence of the previous lemma and a standard ODE argument is the statement that as long as the sonic denominator $B$ is bounded below by some constant $\delta>0$ for all $x\ge \bar x>s(x_\ast)$, we can extend
the solution to the left to some interval $[\bar x-t,x_\ast]$, where $t>0$ depends only on $\delta$ and $\bar x$. The statement and the proof are analogous to Lemma 4.3 in~\cite{GHJ2021} and we state it without proof.

%%%%%%%%%%%%%%%%%%%%%%%%%%%
%%%%%%%%%%%%%%%%%%%%%%%%%%%

\begin{lemma}\label{L:TQUANTITATIVE}
Let $x_\ast\in[\xmin,\xmax]$ be given and consider the unique RLP-type solution $(\R(\cdot;x_\ast),W(\cdot;x_\ast))$ to the left of $x=x_\ast$, given by Theorem~\ref{T:SONICLWP}.
Assume that for some $\bar x\in(0,x_\ast-\tilde r)$ and $\delta>0$ we have $\bar x>s(x_\ast)$ and the conditions
\begin{align}\label{E:ETAZERO}
B(x)>\delta, \ \ \R(x)>0, \ \ x\in[\bar x,x_\ast-\tilde r],
\end{align}
hold. Then there exists a $ t=t(\delta,\bar x)>0$ such that the solution can be continued to the interval $[\bar x- t,x_\ast]$ so that 
\begin{align}
B(x)>0, \ \ \R(x)>0, \ \ x\in[\bar x- t,x_\ast-\tilde r]. \notag 
\end{align}
\end{lemma}

%%%%%%%%%%%%%%%%%%%%%%%%%%%
%%%%%%%%%%%%%%%%%%%%%%%%%%%

%%%%%%%%%%%%%%%%%%%%%%%%%%%
%%%%%%%%%%%%%%%%%%%%%%%%%%%

\begin{proposition}\label{P:LOWERSEMICONTRLP}
Let $x_\ast\in[\xmin,\xmax]$ be given and consider the unique RLP-type solution $(\R(\cdot;x_\ast),W(\cdot;x_\ast))$ to the left of $x=x_\ast$.
\begin{enumerate}
\item[{\em (a)}] (Upper semi-continuity of the sonic time). Then 
\[
\limsup_{\tilde x_\ast \to x_\ast} s(\tilde x_\ast) \le s(x_\ast),
\]
i.e. the map $x_\ast\to s(x_\ast)$ is upper semi-continuous. In particular, if $s(x_\ast)=0$ then the map $s(\cdot)$ is continuous at $x_\ast$. 

\item[{\em (b)}] ([Continuity of the flow away from the sonic time])
Let $\{x^n_\ast\}_{n\in\mathbb N}\subset[\xmin,\xmax]$ and $x_\ast\subset[
\xmin,\xmax]$ satisfy $\lim_{n\to\infty}x^n_\ast = x_\ast$. Let 
$x_\ast-\tilde r>z>\max\{s(x_\ast),\sup_{n\in\mathbb N}s(x_\ast^n)\}$. Then 
\begin{align}
\lim_{n\to\infty}W(x;x_\ast^n) = W(x;x_\ast), \ \ \lim_{n\to\infty}\R(x;x_\ast^n) = \R(x;x_\ast).\notag 
\end{align}

\item[{\em (c)}] 
Let $\{x^n_\ast\}_{n\in\mathbb N}\subset[\xmin,\xmax]$ and $x_\ast\subset[\xmin,\xmax]$ satisfy $\lim_{n\to\infty}x^n_\ast = x_\ast$. Assume that there exist $0<\hat x<x_\ast-\tilde r$ and $\kappa>0$ such that $s(x_\ast^n)<\hat x$ for all $n\in\mathbb N$ and the following uniform bound holds:
\begin{align}\label{E:LOWERBOUNDN}
B[x;x_\ast^n,W,\R]>\kappa,  \ \ n\in\mathbb N, \ \ x\in[\hat x ,x_\ast-\tilde r].
\end{align}
Then there exists a $T=T(\kappa,\hat x)>0$ such that 
\begin{align}\label{E:CLAIMUNIFCONT}
s(x_\ast)<\hat x-T, \ \ s(x_\ast^n)<\hat x-T, \ \ n\in\mathbb N.
\end{align}
\end{enumerate}
\end{proposition}

%%%%%%%%%%%%%%%%%%%%%%%%%%%
%%%%%%%%%%%%%%%%%%%%%%%%%%%

\subsection{The Friedmann shooting argument}

%%%%%%%%%%%%%%%%%%%%%%%%%%%
%%%%%%%%%%%%%%%%%%%%%%%%%%%

The basic idea of this section is inspired by a related proof in our earlier work on
the existence of Newtonian Larson-Penston solutions~\cite{GHJ2021}. We start by
recalling Definition~\ref{D:XYZ}.
%For any fixed $0<\k\le\k_0$ we partition the set $[\xmin,\xmax]$ into the following three subsets.

%\begin{definition}[$\X $, $\Y $, $\Z$, and $X $]\label{D:XYZ}
%Let $\k_0>0$ be a small constant given by Theorem~\ref{T:SONICLWP}. For any $\k\in(0,\k_0]$ and $x_\ast\in[\xmin,\xmax]$ we consider the associated RLP-type solution $(\R(\cdot;x_\ast), W(\cdot;x_\ast))$ given by Theorem~\ref{T:SONICLWP}.
%We introduce the sets
%\begin{align}\label{E:XDEF}
%\X & : = \left\{x_\ast\in[\xmin,\xmax]\,\Big| \inf_{x\in(s(x_\ast),x_\ast)}W(x;x_\ast)>\frac13 \right\}, \\
%\label{E:YDEF}
%\Y &: = \left\{x_\ast\in[\xmin,\xmax]\,\big| \   \exists  \ x\in (s(x_\ast), x_\ast) \ \text{such that } W(x;x_\ast)=\frac13\right\} \\
%\label{E:ZDEF}
%\Z &: = \left\{x_\ast\in[\xmin,\xmax]\,\Big| W(x;x_\ast)>\frac13 \ \text{ for all } \ x\in(s(x_\ast),x_\ast) \ \text{ and } \ \inf_{x\in(s(x_\ast),x_\ast)}W(x;x_\ast)\le\frac13\right\}.
%\end{align}
%Finally, we introduce the \underline{fundamental set} $X\subset \Y$ given by
%\begin{align}
%X: = \left\{x_\ast\in[\xmin,\xmax]\,\big| \ \tilde x_\ast \in \Y \ \ \text{ for all } \ \ \tilde x_\ast\in [x_\ast,\xmax] \right\}. 
%\end{align}
%\end{definition}

%%%%%%%%%%%%%%%%%%%%%%%%%%%
%%%%%%%%%%%%%%%%%%%%%%%%%%%

%%%%%%%%%%%%%%%%%%%%%%%%%%%
%%%%%%%%%%%%%%%%%%%%%%%%%%%

\begin{lemma}[$\X $ and $\Y$ are nonempty] \label{L:XYNONEMPTY}
There exists an $0<\k_0\ll1$ sufficiently small so that the following statements hold for any $0<\k\le\k_0$: 
%be fixed and for any $x_\ast\in[\xmin,\xmax]$ consider the unique local RLP-type solution given by Theorem~\ref{T:SONICLWP}.
%Then 
\begin{enumerate}
\item[{\em (a)}] There exists a $\kappa=\kappa(\k)>0$ such that 
$(\xmax(\k)-\kappa,\xmax(\k)]\subset X \subset\Y$. 
\item[{\em (b)}] Moreover,
$x_\text{min} \in \X $. 
\end{enumerate}
\end{lemma}

%%%%%%%%%%%%%%%%%%%%%%%%%%%

\begin{proof}
{\em Proof of part (a).}
We use the mean value theorem to write 
\[
W(x;x_\ast) = W_0 + W'(\bar x; x_\ast) (x-x_\ast), \ x\in (s(x_\ast), x_\ast) 
\] 
for some $\bar x\in (x,x_\ast)$.  
Note that $ W'(x_\ast; x_\ast) =\frac1{x_\ast} \WH_1 = \frac1{x_\ast}\left(1-\frac2{x_\ast}+O(\k)\right)=\frac1{x_\ast}-\frac2{x_\ast^2}+O(\k)$. By~\eqref{E:XMAXDEF}, it follows that
$W'(\xmax(\k);\xmax(\k)) = \frac19+O(\k)$.
 By Theorem \ref{T:SONICLWP} there exist small enough $r>0$ and $\delta_1>0$ such that $W'(x; x_\ast)>\frac1{18}$ for all $x\in(\xs-r,\xs]$ and $x_\ast\in [\xmax(\k)-\delta_1,\xmax(\k)]$. 
 For such $\xs$ and $x$, we have 
\[
W(x;x_\ast) \le \frac{1}{x_\ast} +\frac1{18} (x-\xs),
\]
where we have used $W_0=\WH_0 <\frac1{\xs}$ for $\k\in(0,\k_0]$, see Lemma~\ref{R0W0}.
Note $ \frac{1}{x_\ast} + \frac1{18} (x-\xs) = \frac 13$ when $x=\tilde x(x_\ast)= \xs- \frac{6(3-x_\ast)}{x_\ast}$. Therefore for all $x_\ast \in (\xmax(\k)-\kappa, \xmax(\k)]\subset(3-\kappa,3)$ with 
$\kappa = \min \{ \delta_1, \xmax(\k)-3+\frac{3r}{6+r}\}$, there exists an $x\ge \tilde x(x_*)$ such that $W(x;x_\ast)=\frac13$, which shows the claim. Here we have used~\eqref{E:XMAXPROP} and the smallness
of $0<\k_0\ll1$.

{\em Proof of part (b).}
We rewrite~\eqref{E:WODE} in the form
\begin{align}\label{5.267}
xW' = 1- 2W - W \frac{B -  2 x^2(1+\k)(W + \k)(\R- W)}{B}. 
\end{align}
We now recall Lemma~\ref{L:BALGEBRA} and express the numerator above in the form 
\begin{align}
&B -  2x^2(1+\k) (W + \k)(\R- W) \notag\\
&= (1-\k) (\F [\R]-x W)(H[\R] +   x  W) 
 - 2(1+\k)x (   W + \k ) (  x \R  - \F [\R ] + \F [\R ] -   x  W) \notag\\
& = 2(1+\k)x( W + \k ) f+  (\F [\R ]-  x  W  )  \left\{ (1-\k) (H[\R ] +   x  W) - 2(1+\k) x(  W + \k)\right\}. \notag
% \label{5.271}
\end{align}
where we recall $f(x) = \F [x;\R]-x\R$.
Using~\eqref{E:HDEF}, 
\begin{align*}
 (1-\k) (H[\R ] +   x  W) - 2(1+\k) x(  W + \k)&= (1-\k) ( \F [\R ] +2\e (1+\R )  x+   x  W) - 2(1+\k)x ( W + \k) \\
&=(1-\k) ( \F [\R ] -   x  W) + 4\k (  x \R  -   x W) + 2\k(1-\k)   x \\
&= (1+3\k) ( \F [\R ] -   x  W) -4\k f + 2\k(1-\k)   x.
\end{align*}
Putting these together, and by \eqref{E:BDEF} and $W>0$,  \eqref{5.267} reads as 
\begin{align}
x W' &= 1- 2 W + 2\e \frac{f}{H[\R ] +   x  W}  W \notag \\
&-  W \frac{ (1+3\k)( \F [\R ] -   x  W)^2 +2(1+\k)  x(  W + \k )(\F [\R ] -   x \R )  + 2\k(1-\k)( \F [\R ] -   x  W)x }{B} \notag \\
&\le 1- (2-2\e)  W,
\end{align}
where we have used $0<f=\F [\R ] -  x \R < H[\R ] + x W$ and $B>0$ for $x\in (s(x_\ast),x_\ast)$. Hence the set 
\be\label{E:INVSET}
\left\{x\in(s(x_\ast),x_\ast))\, \big| W(x;x_\ast)>\frac{1}{2-2\e}\right\}
\ee 
is an invariant set. 

We now use Lemma \ref{L:XMINN} and the definition of $\xmin$ to obtain $W(x_0;\xmin) >\frac{1}{2-2\e}$ for some $x_0\in (s(x_\ast),x_\ast))$. Together with the invariance of~\eqref{E:INVSET} we conclude that $x_{\min} \in \X $.
\end{proof}

%%%%%%%%%%%%%%%%%%%%%%%%%%%
%%%%%%%%%%%%%%%%%%%%%%%%%%%

\begin{remark}
A corollary of Definition~\ref{D:XYZ} and Lemma~\ref{L:XYNONEMPTY} is that the fundamental set $X $ is a simply connected subset of $(\xmin,\xmax]$ which contains the point $\xmax$. The set $X$ is in particular an interval.
%The fundamental set $Y$ is theref
\end{remark}

%%%%%%%%%%%%%%%%%%%%%%%%%%%
%%%%%%%%%%%%%%%%%%%%%%%%%%%
%%%%%%%%%%%%%%%%%%%%%%%%%%%
%%%%%%%%%%%%%%%%%%%%%%%%%%%

Of fundamental importance in our analysis are the following two quantities.
\begin{definition}[Critical point and Friedmann time]\label{D:CRITICAL}
Let $\k_0>0$ be a small constant given by Theorem~\ref{T:SONICLWP}.  For any $\k\in(0,\k_0]$ we introduce the \underline{critical point}
\begin{align}\label{critical_x}
\bar x_\ast : = \inf_{x_\ast\in X } x_\ast,
\end{align}
and the \underline{Friedmann time}
\begin{align}
\xf &  = \xf(x_\ast) : = \inf \left\{x \in (s(x_\ast), x_\ast) \, \big| \ 
\ W(\tau;x_\ast)>\frac13 \ \text{ for } \ \tau \in(x,x_\ast) \right\}. \label{E:FRIEDMANTIME}
\end{align}
\end{definition}

We will show later, the unique local RLP-type solution associated with the sonic point $\bar x_\ast$ is in fact {\em global} and extends all the way to 
the origin $x=0$. The Friedmann time introduced above plays a crucial role in the proof.
%, for any $x_\ast\in[\xmin,\xmax]$ we also introduce the Friedman time

%%%%%%%%%%%%%%%%%%%%%%%%%%%
%%%%%%%%%%%%%%%%%%%%%%%%%%%

The sets $\Y$ and $\Z$ enjoy several important properties which we prove in the next lemma.

%%%%%%%%%%%%%%%%%%%%%%%%%%%
%%%%%%%%%%%%%%%%%%%%%%%%%%%

\begin{lemma}\label{L:PREPX}
Let $\k_0>0$ be a small constant given by Theorem~\ref{T:SONICLWP}. Then for any $\k\in(0,\k_0]$ the following statements hold.
\begin{enumerate}
\item[{\em (a)}] For any $x_\ast\in \Y\cup \Z$ we have 
\begin{align}
W(x;x_\ast)&<\R(x;x_\ast), \ \ x\in (s(x_\ast), x_\ast), \label{E:WLESSTHANRX}\\
%s^+(y_\ast)& <y_{\frac13}(y_\ast) \ \text{ if } \ s^+(y_\ast)>0, \label{E:EXISTENCETIME} \\
W(x;y_\ast) & <\frac13, \ \ x\in (s(x_\ast), x_{F}(x_\ast)),
%\ \text{ if } \ s^+(y_\ast)>0.
\label{E:WLESSTHANTHIRDX}
\end{align}
where~\eqref{E:WLESSTHANTHIRDX} is considered trivially true in the case $s(x_\ast)=x_{F}(x_\ast)$.
\item[{\em (b)}]
For any $x_\ast\in \Y$ we have $W'(x_{F}(x_\ast);x_\ast)>0$. Moreover, the set $\Y$ is relatively open in $[\xmin,\xmax]$. 
%\item[{\em (c)}]
%The map 
%\[
%Y\ni y_\ast \mapsto z_{\frac13}(y_\ast) 
%\]
%is continuous and 
%\be
%\lim_{Y\ni y\to \bar y_\ast} z_{\frac13}(y) = 0 = z_{\frac13}(\bar y_\ast).
%\ee
\end{enumerate}
\end{lemma}

%%%%%%%%%%%%%%%%%%%%%%%%%%%

\begin{proof}
{\it Proof of Part (a).} Let $x_\ast\in \Y\cup \Z$. Then from Lemma \ref{R1W1} we know that 
$W(x) <\R(x)$ for all $x\in [ (1-\bar r) x_\ast, x_\ast)$ for some $\bar r\le r$ where $r$ is given by Theorem \ref{T:SONICLWP}. We first prove \eqref{E:WLESSTHANRX}. By way of contradiction, assume that there exists $x_c\in (s(x_\ast), x_\ast)$ such that 
\[
W(x_c) = \R(x_c) \ \text{ and } \ W(x)< \R(x), \ x \in (x_c, x_\ast). 
\]
We distinguish three cases. 

\noindent{\it Case 1: $x_c\in (\xf, x_\ast)$.} In this case, from \eqref{E:RODE} and \eqref{E:WODE}, we have $W'(x_c) <0$ and $\R'(x_c)=0$. In particular, $(\R-W)'(x_c)>0$ and locally strictly to the left of $x_c$ we have 
\[
W' <0, \ W - \R >0, \ \R'>0, \ W>\frac13 . 
\]
We observe that these conditions are dynamically trapped and since $W'<0$, we deduce that $W$ stays strictly bounded away from $\frac13$ from above for all $x\in (\xf, x_\ast)$, which is a contradiction to the assumption $x_\ast\in \Y\cup \Z$. 

\noindent{\it Case 2: $x_c=\xf$.} In this case, $x_\ast \in \Y$ necessarily and $W(x_c) = \R(x_c)=\frac13$. On the other hand, since $\R-W>0$ for $x\in (\xf, x_\ast)$, it follows $\R'<0$ on $(\xf, x_\ast)$ from \eqref{E:RODE}. Hence, $\R(x_c) > \R(x_\ast)=\WH_0 \ge \frac13$ since $x_\ast\in[\xmin,\xmax]$. This is a contradiction. 

\noindent{\it Case 3: $x_c\in (s(x_\ast), \xf)$.} In this case, $x_\ast \in \Y$ necessarily. Since $x_c<\xf$ we know that $\R-W>0$ locally around $\xf$. Therefore, we have 
\[
W' >0, \ \R-W >0, \ W<\frac13 \ \text{ on } (\xf-\delta, \xf)
\]
for a sufficiently small $\delta>0$. These conditions are dynamically trapped and we conclude that $\R-W>0$ on $(s(x_\ast),\xf)$. This is a contradiction, and hence completing the proof of \eqref{E:WLESSTHANRX}. 

Inequality \eqref{E:WLESSTHANTHIRDX} follows by a similar argument, since the property $W' >0, \ \R-W >0, \ W<\frac13$ is dynamically preserved and all three conditions are easily checked to hold locally to the left of $\xf(x_\ast)$.  

\noindent{\it Proof of Part (b).} For any  $x_\ast\in \Y$ by part (a) and  \eqref{E:WODE} we have 
$W'( \xf(x_\ast); x_\ast) >0$. Therefore there exists a sufficiently small $\delta>0$ so that $W(x; x_\ast)<\frac13$ for all $x\in (\xf(x_\ast)-\delta, \xf(x_\ast))$. By Proposition \ref{P:LOWERSEMICONTRLP}, there exists a small neighborhood of $x_\ast$ such that $W(x;x_\ast)<\frac13$ for some $x\in (\xf(x_\ast)-\delta, \xf(x_\ast))$. Therefore $\Y$ is open. 
\end{proof}

%%%%%%%%%%%%%%%%%%%%%%%%%%%
%%%%%%%%%%%%%%%%%%%%%%%%%%%

\begin{lemma}\label{L:PRELIM2}
Let $\k_0>0$ be a small constant given by Theorem~\ref{T:SONICLWP} and for any $x_\ast\in[\xmin,\xmax]$ consider the unique local RLP-type solution given by Theorem~\ref{T:SONICLWP}.
If $x_F(x_\ast)=s(x_\ast)>0$ then necessarily
\be
W(x_F(x_\ast);x_\ast)>\frac13.
\ee
In particular if $x_\ast\in X $, then necessarily $s(x_\ast)<x_F(x_\ast)$.
\end{lemma}

%%%%%%%%%%%%%%%%%%%%%%%%%%%

\begin{proof}
Assume the claim is not true, in other words $W(x_F(x_\ast);x_\ast)=\frac13$. Since $\R$ is decreasing on  $(x_F,x_\ast-\delta]$  by Lemma~\ref{L:PREPX} and~\eqref{E:RODE} and since  $\R(x_\ast) \ge 1/3$  it follows that 
 $\lim_{x\to x_F}\R(x)>\frac13$.
In particular, since $\F \ge x\R$ on $(s(x_\ast),x_\ast-\delta)$  by \eqref{E:RBOUNDL},  it follows from~\eqref{E:BDEF} that $B(x;x_\ast) \ge  (1-\k)x(\R-W)(H(x)+xW)$.
Therefore
\[
\liminf_{x\to x_F} B(x)>0,
\]
a contradiction to the assumption $x_F(x_\ast)=s(x_\ast)$ due to Lemma~\ref{L:OBSTRUCTION}. The second claim is a consequence of the just proven claim and the definition of the set $X $.
\end{proof}

%%%%%%%%%%%%%%%%%%%%%%%%%%%
%%%%%%%%%%%%%%%%%%%%%%%%%%%

\begin{lemma}\label{L:PRELIM3}
Let $\k_0>0$ be a small constant given by Theorem~\ref{T:SONICLWP} and for any $x_\ast\in\X \subset[\xmin,\xmax]$ consider the unique local RLP-type solution given by Theorem~\ref{T:SONICLWP}.
Then there exists a monotonically increasing continuous function 
\[
m:(0,x_\ast-\tilde r]\to(0,\infty),
\]
such that for any $X\in(0,x_\ast-\tilde r]$
\begin{align}
B(x;x_\ast) >m (X)>0 \ \ \text{ for all } \ \ x\in(\max(X,s(x_\ast)),x_\ast-\tilde r].
\end{align}
\end{lemma}

%%%%%%%%%%%%%%%%%%%%%%%%%%%

\begin{proof}
Since $x_\ast\in\Y$, by Lemma~\ref{L:PREPX} $\F -xW\ge \F -x\R = f$ on $(s(x_\ast),x_\ast-\tilde r)$. Since $H+xW\ge \F -xW\ge f$ we conclude from~\eqref{E:BDEF}
that $B\ge (1-\k)f^2$. Since $W(x;x_\ast)\ge\frac13$ for all $x\in[x_F(x_\ast), x_\ast-\tilde r]$, it follows from Lemma~\ref{L:PREPX} that $\R>\frac13$ on the same interval. Therefore,
by Lemma~\ref{L:LITTLEBPOS} we conclude that $b[x;\R,W]<0$ on $[x_F(x_\ast), x_\ast-\tilde r]$. By Corollary~\ref{C:FFORMULA} we then conclude
\begin{align}\label{E:FLB1}
f(x)\ge f(x_\ast-\tilde r) e^{\int_{x}^{x_\ast-\tilde r} a[z;\R,W]\,dz}, \ \ x\in[x_F(x_\ast), x_\ast-\tilde r].
\end{align}
From~\eqref{E:ATWODEF} it is clear that on $[x_F(x_\ast), x_\ast-\tilde r]$,
\begin{align*}
a_2 & = 2\k\left[\left(\F -xW\right)\left(\R-1\right)  + 2f + 4\k x + x\R\left(5+\k\right) \right] Z^{-1} \notag\\
& =  2\k\left[\left(\F -xW\right)\left(\R+1\right)  + 2xW + 4\k x + x\R\left(3+\k\right) \right] Z^{-1}>0.
\end{align*}
Similarly, the first line of~\eqref{E:AONEDEF} is strictly positive on the same interval, and we conclude from~\eqref{E:FLB1} and~\eqref{E:AONEDEF} that
\begin{align}\label{E:FLB2}
f(x)\ge f(x_\ast-\tilde r) e^{ -2\k \int_{x}^{x_\ast-\tilde r}  \left( \frac1{1-\k}\frac{ \R^{-\e}}{z}
+ (1-\k)z +  (\R+\k) f  \right) Z^{-1}  \,dz}, \ \ x\in[x_F(x_\ast), x_\ast-\tilde r].
%f(x)\ge f(x_\ast-r) e^{ -2\k \int_{x}^{x_\ast-r}  \left( \frac1{1-\k}\frac{c R^{-\e}}{z}
%+ (1-\k)z + R+\k\right) Z^{-1}  \,dz}, \ \ x\in[x_F(x_\ast), x_\ast-r].
\end{align}
From~\eqref{E:ZDEF}, $W\ge\frac13$ on $[x_F(x_\ast), x_\ast-\tilde r]$, and $H\ge \F $ we have   $Z\ge\max\{ \frac23\k z^2, \frac13 \F  z\}$  on $[x_F(x_\ast), x_\ast-\tilde r]$. Similarly,  
\[
\begin{split}
&2\k \left( \frac1{1-\k}\frac{ \R^{-\e}}{z}
+ (1-\k)z+ (\R+\k) f \right) Z^{-1} %+ (R+\k) f 
\\
&\le 2\k \left(  \frac{3^\e}{1-\k} \frac{1}{z} + (1-\k) x_\ast \right) \frac{3}{2\k z^2} +  2\k \frac{(M+\k) \F }{ \frac{1}{3} \F  z}\le C_1 \left(\frac{1}{z^3}+\frac{1}{z^2}+  \frac{1}{z}\right) \le C_2\frac{1}{z^3}, 
\end{split}
\] 
for  and some $C_1,C_2>0$ where we recall the upper bound $M$ of $\R$ in  \eqref{RboundM}. 
Combining the two bounds and plugging it back into~\eqref{E:FLB2} we get
\begin{align}
f(x) \ge f(x_\ast-\tilde r)  e^{- C_2\int_x^{x_\ast-\tilde r}\frac1{z^3}\,dz} \ge f(x_\ast-\tilde r) e^{-\frac{C_2}2 x^{-2}}.
\end{align}
Together with the established lower bound $B\ge (1-\k)f^2$ we conclude the proof.
\end{proof}

%%%%%%%%%%%%%%%%%%%%%%%%%%%
%%%%%%%%%%%%%%%%%%%%%%%%%%%

\begin{proposition}\label{P:GLOBALFRIEDMAN}
The unique RLP-type solution associated with $\bar x_\ast\in(\xmin,\xmax)$ exists globally to the left, i.e.
$
s(\bar x_\ast)=0.
$
\end{proposition}

%%%%%%%%%%%%%%%%%%%%%%%%%%%

\begin{proof}
{\em Case 1.} We assume $x_F(\bar x_\ast)=0$. Since $0\le s(\bar x_\ast) \le x_F(\bar x_\ast)$ we are done.

\noindent
{\em Case 2.} We assume  $x_F(\bar x_\ast)>s(\bar x_\ast)>0$. In this case $\bar x_\ast\in X $ which is impossible, since $X $ is relatively open in $[\xmin,\xmax]$ by Lemma~\ref{L:PREPX} and $\xmin \notin X $ by Lemma~\ref{L:XYNONEMPTY}. 
%(as a simply connected component interval contained in the relatively open)

\noindent
{\em Case 3.} We now assume that $x_F(\bar x_\ast)=s(\bar x_\ast)>0$. By Lemma~\ref{L:PRELIM2} this in particular implies that 
\be\label{E:CASE3}
\bar x_\ast\in[\xmin,\xmax]\setminus \Y
\ee 
and
\[
W(x_F(\bar x_\ast);\bar x_\ast)>\frac13.
\]
Consider now a sequence $\{x_\ast^n\}_{n\in\mathbb N}\subset X $ such that $\lim_{n\to\infty}x_\ast^n = \bar x_\ast$. We consider
\[
\bar x_F : = \limsup_{n\to\infty} x_F(x_\ast^n)
\]
and after possibly passing to a subsequence, we assume without loss that $\lim_{n\to\infty} x_F(x_\ast^n)=\bar x_F$. We now consider two possibilities.

\noindent
{\em Case 3 a).} Assume that $\bar x_F>0$.  Since $\{x_\ast^n\}_{n\in\mathbb N}\subset X $, by Lemma~\ref{L:PRELIM2} necessarily $s(x_\ast^n)<x_F(x_\ast^n)$, $n\in\mathbb N$.
Upon possibly passing to a further subsequence, we can ascertain that there exists some $\delta>0$ such that $\bar x_F>x_F(x_\ast^n)>\bar x_F-\delta>0$ for all $n\in\mathbb N$.
 By Lemma~\ref{L:PRELIM3} we conclude in particular that 
 \begin{align}
 B(x;x_\ast^n)>m(\bar x_F-\delta)>0 \ \ \text{ for all } \ \ x\in [x_F(x_\ast^n),x_\ast-\tilde r], \ \ n\in\mathbb N.
 \end{align}
 Therefore, by part (c) of Proposition~\ref{P:LOWERSEMICONTRLP}, there exists $T=T(m(\bar x_F-\delta),\bar x_F)$ such that 
 \begin{align}
s(\bar x_\ast)<\bar x_F-T, \ \  s(x_\ast^n)<\bar x_F - T, \ \ n\in\mathbb N.
 \end{align}
 Fix an $x\in(\bar x_F-T,\bar x_F)$ and observe that 
 \begin{align}
 W(x;\bar x_\ast)=\lim_{n\to\infty}W(x,x_\ast^n)\le \frac13,
 \end{align}
 where we have used Lemma~\ref{L:PREPX}. This implies $\bar x_\ast\in\Y $, a contradiction to~\eqref{E:CASE3}.
 
 \noindent
{\em Case 3 b).} 
Assume that $\bar x_F=0$. 
For any fixed $\hat x>0$ we can apply the argument from {\em Case 3 a)} to conclude that $s(\bar x_\ast)<\hat x$. Therefore $s(\bar x_\ast)=0$ in this case.
 \end{proof}

%%%%%%%%%%%%%%%%%%%%%%%%%%%
%%%%%%%%%%%%%%%%%%%%%%%%%%%

\begin{lemma}[Continuity of $\Y\ni \xs\mapsto x_{F}(\xs)$]\label{L:XFCONT}
The map 
\[
\Y\ni x_\ast \mapsto x_{F}(x_\ast) 
\]
is continuous and 
\be\label{E:ZONETHIRDCONT}
\lim_{X \ni x\to \bar x_\ast} x_{F}(x) = 0 = x_{F}(\bar x_\ast).
\ee
\end{lemma}

%%%%%%%%%%%%%%%%%%%%%%%%%%%

\begin{proof}
The proof is nearly identical to the proof of Lemma 4.13 in~\cite{GHJ2021}.
\end{proof}

%%%%%%%%%%%%%%%%%%%%%%%%%%%

\subsection{The solution from the origin $x=0$ to the right} \label{SS:TOTHERIGHT}

%%%%%%%%%%%%%%%%%%%%%%%%%%%
%%%%%%%%%%%%%%%%%%%%%%%%%%%
%%%%%%%%%%%%%%%%%%%%%%%%%%%

In this section we consider the solutions $(\R_-,W_-)$ to~\eqref{E:RODE}--\eqref{E:WODE}
generated by the data imposed at $x=0$ and satisfying $W_-(0)=\frac13$. Upon specifying
the value $\R_0=\R_-(0)>0$, the problem is well-posed by Theorem~\ref{T:SONICLWPORIGIN}
on some interval $[0,r]$.

%%%%%%%%%%%%%%%%%%%%%%%%%%%%%%%%%%
%%%%%%%%%%%%%%%%%%%%%%%%%%%%%%%%%%

%%%%%%%%%%%%%%%%%%%%%%%%%%%%%%%%%%
%%%%%%%%%%%%%%%%%%%%%%%%%%%%%%%%%%

\begin{definition}
We introduce the \underline{sonic time from the left} 
\be
s_-(\R_0):= \sup_{x\geq 0}\{x\ \big| \ \F [x;\R_-(x;\R_0)] - x W_- (x;\R_0)>0 \}. 
\ee
\end{definition}

%%%%%%%%%%%%%%%%%%%%%%%%%%%%%%%%%%
%%%%%%%%%%%%%%%%%%%%%%%%%%%%%%%%%%

\begin{lemma}\label{L:APRIORIMINUS0} Let $\R_-(0)=\R_0>\frac13$, $W_-(0)=\frac13$ and $x_\ast\in[\xmin,\xmax]$.  The solution $(\R_-(x;\R_0), W_-(x;\R_0))$ to \eqref{E:RODE}--\eqref{E:WODE} with the initial data
\be\label{E:LEFT}
\R_-(0,\R_0)=\R_0\ge \frac13, \quad W_-(0;\R_0)= \frac13,
\ee
exists on the interval $[0, s_-(\R_0))$ and satisfies the following bounds: 
for $x\in(0, s_-(\R_0))$
\begin{align}
 W_- &> \frac{1}{3} ,\label{E:LBWMINUS}\\
\frac{ \R_-}{1-\k}+ \frac{ W_-}{1+\k}&<\frac{\R_0}{1-\k}+ \frac{1}{3(1+\k)}, \\
 \R_-^{1+\k}  W_-^{1-\k} &<  \frac{ \R_-^{1+\k} }{ 3^{1-\k}},\label{E:USGOOD1}\\
  \R_-  &>  W_- , \label{E:USGOOD2}\\
  \R_-' &<0. 
\end{align}
\end{lemma}

%%%%%%%%%%%%%%%%%%%%%%%%%%%%%%%%%%

\begin{proof}
The proof is analogous to the proof of Lemma 4.14 from~\cite{GHJ2021}.
%follows in the same way as in Lemma 4.14 of LP paper. 
\end{proof}

In the following lemma we identify a spatial scale $x_0\sim\frac1{\R_0^{1+O(\k)}}$ over which we obtain quantitative lower bounds on the density $\R_-$ over $[0,x_0]$.

%%%%%%%%%%%%%%%%%%%%%%%%%%%%%%%%%%
%%%%%%%%%%%%%%%%%%%%%%%%%%%%%%%%%%

\begin{lemma}[Quantitative lower bounds on $\R_-$]\label{L:RHOMINUS}
Let $\R_0>\frac13$ and $x_\ast\in[\xmin,\xmax]$ be given and consider the unique solution $(\R_-(\cdot;\R_0), W_-(\cdot;\R_0))$ to the initial-value problem~\eqref{E:RODE}--\eqref{E:WODE},~\eqref{E:LEFT}.
For any $\R_0>\frac13$ let
\begin{align}\label{E:ZNODDEF}
x_0=x_0(\R_0): =  
\begin{cases}
\frac{3^{\frac{1-\k}{2}}}{2^\frac12  (1+8\k)^\frac12\R_0^{\frac{1}{1-\k} } },%\frac{\sqrt 3 }{\sqrt 2 y_\ast R_0} 
& \R_0> 1; \\
\frac{3^{\frac{1-\k}{2}}}{2^\frac12  (1+8\k)^\frac12 },  & \frac13 < \R_0\le1.
\end{cases} 
\end{align}
Then $s_-(\R_0)>x_0$ for all $\R_0>\frac13$ and 
\begin{align}\label{E:APRIORILEFT}
\R_-(x;\R_0) \ge %R_0 \exp\left(- R_0^{-1}\right)
\begin{cases}
\R_0 \exp\left(- \R_0^{-1+\k}\right), & \R_0>1;\\
 \R_0 \exp\left(-1\right), & \frac13 < \R_0\le1, 
\end{cases}
\ \ x\in[0, x_0].
\end{align}
Moreover, for all sufficiently small $\k>0$, there exists a $\hat\R>1$ 
%\bcr M.H. Restate so that the dependence of $\R$ on $x_0$ is explicit. \ec 
such that for all $\R_0>\hat\R$ we have 
\begin{align}\label{E:RHOZZERO}
\R_-(x_0;\R_0)& >\frac1{x_0^{1-\k}}. 
%R_-(z) & >\frac12 R_0, \ \ z\in[0,z_0].\label{E:RHOZZERO2}
\end{align} 
\end{lemma}

%%%%%%%%%%%%%%%%%%%%%%%%%%%%%%%%%%

\begin{proof}
 Equation~\eqref{E:RODE}  %\eqref{E:RHOEQN} 
 is equivalent to 
\begin{align}\label{E:RHOLEFT}
\R_-(x) = \R_0 \exp\left(-\int_0^x\frac{2 \tau (1-\k)( W_-+\k)(\R_-- W_-)}{B}\,d\tau\right).
\end{align}
Since  $ W_-  =( W_-^{1+\k} W_-^{1-\k})^\frac12< (\R_-^{1+\k} W_-^{1-\k})^\frac12$ and  $\R_-  W_- = (\R_-^{1+\k} W_-^{1-\k})^{\frac{1}{1+\k}} W_-^{\frac{2\k}{1+\k}} $, 
by Lemma \ref{L:APRIORIMINUS0},  %\ref{L:APRIORIMINUS},   
we have 
the following bounds on the interval $(0, s_-(\R_0))$
\begin{align}
 W_-<\R_- & <\R_0, \label{E:B1}\\
\frac13< W_- &   <\left(\frac{\R_0^{1+\k}}{3^{1-\k}}\right)^\frac12, \label{E:B2}\\
\R_-  W_-  & < \frac{\R_0^{1+\k}}{3^{1-\k}}.  \label{E:B21}
\end{align}
%where we have used $ W_-  =( W_-^{1+\k} W_-^{1-\k})^\frac12< (R_-^{1+\k} W_-^{1-\k})^\frac12$ for \eqref{E:B2} and $R_-  W_- = (R_-^{1+\k} W_-^{1-\k})^{\frac{1}{1+\k}} W_-^{\frac{2\k}{1+\k}} $ for \eqref{E:B21}. 
Now from the definition~\eqref{E:BDEF0} of $B$, for any $0\le x\le x_0$ using~\eqref{E:B1}-\eqref{E:B21} and the definition of $x_0$ in \eqref{E:ZNODDEF},  we have 
\begin{align}
\R_-^{\frac{2\k}{1-\k}} B& =1 -\R_-^{\frac{2\k}{1-\k}} \left[ (1-\k) W_-^2  + 4\k \R_- W_- + 4\k  W_-+ \k^2-\k \right] x^2\notag \\
&\geq 1  -\R_0^{\frac{2\k}{1-\k}}\left[ (1+3\k)\frac{\R_0^{1+\k}}{3^{1-\k}}  + 4\k  \left(\frac{\R_0^{1+\k}}{3^{1-\k}}\right)^\frac12+ \k^2-\k \right] x^2\notag \\
&\geq 1  -\R_0^{\frac{2\k}{1-\k}} (1+8\k) \frac{\R_0^{1+\k}}{3^{1-\k}} x^2\notag \\
&\geq 
\begin{cases}
1- \frac{1}{2\R_0^{1-\k}} \geq \frac12 ,& \R_0>1; \\
1- \frac{\R_0^{1+\k+\frac{2\k}{1-\k}}}{2} \geq \frac12, & \frac13<\R_0\le 1. 
\end{cases}
 % \frac{1}{4 R_0} >\frac14 
 \label{E:B3}
\end{align}
Note from the second line to the third line we have used the upper bound of $[\cdot]$ term
\be\label{E:B30}
\frac{\R_0^{1+\k}}{3^{1-\k}} \le (1+3\k)\frac{\R_0^{1+\k}}{3^{1-\k}}  + 4\k  \left(\frac{\R_0^{1+\k}}{3^{1-\k}}\right)^\frac12+ \k^2-\k\le (1+8\k) \frac{\R_0^{1+\k}}{3^{1-\k}}, 
\ee 
which holds true for $\R_0>\frac13$ and sufficiently small $0<\k$. Therefore,   $s_-(\R_0)>x_0$ for all $\R_0>\frac13$. 

From~\eqref{E:B1},~%\eqref{E:APRIORI3},
\eqref{E:B21}, 
and~\eqref{E:B3} for any $x\in [0,x_0]$ we obtain
\begin{align*}
&\int_0^x\frac{2 \tau(1-\k)  ( W_-+\k)(\R_-- W_-)}{  B}\,d\tau
 \le 4 (1-\k)\R_0^{\frac{2\k}{1-\k}} ( \frac{\R_0^{1+\k}}{3^{1-\k}} + \k \R_0) \int_0^x\tau\,d\tau  
 \\
 & \le2x_0^2 (1-\k) \R_0^{\frac{2\k}{1-\k}} ( \frac{\R_0^{1+\k}}{3^{1-\k}} + \k \R_0) 
 %\le \frac{2y_\ast^2 R_0^{\frac32}}{\sqrt 3}  z_0^2 
= \begin{cases}
  \frac{1}{\R_0^{1-\k}}\frac{(1-\k)( \frac{\R_0^{1+\k}}{3^{1-\k}} + \k \R_0)}{ (1+8\k) \frac{\R_0^{1+\k}}{3^{1-\k}} }, &  \R_0>1;\\
\R_0 \frac{(1-\k)  \R_0^\frac{2\k}{1-\k}( \R_0^{\k} + \k3^{1-\k} )}{ (1+8\k) } ,& \frac13<\R_0\le 1 . 
 \end{cases}
% \begin{cases}
 % R_0^{-1}, & R_0>1;\\
% R_0 \le 1, & \frac13 < R_0\le1.
% \end{cases}
\end{align*}
Now comparing the denominator and numerator of the last fractions, we see that 
\[
\begin{split}
&(1+8\k) \frac{\R_0^{1+\k}}{3^{1-\k}} - (1-\k)( \frac{\R_0^{1+\k}}{3^{1-\k}} + \k \R_0) = \k \left[ 9\frac{\R_0^{1+\k}}{3^{1-\k}} - (1-\k)\R_0  \right]\geq 0 \  \text{ for all } \ \R_0 > 1
\end{split}
\] and
\[
\begin{split}
1+8\k - (1-\k) \R_0^\frac{2\k}{1-\k}( \R_0^{\k} + \k3^{1-\k} ) &\ge 1+8\k - (1-\k) (1+\k 3^{1-\k}) \\
&= \k(9-(1-\k) 3^{1-\k}) \ge 0 \  \text{ for all } \  \frac13<\R_0\le 1 
\end{split}
\]
and hence we deduce that for  any $x\in [0,x_0]$
\[
\int_0^x\frac{2\tau(1-\k)  ( W_-+\k)(\R_-- W_-)}{  B}\,d\tau \leq
\begin{cases}
 \frac{1}{\R_0^{1-\k}}, & \R_0>1;\\
 1, &  \frac13<\R_0\le 1 . 
 \end{cases}
\]
Plugging this bound in~\eqref{E:RHOLEFT} we obtain~\eqref{E:APRIORILEFT}.

To show~\eqref{E:RHOZZERO},  we first rewrite \eqref{E:ZNODDEF} for $\R_0>1$ as 
\be\label{rho0L}
\R_0 = \frac{1}{x_0^{1-\k}}  \left( \frac{3^{1-\k}}{2(1+8\k)} \right)^{\frac{1-\k}{2}}. 
\ee
Using this  in~\eqref{E:APRIORILEFT}
\be
\begin{split}
\R_-(x_0; \R_0) &\geq \R_0 e^{-\frac{1}{\R_0^{1-\k}}} %= R_0^{1+\frac{\k(1-\k)}{2}} \R_0^{-\frac{\k(1-\k)}{2}} e^{-\frac{1}{\R_0}}\\
= \frac{1}{x_0^{1-\k}}  \left( \frac{3^{1-\k}}{2(1+8\k)} \right)^{\frac{1-\k}{2}}  e^{-\frac{1}{\R_0^{1-\k}}}. 
\end{split}
\ee
The bound~\eqref{E:RHOZZERO} follows if 
$
 \left( \frac{3^{1-\k}}{2(1+8\k)} \right)^{\frac{1-\k}{2}}  e^{-\frac{1}{\R_0^{1-\k}}}>1
$
which is clearly true for sufficiently large $\R_0$ and sufficiently small $\k$. 
%Bound~\eqref{E:RHOZZERO2} is obvious from~\eqref{E:APRIORILEFT}.
%\begin{align}\label{E:APRIORILEFT}
%R(z) \ge R_0 e^{-\frac{\sqrt 3}{4} R_0^{-\frac12}}, \ \ z\le z_0.
%\end{align}
\end{proof}

\begin{remark} The specific choice of $x_0$ in \eqref{E:ZNODDEF} has been made to ensure the lower bound in \eqref{E:RHOZZERO} to be compared with the bound \eqref{E:FtoRho} satisfied by the solutions emanating from the sonic point. In fact, better bounds are obtained by choosing different $x_0$. For instance, if $x_0= O(\frac{1}{\R_0^{1+\alpha}})$ for $\frac{\k}{1-\k}+\frac{\k}{2}+\frac12<1+\alpha\leq \frac{1}{1-\k}$ in  \eqref{E:ZNODDEF}, then we may deduce that for such $x_0$, $\R_-(x_0;\R_0) > \frac{1}{x_0^{\frac{1}{1+\alpha}} } \geq \frac{1}{ x_0^{1-\k}}$ with the equality being valid for $\alpha=\frac{\k}{1-\k}$ for all sufficiently large $\R_0$ and small $\k$. 
\end{remark}

\begin{remark}\label{R:COMMONINTERVAL}
Since the mapping $\R_0\mapsto x_0(\R_0)$ from~\eqref{E:ZNODDEF} is nonincreasing, it follows that for any fixed $\R_0>\frac13$ we have the uniform bound
on the sonic time:
\begin{align}
s_-(\tilde \R_0)>x_0(\tilde \R_0)\ge x_0(\R_0), \ \ \text{ for all }\ \ \frac13  <\tilde \R_0\le \R_0.\notag
\end{align}
\end{remark}

\

The following lemma shows the crucial monotonicity property of $\R_-(\cdot;\R_0)$ with respect to $\R_0$ on a time-scale of order $\sim \R_0^{-(\frac34+O(\k))}$.

\begin{lemma}\label{L:MONOTONE}
Let $x_\ast\in[\xmin,\xmax]$. There exists a sufficiently small $\kappa>0$ such that for all $\R_0\ge \frac13$  
\begin{align}
\pa_{\R_0}\R_-(x;\R_0)>0 \ \ \text{ for all } \ x\in[0, \kappa \R_0^{-b}], \ \text{ where }  \ b= \tfrac{3+\k+2\e}{4} \notag
\end{align} 
\end{lemma}

\begin{proof}
We introduce the short-hand notation $\pa \R_- = \pa_{ \R_0} \R_-$ and $\pa W_- = \pa_{ \R_0} W_-$, where we note that the map $\R_0\mapsto (\R_-,W_-)$ is $C^1$ by Remark~\ref{R:CONEPROP}. 
It is then easy to check that $(\pa \R_-,\pa W_-)$ solve
\begin{align}
\pa W_-' & = - \frac{3}{x} \pa W_- - \frac{2 x (1+\k) W_- ( W_- +\k)}{B}\pa W_- \notag\\
&\ \ \ \ +  
%\left( \frac{4y_\ast^2 x W_-( R_- - W_-)}{1-y_\ast^2 x^2  W_-^2} 
%+ \frac{4y_\ast^4 x^3  W_-^3 ( R_-- W_-)}{\left(1-y_\ast^2 x^2  W_-^2\right)^2} \right) 
 \frac{2  x(1+\k)( \R_- - W_-)}{B} \left( 2 W_- +\k - \tfrac{\mathfrak D_ W B}{B} W_-( W_-+\k)\right) \pa W_- 
  \label{E:VARONE}\\
 & \ \ \ \ + \frac{2  x (1+\k) W_- ( W_- +\k)}{B} \left( 1- \tfrac{\mathfrak D_{ \R} B}{B} ( \R_-- W_-)  \right)\pa \R_-  , \notag \\
\pa \R_-'  & = - \frac{2  x(1-\k)( W_-+\k)}{B} \left( 2 \R_- -  W_- -\tfrac{\mathfrak D_{ \R} B}{B}  \R_- ( \R_-- W_-) \right) \pr  \notag \\
&  \ \ \ \ -\frac{2  x(1-\k) \R_-}{B}  \left( \R_- -2 W_--\k -\tfrac{\mathfrak D_ W B}{B} ( W_-+\k)( \R_-- W_-) \right) \po , \label{E:VAR2}
\end{align} where
\[
\mathfrak D_ W B = - [2(1-\k) W + 4\k(1+ \R)]  x^2, \quad \mathfrak D_ \R B = -\tfrac{2\k}{1-\k}  \R^{-\frac{2\k}{1-\k}-1} - 4\k W   x^2. 
\]
At $x=0$ we have the initial values
\begin{align}\label{E:BDRYOR}
\pr(0) = 1, \ \ \po(0)=0.
\end{align}

We multiply~\eqref{E:VARONE} by $\po$ and integrate over the region $[0,x]$. By~\eqref{E:BDRYOR} we obtain
\begin{align}
&\frac12\po^2(x) + \int_0^x \left(\frac3\tau +\frac{2  \tau (1+\k) W_- ( W_- +\k)}{B} \right) \po^2 \,d\tau \notag\\
& = \int_0^x 
%\left( \frac{4y_\ast^2 \tau  W_-( R_- - W_-)}{1-y_\ast^2 \tau^2  W_-^2} 
%+ \frac{4y_\ast^4 \tau^3  W_-^3 ( R_-- W_-)}{\left(1-y_\ast^2 \tau^2  W_-^2\right)^2} \right)
 \frac{2  \tau(1+\k)( \R_- - W_-)}{B} \left( 2 W_- +\k - \tfrac{\mathfrak D_ W B}{B} W_-( W_-+\k)\right)  \po^2\,d\tau \notag \\
& \ \ \ \  + \int_0^x \frac{2  \tau(1+\k) W_- ( W_- +\k)}{B} \left( 1- \tfrac{\mathfrak D_{ \R} B}{B} ( \R_-- W_-)  \right)\pa \R_-  \po\,d\tau . \label{E:INT1}
\end{align}
Just like in \eqref{E:B3}, by using  \eqref{E:B1}--\eqref{E:B21}, and \eqref{E:B30}
%Since $ W_-(x)^2\le \frac{ \R_0}{3}$ (by Lemma~\ref{L:APRIORI}) and $y_\ast\le 3$ we have 
we have 
\[
\begin{split}
 \R_-^{\frac{2\k}{1-\k}} B& \geq 1  - \R_0^{\frac{2\k}{1-\k}}\left[ (1+3\k)\frac{ \R_0^{1+\k}}{3^{1-\k}}  + 4\k  \left(\frac{ \R_0^{1+\k}}{3^{1-\k}}\right)^\frac12+ \k^2-\k \right]  \tau ^2 \\
& \geq 1- (1+8\k)  \R_0^{\frac{2\k}{1-\k}} \frac{ \R_0^{1+\k}}{3^{1-\k}}   \tau^2 . 
\end{split}
\]
Therefore 
\begin{align}\label{E:LOWERBOUNDRHOZERO}
 \R_-^{\frac{2\k}{1-\k}} B  \ge \frac12, \text{ for any } \  \tau \in[0, ( 8  \R_0^{1+\k+\frac{2\k}{1-\k}}   )^{-\frac12}]
\end{align}
where we have used $2(1+8\k)  3^{\k-1} < 8 $ for all $x_\ast \leq \xmax$ and sufficiently small $\k$. 
 
Using the bounds $x_\ast\le \xmax$, \eqref{E:B1}--\eqref{E:B21},~\eqref{E:LOWERBOUNDRHOZERO}, %and $ R_- W_-\le \frac{ R_0}3$ (Lemma~\ref{L:APRIORI}), 
we obtain from~\eqref{E:INT1}
\begin{align}
& \frac12\po^2(x) + \int_0^x \left(\frac3\tau +\frac{2  \tau (1+\k) W_- ( W_- +\k)}{B} \right)\po^2 \,d\tau \label{E:INT2} \\
& \le  C \int_0^x  \R_0^{1+\k+\frac{2\k}{1-\k}}
%\left(  R_0^{\frac32} \tau +  R_0^{\frac52} \tau^3\right) 
\tau \po^2\,d\tau 
 + C \R_0^{1+\k+\frac{2\k}{1-\k}} \int_0^x \tau |\pr| |\po| \, d\tau, \ \ x\le ( 8  \R_0^{1+\k+\frac{2\k}{1-\k}}   )^{-\frac12} \notag
\end{align}
for all sufficiently small $\k$. Here we have used $-\tfrac{\mathfrak D_\R B}{B} (\R_--W_-) >0$ and $\tfrac{\mathfrak D_\R B}{B} W_- <0$, and  
\[
\begin{split}
0<- \tfrac{\mathfrak D_{ \R} B}{B} ( \R_-- W_- ) &<- \tfrac{\mathfrak D_{ \R} B}{B} \R_- = - \tfrac{\mathfrak D_{ \R} B}{\R_-^{\e}B} \R_-^{1+\e} \\
&\le C_1\k ( 1 + \R_-^{1+\e} W_- \tau^2  ) \le C_1\k (1+ \frac{\R_0^{1+\k + \e}}{3^{1-\k}}\tau^2)\le C_2 \k 
\end{split}
\]
for any $\tau \in[0, ( 8  \R_0^{1+\k+\frac{2\k}{1-\k}}   )^{-\frac12}]$ so that $\left( 1- \tfrac{\mathfrak D_{ \R} B}{B} ( \R_-- W_-)  \right)$ is bounded by a constant. 

Let $\hat X = \kappa  \R_0^{-b}$ where $b=\frac{3+\k+\frac{4\k}{1-\k}}{4}$ with a sufficiently small $\kappa>0$ to be specified later. Note that $\hat X< ( 8  \R_0^{1+\k+\frac{2\k}{1-\k}}   )^{-\frac12}$ for all $ \R_0\ge \frac13$ and $\kappa$ chosen sufficiently small and independent of $ \R_0$. 
For any $\tau\in [0,\hat X]$ we have $ \R_0 \le \kappa^{\frac{1}{b}} \tau^{-\frac{1}{b}}$. Therefore
%\begin{align}
% R_0^{\frac32} \tau & \le \kappa^2 \tau^{-1}, \ \ \tau\in[0,X] \notag \\
%  R_0^{\frac52} \tau^3 & \le \kappa^{\frac{10}3} \tau^{-\frac13}, \ \ \tau\in[0,X].  \notag
%\end{align}
$ \R_0^{1+\k+\frac{2\k}{1-\k}}  \tau \leq \kappa^{\frac{1}{b}( 1+\k+\frac{2\k}{1-\k}) } \tau^{-\frac{1}{b}+1}$.
From these estimates and~\eqref{E:INT2} we conclude for $ x\in[0,\hat X]$, 
\begin{align}
& \frac12\po^2(x) + \int_0^x \left(\frac3\tau +\frac{2  \tau (1+\k) W_- ( W_- +\k)}{B} \right)\po^2 \,d\tau \notag \\
& \le  C \int_0^x 
%\left( \frac{\kappa^2}{\tau} + \frac{\kappa^{\frac{10}3}}{\tau^{\frac13}}\right) 
\kappa^{\frac{1}{b}( 1+\k+\frac{2\k}{1-\k}) } \tau^{-\frac{1}{b}+1}\po^2\,d\tau 
 + \frac C{\sqrt 3}  \R_0^{1+\k+\e} \left(\int_0^x\frac3\tau \po^2\,d\tau \right)^{\frac12} \left(\int_0^x \tau^3 \pr^2 \, d\tau\right)^{\frac12} \notag \\
 & \le  C \int_0^x 
 %\left( \frac{\kappa^2}{\tau} + \frac{\kappa^{\frac{10}3}}{\tau^{\frac13}}\right) 
\kappa^{\frac{1}{b}( 1+\k+\frac{2\k}{1-\k}) } \tau^{-\frac{1}{b}+1}\po^2\,d\tau 
 + \frac12 \int_0^x\frac3\tau \po^2\,d\tau +  \frac {C^2}{6}  \R_0^{2+2\k+2\e}  \|\pr\|_{\infty}^2 \int_0^x \tau^3\,d\tau.
 \label{E:INT3}
\end{align}
Since $-\frac{1}{b}+1>-1$, with $\kappa$ chosen sufficiently small, but independent of $ \R_0$, we can absorb the first two integrals on the right-most side
into the term $\int_0^x \frac3\tau \po^2\,d\tau$ on the left-hand side. %Since $\int_0^x\tau^3\,d\tau = \frac14 \kappa^4 \R_0^{-3}$ 
We then conclude
\begin{align}\label{E:POBOUND}
\lv \po(x)\rv \le C \R_0^{1+\k+\e}  x^2 \|\pr\|_{\infty}, \ \ x\in[0,\hat X].
\end{align}

We now integrate~\eqref{E:VAR2}, use \eqref{E:B1}--\eqref{E:B21},~\eqref{E:LOWERBOUNDRHOZERO},  and conclude from~\eqref{E:BDRYOR}
\begin{align}
\lv \pr(x) - 1\rv & \le  C \R_0^{1+\k+\frac{2\k}{1-\k}} \|\pr\|_{\infty} \int_0^x \tau\,d\tau   
\\
&\ \ \ \ +  C \int_0^x \left( \R_0^{2+\e} \tau  + \left[  \R_0^{2+2\k+2\e } + \k  \R_0^{3+\k+2\e}\right]\tau^3\right)\lv \po(\tau)\rv \,d\tau \notag \\
& \le C\left(  \R_0^{1+\k+\e} x^2 +  \R_0^{3+\k+2\e} x^4 + (  \R_0^{3+3\k+3\e } + \k  \R_0^{4+2\k+3\e} )  x^6 \right)  \|\pr\|_{\infty} \notag \\
& \le C\kappa^2  \|\pr\|_{\infty},  \ \ x\in[0,\hat X],  \notag
\end{align}
for all sufficiently small $\k$, where we have used~\eqref{E:POBOUND} and $0\le x\le \kappa  \R_0^{-b}$.  Therefore, 
\[
\|\pr\|_{\infty} \le 1 + C\kappa^2 \|\pr\|_{\infty}
\]
and thus, for $\kappa$ sufficiently small so that $C\kappa^2<\frac13$, we have $\|\pr\|_\infty\le \frac32$. From here we infer
\begin{align}
\pr(x) \ge 1 - \frac32C\kappa^2>\frac12>0, \ \ x\in[0,\hat X].\notag
\end{align}
\end{proof}

%%%%%%%%%%%%%%%%%%%%%%%%
%%%%%%%%%%%%%%%%%%%%%%%%

\subsection{Upper and lower solutions}\label{SS:UL}

%%%%%%%%%%%%%%%%%%%%%%%%
%%%%%%%%%%%%%%%%%%%%%%%%

%%%%%%%%%%%%%%%%%%%%%%%%%

\begin{definition}[Upper and lower solution]\label{D:UPPERLOWER}
For any $x_\ast\in [\xmin,\xmax]$ we say that $(\R(\cdot;x_\ast),W(\cdot;x_\ast))$ is an {\em upper} (resp. {\em lower}) solution at $x_0\in(0,x_\ast)$ if there exists $\R_0>0$ such that 
\begin{align}
\R(x_0;x_\ast) = \R_-(x_0;\R_0)\notag 
\end{align}
and 
\begin{align}
W(x_0;x_\ast)> \ (\text{resp. } <) \ W_-(x_0;\R_0).\notag
\end{align}
\end{definition}

%%%%%%%%%%%%%%%%%%%%%%%%%

\begin{lemma}[Existence of a lower solution]\label{L:LOWERSOLUTION}
There exists a $\kappa>0$ such that for any $x_0<\kappa$ there exists an $x_{\ast\ast}\in [\bar x_\ast,\xmax]$ such that $(\R(\cdot;x_{\ast\ast}),W(\cdot;x_{\ast\ast}))$ is a lower solution at $x_0$. 
Moreover, there exists a universal constant $C$ such that $\R_1<\frac C{x_0^{1-\k}}$, where $\R_-(x_0;\R_1)=\R(x_0;x_{\ast\ast})$.
\end{lemma}

%%%%%%%%%%%%%%%%%%%%%%%%%%%%%%%%%%%%

\begin{proof}
For any $x_\ast\in X $ we consider the function 
\begin{align}\label{E:BIGFDEF}
S(x_\ast) : = \sup_{\tilde x_\ast\in[\bar x_\ast, x_\ast]} \left\{x_{F}(\tilde x_\ast)\right\}. 
\end{align}
%For $\bar y_\ast< y_\ast\le \bar y_\ast+\epsilon$ with $\epsilon>0$ sufficiently small we have $y_\ast\in X $. 
The function $x_\ast\mapsto S(x_\ast)$ is clearly increasing, continuous,  and by Lemma~\ref{L:XFCONT}, $\lim_{x_\ast\to\bar x_\ast}S(x_\ast)=0$.  
Therefore, the range of $S$  is of the form $[0,\kappa]$ for some $\kappa>0$.
For any $x_\ast\in X $, by Lemma~\ref{L:XFCONT}, the supremum in~\eqref{E:BIGFDEF} is attained, i.e. there exists $x_{\ast\ast}\in [\bar x_\ast, x_\ast]$ such that 
$S(x_\ast)  = x_{F}(x_{\ast\ast})=:x_0$. Therefore, for any 
$\bar x_\ast < x_\ast< x_{\ast\ast}$ we have
\begin{align}
s(x_\ast)<x_{F}(x_\ast)\le x_0.\notag
\end{align}
By Lemma~\ref{L:PREPX} we have the bound $\R(x_0;x_{\ast\ast})>W(x_0;x_{\ast\ast})=\frac13$. By Lemma~\ref{L:RHOMINUS} choosing 
$\R_0 = \R_0(x_0) =  \frac{1}{x_0^{1-\k}}  \left( \frac{3^{1-\k}}{2(1+8\k)} \right)^{\frac{1-\k}{2}}$ and using Lemmas~\ref{L:APRIORIMINUS} and~\ref{L:PRELIM1} 
%(by choosing $\epsilon$ above sufficiently small we can arrange for $z_0$ to be taken sufficiently small) 
we have
\begin{align}
\R_-(x_0;\R_0) >\frac1{x_0^{1-\k}}> \frac{\F [x_0;\R]}{x_0}>\R(x_0;x_{\ast\ast}),\notag
\end{align}
where we have assumed $x_0$ to be sufficiently small.
% (\bcr M.H. to make Lemma~\ref{L:RHOMINUS} applicable.\ec)
On the other hand $\R_-(x_0;\frac13)=\frac13<\R(x_0;x_{\ast\ast})$ (where we recall that $\R_-(\cdot;\frac 13)$ is the Friedmann solution, see~\eqref{E:FRIEDMANNINTRO}). 
Using Remark~\ref{R:COMMONINTERVAL} and the Intermediate Value Theorem, there exists a $\R_1\in(\frac13,\R_0)$ such that 
\begin{align}
\R(x_0;x_{\ast\ast}) = \R_-(x_0;\R_1).\notag
\end{align}
By~\eqref{E:LBWMINUS} $W_-(x_0;\R_1)>\frac13 = W(x_0;x_{\ast\ast})$ and therefore $(\R(\cdot,x_{\ast\ast}),W(\cdot;x_{\ast\ast}))$ is a lower solution at $x_0$. The upper bound on 
$\R_1$ follows from our choice of $\R_0$ above.
\end{proof}

%%%%%%%%%%%%%%%%%%%%%%%%
%%%%%%%%%%%%%%%%%%%%%%%%

\begin{remark}
The proof of Lemma~\ref{L:LOWERSOLUTION} follows closely the analogous proof in the Newtonian case (Lemma 4.20 in~\cite{GHJ2021}).
\end{remark}

%%%%%%%%%%%%%%%%%%%%%%%%

%%%%%%%%%%%%%%%%%%%%%%%%%%%%%%%%%

\begin{lemma}\label{L:XIMP}
Let $x_\ast\in \X$ (see~\eqref{E:XDEF} for the definition of $\X$) and assume that $s(x_\ast)=0$. 
Then
\begin{enumerate}
\item[(a)]
\begin{align}
\R(x;x_\ast)>W(x;x_\ast), \ \ x\in(0,x_\ast);\notag
\end{align}
\item[(b)]
\begin{align}
\limsup_{x\to0} x^{1-\k} W(x;x_\ast) >0.\notag
\end{align}
\end{enumerate}
\end{lemma}

%%%%%%%%%%%%%%%%%%%%%%%%%%%%%%%%

\begin{proof} 
%We claim that $\rho(z;x_\ast)>\omega(z;x_\ast)$ for all $z\in(0,1)$. 
{\em Proof of Part (a).}
If not let
\begin{align}
x_c : = \sup_{x\in(0,x_\ast)} \left\{\R(\tau;x_\ast)-W(\tau;x_\ast)>0, \ \tau\in(x,x_\ast) , \ \ \R(x;x_\ast)=W(x;x_\ast)\ \right\}>0.
\notag 
\end{align}  
At $x_c$ we have from~\eqref{E:RODE}--\eqref{E:WODE} $W'(x_c;x_\ast) = \frac{1-3W}{x_c}<0$  since $W>\frac13$ for $x_\ast\in\X $, and $\R'(x_c;x_\ast)=0$.  
Therefore there exists a neighbourhood strictly to the left of $x_c$ such that $W'<0$, $\R<W$, and $\R'>0$. 
It is easily checked that this property is  dynamically trapped and we conclude 
\begin{align}
W'(x;x_\ast) \le \frac{1-3W(x;x_\ast)}{x} , \ \ x\le x_c. \label{E:OMEGAPRIMENEWEST}
\end{align}
Integrating the above equation over $[x,x_c]$ we conclude 
\begin{align}\label{E:WCBOUND}
W(x;x_\ast) x^3 \ge \omega(x_c;x_\ast) x_c^3 - \frac13 x_c^3 = \left(W(x_c;x_\ast)-\frac13\right) x_c^3 =:c>0.
\end{align}
We now recall~\eqref{E:LOGID0}. From \eqref{E:WCBOUND}, this  %which 
implies that $\left(\log\left(\R^{2(1+\e)}W^2\right)\right)' <0$ since $x_\ast\in \X $.
In particular we obtain the lower bound
\[
\R(x)^{2(1+\e)}W(x)^2> \R(x_c)^{2(1+\e)}W(x_c)^2 = W(x_c)^{4+2\e}>3^{-(4+2\e)}, \ \ x\le x_c.
\]
It follows that 
\be\label{E:ULONE}
\R^{-\e} < 3^{\frac{\e(2+\e)}{1+\e}}W^{\frac{\e}{1+\e}}, \ \ x\le x_c.
\ee
On the other hand, bound~\eqref{E:WCBOUND} implies
\begin{align}\label{E:ULTWO}
x^2 >  c^{\frac23} W^{-\frac23}, \ \ x\le x_c.
\end{align}
From~\eqref{E:BDEF0},~\eqref{E:ULONE}--\eqref{E:ULTWO} we therefore have
\begin{align}
B < 3^{\frac{\e(2+\e)}{1+\e}}W^{\frac{\e}{1+\e}} - C W^{\frac43}, \ \ x\le x_c.
\end{align}
for some universal constant $C>0$. Since $W$ grows to infinity as $x\to0$, this implies that the right-hand side above necessarily becomes negative, i.e. $s(x_\ast)>0$. A contradiction.

\noindent
{\em Proof of Part (b).}
Since $\R>W$ by part (a) and $W>\frac13$ (since $x_\ast\in\X $), from Lemmas~\ref{L:PRELIM1} and~\ref{L:APRIORIMINUS}
%\begin{align}
%R<\frac{F[x;R]}{x}
%& =-\e (1+R) + \sqrt{\e^2(1+R)^2+\k + \frac{R^{-\e}}{(1-\k)x^2}} \\
%&=\frac{\k + \frac{R^{-\e}}{(1-\k)x^2}}{\e (1+R)+ \sqrt{\e^2(1+R)^2+\k + \frac{R^{-\e}}{(1-\k)x^2}} } \\
%& \le \sqrt{\k + \frac{R^{-\e}}{(1-\k)x^2}} \lesssim \frac{R^{-\frac{\e}{2}}}{x},
%\end{align}
%which implies $R^{1+\frac{\e}{2}}\lesssim \frac 1x$, which, keeping in mind~\eqref{E:ETADEF} gives
\be\label{E:USRBOUND}
\R\le \frac{1}{x^{1-\k}}.
\ee
%for some $C>0$.
By way of contradiction we assume that $\limsup_{x\to0}x^{1-\k} W(x;x_\ast)=0$.
For any $\tilde\k>0$ choose $\delta>0$ so small that 
\be\label{E:USCONTR}
x^{1-\k} W(x;x_\ast)<\tilde\k, \ \ x\in(0,\delta).
\ee
Note that in particular, for any $x\in(0,\delta)$ we have
\begin{align}\label{E:USBOUNDONE}
x^2 W \R^{1+\e} = x^{1-\k}W x^{1+\k}\R^{1+\e} \le  \tilde\k, \ \ x\in(0,\delta),
\end{align}
where we have used~\eqref{E:USRBOUND} and~\eqref{E:USCONTR}.
Similarly, 
\begin{align}\label{E:USBOUNDTWO}
x^2 W \R^{\e} = x^{1-\k} W  x^{2\k} \R^\e  x^{1-\k}  %\le \tilde\k C^\e x^{1-\k} 
 \le \tilde\k \delta^{1-\k} , \ \ x\in(0,\delta).
\end{align}
We next claim that there exists a universal constant $\bar C$ such that 
\begin{align}\label{E:USCLAIM}
2(1+\k)x^2(W+\k)(\R-W) \le \bar C B.
\end{align}
Keeping in mind~\eqref{E:BDEF0}, this is equivalent to the estimate
\begin{align}
&x^2 W^2\R^\e \left(\bar C(1-\k)-2(1+\k)\right) + x^2W\R^\e \left(-2\k(1+\k)+2(1+\k)\R+4\k \bar C(1+\R)\right) \notag \\
&  + x^2\R^\e\left(2\k(1+\k)\R + \bar C(\k^2-\k)\right)\le \bar C. \label{E:LSINTERIM}
\end{align}
For any $0<\bar C<2$ the first term on the left-hand side of~\eqref{E:LSINTERIM} is strictly negative, and so are $-2\k(1+\k)x^2W\R^\e$ and $\bar C(\k^2-\k) x^2\R^\e$.
On the other hand
\begin{align}
x^2W\R^\e \left(2(1+\k)\R+4\k \bar C(1+\R)\right) & = x^2W\R^{1+\e} \left(2(1+\k)+4\k \bar C\right)+4\k \bar Cx^2W\R^\e \notag \\
& \le  \tilde\k \left(2(1+\k)+4\k \bar C\right) + 4\k \bar C \tilde\k \delta^{1-\k},
\end{align}
where we have used~\eqref{E:USBOUNDONE} and~\eqref{E:USBOUNDTWO}.
Similarly, 
\begin{align}
2\k(1+\k) x^2\R^{1+\e}  \le  2\k(1+\k)  x^{1-\k} \le 2\k(1+\k) \delta^{1-\k},
\end{align}
where we have used~\eqref{E:USRBOUND}. It is thus clear that we can choose $\bar C$ of the form $\bar C = \tilde C (\tilde\k + \delta^{1-\k})$ 
for some universal constant $\tilde C>0$ and $\tilde \k,\delta$ sufficiently small, so that~\eqref{E:USCLAIM} is true.
It then follows from~\eqref{E:WODE} that 
\begin{align} \label{E:OMEGAPRIME}
W' \le \frac{1-3W}{x} + \frac{\tilde C(\tilde\k +\delta^{1-\k})W}{x} = \frac{1-(3-\tilde C(\tilde\k +\delta^{1-\k}))W}{x}. 
\end{align}

As a consequence of~\eqref{E:OMEGAPRIME}, for sufficiently small $\tilde\k,\delta>0$ and $x\in(0,\delta)$ we have
\begin{align}
W' \le -\frac{2W}{x}, \notag
\end{align}
which in turn implies $W(x;x_\ast)\ge C x^{-2}$ for some $C>0$ and sufficiently small $x$, a contradiction.  
%Therefore, since $\omega(z;x_\ast)z >1$ for sufficiently small $z$ 
%we conclude $s(\bar x_\ast)>0$, a contradiction to the assumption $s(x_\ast)=0$.
%\ec
\end{proof}

%%%%%%%%%%%%%%%%%%%%%%%%%%%%%%%%
%%%%%%%%%%%%%%%%%%%%%%%%

%%%%%%%%%%%%%%%%%%%%%%%%
%%%%%%%%%%%%%%%%%%%%%%%%

\begin{lemma}[Existence of an upper solution]\label{L:CONTR}
If 
\begin{align}
\lim_{x\to0} W(x;\bar x_\ast) \neq \frac13,\notag
\end{align}
then there exists  a universal constant $C>0$ and an arbitrarily small  $x_0>0$ such that $(\R(\cdot;\bar x_\ast),W(\cdot;\bar x_\ast))$ is an upper solution at $x_0$ and 
$\R_1<\frac C{x_0}$, where $\R_-(x_0;\R_1)=\R(x_0;\bar x_\ast)$.
\end{lemma}

%%%%%%%%%%%%%%%%%%%%%%%%

\begin{proof}
It is clear that $\liminf_{x\to0}W(x;\bar x_\ast)\ge\frac13$ as otherwise we would have $\bar x_\ast\in X $, a contradiction to the definition \eqref{critical_x} %~\eqref{E:BARYDEF} 
of $\bar x_\ast$ and the openness of $X $.
We distinguish three cases.

\noindent
{\em Case 1.}
\begin{align}
\liminf_{x\to 0} W(x;\bar x_\ast) >\frac13.\notag
\end{align}
In this case $\bar x_\ast\in \X$ and by Proposition~\ref{P:GLOBALFRIEDMAN} we have $s(\bar x_\ast)=0$.
By part (b) of Lemma~\ref{L:XIMP} there exists a constant $C>0$ and a sequence $\{x_n\}_{n\in\mathbb N}\subset(0,1)$ such that $\lim_{n\to\infty}x_n=0$ and  
\be\label{E:CRUCIAL00}
W(x_n;\bar x_\ast)>\frac C{x_n^{1-\k}}
\ee
where $C$ is independent of $n$.
For any such $x_n$ we have by part (a) of Lemma~\ref{L:XIMP} and Lemma~\ref{L:APRIORIMINUS} 
\begin{align}\label{E:CONSTANTC}
\R(x_n;\bar x_\ast)>W(x_n;\bar x_\ast) >\frac{C}{x_n^{1-\k}} > C \R(x_n;\bar x_\ast)
\end{align}
where we have used the assumption $\bar x_\ast\in\X $ and part (a) of Lemma~\ref{L:XIMP} to conclude $\R(\cdot;\bar x_\ast)>\frac13$, which is necessary for the application of 
Lemma~\ref{L:APRIORIMINUS}.

For any $0<x_n\ll1$ sufficiently small consider $(\R_-(\cdot;\R_{0,n}),W_-(\cdot;\R_{0,n}))$ with $\R_{0,n}=\R_0(x_n)=
 \frac{1}{x_n^{1-\k}}  \left( \frac{3^{1-\k}}{2(1+8\k)} \right)^{\frac{1-\k}{2}}>1$ for $n$ large enough. By Lemmas~\ref{L:RHOMINUS} and~\ref{L:XIMP}
\be
\R_-(x_n;\R_{0,n})> \frac1{x_n^{1-\k}}>\R(x_n;\bar x_\ast)>W(x_n;\bar x_\ast)>\frac13.\notag 
\ee
On the other hand,
\begin{align}
\R(x_n;\bar x_\ast)>\frac13 = \R_-(x_n;\frac13),\notag
\end{align}
where we recall that $\R_-(\cdot;\frac13)\equiv\frac13$ is the Friedman solution.
Moreover, by Remark~\ref{R:COMMONINTERVAL} $[0,x_n]\subset[0,s_-(\tilde \R_0))$ for all $\tilde \R_0\subset[\frac13,\R_{0,n}]$. By the continuity of the map 
$[\frac13,\R_{0,n}]\ni \tilde \R_0\mapsto \R_-(x_n;\tilde \R_0)$  the Intermediate Value Theorem implies that there exists  $\R_{n,1}\in(\frac13,\R_{0,n})$ such that 
\begin{align}\label{E:EQUALITY}
\R_-(x_n;\R_1^n) = \R(x_n;\bar x_\ast) \ \ \text{ for all sufficiently large $n\in\mathbb N$}.
\end{align}
Let 
\begin{align}
c_{1,n}: =
\begin{cases}
 \exp\left(- \left(\R_{n,1}\right)^{-1+\k}\right), & \text{ if }\ \R_{n,1}> 1; \\
 \exp\left(-1 \right),  &  \text{ if }\ \frac13 <\R_{n,1} \le1. \notag
\end{cases}
\end{align}
Clearly $c_{1,n}\ge e^{-1}=:c_1$ for all $n\in\mathbb N$.
By Lemma~\ref{L:RHOMINUS} $\R_-(x_n;\R_{n,1})\ge c_1\R_{n,1}$. Using~\eqref{E:USGOOD1}--\eqref{E:USGOOD2} we also have
 $W_-(x_n;\R_{n,1})< \R_{n,1}^{\frac{1+\k}{2}} 3^{-\frac{1-\k}{2}}$. Together with~\eqref{E:CONSTANTC} and~\eqref{E:EQUALITY} we conclude that for all $n$ sufficiently large
\begin{align}
W_-(x_n;\R_{n,1})&<\R_{n,1}^{\frac{1+\k}{2}} 3^{-\frac{1-\k}{2}}<\left(\frac{\R_-(x_n;\R_{n,1})}{c_1}\right)^{\frac{1+\k}{2}}3^{-\frac{1-\k}{2}}
= \left(\frac{\R(x_n;\bar x_\ast)}{c_1}\right)^{\frac{1+\k}{2}}3^{-\frac{1-\k}{2}} \notag\\
& < 3^{-\frac{1-\k}{2}} (c_1 C)^{-\frac{1+\k}{2}} W(x_n;\bar x_\ast)^{\frac{1+\k}{2}}. \label{E:UPPERONE}
\end{align}
By~\eqref{E:CONSTANTC} $W(x_n;\bar x_\ast)$ grows to positive infinity as $x_n$ approaches zero. Therefore, we may choose a sufficiently large $N\in\mathbb N$ and set $x_0=x_N\ll1$, 
$\R_0=\R_{0,N}$, $\R_1=\R_{1,N}$ so that the right-hand side of~\eqref{E:UPPERONE} is bounded from above by
$ W(x_0;\bar x_\ast)$. This gives
\begin{align}
W_-(x_0;\R_1) <W(x_0;\bar x_\ast).\notag
\end{align}
We conclude that $(\R(\cdot;\bar x_\ast), W(\cdot;\bar x_\ast))$ is an upper solution (see Definition~\ref{D:UPPERLOWER}) at $x_0$ and the upper bound on $\R_1$ follows from our choice of $\R_0$.

\noindent
{\em Case 2.}
\begin{align}\label{E:USCONTR2}
\frac13<\limsup_{x\to 0} W(x;\bar x_\ast) <\infty, \ \ \liminf_{x\to 0}W(x;\bar x_\ast)=\frac13.
\end{align}
In particular $\bar x_\ast\in \Z$ (see~\eqref{E:ZDEF}) and by Lemma~\ref{L:PREPX} 
$\R(x;\bar x_\ast)>W(x;\bar x_\ast)$.
Assumption~\eqref{E:USCONTR2}  also implies that there exists a constant $c>0$ independent of $x$ such that 
\begin{align}\label{E:OMEGAUNIFORMBOUND}
W(x;\bar x_\ast)<c, \ \ x\in(0,\bar x_\ast].
\end{align}

We now claim that there exists a constant $C$ such that 
\begin{align}\label{E:USRUB}
\R'\ge - \frac{C \R}{x^\k},
\end{align}
for all sufficiently small $x$.
To prove~\eqref{E:USRUB} we note that by~\eqref{E:RODE} this is equivalent to the bound
$
x^{1+\k}(W+\k)(\R-W)\le C B,
$
which, by~\eqref{E:BDEF0} is equivalent to the bound
\begin{align}
&x^2W^2\R^\e \left(-\frac{1}{x^{1-\k}}+C(1-\k)\right) + x^2W\R^{1+\e} \left(\frac{1}{x^{1-\k}}+4\k C)\right) \notag\\
& + x^2 W \R^\e \left(-\frac{\k}{x^{1-\k}}+4\k C\right)
+ x^2 \R^\e \left( \frac{\k}{x^{1-\k}}+C(\k^2-\k)\right) \le C. 
\end{align}
Using~\eqref{E:OMEGAUNIFORMBOUND} and Lemma~\ref{L:APRIORIMINUS}, we have
\[
x^2W^2\R^\e\le c^2 x^{2-2\k}, \ \ x^2W\R^{1+\e} \le c x^{1-\k},  \ \ x^2W\R^\e\le c x^{2-2\k}, \ \ x^2\R^\e\le x^{2-2\k}.
\]
Choosing $\k>0$ sufficiently small, $x$ so small that $-\frac{1}{x^{1-\k}}+(1-\k)<0$, and $C>1$ sufficiently large, but independent of $\k$, we use the above bounds to conclude the claim.

%Since $B>0$ for any $x\in(0,\bar x_\ast)$, by~\eqref{E:BDEF0} this equivalently reads
%\begin{align}
%R^{-\e} + (\k-\k^2)x^2 > 
%\end{align}
%

%From~\eqref{E:RODE} and the bound~\eqref{E:RHOAPRIORI} 
%we conclude
%\begin{align}
%\rho'(z;\bar y_\ast) \ge - C \rho(z;\bar y_\ast), \notag 
%\end{align}
%or equivalently $\left(\rho e^{Cz}\right)'\ge0$; here $C>0$.  This implies the boundedness of $\rho(\cdot;\bar y_\ast)$, i.e. 

Bound~\eqref{E:USRUB} gives $\left(\log \R + \frac{C}{1-\k}x^{1-\k}\right)'\ge 0$, which in turn gives the bound
\begin{align}
\R(x) \le \R(\bar x)  e^{ \frac{C}{1-\k}\bar x^{1-\k} - \frac{C}{1-\k} x^{1-\k}}, \ \ x\in(0,\bar x],
\end{align}
for some $\bar x$ sufficiently small. This immediately implies the uniform boundedness of $\R(\cdot;\bar x_\ast)$, i.e. 
\begin{align}\label{E:RHOUNIFORMBOUND}
\R(x;\bar x_\ast)<c, \ \ x\in(0,x_\ast],
\end{align}
where we have (possibly) enlarged $c$ so that~\eqref{E:OMEGAUNIFORMBOUND} and~\eqref{E:RHOUNIFORMBOUND} are both true. 
There exists an $\delta>0$ and a sequence $\left\{x_n\right\}_{n\in\mathbb N}$ such that $\lim_{n\to\infty}x_n=0$
and 
\[
\frac13 + \delta < W(x_n;\bar x_\ast), \ \ \text{ and } \ \ \lim_{n\to \infty} W(x_n;\bar x_\ast) = \limsup_{x\to 0} W(x;\bar x_\ast).
\] 
Since $\{ \R(x_n;\bar x_\ast)\}_{n\in\mathbb N}$ is bounded, by Lemma~\ref{L:RHOMINUS} we can choose an $\R_0>1$ such that 
$
\R_-(x_n;\R_0) > \R(x_n;\bar x_\ast)
$
for all $n\in\mathbb N$. 
On the other hand $\R(x_n;\bar x_\ast)>\frac13=\R_-(x_n;\frac13)$. By the intermediate value theorem there exists a 
sequence $\{\R_{0,n}\}_{n\in\mathbb N}\subset (\frac13,\R_0)$ such that
\begin{align}
\R_-(x_n;\R_{0,n}) = \R(x_n;\bar x_\ast).\notag
\end{align}
Since $W_-(x;\R_{0,n})^2\le \frac{\R_{0,n}^{1+\k}}{3^{1-\k}}<\frac{\R_0^{1+\k}}{3^{1-\k}}$ and $\R_-(x;\R_{0,n})<\R_{0,n}<\R_0$ (Lemma~\ref{L:APRIORIMINUS0}) we conclude
from~\eqref{E:RODE}--\eqref{E:WODE} 
and Theorem~\ref{T:SONICLWPORIGIN} that $\lv \R_-'(x_n;\R_{0,n}) \rv$ and $\lv W_-'(x_n;\R_{0,n})\rv$ 
are bounded uniformly-in-$n$, by some constant $\tilde C>0$. 
Therefore
\begin{align}
W_-(x_n;\R_{0,n}) \le \frac13 + \tilde C x_n.\notag
\end{align}
We thus conclude that for a fixed $n$ sufficiently large $W_-(x_n;\R_{0,n})\le \frac13 + \delta <W(x_n;\bar x_\ast)$. Therefore, $W(\cdot;\bar x_\ast)$ is an upper solution (see Definition~\ref{D:UPPERLOWER}) at 
$x_0:=x_n$ with $\R_1=\R_{0,n}$.
The claimed upper bound on $\R_1$ is clear.

\noindent
{\em Case 3.}
\begin{align}
\frac13<\limsup_{x\to 0} W(x;\bar x_\ast) =\infty, \ \ \liminf_{x\to0}W(x;\bar x_\ast)=\frac13.\notag
\end{align}
As $W(\cdot;\bar x_\ast)$ must oscillate between $\frac13$ and $\infty$ we can use the mean value theorem to conclude that there exists a sequence $\left\{x_n\right\}_{n\in\mathbb N}$ 
such that $\lim_{n\to\infty}x_n=0$ and
\begin{align}\label{E:CRUCIAL0}
W(x_n;\bar x_\ast)>n, \ \ \text{ and } \ \ W'(x_n;\bar x_\ast)=0.
\end{align}
We claim that 
there exist $N_0>0$ and $0<\eta\ll1$ such that
\begin{align}\label{E:CRUCIAL}
W(x_n;\bar x_\ast)\ge  \frac{1}{2x_n^{1-\k}}, \ \ n\ge N_0. 
%n\in\mathbb N,
\end{align}
%where $r<1$ is the positive root of the quadratic polynomial $3x^2+2x-3$.  
To prove this, assume that~\eqref{E:CRUCIAL} is not true. 
Then there exists a subsequence of $\{x_n\}_{n\in\mathbb N}$ such that $W(x_n;\bar x_\ast)< \frac{1}{2x_n^{1-\k}}$. 
We now rewrite~\eqref{E:WODE} in the form
\begin{align}\label{E:USWODEBOUND}
W' = \frac{W}{x}\left(\frac1W - (3-\gamma) +\frac{-\gamma B + 2(1+\k)x^2(W+\k)(\R-W)}{B}\right),
\end{align}
where $\gamma\in(0,3)$ is a control parameter to be chosen below.
We use~\eqref{E:BDEF0} to evaluate
\begin{align}
& -\gamma B + 2x^2(1+\k)(W+\k)(\R-W) \notag \\
& =x^2\left[\gamma\left(( W + \k)^2 - \k (W - 1)^2 + 4\k \R W\right) - 2(1+\k)(W+\k)(W-\R)\right] -\gamma \R^{-\e}\notag \\
& = x^2\Big\{ \left[\gamma(1-\k)-2(1+\k)\right] W^2 + \left[4\gamma \k- 2(1+\k)\k+\left(4\gamma \k+2(1+\k)\right)\R\right]W \notag \\
& \ \ \ \ \qquad + \gamma(\k^2-\k) + 2 \k(1+\k)\R \Big\} - \gamma \R^{-\e}.
%& =: G_\gamma[x;R,W]\label{E:GAMMAONE}
\end{align}
Since $\R\ge\frac13$, by Lemma~\ref{L:APRIORIMINUS} we also have the bound $\R(x)\le \frac1{x^{1-\k}}$.
We may therefore estimate 
\begin{align}
&\R^\e\left( -\gamma B + 2x^2(1+\k)(W+\k)(\R-W)\right) \notag\\
& \le x^2 \R^\e \left\{\left[\gamma(1-\k)-2(1+\k)\right] W^2 + 4\gamma \k W\right\} + \left[\left(4\gamma \k+2(1+\k)\right)\right]x^2W\R^{1+\e} \notag \\
& \ \ \ \  + 2 \k(1+\k)x^2\R^{1+\e}  - \gamma.
\end{align}
By~\eqref{E:CRUCIAL0}, for any $\gamma\in(0,\frac32)$ we have $\left[\gamma(1-\k)-2(1+\k)\right] W(x_n;\bar x_\ast)^2 + 4\gamma \k W(x_n;\bar x_\ast)<0$ for all sufficiently large $n$.
We use the upper bounds on $(\R,W)$ along the sequence $\{x_n\}_{n\in\mathbb N}$ to obtain 
\[
x^2W\R^{1+\e} \le \frac12 \ \text{ and } \ %\ x^2 W R^\e \le  
x^2 \R^{1+\e} \le  x^{1-\k}, \ \ 
\]
and therefore
\begin{align}
&\R^\e\left( -\gamma B + 2x^2(1+\k)(W+\k)(\R-W)\right)\Big|_{x=x_n}  \le2\gamma\k + 1 + \k + 2\k(1+\k)x_n^{1-\k}-\gamma <0 \notag\\
\end{align}
for some $\gamma\in(1,\frac32)$ and all sufficiently small $\k$. Feeding this back into~\eqref{E:USWODEBOUND}, we conclude
\begin{align}
W'(x_n;\bar x_\ast)&  \le \frac{W(x_n;\bar x_\ast)}{x_n} \left(\frac1{W(x_n;\bar x_\ast)} - (3-\gamma)\right)<0, 
\end{align}
 for all sufficiently large $n$, where we use the first bound in~\eqref{E:CRUCIAL0}, which contradicts the second claim in~\eqref{E:CRUCIAL0}.
We can therefore repeat the same argument following~\eqref{E:CRUCIAL00} to conclude that $W(\cdot;\bar x_\ast)$ is an upper solution at $x_0:=x_n$, for some $n$ sufficiently large. 
The upper bound on $\R_1$ follows in the same way.
\end{proof}

We next show that $W(\cdot;\bar x_\ast)$ takes on value $\frac13$ at $x=0$.

%%%%%%%%%%%%%%%%%%%%%%%%%%%%%%%
%%%%%%%%%%%%%%%%%%%%%%%%%%%%%%%

\begin{proposition}\label{P:ONETHIRD}
The limit $\lim_{x\to0}W(x;\bar x_\ast)$ exists and
\be
\lim_{x\to0}W(x;\bar x_\ast) = \frac13.\notag
\ee
\end{proposition}

%%%%%%%%%%%%%%%%%%%%%%%%%%%%%%%

\begin{proof}
Assume that the claim is  not true. By Lemmas~\ref{L:LOWERSOLUTION} and~\ref{L:CONTR} we can find a $0<x_0\ll1$ and $x_{\ast\ast}\in X $ so that $(\R(\cdot;x_{\ast\ast}),W(\cdot;x_{\ast\ast}))$  
and $(\R(\cdot;\bar x_\ast),W(\cdot;\bar x_\ast))$   are respectively a lower and an upper solution at $x_0$.
Without loss of generality let 
\[
A: =\R(x_0;x_{\ast\ast})< \R(x_0;\bar x_\ast)=:B.
\]
By Lemmas~\ref{L:LOWERSOLUTION} and~\ref{L:CONTR} there exist $\R_A,\R_B>\frac13$ such that $A=\R_-(x_0;\R_A)$,  $B=\R_-(x_0;\R_B)$, and
$\R_A,\R_B\in(\frac13,\R_0)$, where $\R_0\gg1$ and $x_0 \le C\frac1{\R_0}$. Therefore by Lemma~\ref{L:MONOTONE}, $\pa_{\R_0}\R_-(x_0;\tilde \R_0)>0$ for all $\tilde \R_0\in[\frac13,\R_0]$, 
since $\R_0^{- \tfrac{3+\k+2\e}{4} }\gg \R_0^{-1}$ for $\R_0$ large and all $\k\le\k_0$ with $\k_0$ sufficiently small. By the inverse function theorem, there exists a continuous function $\tau\mapsto g(\tau)$ such that 
\begin{align}
\R_-(x_0;g(\tau)) & = \tau,  \ \ \tau\in[A,B], \notag\\
g(\R_A) & = A.\notag
\end{align} 
By strict monotonicity of $\tilde \R_0\mapsto \R_-(x_0;\tilde \R_0)$ on $(0, \R_0]$ the inverse $g$ is in fact injective and therefore $g( \R_B)=B$. 
We consider the map 
\begin{align}
[\bar x_\ast,x_{\ast\ast}] \ni x_\ast \mapsto W(x_0;x_\ast) - W_-(x_0;g( \R(x_0;x_\ast))) = : h(x_\ast).\notag
\end{align}
By the above discussion $h$ is continuous, $h(A)<0$ and $h(B)>0$. Therefore, by the Intermediate Value Theorem there exists an $x_s\in(\bar x_\ast,x_{\ast\ast})$ such that $h(x_s)=0$.  
The solution $( \R(\cdot;x_s),W(\cdot;x_s))$ exists on $[0,x_s]$, satisfies $W(0)=\frac13$ and belongs to $X $. This is a contradiction to the definition of $X $. %~\eqref{E:WLESSTHANTHIRDX}.
\end{proof}

%%%%%%%%%%%%%%%%%%%%%%%%%%%%%%%
%%%%%%%%%%%%%%%%%%%%%%%%%%%%%%%

In the next proposition we will prove that the solution $(\R(\cdot;x_\ast), W(\cdot;x_\ast))$ is analytic in a left neighbourhood of $x=0$, by showing that it coincides with a solution emanating from the origin.

\begin{proposition}\label{P:GOODPROPERTIES}
There exists a constant $C_\ast>0$ so that
\begin{align}
\lv \R(x;\bar x_\ast)\rv + \lv W(x;\bar x_\ast)\rv
 +  \lv \frac{W(x;\bar x_\ast)-\frac13}{x^2}\rv
\le C_\ast, \ \ x\in(0,\bar x_\ast].\notag
\end{align}
The solution $\R(\cdot;\bar x_\ast):(0,\bar x_\ast]\to\mathbb R_{>0}$ extends continuously to $x=0$ and 
$\R_\ast:=\R(0;\bar x_\ast)<\infty$. Moreover, the solution $(\R(\cdot;\bar x_\ast), W(\cdot;\bar x_\ast))$ coincides with
$(\R_-(\cdot;\R_\ast), W_-(\cdot;\R_\ast))$ and it is therefore analytic at $x=0$ by Theorem~\ref{T:SONICLWPORIGIN}.
\end{proposition}

%%%%%%%%%%%%%%%%%%%%%%%%%%%%%%%
\begin{proof}
By Proposition~\ref{P:ONETHIRD} it is clear that $W(\cdot;\bar x_\ast)$ is bounded on $[0,\bar x_\ast]$. Moreover, $\R\ge W\ge \frac13$ on $[0,\bar x_\ast]$ and therefore by Lemma~\ref{L:APRIORIMINUS} $\R\le\frac1{x^{1-\k}}$. We may now apply the exact same argument as in the proof of~\eqref{E:RHOUNIFORMBOUND} to conclude that $\R$ is uniformly bounded on $[0,\bar x_\ast]$. Since $B>0$ and $\R>W$ on $[0,\bar x_\ast]$, by~\eqref{E:RODE} $\R'\le0$
and therefore the limit $\R_\ast:=\lim_{x\to0}\R(x;\bar x_\ast)$ exists and  it is finite. It is also clear that there exists a constant $c_0>0$ such that 
\be\label{E:BLBUNIF}
B(x)>c_0, \ \  x\in[0,\bar x_\ast-\tilde r].
\ee

%From Lemma~\ref{L:APRIORIMINUS} we have $x(R-W)\le x^{\k}$. From~\eqref{E:RODE} we conclude $|R'| \le CR$ and thus $R$ is  bounded up to $x=0$. Using the boundedness of $R$ and $W$, equation~\eqref{E:RODE} immediately implies $|R'(x)|\lesssim x$ for all $z\in[0,1]$.

%Since $R(\cdot;\bar y_\ast)- W(\cdot;\bar y_\ast)\ge0$ on $(0,1]$ and  both $R(\cdot;\bar y_\ast)$ and $ W(\cdot;\bar y_\ast)$ are positive, we conclude from~\eqref{E:RHOEQN} that $R'\le0$ and therefore the limit $R_\ast:=\lim_{z\to0}R(z;\bar y_\ast)$ exists and by the above it is finite.
Let
\[
\zeta =  W-\frac13\ge0.
\]
It is then easy to check from~\eqref{E:WODE}
\begin{align}
\left(\zeta x^3\right)' = x^3 \frac{2(1+\k) x  W(W+\k)(\R- W)}{B}\notag
\end{align}
and therefore, by~\eqref{E:BLBUNIF} and the uniform boundedness of $W$ and $\R$, for any $0<x_1<x$ we obtain
\[
\zeta(x) x^3 - \zeta(x_1) x_1^3 \le C \int_{x_1}^{x}\tau^4\,d\tau  = \frac C5\left(x^5-x_1^5\right), \ \ x\in(0,\bar x_\ast-\tilde r].
\]
We now let $x_1\to0^+$ and conclude 
\be\label{E:ZETABOUND}
\zeta(x)\lesssim x^2. 
\ee

Consider now $\bar \R(x):=\R(x;\bar x_\ast)-\R_-(x;\R_\ast)$ and $\bar W(x):= W(x;\bar x_\ast)- W_-(x;\R_\ast)$, both are defined in a (right) neighbourhood of $x=0$ and satisfy
\begin{align}
\bar \R' &= O(1)\bar \R + O(1)\bar W, \notag \\
\bar W' & = - 3\frac{\bar W}{x} + O(1)\bar \R + O(1)\bar W, \notag
\end{align}
where $\bar \R(0)=\bar W(0)=0$. Here we have used the already proven boundedness of $(\R(\cdot;\bar x_\ast), W(\cdot;\bar x_\ast))$ and the boundedness of $(\R_-(\cdot;\R_\ast), W_-(\cdot;\R_\ast))$, see Lemma~\ref{L:APRIORIMINUS0}.
We multiply the first equation by $\bar \R$, the second by $\bar W$, integrate over $[0,x]$ and use Cauchy-Schwarz to get
\begin{align}\label{E:UNIQUE}
\bar \R(x)^2+\bar W(x)^2 + 3\int_0^x\frac{\bar W(\tau)^2}{\tau}\,d\tau  \le C \int_0^x \left(\bar \R(\tau)^2+\bar W(\tau)^2\right)\,d\tau.
\end{align}
We note that $\int_0^x\frac{\bar W(\tau)^2}{\tau}\,d\tau $ is well-defined, since $\bar W = \zeta - \zeta_-$, where $\zeta_-= W_--\frac13$; we use~\eqref{E:ZETABOUND} and observe $\zeta_-\lesssim x^2$ in the vicinity of $x=0$ by the analyticity of $ W_-$, see Theorem~\ref{T:SONICLWPORIGIN}. Therefore $\bar \R(x)^2+\bar W(x)^2=0$ by~\eqref{E:UNIQUE}. The analyticity claim now follows from Theorem~\ref{T:SONICLWPORIGIN}.
\end{proof}

\begin{remark}
Propositions~\ref{P:ONETHIRD} and~\ref{P:GOODPROPERTIES} follow closely the arguments in the Newtonian case (Propositions 4.22 and 4.23 in~\cite{GHJ2021}). 
However, an important difference is the use of~\eqref{E:RHOUNIFORMBOUND} in the proof of Proposition~\ref{P:GOODPROPERTIES}.
\end{remark}

{\em Proof of Theorem~\ref{T:FRIEDMANN}.}
Theorem~\ref{T:FRIEDMANN} is now a simple corollary of Propositions~\ref{P:GLOBALFRIEDMAN},~\ref{P:ONETHIRD}, and~\ref{P:GOODPROPERTIES}.
\prfe

%%%%%%%%%%%%%%%%%%%%%%%%%%%%%%%
%%%%%%%%%%%%%%%%%%%%%%%%%%%%%%%
%%%%%%%%%%%%%%%%%%%%%%%%%%%%%%%
%%%%%%%%%%%%%%%%%%%%%%%%%%%%%%%
%%%%%%%%%%%%%%%%%%%%%%%%%%%%%%%

% The following section contains the analysis to the RIGHT of the sonic point

\section{The far field connection}\label{S:FARFIELD}

%%%%%%%%%%%%%%%%%%%%%%%%%%%%%%%
%%%%%%%%%%%%%%%%%%%%%%%%%%%%%%%
%%%%%%%%%%%%%%%%%%%%%%%%%%%%%%%
%%%%%%%%%%%%%%%%%%%%%%%%%%%%%%%
%%%%%%%%%%%%%%%%%%%%%%%%%%%%%%%

The goal of this section is to construct a global-to-the-right solution of the dynamical system~\eqref{E:RODE}--\eqref{E:WODE}
for all $\xs$ in the sonic window $[\xmin,\xmax]$. We refer to Section~\ref{SS:FFINTRO} for a detailed overview.

\subsection{A priori bounds} 

\begin{remark}\label{R:EQ}
The condition $\F >x\R$ is equivalent to the inequality
\begin{align}\label{E:SONICASS3}
\R^{-\e}  > x^2\left((1-\k)\R^2-\k(1-\k)+4\k \R(1+\R)\right) ,
\end{align}
which can be seen directly from the definition~\eqref{E:FDEF} of $\F [x;\R]$. Similarly, the inequality
$xW>\F $ is equivalent to
\begin{align}\label{E:SONICASS4}
\R^{-\e} < x^2\left((1-\k)W^2- \k(1-\k) +4\k W(1+\R) \right).
\end{align}
\end{remark}

%%%%%%%%%%%%%%%%%%%%%%%
%%%%%%%%%%%%%%%%%%%%%%%

%%%%%%%%%%%%%%%%%%%%%%%
%%%%%%%%%%%%%%%%%%%%%%%

\begin{lemma}[A priori bounds to the right]\label{L:APRIORI}
Let $x_\ast\in[\xmin,\xmax]$.
Let $X>x_\ast$ be the maximal time of existence to the right of the associated RLP-type solution on which we have
\begin{align}
xW(x) & >\F [x;\R],   \ \ x\in(x_\ast+\delta,X), \label{E:SONICASS}\\
%W& >R,   \ \ x\in(x_\ast,X).\label{E:WRASS} 
\F [x;\R]&>x \R(x), \ \ x\in(x_\ast+\delta,X), \label{E:SONICASS2}
\end{align}
where $0<\delta\ll1$ is an $\k$-independent constant from Lemma~\ref{L:INITIAL}.
Then the following claims hold
\begin{enumerate}
\item[{\em (a)}]
\begin{align}
\R' & <0,  \ \ x\in(x_\ast+\delta,X) ,\label{E:RBOUND}\\
W & >\frac16,  \ \ x\in(x_\ast+\delta,X) ,\label{E:WLOWERBOUND} \\
\R(x)W(x) & > \R_\delta^2 \left(\frac{x_\ast+\delta}{x}\right)^3, \ \ x\in(x_\ast+\delta,X), \label{E:RWLOWERBOUND}
\end{align}
where 
\[
%W_\delta = \frac13 - O(\delta)>\frac16
\R_\delta:=\R(x_\ast+\delta).
\]
\item[{\em (b)}]
There exists a constant $\tilde C$ {\em independent} of $X$ and $\k_0>0$ such that for all $0<\k\le\k_0$
\begin{align}\label{E:WUPPERBOUND}
W(x)\le \tilde C, \ \  x\in(x_\ast+\delta,X).
\end{align}
\item[(c)]
There exist constants $C>0$, $0<\k_0\ll1$ 
%and a $0<1-C\k_0<\gamma<1$
such that for all $\k\in(0,\k_0)$ 
\begin{align}\label{E:RUB}
\frac{\R'x}{\R} \le - \gamma, \ \ x\in(x_\ast+\delta,X), \ \ \gamma := 1-C\k>0
\end{align}
and therefore
\begin{align}\label{E:RUB2}
\R(x) \le \R_\delta \left(\frac{x}{x_\ast+\delta}\right)^{-\gamma}.
\end{align}
In particular, with $0<\k_0\ll1$ sufficiently small, we may choose $\gamma = 1- C\k_0$ uniform for all $\k\in(0,\k_0)$.
\end{enumerate}
\end{lemma}

\begin{proof}
We note that by Lemma~\ref{L:INITIAL} there exists a $\delta>0$ such that $X>x_\ast+2\delta$. 

\noindent
{\em Proof of part {\em (a)}.}
Notice that by assumptions~\eqref{E:SONICASS}--\eqref{E:SONICASS2} 
we have
\begin{align}
W& >\R,   \ \ x\in(x_\ast+\delta,X).\label{E:WRASS} 
\end{align}
Bound~\eqref{E:RBOUND} follows from~\eqref{E:RODE},~\eqref{E:BDEF},~\eqref{E:SONICASS}, and~\eqref{E:WRASS}.

By~\eqref{E:WODE} and~\eqref{E:SONICASS}--\eqref{E:SONICASS2} we conclude $W'\ge \frac{1-3W}{x}$ or equivalently 
$(Wx^{3})'\ge x^{2}$.
%\left(1+\lambda x^{2+\lambda}W\right)\ge x^{2+\lambda}$.
We conclude that for any $x\in(x_\ast+\delta,X)$ we have
\begin{align}
W(x) \ge W_\delta \left(\frac{x_\ast+\delta}{x}\right)^{3} + \frac1{3}\left(1-\left(\frac{x_\ast+\delta}{x}\right)^{3}\right) \ge\frac1{6},
\end{align}
since $W_\delta:=W(x_\ast+\delta)%W_0(\k,\xs+\delta)
=W_0(\k,\xs)-O(\delta)\ge W_0(\k;\xmax(\k))-O(\delta)=\frac1{3}-O(\delta)$, for $0<\delta\ll1$ sufficiently small and independent of $\k$. Here we have
 also used Remark~\ref{R:XSDECREASING} and the definition of $\xmax$ in~\eqref{E:XMAXDEF}.

From~\eqref{E:RODE}--\eqref{E:WODE} we have
\begin{align}
x^{1-\beta}\left(W\R x^{\beta}\right)' & = \R' W x + \R W' x + \beta \R W \notag \\
& = \R + (\beta-3) \R W +  \frac{4\k x^2 \R W (W + \k) (\R-W)}{B}. 
\end{align}
We may let $\beta=3$ and conclude from~\eqref{E:SONICASS},~\eqref{E:WRASS} that $\frac{d}{dx}\left(W\R x^{\beta}\right)>0$ on $(x_\ast+\delta,X)$. In particular
\begin{align}
\R(x) W(x) > \R_\delta^2 \left(\frac{x_\ast+\delta}{x}\right)^3, \ \ x\in(x_\ast+\delta,X),
\end{align}
which is~\eqref{E:RWLOWERBOUND}.

\noindent
{\em Proof of part {\em (b)}.}
Let now $\gamma\in(0,3)$ and rewrite~\eqref{E:WODE} in the following way:
\begin{align}\label{E:WGAMMA}
W' = \frac{1-(3-\gamma)W}{x} + \frac{W}{xB}\left(-\gamma B + 2x^2(1+\k)(W+\k)(\R-W)\right). 
\end{align}
We use~\eqref{E:BDEF0} to evaluate
\begin{align}
& -\gamma B + 2x^2(1+\k)(W+\k)(\R-W) \notag \\
& =x^2\left[\gamma\left(( W + \k)^2 - \k (W - 1)^2 + 4\k \R W\right) - 2(1+\k)(W+\k)(W-\R)\right] -\gamma \R^{-\e}\notag \\
& = x^2\Big\{ \left[\gamma(1-\k)-2(1+\k)\right] W^2 + \left[4\gamma \k- 2(1+\k)\k+\left(4\gamma \k+2(1+\k)\right)\R\right]W \notag \\
& \ \ \ \ \qquad + \gamma(\k^2-\k) + 2 \k(1+\k)\R \Big\} - \gamma \R^{-\e} \notag \\
& =: G_\gamma[x;\R,W]. \label{E:GAMMAONE}
\end{align}
By~\eqref{E:RWLOWERBOUND} we know that for any $x\in(x_\ast+\delta,X)$,
%\begin{align}
$\R(x) > \R_\delta^2  (x_\ast+\delta)^3 x^{-3}W(x)^{-1},$
%\end{align}
and therefore
\begin{align}\label{E:GAMMATWO}
\R^{-\e} <\R_\delta^{-2\e}  (x_\ast+\delta)^{-3\e} x^{3\e}W(x)^{\e}, \ \ x\in(x_\ast+\delta, X).
\end{align}

Fix now any $\gamma\in(\frac{2(1+\k)}{1-\k},3)$, which is possible since $\k$ can be chosen small. The expression $G_\gamma[x;\R,W]$ in~\eqref{E:GAMMAONE}
can be bounded from below with the help of~\eqref{E:GAMMATWO} by 
\begin{align}\label{E:GAMMATHREE}
g[x,W]: = x^2(1-\k)\left(\gamma - \frac{2(1+\k)}{1-\k} \right)W^2 + x^2\chi_1[\R] W + x^2\chi_2[\R] - C_0(x_\ast) x^{3\e}W(x)^{\e},
\end{align}
where
\begin{align}
\chi_1[\R] &: = 4\gamma \k- 2(1+\k)\k+\left(4\gamma \k+2(1+\k)\right)\R,  \notag \\
\chi_2[\R] & : = \gamma(\k^2-\k) + 2 \k(1+\k)\R,  \notag \\
C_0(x_\ast) & : = \gamma \R_\delta^{-2\e}  (x_\ast+\delta)^{-3\e} \notag.
\end{align}
By the already proven bound~\eqref{E:RBOUND} it follows that $\R(x)\le \R_\delta$. Therefore, from the above formulas, there exists a universal constant $C$ such that 
\[
|\chi_1[\R]| + |\chi_2[\R]| + |C_0(x_\ast)| \le C, \ \ x\in(x_\ast+\delta, X), \ \ x_\ast\in[\xmin,\xmax]. 
%[2,3].
\]
Since $x>x_\ast+\delta\ge1$ and for sufficiently small $\k$ we have $3\e<2$ (recall~\eqref{E:ETADEF}), it follows that there exists a constant $\bar C=\bar C(\gamma)$ {\em independent} of $X$ and $\k$ such that 
\begin{align}\label{E:GGAMMA}
g[x,W]>1, \ \ \text{ if } \ W(x)>\bar C.
\end{align}
Let now $\tilde C = \max\{\bar C, \frac{2}{3-\gamma}\}$. It then follows  that both summands on the right-hand side of~\eqref{E:WGAMMA} are strictly negative when $W\ge \tilde C$; the first term is negative due to $1-(3-\gamma) C<0$ and the second one due to~\eqref{E:GGAMMA} and the negativity of $B$. Therefore the region $W\ge \tilde C$ is dynamically inaccessible. This completes the proof of~\eqref{E:WUPPERBOUND}.

\noindent
{\em Proof of part {\em (c)}.}
For any $\gamma\in\mathbb R$ it follows from~\eqref{E:RODE} that
\begin{align}\label{E:RPRIMEGAMMA}
\frac{\R'x}{\R} + \gamma = \frac{\gamma B - 2(1-\k)x^2(W+\k)(\R-W)}{B}.
\end{align}
We focus on the numerator of the right-hand side above. By \eqref{E:BDEF0} %~\eqref{E:BDEF} 
we have
\begin{align}
& \gamma B - 2(1-\k)x^2(W+\k)(\R-W) \notag \\
& = \gamma \R^{-\e} -\gamma \left[ ( W + \k)^2 - \k (W - 1)^2 + 4\k \R W \right] x^2- 2(1-\k)x^2(W+\k)(\R-W) \notag \\
& \ge \gamma x^2\left((1-\k)\R^2-\k(1-\k)+4\k \R(1+\R)\right) \notag \\
& \ \ \ \ -\gamma \left[ ( W + \k)^2 - \k (W - 1)^2 + 4\k \R W \right] x^2- 2(1-\k)x^2(W+\k)(\R-W) \notag \\
& = x^2 (W-\R) \left[\gamma (W-\R) + (2-2\gamma) W - \k \left(3\gamma \R + W(2-\gamma) + 4\gamma+2\k -2\right)\right]. \label{E:BETASHORTER}
\end{align}
In the third line above we crucially used the bound~\eqref{E:SONICASS3}, which by Remark~\ref{R:EQ} is equivalent to the assumption~\eqref{E:SONICASS2}.
Since $W>\R$, $\tilde C>W\ge\frac16$, and $0<\R<\R_\delta$ on $(x_\ast+\delta,X)$, we may choose $\gamma=1-C\k$ with $C$ sufficiently large, but $\k$-independent, so that the above expression is strictly positive, since in this case $(2-2\gamma) W\ge \frac{C\k}{3}$.
Therefore, for sufficiently small $\k$ and $\gamma=1-C\k$, 
the right-hand side of~\eqref{E:RPRIMEGAMMA} is negative, since $B<0$ on $(x_\ast,X)$. This proves~\eqref{E:RUB}, which after an integration gives~\eqref{E:RUB2}.
\end{proof}

%%%%%%%%%%%%%%%%%%%%%%
%%%%%%%%%%%%%%%%%%%%%%
%%%%%%%%%%%%%%%%%%%%%%

\subsection{Monotonicity and global existence}

%%%%%%%%%%%%%%%%%%%%%%
%%%%%%%%%%%%%%%%%%%%%%
%%%%%%%%%%%%%%%%%%%%%%

%%%%%%%%%%%%%%%%%%%%%%
%%%%%%%%%%%%%%%%%%%%%%

\begin{lemma}[Monotonicity properties of the flow]\label{L:FLOWERBOUND}
Let $x_\ast\in[\xmin,\xmax]$.
Let $X>x_\ast$ be the maximal time of existence to the right of the associated RLP-type solution on which we have
\begin{align}
xW(x) & >\F [x;\R],   \ \ x\in(x_\ast+\delta,X). \label{E:SONICASS1.1}
%W& >R,   \ \ x\in(x_\ast,X).\label{E:WRASS1.1} 
\end{align}
Then there exists an $0<\k_0\ll1$ such that for all $\k\in(0,\k_0]$ the associated RLP-type solution satisfies,
\begin{align}
\F [x;\R]>x \R(x), \ \ x\in(x_\ast+\delta,X). \label{E:SONICASS2.1}
\end{align}
If in addition $X<\infty$, then 
there exists a constant $\bar{\kappa} = \bk(X)>0$ such that $f(x)=\F [x;\R]-x\R(x)>\bk$ for all $x\in(x_\ast+\delta, X)$.
\end{lemma}

%%%%%%%%%%%%%%%%%%%%%%%
%%%%%%%%%%%%%%%%%%%%%%%

\begin{proof}
We observe that $b_1[x;\R,W]$ is strictly positive on $(x_\ast+\delta,X)$ by our assumption.
Moreover, from the definition of~\eqref{E:BTWODEF}, it is clear that $b_2[x;\R,W]\ge \tilde b_2[x;\R,W]$, where
\begin{align}
\tilde b_2[x;\R,W]
= &- 2 x^2\R\Big\{\R^2+(3+\k)\R -(1-\k) \notag \\
& \ \ \ \ +\k \left[\frac4{1-\k}\R^2+\frac{2(1+\k)}{1-\k}(1+\R) + \frac{5-\k}{1-\k}\R+3\right]\Big\}Z^{-1}; 
\label{E:BTILDETWODEF}
\end{align}
It is clear that for any $\R\le\frac14$ and $\k_0$ sufficiently small the expression $\tilde b_2[x;\R,W]$ is strictly positive for all $0<\k\le\k_0$. 
If there exists an $\tilde x>x_\ast+\delta$ such that $\R(\tilde x)\le\frac14$,  then by the monotonicity of $x\mapsto \R(x)$, $\R(x)\le\frac14$ for all $x\ge\tilde x$. Thus $\tilde b_2[x;\R,W]$ and therefore $b[x;\R,W]$ remain positive.
It follows that if $f(\tilde x)>0$, then by Corollary~\ref{C:FFORMULA}
$f(x)>0$ for all $x\in[\tilde x, X)$. 

We next show that for a sufficiently small $\k_0$ the value of $\R$ always drops below $\frac14$ at some $\tilde x$ for all $\k\in(0,\k_0]$  so that $f$ remains strictly positive on 
$(x_\ast+\delta,\tilde x]$. By the bound~\eqref{E:RUB2} and $\R_\delta\le1$ it follows that $\R(x)\le \frac14$ as long as $x\ge x_\ast 4^{\frac1\gamma}$. Since $\gamma = 1+O(\k)$ and $x_\ast\le3$, then 
with 
\be\label{E:XTILDEDEF}
\tilde x = 3 \times 4^{\frac1{1-\frac12}} = 48
\ee
it follows that if $X>\tilde x$ and $f(x)>0$ on $(x_\ast+\delta, \tilde x]$, then 
$f(x)>0$ for all $x\in(x_\ast+\delta, X)$, with $\k_0$ chosen sufficiently small.

It remains to show that for all $\k\in(0,\k_0]$ with $\k_0$ chosen sufficiently small, we indeed have $f>0$ on $(x_\ast+\delta, \min\{\tilde x,X\})$.
Let $(x_\ast+\delta, X_1)\subset (x_\ast+\delta, X)$ be the maximal subinterval of $(x_\ast+\delta, X)$, such that $f>0$ on $(x_\ast+\delta, X_1)$. By Lemma~\ref{L:INITIAL} we have $X_1\geq x_\ast+2\delta$.
Notice that on the interval $(x_\ast+\delta,X_1)$, by the strict negativity of $B$ the factor $a_1[x;\R,W]$~\eqref{E:AONEDEF} is strictly negative. Since $b_1$ is positive and $b_2\ge \tilde b_2$ we have from Lemma~\ref{L:JUHILEMMA} 
\[
f' + a_2 f\ge \k \tilde b_2. 
\]
For any $x_\ast+\delta<x_1<x<X_1$ this immediately yields
\begin{align}\label{E:FLOWERBOUND}
f(x) \ge  f(x_1) e^{-\int_{x_1}^x a_2[z;\R,W]\,dz}  + \k e^{-\int_{x_1}^x a_2[z;\R,W]\,dz}  \int_{x_1}^x \tilde b_2[z;\R,W] e^{\int_{x_1}^z a_2[s;\R,W]\,ds}\,dz. 
\end{align}

Observe that 
\begin{align}\label{E:ZINVERSEBOUND}
Z^{-1} \le \frac1{x^2\R W}\le \frac6{x^2 \R},
\end{align}
where we have used $H+xW\ge xW$, $(1-\k)\F  + 2\k x \left(1+\R\right)\ge (1-\k)x\R + 2\k x \left(1+\R\right)\ge x\R$, and finally $W\ge\frac16$.
Therefore
\begin{align}
\lv \tilde b_2[x;\R,W]\rv & \le  \frac{2x^2\R}{Z} \Big(\lv \left[\R^2+(3+\k)\R -(1-\k)\right]\rv \notag \\
& \ \ \ \ + 2\k \lv \left[\frac4{1-\k}\R^2+\frac{2(1+\k)}{1-\k}(1+\R) + \frac{5-\k}{1-\k}\R+3\right] \rv \Big) \notag \\
& \le C,\label{E:TILDEBTWOBOUND}
\end{align}
for some universal constant $C$, which follows from~\eqref{E:ZINVERSEBOUND} and the a priori bound $\R\le \R_\delta\le 1$.

We note that by~\eqref{E:ZDEF} and~\eqref{E:HDEF},
\begin{align}
2\k \lv\left(\F -xW\right)\left(\R-1\right) Z^{-1}\rv
& = 2\k\lv \frac{ \left(\F -xW\right)\left(\R-1\right) }{\left[(1-\k)\F  + 2\k x \left(1+\R\right)\right]\left[H+xW\right]} \rv \notag \\
& = 2\k\lv \frac{ \left(\F -xW\right)\left(\R-1\right) }{\left[(1-\k)\F  + 2\k x \left(1+\R\right)\right]\left[\F [\R] +\frac{4\k}{1-\k}(1+\R)x+xW\right]} \rv  \notag \\
& \le 2\k \frac{ \lv \F +xW\rv \left(1-\R_\delta\right) }{\left[(1-\k)\F  + 2\k x \left(1+\R\right)\right]\left[\F  +\frac{4\k}{1-\k}(1+\R)x+xW\right]} \notag \\
& \le  \frac{1-\R_\delta}{ x} \le 1, \label{E:ATWOBOUND1}
\end{align}
where we have used $[(1-\k)\F  + 2\k x \left(1+\R\right) \ge 2\k x$, $\F  +\frac{4\k}{1-\k}(1+\R)x+xW\ge \F +xW$, and $x\ge x_\ast\ge1$.
Similarly, 
\begin{align}
2\k \lv\frac{f}{Z}\rv &\le \frac{\F +x\R}{\left[(1-\k)\F  + 2\k x \left(1+\R\right)\right]\left[H+xW\right]} \notag \\
& \le  2\k \frac{\F +xW}{\left[(1-\k)x\R + 2\k x \left(1+\R\right)\right]\left[H+xW\right]} \notag \\
& \le 2\k \frac1{2\k x} = \frac1x\le 1, \label{E:ATWOBOUND2}
\end{align}
where we have used $f = \F -x\R\le \F +x\R$ in the first line, $\R\le W$ and $\F >x\R$ in the second, and finally $\F +xW\le H+xW$, $(1-\k)x\R + 2\k x \left(1+\R\right)\ge 2\k x$, and $x\ge x_\ast\ge 1$ in the last line.
Also, since $H+xW\ge xW\ge \frac16 x$ and $(1-\k)x\R + 2\k x \left(1+\R\right)\ge 2\k x$, we have the following simple bound
\begin{align}\label{E:ZINVERSEBOUNDTWO}
Z^{-1} \le \frac3{\k x^2}. 
\end{align}
Since $\R\le \R_\delta\le1$, we can use~\eqref{E:ATWOBOUND1},~\eqref{E:ATWOBOUND2},~\eqref{E:ZINVERSEBOUND},~\eqref{E:ZINVERSEBOUNDTWO}, and the definition~\eqref{E:ATWODEF} of $a_2[x;\R,W]$ to conclude
that 
\begin{align}
\lv a_2[x;\R,W]\rv & \le 2\k  \lv\left(\F -xW\right)\left(\R-1\right) Z^{-1}\rv + 4\k  \lv\frac{f}{Z}\rv
+ 2\k Z^{-1} \left(4\k x + x\R\left(5+\k\right)\right) \notag \\
& \le  C.\label{E:ATWOBOUND3}
\end{align}

We now feed the bounds~\eqref{E:TILDEBTWOBOUND} and~\eqref{E:ATWOBOUND3} into~\eqref{E:FLOWERBOUND} to conclude
\begin{align}\label{E:FLOWERBOUND2}
e^{C(x-x_1)}f(x)\ge f(x)e^{\int_{x_1}^x a_2[z;\R,W]\,dz} \ge f(x_1) - \k C e^{C(x-x_1)}(x-x_1), \ \ x\in (x_1, X), \ \ x_1\in[x_\ast+\delta, X_1).
\end{align}
Let $x_1 = x_\ast+\delta$ so that by Lemma~\ref{L:INITIAL} $f(x_1)>\bar c$ for some $\bar c>0$ and all $\k\in(0,\k_0]$ and all $x_\ast\in[\xmin,\xmax]$. 
From~\eqref{E:FLOWERBOUND2} we conclude that we can choose $\k_0$ so that $f(x)>c>0$ for all $x\in (x_\ast+\delta, \min\{\tilde x,X_1\})$, where we recall $\tilde x = 48$. 
By the definition of $X_1$, this concludes the proof of~\eqref{E:SONICASS2.1}. 

To prove the remaining claim, by~\eqref{E:FLOWERBOUND2} we need only to address the case when $X>\tilde x=48$. In this case $\tilde b_2$ is strictly positive for $x>\tilde x$ by
the argument following~\eqref{E:BTILDETWODEF}
and therefore by~\eqref{E:FLOWERBOUND} and~\eqref{E:ATWOBOUND3}
\begin{align}
f(x) \ge  f(\tilde x) e^{-\int_{\tilde x}^x a_2[z;\R,W]\,dz} \ge f(\tilde x) e^{-C(x-\tilde x)},
\end{align}
and thus the claim follows.
\end{proof}

%%%%%%%%%%%%%%%%%%%%%%%%%%%%%
%%%%%%%%%%%%%%%%%%%%%%%%%%%%%

%\begin{theorem}\label{T:GLOBALRIGHT}
%Let $x_\ast\in[\xmin,\xmax]$. There exists an $0<\k_0\ll1$ such that the unique RLP-type solution $(W,\R)$ exists globally to the right for all $\k\in(0,\k_0]$.
%\end{theorem}

%%%%%%%%%%%%%%%%%%%%%%%%%%%%%

We are now ready to prove Theorem~\ref{T:GLOBALRIGHT}, which asserts global-to-the-right existence of RLP-type solutions.

{\em Proof of Theorem~\ref{T:GLOBALRIGHT}.}
Let $[x_\ast,X)$ be the maximal time of existence on which 
\begin{align}
xW>\F , \ \ x\in(x_\ast,X).
\end{align}
By Lemma~\ref{L:FLOWERBOUND} we conclude that $\F >x\R$ on $(x_\ast+\delta, X)$ for all $\k\le\k_0$ sufficiently small and therefore all the conclusions of Lemma~\ref{L:APRIORI} apply.
We argue by contradiction and assume that $X<\infty$. By Lemma~\ref{L:FLOWERBOUND} we know that there exists a constant $\bk = \bk(X)$ such that $f(x)>\bk$ for all $x\in(x_\ast+\delta, X)$. 
In particular $\limsup_{x\to X^-}(W(x)-\R(x))>0$. It follows that 
\be\label{E:CONTR0}
\limsup_{x\to X^-}\left((xW)' - \F '\right)\le0.
\ee
By~\eqref{E:WODE} and~\eqref{E:FPRIME}, we have
\begin{align}
(xW)' - \F ' & =  1-2W(x) + \frac{2x^2(1+\k) W (W + \k) (\R-W)}{B}  \notag\\
&  \ \ \ \ + \frac{ 2\k \F  (1+\R) - \k(1-\k)x}{(1-\k) \F  + 2\k x (1+\R)} +\frac{ (2\k x \F  + \frac{\k}{1-\k} \R^{-\e -1}) \R' }{(1-\k) \F  + 2\k x (1+\R) } \notag \\
& = 1-2W(x) + \frac{ 2\k \F  (1+\R) - \k(1-\k)x}{(1-\k) \F  + 2\k x (1+\R)} \notag \\
& \ \ \ \ + \left( - \frac{x(1+\k)W}{(1-\k)\R} + \k \frac{ 2 x \F  + \frac{1}{1-\k} \R^{-\e -1} }{(1-\k) \F  + 2\k x (1+\R)}  \right) \R' . \label{E:CONTR1}
\end{align}

By the uniform bounds~\eqref{E:WLOWERBOUND} and~\eqref{E:WUPPERBOUND} we have $\frac16\le W \le \tilde C$. By~\eqref{E:RBOUND} and~\eqref{E:RWLOWERBOUND} we have
\begin{align}
\R_\delta \ge \R(x) \ge  \R_\delta^2 \frac{(x_\ast+\delta)^3}{x^3 \tilde C} \ge \R_\delta^2 \frac{(x_\ast+\delta)^3}{X^3 \tilde C}.
\end{align}
Therefore the first line of the right-most side of~\eqref{E:CONTR1} is bounded. Since $\lim_{x\to X^-}\left(xW - \F \right)=0$ it follows that $\lim_{x\to X^-}B = 0$ and therefore
$\limsup_{x\to X^-}\R'(x) = -\infty$ by~\eqref{E:RODE}. On the other hand, the second line of the right-most side of~\eqref{E:CONTR1} is of the form $g(x) \R'$, where
\begin{align}
g(x) & =  - \frac{x(1+\k)W}{(1-\k)\R} + \k \frac{ 2 x \F  + \frac{1}{1-\k} \R^{-\e -1} }{(1-\k) \F  + 2\k x (1+\R)} \notag \\
& = \frac{-xW(1+\k)\left[(1-\k) \F  + 2\k x (1+\R)\right] + \left[2 \k x \F  + \frac{\k}{1-\k} \R^{-\e -1}\right](1-\k)\R }{(1-\k)\R\left[(1-\k) \F  + 2\k x (1+\R)\right]} \notag \\
& = \frac{-(1-\k^2)xW\F  - 2\k(1+\k)x^2 W(1+\R)+2\k(1-\k)x\R \F  + \k \R^{-\e}}{(1-\k)\R\left[(1-\k) \F  + 2\k x (1+\R)\right]}. \label{E:CONTR2}
\end{align}
We now use~\eqref{E:SONICASS4}, which for any $x<X$ allows us to estimate the numerator of~\eqref{E:CONTR2} from above:
\begin{align}
& -(1-\k^2)xW\F  - 2\k(1+\k)x^2 W(1+\R)+2\k(1-\k)x\R \F  + \k \R^{-\e} \notag \\
& < -(1-\k^2)xW\F  - 2\k(1+\k)x^2 W(1+\R)+2\k(1-\k)x\R \F  \notag \\
& \ \ \ \  + \k x^2\left((1-\k)W^2- \k(1-\k) +4\k W(1+\R) \right) \notag \\
& = x^2\left[-(1-\k^2)W\frac{\F }x + 2\k(1-\k)\R\frac \F x +\k(1-\k)W^2 \right] \notag \\
& \ \ \ \ - (2\k-2\k^2) x^2 W(1+\R) - x^2 \k^2(1-\k). 
\end{align}
As $x\to X$, $\frac \F x$ approaches $W$ and therefore the above expression is negative for $\k$ sufficiently small (and independent of $X$). 
Therefore by~\eqref{E:CONTR2} $g(x)<0$ as $x\to X^-$, which implies that the right-most side of~\eqref{E:CONTR1} blows up to $+\infty$ as $x\to X^-$, which is a contradiction
to~\eqref{E:CONTR0}. This concludes the proof of the theorem. %proposition.
\prfe
%%%%%%%%%%%%%%%%%%%

\subsection{Asymptotic behaviour as $x\to\infty$}\label{SS:AS}

We introduce the unknown $\DS(x)=\Sigma(y)$, where $\Sigma$ is the energy density introduced in~\eqref{E:SS1}. Recalling~\eqref{E:RWDEF} and~\eqref{E:LITTLEBOLDDDEF}, we have
\be\label{E:PIDEF}
\DS = \R^{1+\eta} = \R^{\frac{1+\k}{1-\k}}.
\ee
It is straightforward to check that the system~\eqref{E:RODE}--\eqref{E:WODE} can be reformulated as follows
\begin{align}
\DS' &= - \frac{2x(1+\k)\DS(W+\k)(\R-W)}{B}, \label{E:RODE1}\\
W' & = \frac{(1 -3W )}{x} + \frac{2x(1+\k) W (W + \k) (\R-W)}{B}.  \label{E:WODE1}
\end{align}

%Let
%\begin{align}
%u(x): = \frac{D'(x)x}{D(x)}+2, \ \ v(x): = (W-1)x
%\end{align}

\begin{lemma}\label{L:DECAY1}
There exist an $x_0>x_\ast$, a constant $C$ and $\k_0>0$ sufficiently small so that 
for all $x>x_0$
\begin{align}
\frac{\DS'(x)x}{\DS(x)}+2 = \k \alpha(x) (W-1) + \beta(x), \ \ x\in(x_\ast,\infty),
\end{align}
where 
\begin{align}
|\alpha(x)|&\le C \ \  x\ge x_0, \label{E:ALPHABOUND}\\ 
|\beta(x)|& \le C x^{-\frac32}, \ \  x\ge x_0.\label{E:BETABOUND}
\end{align}
%Here $\gamma=1-C\k_0$ is a constant defined in~\eqref{E:RUB}.
\end{lemma}

\begin{proof}
From~\eqref{E:RODE1} and the definition~\eqref{E:BDEF} of $B$ it follows that 
\begin{align}
\frac{\DS' x}{\DS} + 2 & = \frac{2B - 2x^2(1+\k)(W+\k)(\R-W)}{B} \notag \\
& =  2\frac{\R^{-\e} - x^2\left[(W+\k)^2-\k(W-1)^2+4\k \R W + (1+\k)(W+\k)(\R-W)\right]}{\R^{-\eta}-x^2\left((W+\k)^2-\k(W-1)^2+4\k \R W\right)}.  \label{E:DECAYD0} 
\end{align}
We rewrite the rectangular brackets in the numerator above in the form
\begin{align}
&-\left[(W+\k)^2-\k(W-1)^2+4\k \R W + (1+\k)(W+\k)(\R-W)\right] \notag \\
& = \k(W-1)\left[2W+\k-1\right] - \R\left(5\k W +W+\k(1+\k)\right). \label{E:DECAY1}
\end{align}
We feed this back into~\eqref{E:DECAYD0}, divide the numerator and the denominator  by $x^2$, and obtain
\begin{align}
\frac{\DS' x}{\DS} + 2 & = 2\frac{ \k(W-1)\left[2W+\k-1\right] - \R\left(5\k W +W+\k(1+\k)\right)+\R^{-\e}x^{-2}}{-\left((W+\k)^2-\k(W-1)^2+4\k \R W\right)+\R^{-\e}x^{-2}} \notag \\
& = \k \alpha(x) (W-1) + \beta(x),
\end{align}
where
\begin{align}
\alpha(x):= \frac{2\left(2W+\k-1\right)}{-\left((W+\k)^2-\k(W-1)^2+4\k \R W\right)+\R^{-\e}x^{-2}} , \label{E:ALPHADEF}\\
\beta(x):=\frac{-2\R\left(5\k W +W+\k(1+\k)\right)+2\R^{-\e}x^{-2}}{-\left((W+\k)^2-\k(W-1)^2+4\k \R W\right)+\R^{-\e}x^{-2}}. \label{E:BETADEF}
\end{align}
By~\eqref{E:RWLOWERBOUND} and $W\le \tilde C$, it follows that $\R(x)\ge c_1 x^{-3}$ for some universal constant $c_1>0$. Together with~\eqref{E:RUB2} we conclude
\begin{align}
c_1 x^{-3} \le \R(x) \le c_2 x^{-\gamma},
\end{align}
where $\gamma=1-C\k_0$ for some $\k_0\ll1$. We conclude that $\R^{-\eta}x^{-2}\le c_1^{-\eta}x^{3\eta-2}$. Since $\frac16\le W\le \tilde C$ by Lemma~\ref{L:APRIORI}, for $\k\le\k_0$ sufficiently small we conclude
that there exist a constant $C$ and $x_0>x_\ast$ such that for all $x>x_0$
\[
\frac1C\le \lv -\left((W+\k)^2-\k(W-1)^2+4\k \R W\right)+\R^{-\e}x^{-2}\rv \le C. 
\]
This immediately yields~\eqref{E:ALPHABOUND}.  
%$\lv \alpha(x)\rv\le C$ for all $x\ge x_0$. 
Note that the leading order behaviour in the rectangular brackets in the very last line of~\eqref{E:BETASHORTER} is of the form
$\gamma (W-\R) + (2-2\gamma) W+O(\k)$. We use the decay bound~\eqref{E:RUB2}, namely $\R(x)\lesssim x^{-1+C\k}$, and we see that the
right-most side of~\eqref{E:BETASHORTER} becomes positive with $\gamma=\frac32$, with $x$ sufficiently large and $\k$ sufficiently small.
The identity~\eqref{E:RPRIMEGAMMA} yields~\eqref{E:BETABOUND}.
%
%We note that the absolute value of the numerator in~\eqref{E:BETADEF} is bounded by $C|R| + C|R^{-\e}x^{-2}|\le C x^{-\gamma}+C x^{3\e-2} \le C x^{-\gamma}$, for $0<\k\le\k_0$ sufficiently small. Here we recall $\gamma=1-C\k$ and the definition~\eqref{E:ETADEF} of $\e$. Therefore~\eqref{E:BETABOUND} follows.
\end{proof}

\begin{lemma}\label{L:ROUGHASYMPTOTICS}
Let $x_\ast\in[\xmin,\xmax]$ and let $(W,\R)$ be  the unique RLP-type solution defined globally to the right for all $\k\in(0,\k_0]$.
Then for any $\k\in(0,\k_0]$ there exists a constant $\RY_2>0$ such that
\begin{align}
\lim_{x\to\infty}W(x) & = 1, \label{E:WROUGH}\\
\lim_{x\to\infty}\DS(x)x^2 &= \RY_2 . \label{E:DROUGH}
\end{align}
Claim~\eqref{E:DROUGH} equivalently reads
\begin{align}\label{E:RPRECISE}
\lim_{x\to\infty}\R(x)x^{2\frac{1-\k}{1+\k}} = \RY_2^{\frac{1-\k}{1+\k}}.
\end{align}
\end{lemma}

%%%%%%%%%%%%%%%%%%%%%%%
%%%%%%%%%%%%%%%%%%%%%%%

\begin{proof}
Note that by~\eqref{E:WODE1}
\begin{align}
\left((W-1)x\right)' &= W'x + W-1 = -2W + \frac{2x^2(1+\k) W (W + \k) (\R-W)}{B} \notag \\
& = -W \frac{2B - 2x^2(1+\k)(W+\k)(\R-W)}{B}.
\end{align}
From this and the first line of~\eqref{E:DECAYD0} we conclude
\begin{align}\label{E:WDIDENTITY}
\left((W-1)x\right)' =- W \left(\frac{\DS' x}{\DS} + 2\right).
\end{align}
Let now
\begin{align}\label{E:VDEF}
\vartheta(x): = (1-W)x.
\end{align}
Equation~\eqref{E:WDIDENTITY} and Lemma~\ref{L:DECAY1} now give
\begin{align}\label{E:VARTHETAEQN}
\vartheta'(x) =  -\k \frac{W(x)\alpha(x)}{x} \vartheta(x) + W(x)\beta(x).
\end{align}
We now use the integrating factor to solve for $\vartheta$. For any $x_0>x_\ast$ we have
\begin{align}\label{E:VFORMULA}
\vartheta(x) e^{\k\int_{x_0}^x \frac{W(s)\alpha(s)}{s}\,ds} = \vartheta(x_0) + \int_{x_0}^x W(s)\beta(s) e^{\k\int_{x_0}^s \frac{W(s')\alpha(s')}{s'}\,ds'} \,ds.
\end{align}
By Lemma~\ref{L:DECAY1} and the bound $\frac16\le W\le \tilde C$ from Lemma~\ref{L:APRIORI} we have $|\alpha W|\le C$ and therefore
\begin{align}\label{E:IFBOUND}
\left(\frac{x}{x_0}\right)^{-C\k}\le e^{\k\int_{x_0}^x \frac{W(s)\alpha(s)}{s}\,ds} \le  \left(\frac{x}{x_0}\right)^{C\k}. 
%\le \left(\frac{x}{x_0}\right)^{\frac12},
\end{align} 

We use this bound in~\eqref{E:VFORMULA} and together with~\eqref{E:BETABOUND} we conclude that 
\begin{align}
\left(\frac{x}{x_0}\right)^{-C\k} |\vartheta(x)| & \le |\vartheta(x_0)| %|v_0(x)|
 + C \int_{x_0}^x s^{-\frac32+C\k}\,ds \le C.
 % \notag \\
%& \le C + \frac{C}{1-\gamma+C\k} x^{1-\gamma+C\k}.
 \label{E:MIDARG}
\end{align}
We now recall the definition~\eqref{E:VDEF} of $\vartheta$ and conclude from the above bound that 
\begin{align}\label{E:WROUGH1}
|1-W| \le  C x^{-1+C\k},  \text{ for all }\ \ \k\in(0,\k_0].
\end{align}
where we recall~\eqref{E:RUB}. The claim~\eqref{E:WROUGH} follows.

%Fix now an $0<\zeta\ll1$. If we replace the upper bound from~\eqref{E:IFBOUND} by
%\be\label{E:IFBOUND2}
%e^{\k\int_{x_0}^x \frac{W(s)\alpha(s)}{s}\,ds} \le  \left(\frac{x}{x_0}\right)^{\zeta},
%\ee
%which clearly holds for all $\k\le\k_0$ with $\k_0$ sufficiently small, the argument analogous to~\eqref{E:WROUGH1} 
%will give the bound
%\begin{align}\label{E:WROUGH2}
%|1-W| \le \frac C{\zeta} x^{-1+C\zeta}, \ \ \text{ for all } \ \ \k\in(0,\k_0].
%\end{align}

From Lemma~\ref{L:DECAY1} we conclude
\begin{align}
 \log(\DS(x)x^2) & =  \log (\DS(x_0)x_0^2) + \int_{x_0}^x \left(\k \frac{\alpha(s)}{s} (W(s)-1) + \frac{\beta(s)}{s} \right)\,ds,
\end{align}
and thus
\begin{align}\label{E:DFORMULA}
\DS(x) x^2 = \DS(x_0) x_0^2 \,  e^{\int_{x_0}^x \left(\k \frac{\alpha(s)}{s} (W(s)-1) + \frac{\beta(s)}{s} \right)\,ds}. 
\end{align}
We now use~\eqref{E:WROUGH1} and Lemma~\ref{L:DECAY1} to conclude that 
$\lv \k \frac{\alpha(s)}{s} (W(s)-1) + \frac{\beta(s)}{s}\rv \lesssim s^{-2+O(\k)}$ and is therefore
integrable. Letting $x\to\infty$ we conclude~\eqref{E:DROUGH}. That~\eqref{E:RPRECISE} is equivalent to~\eqref{E:DROUGH}
follows from~\eqref{E:PIDEF}. 
% bound the right-hand side:
%\begin{align}
%\log(D(x)x^2) & \le C +  \int_{x_0}^x \left( \k s^{-2+O(\k)} + s^{-\gamma-1}\right) \, ds \\
%& \le C + 
%\end{align}
\end{proof}

%%%%%%%%%%%%%%%%%%%%%%%%%%%%%%%%%
%%%%%%%%%%%%%%%%%%%%%%%%%%%%%%%%%

\begin{remark}
Stronger versions of the claims~\eqref{E:WROUGH} and~\eqref{E:DROUGH} are contained in~\eqref{E:WROUGH1} and~\eqref{E:DFORMULA}
respectively. Formula~\eqref{E:DFORMULA} can be used to give a quantitative decay bound for 
\[
\lv \DS(x)x^2-\RY_2\rv , \ \ \text{ as }  \ x\to\infty.
\]
\end{remark}

%\begin{remark} 
%Claim~\eqref{E:DROUGH} equivalently reads
%\begin{align}\label{E:RPRECISE}
%\lim_{x\to\infty}R(x)x^{2\frac{1-\k}{1+\k}} = c^{\frac{1-\k}{1+\k}}.
%\end{align}
%\end{remark}

%%%%%%%%%%%%%%%%%%%%%%%%%%%%%%%%%
%%%%%%%%%%%%%%%%%%%%%%%%%%%%%%%%%

In the following lemma we establish the strict upper bound $W<1$ for the RLP-type solutions and provide 
the crucial $\k$-independent bounds at spacelike hypersurface $x=+\infty$, which will play an important role in 
the construction of the maximal self-similar extension of the RLP spacetime.

\begin{lemma}[$W$ stays below $1$]\label{L:WSTAYSBELOWONE}
Let $x_\ast\in[\xmin,\xmax]$ and let $(W,\R)$ be  the unique RLP-type solution defined globally to the right for all $\k\in(0,\k_0]$.
Then there exist a sufficiently small $0<\k_0\ll1$ and  constants $c_2>c_1>0$ such that 
\begin{align}\label{E:WUPPERBOUND2}
W(x)<1, \ \ x\in[x_\ast,\infty), \ \ \text{ for all } \ \k\in(0,\k_0], 
\end{align}
and
%Moreover, there exist constants $c_2>c_1>0$ such that 
\begin{align}\label{E:PIWX2LB}
1+ c_2\ge\RY_2 \ge 1+c_1 \ \ \text{ for all } \ \k\in(0,\k_0],
\end{align}
where $\RY_2 =  \lim_{x\to\infty} \DS(x)W(x)x^2 = \lim_{x\to\infty} \DS(x)x^2$.
\end{lemma}

\begin{proof}
We first observe that by~\eqref{E:RODE1}--\eqref{E:WODE1}
\begin{align}\label{E:PIWX2}
(\DS W x^2)' = \DS x (1-W) = \Pi \vartheta,
\end{align}
where we have used~\eqref{E:VDEF} in the second equality.
By way of contradiction, assume that there exists an $\bar x\in(x_\ast,\infty)$ such that $W(\bar x)=1$ and $W(x)<1$ for all $x\in(x_\ast,\bar x)$. Integrating~\eqref{E:PIWX2} over $[x_\ast,\bar x]$
we conclude that 
\begin{align}\label{E:C10}
\DS(\bar x) \bar x^2 &=\DS(\bar x) W(\bar x) \bar x^2  \notag \\
& = \DS(x_\ast) W(x_\ast)\xs^2   +\int_{x_\ast}^{\bar x}  \DS(s) \vartheta(s)\,ds  \notag \\
& = W_0^{2+\e}\xs^2  +\int_{x_\ast}^{\bar x}  \DS(s) \vartheta(s)\,ds,
\end{align}
where we recall~\eqref{E:PIDEF}. On the other hand, clearly $W'(\bar x)\ge0$ and therefore, by~\eqref{E:WODE1},
at $\bar x$ we have 
\begin{align}
0\le W'(\bar x) &= -\frac{2}{\bar x} + \frac{2\bar x(1+\k)^2(\R-1)}{\R^{-\e}-\left((1+\k)^2+4\k \R\right)\bar x^2} \notag\\
& = \frac{-2\left(\R^{-\e}-\left((1+\k)^2+4\k \R\right)\bar x^2\right)+2\bar x^2(1+\k)^2(\R-1)}{\bar x\left(\R^{-\e}-\left((1+\k)^2+4\k \R\right)\bar x^2\right)} \notag \\
& = \frac{-2 \R^{-\e} +2\R\bar x^2 \left(4\k + (1+\k)^2\right)}{\bar x\left(\R^{-\e}-\left((1+\k)^2+4\k \R\right)\bar x^2\right)} \notag \\
& = - 2 \R^{-\e} \frac{1- \DS \bar x^2\left(4\k + (1+\k)^2\right) }{\bar x\left(\R^{-\e}-\left((1+\k)^2+4\k \R\right)\bar x^2\right)}. 
\end{align}
Since the denominator is strictly negative, it follows that 
\begin{align}\label{E:C20}
\DS(\bar x) \bar x^2 \le \frac1{4\k + (1+\k)^2}. 
\end{align}
Combining~\eqref{E:C10} and~\eqref{E:C20} we conclude that
\begin{align}\label{E:C3}
\frac1{4\k + (1+\k)^2} \ge W_0^{2+\e}\xs^2  + \int_{x_\ast}^{\bar x}  \DS(s)  \vartheta(s)\,ds.
\end{align}

%For $0<\zeta\ll1$ fixed as in the proof of Lemma~\ref{L:ROUGHASYMPTOTICS}, by \eqref{E:BETABOUND} %~\eqref{E:WODE1} 
%note the rough bound
%\be\label{E:BETAROUGHBOUND}
%|\beta(s)|\le C x^{-\gamma} = C x^{-1+C\k} \le C x^{-1+C\zeta},
%\ee
%where $\beta$ is defined in~\eqref{E:BETADEF}.
Our goal is to provide a lower bound for $\vartheta(x)$ which is uniform-in-$\k$. Using~\eqref{E:VFORMULA},~\eqref{E:IFBOUND}, and~\eqref{E:BETABOUND}, we conclude that for any $\bar x>x>x_0>x_\ast$
\begin{align}
%\left(\frac{x}{x_0}\right)^{C\zeta} |\vartheta(x)| \ge 
\left(\frac{x}{x_0}\right)^{C\k} |\vartheta(x)| & \ge \vartheta(x_0) - C \int_{x_0}^x s^{-\frac32+C\k}\,ds. 
% \notag \\
%& \ge \vartheta(x_0) - C\int_{x_0}^x s^{-\frac32+C\k}\,ds .  \notag
%& = \vartheta(x_0) - \frac{C}{\zeta}\left(\frac1{x_0^{C\zeta}} -\frac{1}{x^{C\k+C\zeta}}\right)
\end{align}
Therefore
\begin{align}\label{E:VLB1}
|\vartheta(x)| \ge \vartheta(x_0)\left(\frac{x}{x_0}\right)^{-C\k} - C\int_{x_0}^x s^{-\frac32+C\k}\,ds \left(\frac{x}{x_0}\right)^{-C\k}.
\end{align}
Now choose $x_0=x_\ast+\delta$ with $\delta$ as in Lemma~\ref{L:INITIAL} so that, by Lemma~\ref{L:INITIAL} we can ascertain the {\em uniform-in-$\k$} bound
\begin{align}
\vartheta(x_0)\ge \frac12, \ \ \text{ for all} \ \k\in(0,\k_0].
\end{align}
%Here, we note that $\vartheta(x_\ast) = (1-W_0)x_\ast \ge (\frac12-O(\k))(2-O(\k)) = 1- O(\k)$.
This bound together with the lower bound~\eqref{E:VLB1} yields
\begin{align}\label{E:VLB2}
|\vartheta(x)| \ge \frac12\left(\frac{x}{x_0}\right)^{-C\k} - C\int_{x_0}^x s^{-\frac32+C\k}\,ds
\left(\frac{x}{x_0}\right)^{-C\k}, \ \ \text{ for all} \ \k\in(0,\k_0].
\end{align}
The right-hand side of~\eqref{E:VLB2} is a continuous function in $x$ which converges to $\frac12$ as $x\to x_0^-$. Therefore, there exists an $\k$-independent real number $\tilde x>x_0$ such that 
\begin{align}\label{E:VLBKEY}
\vartheta(x) \ge \frac14, \ \ \text{ for all } \ x\in[x_0,\tilde x] \ \ \text{ and all } \ \k\in(0,\k_0].
\end{align}
Observe that the bound~\eqref{E:VLBKEY} also ascertains that $\tilde x<\bar x$, see~\eqref{E:VDEF}.

It is easily checked from~\eqref{E:RODE1}--\eqref{E:WODE1} that $(\DS W x^3)'=\DS x^2\ge0$. Since $|W|\le \tilde C$ by Lemma~\ref{L:APRIORI} it then follows that 
\begin{align}\label{E:PILBKEY}
\Pi(x) \ge \frac{\Pi(x_\ast)W(x_\ast)x_\ast^3}{\tilde C x^3}  = \frac{W_0^{2+\e}x_\ast^3}{\tilde C x^3}  \ge C x^{-3}, \ \ x>x_\ast,
\end{align}
where, as usual, the constant $C$ is independent of $\k$.

We now use the bounds~\eqref{E:VLBKEY} and~\eqref{E:PILBKEY} in~\eqref{E:C3} to conclude
\begin{align}
\frac1{4\k + (1+\k)^2} & \ge W_0^{2+\e}\xs^2  + \int_{x_\ast}^{\bar x}  \DS(s)  \vartheta(s)\,ds 
 \ge W_0^{2+\e}\xs^2  + \int_{x_0}^{\tilde x}  \DS(s)  \vartheta(s)\,ds \notag\\
& \ge W_0^{2+\e}\xs^2  + C \int_{x_0}^{\tilde x} s^{-3} \,ds 
 \ge W_0^{2+\e}\xs^2   + C \left(\frac{1}{x_0^2}-\frac1{\tilde x^2}\right).
\end{align}
Observe that we have used the non-negativity of $\vartheta$ on $[x_\ast,\bar x]$ in the second bound above.

By the $\k$-asymptotic behaviour of $W_0$ from Lemma~\ref{R0W0} 
we have
\[
W_0^{2+\e}\xs^2  = 1+O(\k),
\]
and clearly $\frac1{4\k + (1+\k)^2} = 1+O(\k)$. Since however the term 
$C \left(\frac{1}{x_0^2}-\frac1{\tilde x^2}\right)$
is bounded from below by some  $\k$-independent positive  real number $\tilde\delta>0$, we obtain contradiction by choosing $\k_0$ sufficiently small. 
This completes the proof of~\eqref{E:WUPPERBOUND2}.

By Lemma~\ref{L:ROUGHASYMPTOTICS} we may integrate~\eqref{E:PIWX2} over the whole interval 
$[x_\ast,\infty)$ to conclude
\begin{align}
\lim_{x\to\infty}\DS(x) W(x) x^2 & = \DS(x_\ast) W(x_\ast)\xs^2   + \int_{x_\ast}^\infty \DS(s) \vartheta(s)\,ds \notag \\ 
& \ge \DS(x_\ast) W(x_\ast)\xs^2    + \int_{x_0}^{\tilde x} \DS(s) \vartheta(s)\,ds \notag \\
& \ge W_0^{2+\e}\xs^2   + C \left(\frac{1}{x_0^2}-\frac1{\tilde x^2}\right) \notag \\
&  \ge 1 + O(\k) + \tilde \delta > 1,
\end{align}
for $\k\le\k_0$ sufficiently small. This proves the lower bound stated in~\eqref{E:PIWX2LB}. 

To prove the upper bound on $\DS(x)W(x) x^2$, we rewrite~\eqref{E:PIWX2} in the form $\log(\DS W x^2)' = \frac{1-W}{xW}$ and integrate
to obtain the identity
\begin{align}\label{E:LOGID}
\log(\DS(x) W(x) x^2) = \log (\DS(x_0) W(x_0) x_0^2) + \int_{x_0}^x \frac{1-W(s)}{sW(s)}\,ds, \ \ x_\ast\le x_0 < x.
\end{align}
 Using~\eqref{E:WROUGH1} and~\eqref{E:WLOWERBOUND} we conclude from~\eqref{E:LOGID} that
\begin{align}\label{E:LOGID2}
\log(\DS(x) W(x) x^2) \le  \log (\DS(x_0) W(x_0) x_0^2) + 6C \int_{x_0}^x s^{-2+C\k}\,ds,
\end{align}
where the constant $C$ does not depend on $\k$. We let $x_0=x_\ast$, so that for sufficiently small $\k\le\k_0$ the $\k$-dependent quantity
$\DS(x_\ast) W(x_\ast)  $ is bounded uniformly-in-$\k$. On the other hand 
\[
 6C \int_{x_\ast}^x s^{-2+C\k}\,ds \le  6C \int_{x_\ast}^\infty s^{-2+C\k}\,ds
 =  \frac{6C}{1-C\k} x_\ast^{-1+C\k}, 
\]
which, since $\xs\in[\xmin,\xmax]$, is also bounded uniformly-in-$\k$. This completes the proof of the lemma.
\end{proof}

In the following proposition, we establish the sharp asymptotic behaviour of the variable $W$ as $x\to\infty$ - this will play 
an important role in constructing the unique extension of the solution in Section~\ref{S:MAE}.

\begin{proposition}[Precise asymptotic behaviour for $W$]\label{P:PRECISEW}
Let $x_\ast\in[\xmin,\xmax]$ and let $(W,\R)$ be  the unique RLP-type solution defined globally to the right for all $\k\in(0,\k_0]$.
After choosing a possibly smaller $\k_0>0$, there exists a constant $\tilde c$ such that for any $\k\in(0,\k_0]$ there exists a constant $\vY_1<0$ such that 
\begin{align}\label{E:ONEMINUSWEXPANSION}
1-W(x) = -\vY_1 x^{-\frac{1-\k}{1+\k}} + o_{x\to\infty}(x^{\frac{1-\k}{1+\k}}), \ \ x\ge x_\ast,
\end{align}
where $\vY_1< - \tilde c<0$ for all $\k\in(0,\k_0]$.
\end{proposition}

\begin{proof}
We may now derive more precise asymptotics for the functions $\alpha(x)$ and $\beta(x)$ from~\eqref{E:ALPHADEF}--\eqref{E:BETADEF}. Using~\eqref{E:RPRECISE} and~\eqref{E:WROUGH}
we see that 
\begin{align}
\lim_{x\to\infty}\frac{\beta(x)}{\R(x)}= \frac{2}{(1+\k)^2} \left(1+\k(6+\k) - \lim_{x\to\infty}\left(\R^{-\frac{1+\k}{1-\k}}x^{-2}\right)\right) = 
 \frac{2}{(1+\k)^2} \left(1+\k(6+\k) -\frac1c\right),
\end{align}
where $1+c_2>c= \lim_{x\to\infty}\left(\R^{\frac{1+\k}{1-\k}}x^{2}\right)>1+c_1$ for all $\k\in(0,\k_0]$ by Lemma~\ref{L:WSTAYSBELOWONE}. We conclude that there exists a constant $\tilde c>0$ such that
for all $\k\in(0,\k_0]$,
%It follows from~\eqref{E:RPRECISE} that 
\be\label{E:BETAPRECISE}
\frac1{\tilde c}>\lim_{x\to\infty} \frac{\beta(x)}{x^{-2\frac{1-\k}{1+\k}}} >\tilde c.
\ee 
Using~\eqref{E:WROUGH} and~\eqref{E:RPRECISE}, it is easy to see that
\begin{align}\label{E:ALPHAPRECISE}
\alpha(x) \asymp_{x\to\infty}  -\frac2{1+\k}.
\end{align}

Recall now $\vartheta$ from~\eqref{E:VDEF}. Letting $W = 1 + (W-1) = 1- \frac{\vartheta}{x}$, we may rewrite~\eqref{E:VARTHETAEQN} in the form
\begin{align}\label{E:VARTHETAEQN2}
\vartheta'(x) =  \frac{2\k}{1+\k} \frac{\vartheta(x)}{x} + \frac{\k \vartheta(x)^2}{x^2} \alpha(x) -\k \tilde\alpha(x)\frac{\vartheta(x)}{x} + W(x)\beta(x), \ \ \alpha =   -\frac2{1+\k} + \tilde \alpha,
\end{align}
with $\lim_{x\to\infty}\tilde\alpha(x)=0$.
This yields the identity
\begin{align}
\left(\vartheta x^{-\frac{2\k}{1+\k}}\right) ' = x^{-\frac{2\k}{1+\k}} \left(\frac{\k \vartheta(x)^2}{x^2} \alpha(x) -\k \tilde\alpha(x)\frac{\vartheta(x)}{x} + W(x)\beta(x)\right).
\end{align}
By the bound~\eqref{E:WROUGH1} 
%with $0<\zeta\ll1$ sufficiently small 
and by~\eqref{E:BETAPRECISE}--\eqref{E:ALPHAPRECISE}, it follows that the right-hand side of of the above identity is integrable on $[x_\ast,\infty)$
and therefore~\eqref{E:ONEMINUSWEXPANSION} holds for some $v_1\le0$ (note that $v_1$ cannot be positive by Lemma~\ref{L:WSTAYSBELOWONE}). 
To prove the $\k$-uniform upper bound on $v_1$, we must first estimate the rate of decay of $\tilde\alpha(x)$ as $x\to\infty$.
From~\eqref{E:VARTHETAEQN2} and~\eqref{E:ALPHADEF}, we directly check 
\begin{align}
\tilde\alpha(x) = - \frac{2\left((1-\k)(1-W)^2+4\k \R W-\R^{-\e}x^{-2}\right)}{(1+\k)\left(\R^{-\e}x^{-2}-\left((W+\k)^2-\k(W-1)^2+4\k \R W\right)\right)}.
\end{align}
Therefore, for sufficiently large $x\gg1$ we have the upper bound
\begin{align}\label{E:ALPHATILDEPRECISE}
|\tilde\alpha(x)| \le K x^{-2+O(\k)},
\end{align}
for some $K>0$ independent of $\k$.
Here we have used~\eqref{E:WROUGH1}, the bound~\eqref{E:RPRECISE}, and Lemma~\ref{L:ROUGHASYMPTOTICS}.
Keeping in mind~\eqref{E:VDEF}, the bound $1-W>0$ (from Lemma~\ref{L:WSTAYSBELOWONE}) and~\eqref{E:BETAPRECISE} we conclude that for any $x\ge x_\ast$ sufficiently large we have
\begin{align}
& -\vY_1=\lim_{x\to\infty}\left((1-W(x)) x^{\frac{1-\k}{1+\k}}\right) = \lim_{x\to\infty}\left(\vartheta(x)x^{-\frac{2\k}{1+\k}}\right) \notag\\
& = \left(1-W(x_0)\right)x_0^{\frac{1-\k}{1+\k}} + \int_{x_0}^\infty \k x^{-\frac{2\k}{1+\k}} \left(\frac{ \vartheta(x)^2}{x^2} \alpha(x) - \tilde\alpha(x)\frac{\vartheta(x)}{x}\right) + x^{-\frac{2\k}{1+\k}} W(x)\beta(x) \, dx \notag\\
& \ge \k  \int_{x_0}^\infty x^{-\frac{2\k}{1+\k}} \left(\frac{ \vartheta(x)^2}{x^2} \alpha(x) - \tilde\alpha(x)\frac{\vartheta(x)}{x}\right) \, dx + C_1 \int_{x_0}^\infty x^{-\frac{2\k}{1+\k}} x^{-2\frac{1-\k}{1+\k}}\,dx \notag \\
& \ge C_1 \int_{x_0}^\infty  x^{-\frac{2}{1+\k}}\,dx - C_2 \k \int_{x_0}^\infty x^{-3+C\k}\,dx,
\end{align}
for some constants $C_1,C_2>0$ and independent of $\k$. Note that we have used $W\ge\frac16$ in the third line and~\eqref{E:ALPHATILDEPRECISE} in the last. This implies the claim for $\k$ sufficiently small.
\end{proof}

%%%%%%%%%%%%%%%%%%%%%%%%%%%%%%%%%%%
%%%%%%%%%%%%%%%%%%%%%%%%%%%%%%%%%%%

\section{Maximal analytic extension} \label{S:MAE}

%%%%%%%%%%%%%%%%%%%%%%%%%%%%%%%%%%%%
%%%%%%%%%%%%%%%%%%%%%%%%%%%%%%%%%%%%

%%%%%%%%%%%%%%%%%%%%%%%%%%%%%%%%%%%%
%%%%%%%%%%%%%%%%%%%%%%%%%%%%%%%%%%%%

The main results of this section are the description of the local and  the global extension of the flow across the 
coordinate singularity at $y=\infty$, see Theorems~\ref{T:LE} and~\ref{T:GE} respectively. For a detailed overview, we refer to Section~\ref{SS:MAEINTRO}.

\subsection{Adapted comoving coordinates}

%%%%%%%%%%%%%%%%%%%%%%%%%%%%%%%%%%%%
%%%%%%%%%%%%%%%%%%%%%%%%%%%%%%%%%%%%

%To set it up, we first
%revert to the original comoving coordinates.

%We observe that the system~\eqref{E:LDEQN1}--\eqref{E:LREQN1} is analogous to its Newtonian version~\eqref{E:DEQN1.0}--\eqref{E:TILDEREQN1.0}, and formally reduces to it as $\k\to0$.
As explained in Section~\ref{SS:MAEINTRO}
the metric $g$ given by~\eqref{E:METRIC}
becomes singular as $y\to\infty$. 
%(or equivalently $x\to\infty$). 
%As we will next show, this is merely a coordinate singularity, and the spacetime extends smoothly (in fact analytically in a suitable choice of coordinates) across the surface
%$\{(\tau\equiv 0,R)\, \big|\, R>0\}$, see Theorems~\ref{T:LE} and~\ref{T:GE}.
To show that this is merely a coordinate singularity, we introduced the adapted comoving chart~\eqref{E:TILDETAUDEF}, and the new self-similar variable $Y$ given by~\eqref{E:YDEF}--\eqref{E:YLITTLEY}.
%%%%%%%%%%%%%%%%%%%%%%%%%%
%%%%%%%%%%%%%%%%%%%%%%%%%%
%%%%%%%%%%%%%%%%%%%%%%%%%%
%%%%%%%%%%%%%%%%%%%%%%%%%%
%%%%%%%%%%%%%%%%%%%%%%%%%%
%%%%%%%%%%%%%%%%%%%%%%%%%%
A simple manipulation of~\eqref{E:TILDETAUDEF} gives the relation
\begin{align}\label{E:TAUTILDETAU}
\tau = \tau(\ttau,R) = Y^\e \ttau, \ \ \tau<0, \ R>0.
\end{align}

Recall the new unknowns $\chi(Y)$, $\RY(Y)$, and $w(Y)$ from~\eqref{E:YVARDEF}.
In the new chart the spacetime metric~\eqref{E:METRIC} by~\eqref{E:SS1}--\eqref{E:SS6},~\eqref{E:TAUTILDETAU}, and~\eqref{E:YVARDEF} takes the form
\[
g = - e^{2\tilde\mu(Y)}\,d\tilde\tau^2  -  \frac{4\sqrt\k}{1+\k} Y e^{2\tilde\mu(Y)}\,d\tt\,dR + \left(e^{2\tilde\l(Y)}-\frac{4\k}{(1+\k)^2} Y^2 e^{2\tilde\mu(Y)} \right)\,dR^2 + 
R^2 \chi(Y)^2\, \gamma,
\]
where we recall~\eqref{E:MUTILDE}--\eqref{E:LAMBDATILDE}.
Note that by~\eqref{E:GTT} and~\eqref{E:MUTILDE}, we have
\begin{align}\label{E:NEWMU}
e^{2\tilde\mu(Y)} = \frac1{(1-\k)^2} \left(\RY Y^{-2}\right)^{-\e}, \ \ Y>0.
\end{align}
It is clear that in the limit $y\to\infty$, i.e. $Y\to0^+$ the metric coefficient
$ e^{2\tilde\mu(Y)}$ approaches a positive constant due to~\eqref{E:MUAS} and~\eqref{E:YLITTLEY}. 
%Moreover $e^{2\tilde\l(Y)}=e^{2\l(y)}$ for all $y>0$.

%%%%%%%%%%%%%%%%%%%%%%%%%
%%%%%%%%%%%%%%%%%%%%%%%%%

\begin{lemma}\label{L:ADAPTED}
Let the triple $(\tr,\d,\w)$ be a smooth solution to~\eqref{E:DSSEQN}--\eqref{E:WSSEQN}. Then the new unknowns $(\chi,\RY,w)$ defined in~\eqref{E:YVARDEF} solve the system~\eqref{E:LDEQN2}--\eqref{E:CDEF}. 
\end{lemma}

%%%%%%%%%%%%%%%%%%%%%%%%%

\begin{proof}
The Lagrangian system~\eqref{E:DSSEQN}--\eqref{E:WSSEQN} now takes the form
\begin{align}
\d' &= - \frac{\wl+\k}{1+\k} \ \frac{2(1-\k)\rl^2}{y} \frac{ \d (\w +\k)(\Rl-\w )}{\mathcal B} ,\label{E:LDEQN1}\\
\wl' &= \frac{\wl+\k}{1+\k}\left(\frac{1-3\wl}{y} +\frac{2(1+\k)\rl^2}{y} \frac{\wl(\wl+\k)(\Rl-\wl)}{\mathcal B}\right) ,  \label{E:LWEQN1} \\
\rl' & = \frac{\rl }{y} \  \frac{\wl+\k}{1+\k}, \label{E:LREQN1}
\end{align}
where
\begin{align}
\mathcal B = \d^{-\e} - \rl^2 \left[(\wl + \k)^2-\k(\wl-1)^2 + 4\k \wl \Rl \right] \label{E:LBDEF},
\end{align}
and $'$ refers to differentiation with respect to $y$. Using~\eqref{E:YVARDEF} it is then straightforward to check that~\eqref{E:LDEQN2}--\eqref{E:CDEF} hold.
\end{proof}

%%%%%%%%%%%%%%%%%%%%%%%%%
%%%%%%%%%%%%%%%%%%%%%%%%%

\begin{remark}\label{R:TAYLORR}
%The system~\eqref{E:LDEQN2}--\eqref{E:CHIEQN} is  analogous to the Newtonian system~\eqref{E:DEQN1}--\eqref{E:CHIEQN1}. 
%It is possible that it would be easier to remo\vYe $\frac1{1+\k}$ from the right-hand side of~\eqref{E:LDEQN2}--\eqref{E:CHIEQN} by a further change of variables, but, I leave it as is for the moment.
From \eqref{E:LREQN1} %~\eqref{E:LREQN} 
and the leading order behaviour of $\wl(y)=\vY(Y)$ at $y=\infty$ it is easy to see that 
\begin{align}\label{E:RLASYMPTOTICS}
& \rl(y) = a y + o_{y\to+\infty}(y) = a \frac1{Y^{1+\e}} + o_{Y\to 0^+}(\frac1{Y^{1+\e}}), \ Y>0,\\
& \lim_{Y\to 0^+}\chi(Y) = a,
\end{align}
for some constant $a>0$. Here the constant $a$ corresponds to the labelling gauge freedom and we set without loss of generality
\be\label{E:ANORM}
a=1.
\ee
We note that the unknown $\RY$ corresponds  to the modified density $\Rl$ defined in \eqref{E:YVARDEF}. %~\eqref{E:TILDERWDEF}. 
By Lemma~\ref{L:ROUGHASYMPTOTICS} and Proposition~\ref{P:PRECISEW}, 
and the asymptotic behaviour~\eqref{E:RLASYMPTOTICS} with $a=1$, 
we have the leading order asymptotic behaviour
\begin{align}
%d(Y) & = c_\k Y^2 + o_{Y\to0^+}\left(Y^2\right)\\
\RY(Y) & = \RY_2 Y^2 + o_{Y\to0^+}\left(Y^2 \right), \label{E:SIGMAASYMPTOTICS}\\
\vY(Y) & = 1 + \vY_1 Y + o_{Y\to0^+}\left(Y\right), \label{E:SMALLVASYMPTOTICS}
\end{align}
where $1+c_2>\RY_2>1+c_1$ and 
\be\label{E:VONEUNIFORM}
\vY_1<-\tilde c<0
\ee 
by Lemma~\ref{L:WSTAYSBELOWONE} and Proposition~\ref{P:PRECISEW}.
We note that 
\be\label{E:ASIMP}
\lim_{Y\to0^+} \frac{\chi(Y)^2}{Y^{2+2\e}} \vY(Y)\RY^{1+\e}(Y) = \lim_{x\to\infty} \DS(x)W(x)x^2 = \RY_2 >1
\ee
by Lemma~\ref{L:WSTAYSBELOWONE}. 
\end{remark}

%%%%%%%%%%%%%%%%%%%%%%%%%
%%%%%%%%%%%%%%%%%%%%%%%%%
%%%%%%%%%%%%%%%%%%%%%%%%%

\subsection{Local extension}

%%%%%%%%%%%%%%%%%%%%%%%%%
%%%%%%%%%%%%%%%%%%%%%%%%%
%%%%%%%%%%%%%%%%%%%%%%%%%

%%%%%%%%%%%%%%%%%%%%%%%%%%%%%%
%%%%%%%%%%%%%%%%%%%%%%%%%%%%%%

%%%%%%%%%%%%%%%%%%%%%%%%%%%%%%

To prove the existence of a solution to~\eqref{E:LDEQN2}--\eqref{E:CHIEQN} with suitable boundary conditions (see~\eqref{E:SIGMACOND}--\eqref{E:CHICOND}), we formally Taylor-expand the unknowns
$\RY,\vY,\chi$ around $Y=0$ and prove the convergence of the series.
Assume that the following expansions hold
\begin{align}
\RY(Y) & = \RY_2 Y^2 + \sum_{N=3}^\infty \RY_N Y^N, \label{E:SIGMAEXPLOC}\\
\vY(Y)& = 1 + \sum_{N=1}^\infty \vY_N Y^N ,\label{E:VEXPLOC}\\
\chi(Y) & = \chi_0 + \sum_{N=1}^\infty \chi_N Y^N, \label{E:CHIEXPLOC}
\end{align}
where by~\eqref{E:ANORM} $\chi_0=1$.
From this and~\eqref{E:CDEF} we have the formal expansion
\begin{align}
\mathcal C = \sum_{N=0}^\infty \mathcal C_N Y^N, 
\end{align}
where for any $N\ge0$
\begin{align}
\mathcal C_N = \left((\RY Y^{-2})^{-\e}\right)_{N-2} - \sum_{k+\ell=N} (\chi^2)_k\left[(1-\k)(\vY^2)_\ell + 4\k \vY_\ell +(\k^2-\k)\delta_{0\ell}+4\k (\vY \RY)_\ell\right]. \label{E:CEXPLOC}
\end{align}
Here we employ the convention that $(f)_j=0$ if $j<0$. We may single out the leading order contribution on the right-hand side of~\eqref{E:CEXPLOC}:
\begin{align}
\mathcal C_N & = - 4\k \chi_0^2 \RY_N - 2(1+\k)\chi_0^2 \vY_N - 2(1+\k)^2\chi_0 \chi_N + \mathcal C_{N,1}, \label{E:CNLEADINGORDER}
\end{align}
where
\begin{align}
\mathcal C_{N,1}:=&\left((\RY Y^{-2})^{-\e}\right)_{N-2} - \sum_{\substack{k+\ell=N \\ k,\ell\le N-1}} (\chi^2)_k\left[(1-\k)(\vY^2)_\ell + 4\k \vY_\ell +(\k^2-\k)\delta_{0\ell}+4\k (\vY \RY)_\ell\right] \notag\\
& - \chi_0^2 \left[(1-\k)\sum_{\substack{k+\ell =N \\ k,\ell\le N-1}}\vY_k \vY_\ell + 4\k \sum_{\substack{k+\ell =N \\ \ell \le N-1}}\vY_k \RY_\ell\right].
\end{align}

%%%%%%%%%%%%%%%%%%%%%%%%%%%%%%%
%%%%%%%%%%%%%%%%%%%%%%%%%%%%%%%

\begin{lemma}
For any $N\ge1$, $N\in\mathbb N$, the Taylor coefficients $(\RY_N,\vY_N,\chi_N)$ satisfy the following recursive relations
\begin{align}
 (N-2)\chi_0^2(1+\k)^2\RY_N  &=  \mathcal U_N,
\label{E:SIGMARECURSIVE}\\
 (1-\k)(1+\k)^2 (N-1)\chi_0^2 \vY_N - 2(1+\k)\left((1+\k)^2+4\k\right)\chi_0^2 \RY_N  & = \mathcal V_N,   \label{E:VYRECURSIVE} \\
\chi_0 \vY_N + (1-\k) N \chi_N  & =- \sum_{\substack{k+\ell=N \\ k,\ell\le N-1}}\chi_k \vY_\ell , \label{E:CHIYRECURSIVE}
\end{align}
where
\begin{align}
\mathcal U_N & = \sum_{\substack{k+j=N \\ k\le N-1}} k \RY_k \mathcal C_j  -2\sum_{\substack{k+\ell+m+n=N \\ \ell\le N-1}} (\chi^2)_k \RY_\ell (\RY_m-\vY_m)\left((\vY^2)_n+2\k \vY_n + \k^2\delta_{0n}\right), \ \ N\ge3, 
\label{E:MUNDEF}\\
 \mathcal V_N & = (1-\k)\sum_{\substack{k+j=N \\ k\le N-1}} k \vY_k \mathcal C_j  \notag\\
 & \ \ \ \ -  \sum_{\substack{k+\ell=N \\ k,\ell \le N-1}} \left(3(\vY^2)_k-(1-3\k)\vY_k\right) \mathcal C_\ell + 2(1+\k)\mathcal C_{N,1} 
+ 3(1+\k)^2 \chi_0^2 \sum_{\substack{k+\ell=N \\ k,\ell \le N-1}} \vY_k \vY_\ell \notag\\
& \ \ \ \ +2(1+\k)\sum_{\substack{k+m+n=N \\ k,m,n\le N-1}} (\chi^2)_k  (\RY_m-\vY_m)\left((\vY^3)_n+2\k (\vY^2)_n + \k^2\vY_n\right) \notag\\
&- 2(1+\k)\chi_0^2 \left(\sum_{\substack{k+\ell+n=N\\ k,\ell,n\le N-1}} \vY_k\vY_\ell\vY_n + 2\k \sum_{\substack{k+\ell=N \\ k,\ell \le N-1}} \vY_k\vY_\ell \right)
 - 2(1+\k)^3 \sum_{\substack{k+\ell=N \\ k,\ell \le N-1}} \chi_k \chi_\ell. \label{E:MVNDEF}
\end{align}
\end{lemma}

%%%%%%%%%%%%%%%%%%%%%%%%%%%%%%%

\begin{proof}
{\em Proof of~\eqref{E:SIGMARECURSIVE}.}
We multiply~\eqref{E:LDEQN2} by $Y\mathcal C$ and plug in the formal 
expansions~\eqref{E:SIGMAEXPLOC}--\eqref{E:CEXPLOC}. We formally obtain that the $N$-th Taylor coefficient in the expansion of $\RY' Y \mathcal C$ is given by
\begin{align}\label{E:LHSSIGMAEXP}
\sum_{k+j=N} k \RY_k \mathcal C_j.
\end{align} 
On the other hand, the $N$-th Taylor coefficient in the expansion of $2\chi^2\RY(\vY+\k)^2(\RY-\vY)=2\chi^2\RY(\vY^2+2\k\vY+\k^2)(\RY-\vY)$ is easily checked to be
\begin{align}\label{E:RHSSIGMAEXP}
2\sum_{k+\ell+m+n=N} (\chi^2)_k \RY_\ell (\RY_m-\vY_m)\left((\vY^2)_n+2\k \vY_n + \k^2\delta_{0n}\right).
\end{align}
We now extract the leading order terms in~\eqref{E:LHSSIGMAEXP} and~\eqref{E:RHSSIGMAEXP} - these are the factors containing either $\RY_N$ or $\vY_N$. 
We see that 
\begin{align}
\sum_{k+j=N} k \RY_k \mathcal C_j & = N \RY_N \mathcal C_0 + \sum_{\substack{k+j=N \\ k\le N-1}} k \RY_k \mathcal C_j  \notag \\
&= - N\chi_0^2(1+\k)^2 \RY_N + \sum_{\substack{k+j=N \\ k\le N-1}} k \RY_k \mathcal C_j, \label{E:LHSEXP2}
\end{align}
where we have used $\mathcal C_0 = - \chi_0^2 (1+\k)^2$. Note that the term $\mathcal C_N$ does not contribute a copy of $\RY_N$ on the right-hand side above
since $k=0$ when $j=N$.
Similarly, the expression~\eqref{E:RHSSIGMAEXP} can be split into
\begin{align}\label{E:RHSEXP2}
-2 \chi_0^2 (1+\k)^2\RY_N+ 2\sum_{\substack{k+\ell+m+n=N \\ \ell\le N-1}} (\chi^2)_k \RY_\ell (\RY_m-\vY_m)\left((\vY^2)_n+2\k \vY_n + \k^2\delta_{0n}\right).
\end{align}
Note that we have repeatedly used the assumption $\RY_0=\RY_1 = 0$. Equating~\eqref{E:LHSEXP2} and~\eqref{E:RHSEXP2} we
obtain the recursive relation~\eqref{E:SIGMARECURSIVE}.

{\em Proof of~\eqref{E:VYRECURSIVE}.}
We multiply~\eqref{E:LWEQN2} by $(1-\k)Y\mathcal C$ and expand the two sides of the equation by analogy to the above.
We obtain formally
\begin{align}\label{E:VPRIMEEXP}
(1-\k)\vY' Y\mathcal C = (1-\k)\sum_{N=0}^\infty\left(\sum_{k+j=N} k \vY_k \mathcal C_j\right) Y^N.
\end{align}
Upon extracting the leading order term in the $N$-th Taylor coefficient on the right-hand side of~\eqref{E:VPRIMEEXP}, by analogy to~\eqref{E:LHSEXP2}
we obtain
\begin{align}
(1-\k)\sum_{k+j=N} k \vY_k \mathcal C_j= 
- (1-\k)N\chi_0^2(1+\k)^2 \vY_N + (1-\k)\sum_{\substack{k+j=N \\ k\le N-1}} k \vY_k \mathcal C_j. \label{E:1}
\end{align} 
We now turn our attention to the right-hand side.
Observe that formally the $N$-th Taylor coefficient in the expansion of $-(\vY+\k)(1-3\vY)\mathcal C  = \left(3\vY^2-(1-3\k)\vY-\k\right)\mathcal C $ equals
\begin{align}\label{E:VYEXP1}
\sum_{k+\ell=N} \left(3(\vY^2)_k\mathcal C_\ell-(1-3\k)\vY_k\mathcal C_\ell-\k\delta_{0k}\mathcal C_\ell\right).
\end{align}
To single out the leading order contribution we use~\eqref{E:CNLEADINGORDER} and thus~\eqref{E:VYEXP1} can be rewritten in the form
\begin{align}
& 2(1+\k) \mathcal C_N + \left(3(\vY^2)_N-(1-3\k)\vY_N\right)\mathcal C_0 
+ \sum_{\substack{k+\ell=N \\ k,\ell \le N-1}} \left(3(\vY^2)_k-(1-3\k)\vY_k\right) \mathcal C_\ell \notag\\
& = -8\k(1+\k)\chi_0^2 \RY_N - 3(1+\k)^2(3+\k)\chi_0^2 \vY_N - 4(1+\k)^3\chi_0 \chi_N \notag\\
& \ \ \ \ +  \sum_{\substack{k+\ell=N \\ k,\ell \le N-1}} \left(3(\vY^2)_k-(1-3\k)\vY_k\right) \mathcal C_\ell + 2(1+\k)\mathcal C_{N,1} 
- 3(1+\k)^2 \chi_0^2 \sum_{\substack{k+\ell=N \\ k,\ell \le N-1}} \vY_k \vY_\ell . \label{E:2}
\end{align}
By analogy to~\eqref{E:RHSSIGMAEXP} and~\eqref{E:RHSEXP2} we also check that the $N$-the Taylor coefficient in the expansion of 
$-2(1+\k)\chi^2\vY(\vY+\k)^2(\RY-\vY)$ formally corresponds to
\begin{align}
-2(1+\k)\sum_{k+m+n=N} (\chi^2)_k  (\RY_m-\vY_m)\left((\vY^3)_n+2\k (\vY^2)_n + \k^2\vY_n\right). 
\end{align}
After isolating the leading order coefficients, the above expression can be rewritten in the form
\begin{align}
&- 2(1+\k)^3 \chi_0^2 \RY_N+4(1+\k)^2(2+\k)\chi_0^2 \vY_N + 4(1+\k)^3 \chi_0\chi_N \notag \\
& -2(1+\k)\sum_{\substack{k+m+n=N \\ k,m,n\le N-1}} (\chi^2)_k  (\RY_m-\vY_m)\left((\vY^3)_n+2\k (\vY^2)_n + \k^2\vY_n\right) \notag\\
&+ 2(1+\k)\chi_0^2 \left(\sum_{\substack{k+\ell+n=N\\ k,\ell,n\le N-1}} \vY_k\vY_\ell\vY_n + 2\k \sum_{\substack{k+\ell=N \\ k,\ell \le N-1}} \vY_k\vY_\ell \right)
 + 2(1+\k)^3 \sum_{\substack{k+\ell=N \\ k,\ell \le N-1}} \chi_k \chi_\ell. \label{E:3}
\end{align}
Claim~\eqref{E:VYRECURSIVE} now follows from~\eqref{E:1}~\eqref{E:2}, and~\eqref{E:3}.

{\em Proof of~\eqref{E:CHIYRECURSIVE}.} The claim follows by substituting the formal expansions for $\chi$ and $\vY$ into~\eqref{E:CHIEQN},  comparing the coefficients and using $\vY_0=1$.
\end{proof}

\begin{theorem}[Local extension]\label{T:LE}
There exists and $0<\k_0\ll1$ and a $Y_0<0$ such that for  any $\k\in(0,\k_0]$ 
there exists a unique analytic-in-$Y$ solution to~\eqref{E:LDEQN2}--\eqref{E:CHIEQN} on the interval $(-|Y_0|,|Y_0|)$
such that 
\begin{align}
\lim_{Y\to 0} \frac{\RY(Y)}{Y^2} & = \RY_2, \label{E:SIGMACOND}\\
\vY(0) & = 1, \ \ \vY'(0)=\vY_1,  \label{E:VCOND}\\
\chi(0) & =1, \label{E:CHICOND}\\
\mathcal C(Y) & <0, \ \ Y\in (-|Y_0|,|Y_0|),
\end{align}
where $\RY_2, \vY_1$ are constants given by Lemma~\ref{L:WSTAYSBELOWONE} and Proposition~\ref{P:PRECISEW}.
Moreover,
\begin{enumerate}
\item[{\em (a)}] 
There exists a $\delta>0$ such that 
\begin{align}\label{E:PARTA}
\RY(Y_0)>\delta,  \ \ \chi(Y_0)>\delta, \ \  \text{ and } \ \ \vY(Y_0)<\frac1\delta \ \ \text{ for all } \ \ \k\in(0,\k_0].
\end{align}
\item[{\em (b)}]
There exists a constant $c_0>0$ such that
\be\label{E:GAMMASETUP0}
\chi^2 \vY(\RY Y^{-2})^{1+\e}>1+c_0, \ \ \text{ for all } \ Y\in [Y_0,0] \ \ \text{ and all } \ \k\in(0,\k_0].  
\ee
\item[{\em (c)}]
There exists a constant $\tilde\delta>0$ such that 
\begin{align}\label{E:SIGMAOVERV}
\frac1{100}>\frac{\RY}{\vY}\Big|_{Y=Y_0} > \tilde\delta, \ \ \k\in(0,\k_0].
\end{align}
\item[{\em (d)}]
There exists a constant $c_{\vY}>0$ such that 
\begin{align}
\vY(Y_0)>1+ c_{\vY},  \ \ \text{ for all } \ \ \k\in(0,\k_0].
\end{align}
\end{enumerate}
\end{theorem}

\begin{proof}
The proof of existence of a local real-analytic solution is analogous to the proof of Theorem~\ref{T:SONICLWP}.
We observe that the occurrence of the factors $(N-2)$ and $(N-1)$ in~\eqref{E:SIGMARECURSIVE} and~\eqref{E:VYRECURSIVE}
respectively reflects the fact that $\RY_2$ and $\vY_1$ must be prescribed in order to consistently solve the ensuing recursive relations
between the higher-order coefficients. We leave out the details, as they are similar to the ideas of the proof of Theorem~\ref{T:SONICLWP}.

\noindent
{\em Proof of part {\em (a)}.}
Estimates~\eqref{E:PARTA} are a trivial consequence of the leading order Taylor expansions from Remark~\ref{R:TAYLORR} the uniform-in-$\k$ positivity of $\RY_2$ and negativity of $\vY_1$.

\noindent
{\em Proof of part {\em (b)}.}
Since there exists a constant $\tilde c_0$ such that $\lim_{Y\to0}\left(\chi^2 \vY (\RY Y^{-2})^{1+\e}\right)>1+\tilde c_0$ for all $\k\in(0,\k_0]$,
by Taylor expansion around $Y=0$ we can ascertain that there is a constant $c_0>0$ and an interval $[-|Y_0|,|Y_0|]$, $Y_0<0$, such that~\eqref{E:GAMMASETUP0} holds. 
%\be\label{E:GAMMASETUP}
% \chi^2 v (\RY Y^{-2})^{1+\e}>1+c_0, \ \ \text{ for all } \ Y\in [Y_0,0] \ \ \text{ and all } \ \k\in(0,\k_0].  
% \ee

\noindent
{\em Proof of part {\em (c)}.} Observe that $\frac{\RY}{\vY}\Big|_{Y=0} =0$. Since $\RY$ is locally bounded from below and $\vY$ locally increases to the left, the bound follows from
continuity.

\noindent
{\em Proof of part {\em (d)}.}
By the Taylor expansion~\eqref{E:VEXPLOC} around $Y=0$ and the uniform-in-$\k$ bound~\eqref{E:VONEUNIFORM} we can guarantee that $\vY'<0$ on $[Y_0,0)$ and that $c_{\vY}$ can be chosen independently of $\k$. 
\end{proof}

\begin{remark}
Note that $Y_0$ is $\k$-independent, which plays an important role in our proof of the existence of outgoing null-geodesics from the scaling origin ${\mathcal O}$.
The constants $\RY_2$ and $\vY_1$  act as initial conditions for the system~\eqref{E:LDEQN2}--\eqref{E:CHIEQN} to extend the solution to the left.   
\end{remark}

It is a priori possible that the solution constructed by Theorem~\ref{T:LE} does not coincide with the RLP solution constructed for $x\in(0,\infty)$.
The overlapping region $(0,|Y_0|)$ expressed in the $x$-coordinate is given by $(X_0,\infty)$, where
$X_0 = \left(\tilde r\left(|Y_0|^{-(1+\e)}\right)\right)^{-\frac1{1+\e}}$.
 As it is less clear how to apply the standard uniqueness theorem 
for the problem phrased in the $Y$-coordinate, we shall revert to the coordinate
\be\label{E:XXCOORD}
X = x^{-\frac{1}{1+\e}}, \ \ x = X^{-1-\e},
\ee
and reduce the question of uniqueness to the standard ODE theory.

%%%%%%%%%%%%%%%%%%%%%%%%%%%%%%
%%%%%%%%%%%%%%%%%%%%%%%%%%%%%%
 
\begin{proposition}[Uniqueness of the RLP-extension]
In the region $Y>0$, the solution constructed in Theorem~\ref{T:LE} coincides with the RLP-type solution emanating from the sonic point $\bar x_\ast$ constructed in Section~\ref{S:FARFIELD}.
\end{proposition}

%%%%%%%%%%%%%%%%%%%%%%%%%%%%%%

\begin{proof}
Introduce the unknowns
\begin{align}\label{E:RWXNORM}
\R(x)= X^2 \bar \R(X), \ \ W(x) = 1+ X \bar W(X).
\end{align}
A direct calculation gives
\begin{align}
B(x) & = X^{-2-2\e} \left(-(1+\k)^2 -2(1+\k)X\bar W+ X^2K_1(X;\bar \R,\bar W)\right), \\
 K_1(X;\R,W) &:= \bar \R^{-\e}  - (1-\k) %X 
 \bar W^2 - 4\k \bar \R - 4\k X\bar \R\bar W.
\end{align}
Using~\eqref{E:XXCOORD} and~\eqref{E:RODE}--\eqref{E:WODE}, we may further compute
\begin{align}
\frac{d}{dX}\R & = \R' \frac{dx}{dX} \notag\\
& = 2X \frac{(1+\k)\bar \R\left(1+\k+X\bar W\right)\left(X^2\bar \R-1-X\bar W\right)}{-(1+\k)^2 -2(1+\k)X\bar W+ X^2K_1(X;\bar \R,\bar W)} \notag\\
& = :2X\bar \R + X^2 K_2(X,\bar \R,\bar W), 
\end{align}
where it is easy to check that the function $K_2(X,\bar \R,\bar W)$ is Lipschitz.
An analogous calculation based on~\eqref{E:WODE} then gives
\begin{align}
\frac{d}{dX}W  &= W'  \frac{dx}{dX}  \notag\\
& = (1+\e)\left[\frac2X +3\bar W + \frac2{X(1+\k)} \frac{\left(1+X\bar W\right)\left(1+\k+X\bar W\right)(X^2\bar \R-1-X\bar W)}{1 + \frac2{1+\k} X\bar W - \frac{X^2}{(1+\k)^2}K_1}\right] \notag\\
& =:  (1+\e)\left[\frac2X +3\bar W  - \frac2X +\frac1{1+\e}\bar W + X K_3(X;\bar \R,\bar W) \right] \notag\\
 & = \bar W + (1+\e)X K_3(X;\bar \R,\bar W),
\end{align}
where it is easy to check that the function $K_3(X,\bar \R,\bar W)$ is Lipschitz.
Recalling~\eqref{E:RWXNORM}, we conclude that 
\begin{align}\label{E:BARRWSYSTEM}
\frac{d}{dX}\bar \R = K_2(X;\bar \R,\bar W), \ \ \frac{d}{dX}\bar W =(1+\e)  K_3(X;\bar \R,\bar W).
\end{align}
Therefore, the dynamical system~\eqref{E:BARRWSYSTEM} is regular at $X=0$ and must coincide with the RLP-solution emanating from the sonic point $X_\ast = \xs^{-\frac{1}{1+\e}}$. 
Since the mapping $X\to Y$ is smooth and invertible locally around $X=0$, the claimed uniqueness statement follows.
%problem is well-posed with initial data $\bar R(0)=\RY_2$ and $\bar W(0) = \vY_1<0$.
\end{proof}

%%%%%%%%%%%%%%%%%%%%%%%%%%%%%%
%%%%%%%%%%%%%%%%%%%%%%%%%%%%%%

%\ec
%%%%%%%%%%%%%%%%%%%%%%%%%%%%%%
%%%%%%%%%%%%%%%%%%%%%%%%%%%%%%

%We {\em expect} to prove the following theorem.

%%%%%%%%%%%%%%%%%%%%%%%%%%%%%%

%%%%%%%%%%%%%%%%%%%%%%%%%
%%%%%%%%%%%%%%%%%%%%%%%%%
%%%%%%%%%%%%%%%%%%%%%%%%%

\subsection{Maximal extension}

%%%%%%%%%%%%%%%%%%%%%%%%%
%%%%%%%%%%%%%%%%%%%%%%%%%
%%%%%%%%%%%%%%%%%%%%%%%%%

%%%%%%%%%%%%%%%%%%%%%%%%%%%%%
%%%%%%%%%%%%%%%%%%%%%%%%%%%%%

By analogy to Lemma~\ref{L:BALGEBRA}, we factorise
the denominator $\mathcal C$ into
\begin{align}
\mathcal C = (1-\k)\left(\mathcal \F[Y;\RY,\chi] - \chi \vY\right)\left( \mathcal H[Y;\RY,\chi] + \chi\vY\right),
\end{align}
where
\begin{align}\label{E:FDEFNEW}
\mathcal \F [Y;\RY,\chi]= \mathcal \F  & : = -\e(1 +\RY)\chi + \sqrt{\e^2(1 +\RY)^2 \chi^2+  \k \chi^2 
+ \frac{\left(\RY Y^{-2}\right)^{-\eta}Y^2}{1-\k}}, \\
\mathcal H[Y;\RY,\chi]= \mathcal H & : = \mathcal \F  + 2\e (1+\RY)\chi.
\end{align}
Clearly, for any fixed $Y$ and $\RY$ , $\mathcal \F [Y, \RY]$ is the solution of the equation $\mathcal C=0$ viewed as a quadratic equation in $\chi \vY$.
Just like in the proof of Lemma~\ref{L:BALGEBRA}, it can be checked that 
\begin{align}
\mathcal \F ' & = - \frac{4\k \chi \mathcal \F +\e (\RY Y^{-2})^{-\e-1}}{2(1-\k)\mathcal \F +4\k \chi(1+\RY)}\RY' \notag \\
& \ \ \ \ 
+ \frac{-4\k \chi'(1+\RY)\mathcal \F -2(\k^2-\k)\chi\chi'+\frac{2\e}{Y}(\RY Y^{-2})^{-\e-1}\RY+2(\RY Y^{-2})^{-\e}Y}
{2(1-\k)\mathcal \F +4\k \chi(1+\RY)} . \label{E:MATHCALFDER}
\end{align}

Our next goal is to prove a {\em global extension} result, which is shown later in Theorem~\ref{T:GE}. To that end define 
\be\label{E:BARYDEF}
\yms : = \inf\left\{Y<0\, \Big|  \text{$\exists$ a smooth solution to~\eqref{E:LDEQN2}--\eqref{E:CHIEQN} and} \ \vY(Y)>1, \  \mathcal C(Y)<0, \ \chi(Y)>0\right\}.
\ee

%%%%%%%%%%%%%%%%%%%%%%%%%%%%%
%%%%%%%%%%%%%%%%%%%%%%%%%%%%%

\begin{lemma}\label{L:PRELIM}
Let $(\chi,\vY,\RY)$ be a local-in-$Y$ solution to~\eqref{E:LDEQN2}--\eqref{E:CHIEQN}. 
\begin{enumerate}
\item[{\em (a)}]
Then for all $Y\in(\yms,0]$ the following identities hold:
\begin{align}
(\RY^{1+\e}\vY)' &= - \RY^{1+\e} \frac{(\vY+\k)(1-3\vY)}{(1-\k)Y}, \label{E:DVIDENTITY} \\
(\chi^2\RY^{1+\e}\vY)' & = \frac{\chi^2 \RY^{1+\e}}{(1-\k)Y}\left(\vY^2 +\vY(1+3\k)-\k\right). \label{E:CHIVSIGMAID}
\end{align}
\item[{\em (b)}]
For all $Y\in(\yms,0)$ we have the bounds
\begin{align}\label{E:SIGMAVUPPER}
0<\RY<\vY.
\end{align}
\end{enumerate}
\end{lemma}

%%%%%%%%%%%%%%%%%%%%%%%%%%%%%
\begin{proof}
{\em Proof of part {\em (a)}.}
Dividing~\eqref{E:LDEQN2} by $\RY$ and~\eqref{E:LWEQN2} by $\vY$ and the summing the $(1+\e)$-multiple of the first equation with the second,  we obtain~\eqref{E:DVIDENTITY}.
Using~\eqref{E:DVIDENTITY} and~\eqref{E:CHIEQN} we then obtain
\begin{align}
(\chi^2\RY^{1+\e}\vY)' &= \frac{\chi^2 \RY^{1+\e}}{(1-\k)Y} \left(2\vY(1-\vY) -(\vY+\k)(1-3\vY)\right) \notag \\
& = \frac{\chi^2 \RY^{1+\e}}{(1-\k)Y}\left(\vY^2 +\vY(1+3\k)-\k\right). \notag
\end{align}

{\em Proof of part {\em (b)}.}
The strict positivity of $\RY$ in a small open left neighbourhood of $Y=0$ follows from Theorem \ref{T:LE}.  
The global positivity then follows by integrating~\eqref{E:LDEQN2}, which can be rewritten as 
\[
\left(\log\RY\right)' =  \frac{2\chi^2}{Y} \frac{(\vY+\k)^2(\RY-\vY)}{\mathcal C}.
\]
The upper bound $\RY<\vY$ clearly holds at $Y=0$ and in its small neighbourhood due to~\eqref{E:SIGMAASYMPTOTICS}--\eqref{E:SMALLVASYMPTOTICS}. Assume now by contradiction that there exists $\yms<Y_\ast<0$ so that $Y_\ast$ is the infimum over all values of $Y\in(\yms,0]$ such that $\RY(Y)<\vY(Y)$. By continuity obviously $\RY(Y_\ast)=\vY(Y_\ast)$. 
However, by~\eqref{E:LDEQN2}--\eqref{E:LWEQN2} 
%\begin{align}
$\RY'(Y_\ast) - \vY'(Y_\ast) = 
% \frac{2\chi^2}{Y} \frac{(\vY+\k)^2(\RY-\vY)}{\mathcal C} \left(\Sigma + 2\frac{1+\k}{1-\k}\vY\right) + 
\frac{(\vY+\k)(1-3\vY)}{(1-\k)Y}>0$,
%\end{align}
since $\vY(Y_\ast)>1$ by~\eqref{E:BARYDEF}. This is a contradiction, as this means that $\RY-\vY$ decays locally going to the left of $Y_\ast$.
\end{proof}

%%%%%%%%%%%%%%%%%%%%%%%%%%%%%
%%%%%%%%%%%%%%%%%%%%%%%%%%%%%

\begin{lemma}\label{L:GAMMABOUND}
Let $(\chi,\vY,\RY)$ be a local-in-$Y$ solution to~\eqref{E:LDEQN2}--\eqref{E:CHIEQN}
and let 
\be\label{E:BIGGAMMADEF}
\Gamma_\delta(Y) : = (Y^2)^{1+\e} - (1-\delta)\chi^2 \vY\RY^{1+\e}, \ \ \delta>0.
\ee
Then there exist a $0<\delta<1$ and $0<\k_0\ll1$ sufficiently small such that for all $\k\in(0,\k_0]$ we have
\begin{align}
\Gamma_\delta(Y)< 0, \ \ Y\in(\yms,0).
\end{align}
\end{lemma}

%%%%%%%%%%%%%%%%%%%%%%%%%%%%%

\begin{proof}
%We now let 
%\[
%\Gamma(Y) : = Y^{2+2\e} - (\k^2+6\k+1)\chi^2 v \RY^{1+\e}.
%\]
%We may therefore rewrite
%\begin{align}
%Y^{2+2\e} - 4\k \chi^2 v\RY^{1+\e} - \chi^2 \RY^{1+\e}(\vY+\k)^2
%&=\Gamma(Y) +\chi^2  \RY^{1+\e}\left((\k^2+6\k+1)v - 4\k v -(\vY+\k)^2 \right) \notag \\
%& = \Gamma(Y) -\chi^2 \RY^{1+\e}(v-\k^2)(v-1). 
%\end{align}
%Therefore,~\eqref{E:VPRIME2} may be rewritten in the form
%\begin{align}
% &\vY'Y - \left[\frac{3v+\k}{1-\k}+\frac{1+\k}{1-\k} \frac{v\left(2\k \chi^2 (v-1)+2\chi^2(v+\k)^2\right)}{\mathcal C}
% -\frac{1+\k}{1-\k} \frac{2v\chi^2 \RY(v-\k^2)}{\mathcal C}\right](v-1) \notag \\
% & =  \frac{2v(1+\k)\RY^{-\e}}{(1-\k)\mathcal C} \Gamma(Y). \label{E:VYID}
%\end{align}

We observe that by~\eqref{E:CHIVSIGMAID},
\begin{align}
\Gamma_\delta'(Y) & = (2+2\e)(Y^2)^\e Y - (1-\delta)(\chi^2 \vY\RY^{1+\e})' \notag \\
& =  (2+2\e)(Y^2)^\e Y- (1-\delta) \frac{\chi^2  \RY^{1+\e}}{(1-\k)Y}\left(\vY^2 +\vY(1+3\k)-\k\right) \notag \\
& = \frac1{(1-\k)Y} \left[2(1+\k)\Gamma_\delta(Y)+ (1-\delta)\chi^2 \RY^{1+\e}\left(-\vY^2+ \vY- \k \vY+ \k \right) \right] \notag \\
& =   \frac1{(1-\k)Y} %\frac1Y
 \left[2(1+\k)\Gamma_\delta(Y)- (1-\delta)\chi^2 \RY^{1+\e}(\vY-1)(\vY+\k) \right].  \label{E:GAMMAPRIME}
\end{align}
Since $\vY>1$ by our assumptions, it follows that 
\[
 (1-\delta)\chi^2 \RY^{1+\e}(\vY-1)(\vY+\k)>0.
\] 
Since $\Gamma_\delta(Y) = (Y^2)^{1+\e} \left(1-  (1-\delta)\chi^2 \vY (\RY Y^{-2})^{1+\e}\right)$ it follows from Theorem~\ref{T:LE}, inequality~\eqref{E:GAMMASETUP0} that we can choose $\delta=\delta(c_0)>0$ such that $\Gamma_\delta<0$ on $[Y_0,0)$.
Using the standard integrating factor argument it then follows from~\eqref{E:GAMMAPRIME} that
\be\label{E:GAMMASIGN}
\Gamma_\delta(Y)<0, \ \ Y\in(\yms,0).
\ee 
 \end{proof}

\begin{lemma}\label{L:SIGMAVBOUND}
Let $(\chi,\vY,\RY)$ be a local-in-$Y$ solution to~\eqref{E:LDEQN2}--\eqref{E:CHIEQN} with the radius of analyticity $(Y_0,-Y_0)$ given by Theorem \ref{T:LE}. 
Then there exists a constant $\tilde\delta>0$ such that 
\begin{align}
\frac{\RY}{\vY}>\tilde\delta \ \ \text{ for all } \ \yms<Y\le Y_0, \ \ 0<\k\le\k_0.
\end{align}
\end{lemma}

\begin{proof}
By~\eqref{E:LDEQN2}--\eqref{E:LWEQN2} we obtain
\begin{align}
\RY'\vY - \RY\vY' & =  \frac{2(2+\e)\chi^2\RY \vY(\vY+\k)^2(\RY-\vY)}{Y\mathcal C} + \frac{\RY(\vY+\k)(1-3\vY)}{(1-\k)Y} \notag \\
& = \frac{\RY}{(1-\k)Y\mathcal C} \Big\{4\chi^2 \vY (\vY+\k)^2(\RY-\vY) \notag \\
& \ \ \ \ 
+(\vY+\k)(1-3\vY)\left(\left(\RY Y^{-2}\right)^{-\eta}Y^2 - \chi^2 \left[(\vY+ \k)^2-\k(\vY-1)^2 + 4\k \vY\RY \right] \right)  \Big\} \notag \\
& = :  \frac{\RY}{(1-\k)Y\mathcal C} \bar A, \label{E:BARA1}
\end{align}
where we used $2(2+\e) = \frac4{1-\k}$.
We now rewrite $\bar A$ to obtain 
\begin{align}
\bar A & = - 4\chi^2 \vY^2 (\vY+\k)^2 - (\vY+\k)(1-3\vY)\chi^2 (\vY+\k)^2  \notag \\
& \ \ \ \ +4\chi^2 \vY\RY (\vY+\k)^2 - 4\k \chi^2 \vY \RY (\vY+\k)(1-3\vY)\notag \\
& \ \ \ \ + (\vY+\k)(1-3\vY)\left(\RY Y^{-2}\right)^{-\eta}Y^2 + \k \chi^2 (\vY+\k)(1-3\vY)(\vY-1)^2 \notag \\
& = \chi^2 (\vY+\k)^2 \left(-\vY^2-\vY -\k + 3\k \vY\right) + 4\chi^2 \vY^2\RY (\vY+\k)(1+3\k) \notag \\
& \ \ \ \ + (\vY+\k)(1-3\vY)\left(\RY Y^{-2}\right)^{-\eta}Y^2 + \k \chi^2 (\vY+\k)(1-3\vY)(\vY-1)^2 \notag \\
& = \chi^2 (\vY+\k) \left\{4 \vY^2 \RY(1+3\k) - (\vY+\k)\left(\vY^2+\vY+\k(1-3\vY)\right)\right\} \notag \\
& \ \ \ \ + (\vY+\k)(1-3\vY)\left(\RY Y^{-2}\right)^{-\eta}Y^2 + \k \chi^2 (\vY+\k)(1-3\vY)(\vY-1)^2 \notag \\
& =  \chi^2 (\vY+\k) \left\{\vY^2 \left[4(1+3\k)\RY-\vY-\k\right] - (\vY+\k)\left(\vY+\k(1-3\vY)\right) \right\} \notag \\
& \ \ \ \ + (\vY+\k)(1-3\vY)\left(\RY Y^{-2}\right)^{-\eta}Y^2 + \k \chi^2 (\vY+\k)(1-3\vY)(\vY-1)^2 \notag \\
& <\chi^2 (\vY+\k) \vY^2\left[4(1+3\k)\RY-\vY\right], \label{E:BARA2}
\end{align}
where we have used $\vY>1$, which in turn gives 
%(Lemma~\ref{L:VBOUND})
 the bounds $ -(\vY+\k)\left(\vY+\k(1-3\vY)\right)<0$, $(\vY+\k)(1-3\vY)\left(\RY Y^{-2}\right)^{-\eta}Y^2<0$, $ \k \chi^2 (\vY+\k)(1-3\vY)(\vY-1)^2<0$. 
 Recall that $\frac{\RY}{\vY}\Big|_{Y=0}=0$ by~\eqref{E:SIGMACOND}--\eqref{E:VCOND}.
 From~\eqref{E:BARA1} and~\eqref{E:BARA2} it then follows that 
\begin{align}\label{E:SIGMAV}
\left(\frac{\RY}{\vY}\right)' <0 \ \ \text{ if } \ \frac{\RY}{\vY}<\frac1{4(1+3\k)}.
\end{align}

Therefore by Theorem~\ref{T:LE}, inequality~\eqref{E:SIGMAOVERV} and~\eqref{E:SIGMAV}, the ratio $\frac{\RY}{\vY}$ increases to the left of $Y=Y_0$ as long as $\frac{\RY}{\vY}<\frac1{4(1+3\k)}$. If the ratio ever exceeds $\frac1{4(1+3\k)}$ going to the left then it must stay larger than $\frac1{4(1+3\k)}$ as seen by a contradiction argument using~\eqref{E:BARA2}. 
Therefore the claim follows.
%\begin{align}\label{E:SIGMAVBOUND}
%\frac{\RY}{\vY}>\tilde\delta \ \ \text{ for all } \ Y\le Y_0.
%\end{align}
\end{proof}
%%%%%%%%%%%%%%%%%%%%%%%%%%%%%

%%%%%%%%%%%%%%%%%%%%%%%%%%%%%

The next lemma establishes the monotonicity of the function $\chi \vY$, and as a consequence the strict lower bound $\vY>1+c$ on $(\yms,0]$ for some $c>0$. The former is a preparatory step to prove that the flow remains supersonic in Lemma~\ref{L:SUPERSONIC}.

%%%%%%%%%%%%%%%%%%%%%%%%%%%%%

\begin{lemma}\label{L:CHIVBOUND}
Let $(\chi,\vY,\RY)$ be a local-in-$Y$ solution to~\eqref{E:LDEQN2}--\eqref{E:CHIEQN} analytic in $(Y_0,-Y_0)$ given by Theorem~\ref{T:LE}. Then there exists an $0<\k_0\ll1$ sufficiently small so that 
\be\label{E:CHIVBOUND}
(\chi\vY)'<0, \ \ \yms<Y\le Y_0.
\ee
Moreover, there exists a constant $c>0$ such that 
\be\label{E:VBOUNDREFINED}
\vY(Y)>1+c,  \ \ \ \yms<Y\le Y_0, \ \ 0<\k\le\k_0.
\ee
In particular, $\inf_{(\yms,0]}\vY(Y)>1$.
\end{lemma}

\begin{proof}
By~\eqref{E:LWEQN2}--\eqref{E:CHIEQN} we have
\begin{align}
(\chi \vY)' & = \frac{(1-\vY)\chi \vY}{(1-\k)Y} - \frac{\chi (\vY+\k)(1-3\vY)}{(1-\k)Y} - \frac{2(1+\e)\chi^3 \vY(\vY+\k)^2(\RY-\vY)}{Y\mathcal C} \notag \\
& = \frac{2\chi\vY^2-\k\chi+3\k\vY \chi}{(1-\k)Y}  - \frac{2(1+\e)\chi^3 \vY(\vY+\k)^2(\RY-\vY)}{Y\mathcal C} \notag \\
& = \frac{\chi}{(1-\k)Y\mathcal C} \left\{\left(2\vY^2-\k+3\k\vY\right)\mathcal C - 2(1+\k) \chi^2 \vY(\vY+\k)^2(\RY-\vY)\right\} \notag \\
& =: \frac{\chi}{(1-\k)Y\mathcal C}\, A. \label{E:ADEF}
\end{align}
From~\eqref{E:CDEF} we have
\begin{align}
A & = \left(2\vY^2-\k+3\k\vY\right)\left(\left(\RY Y^{-2}\right)^{-\eta}Y^2 - \chi^2 \left[(\vY+ \k)^2-\k(\vY-1)^2 + 4\k \vY\RY \right] \right)  \notag \\
& \ \ \ \ - 2(1+\k) \chi^2 \vY \RY (\vY+\k)^2 + 2(1+\k) \chi^2 \vY^2(\vY+\k)^2 \notag \\
& =  \chi^2(\vY+\k)^2 \left[-2\vY^2-3\k\vY+\k + 2(1+\k)\vY^2\right]
+ \k \chi^2  \left(2\vY^2-\k+3\k\vY\right) (\vY-1)^2 \notag \\
& \ \ \ \ - 4\k \chi^2 \vY \RY \left(2\vY^2-\k+3\k\vY\right) - 2(1+\k) \chi^2 \vY \RY (\vY + \k)^2
+  \left(2\vY^2-\k+3\k\vY\right)\left(\RY Y^{-2}\right)^{-\eta}Y^2 \notag \\
%& \ \ \ \ + \k \chi^2  \left(2\vY^2-\k+3\k\vY\right) (\vY-1)^2 \notag \\
& = \k \chi^2(\vY+\k)^2 (\vY-1)(2\vY-1) + \k \chi^2  \left(2\vY^2-\k+3\k\vY\right) (\vY-1)^2 \notag \\
& \ \ \ \ - 4\k \chi^2 \vY \RY \left(2\vY^2-\k+3\k\vY\right) - 2(1+\k) \chi^2 \vY \RY (\vY + \k)^2
+  \left(2\vY^2-\k+3\k\vY\right)\left(\RY Y^{-2}\right)^{-\eta}Y^2
\end{align}

By Lemma~\ref{L:GAMMABOUND} there exists a $\delta>0$ such that
%\begin{align}
$\left(\RY Y^{-2}\right)^{-\eta}Y^2 < (1-\delta) \chi^2 \vY \RY$.
%\end{align}
Therefore
\begin{align}
&- 4\k \chi^2 \vY \RY \left(2\vY^2-\k+3\k\vY\right) - 2(1+\k) \chi^2 \vY \RY (\vY + \k)^2
+  \left(2\vY^2-\k+3\k\vY\right)\left(\RY Y^{-2}\right)^{-\eta}Y^2 \notag \\
& <  \chi^2 \vY \RY \left(-4\k\left(2\vY^2-\k+3\k\vY\right) - 2(1+\k)  (\vY + \k)^2 +\left(2\vY^2-\k+3\k\vY\right)(1-\delta)  \right) \notag \\
& =  \chi^2 \vY \RY  \left\{ (-2\delta-10\k)\vY^2 + \left(-12\k^2-4\k(1+\k)+3\k(1-\delta)\right) \vY
+ 4\k^2-2\k^2(1+\k)-\k(1-\delta)\right\} \notag \\
& < -2\delta \chi^2 \vY^3 \RY,
\end{align}
for $\k\le\k_0$ sufficiently small.
Therefore
\begin{align}
A & < \k \chi^2(\vY+\k)^2 (\vY-1)(2\vY-1) + \k \chi^2  \left(2\vY^2-\k+3\k\vY\right) (\vY-1)^2 
-2\delta \chi^2 \vY^3 \RY \notag \\
& = \k \chi^2 (\vY-1)\left(4\vY^3 + (-3+7\k)\vY^2 + (2\k^2-6\k)\vY+ \k-\k^2\right)   -2\delta \chi^2 \vY^3 \RY  \notag \\
& < 4\k \chi^2 \vY^4 -2\delta \chi^2 \vY^3 \RY  \notag \\
& = 2\chi^2 \vY^3\left(-\delta\RY + 2\k \vY\right),  \label{E:ABOUND}
\end{align}
where we have used the bounds $\vY-1<\vY$ and $(-3+7\k)\vY^2 + (2\k^2-6\k)\vY+ \k-\k^2<0$, where the latter follows from the 
assumption $\vY>1$ on $(\yms,0]$. 
%(Lemma~\ref{L:VBOUND}). 
Plugging~\eqref{E:ABOUND} into~\eqref{E:ADEF} we obtain the estimate
\begin{align}
(\chi \vY)' &< \frac{2\chi^3 \vY^3 }{(1-\k)Y\mathcal C}\left(-\delta\RY + 2\k \vY\right) 
 =  \frac{2\chi^3 \vY^4 }{(1-\k)Y\mathcal C}\left(-\delta\frac{\RY}{\vY} + 2\k \right)  \notag\\
& <  \frac{2\chi^3 \vY^4 }{(1-\k)Y\mathcal C}\left(-\delta\tilde\delta + 2\k \right) 
 < -\frac{\delta\tilde\delta\chi^3 \vY^4 }{(1-\k)Y\mathcal C}<0, \label{E:CHIVBOUND0}
\end{align}
for $0<\k\le\k_0$ sufficiently small. Here we have crucially used Lemma~\ref{L:SIGMAVBOUND}.

To prove~\eqref{E:VBOUNDREFINED}
we note that by~\eqref{E:CHIVBOUND0}
\begin{align}
\chi(Y)\vY(Y) > \chi(Y_0)\vY(Y_0), \ \ \yms<Y<Y_0
\end{align}
and therefore
\begin{align}
\vY(Y) >  \frac{\chi(Y_0)}{\chi(Y)}\vY(Y_0)\ge \vY(Y_0)>1+ c_{\vY},
\end{align}
where we have used part (d) of Theorem~\ref{T:LE} in the last inequality.
We have also used the bound $\chi(Y)\le \chi(Y_0)$ for $Y\le Y_0$ which follows from $\chi'>0$ on $(\yms, Y_0)$, which in turn follows from the bound $\vY>1$ and~\eqref{E:CHIEQN}. 
\end{proof}

%%%%%%%%%%%%%%%%%%%%%%%%%%%%%

\begin{lemma}\label{L:SUPERSONIC}
Let $(\chi,\vY,\RY)$ be a local-in-$Y$ solution to~\eqref{E:LDEQN2}--\eqref{E:CHIEQN} with the radius of analyticity $(Y_0,-Y_0)$ given by Theorem \ref{T:LE}. 
Then 
\[
\limsup_{Y\to(\yms)^-}\mathcal \, \mathcal C(Y)<0,
\]
in other words  - the flow remains supersonic.
\end{lemma}

\begin{proof}
Using~\eqref{E:MATHCALFDER} and~\eqref{E:CHIVBOUND}, and the bounds 
\[
- \frac{4\k \chi \mathcal \F +\e (\RY Y^{-2})^{-\e-1}}{2(1-\k)\mathcal \F +4\k \chi(1+\RY)}\RY'>0, 
\ \ 
%-4\k \chi'(1+\RY)\mathcal F >0
-2(\k^2-\k)\chi\chi'>0,
\]
we get
\begin{align}
\mathcal \F ' - (\chi \vY)'  > \frac{\delta\tilde\delta\chi^3 \vY^4 }{(1-\k)Y\mathcal C} 
+  \frac{
%-2(\k^2-\k)\chi\chi'
-4\k \chi'(1+\RY)\mathcal \F  
+(2\e+2)(\RY Y^{-2})^{-\e} Y}
{2(1-\k)\mathcal \F +4\k \chi(1+\RY)}. \label{E:CONTRY0}
\end{align}

Since $(\RY Y^{-2})^{-\e}Y^2 = (1-\k)\mathcal \F ^2 + 4\k \mathcal \F (1+\RY)\chi - \k(1-\k) \chi^2$, it follows that 
\[
(\RY Y^{-2})^{-\e}Y^2 <  (1-\k)\mathcal \F ^2+ 4\k \mathcal \F (1+\RY)\chi < \mathcal \F \left(2(1-\k)\mathcal \F + 4\k (1+\RY)\chi\right).
\]
Therefore
\begin{align}
\lv  \frac{(2\e+2)(\RY Y^{-2})^{-\e} Y}
{2(1-\k)\mathcal \F +4\k \chi(1+\RY)} \rv \le \frac{C\mathcal \F }{|Y|}. \label{E:CONTRY1}
\end{align}
From~\eqref{E:CHIEQN} and the bound $\vY>1$ we have the rough bound $|\chi'|\le C\frac{\vY\chi}{|Y|}$. Therefore
\begin{align}
\lv\frac{-4\k \chi'(1+\RY)\mathcal \F }{2(1-\k)\mathcal \F +4\k \chi(1+\RY)} \rv 
\le \frac{C\k \vY\chi(1+\RY)\mathcal \F }{|Y|(2(1-\k)\mathcal \F +4\k \chi(1+\RY))}
\le\frac{C \vY \mathcal \F }{|Y|}.\label{E:CONTRY2}
%\frac{C\k \chi^2 \vY}{4\k |Y|\chi} \le \frac{C\chi \vY}{|Y|}.\label{E:CONTRY2}
\end{align}
Using~\eqref{E:CONTRY1}--\eqref{E:CONTRY2} in~\eqref{E:CONTRY0} we obtain
\begin{align}
\mathcal \F ' - (\chi \vY)'  > \frac{\delta\tilde\delta\chi^3 \vY^4 }{(1-\k)Y\mathcal C} 
- C\frac{\mathcal \F  (1+ \vY)}{|Y|} = \chi \vY^2\left(\frac{\delta\tilde\delta\chi^2\vY^2 }{(1-\k)Y\mathcal C}- C\frac{\frac{\mathcal \F }{\chi \vY}(\frac1{\vY}+1)}{|Y|}\right).
\label{E:CONTRY4}
\end{align}
Assume now that $\lim_{Y\to(\yms)^-} (\mathcal \F  - \chi \vY)=0$. In that case $\lim_{Y\to (\yms)^-}\mathcal C = 0$ and it is clear from~\eqref{E:CONTRY4}
that $\mathcal \F ' - (\chi \vY)' $ is strictly positive in some right neighbourhood of $\yms$. Here we use the uniform positivity of $\chi \vY$ on $(\yms,0]$, which follows from Lemma~\ref{L:CHIVBOUND}. A contradiction.
\end{proof}

%%%%%%%%%%%%%%%%%%%%%%%%%%%%%

We have shown that the flow remains supersonic to the left of $Y=0$ and therefore the only obstruction to the global existence of the solution is the finite-time blow-up of the unknowns. We shall show that this is precisely the case.
The intuition is that right-hand side of~\eqref{E:LWEQN2} will be, in a suitable sense, dominated by the first term on the right-hand side, which will lead to the blow-up of $\vY$ through a Riccati-type argument.

%\begin{theorem}[Maximal extension]\label{T:GE}
%There exists an $0<\k_0\ll1$ sufficiently small such that for any $\k\in(0,\k_0]$  there exists a $\yms<0$ such that the unique solution to the initial value problem~\eqref{E:LDEQN2}--\eqref{E:CHIEQN} exists on the interval $(\yms,0]$,
%and
%\begin{align}
%\lim_{Y\to (\yms)^-}\vY_\k(Y) = \lim_{Y\to (\yms)^-}\RY_\k(Y) & = \infty, \\
%\chi_\k(Y)& >0, \ \ Y\in(\yms,0], \\
%\lim_{Y\to (\yms)^-}\chi_\k(Y) & = 0.
%\end{align}
%\end{theorem}

{\em Proof of Theorem~\ref{T:GE}.}
Let $\alpha>0$ be a positive constant to be specified later.
We rewrite~\eqref{E:LWEQN2} in the form
\begin{align}\label{E:VALPHA}
\vY' =  -\frac{(\vY+\k)(1-(3-2\alpha)\vY)}{(1-\k)Y} +\frac{2(\vY+\k)}{(1-\k)Y\mathcal C}\left(\alpha \vY \mathcal C-(1+\k)\chi^2\vY(\vY+\k)(\RY-\vY)\right).
\end{align}
We focus on the term 
\[
E: = \alpha \vY \mathcal C-(1+\k)\chi^2\vY(\vY+\k)(\RY-\vY).
\]
By~\eqref{E:CDEF} we have
\begin{align}
E & = \alpha \vY \RY^{-\e}(Y^2)^{1+\e}
-\alpha \vY(1-\k)\chi^2\vY^2 -4\k \alpha \vY^2 \chi^2  \notag \\
& \ \ \ \ + \alpha \vY \chi^2(\k-\k^2) - 4\k \alpha \chi^2 \vY^2 \RY \notag \\
& -(1+\k)\chi^2 \vY \RY (\vY+\k) +(1+\k)\chi^2 \vY^2 (\vY+\k).
\end{align}
We let 
\be\label{E:ALPHADEF2}
\alpha = \frac{1+\bar\delta}{1-\k}
\ee
where $\bar\delta>0$ is a constant to be specified later. After regrouping terms in $E$ above we obtain
\begin{align}
E & = \chi^2 \vY^2\left(-(1+\bar\delta) \vY - 2\e (1+\bar\delta)+ (1+\k)(\vY+\k)+\frac{\k (1+\bar\delta)}{\vY}\right) \notag \\
& \ \ \ \ + \frac{1+\bar\delta}{1-\k} \vY \RY^{-\e}(Y^2)^{1+\e}- 2\e(1+\bar\delta) \chi^2 \vY^2 \RY-(1+\k)\chi^2 \vY \RY (\vY+\k)  \notag \\
& < \chi^2 \vY^2\left(- (\bar\delta-\k) \vY - 2\e (1+\bar\delta) + \k + \k^2+ \frac{\k (1+\bar\delta)}{\vY} \right) \notag\\
&  \ \ \ \ + \chi^2 \vY^2 \RY\left( \frac{1+\bar\delta}{1-\k}(1-\delta) - 2\e(1+\bar\delta)-(1+\k) \right),\label{E:TOWARDSR}
\end{align}
where we have used Lemma~\ref{L:GAMMABOUND} in the last inequality. It is now clear that with the choice
\begin{align}
\bar\delta \le \frac12\delta,
\end{align}
there exists an $0<\k_0\ll1$ sufficiently small, so that both expressions on the right-most side of~\eqref{E:TOWARDSR} are strictly negative for all $0<\k\le\k_0$. Here we use $\vY>1$.

We conclude therefore from~\eqref{E:VALPHA} and~\eqref{E:ALPHADEF2} that
\begin{align}
\vY' \le -\frac{(\vY+\k)(1-\frac{1-3\k-2\bar\delta}{1-\k}\vY)}{(1-\k)Y}  = -\frac{1-3\k-2\bar\delta}{1-\k} \ \frac{(\vY+\k)(\frac{1-\k}{1-3\k-2\bar\delta} - \vY)}{(1-\k)Y}. \label{E:VPRIMER}
\end{align}
Now choose $\bar\delta$ sufficiently small (but independent of $\k$) so that $\vY(Y)>\frac{1-\k}{1-3\k-2\bar\delta}$ for all $Y\in(\yms,Y_0)$, which is possible due to~\eqref{E:VBOUNDREFINED}. Let now
\[
K: = \frac{1-3\k-2\bar\delta}{1-\k}. 
\]
We conclude from~\eqref{E:VPRIMER} that
\begin{align}\label{E:RICC}
\vY' \le - C \frac{(\vY+\k)(K\vY-1)}{|Y|} \le - a \frac{\vY(K\vY-1)}{|Y|} \ , \ \ \yms < Y < Y_0,
\end{align}
for some universal constant $a>0$.
Inequality~\eqref{E:RICC} is a Riccati-type differential inequality and leads to finite $Y$ blow-up of $\vY$. We provide here the standard argument for the sake of completeness.
Upon multiplying $\vY$ by $K$ and then redefining $a>0$, we may assume without loss of generality that $K=1$. Divide~\eqref{E:RICC} (with $K=1$) by
 $\vY(\vY-1)$ and express both sides as an exact derivative to conclude
 \begin{align}
 \log\left(\left(1-\frac1\vY\right)|Y|^{-a}\right)' \le 0, \ \ Y\le Y_0.
 \end{align}
 Upon integration this yields the bound
 \begin{align}\label{E:VBLOWUP}
\vY \ge \frac1{1 -  C{|Y|^a}}, \ \ Y<Y_0.
 \end{align}
 and therefore $\vY$ necessarily blows up as $Y\to \tilde Y^+$ for some $-\infty<\tilde Y<Y_0<0$.
 By Lemma~\ref{L:SIGMAVBOUND} this also implies that $d$ blows up as $Y\to\tilde Y^+$.
 
 We next want to show that $\lim_{Y\to\tilde Y^+}\chi(Y)=0$ and therefore $\tilde Y = \yms$. We note that for any $Y\in(\tilde Y,0]$ we necessarily have $\chi(Y)>0$, which follows from~\eqref{E:CHIEQN}. Since $\chi'>0$ by~\eqref{E:CHIEQN} and the bound $\vY>1$ it follows that $\chi$ decreases to the left of $Y=0$ and the limit
 \begin{align}
 \chi_\ast : = \lim_{Y\to\tilde Y^+}\chi(Y)\ge0
 \end{align}
 exists. Assume by the way of contradiction that $\chi_\ast>0$. We look more closely at the leading order behaviour of the right-hand side of~\eqref{E:LWEQN2} on approach to the blow-up point $\tilde Y$. Since 
 \be\label{E:SIGMAV2WAYS}
 0<\tilde\delta < \frac{\RY}{\vY} <1
 \ee
 by Lemmas~\ref{L:PRELIM} and~\ref{L:SIGMAVBOUND}, and the assumption $\chi_\ast>0$
 it is easily seen that 
 \begin{align}\label{E:EFF}
 \vY' = \frac{3\vY^2}{(1-\k)Y} + \frac{2(1+\k)\vY^3(\RY-\vY)}{Y\left((1-\k)\vY^2+4\k\vY\RY\right)} + \mathcal R,
 \end{align}
 where $\mathcal R$ has the property 
 \[
 \lim_{Y\to \tilde Y^+}\frac{\mathcal R(Y)}{ \frac{2(1+\k)\vY^3(\RY-\vY)}{Y\left((1-\k)\vY^2+4\k\vY\RY\right)} } = 0.
 \]
 From~\eqref{E:SIGMAV2WAYS} and~\eqref{E:EFF} it is now clear that there exist constants $0<\kappa_1<\kappa_2$ such that 
 \begin{align}
\frac{\kappa_1}{Y} < \frac{\vY'}{\vY^2} < \frac{\kappa_2}{Y}, \ \ Y\in(\tilde Y, \tilde Y+\beta),
 \end{align}
 for some constant $\beta>0$. Integrating the above differential inequalities over an interval $[Y_0,Y]\subset (\tilde Y, \tilde Y+\beta)$ and letting $Y_0\to\tilde Y$, we conclude
 \begin{align}\label{E:NEWVBOUNDS}
 \frac{1}{\kappa_2 |\tilde Y-Y| + O(|\tilde Y-Y|^2)} \ge \vY(Y) \ge  \frac{1}{\kappa_1 |\tilde Y-Y| + O(|\tilde Y-Y|^2)} 
 \end{align}
 in a possibly smaller open right neighbourhood of $\tilde Y$. By~\eqref{E:CHIEQN} we have $(\log\chi)'=\frac{1-\vY}{Y}$, which together with~\eqref{E:NEWVBOUNDS} shows that there exist some positive constants $0<\tilde\kappa_1<\tilde\kappa_2$ such that
 \begin{align}\label{E:CHIEFF}
\tilde\kappa_1 < \left(\log \chi\right)'|\tilde Y-Y| < \tilde\kappa_2
 \end{align}
  in a small open right neighbourhood of $\tilde Y$.  Integrating~\eqref{E:CHIEFF} we conclude that 
  \begin{align}
  \lim_{Y\to\tilde Y^+}\chi(Y) = 0,
  \end{align}
  which contradicts the assumption $\chi_\ast>0$. It follows that in particular $\tilde Y = \yms$.
% 
% Finally, estimate~\eqref{E:VBLOWUP} implies that $\vY-1\ge \frac{C|Y|^a}{1-C|Y|^a}$ for $Y> \tilde Y$ and therefore by~\eqref{E:CHIEQN}
% \begin{align}
% (\log \chi)' \ge  \frac{C|Y|^{a-1}}{1-C|Y|^a}, \ \ Y>\tilde Y,
% \end{align}
% which in turn 
\prfe

%%%%%%%%%%%%%%%%%%%%%%%%%%%%

\begin{corollary}[Uniformity-in-$\k$]\label{C:UNIF}
There exist a constant $A>0$ and $0<\k_0\ll1$ such that for all $\k\in(0,\k_0]$ we have the uniform bounds
\begin{align}
\frac1A< |\yms| < A,
\end{align}
where $-\infty<\yms <0$ is the maximal existence interval to the left from Theorem~\ref{T:GE}. 
\end{corollary}

%%%%%%%%%%%%%%%%%%%%%%%%%%%%%
\begin{proof}
Both constants $C$ and $a$ in~\eqref{E:VBLOWUP} can be chosen to be $\k$-independent for $\k_0$ sufficiently small and
so we obtain a uniform upper bound on the maximal time $\yms$. Since by the construction $\yms<Y_0<0$, where $Y_0$ is the
$\k$-independent constant from Theorem~\ref{T:LE}, we conclude the proof.
\end{proof}

%%%%%%%%%%%%%%%%%%%%%%%%%%%%%
%%%%%%%%%%%%%%%%%%%%%%%%%%%%%

%%%%%%%%%%%%%%%%%%%%%%%%%%
%%%%%%%%%%%%%%%%%%%%%%%%%%

\subsection{The massive singularity}\label{SS:MASSIVE}

%%%%%%%%%%%%%%%%%%%%%%%%%%
%%%%%%%%%%%%%%%%%%%%%%%%%%
%%%%%%%%%%%%%%%%%%%%%%%%%%

\begin{definition}[The massive singularity]\label{D:MS}
The hypersurface $\MS$ defined through
\begin{align*}
\MS 
%&: = \left\{ (\tilde\tau, R) \in\partial \DRLPtilde \, \Big| \,  Y = \yms \right\}  \\
& = \left\{ (\tilde\tau, R) \, \Big| \, R = \frac{\sqrt\k}{|\yms|} \tilde\tau\right\} \setminus \{(0,0)\} 
\end{align*}
is called the \underline{massive singularity}. 
\end{definition}

In this section we compute the precise blow up rates of 
the RLP-solution at the massive singularity. To this end,
it is convenient to introduce the quantity 
\be\label{E:BARCDEF}
\bar{\mathcal C} : = -\chi^{-2}\vY^{-2} \mathcal C =  -\chi^{-2}\vY^{-2} \left(\RY Y^{-2}\right)^{-\eta}Y^2 + \left[(1 + \frac\k{\vY})^2-\k(1-\frac1{\vY})^2 + 4\k  \frac{\RY}{\vY} \right].
\ee

%%%%%%%%%%%%%%%%%%%%%%%%%%
%%%%%%%%%%%%%%%%%%%%%%%%%%

\begin{lemma}\label{L:QLIMIT}
The limits 
$Q_0:=\lim_{Y\to\yms}\frac{\RY(Y)}{\vY(Y)}$
and $\lim_{Y\to\yms}\bar{\mathcal C}(Y)$
exist and are finite. Moreover $0< Q_0\le1$ and
\[
\lim_{Y\to\yms}\bar{\mathcal C}(Y)=1 - \k + 4\k Q_0.
\]
\end{lemma}

%%%%%%%%%%%%%%%%%%%%%%%%%%

\begin{proof}
We let 
\be
\QY:=\frac{\RY}{\vY}.
\ee
We use~\eqref{E:BARA1}--\eqref{E:BARA2} to derive a differential equation for $Q$:
\begin{align}
Q' &= \frac{\RY \vY (\vY+\k)}{(1-\k)|Y|\vY^2\bar{\mathcal C}} \left[(1+3\k)(4Q-1) - \frac1{\vY} \alpha(Y)\right]
+ \frac{\RY^{1-\e}(\vY+\k)(1-3\vY)|Y|^{2+2\e}}{(1-\k)|Y|\chi^2 \vY^4 \bar{\mathcal C}} \notag \\
&=: \alpha_1(Y) + \alpha_2(Y)\RY^{1-\e}, \label{E:QDER}
\end{align}
where the function $\alpha(Y)$ is given by 
\begin{align}
\alpha(Y) = 1-9\k +\frac{7\k-3\k^2}{\vY} - \frac{\k-\k^2}{\vY^2}.
\end{align}
By the invariance of the flow, we know that $(1+3\k)(4Q-1)\ge \frac1{\vY}\alpha(Y)$ and in particular the function $\alpha_1$ is nonnegative on $(\yms,Y_0]$.
Integrating~\eqref{E:QDER} we conclude that 
\be\label{E:QQ}
Q(Y)-Q(\yms) = \int_{\yms}^Y \alpha_1(s)\,ds + \int_{\yms}^Y \RY^{1-\e}\alpha_2(s)\,ds. 
\ee
We observe that $\alpha_2$ is a bounded function on $(\yms,Y_0]$ and $\RY^{1-\e}$ is bounded from above by $\kappa^{1-\e}(\dy)^{-1+\e}$ and therefore
$\lv \int_{\yms}^Y \RY^{1-\e}\alpha_2(s)\,ds \rv \lesssim \frac1\k |Y-\yms|^{\e}$. Since $Q$ is bounded, it follows from~\eqref{E:QQ} that $\alpha_1\in L^1((\yms,Y_0])$.
For any $Y_1,Y_2\in(\yms,Y_0]$ we conclude that 
\[
Q(Y_1)-Q(Y_2) = \int_{Y_1}^{Y_2} \alpha_1(s)\,ds + \int_{Y_1}^{Y_2} \RY^{1-\e}\alpha_2(s)\,ds \lesssim \int_{Y_1}^{Y_2} \alpha_1(s)\,ds + \frac1\k |Y-\yms|^{\e}
\]
In particular, $Q$ is uniformly continuous on $(\yms,Y_0]$ and therefore $Q_0=\lim_{Y\to\yms}Q(Y)$ exists as claimed. Since $0<\frac{\RY}{\vY}\le1$ by Lemma~\ref{L:SIGMAVBOUND} 
we have $Q_0\in(0,1]$.
To show that $\lim_{Y\to\yms}\bar{\mathcal C}(Y)$ exists it follows from \eqref{E:NEWVBOUNDS} and the continuity of $Q$ that we only need to show the continuity of 
$\chi^{-2}\vY^{-2} \left(\RY Y^{-2}\right)^{-\eta}Y^2$ at $\yms$. However, by Lemma~\ref{L:CHIVBOUND} the quantity $\chi^{-2}\vY^{-2}$ is uniformly bounded from above
on approach to $\yms$ and by \eqref{E:NEWVBOUNDS} $ \left(\RY Y^{-2}\right)^{-\eta}Y^2 \lesssim |Y-\yms|^\e$ as $Y\to\yms$, so that $\chi^{-2}\vY^{-2} \left(\RY Y^{-2}\right)^{-\eta}Y^2$
necessarily converges to $0$ as $Y\to\yms$. Therefore $Y\mapsto \bar{\mathcal C}(Y)$ is continuous and it converges to  $1 - \k + 4\k Q_0$ at $\yms$.  %and converges to $1$ at $\yms$. is continuous.
%Let $Y_\ast>\yms$ be given. We then have
%We now regroup~\eqref{E:QDER} and rewrite it in the form 
%\begin{align}
%Q'  =  \frac{\RY \vY (\vY+\k)}{(1-\k)|Y|\vY^2\bar{\mathcal C}} \left[(1+3\k)(4Q-1) + \frac{(1-3\vY)}{\vY} \frac{\RY^{-\e}}{\chi^2\vY^2}|Y|^{2+2\e}\right]
% - Q \frac{\vY(\vY+\k)}{(1-\k)|Y|\vY^2 \bar{\mathcal C}}\alpha(Y).
%\end{align}
%By the invariance of the flow, we know that $(1+3\k)(4Q-1)\ge \frac1v\alpha(Y)$ and therefore the expression in the rectangular brackets satisfies the bound
%\begin{align}
%\big[ \dots \big] \ge \frac1{\vY} \left(\alpha(Y)\frac{(1-3\vY)}{\vY} \frac{\RY^{-\e}}{\chi^2\vY^2}|Y|^{2+2\e}\right)
%\end{align}
\end{proof}

%%%%%%%%%%%%%%%%%%%%%%%%%%%%%%%%%%%
%%%%%%%%%%%%%%%%%%%%%%%%%%%%%%%%%%%

\begin{lemma}\label{L:LIMAPRIORI2}
There exist positive constants %nonnegative real constants 
$\hat\vY, \hat\RY>0$ such that 
\begin{align}
\vY(Y) & = \frac{\hat\vY}{|Y-\yms|} + o_{Y\to\yms}(|Y-\yms|^{-1}), \label{E:VELLIM}\\
\RY(Y) & = \frac{\hat\RY}{|Y-\yms|} + o_{Y\to\yms}(|Y-\yms|^{-1}). \label{E:SIGMALIM}
\end{align} 
Moreover
\begin{align}\label{E:SVFORMULA}
\hat\RY = \frac{|\yms|}{6}, \ \ \hat\vY = \frac{2|\yms|}{3}.
\end{align}
\end{lemma}

%%%%%%%%%%%%%%%%%%%%%%%%%%%%%%%%%%%

\begin{proof}
%From Lemmas~\ref{L:SIGMAVBOUND} and~\ref{L:QLIMIT} it is clear that~\eqref{E:VELLIM} follows from~\eqref{E:SIGMALIM} and it
%therefore suffices to prove~\eqref{E:SIGMALIM}.
Dividing~\eqref{E:LDEQN2} by $\RY^2$ and using~\eqref{E:BARCDEF} we obtain
\begin{align}\label{E:SIGMALIM2}
-\frac{\RY'}{\RY^2} =\frac{2}{|Y|} \frac{  (1+\frac{\k}{\vY})^2(\frac{\vY}{\RY}-1)}{\bar{\mathcal C}}.
\end{align}
The right-hand side has a strictly positive limit $\frac1{\hat\RY}$ by Lemma~\ref{L:QLIMIT}, where $\hat{\RY}$ equals
\begin{align}\label{E:SIGMA0FORMULA}
\hat\RY = \frac{|\yms|}{2} \frac{Q_0(1+\k\left(4 Q_0-1\right))}{1-Q_0},
\end{align}
%which we denote by $\RY_0$
 and we may write the limit in the form $\frac{1}{\hat\RY} +o_{Y\to\yms}(|Y-\yms|)$. 
We may now integrate~\eqref{E:SIGMALIM2} to conclude~\eqref{E:SIGMALIM}.
Analogously, we divide~\eqref{E:LWEQN2} by $\vY^2$ and obtain
\begin{align}\label{E:VELLIM2}
-\frac{\vY'}{\vY^2} =\frac{\left(1+\frac{\k}{\vY}\right)\left(\frac1{\vY}-3\right)}{(1-\k)Y} - \frac{2(1+\e)}{Y} \frac{  (1+\frac{\k}{\vY})^2(\frac{\RY}{\vY}-1)}{\bar{\mathcal C}}.
\end{align}
%By Lemmas~\ref{L:LIMAPRIORI} and
By Lemma~\ref{L:QLIMIT}, the right-hand side converges to a limit denoted by $\frac1{\hat\vY}$, given by
\begin{align}\label{E:V0FORMULA}
\frac1{\hat\vY} = \frac{3}{(1-\k)|\yms|} - \frac{2(1+\e)}{|\yms|} \frac{1- Q_0}{1+\k(4Q_0-1)}.
\end{align}
We may now integrate~\eqref{E:VELLIM2} to conclude~\eqref{E:VELLIM}.
Since by~\eqref{E:VELLIM} and~\eqref{E:SIGMALIM} $Q_0=\frac{\hat\RY}{\hat\vY}$, it follows by multiplying~\eqref{E:SIGMA0FORMULA} and~\eqref{E:V0FORMULA} that 
\begin{align}
Q_0 = \frac{|\yms|}{2} \frac{Q_0(1+\k\left(4 Q_0-1\right))}{1-Q_0} \left(\frac{3}{(1-\k)|\yms|} - \frac{2(1+\e)}{|\yms|} \frac{1- Q_0}{1+\k(4Q_0-1)}\right).
\end{align}
Since $Q_0>0$ by Lemma~\ref{L:QLIMIT}, we may divide by $Q_0$ above and reduce the problem to the linear equation
$(4Q_0-1)\left(1+\k\right)=0$, hence $\frac{\hat\RY}{\hat\vY}=Q_0=\frac14$.
From~\eqref{E:SIGMA0FORMULA} and~\eqref{E:V0FORMULA} we now conclude~\eqref{E:SVFORMULA}.
%
%Upon dividing~\eqref{E:CHIEQN} by $\chi$ and integrating, for any $\yms<Y<Y_1$ we obtain the formula
%\begin{align}
%\chi(Y) = \chi (Y_1)
%\end{align}
\end{proof}

%%%%%%%%%%%%%%%%%%%%%%%%%%%%%%%%%%%
%%%%%%%%%%%%%%%%%%%%%%%%%%%%%%%%%%%

\begin{proposition}[Massive singularity]\label{P:SINGBEHAVIOUR}
%\item[(a).] (Singularity at $(0,0)$) Describe how the density and the curvature blow up as we approach $(0,0)$.
%\item[(b)] (Massive singularity)
Let $(\hat\vY,\hat\RY)$ be given by~\eqref{E:SVFORMULA}. There exists a $\hat\chi>0$ such that
on approach to the massive singularity $\MS$ the solution $(\RY,\vY,\chi)$ of~\eqref{E:LDEQN2}--\eqref{E:CHIEQN} obeys the following asymptotic behaviour:
\begin{align}
\vY(Y) & = \frac{\hat\vY}{|Y-\yms|} \left(1 + O_{Y\to\yms}(|Y-\yms|^{\e})\right),\label{E:VELLIM3}\\
\RY(Y) & = \frac{\hat\RY}{|Y-\yms|} \left(1 + O_{Y\to\yms}(|Y-\yms|^{\e})\right), \label{E:SIGMALIM3} \\
\chi(Y)& = \hat\chi |Y-\yms|^{\frac2{3(1-\k)}}\left(1 + O_{Y\to\yms}(|Y-\yms|^{\e})\right). \label{E:CHILIM3}
\end{align} 
%\begin{align}
%\RY(Y) \asymp_{Y\to\yms} (Y-\yms)^{-1}, \ \ \vY\asymp_{Y\to\yms} (Y-\yms)^{-1}, \ \ \chi\asymp_{Y\to\yms} (Y-\yms)^{\frac{2}{3(1-\k)}}. \label{E:BETAGAMMA}
%\end{align}
Moreover, the quantities $\tilde\mu$, $\tilde\lambda$, defined by the extension of~\eqref{E:MUTILDE} and~\eqref{E:LAMBDAFLUID} to $Y\in(\yms,0)$ respectively, satisfy 
\begin{align}
  e^{2\tilde\mu} \asymp_{Y\to\yms} (Y-\yms)^{\frac{2\k}{1-\k}}, \ \ e^{2\tilde\lambda}\asymp_{Y\to\yms} (Y-\yms)^{-\frac{2}{3(1-\k)}}. \label{E:METRICYMSASYMP}
\end{align}
The star density $\rho$ and the Ricci scalar $\mathcal R$ blow up on approach to $\MS$. 
%In particular, the metric $\grlp$ is inextendible across $\MS$ a $C^2$-metric.
%\end{enumerate}
\end{proposition}

%%%%%%%%%%%%%%%%%%%%%%%%%%%%%%%%%%%%%%%%%%

\begin{proof}
We introduce $q:=Q-\frac14=\frac{\RY}{\vY}-\frac14$, we can rewrite~\eqref{E:QDER} in the form
\begin{align}\label{E:LITTLEQEQN}
q' = \frac{A}{Y-\yms} q + B,
\end{align}
where
\[
A: = \frac{4(1+3\k)\RY|Y-\yms| \vY (\vY+\k)}{(1-\k)|Y|\vY^2\bar{\mathcal C}}, \ \ 
B: = - \frac{Q  (\vY+\k)}{(1-\k)|Y|\vY\bar{\mathcal C}}\alpha +  \frac{\RY^{1-\e}(\vY+\k)(1-3\vY)|Y|^{2+2\e}}{(1-\k)|Y|\chi^2 \vY^4 \bar{\mathcal C}}.
\]
From~\eqref{E:NEWVBOUNDS} and Lemma~\ref{L:LIMAPRIORI2} we immediately have
\begin{align}
\lim_{Y\to\yms} A(Y) & = \frac{4(1+3\k)\hat\RY}{(1-\k)|\yms|} = \frac23 (1+2\e)=A_0, \\
B(Y) & = O_{Y\to\yms}(|Y-\yms|^{-1+\e}).
\end{align}
We now consider $\yms<Y<Y_1$ for some fixed $Y_1$ and integrate~\eqref{E:LITTLEQEQN}. We obtain 
\begin{align}\label{E:LITTLEQFORMULA}
q(Y) = q(Y_1)e^{\int_{Y_1}^Y\frac{A(\tau)}{\tau-\yms}\,d\tau} +  \int_{Y_1}^Y e^{\int_{s}^Y\frac{A(\tau)}{\tau-\yms}\,d\tau} B(s)\,ds.
\end{align}
For any $0<\delta\ll1$ there exists a $Y_1>\yms$ such that $|A-A_0|<\delta$. It is then easy to see
that 
\[
e^{\int_{Y_1}^Y\frac{A(\tau)}{\tau-\yms}\,d\tau} = e^{-\int_{Y}^{Y_1}\frac{A(\tau)}{\tau-\yms}\,d\tau} \le e^{-\int_{Y}^{Y_1}\frac{A_0-\delta}{\tau-\yms}\,d\tau} = 
\frac{|Y-\yms|^{A_0-\delta}}{|Y_1-\yms|^{A_0-\delta}}.
\]
We use the bound in~\eqref{E:LITTLEQFORMULA} and conclude
\begin{align}
\lv q(Y) \rv &\lesssim \lv q(Y_1)\rv \frac{|Y-\yms|^{A_0-\delta}}{|Y_1-\yms|^{A_0-\delta}} 
+ \int_{Y_1}^Y  \frac{|Y-\yms|^{A_0-\delta}}{|s-\yms|^{A_0-\delta}} |s-\yms|^{-1+\e}\,ds \notag \\
& \lesssim  |Y-\yms|^{A_0-\delta}  + |Y-\yms|^{\e} \lesssim  |Y-\yms|^{\e}.
\end{align}
Plugging this back into~\eqref{E:SIGMALIM2} and~\eqref{E:VELLIM2} allows us to obtain the (suboptimal) rates~\eqref{E:VELLIM3}--\eqref{E:SIGMALIM3}.
%\begin{align}
%\vY(Y) & = \frac{\vY_0}{|Y-\yms|}\left(1 + O_{Y\to\yms}(|Y-\yms|^{\e})\right) \label{E:VELLIMSUB}\\
%\RY(Y) & = \frac{\RY_0}{|Y-\yms|} \left(1+ O_{Y\to\yms}(|Y-\yms|^{\e})\right). \label{E:SIGMALIMSUB}
%\end{align}
Upon dividing~\eqref{E:CHIEQN} by $\chi$ and integrating, using~\eqref{E:VELLIM3}, %--\eqref{E:SIGMALIM3}, 
for any $\yms<Y<Y_1$ we obtain
%\begin{align}
%\chi(Y) & = \chi (Y_1) e^{-\frac1{1-\k}}\int_{Y}^{Y_1} \frac{\hat\vY}{\tau (\tau-\yms)}\left(1 + O_{\tau\to\yms}(|\tau-\yms|^{\e})\right) \, d\tau \notag\\
%& =  \chi (Y_1) e^{-\frac2{3(1-\k)}}\int_{Y}^{Y_1} \frac{|\yms|}{\tau (\tau-\yms)}\left(1 + O_{Y\to\yms}(|Y-\yms|^{\e})\right) \, d\tau \notag\\
%& = \chi (Y_1) e^{-\frac2{3(1-\k)}}\int_{Y}^{Y_1} \frac{1}{ \tau-\yms}\left(1 + O_{Y\to\yms}(|Y-\yms|^{\e})\right) \, d\tau \notag\\
%& = O(1) |Y-\yms|^{\frac2{3(1-\k)}}\left(1 + O_{Y\to\yms}(|Y-\yms|^{\e})\right),
%\end{align}
\begin{align}
\chi(Y) & = \chi (Y_1) e^{\frac1{1-\k}\int_{Y}^{Y_1} \frac{\hat\vY}{\tau (\tau-\yms)}\left(1 + O_{\tau\to\yms}(|\tau-\yms|^{\e})\right) \, d\tau} \notag\\
& =  \chi (Y_1) e^{\frac2{3(1-\k)}\int_{Y}^{Y_1} \frac{|\yms|}{\tau (\tau-\yms)}\left(1 + O_{Y\to\yms}(|Y-\yms|^{\e})\right) \, d\tau} \notag\\
& = \chi (Y_1) e^{\frac2{3(1-\k)}\int_{Y}^{Y_1} \left(\frac{1}{\tau}-\frac{1}{ \tau-\yms}\right)\left(1 + O_{Y\to\yms}(|Y-\yms|^{\e})\right) \, d\tau }\notag\\
& = O(1) |Y-\yms|^{\frac2{3(1-\k)}}\left(1 + O_{Y\to\yms}(|Y-\yms|^{\e})\right),
\end{align} 
which proves~\eqref{E:CHILIM3}.

The asymptotics for $ e^{2\tilde\mu}$ in~\eqref{E:METRICYMSASYMP} follows directly from~\eqref{E:SIGMALIM3} and~\eqref{E:NEWMU}.
The asymptotics for $ e^{2\tilde\l}$ in~\eqref{E:METRICYMSASYMP} can be read off from~\eqref{E:SIGMALIM3}--\eqref{E:CHILIM3},~\eqref{E:LAMBDATILDE}, and the identity~\eqref{E:LAMBDAFORMULA}, 
which in the $(\RY,\vY,\chi)$ variables reads
\begin{align}\label{E:TILDELAMBDAEQN}
\RY^{\frac1{1-\k}}e^{\tilde\l}\chi^2 = \alpha Y^{2+\e},
\end{align}
for some constant $\alpha>0$. It then follows from~\eqref{E:TILDELAMBDAEQN} that as $\dy \to 0^+$, $e^{2\tilde\lambda} \asymp (\dy)^{-\frac{2}{3(1-\k)}}$.

From%~\eqref{E:TILDERWDEF} and
~\eqref{E:YVARDEF}, we conclude
that $ \Sigma(Y) = d(Y)^{\frac{1+\k}{1-\k}}\asymp_{Y\to\yms} \dy^{-\frac{1+\k}{1-\k}}$, where we slightly abuse notation by continuing to denote $\Sigma(Y)$
to signify the self-similar energy density (see~\eqref{E:SS1}).
Therefore by~\eqref{E:SS1}, for any $\ttau>0$
\begin{align}
\rho(\ttau,R)& = \frac1{\tau(\ttau,R)^2}  \Sigma(Y)
= \frac{R^{2\e}}{\ttau^{2\frac{1+\k}{1-\k}}\k^\e}  \Sigma(Y) 
 =\frac1{\ttau^2 Y^2} \Sigma(Y) \asymp_{Y\to\yms}  \dy^{-\frac{1+\k}{1-\k}}.
\end{align}
The fluid density therefore implodes along $\MS$ and as a consequence of~\eqref{E:CURVATUREDENSITY}, the Ricci scalar blows up at $\MS$.
\end{proof}

%%%%%%%%%%%%%%%%%%%%%%%%%%%%%%%%%%%
%%%%%%%%%%%%%%%%%%%%%%%%%%%%%%%%%%%

\begin{remark}
We can now use the asymptotic rates from from the previous lemma to replace the rough upper bound $\chi^{-2}\vY^{-2}\lesssim 1$, by the sharp upper bound rate $|Y-\yms|^{2-\frac{4}{3(1-\k)}}$. This can then be bootstrapped to obtain the 
near optimal next order correction in the rates~\eqref{E:SIGMALIM3}--\eqref{E:VELLIM3}, where $\e$ can be replaced by $\frac23+O(\k)$. We do not pursue this here, as it will not be needed in the rest of the paper.
\end{remark}

%%%%%%%%%%%%%%%%%%%%%%%%%%%%%%
%%%%%%%%%%%%%%%%%%%%%%%%%%%%%%
%%%%%%%%%%%%%%%%%%%%%%%%%%%%%%

\section{Causal structure of the RLP family of solutions}\label{S:RLP}

%%%%%%%%%%%%%%%%%%%%%%%%%%%%%
%%%%%%%%%%%%%%%%%%%%%%%%%%%%%
As explained in Section~\ref{SS:RLPINTRO},
Theorems~\ref{T:FRIEDMANN},~\ref{T:GLOBALRIGHT}, and~\ref{T:GE} imply the existence of 
a maximally self-similarly extended RLP spacetime, recall Definition~\ref{D:RLPST}. The standard ODE-theory
implies that the solution is indeed real-analytic in $Y$ on $(\yms,\infty)$.
We next show that the RLP-spacetimes are not asymptotically flat and compute the associated mass-aspect function as $r\to\infty$.

%%%%%%%%%%%%%%%%%%%%%%%%%%%%%%
%%%%%%%%%%%%%%%%%%%%%%%%%%%%%%

\begin{lemma}[$(\MRLP,\grlp)$ is not asymptotically flat]
The RLP spacetimes $(\MRLP,\grlp)$ are not asymptotically flat, i.e. for any $\tau<0$, $\lim_{R\to\infty}m(\tau,R)=\infty$. 
More precisely the mass aspect function $\frac{2m(\tau,r)}{r}$ satisfies
\[
\lim_{r\to\infty}\frac{2m(\tau,r)}{r}=\lim_{R\to\infty} \frac{2m(\tau,R)}{r(\tau,R)} = 4\k %\frac{\k}{\pi} 
d_2,
\]
where $d_2>1$ is the $\k$-independent constant introduced in Lemma~\ref{L:ROUGHASYMPTOTICS}.
\end{lemma}

%%%%%%%%%%%%%%%%%%%%%%%%%%%%%%
%%%%%%%%%%%%%%%%%%%%%%%%%%%%%%

\begin{proof}
Fix a $\tau<0$. As $R$ increases to $\infty$ we have $y=\frac{R}{-\sqrt\k \tau}\to\infty$. Recalling~\eqref{E:HAWKINGCOMOVING} we have
\begin{align}
\lim_{R\to\infty} \frac{m(\tau,R)}{R} & = \lim_{R\to\infty} \frac{4\pi \int_0^R \frac1{2\pi \tau^2}  \Sigma(y)\, \k \tau^2 \tr^2(y) \tr'(y) \,d\bar R}{R}   \notag\\
& =2\k %\frac\k{2\pi}
 \lim_{y\to\infty} \frac{\int_0^y  \Sigma(z) \tr(z)^2 \tr'(z) \,dz}{y} \notag\\
& = 2\k %\frac\k{2\pi} 
\lim_{y\to\infty}  \Sigma(y)\tr(y)^2\tr'(y) =2\k %\frac{\k}{2\pi} 
d_2, \label{E:NOTAF}
\end{align}
where we have changed variables and used $y=\frac{R}{-\sqrt\k \tau}$ in the second equality, the l'Hospital rule in the third, and~\eqref{E:WROUGH}--\eqref{E:DROUGH},~\eqref{E:TILDERPRIME}, Remark~\ref{R:TAYLORR} in the last.
Note that $\lim_{R\to\infty}\frac{r}{R}=\lim_{r\to\infty}\frac{r}{R}=1$ by~\eqref{E:RLASYMPTOTICS}--\eqref{E:ANORM}.
\end{proof}

%\subsection{Causal structure}

%%%%%%%%%%%%%%%%%%%%%%%%%%%%%%
%%%%%%%%%%%%%%%%%%%%%%%%%%%%%%
%%%%%%%%%%%%%%%%%%%%%%%%%%%%%%

As stated in~\cite{OP1990}, it is clear that along any line of the form $(\tau,\alpha\tau)$ the density $\rho(\tau,\alpha\tau) = \frac1{2\pi\tau^2} \Sigma(\frac{-\alpha}{\sqrt\k})$ diverges as $\tau\to0^-$. A similar statement applies when we approach the scaling origin $(0,0)$ along the lines of the form $(\tilde\tau,\alpha\tilde\tau)$, $\ttau>0$. In particular, by~\eqref{E:CURVATUREDENSITY} the Ricci scalar blows up at $(0,0)$ and this is a geometric singularity. 
We now proceed to study the radial null-geodesics that ``emanate" from this singularity. We shall henceforth use the abbreviation RNG for the radial null geodesics.
%We follow here the strategy outlined in~\cite{OP1990} and provide rigorous proofs.

%%%%%%%%%%%%%%%%%%%%%%%%%%%
%%%%%%%%%%%%%%%%%%%%%%%%%%%

\begin{lemma}\label{L:RNGS}
\begin{enumerate}
\item[{\em (a)}]
In the $(\tau,R)$-plane the outgoing/ingoing RNG-s respectively satisfy the equations
\begin{align}\label{E:GEODESICL}
\frac{dR}{d\tau} = \pm e^{\mu(y)-\l(y)}, \ \ y = \frac{R(\tau)}{-\sqrt\k \tau},
\end{align}
whenever the right-hand side is well-defined.
\item[{\em (b)}]
Similarly, in the $(\ttau,R)$-plane the outgoing/ingoing RNG-s respectively satisfy the equations
\begin{align}\label{E:GEQN3}
\frac{dR}{d\tilde\tau} = \frac{\pm 1}{e^{\tilde\l(Y)-\tilde\mu(Y)}\mp \frac{2\sqrt\k}{1+\k}Y}, \ \ Y =  \frac{-\sqrt \k\tilde\tau}{R(\ttau)},
\end{align}
whenever the right-hand side is well-defined.
\item[{\em (c)}]
%An RNG of the form
%\begin{align}\label{E:SIMPLERNGDEF}
%R(\tilde\tau) = \sigma \tilde\tau, \ \ \sigma\in\mathbb R\setminus\{0\},
%\end{align}
%is called a \underline{simple radial null-geodesic} (simple RNG). 
If we let $Y = -\frac{\sqrt\k}{\sigma}$, then the curve 
$(\sigma\ttau,\ttau)$ is a simple outgoing/ingoingRNG (see Definition~\ref{D:SIMPLERNG}) if and only if
\begin{align}
G_\pm(Y) = 0,
\end{align}
where
\be\label{E:GPMDEF}
G_\pm(Y):=\sqrt\k e^{\tilde\l(Y)-\tilde\mu(Y)} \pm \frac{1-\k}{1+\k}Y.
\ee
\end{enumerate}
\end{lemma}

%%%%%%%%%%%%%%%%%%%%%%%%%%%

\begin{proof}
{\em Proof of part (a).}
Equation~\eqref{E:GEODESICL} is just the condition that the radial geodesic  has null length in the local coordinates~\eqref{E:METRICRLP1}.

\noindent
{\em Proof of part (b)}.
In the local coordinates~\eqref{E:METRICRLP2} the RNG-s satisfy
\begin{align}
-e^{2\tilde\mu(Y)}\left(\dot{\tilde\tau}(s)\right)^2- \frac{4\sqrt\k}{1+\k} Ye^{2\tilde\mu} \dot R(s)\dot\tt(s) + \left(e^{2\tilde\l(Y)}-\frac{4\k}{(1+\k)^2} Y^2 e^{2\tilde\mu(Y)} \right)\left(\dot R(s)\right)^2=0,
\end{align}
where it is understood that $Y = \frac{-\sqrt\k \tt(s)}{R(s)}$ in the above expression. 
Reparametrising by $\tilde\tau$ the above ODE, we obtain
\begin{align}\label{E:GEQN}
1& = \left(e^{2\tilde\l(Y)-2\tilde\mu(Y)}-\frac{4\k}{(1+\k)^2} Y^2  \right)\left(\frac{dR}{d\tilde\tau}\right)^2 - \frac{4\sqrt\k}{1+\k} Y \frac{dR}{d\tt},
\end{align}
or equivalently
\begin{align}
\left(\frac{dR}{d\tilde\tau}\right)^2e^{2\tilde\l(Y)-2\tilde\mu(Y)}  = \left(1+\frac{2\sqrt\k}{1+\k} Y \frac{dR}{d\tt}\right)^2.\label{E:GEQN2}
%\ \ y = \frac{R}{\sqrt \k\tau}.
\end{align}
Upon taking the square root, solutions to the equation~\eqref{E:GEQN2}
are given by the solutions of~\eqref{E:GEQN3}.

\noindent
{\em Proof of part (c)}.
The proof follows by direct substitution in~\eqref{E:GEQN3}.
%\begin{align}\label{E:GEQN3}
%\frac{dR}{d\tilde\tau} \left(e^{\tilde\l(Y)-\tilde\mu(Y)} \mp \frac{2\sqrt\k Y}{1+\k}\right) = \pm 1.
%\end{align}
\end{proof}
%%%%%%%%%%%%%%%%%%%%%%%%%%%%%%%
%%%%%%%%%%%%%%%%%%%%%%%%%%%%%%%

%\begin{definition}
%A null-geodesic of the form
%\begin{align}\label{E:SIMPLERNGDEF}
%R(\tilde\tau) = \sigma \tilde\tau, \ \ \sigma\in\mathbb R,
%\end{align}
%is called a \underline{simple radial null-geodesic}. 
%\end{definition}

%%%%%%%%%%%%%%%%%%%%%%%%%%%%%%%
%%%%%%%%%%%%%%%%%%%%%%%%%%%%%%%

%%%%%%%%%%%%%%%%%%%%%%%%%%%%%%%
%%%%%%%%%%%%%%%%%%%%%%%%%%%%%%%

%%%%%%%%%%%%%%%%%%%%%%%%%%%%%%%

\begin{remark}
For the purpose of describing the causal structure of the self-similar spacetimes under consideration, 
it is convenient to introduce the function
\begin{align}\label{E:FDEFLITTLEY}
F_\k(Y): = Y^2 e^{2\tilde\mu_\k(Y)-2\tilde\l_\k(Y)}, \ \ Y>\yms, 
\end{align}
where $Y = \frac{-\sqrt\k \tilde\tau}{R}$ is the self-similar coordinate associated with the patch~\eqref{E:METRICRLP2}
and we keep the index $\k$ to emphasise the dependence on $\k$.
It is then straightforward to check from part (c) of Lemma~\ref{L:RNGS} that an RNG $(\sigma\ttau,\ttau)$, $\sigma\neq0$, is simple if and only if 
\begin{align}\label{E:FEQN}
\frac1\k F_\k(Y) = \left(\frac{1+\k}{1-\k}\right)^2, \ \ Y = - \frac{\sqrt\k}{\sigma}.
\end{align}
Formulas~\eqref{E:GPMDEF} and~\eqref{E:FDEFLITTLEY} give the obvious factorisation property:
\begin{align}\label{E:FGGFACTORISATION}
G_+(Y) G_-(Y) = - \k \left(\frac{1-\k}{1+\k}\right)^2e^{2\tilde\l_\k(Y)-2\tilde\mu_\k(Y)}\left(\frac1\k F_\k(Y) -\left(\frac{1+\k}{1-\k}\right)^2\right).
\end{align}
\end{remark}

%%%%%%%%%%%%%%%%%%%%%%%%%%%%%%%
%%%%%%%%%%%%%%%%%%%%%%%%%%%%%%%

\begin{proposition}\label{P:UNIF}
Let $F_\k(\cdot)$ be the function given by~\eqref{E:FDEFLITTLEY}. 
There exists an $0<\k_0\ll1$ such that the following statements are true.
\begin{enumerate}
\item[{\em (a)}]
The function $F_\k$ satisfies the formula
\begin{align}\label{E:FDEFY}
F_\k(Y) =\frac{(1+\k)^2}{(1-\k)^2}  \frac{(\RY_\k Y^{-2})^{-\e}Y^2 + \k \chi_\k^2\left[(\vY_\k-1)^2 - 4 \vY_\k\RY_\k\right]}{(\vY_\k+\k)^2\chi_\k^2}, \ \ Y>\yms.
\end{align}
\item[{\em (b)}]
There exists an $Y<0$ such that $Y>\yms(=\yms_\k)$ for all $\k\in(0,\k_0]$ and
\begin{align}
\frac1\k F_\k(Y) >2 \ \ \text{ for all } \ \k\in(0,\k_0].
\end{align}
\item[{\em (c)}]
For any fixed $\k\in(0,\k_0]$ we have
\begin{align}
\lim_{Y\to 0} F_\k(Y) & = 0 ,\label{E:C1}\\
\lim_{Y\to \yms} F_\k(Y) & = 0.\label{E:C2}
\end{align}
\end{enumerate}
\end{proposition}

%%%%%%%%%%%%%%%%%%%%%%%%%%%%%%%

\begin{proof}
{\em Proof of part {\em (a)}.}
By~\eqref{E:MUTILDE} we have for any $Y>0$
\begin{align}
F_\k(Y) & = Y^2 e^{2\tilde\mu_\k(Y)-2\tilde\l_\k(Y)} 
 =   \frac{(1+\k)^2}{(1-\k)^2} Y^{2+2\e} e^{2\mu_\k(y)-2\l_\k(y)} \notag \\ 
 &=  \frac{(1+\k)^2}{(1-\k)^2}  y^{-2}e^{2\mu_\k(y)-2\l_\k(y)}
%%\label{E:MUTILDE}
%\end{align}
%It follows from~\eqref{E:SONICFORMULA} that 
%\begin{align}
%y^{-2}e^{2\mu_\k(y)-2\l_\k(y)} 
 =\frac{(1+\k)^2}{(1-\k)^2} \frac{ \d_\k^{-\e} +\k \tr_\k^2 \left[(\w _\k-1)^2- 4  \Rl_\k \w _\k\right]}{(\w _\k+\k)^2 \tr_\k^2} \notag\\
& =\frac{(1+\k)^2}{(1-\k)^2}  \frac{(\RY_\k Y^{-2})^{-\e}Y^2 + \k \chi_\k^2\left[(\vY_\k-1)^2 - 4 \vY_\k\RY_\k\right]}{(\vY_\k+\k)^2\chi_\k^2},  \ \ y, Y>0, \label{E:FDEF2}
\end{align}
where the index $\k$ is added to emphasise the dependence on $\k$. We used the formula~\eqref{E:YLITTLEY} in the third equality, and~\eqref{E:LBDEF} to express $y^{-2}e^{2\mu_\k(y)-2\l_\k(y)}$ in terms of $\w _\k$, $ \d_\k$, and $\tr_\k$ in the fourth.
%\begin{align}
%y^{-2}\left(e^{2\mu_\k-2\lambda_\k}-y^2\right) & = 
%\frac{\tilde R_\k^{-\e}}{(\tilde w_\k+\k)^2\tilde r_\k^2} - \frac{(\tilde w_\k+\k)^2 - \k(\tilde w_\k-1)^2 + 4\k \tilde w_\k \tilde R_\k}{(\tilde w_\k+\k)^2} \notag \\
%& = \frac{\RY_\k^{-\e}Y^{2+2\e}- \chi^2\left[(\vY_\k+\k)^2 - \k(\vY_\k-1)^2 + 4\k \vY_\k\RY_\k\right]}{(\vY_\k+\k)^2\chi^2} \notag \\
%& = \frac{\mathcal C_\k}{(\vY_\k+\k)^2 \chi_\k^2}, \ \ Y>0,\label{E:FDEF1}
%\end{align}
%where $\mathcal C_\k$ is the sonic denominator and the index $\k$ is added to emphasise its dependence on $\k$.
%It follows in particular that 
%\begin{align}\label{E:FDEF2}
%F_\k(Y) = \frac{(1+\k)^2}{(1-\k)^2} \left(1+\frac{\mathcal C_\k}{(\vY_\k+\k)^2 \chi_\k^2}\right), \ \ Y>0.
%\end{align}
Since the right-most side of~\eqref{E:FDEF2} extends analytically to $Y\in(\yms,0]$ by Theorems~\ref{T:LE} and~\ref{T:GE} the claim in part (a) follows.

\noindent
{\em Proof of part {\em (b)}.}
Note that for $Y\in(\yms,0]$ by~\eqref{E:FDEF2}
\begin{align}
\frac1\k F_\k  = \frac{(1+\k)^2}{\k (1-\k)^2}  \frac{(\RY_\k Y^{-2})^{-\e}Y^2}{(\vY_\k+\k)^2\chi_\k^2} 
+ \frac{(1+\k)^2}{(1-\k)^2}  \frac{\left[(\vY_\k-1)^2 - 4 \vY_\k\RY_\k\right]}{(\vY_\k+\k)^2}.\label{E:UNIF}
%%\frac{(1-\k)^2}{(1+\k)^2} F_\k=1 + \frac{\mathcal C_\k}{(\vY_\k+\k)^2 \chi_\k^2} 
%= \frac{\RY_\k^{-\e}(Y^2)^{1+\e}}{(\vY_\k+\k)^2 \chi_\k^2} + \k\left[\frac{(\vY_\k-1)^2}{(\vY_\k+\k)^2} - \frac{4\vY_\k\RY_\k}{(\vY_\k+\k)^2}\right] \label{E:FDEF3}
\end{align}
%\begin{align}
%\frac1{\k} F_\k(Y) - 1 &= \frac1{(\vY_\k+\k)^2} \left(\frac{\RY_\k^{-\e}(Y^2)^{1+\e}}{\k  \chi_\k^2} +(\vY_\k-1)^2 - 4\vY_\k\RY_\k- (\vY_\k+\k)^2 \right) \notag \\
%& = \frac1{(\vY_\k+\k)^2} \left(\frac{\RY_\k^{-\e}(Y^2)^{1+\e}}{\k  \chi_\k^2} -2\vY_\k(1+\k+2\RY_\k) + 1-\k^2 \right). \label{E:UNIF}
%\end{align}
We now fix $Y=Y_0$, a constant provided by the local extendibility statement of Theorem~\ref{T:LE}. In particular, by part (a) of Theorem~\ref{T:LE} there exists a $\delta>0$ such that 
for all $\k\in(0,\k_0]$
we have $\RY_\k(Y_0)>\delta$, $\chi_\k(Y_0)>\delta$ and $\RY_\k(Y_0)<\vY_\k(Y_0)<\frac1\delta$. By letting $\k\to0$ in~\eqref{E:UNIF} we conclude that
\begin{align}
\frac1\k F_\k(Y_0) >2  \ \ \text{ for } \ \k \text{ sufficiently small}. 
\end{align}

\noindent
{\em Proof of part {\em (c)}.}
Claim~\eqref{E:C1} follows from the formula~\eqref{E:FDEF2} and the asymptotic behaviour~\eqref{E:RLASYMPTOTICS}--\eqref{E:SMALLVASYMPTOTICS}.
%it follows that 
%\begin{align}\label{E:FASYMP1}
%\lim_{Y\to0} F_\k(Y) = \infty.
%\end{align}
To prove~\eqref{E:C2} we work directly with~\eqref{E:FDEFLITTLEY} and use~\eqref{E:METRICYMSASYMP}. This gives
%Similarly, by part (b) of Proposition~\ref{P:SINGBEHAVIOUR} we conclude that
\begin{align}
\lim_{Y\to(\yms)^+}F_\k(Y) & = \lim_{Y\to(\yms)^+} \left(Y^2 e^{2\mu_\k(Y)-2\l_\k(Y)}\right) \notag\\
& \asymp_{Y\to\yms} c (\yms)^2  \lim_{(Y-\yms)\to0^+} (\dy)^{\frac{2(1+3\k)}{3(1-\k)}} = 0.\label{E:FASYMP2}
\end{align}
%since $\gamma = \frac23+O(\k)$. 
%HERE I NEED TO CORRECT THE EXPONENTS..
\end{proof}
%is given by
%\begin{align}
%F_\k(Y) = Y^
%\end{aligb}

We are now ready to prove Theorem~\ref{T:NS}.

{\em Proof of Theorem~\ref{T:NS}.}
Our first goal is to show that for  $0<\k\ll1$ there exist at least two solutions to the equation~\eqref{E:FEQN}. 
By parts (b) and (c) of Proposition~\ref{P:UNIF} it is now clear that the function $Y\mapsto \frac1\k F_\k(Y)$ converges to $0$ at $Y=\yms$ and $Y=0$, but 
necessarily peaks above $\left(\frac{1+\k}{1-\k}\right)^2$ at $Y=Y_0$, where $Y_0$ is given by Theorem~\ref{T:LE}. Therefore there
exist $Y_1\in(Y_0,0)$ and $Y_2\in(\yms,Y_0)$ such that~\eqref{E:FEQN} holds with $Y=Y_1$ and $Y=Y_2$. Since $Y_1, Y_2$ are strictly negative and necessarily
zeroes of $G_\pm$ defined by~\eqref{E:GPMDEF}, they must in fact be zeroes of $G_+$ and therefore represent outgoing simple RNG-s.
Note that the function $Y\mapsto  F_\k(Y) -\k \left(\frac{1+\k}{1-\k}\right)^2$ is real analytic on $(\yms,0)$ and therefore
the number of zeroes is finite.  By slight abuse of notation we enumerate the zeroes as in~\eqref{E:SRNGS}.
\prfe
%%%%%%%%%%%%%%%%%%%%%%%%%%%%%%
%%%%%%%%%%%%%%%%%%%%%%%%%%%%%%

%%%%%%%%%%%%%%%%%%%%%%%%%%%%%%
%%%%%%%%%%%%%%%%%%%%%%%%%%%%%%
%%%%%%%%%%%%%%%%%%
%%%%%%%%%%%%%%%%%%

We recall the relationship~\eqref{E:YLITTLEY} between the comoving self-similar coordinates $y$ and $Y$, as well as the relation between the comoving coordinate $y$ and 
the Schwarzschild coordinates $x$ in Subsection~\ref{SS:SCHW}. Recalling the sonic point $\bar x_\ast$, the slope of the sonic line in the $Y$-coordinates is given by 
\begin{align}
\Ysp:=\left(\tilde r^{-1}(\bar x_\ast)\right)^{-\frac1{1+\e}}.
\end{align}
The next lemma is important for the description of ingoing null-geodesics. It in particular implies that the curve $\mathcal N$
is the unique simple ingoing RNG.

\begin{lemma}\label{L:YN}
For any $\k\in(0,\k_0]$, consider the relativistic Larson-Penston spacetime given by Definition~\ref{D:RLPST}.
Then there exists a $Y_{\mathcal N}\in(0,\Ysp)$ such that the curve 
\begin{align}
\mathcal N : = \{(\tt, R)\in\MRLP\,\big| \, \frac{-\sqrt\k \tt}{R} = Y_{\mathcal N}\},
\end{align}
represents a simple ingoing null-geodesics i.e. the boundary of the past light cone of the scaling origin $\mathcal O$.
Moreover, the curve $\mathcal N$
is the unique simple ingoing RNG and the following bounds hold:
\begin{align}\label{E:AE}
G_-(Y)&>0, \ \ Y\in (\yms, Y_{\mathcal N}), \\
G_-(Y)&<0, \ \ Y\in (Y_{\mathcal N},\infty), \label{E:AE2}
\end{align}
where $G_-(\cdot)$ is defined in~\eqref{E:GPMDEF}.
\end{lemma}

%%%%%%%%%%%%%%%%%%

\begin{proof}
By Theorem~\ref{T:NS} it suffices to show that there exists a $Y_{\mathcal N}\in(0,\Ysp)$ which solves the 
equation~\eqref{E:FEQN}. By part (c) of Proposition~\ref{P:UNIF} we know that the function $Y\mapsto \frac1\k F_\k(Y)$ converges to $0$ at $Y=0$. On the other hand, from~\eqref{E:FDEFLITTLEY} and~\eqref{E:MUTILDE} we have
\begin{align}\label{E:SHORTERPROOF}
F_\k(Y) & = \frac{(1+\k)^2}{(1-\k)^2} y^{-2} e^{2\mu(y)-2\l(y)} = \frac{(1+\k)^2}{(1-\k)^2} y^{-2} \left(e^{2\mu(y)-2\l(y)}-y^2\right) + \frac{(1+\k)^2}{(1-\k)^2}.
\end{align}
Therefore, at the sonic point $\Ysp$ we conclude $F_\k(\Ysp)=\frac{(1+\k)^2}{(1-\k)^2}$, which implies that $\frac1\k F_\k(\Ysp)$ is larger than $\left(\frac{1+\k}{1-\k}\right)^2$ for all $\k\in(0,1)$. By continuity, there exists an $Y_{\mathcal N}\in(0,\Ysp)$ such that $\frac1\k F_\k(Y_{\mathcal N})= \left(\frac{1+\k}{1-\k}\right)^2$.
Now observe that by~\eqref{E:SHORTERPROOF}, the zeroes of the function $\frac1\k F_\k(Y)-\left(\frac{1+\k}{1-\k}\right)^2$ are in 1-1 relationship with the zeroes of the function $(0,\infty)\ni y\mapsto e^{2\l}y^2 - \frac1\k e^{2\mu}$. 
It is shown in Lemma~\ref{L:MATHCALHMONOTONE} that the map 
$y\mapsto \tr(y)^2\left(e^{2\l}y^2 - \frac1\k e^{2\mu}\right)$ is strictly monotone on $(0,\infty)$ and the uniqueness claim follows. Inequalities~\eqref{E:AE}--\eqref{E:AE2} are a simple consequence of the factorisation~\eqref{E:FGGFACTORISATION}, the above monotonicity, and the positivity of $G_+$ for all $Y>0$.
\end{proof}

%%%%%%%%%%%%%%%%%%
%%%%%%%%%%%%%%%%%%
%%%%%%%%%%%%%%%%%
%%%%%%%%%%%%%%%%%

\begin{definition}\label{D:EXTERIOR}
We refer to the region in the future of the backward null-curve $\mathcal N$ as the \underline{exterior} region, and the region in the past of the null-curve $\mathcal N$ as the
\underline{interior} region, following here the terminology in~\cite{Ch1994}, see Figure~\ref{F:EXTERIOR}.
\end{definition}

%%%%%%%%%%%%%%%%%
%%%%%%%%%%%%%%%%%

%%%%%%%%%%%%%%%%%
%%%%%%%%%%%%%%%%%

\begin{remark}\label{R:WELLDEFINEDGEODESICS}
For any $(R,\ttau)$ in the exterior region, we note that 
$\sqrt\k e^{\tilde\mu-\tilde\l}-\frac{2\k}{1+\k}Y$ is clearly positive for $Y\in(\yms,0]$. When $Y\in(0,Y_{\mathcal N})$ we may rewrite this expression as $G_-(Y)+\frac{1-3\k}{1+\k}Y$, which is 
then positive by Lemma~\ref{L:YN} for sufficiently small $\k$. 
On the other hand, the expression $\sqrt\k e^{\tilde\mu-\tilde\l}+\frac{2\k}{1+\k}Y$ is clearly positive for $Y\ge0$. If $Y\in(Y_1,0)$, we may rewrite it as $G_+(Y)-\frac{1-3\k}{1+\k}Y$, which 
is then necessarily positive. This shows that the right-hand side of the null-geodesic equation appearing in~\eqref{E:GEQN3} is well-defined in the exterior region $\{Y_1<Y<Y_{\mathcal N}\}$.
\end{remark}

%%%%%%%%%%%%%%%%%
%%%%%%%%%%%%%%%%%
%%%%%%%%%%%%%%%%%%%%%%
%%%%%%%%%%%%%%%%%%%%%%

\begin{lemma}[General structure of radial null geodesics]\label{L:INCOMING}
\begin{enumerate}
\item[(a)]
 For any point $(R_0,\tt_0)$ in the interior region,
the future oriented ingoing null curve $\tt\mapsto R(\tt)$ through $(R_0,\tt_0)$ remains in the interior region and intersects the surface $\{R=0\}$ at some $\ttau<0$.
\item[(b)]
For any point $(R_0,\tt_0)$ in the exterior region,
%Consider an ingoing radial null-geodesic passing through a given point $(R_0,\tt_0)$ in the {\em exterior region}. 
the future oriented ingoing null curve $\tt\mapsto R(\tt)$ exits the exterior region by intersecting $\mathcal B_1$ 
and also intersects $\mathcal B_2$ at positive values of the $R$-coordinate.
\item[(c)]
%Consider an outgoing radial null-geodesic passing through a given point $(R_0,\tt_0)$ in the {\em exterior region}. 
For any $(R_0,\tt_0)$ in the union of the exterior and the interior region, the outgoing null curve $\tt\mapsto R(\tt)$ through $(R_0,\tt_0)$  exists globally-in-$\ttau$ and $\lim_{\ttau\to\infty}R(\ttau)=\infty$.
Moreover, no such curve can converge to the scaling origin $\mathcal O$ to the past.
\end{enumerate}
\end{lemma}

%%%%%%%%%%%%%%%%%
%%%%%%%%%%%%%%%%%

\begin{proof}
\noindent
{\em Proof of part (a)}. 
Assume now that $(R_0,\ttau_0)$ is in the interior region. The future oriented ingoing geodesic must stay in the interior
region, as it cannot cross the backward light cone $Y=Y_{\mathcal N}$ by the ODE uniqueness theorem. In the interior region
it is more convenient to switch to the original comoving coordinates $(R,\tau)$ as they also cover the centre of symmetry surface 
$\{(R,\tau)\,\big| \, R=0, \ \tau<0\} = \{(r,\tau), \, \big| \, r=0,\tau<0\}$. Let $(R_0,\tau_0)$ correspond to $(R_0,\ttau_0)$.
%By analogy to~\eqref{E:CHANGEOFV}, we parametrise the
%geodesic by $y$, using the change of variables $\tau\to y$. Noting that $\frac{dy}{d\tau}=-\frac{y}{\tau}>0$ in the interior region, 
The ingoing geodesic equation reads $\pa_\tau R = - e^{\mu-\l}$ and the interior region is characterised by the condition $e^{\mu(y)-\l(y)}>\sqrt\k y$. 
Let now $T=-\log(-\tau)$ for $\tau<0$. We then have $\frac{dy}{dT} =\frac{dy}{d\tau}\frac{d\tau}{dT}  =\left( \frac{R_\tau}{-\sqrt\k \tau} + \frac{R}{\sqrt\k \tau^2}\right) (-\tau)=-\frac{e^{\mu(y)-\l(y)}-\sqrt\k y}{\sqrt\k}.$
In particular $\frac{dy}{dT}<0$,  the right-hand side is smooth, and there are no fixed points of the above ODE on the interval $(0, y_{\mathcal N})$. We wish to show that the time $T$ it takes to reach $y=0$ is finite.
Integrating the above ODE, we see that 
\[
T(y)-T_0 = \int_{y_0}^y\frac{-\sqrt\k}{e^{\mu(\theta)-\l(\theta)}-\sqrt\k \theta}\, d\theta,
\]
where $T_0 = -\log|\tau_0|$, $y_0 = \frac{R_0}{-\sqrt\k \tau_0}$.
Note that $e^{-\l(y)}\asymp_{y\to0^+} \frac1{\tr'(y)}$ by~\eqref{E:LAMBDACONSTRAINTSS}. 
We next use~\eqref{E:TILDERREG} and $e^{\mu(0)}>0$ to conclude that,
as $y\to0^+$, the denominator inside the integral above asymptotes to a constant multiple $y^{-\frac2{3(1+\k)}}$. 
Since the latter is integrable near $y=0$, it follows that $\lim_{y\to0^+}T(y)<\infty$, as desired.

\noindent
{\em Proof of part (b)}. 
Let $(R_0,\tt_0)$ belong to the exterior region. 
%By~\eqref{E:GEQN3}, the ingoing RNG satisfies the equation 
%\begin{align}\label{E:GEQN1}
%\frac{dR}{d\tilde\tau} =- \frac1{e^{\tilde\l_\k(Y)-\tilde\mu_\k(Y)}+\frac{2\sqrt\k Y}{1+\k}}
%\end{align}
We consider the change of variables $\tilde\tau\mapsto Y$ and the particle label $R$ as a function of $Y$. 
By~\eqref{E:YDEF} and the geodesic equations~\eqref{E:GEQN3}, along any null-geodesic we have
\begin{align}
\frac{dY}{d\tilde\tau} & = \frac{d}{d\tilde\tau}\left(-\frac{\sqrt\k \tilde\tau}{R(\tt)}\right)  =  \frac{-\sqrt\k }{R(\tt)}  \left(1-\tt \frac1{R(\tt)}\frac{dR}{d\tt}\right) \notag\\
& = \frac{-\sqrt\k }{R(\tt)} \left(1 \pm \frac1{\sqrt\k} \frac Y{e^{\tilde\l(Y)-\tilde\mu(Y)}\mp\frac{2\sqrt\k Y}{1+\k}}\right) 
 =  -\frac{1}{R(\tt)} \frac{G_\pm(Y)}{ e^{\tilde\l(Y)-\tilde\mu(Y)}\mp\frac{2\sqrt \k Y}{1+\k}}, \label{E:CHANGEOFV}
\end{align} 
where
\be
G_\pm(Y):=\sqrt\k e^{\tilde\l(Y)-\tilde\mu(Y)} \pm \frac{1-\k}{1+\k}Y.
\ee
We note that by Remark~\ref{R:WELLDEFINEDGEODESICS} all the denominators appearing above are nonzero.
Therefore the geodesic equations~\eqref{E:GEQN3} transform into
\begin{align}
\frac{\pm1}{e^{\tilde\l(Y)-\tilde\mu(Y)}\mp\frac{2\sqrt\k Y}{1+\k}}& = \frac{dR}{dY}\frac{dY}{d\tt}
 = \frac{-1}{R(Y)}\frac{G_\pm(Y)}{ e^{\tilde\l(Y)-\tilde\mu(Y)}\mp\frac{2\sqrt \k Y}{1+\k}}\frac{dR}{dY}.
\end{align}
For as long as $R>0$ we can rewrite the above ODE in the form
\begin{align}\label{E:RYEQN}
\frac{d}{dY}(\log R) = \frac{\mp 1}{G_\pm(Y)}.
\end{align}
We remark that by Lemma~\ref{L:YN} the denominator is strictly positive in the exterior region and in the case of 
outgoing geodesics, function $G_+$ is in fact strictly positive for all $Y\in (Y_1,\infty)$ (positivity of $G_+$ characterises
the region ``below" $\mathcal B_1$).
%Assuming that the geodesic passes through a point $(R_0,\tt_0)$ in the {\em exterior region} 
%We integrate~\eqref{E:RYEQN} to conclude that (with $Y_\bullet:=-\frac{\sqrt\k \tt_0}{R_0}$)
%\begin{align}\label{E:LOGR}
%\log R(Y) - \log R_0 = \int_{Y}^{Y_\bullet} \frac{\mp 1}{\sqrt\k e^{\tl-\tm}\pm\frac{1-\k}{1+\k}Z}\, dZ
%\end{align}

For the ingoing geodesics, the exterior region is invariant by the flow. 
We integrate~\eqref{E:RYEQN} to conclude that (with $Y_\bullet:=-\frac{\sqrt\k \tt_0}{R_0}$)
\begin{align}\label{E:LOGR}
\log R(Y) - \log R_0 =- \int_{Y}^{Y_\bullet} \frac{1}{G_-(Z)}\, dZ.
\end{align}
The map $Z\mapsto G_-(Z)$ is smooth and by Lemma~\ref{L:YN} it is strictly positive on the interval $(\yms, Y_{\mathcal N})$. 
%As a consequence of~\eqref{E:FDEFY} we have the formula
%\begin{align}
%e^{2\tm(Y) - 2\tl(Y)}
%=\frac{(1+\k)^2}{(1-\k)^2}  \frac{(\RY Y^{-2})^{-\e} + \k \chi^2\left[(\vY_\k-1)^2Y^{-2} - 4\k \vY_\k\RY_\k Y^{-2}\right]}{(\vY_\k+\k)^2\chi_\k^2}.
%\end{align}
%In particular
%\begin{align}
%\lim_{Y\to 0^+} e^{2\tm(Y) - 2\tl(Y)} = \frac1{(1-\k)^2} \left(\frac{\Sigma_0^{-\e}}{\chi_0^2} -4\k^2\Sigma_0 \right) =
%\frac{\Sigma_0^{-\e}}{\chi_0^2(1-\k)^2} \left(1 + \k \chi_0^2\left( \vY_1^2 -4\k \Sigma_0^{1+\e}\right) \right) =: c_0^2 >0.
%\end{align}
%As $Z$ approaches $0$ from the right, the denominator inside the integral on the right-hand side of~\eqref{E:LOGR} 
%converges to $\frac{\sqrt\k}{c_0}$ and the integral $\int_{0}^{Y_\bullet} \frac{1}{\sqrt\k e^{\tl(Z)-\tm(Z)}-\frac{1-\k}{1+\k}Z}\, dZ$ is finite. 
Therefore, $R$ is positive as the ingoing RNG traverses $\mathcal B_1$ and $\mathcal B_2$, as the right-hand side of~\eqref{E:LOGR}
is finite for any $Y>\yms$.
%\[
%\sqrt\k e^{\tl(Y_i)-\tm(Y_i)} = -\frac{1-\k}{1+\k} Y_i, \ \ i=1,2,
%\]
%the integral on the right-hand side of~\eqref{E:LOGR} remains finite for all $Y\ge Y_2$ and the claim follows.

%Therefore, since 
%$y(\tau) = -\frac{R(\tau)}{\sqrt\k \tau}$, we obtain the bound
%%then reads
%$\pa_\tau R< - \sqrt\k y = \frac{R(\tau)}{\tau}$. This in turn gives
%\[
%hh
%\]
%\[
%R(\tau) = R_0 - \int_{\tau_0}^\tau e^{\mu(s)-\l(s)}\,ds
%\]
%The interior region is characterised by the condition $e^{\mu(y)-\l(y)}>\sqrt\k y$ and therefore
%\[
%R(\tau) \le R_0 - \sqrt \k \int_{\tau_0}^\tau s\,ds = R_0 - \frac{\sqrt\k}{2}(\tau-\tau_0)^2.
%\]

\noindent
{\em Proof of part (c)}. 
Let $(R_0,\ttau_0)$ belong to the exterior region. The outgoing null-geodesic solves the ODE
\[
\frac{d}{dY}(\log R) = \frac{- 1}{G_+(Y)}.
\]
Note that $G_+$ is smooth on $(Y_1,Y_{\mathcal N})$. Since $Y_1<0$ is the largest negative root of $Z \mapsto G_+(Z)$,  the right-hand side above is negative.
It follows that $Y\mapsto R(Y)$ is decreasing on $(Y_1,Y_{\bullet})$, i.e. $R(Y)$ increases as 
$Y$ approaches $Y_1$ from the right.
In particular, the solution exists for all $Y\in (Y_1, Y_{\bullet}]$ by the strict negativity of the right-hand side above. 
We consequently have the formula
\begin{align}\label{E:ROUTGOING}
R(Y) = R_0 \exp\left( \int_{Y}^{Y_\bullet} \frac{1}{G_+(Z)}\, dZ\right), \ \ Y>Y_1.
\end{align}
Since the function $G_+$ is positive for $Z>Y_1$,
if $G(Z)=C(Z-Y_1)^m(1+O(|Z-Y_1|))$ is the first term of Taylor expansion of $G$ at $Y_1$, with $m\in\mathbb N$ (recall that $G$ is real analytic), then necessarily $C>0$. Plugging this expansion into~\eqref{E:ROUTGOING}
and integrating we conclude $\lim_{Y\to Y_1^+}R(Y)=\infty$, which shows that the outgoing null-geodesic asymptotes to $\mathcal B_1$. Since $Y= - \frac{\sqrt{\k}\ttau(Y)}{R(Y)}$ and
 $\lim_{Y\to (Y_1)^+}R(Y)=\infty$, it follows that $\lim_{Y\to (Y_1)^+}\ttau(Y)=\infty$, and therefore the outgoing geodesic exists globally on $[\ttau_0,\infty)$ and $\lim_{\ttau\to\infty}R(\ttau)=\infty$.
 If $(R_0,\ttau_0)$ belongs to the interior region, a similar analysis in the original comoving coordinates $(R,\tau)$ yields the same conclusion.

Since the function $G_+(Z)$ is smooth and bounded in a neighbourhood of $Z=0$, it follows from~\eqref{E:ROUTGOING} that any outgoing geodesic starting at $(R_0,\ttau_0)$ with $\ttau_0>0$  intersects $\{\ttau=0\}$ axis (i.e. $Y=0$)
at a positive value of $R$ to the past. Due to the monotonicity of the flow~\eqref{E:ROUTGOING}, any outgoing geodesic emanating from $(R_0,\ttau_0)$ with $\ttau_0\le0$ remains below $\{\ttau=0\}$ axis to the past.
\end{proof}

%%%%%%%%%%%%%%%%%%%%%%%%%%%%%%%%%%%
%%%%%%%%%%%%%%%%%%%%%%%%%%%%%%%%%%%

\section{Asymptotic flattening of the self-similar profile}\label{S:DOUBLENULL}

The key result of this section is the local well-posedness for the characteristic initial value problem
for the Einstein-Euler system, see Theorem~\ref{thm:LWP}. The idea is to suitably truncate the
self-similar spacetime as described in Section~\ref{SS:DOUBLENULLINTRO}. We work with the double-null formulation, see Section~\ref{SS:DNFORM},
%In order to address the evolution of the isothermal gas, 
and our starting point is the reformulation of the fluid
evolution equations~\eqref{E:CONTNULL13}--\eqref{E:MOMENTUMNULL13}.
%\subsection{Spherically symmetric Einstein-Euler system in double null coordinates}

%%%%%%%%%%%%%%%%%%%%%%%%%%%%%%%%%%%
%%%%%%%%%%%%%%%%%%%%%%%%%%%%%%%%%%%

%%%%%%%%%%%%%%%%%%%%%%%%%%%

\subsection{Reformulation of the fluid evolution and the effective transport velocity}

%In order to address the solvability of~\eqref{E:CONTNULL13}--\eqref{E:MOMENTUMNULL13}, we shall 
%first reformulate as a more convenient set of transport equations. 
We introduce the constant 
\be\label{(2.93)}
k_\pm : = 1\pm \frac{\sqrt{2\e +\e^2}}{1+\e},
\ee 
where we recall $\e=\e(\k)$ is given by~\eqref{E:ETADEF} and from~\eqref{(2.93)} it is clear 
that $k_\pm = 1\pm O(\sqrt{\e})$. 

%%%%%%%%%%%%%%%%%%%%%%%%%%%%
%%%%%%%%%%%%%%%%%%%%%%%%%%%%

\begin{lemma}[Reformulation of the Euler equations]\label{L:EULERREFORM}
Assume that $(\rho,u^\nu,r,\Omega)$ is a $C^1$ solution of~\eqref{E:CONTNULL13}--\eqref{E:MOMENTUMNULL13}. 
Let
\begin{align}
%(1+\e)^2\frac{\Om^2}{r^{2\e}}\frac{V}{U} &= \frac{(1+\e)}{ \Om^2 (u^\u)^2} =:
 \mathcal U &:=(1+\e) \Om^2 (u^\v)^2, \label{E:TRANSPORTVELDEF}\\
%\frac{V^{k_+}}{U} &=  \frac{(1-\k)^{k_+}}{1+\k} \frac{r^{(2+2\e)k_+-2}\rho^{k_+-1}}{\Om^4 (u^\u)^2}=:
 f^+&: = % \frac{(1-\k)^{k_+}}{1+\k} \frac{r^{(2+2\e)k_+-2}\rho^{k_+-1}}{\Om^4 (u^\u)^2} \\
 \frac{(1-\k)^{k_+}}{1+\k} {r^{(2+2\e)k_+-2}\rho^{k_+-1}} (u^\v)^2,  \label{E:FPLUSDEF}\\ 
%\frac{U}{V^{k_-}} &= \frac{1+\k}{(1-\k)^{k_-} } \Om^4 (u^\u)^2 r^{2-(2+2\e)k_-} \rho^{1-k_-} =: 
f^-&: = \frac{1+\k}{(1-\k)^{k_-} } \frac{ r^{2-(2+2\e)k_-} \rho^{1-k_-}}{(u^\v)^2}.\label{E:FMINUSDEF}
\end{align}
%The fluid equations~\eqref{E:CONTNULL1}--\eqref{E:MOMENTUMNULL1} 
Then the new unknowns $f^\pm$ 
%can be rewritten in the from 
satisfy
\begin{align}
\pa_\u f^+ + k_+ \mathcal U \pa_\v f^+ + 2k_+ (2\frac{\pa_\v \Om}{\Om} - 2\e\frac{\pa_\v r}{r})  \mathcal U f^+  & =0, \label{E:FPLUSEVOLUTION}\\
\pa_\u f^- + k_- \mathcal U \pa_\v f^- - 2k_- (2\frac{\pa_\v \Om}{\Om} - 2\e\frac{\pa_\v r}{r})  \mathcal U f^-  & =0. \label{E:FMINUSEVOLUTION}
\end{align}
\end{lemma}

%%%%%%%%%%%%%%%%%%%%%%%%%%%%

\begin{proof}
%Recall \eqref{E:CONTNULL12} and \eqref{E:MOMENTUMNULL12}
%\begin{align}
%\pa_\u (\Om^4r^2T^{\u\u}) + \frac{\Om^2}{r^{2\e}} \pa_\v ( \Om^2r^{2+2\e} T^{\u\v}  ) &=0 
% \label{E:CONTNULL13}\\
%\pa_\u ( \Om^2r^{2+2\e} T^{\u\v}  )  + \frac{r^{2\e}}{\Om^2}\pa_\v (\Om^4r^2T^{\v\v}) & = 0
%\label{E:MOMENTUMNULL13}
%\end{align}
Let $U:= \Om^4r^2T^{\u\u}$ and $V:=  \Om^2r^{2+2\e} T^{\u\v} $. Using~\eqref{E:TCONSTRAINT} we rewrite 
$\Om^4r^2T^{\v\v}= (1+\e)^2 \frac{\Om^4}{r^{4\e}} \frac{V^2}{U}$.
% by using $T^{\u\u} T^{\v\v}
%= (1+\e)^2(T^{\u\v})^2$, we may rewrite the above 
Therefore,~\eqref{E:CONTNULL13}--\eqref{E:MOMENTUMNULL13} can be rewritten in the form
\be\label{E:FMINUSEVOLUTION1}
\pa_\u U +\frac{\Om^2}{r^{2\e}} \pa_\v V =0, \quad \pa_\u V +(1+\e)^2  \frac{r^{2\e}}{\Om^2}\pa_\v ( \frac{\Om^4}{r^{4\e}} \frac{V^2}{U} )=0.
\ee
For any $k\in\mathbb R$ we now compute $\pa_p\left(\frac{V^k}{U}\right)$ and thereby use~\eqref{E:FMINUSEVOLUTION1}: 
\begin{align}
\pa_\u (\frac{V^k}{U}) &= k \frac{V^{k-1}}{U} \pa_\u V - \frac{V^k}{U^2}\pa_\u U \notag \\
&=- k (1+\e)^2 \frac{\Om^2}{r^{2\e}} \frac{V}{U} \pa_\v (\frac{V^k}{U}) - 2k (1+\e)^2 \frac{V^{k+1}}{U^2}\pa_\v(\frac{\Om^2}{r^{2\e}})\notag \\
&\quad - k (2-k) (1+\e)^2 \frac{\Om^2}{r^{2\e}} \frac{V^k}{U^2} \pa_\v V + \frac{V^k}{U^2} \frac{\Om^2}{r^{2\e}} \pa_\v V . \label{(2.79)}
\end{align}
We see that  the last line \eqref{(2.79)} vanishes if $k$ is a solution of the quadratic equation
\be\label{(2.92)}
k^2 - 2k+ \frac{1}{(1+\e)^2} =0.
\ee
The two distinct roots $k_\pm$ of~\eqref{(2.92)} are given in~\eqref{(2.93)}, 
and the equation for $\frac{V^{k_{\pm}}}{U}$ reads
\[
\pa_\u (\frac{V^{k_\pm}}{U}) + k_\pm (1+\e)^2 \frac{\Om^2}{r^{2\e}} \frac{V}{U} \pa_\v (\frac{V^{k_\pm}}{U}) +  
2k_\pm (1+\e)^2 \frac{V^{k_\pm+1}}{U^2}\pa_\v(\frac{\Om^2}{r^{2\e}}) = 0 . 
\]
When $k=k_+>1$, we keep $\frac{V^{k_+}}{U}$ as the unknown.
When however $k=k_-<1$, we work with $\frac{U}{V^{k}}$ instead, to avoid singularities for small values of $\rho$.
From above we obtain the equation
\[
\pa_\u (\frac{U}{V^{k_-}}) + k_- (1+\e)^2 \frac{\Om^2}{r^{2\e}} \frac{V}{U} \pa_\v (\frac{U}{V^{k_-}}) - 2k_- (1+\e)^2 V^{1-k_-}\pa_\v(\frac{\Om^2}{r^{2\e}}) = 0. 
\]
Going back to original variables, note that $\frac{V}{U}= \frac{r^{2\e}}{(1+\e) \Om^4 (u^\u)^2}$ and 
$\frac{V^k}{U} = \frac{(1-\k)^k}{1+\k} \frac{r^{(2+2\e)n-2}\rho^{k-1}}{\Om^4 (u^\u)^2}$ so that 
\begin{align*}
(1+\e)^2\frac{\Om^2}{r^{2\e}}\frac{V}{U} &= \frac{(1+\e)}{ \Om^2 (u^\u)^2} =(1+\e) \Om^2 (u^\v)^2= \mathcal U, \\
\frac{V^{k_+}}{U} &=  \frac{(1-\k)^{k_+}}{1+\k} \frac{r^{(2+2\e)k_+-2}\rho^{k_+-1}}{\Om^4 (u^\u)^2} = f^+, \\
\frac{U}{V^{k_-}} &= \frac{1+\k}{(1-\k)^{k_-} } \Om^4 (u^\u)^2 r^{2-(2+2\e)k_-} \rho^{1-k_-} = f^-,
\end{align*} 
where we recall~\eqref{E:TRANSPORTVELDEF}--\eqref{E:FMINUSDEF}.
\end{proof}

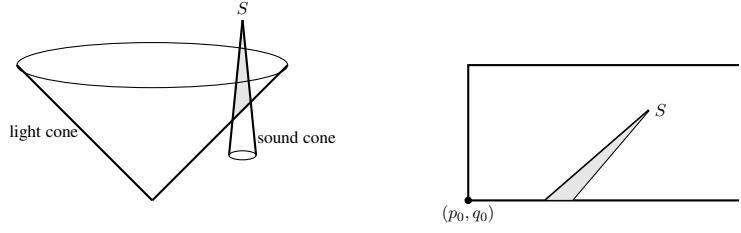
\begin{figure}

\begin{center}
\begin{tikzpicture}
%[domain=0:5, scale = 0.6]
\begin{scope}[scale=0.6, transform shape]

% Tip of the sound cone
\coordinate [label=above:$S$](S) at (2,4);

% lower left tip of the sound cone
\coordinate (S1) at (1.7,1);

% lower right tip of the sound cone
\coordinate (S2) at (2.3,1);

% fill in gray the sound cone inside the cone picture
\draw [fill=gray!20, dashed] (S) -- (S1) -- (S2);

%origin
\coordinate [] (A) at (0,0){};

%% right horizontal axis coordinate
\coordinate [] (B) at (3,0){};

%% left horizontal axis coordinate
\coordinate (C) at (-3,0){};

%% upper vertical axis coordinate
\coordinate (D) at (0,3){};

% lower vertical axis
%\coordinate [label=left:$r=0$] (G) at (-2,-4){};

% Massive singularity
%\coordinate [label=above:$\MS$] (C) at (-2,10){};

% \mathcal B_2 curve
%\coordinate [label=right:$\mathcal B_2$] (D) at (-1,10){};

% denote \mathcal B_1 curve
%\coordinate [label=above:$\mathcal B_1$] (E) at (3,3.2){};

% sonic line
%\coordinate [label=below:sonic line] (H) at (-2,4.1){};

% Upper - right tip of the light cone
\coordinate (C1) at (3,3);

% Upper - left tip of the light cone
\coordinate (C2) at (-3,3);

%\node at (-2,-3) {$r=0$};

% fill in white triangle below the right side of the light cone 
\draw [fill=white,white] (A) -- (B) -- (C1);

% label light cone
\node at (-2.4,1.5) {light cone};

% draw upper right line of the cone
\draw[thick] (A)--(C1);

% draw upper left line of the cone
\draw[thick] (A)--(C2);

% draw lower left tip of the sound cone
\draw[thick] (S)--(S1);

% draw lower right tip of the sound cone
\draw[thick] (S)--(S2);

% label sound cone
\node at (3.2,1.4) {sound cone};

 % drawing the region D version of the picture
 
% point (p_0,q_0)
\coordinate [](L0) at (7,0);

% denote (p-0,q_0)
\node at (7,-0.3) {$(\u_0,\v_0)$};

% (p_0,q_0)
\draw[fill=black] (7,0) circle (2pt);

% point (p_0,q_0+6);
\coordinate (L1) at (13,0);

% point (p_0+3, q_0);
\coordinate (L2) at (7,3);

% point (p_0+3, q_0+6);
\coordinate (L3) at (13,3);

% draw D
\draw [thick] (L1) -- (L0) -- (L2) -- (L3);

% tip of the sound cone in D
\coordinate [label=right:$S$] (SD0) at (11,2){};

% left intersection of the sound cone
\coordinate [] (SD1) at (8.7,0);

% right intersection of the sound cone
\coordinate [] (SD2) at (9.3,0);

% draw left portion of the backward sound cone from SD0
\draw [thick] (SD0) -- (SD1);

% draw right portion of the backward sound cone from SD0
\draw [thick] (SD0) -- (SD2);

% fill in gray the sound cone inside D
\draw [fill=gray!20] (SD0) -- (SD1) -- (SD2);

\end{scope}

\begin{scope}[scale=0.6, transform shape]
% Draw the light cone ellipse.
  \draw[] (0, 3) ellipse (3 cm and 0.5 cm);
\end{scope}

\begin{scope}[scale=0.6, transform shape]
% Draw the sound cone ellipse.
  \draw[] (2, 1) ellipse (0.3 cm and 0.1 cm);
\end{scope}

 \end{tikzpicture}
   \end{center}
    \caption{The grey shaded area in the infinite rectangular region $\D$ is a schematic depiction of the region bounded by the backward fluid characteristics emanating from a point $S\in\mathcal D$. The opening angle is of order $\sqrt\k\ll1$, which of course is precisely the speed of sound.}
\label{F:SOUNDVERSUSLIGHT}
\end{figure}

%%%%%%%%%%%%%%%%%%%%%%%%%%%%
%%%%%%%%%%%%%%%%%%%%%%%%%%%%

\begin{remark} \label{R:OLDNEWFLUID}
It follows from~\eqref{E:FPLUSDEF}--\eqref{E:FMINUSDEF} that 
%Since $f^+ f^- = \left(  (1-\k) r^{2+2\e} \rho  \right)^{k_+ - k_-}$ we have the following expression for $\Om^2T^{\u\v}$: 
\be
 %\Om^2 T^{\u\v} = (1-\k) 
 \rho =\frac1{1-\k} \frac{ (f^+  f^-)^{\frac{1}{k_+-k_-}}}{r^{2+2\e}}. \label{2.99}
\ee 
Moreover, from~\eqref{E:NORMALISATIONNULL} and~\eqref{E:FPLUSDEF}--\eqref{E:FMINUSDEF} we have
 $\frac{f^+}{f^-} = \frac{1}{(1+\e)^2} r^{4\e} (u^\v)^4$, which leads to the relation 
\be\label{2.100}
 \mathcal U =(1+\e)\Om^2 (u^\v)^2= (1+\e)^2 \frac{\Om^2}{r^{2\e}} \left( \frac{f^+}{f^-} \right)^\frac12.  
\ee  
\end{remark}

%%%%%%%%%%%%%%%%%%%%%%%%%%%%
%%%%%%%%%%%%%%%%%%%%%%%%%%%%

\subsection{Statement of the local well-posedness theorem}\label{SS:LWPSTATEMENT}

%\subsubsection{Data along the characteristic cones and compatibility}

In order to flatten the the $\grlp$ metric at asymptotic infinity, we shall treat the system~\eqref{E:RWAVE1}--\eqref{E:OMEGAWAVE1} and~\eqref{E:FPLUSEVOLUTION}--\eqref{E:FMINUSEVOLUTION} as the evolutionary part
and equations~\eqref{E:CONSTRAINTU1}--\eqref{E:CONSTRAINTV1} as the constraints. 

\subsubsection{Fixing the choice of double-null coordinates}

We now fix a choice of double-null coordinates which will then be used to dampen the tails of the solution and produce an asymptotically flat spacetimes containing a naked singularity. 
%To that end, let $\mathcal N\subset \MRLP$ be the ingoing curve corresponding to the  boundary of the past of the scaling origin $\mathcal O$.
Let $(\tilde\tau_0, R_0)\in \DRLPtilde$ be a given point in the exterior region (see Definition~\ref{D:EXTERIOR} and~\eqref{E:MATHCALDTILDE} for the definition of $\DRLPtilde$).
Through $(\tilde\tau_0, R_0)$ we consider the ingoing null-curve which intersects the outgoing null-curve $\mathcal B_1$ 
(given by $R = \frac{\sqrt\k}{|Y^1|} \tilde\tau$) 
%= \sqrt\k |y^1| \tau$) 
at some $(R_1,\ttau_1)$ where by Lemma~\ref{L:INCOMING}, $R_1>0$. We then fix the  null-coordinate $\u$ by 
demanding that 
%$\u = - 2(r-r_1)$ along the ingoing null-geodesic going through $(\tilde\tau_0,R_0)$, where
\be\label{E:NORMIN}
\u = - 2(r-r_1) \ \ \text{along the ingoing null-geodesic  through $(R_0, \tilde\tau_0)$, } \ \ 
r_1 : = \tr(R_1,\ttau_1);
\ee
% and extending to the interesection massive singularity ;
and demanding that the level sets of $\u$ correspond to outgoing null-geodesics.
The choice~\eqref{E:NORMIN} normalises $\mathcal B_1$ to correspond to the hypersruface $\{\u = 0\}$ in the RLP-spacetime. Note that in the RLP spacetime, 
the ingoing curve through $(\tilde\tau_0, R_0)$ terminates at the massive singularity $\MS$ to the future. 
%and extends indefinitely to the past by Lemma~\ref{L:INCOMING}.

Let now the outgoing null geodesic through $(\tilde\tau_0, R_0)$ intersect the surface $\mathcal N$ (boundary of the past of the scaling origin $\mathcal O$) at $(R_\ast,\ttau_\ast)$. We then fix the null coordinate $\v$ by 
demanding that 
%$\v =  2(r-r_\ast)$ along the outgoing null-geodesic going through $(\tilde\tau_0,R_0)$, where
\be\label{E:NORMOUT}
\v =  2(r-r_\ast) \ \  \text{along the outgoing null-geodesic through $(\tilde\tau_0, R_0)$,} \ \ r_\ast : = \tr(R_\ast,\ttau_\ast);
\ee
% and extending to the interesection massive singularity ;
and demanding that the level sets of $\v$ correspond to ingoing null-geodesics.
The normalisation~\eqref{E:NORMOUT} makes the surface $\mathcal N$ correspond to the hypersurface $\{\v = 0\}$.
A more detailed description of the RLP-spacetime in this double-null gauge is given in Lemma~\ref{L:RLPDN}.

Let $(\u_0,\v_0)$ be the point $(\ttau_0,R_0)$ in the above double-null gauge. 
We shall consider the seminfinite rectangular domain 
%We wish to solve the system of \eqref{E:FPLUSEVOLUTION}, \eqref{E:FMINUSEVOLUTION}, \eqref{E:RWAVE13},  \eqref{E:OMEGAWAVE13} in the rectangular domain 
\be
\mathcal D : =\{(\u,\v): \u_0 < \u < 0, \ \  \v>\v_0 \},
\ee
where $|\u_0|>0$ is sufficiently small, with data prescribed on the set 
$
\underline{\mathcal C} \cup \mathcal C,
$
where 
\begin{align}
\underline{\mathcal C} &: = \left\{(\u,\v) \ \Big| \  \v=\v_0, \ \ \u\in[\u_0,0]\right\}, \\
\C & : = \left\{(\u,\v) \ \Big| \ \u=\u_0, \ \ \v\in[\v_0,\infty)\right\}
\end{align} 
correspond to the
ingoing and the outgoing null-curves emanating from $(\u_0,\v_0)$ respectively. See Figure~\ref{F:CUTOFF}.

%%%%%%%%%%%%%%%%%%%%%%%%%
%%%%%%%%%%%%%%%%%%%%%%%%%
%%%%%%%%%%%%%%%%%%%%%%%%%

\subsubsection{Norms and local well-posedness}

%%%%%%%%%%%%%%%%%%%%%%%%%
%%%%%%%%%%%%%%%%%%%%%%%%%
%%%%%%%%%%%%%%%%%%%%%%%%%

Fix $N_\pm>0$ so that  
\be\label{2.122}
N_+ = N_- - 4\e , \quad N_- > \frac{k_+-k_-}{2} + 2\e  
\ee
and let $0<\theta\ll 1$ be a small  fixed constant. 
For any $f^\pm \in W^{2,\infty}(\overline{\mathcal D}) $ and $(\Om,\frac{r}{\v})\in W^{3,\infty}(\overline{\mathcal D})$, we define the total norm by 
\begin{align}\label{E:TRIPLENORM}
\vertiii{[f^\pm, \Om, r]}&:= \vertiii{[f^\pm]} + \vertiii{[\Om, r]}
\end{align}
where 
\begin{align}
 \vertiii{[f^\pm]}&:=\|\log (\v^{N_+}  f^+) \|_{\infty} +\sum_{j=1}^2\|\v^{j} \pa_\v^j \log f^+ \|_{\infty}\\
 &\quad+\|\log (\v^{N_-}  f^-) \|_{\infty} + \sum_{j=1}^2\|\v^{j} \pa_\v^j \log f^- \|_{\infty},\notag %\|\log (\v^{N_\pm}  f^\pm) \|_{\infty}+ \sum_{j=1}^2\|\v^{j} \pa_\v^j \log f^\pm \|_{\infty}
 \end{align}
 \begin{align}
 \vertiii{[\Om, r]} &=
\|\log \Om\|_\infty + \sum_{j=1}^3  \|\v^{j+\theta}\pa_\v^j \log \Om \|_{\infty} \\
&\quad+\sum_{j=0}^3 \|\v^{j-2} \pa_\v^j (r^2) \|_{\infty}   %\|\log (\frac{r}{\v})\|_\infty +\|\log \pa_\v r \|_\infty+ \sum_{j=2}^3 \|\v^{j-2} \pa_\v^j (r^2) \|_{\infty} 
 %\sum_{j=0}^3 \|\v^{j-2} \pa_\v^j (r^2) \|_{\infty}  \notag
+\left\| \frac{\v^2}{r^2}\right\|_\infty + \left\| \frac{\v}{\pa_\v(r^2)}\right\|_\infty,%+ \sum_{j=0}^1\|\v^{j-1}\pa_\v^j\pa_\u (r^2)\|_\infty
\end{align}
and the data norm 
\begin{align*}
\vertiii{[f^\pm, \Om, r]|_{\uC \cup \C}}&:=\vertiii{[f^\pm]|_{\uC \cup \C}}+\vertiii{[\Om, r]|_{\uC \cup \C}} \\
&:=\|\log(\v^{N_+}  f^+) |_{\uC \cup \C}\|_{\infty}+ \sum_{j=1}^2\|\v^{j} \pa_\v^j \log f^+ |_{\uC \cup \C}\|_{\infty}\\
&\quad+\|\log(\v^{N_-}  f^-) |_{\uC \cup \C}\|_{\infty}+ \sum_{j=1}^2\|\v^{j} \pa_\v^j \log f^- |_{\uC \cup \C}\|_{\infty}\\
%\sum_{j=0}^2\|\v^{N_\pm+j} \pa_\v^j f^\pm|_{\pa\mathcal D} \|_{\infty} 
&\quad+\|\log \Om |_{\uC \cup \C} \|_\infty + \sum_{j=1}^3 \|\v^{j+\theta}\pa_\v^j \log \Om |_{\uC \cup \C} \|_{\infty} \notag\\%+\sum_{j=0}^3  \|\v^{j-2} \pa_\v^j (r^2) |_{\uC \cup \C} \|_{\infty}  \notag
&\quad%+ \|\log (\frac{r}{\v}) |_{\uC \cup \C}\|_\infty +\|\log \pa_\v r |_{\uC \cup \C} \|_\infty
+ \sum_{j=0}^3 \|\v^{j-2} \pa_\v^j (r^2)  |_{\uC \cup \C}\|_{\infty}+\left\| \frac{\v^2}{r^2}\Big|_{\uC \cup \C} \right\|_\infty + \left\| \frac{\v}{\pa_\v(r^2)}\Big|_{\uC \cup \C} \right\|_\infty\\
&\quad +\|\v^{-1}\pa_\u( r^2) |_{\uC }\|_\infty+ \left\| \frac r2 \left(1+  \frac{4\pau r \pav r}{\Omega^2}\right)\Big|_{\uC} \right\|_\infty\notag
\end{align*}

We choose the data $(\hat f^\pm,\hat r,\hat \Om)$ on $\uC$ to coincide with the corresponding data obtained by restricting the RLP-solution to $\uC$. 
Let $A_0>q_0$ be a real number to be specified later. 
\begin{itemize}
\item[(i)]
The data on $\C$ is chosen so that  $[\hat{f}^\pm,\hat{\Om},\hat{r}]$ coincide with the exact self-similar $\grlp$ solution
on the segment $\{(\u,\v) \ \Big| \u = \u_0, \ \v\in[q_0,A_0]\}$;
\item[(ii)]  $\hat{f}^\pm\in W^{2,\infty}(\C\cup\uC)$, and $\hat\Om, \frac{\hat r}{\v}\in  W^{3,\infty}(\uC \cup \C)$ with $\vertiii{[f^\pm, \hat\Om, \hat r]|_{\uC \cup \C}} <\infty$ and 
$\hat f^\pm, \hat\Om, \hat r>0$;
\item[(iii)] the constraint equations~\eqref{E:CONSTRAINTV1}--\eqref{E:CONSTRAINTU1} hold on $\C$ and $\uC$ respectively.
%the Hawking mass $m(\u_0,\v)$ is bounded by a constant uniformly in $\v$, for all $\v\in[\v_0,\infty)$.
\end{itemize}
%We observe that by the assumption (i), the compatibility conditions at $(\u_0,\v_0)$ are automatically satisfied to arbitrary order.
%We consider the data $[\hat{f}^\pm,\hat{\Om},\hat{r}]$ on $\uC \cup \C$ satisfying 
%\begin{enumerate}
%\item[(A1)]   $\hat{f}\in  W^{2,\infty}(\uC \cup \C)$ and $\hat\Om, \frac{\hat r}{\v}\in  W^{3,\infty}(\uC \cup \C)$ with $\vertiii{[\hat f^\pm, \hat\Om, \hat r]|_{\uC \cup \C}} <\infty$ and $\hat f^\pm, \hat\Om, \hat r>0$,
%\item[(A2)] The constraint equations~\eqref{E:CONSTRAINTV1}--\eqref{E:CONSTRAINTU1} hold on $\C$ and $\uC$ respectively.
%%\item[(A3)] $\v^{N^\pm} \hat f_{\pm,0}\sim 1$, \  $\hat\Om\sim 1$, \ $\frac{\hat r}{\v}\sim 1$, \ $\pa_\v \hat r \sim 1$ %there exists a constant $C>1$ such that $\frac{1}{C}<\v^{N^\pm} \hat f_{\pm,0}, \Om, <C$ for all $\v>\v_0$
%\end{enumerate}
Due to the choice of the double-null gauge~\eqref{E:NORMIN}--\eqref{E:NORMOUT}, the metric coefficient $\hat r$ is determined along $\C$. In order to impose the constraint~\eqref{E:CONSTRAINTV1}
we can solve it for  $\hat\Om$. By~\eqref{E:NORMOUT} we have $\pa_\v r=\frac12$ and by~\eqref{E:TPP} and Remark~\ref{R:OLDNEWFLUID} we may rewrite~\eqref{E:CONSTRAINTV1} as
%via
%\[
%- \pi r \Om^2 T^{\u\u} = - \pi(1+\e)^2 r^{-1-4\eta} (f^+f^-)^{\frac1{k_+-k_-}}\left(\frac{f^+}{f^-}\right)^{\frac12} \Omega^2.
%\]
%on $\C$ and  
%therefore~\eqref{E:CONSTRAINTV1} reduces to
\begin{align}
\frac12\pa_\v(\Omega^{-2}) = - \frac{\pi}{r}(f^+f^-)^{\frac1{k_+-k_-}}\left(\frac{f^-}{f^+}\right)^{\frac12} \Omega^{-2}.
\end{align}
Moving the copy of $\Omega^{-2}$ to the left-hand side and then integrating, we obtain the formula
%or alternatively
%\be
%\pa_\v(\Omega^{-4}) = - \pi(1+\e)^2 (2\v+c)^{-1-4\eta} (f^+f^-)^{\frac1{k_+-k_-}}\left(\frac{f^+}{f^-}\right)^{\frac12} .
%\ee
%We now see that $\Omega$ will be order $1$ at future null-infinity precisely when the right-hand side is integrable and this will additionally determine $\Omega(\u_0,\v)$ through the formula
\begin{align}\label{E:OMEGAOUTGOING}
\Omega(\u_0,\v) = \Omega(\u_0,\v_0) \exp\left( \int_{\v_0}^{\v}  \frac{2\pi}{s+2r_\ast} (f^+f^-)^{\frac1{k_+-k_-}}\left(\frac{f^-}{f^+}\right)^{\frac12} \,ds\right).
\end{align}

%%%%%%%%%%%%%%%%%%%%%%%%%%%%%%%%%%

\begin{remark}[The Hawking mass]
We recall the Hawking mass introduced in~\eqref{E:HAWKINGCOMOVING}. It can 
be alternatively expressed via the formula
\begin{align}\label{E:HAWKINGMASSDEF}
m = \frac r2 \left(1+  \frac{4\pau r \pav r}{\Omega^2}\right).
\end{align}
Using~\eqref{E:RWAVE1}--\eqref{E:CONSTRAINTU1} one can show that for classical solutions of the Einstein-Euler system we have the identities
\begin{align}
\pau m & = 2\pi r^2 \Om^{2}\left(T^{\u\v}\pau r - T^{\v\v}\pav r\right), \\
\pav m & = 2\pi r^2 \Om^{2}\left(T^{\u\v}\pav r - T^{\u\u}\pau r\right), \label{E:PAVM}
\end{align}
see for example Section 1.2 of~\cite{Da2005}. 
Using~\eqref{E:TPP} we may  rewrite the right-hand side of~\eqref{E:PAVM} as $2\pi r^2 (1-\k)\rho\left(\pa_\v r - (1+\e)\Om^2 (u^{\u})^2 \pa_{\u} r\right)$. Integration 
along a constant $\u$-slice then gives
\begin{align}\label{E:HM}
m(\u,\v) = m(\u,\v_0) + 2\pi (1-\k) \int_{\v_0}^\v r^2 \rho \left(\pa_\v r - (1+\e)\Om^2 (u^{\u})^2 \pa_{\u} r\right) \, ds.
\end{align}
%We also note that $\Om^2 T^{\u\v}=(1-\k)\rho$ is given in terms of $f^\pm$ and $r$ via~\eqref{2.99}. Using~\eqref{E:NORMALISATIONNULL} and~\eqref{2.100}, we also have the expression
%\begin{align}
%\Om^2 (u^{\u})^2 = \frac1{1+\e}\frac{r^{2\e}}{\Om^2} \left(\frac{f^-}{f^+}\right)^{\frac12}.
%\end{align}
%
%Therefore, after we integrate~\eqref{E:PAVM} along a constant $\u$-slice, we obtain
%\begin{align}
%m(\u,\v) = m(\u,\v_0) + \int_{\v_0}^\v 
%\end{align}
\end{remark}

We now state the main local existence and uniqueness theorem for the characteristic problem described above.

\begin{theorem}\label{thm:LWP} There exist sufficiently small $\delta_0>0$ such that for any $\delta\in (0,\delta_0)$ and $\u_0=-\delta$, with initial boundary data satisfying {\em (i)-(iii)} above,
 there exists a unique asymptotically flat solution $[f^\pm,\Om,r]$ to the system 
\eqref{E:RWAVE1}--\eqref{E:OMEGAWAVE1} and \eqref{E:FPLUSEVOLUTION}--\eqref{E:FMINUSEVOLUTION}, 
with $f^\pm\in  W^{2,\infty}(\overline{\mathcal D})$ , $[\Om,\frac{r}{\v}]\in W^{3,\infty}(\overline{\mathcal D})$, and 
%\eqref{E:RWAVE12},  \eqref{E:OMEGAWAVE12} 
such that $\vertiii{[f^\pm, \Om, r]}<\infty$. 
Moreover, this solution is a solution of the original system~\eqref{E:RWAVE1}--\eqref{E:MOMENTUMNULL13}.
\end{theorem}

We shall prove Theorem~\ref{thm:LWP} by the method of characteristics. Our strategy is to first solve the fluid evolution equations for $f^\pm$ given the effective fluid velocity $\mathcal U$ and the metric components $\Om,r$. 
Then we will feed that back into the wave equations~\eqref{E:RWAVE1}--\eqref{E:OMEGAWAVE1} for the metric components to obtain the bounds for $[r,\Om]$. To make this strategy work, in Section~\ref{SS:CHARACTERISTICS}
we carefully look at the characteristics associated with the fluid evolution. After collecting some preparatory a priori bounds in Section~\ref{SS:APRIORIPDE}, in Section~\ref{SS:LWPPDE} we use an iteration scheme to conclude the 
proof of Theorem~\ref{thm:LWP}.
%Theorem \ref{thm:LWP} will be proved by iteration methods. For given $\mathcal U=\mathcal U_n$, $\Om=\Om_n$, $r=r_n$, we solve the Euler part $f^\pm_{n+1}$ of \eqref{E:FPLUSEVOLUTION}, \eqref{E:FMINUSEVOLUTION} by the method of characteristics and solve the metric part $\Om_{n+1}$, $r_{n+1}$ of \eqref{E:RWAVE12}, \eqref{E:OMEGAWAVE12} by $\u,\v$ integrations. 

\subsection{Characteristics for the fluid evolution}\label{SS:CHARACTERISTICS}

Let $\mathfrak{\v}_\pm(s)=\mathfrak{\v}_\pm (s; \u,\v)$ be the backward characteristics associated with the speeds $k_\pm \mathcal U$ such that 
\begin{align}
\frac{d \mathfrak{\v}_\pm}{ds}(s; \u,\v)& = k_\pm \mathcal U (s,  \mathfrak{\v}_\pm(s;\u,\v)), \ \  s<\u,  \label{2.103}  \\
\mathfrak{\v}_\pm (\u; \u,\v) &= \v .  \label{2.104}
\end{align} 

Since our solutions as well as $\mathcal U$ reside in the domain $\mathcal D$ and the boundary, we track the backward characteristics $(s, \mathfrak{\v}_\pm(s))$ until they leave the domain. In the next lemma, we show the existence and regularity properties of $\mathfrak{\v}_\pm(s)$ and exit time $\u_*=\u_*(\u,\v)$ and position $\v_*=\v_*(\u,\v)$. %We denote $\mathcal U \cup  \C   \cup \pa \mathcal D_{\text{in}}$ by $\overline{\mathcal D}$. 
%This leads to the following definition of the backward exit time and position $(\u_*,\v_*)$: 
 %For any given $(\u,\v) \mathcal D$, the backward exit time $\u_*=\u_*(\u,\v)$ and position $\v_*= \mathfrak{\v}_\pm (\u_*)$. 

%%%%%%%%%%%%%%%%%%%%%%%%%%%%%%%
%%%%%%%%%%%%%%%%%%%%%%%%%%%%%%%

\begin{lemma}\label{lem:Char} Let $\ell\in \N$. Suppose $\mathcal U\in C^\ell(\overline{\mathcal D})$ or $W^{\ell,\infty}(\overline{\mathcal D})$ and $\frac{1}{C_0}< \mathcal U < C_0$ on $\overline{\mathcal D}$ for some $C_0>1$. Then for any given $(\u,\v)\in \mathcal D$, there exist a unique exit time $\u_*=\u_*(\u,\v)$ and position $\v_*=\v_*(\u,\v)$ such that
\begin{enumerate}
\item A unique solution $\mathfrak{\v}_\pm  \in C^\ell((\u_*, \u]; C^{\ell}( \mathcal D))$ or $C^\ell((\u_*, \u]; W^{\ell,\infty}( \mathcal D))$ of \eqref{2.103} and \eqref{2.104} exists so that  %for $\u_*< s<\u$, 
$(s,\mathfrak{\v}_\pm (s; \u,\v) )\in \mathcal D$. 
\item At $s=\u_*$, $(\u_*, \v_* ) \in \uC\cup\C$ where $\v_*=\mathfrak{\v}_\pm (\u_*; \u,\v)$ and it satisfies 
\be\label{exit}
\v - \v_* = \int_{\u_*}^\u k_\pm \mathcal U ( s,  \mathfrak{\v}_\pm( s;\u,\v))\,d s. 
\ee
If $\u_*>\u_0$ then $\v_*=\v_0$. 
%and if $\u_*<\u_0$, then $\v_*=\v_0 + \frac{\u_0-\u_*}{\kappa}$. 
%and $\mathfrak{\v}_\pm(s) \in C^k(\mathcal D)$, 
%\item At $s=\u_*$, $(\u_*, \v_* ) \in \C   \cup \uC $.
%\item if $\v>\v_0 +k_\pm C_0 (\u-\u_0)$, then $(\u_*, \mathfrak{\v}_\pm(\u_*) ) \in \C  $
\item $(\u,\v)\mapsto \u_*(\u,\v)\in C(\mathcal D)$.
\item If $\u_*\neq \u_0$, $\u_* \in C^\ell(\mathcal D)$ or $W^{\ell,\infty}(\overline{\mathcal D})$.  %is continuous in $\u$ and $\v$
%\item $\mathfrak{\v}_\pm(s) \in C^k_{\u,\v}$ if $\u_*< s<\u$
 \end{enumerate}
In particular, if $\v>\v_0 +k_\pm C_0 (\u-\u_0)$, then $(\u_*,\v_*) = (\u_0,\v_\ast) \in \C$. 
\end{lemma}

%%%%%%%%%%%%%%%%%%%%%%%%%%%%%%%

\begin{proof} We prove the claims for $\mathcal U\in C^\ell$ as the case of $\mathcal U\in W^{\ell,\infty}$ follows in the same way. 
The local existence and uniqueness of $\mathfrak{\v}_\pm\in C^\ell$ follows from $\mathcal U\in C^\ell$ via the Picard iteration 
\be\label{(2.111)}
\mathfrak{\v}_\pm(s)  = \v - \int_{s}^\u k_\pm \mathcal U (\tilde s,  \mathfrak{\v}_\pm(\tilde s;\u,\v))\,d\tilde s. 
\ee
Thanks to the positive uniform bound of $\mathcal U$ and Gr\"onwall, the solution can be continued as long as the characteristics belong to the domain $(s,\mathfrak{\v}_\pm (s; \u,\v) )\in {\mathcal D}$. Since $ \mathcal U>\frac{1}{C_0}$, $\mathfrak{\v}_\pm(s)  < q - \frac{ k_\pm }{C_0} (p-s)$, $(s,\mathfrak{\v}_\pm(s))$ will exit the domain $\mathcal D$ either through the outgoing boundary
$\C$ or through the ingoing boundary $\uC$. We denote such an exit time by $\u_*$ and the associated value $\v_*=\v_\pm(\u_*; \u,\v)$ where 
$(\u_*, \v_* ) \in \uC \cup \C$. Note that for each given $(\u,\v)\in \mathcal D$, $(\u_*,\v_*)$ is uniquely determined since the backward characteristics $\mathfrak{\v}_\pm$ are unique. By integrating \eqref{2.103} from $s=\u$ to  $s=\u_*$, we see that $(\u_*, \v_* ) $ satisfies \eqref{exit}.  %where  $\u_0- \kappa (\v-\v_0)\le \u_*= \u_*(\u,\v)< \u$.
We observe that if $(\u_*, \v_* ) \in \uC $, $\v_*=\v_0$ and $\u_0\le\u_*< \u$, and \eqref{exit} reads as  
\be\label{exit1}
\v- \v_0 = \int_{\u_*}^\u k_\pm \mathcal U (s,  \mathfrak{\v}_\pm(s;\u,\v)) ds. 
\ee 
If $(\u_*, \v_* ) \in \C  $, $\u_*= \u_0$ and $\v_*$ is given by \eqref{exit}. 
% \kappa (\v_*-\v_0) \le \u_0$ with strict inequality for $\kappa>0$. If $\kappa=0$, $\u_*=\u_0$ and $\v_*$ is given by \eqref{exit}. 
%\be
%\v_* = \v - \int_{\u_0}^\u k_\pm \mathcal U (\tilde s,  \mathfrak{\v}_\pm(\tilde s;\u,\v))\,d\tilde s
%\ee
%On the other hand, if $\kappa>0$, from \eqref{exit} we obtain the relation for $\u_*$
%\be\label{exit2}
%\v - \v_0 + \frac{\u_*-\u_0}{\kappa} = \int_{\u_*}^\u k_\pm \mathcal U (s,  \mathfrak{\v}_\pm( s;\u,\v))\,d s
%\ee  

Clearly $\u_*$ is continuous in $\u$ and $\v$. Since higher regularity of $\u_*$ fails in general at the corner of the domain where $(\u_*,\v_*)=(\u_0,\v_0)$, we show the regularity when $\u\neq \u_0$. First let $(\tilde \u,\tilde \v)\in \mathcal D$ be given.  Suppose $\tilde\u_{*}= \u_*(\tilde \u,\tilde \v)>\u_0$. Consider a small neighborhood $\mathcal B$ of $(\tilde \u,\tilde \v)$ in $\mathcal D$ such that $I_1=\u_*(\mathcal B)$, $\inf I_1>\u_0$. Recalling \eqref{exit1}, we define an auxiliary function $H: I_1 \times \mathcal B\to \mathbb R$ by 
\[
H(\bar\u_*,\u,\v) = \v-\v_0- \int_{\bar\u_*}^\u k_\pm \mathcal U (s,  \mathfrak{\v}_\pm(s;\u,\v)) ds . 
\]
Then $H\in C^\ell$ since $\mathcal U\in C^\ell$, while we have  $H(\tilde\u_*, \tilde\u,\tilde\v)=0$  and $\pa_{\bar\u_*} H (\tilde\u_*,\tilde\u,\tilde\v)  =k_\pm \mathcal U (\tilde\u_*,\tilde\v_* ) >0$. Therefore,  by the implicit function theorem, $\u_*=\u_*(\u,\v)\in C^\ell$ in a small neighborhood of $(\tilde \u,\tilde \v)$. 
%Next let $\tilde\u_*<\u_0$ and consider a small neighborhood $\mathcal B$ so that $I_2=\u_*(\mathcal B) $, $\sup I_2<\u_0$.   Recalling \eqref{exit2}, we 
%define an auxiliary function $H: I_2 \times \mathcal B\to \mathbb R$ by 
%\[
%H(\bar\u_*,\u,\v) =\v - \v_0 + \frac{\bar\u_*-\u_0}{\kappa} - \int_{\bar\u_*}^\u k_\pm \mathcal U (\tilde s,  \mathfrak{\v}_\pm(\tilde s;\u,\v))\,d\tilde s
%\]
%Then we have $H(\tilde\u_*, \tilde\u,\tilde\v)=0$  and $\pa_{\bar\u_*} H (\tilde\u_*,\tilde\u,\tilde\v)  =\frac{1}{\kappa}+ k_\pm \mathcal U (\tilde\u_*,\tilde\v_* ) >0$. Therefore, we deduce that $\u_*=\u_*(\u,\v)\in C^\ell$ in a small neighborhood of $(\tilde \u,\tilde \v)$.

Lastly, since $\mathcal U<C_0$, if $\v>\v_0 +k_\pm C_0 (\u-\u_0)$, the backward characteristics will intersect the outgoing surface $\C$, in which case $\u_*= \u_0$.  
 \end{proof}
 
%%%%%%%%%%%%%%%%%%%%%%%%%%%%%%%
%%%%%%%%%%%%%%%%%%%%%%%%%%%%%%%
%%%%%%%%%%%%%%%%%%%%%%%%%%%%%%%

\subsection{A priori bounds}\label{SS:APRIORIPDE}

In this subsection, we provide estimates for various quantities appearing the iteration scheme in terms of our norms. We will frequently use the following inequality: for any positive function $g>0$
\be\label{logto}
\max\{\|g\|_\infty,  \ \| g^{-1}\|_\infty\} \le e^{\|\log g \|_\infty}
\ee
which directly follows from $g= e^{\log g}$ and $g^{-1}= e^{-\log g}$. 

%%%%%%%%%%%%%%%%%%%%%%%%%%%%%%%
%%%%%%%%%%%%%%%%%%%%%%%%%%%%%%%

\begin{lemma} Suppose $\vertiii{[\Om, r]}<\infty$. Then the following holds:
\begin{align}
\label{estimate_pre}
\sum_{j=0}^2\left\| \v^{j+1}\pa_\v^j\left( \frac{\pa_\v \Om}{\Om} - \e\frac{\pa_\v r}{r}\right) \right\|_\infty  \le M_1(\vertiii{[\Om, r]}) ,\\
\sum_{j=1}^2 \left\| \v^{j+3}\pa_\v^j\left( \frac{\Om^2}{r^3}\right) \right\|_\infty  \le M_2(\vertiii{[\Om, r]}), \label{est_aux}
\end{align}
where  $M_1$ and $M_2$ are continuous functions of their arguments. 
\end{lemma}

%%%%%%%%%%%%%%%%%%%%%%%%%%%%%%%

\begin{proof} We start with \eqref{estimate_pre}. For $j=0$, it is easy to see that 
\be\label{estimate_pre0}
\left\| \v \left( \frac{\pa_\v \Om}{\Om} - \e\frac{\pa_\v r}{r}\right) \right\|_\infty \le \frac{1}{\v_0^\theta} \|\v^{1+\theta} \pa_\v \log \Om\|_\infty + \frac{\e}{2} \left\|\frac{\v^2}{r^2}\right\|_\infty \left\|\frac{\pa_\v (r^2)}{\v}\right\|_\infty  
\ee
where the right-hand side is a continuous function of $\vertiii{[\Om, r]}$ by \eqref{logto}. For $j=1$, since $$\pa_\v ( \frac{\pa_\v \Om}{\Om} - \e\frac{\pa_\v r}{r} ) = \pa_\v^2\log \Om  %- \e \frac{\pa_\v^2r}{r} + \e\frac{(\pa_\v r)^2}{r^2}
- \frac{\e}{2} \frac{\pa_\v^2(r^2)}{r^2} + \frac{\e}{2}\frac{(\pa_\v (r^2))^2}{r^4}
$$
we have 
\be\label{estimate_pre1}
\left\| \v^{2}\pa_\v\left( \frac{\pa_\v \Om}{\Om} - \e\frac{\pa_\v r}{r}\right) \right\|_\infty  \le \frac{1}{\v_0^\theta}\|\v^{2+\theta} \pa_\v^{2}\log\Om\|_\infty + \frac{\e}{2} \left\|\frac{\v^2}{r^2}\right\|_\infty \|\pa_\v^2( r^2)\|_\infty +\frac{ \e }{2}\left\|\frac{\v^2}{r^2}\right\|_\infty^2 \left\|\frac{\pa_\v (r^2)}{\v}\right\|_\infty^2
\ee
which shows \eqref{estimate_pre}. Lastly, from 
$$\pa_\v^2 ( \frac{\pa_\v \Om}{\Om} - \e\frac{\pa_\v r}{r} ) = \pa_\v^3\log \Om  - \frac{\e}{2} \frac{\pa_\v^3(r^2)}{r^2} + \frac{3\e}{2}\frac{\pa_\v (r^2)\pa_\v^2( r^2)}{r^4}- \e \frac{(\pa_\v (r^2))^3}{r^6}$$
we obtain 
\be
\begin{split}
\left\| \v^{3}\pa_\v^2\left( \frac{\pa_\v \Om}{\Om} - \e\frac{\pa_\v r}{r}\right) \right\|_\infty&\le 
\frac{1}{\v_0^\theta}\|\v^{3+\theta} \pa_\v^{3}\log\Om\|_\infty + \frac{\e}{2} \left\|\frac{\v^2}{r^2}\right\|_\infty \|\v \pa_\v^3( r^2)\|_\infty \\
&\quad+\frac{3\e}{2}\left\|\frac{\v^2}{r^2}\right\|_\infty^2  \left\|\frac{\pa_\v (r^2)}{\v}\right\|_\infty^2 \| \pa_\v^2( r^2)\|_\infty + \frac{\e}{2} \left\|\frac{\v^2}{r^2}\right\|_\infty^3 \left\|\frac{\pa_\v (r^2)}{\v}\right\|_\infty^3. 
\end{split}
\ee
This completes the proof of \eqref{estimate_pre}. The estimation of \eqref{est_aux} follows similarly, we omit the details. 
\end{proof}

%%%%%%%%%%%%%%%%%%%%%%%%%%%%%%%
%%%%%%%%%%%%%%%%%%%%%%%%%%%%%%%

%%%%%%%%%%%%%%%%%%%%%%%%%%%%%%%
%%%%%%%%%%%%%%%%%%%%%%%%%%%%%%%

\begin{lemma}[$\mathcal U$ bounds] 
Suppose $\vertiii{[f^\pm, \Om, r]}<\infty$. Then the following holds:
\begin{align}\label{est_U}
%(1+\e)^{-2}\|\mathcal U\|_\infty,  \ (1+\e)^2 \|\mathcal U^{-1}\|_\infty  \le  e^{2\vertiii{[f^\pm, \Om, r]} } \\
%\sum_{j=1}^2\|\v^{j}\pa^j \mathcal U\|_{\infty} \le C\vertiii{[f^\pm, \Om, r]} e^{2\vertiii{[f^\pm, \Om, r]} }\label{est_Uj}
\|\mathcal U\|_\infty+  \|\mathcal U^{-1}\|_\infty+ \sum_{j=1}^2\|\v^{j}\pa_\v^j \mathcal U\|_{\infty} \le M_3 (\vertiii{[f^\pm, \Om, r]})
\end{align}
where $M_3$ is a continuous function of its argument. 
\end{lemma}

\begin{proof} From the first condition of our choice $N_\pm$ in \eqref{2.122}, we may rewrite $\mathcal U$ as 
\[
\mathcal U=(1+\e)^2\Om^2 \left( \frac{\v}{r}\right)^{2\e} \left( \frac{\v^{N_+}f^+}{\v^{N_-}f^-} \right)^\frac12. 
\]
Then by \eqref{logto}, we see that 
\be
\|\mathcal U^{\pm 1}\|_\infty \le (1+\e)^{\pm 2}\left\|\frac{\v^2}{r^2}\right\|^{\pm \e}_\infty e^{2\|\log \Om\|_\infty + \frac{1}{2}  \|\log (\v^{N_+}f^+) \|_\infty +\frac12 \|\log (\v^{N_-}f^-) \|_\infty }
\ee
which shows \eqref{est_U} for $\|\mathcal U\|_\infty + \|\mathcal U^{-1}\|_\infty$. 
%\[
%\begin{split}
%\log \mathcal U&=2\log (1+\e)+2 \log \Om - 2\e \log r + \frac12 (\log f^+ - \log f^-) \\
%&= 2\log (1+\e)+2\log \Om - 2\e \log \frac{r}{\v} + \frac{1}{2}( \log (\v^{N_+}f^+) - \log (\v^{N_-} f^-))
%\end{split}
%\]
% Therefore using \eqref{logto}
% we obtain 
%\be
%\|\log \tfrac{\mathcal U}{(1+\e)^2}\|_\infty\le 2 \|\log\Om\|_\infty + 2\e \|\log \frac{r}{\v}\|_\infty + \frac12 \|\log (\v^{N_+}f^+) \|_\infty +\frac12 \|\log (\v^{N_-}f^-) \|_\infty 
%\ee
Next computing $\pa_\v\mathcal U$ as 
\be
\pa_\v \mathcal U = \left( 2\left(\frac{\pa_\v \Om}{\Om} - \e\frac{\pa_\v r}{r}\right)  + \frac12 \pa_\v\log f^+ - \frac12\pa_\v\log f^- \right) \mathcal  U 
\ee
we obtain 
\be\label{est_U1}
\|\v\pa_\v \mathcal U\|_\infty \le \left( 2\left\| \v \left( \frac{\pa_\v \Om}{\Om} - \e\frac{\pa_\v r}{r}\right) \right\|_\infty %\left\| \frac{\v\pa_\v \Om}{\Om} - \e\frac{\v \pa_\v r}{r} \right\|_\infty 
+\frac12\|\v\pa_\v\log f^+\|_\infty +\frac12\|\v\pa_\v\log f^-\|_\infty \right)  \| \mathcal U\|_\infty. 
\ee
Together with \eqref{estimate_pre0}, it implies \eqref{est_U} for $j=1$. Moreover, since 
\be
\pa_\v^2 \mathcal U = \left(2\pa_\v\left(\frac{\pa_\v \Om}{\Om} - \e\frac{\pa_\v r}{r}\right)  + \frac12 \pa_\v^2\log f^+ - \frac12\pa_\v^2\log f^- \right) \mathcal  U + \frac{(\pa_\v\mathcal U)^2}{\mathcal U}
\ee
we have 
\be
\begin{split}
\|\v^2\pa_\v^2 \mathcal U\|_\infty &\le \Big(2\left\| \v^{2}\pa_\v\left( \frac{\pa_\v \Om}{\Om} - \e\frac{\pa_\v r}{r}\right) \right\|_\infty +\frac12\|\v^2\pa_\v^2\log f^+\|_\infty +\frac12\|\v^2\pa_\v^2\log f^-\|_\infty \Big)  \| \mathcal U\|_\infty \\
& \qquad
+ \|\mathcal U^{-1}\|_\infty \| \v\pa_\v\mathcal U\|_\infty^2. 
\end{split}
\ee
Using \eqref{estimate_pre1} and \eqref{est_U1}, we deduce \eqref{est_U}, where we recall~\eqref{E:TRIPLENORM}.
\end{proof}

%%%%%%%%%%%%%%%%%%%%%%%%%%%%%%%
%%%%%%%%%%%%%%%%%%%%%%%%%%%%%%%

\begin{remark} 
The relation $N_+ = N_- - 4\e$ in \eqref{2.122} is importantly used to ensure the boundedness (both upper and lower) of the transport speed $\mathcal U$. 
\end{remark}

We introduce the constant $\beta>0$:
\be\label{beta}
\beta :=2\e+ \frac{N_++N_-}{k_+-k_-}  - 1=2\e+ \frac{2N_-- (k_+-k_- +4\e)}{k_+-k_-} >0,
\ee
where we have used~\eqref{2.122} in the second equality.

\begin{lemma} Suppose $\vertiii{[f^\pm, \Om, r]}<\infty$. Then the following bounds hold:
\begin{align}
\sum_{j=0}^2\left\|\v^{j+1+\beta}\pa_\v^j\left( \frac{\Om^2}{r^{2\e}} (f^+f^-) ^{\frac{1}{k_+-k_-}}\right)\right\|_\infty \le M_4(\vertiii{[f^\pm, \Om, r]}), \label{est_S1}\\
\sum_{j=0}^2\left\|\v^{j+3+\beta}\pa_\v^j\left( \frac{\Om^2}{r^{2+2\e}} (f^+f^-) ^{\frac{1}{k_+-k_-}}\right)\right\|_\infty \le M_5 (\vertiii{[f^\pm, \Om, r]}), \label{est_S2}
\end{align}
where $M_4$ and $M_5$ are continuous functions of their arguments.
\end{lemma}

\begin{proof} We will first prove \eqref{est_S1}. We start with $j=0$. Using \eqref{beta}, we rewrite $\mathfrak H:=\frac{\Om^2}{r^{2\e}} (f^+f^-) ^{\frac{1}{k_+-k_-}}$ as 
\[
\v^{1+\beta}\mathfrak H= \v^{1+\beta}\frac{\Om^2}{r^{2\e}} (f^+f^-) ^{\frac{1}{k_+-k_-}} = \Om^2\left(\frac{\v}{r}\right)^{2\e} \left(\v^{N_+}f^+\v^{N_-}f^-\right) ^{\frac{1}{k_+-k_-}}
\]
from which we have 
\be\label{est_S11}
\left\|\v^{1+\beta} \mathfrak H\right\|_\infty \le \|\Om\|^2_\infty\left \|\frac{\v^2}{r^2}\right\|_\infty^{\e}(\|\v^{N_+} f^+\|_\infty \|\v^{N_-} f^-\|_\infty )^{\frac{1}{k_+-k_-}} . 
\ee
By \eqref{logto}, the claim immediately follows. We next compute
\be
\pa_\v\mathfrak H =\left( 2\left(\frac{\pa_\v \Om}{\Om} - \e\frac{\pa_\v r}{r}\right) + \frac{\pa_\v \log f^+ + \pa_\v\log f^-}{k_+-k_-}  \right) \mathfrak H
\ee
and obtain 
\be\label{est_S12}
\|\v^{2+\beta}\pa_\v\mathfrak H \|_\infty\le \left( 2\left\| \v \left( \frac{\pa_\v \Om}{\Om} - \e\frac{\pa_\v r}{r}\right) \right\|_\infty %\left\| \frac{\v\pa_\v \Om}{\Om} - \e\frac{\v \pa_\v r}{r} \right\|_\infty 
+\frac{\|\v\pa_\v\log f^+\|_\infty +\|\v\pa_\v\log f^-\|_\infty}{k_+-k_-} \right) \|\v^{1+\beta}\mathfrak H \|_\infty . 
\ee
With \eqref{estimate_pre0} and \eqref{est_S11}, it gives \eqref{est_S1} for $j=1$. We next have
\be
\begin{split}
\pa_\v^2 \mathfrak H &= \left(2\pa_\v\left(\frac{\pa_\v \Om}{\Om} - \e\frac{\pa_\v r}{r}\right)  + \frac{\pa_\v^2 \log f^+ + \pa_\v^2\log f^-}{k_+-k_-}  \right) \mathfrak H \\
&\quad + \left( 2\left(\frac{\pa_\v \Om}{\Om} - \e\frac{\pa_\v r}{r}\right) + \frac{\pa_\v \log f^+ + \pa_\v\log f^-}{k_+-k_-}  \right) \pa_\v\mathfrak H,
\end{split}
\ee
and therefore
\be
\begin{split}
\|\v^{3+\beta}\pa_\v^2 \mathfrak H\|_\infty &\le \Big(2\left\| \v^{2}\pa_\v\left( \frac{\pa_\v \Om}{\Om} - \e\frac{\pa_\v r}{r}\right) \right\|_\infty +\frac{\|\v^2\pa_\v^2\log f^+\|_\infty +\|\v^2\pa_\v^2\log f^-\|_\infty}{k_+-k_-} \Big)  \|\v^{1+\beta} \mathfrak H\|_\infty \\
&+ \left( 2\left\| \v \left( \frac{\pa_\v \Om}{\Om} - \e\frac{\pa_\v r}{r}\right) \right\|_\infty %\left\| \frac{\v\pa_\v \Om}{\Om} - \e\frac{\v \pa_\v r}{r} \right\|_\infty 
+\frac{\|\v\pa_\v\log f^+\|_\infty +\|\v\pa_\v\log f^-\|_\infty}{k_+-k_-} \right) \|\v^{2+\beta}\pa_\v\mathfrak H \|_\infty . 
\end{split}
\ee
Hence the claim follows from \eqref{est_S11}, \eqref{est_S12} and \eqref{estimate_pre}. 

The proof of \eqref{est_S2} follows easily from \eqref{est_S1} by  applying the product rule for $\frac{\Om^2}{r^{2+2\e}} (f^+f^-) ^{\frac{1}{k_+-k_-}}= r^{-2}\mathfrak H$ or by estimating them directly in the same way as done for \eqref{est_S1}. We omit the details. 
\end{proof}

In the iteration scheme, we will make use of the Hawking mass $m$ given in \eqref{E:HAWKINGMASSDEF}. From \eqref{E:PAVM} and~\eqref{E:NORMALISATIONNULL} we see that the Hawking mass satisfies 
%\bcr
\be
\begin{split}
\pa_\v m &= 2\pi r^2 \left[(1-\k)\rho \pa_\v r - (1+\k) \Om^{-2}\rho (u^{\v})^{-2}\pau r\right] \\
&= 2\pi (f^+f^-)^\frac{1}{k_+-k_-} \left[ \frac{1 }{r^{2\e}} \pa_\v r -  \Om^{-2} \left(\frac{f^-}{f^+}\right)^\frac12 \pau r\right]. 
\end{split}
\ee
%\ec
We observe that $p$-derivatives are not featured in our function space hierarchy, however $\pau r$ appears in the above expression. 
To go around this difficulty, we observe that~\eqref{E:RWAVE1} can be formally rewritten in the form
\be\label{E:RPQ}
\pa_{\u\v}(r^2) = -\frac{\Om^2}{2}+2\pi r^2 \Om^4 T^{\u\v} = -\frac{\Om^2}{2} + \frac{2\pi \Om^2  }{r^{2\e}  } (f^+  f^-)^{\frac{1}{k_+-k_-}}, 
\ee
where we have used~\eqref{E:TPP} and~\eqref{2.99} in the second equality.
Therefore, rather than directly estimating $m$ from~\eqref{E:HAWKINGMASSDEF}, we slightly abuse notation, and redefine $m$ to be 
%\bcr
\be\label{def_m}
%m(\u,\v) := m(\u,\v_0 )+ \int^{\v}_{\v_0 } \left(2\pi (f^+f^-)^\frac{1}{k_+-k_-} \left[ \frac{1 }{r^{2\e}} \pa_\v r -  (1+\e)^2\frac{1 }{r^{4\e}} \left(\frac{f^+}{f^-}\right)^\frac12 \pau r\right]  \right) d\tilde \v
m(\u,\v) := m(\u,\v_0)+ \int^{\v}_{\v_0} \left(2\pi (f^+f^-)^\frac{1}{k_+-k_-} \left[ \frac{1 }{r^{2\e}} \pa_\v r -  \frac{ \Om^{-2}  }{2r} \left(\frac{f^-}{f^+}\right)^\frac12 \mathfrak s\right]  \right) d\tilde \v
\ee
%\ec
where 
\begin{align}
m(\u,\v_0)&:=  \frac{\hat r}{2} \left(1+  \frac{4\pau\hat r \pav \hat r}{\hat \Omega^2}\right) \Big|_{(\u,\v_0)}, \\
 \mathfrak s(\u,\v)&:= \pa_\u (\hat r^2) \Big|_{(\u,\v_0)} + \int_{\v_0}^\v \left[  - \frac{\Om^2}{2} + \frac{2\pi \Om^2  }{r^{2\e}  } (f^+  f^-)^{\frac{1}{k_+-k_-}} \right] d\tilde\v. \label{mathfraks}
 \end{align}
 We observe that $\pau \hat r$ is well-defined at $(\u,\v_0)$ since the data there is given by the exact selfsimilar RLP-solution.
We note that $\mathfrak s$ corresponds exactly to $\pau (r^2)$ for a $C^2$-solution of the problem.
%This definition is consistent with \eqref{1.59}. If $[f^\pm,\Om,r]$ are solutions, it should coincide with the Hawking mass introduced in \eqref{E:HAWKINGMASSDEF}. 
In the following, we show that $m$ in \eqref{def_m} can be estimated by using the norm $\vertiii{[f^\pm, \Om, r]}$. 

\begin{lemma}\label{lem:est_m} Suppose $\vertiii{[f^\pm, \Om, r]}<\infty$.  Then the following holds: 
\begin{align}
\|m\|_\infty &\le \vertiii{[ \Om, r]|_{\uC\cup\C}} + M_6 ( \vertiii{[f^\pm, \Om, r]}), \label{est_m}\\
\sum_{j=1}^2\| \v^{j+\beta} \pa_\v^j m \|_\infty& \le  M_7 (\vertiii{[f^\pm, \Om, r]}), \label{est_m1}
\end{align}
where $M_6$ and $M_7$ are continuous functions of its argument. 
\end{lemma}

\begin{proof} We first observe that from \eqref{mathfraks}  using $\v>\v_0$ and $\frac{\v-\v_0}{\v}<1$ 
\be
\|\v^{-1} \mathfrak s\|_\infty \le \left\| \frac{\pa_\u (\hat r^2)}{\v}\Big|_{(\u,\v_0)}\right\|_\infty + \frac{\| \Om\|_\infty^2}{2} +\frac{2\pi}{\v_0^{1+\beta}} \left\| \v^{1+\beta}  \frac{ \Om^2  }{r^{2\e}  } (f^+  f^-)^{\frac{1}{k_+-k_-}} \right\|_\infty 
\ee
and 
\be
\|\pa_\v\mathfrak s \|_\infty \le \frac{\| \Om\|_\infty^2}{2} +\frac{2\pi}{\v_0^{1+\beta}} \left\| \v^{1+\beta}  \frac{ \Om^2  }{r^{2\e}  } (f^+  f^-)^{\frac{1}{k_+-k_-}} \right\|_\infty 
\ee
From the definition of the boundary data norm, we have $|m(\u,\v_0)|\le \vertiii{[ \Om, r]|_{\uC\cup\C}} $. 
For the integral term in \eqref{def_m}, using \eqref{beta} and \eqref{2.122}, we rewrite the integrand as 
%\bcr
\be\label{est_m11}
\begin{split}
\v^{-1-\beta}(\v^{N_+}f^+\v^{N_-} f^-)^\frac{1}{k_+-k_-}   \left[ \frac{\v^{2\e} }{r^{2\e}} \pa_\v r -  \Om^{-2}   \frac{\v }{2r} \left(\frac{\v^{N_-} f^-}{\v^{N_+}f^+}\right)^\frac12 \frac{\mathfrak s}{\v}\right]
\end{split}
\ee
so that 
\[
\begin{split}
&\left|  \int^{\v}_{\v_0 } \left(2\pi (f^+f^-)^\frac{1}{k_+-k_-} \left[ \frac{1 }{r^{2\e}} \pa_\v r - \frac{ \Om^{-2} }{2r} \left(\frac{f^-}{f^+}\right)^\frac12\mathfrak s\right]  \right) d\tilde \v \right| \\
&\le \frac{1}{\beta \v_0^\beta} (\|\v^{N_+}f^+\|_\infty\|\v^{N_-} f^-\|_\infty)^\frac{1}{k_+-k_-}  \\
&\quad \cdot \left[\frac12\|\frac{\v}{r}\|^{2\e+1}_\infty \|\v^{-1}\pa_\v( r^2)\|_\infty + \frac{1}{2}\|\Om^{-1}\|^2_\infty \|\frac{\v}{r}\|_\infty \|\v^{N_-} f^-\|_\infty^\frac12 \|(\v^{N_+}f^+)^{-1}\|_\infty^\frac12 \|\v^{-1} \mathfrak s\|_\infty \right] 
\end{split}
\]
where we have used $\int_{\v_0 }^\v\tilde\v^{-1-\beta} d\tilde\v \le \frac{1}{\beta \v_0^\beta}$. Hence, using \eqref{logto}, we deduce \eqref{est_m}. Moreover since 
\be
\pa_\v m= 2\pi (f^+f^-)^\frac{1}{k_+-k_-} \left[ \frac{1 }{r^{2\e}} \pa_\v r - \frac{ \Om^{-2} }{2r} \left(\frac{f^-}{f^+}\right)^\frac12\mathfrak s\right] 
\ee
from \eqref{est_m11}, we immediately obtain 
\be
\begin{split}
\|\v^{1+\beta} \pa_\v m\|_\infty&\le 2\pi  (\|\v^{N_+}f^+\|_\infty\|\v^{N_-} f^-\|_\infty)^\frac{1}{k_+-k_-} \\
&\cdot \left[\frac12\|\frac{\v}{r}\|^{2\e+1}_\infty \|\v^{-1}\pa_\v( r^2)\|_\infty + \frac{ 1}{2}\|\Om^{-1}\|_\infty^2 \|\frac{\v}{r}\|_\infty \|\v^{N_-} f^-\|_\infty^\frac12 \|(\v^{N_+}f^+)^{-1}\|_\infty^\frac12 \|\v^{-1} \mathfrak s\|_\infty \right] 
\end{split}
\ee
which shows \eqref{est_m1} for $j=1$. Lastly, %To estimate $\pa_\v^2m$, 
we compute $\pa_\v^2m$ as 
\be
\begin{split}
&\pa_\v^2 m= \pi (f^+f^-)^\frac{1}{k_+-k_-} \left[ \frac{1 }{r^{2\e+1}} \pa_\v^2 (r^2) -  \frac{\Om^{-2} }{r} \left(\frac{f^-}{f^+}\right)^\frac12\pa_\v\mathfrak s\right] \\
&+ \pi  \frac{ (f^+f^-)^\frac{1}{k_+-k_-}}{r^{2\e}} \frac{\pa_\v (r^2)}{r}  \left( \frac{\pa_\v\log f^+ + \pa_\v\log f^-}{k_+-k_-} - \frac{2\e+1}{2}\frac{\pa_\v (r^2)}{r^2}\right) \\
&- \pi  \Om^{-2}  (f^+f^-)^\frac{1}{k_+-k_-} \left(\frac{f^-}{f^+}\right)^\frac12 \frac{\mathfrak s}{r} \left( -2\pa_\v \log \Om+ c_1 \pa_\v\log f^+ + c_2\pa_\v \log f^+  -\frac{1}{2}\frac{\pa_\v (r^2)}{r^2}  \right)
\end{split}
\ee
where $c_1=\frac{1}{k_+-k_-} -\frac12$ and $c_2=  \frac{1}{k_+-k_-} +\frac12$. Following the same strategy, it is now clear that $\|\v^{2+\beta}\pa_\v^2m\|_\infty$ is bounded by a continuous function of $\vertiii{[f^\pm, \Om, r]}$. This finishes the proof. 
%\ec
\end{proof}

%%%%%%%%%%%%%%%%%%%%%%%%%%%%%%%
%%%%%%%%%%%%%%%%%%%%%%%%%%%%%%%
%%%%%%%%%%%%%%%%%%%%%%%%%%%%%%%

\subsection{Proof of the local well-posedness}\label{SS:LWPPDE}

\subsubsection{Iteration scheme}

%%%%%%%%%%%%%%%%%%%%%%%%%%%%%%%
%%%%%%%%%%%%%%%%%%%%%%%%%%%%%%%
%%%%%%%%%%%%%%%%%%%%%%%%%%%%%%%

We now set up the iteration scheme. Let $[f^\pm_n,\Om_n,r_n]$ be given  so that $\vertiii{[f^\pm_n, \Om_n, r_n]}<\infty$. And let $\mathcal U_n$ and $m_n$ be given in 
\eqref{2.100} and \eqref{def_m} where $[f,\Om,r]=[f^\pm_n,\Om_n,r_n]$:
\[
\mathcal U_n= (1+\e)^2 \frac{(\Om_n)^2}{(r_n)^{2\e}} \left( \frac{f^+_n}{f^-_n} \right)^\frac12
\]
and 
%\bcr
\[
m_n= m(\u,\v_0)+ \int^{\v}_{\v_0} \left(2\pi (f_n^+f_n^-)^\frac{1}{k_+-k_-} \left[ \frac{1 }{(r_n)^{2\e}} \pa_\v r_n -  \frac{1}{2}\frac{(\Om_n)^{-2} }{(r_n)} \left(\frac{f^-_n}{f_n^+}\right)^\frac12 \mathfrak s_n\right]  \right) d\tilde \v . 
\]
%\ec
Here $\mathfrak s_n$ is given by the right-hand side of \eqref{mathfraks} where $[f^\pm,\Om,r]=[f^\pm_n,\Om_n,r_n]$. 
We then define  $[f^\pm_{n+1},\Om_{n+1},r_{n+1}]$ to be the solution of the following system 
\begin{align}
&\pa_\u f^\pm_{n+1} + k_\pm \mathcal U_n \pa_\v f^\pm_{n+1} \pm 2k_\pm (2\frac{\pa_\v \Om_n}{\Om_n} - 2\e\frac{\pa_\v r_n}{r_n})  \mathcal U_n f^\pm_{n+1}  =0, \label{eq:fn+1} \\
& \pa_\u \pa_\v [(r_{n+1})^2]= - \frac{(\Om_n)^2}{2} + 2\pi \frac{(\Om_n)^2}{(r_n)^{2\e}} (f^+_nf^-_n) ^{\frac{1}{k_+-k_-}}, \label{recast1} \\
&  \pa_\u \pa_\v \log {\Om_{n+1}} = \frac{ (\Om_n)^2}{2 (r_n)^2} \frac{m_n}{r_n} - (1+\e) \pi \frac{ (\Om_n)^2 }{(r_n)^{2+2\e}}(f^+_n  f^-_n)^{\frac{1}{k_+-k_-}}, \label{recast2} 
\end{align} 
with given characteristic data $[f^\pm_{n+1},\Om_{n+1},r_{n+1}]= [\hat f^\pm, \hat\Om, \hat r]|_{\uC \cup \C}$ satisfying the conditions (i)--(iii) as in Theorem \ref{thm:LWP}. We note that for exact solutions, \eqref{recast1}-\eqref{recast2} are equivalent to  \eqref{E:RWAVE1},  \eqref{E:OMEGAWAVE1}. 

Our next goal is to prove the solvability of the iterative system \eqref{eq:fn+1}-\eqref{recast2}  and derive the uniform bounds of $\vertiii{[f^\pm_{n+1},\Om_{n+1},r_{n+1}]}$ for sufficiently small $\delta>0$. 
%and $\kappa\ge 0$. 
Let $A:= \max\{4, 1+ \frac{1}{\theta \v_0^\theta}\}$. 

\begin{proposition}\label{prop:sol} Let the characteristic data $[\hat f^{\pm}, \hat\Om, \hat r]$ satisfying the assumptions of Theorem \ref{thm:LWP}  be given. Suppose $\vertiii{[f^\pm_{n},\Om_{n},r_{n}]}\le 2A\vertiii{[\hat f^\pm, \hat\Om, \hat r]|_{\uC \cup \C}} $. 
Then there exist sufficiently small $\delta_0>0$  depending only on $\vertiii{[\hat f^\pm, \hat\Om, \hat r]|_{\uC \cup \C}}$ such that for all $\delta\in(0,\delta_0)$
%$f_{\pm, 0}$, $f_{\pm,\text{in}}\in \mathcal S^{k_\pm}$ {\color{red} satisfying the compatibility condition   \eqref{compatibility} for $n=0,1$},  
there exists a unique solution $f^\pm_{n+1} \in W^{2,\infty}(\overline{\mathcal D})$, $[\Om_{n+1}, \frac{r_{n+1}}{\v}]\in W^{3,\infty}(\overline{\mathcal D})$ to \eqref{eq:fn+1}-\eqref{recast2} %$f^\pm \in \mathcal S^{k_\pm}$ 
%to \eqref{E:FPLUSEVOLUTION} and \eqref{E:FMINUSEVOLUTION} 
with $[f_{n+1}^{\pm} ,\Om_{n+1},r_{n+1}]= [\hat f^{\pm}, \hat\Om, \hat r]$ on $\uC \cup \C$ satisfying 
\be
\vertiii{[f^\pm_{n+1},\Om_{n+1},r_{n+1}]}\le 2A\vertiii{[\hat f^\pm, \hat\Om, \hat r]|_{\uC \cup \C}} . 
\ee
\end{proposition}

Proposition \ref{prop:sol} immediately follows from the following two lemmas.  

%%%%%%%%%%%%%%%%%%%%%%%%%%%%%%%
%%%%%%%%%%%%%%%%%%%%%%%%%%%%%%%

\begin{lemma}[Solving the Euler part]\label{lem:fn+1} Assume the same as in Proposition \ref{prop:sol}. 
Then there exist sufficiently small $\delta_0>0$  depending only on $\vertiii{[\hat f^\pm, \hat\Om, \hat r]|_{\uC \cup \C}}$ such that for all $\delta\in(0,\delta_0)$
there exists a unique solution $f^\pm_{n+1} \in W^{2,\infty}(\overline{\mathcal D})$ to \eqref{eq:fn+1} %$f^\pm \in \mathcal S^{k_\pm}$ 
%to \eqref{E:FPLUSEVOLUTION} and \eqref{E:FMINUSEVOLUTION} 
with $f_{n+1}^{\pm} = \hat f^{\pm}$ on $\uC \cup \C$ satisfying 
\be\label{bound:fn+1}
\vertiii{[f^\pm_{n+1}]}\le 2\vertiii{[\hat f^\pm]|_{\uC \cup \C}} + A \vertiii{[\hat f^\pm, \hat\Om, \hat r]|_{\uC \cup \C}} . 
\ee
\end{lemma}

%%%%%%%%%%%%%%%%%%%%%%%%%%%%%%%
%%%%%%%%%%%%%%%%%%%%%%%%%%%%%%%

\begin{lemma}[Solving the metric part]\label{lem:Omrn+1} Assume the same as in Proposition \ref{prop:sol}. 
Then there exist sufficiently small $\delta_0>0$  depending only on $\vertiii{[\hat f^\pm, \hat\Om, \hat r]|_{\uC \cup \C}}$ such that for all $\delta\in(0,\delta_0)$
there exists a unique solution $[\Om_{n+1}, \frac{r_{n+1}}{\v}]\in W^{3,\infty}(\overline{\mathcal D})$ to \eqref{recast1}-\eqref{recast2} with $[\Om_{n+1},r_{n+1}]= [ \hat\Om, \hat r]$ on $\uC \cup \C$ satisfying 
\be
\vertiii{[\Om_{n+1},r_{n+1}]}\le A\vertiii{[\hat\Om, \hat r]|_{\uC \cup \C}} +A \vertiii{[\hat f^\pm, \hat\Om, \hat r]|_{\uC \cup \C}} . 
\ee
\end{lemma}

%%%%%%%%%%%%%%%%%%%%%%%%%%%%%%%
%%%%%%%%%%%%%%%%%%%%%%%%%%%%%%%

In what follows, we prove the above two lemmas. We start with Lemma  \ref{lem:fn+1}.

\begin{proof}[Proof of  Lemma \ref{lem:fn+1} (Solving the Euler part).] For the sake of notational convenience, throughout the proof, we use $\mathcal U=\mathcal U_n$, $\Om=\Om_n$, $r=r_n$, $f^\pm=f_n^\pm$ so that the variables without the indices refer to the ones from the $n$-th step.

 Uniqueness follows from the uniqueness of characteristics since $ W^{2,\infty}\subset  C^1$ solutions $f_{n+1}^\pm>0$ satisfy the ODE along the characteristics 
\be\label{fODE}
\frac{d}{d\u} \log f_{n+1}^\pm (\u, \mathfrak{\v}_\pm ) = \mp 2k_\pm (2\frac{\pa_\v \Om}{\Om} - 2\e\frac{\pa_\v r}{r})  \mathcal U (\u, \mathfrak{\v}_\pm ) . 
\ee
The existence follows by the integral representation of  \eqref{fODE}
\be\label{int_rep}
 f_{n+1}^\pm (\u, \v ) = \hat f^{\pm}_*(\u,\v) \exp\left\{\mp\int_{\u_*}^\u 2 k_\pm  \left[ (2\frac{\pa_\v \Om}{\Om} - 2\e\frac{\pa_\v r}{r})  \mathcal U \right] (s, \mathfrak{\v}_\pm (s)) ds \right\}
\ee
%We focus on the existence. %For given smooth,  bounded functions $\mathcal U$, $\pa_\v\Om$, and $\pa_\v r$, 
%The first goal is to show the existence of a unique mild solution to the integral representation of  \eqref{fODE}
%\be\label{2.106}
%f^\pm (\u,\v) := f_{\pm,*} (\u,\v) \mp\int_{\u_*}^\u 2 k_\pm  \left[ (2\frac{\pa_\v \Om}{\Om} - 2\e\frac{\pa_\v r}{r})  \mathcal U f^\pm\right] (s, \mathfrak{\v}_\pm (s)) ds 
%\ee
where 
$
\hat f^{\pm}_*(\u,\v) = \hat f^{\pm }(\u_*, \v_*), 
%\begin{cases}
%\hat f^{\pm }(\u_*, \v_*) & \text{ if } \u_\ast =  \u_0%- \kappa (\v-\v_0)  
%\\
%\underline{\hat f}^{\pm}(\u_*, \v_*)  & \text{ if } \u_\ast >\u_0
%\end{cases}
$
and $(\u_*,\v_*)$ is the exit time and position associated with $(\u,\v)$ constructed in Lemma \ref{lem:Char}. 
We focus on verifying the desired regularity and estimates. 

First of all, clearly $f_{n+1}^\pm \in C(\overline{\mathcal D})$ since $\u_*$ is continuous.
% and the data satisfy the zeroth-order compatibility conditions 
%\eqref{compatibility}. 
To estimate $\|\log (\v^{N_\pm} f^\pm_{n+1})\|_\infty$, we take log of \eqref{int_rep} and rewrite it as 
\be\label{est_logf}
\log (\v^{N_\pm} f^\pm_{n+1}) = \log (\v^{N_\pm} \hat f^{\pm}_*) \mp\int_{\u_*}^\u 2 k_\pm  \left[ (2\frac{\pa_\v \Om}{\Om} - 2\e\frac{\pa_\v r}{r})  \mathcal U \right] (s, \mathfrak{\v}_\pm (s)) ds.
\ee
For the first term, we may write it as 
\[
 \log (\v^{N_\pm} \hat f^{\pm}_*) =  \log (\v_*^{N_\pm} \hat f^{\pm}_*) + N_\pm \log (\frac{\v}{\v_*}). 
\]
Since $| \log (\v_*^{N_\pm} \hat f^{\pm}_*)|\le  \|\log (\v^{N_\pm} f^\pm) |_{\uC \cup \C} \|_\infty$, it suffices to estimate $ \log (\frac{\v}{\v_*})$. To this end, 
%To show the bound \eqref{est_fpm} for $j=0$, we work with another equivalent expression of $f^\pm$: 
%\be\label{int_rep0}
 %f^\pm (\u, \v ) = \hat f^{\pm}_*(\u,\v) \mp\int_{\u_*}^\u 2 k_\pm  \left[ (2\frac{\pa_\v \Om}{\Om} - 2\e\frac{\pa_\v r}{r})  \mathcal U f^\pm  \right] (s, \mathfrak{\v}_\pm (s)) ds 
%\ee
%we need to estimate $\hat{f}_{\pm,*}(\u,\v)$ and the exponent in \eqref{int_rep}. 
%We start with $|\log(\v^{N_\pm} \hat{f}_{\pm,*} (\u,\v))|$. 
first let $\u_*= \u_0$.
% so that $\u_*=\u_0-\kappa(\v_*-\v_0)$. 
Then we have %for $\u_*\le s<\u$, 
\[
\v_*=\mathfrak{q}_\pm (\u_0;\u,\v) = \v - \int_{\u_0}^\u k_\pm \mathcal U (s) ds > \v - k_\pm \| \mathcal U\|_\infty (\u-\u_0)\ge 
 \v - k_\pm \| \mathcal U\|_\infty \delta . 
\]
In particular, using $\v\ge \v_0$, we have $\v_* \ge \left(1-k_\pm \| \mathcal U\|_\infty \v_0^{-1}\delta\right)\v$. Hence, 
\be\label{est1}
| \log (\frac{\v}{\v_*}) | \le - \log (1-k_\pm \| \mathcal U\|_\infty \v_0^{-1}\delta) \le  \v_0^{-1}\delta k_\pm \| \mathcal U\|_\infty  
\ee
for  sufficiently small $\delta>0$. 
% we have $\v_*>\frac{\v}{2}$ and in turn 
%\be\label{est1}
%|\v^{N_\pm} \hat{f}_{\pm,*}(\u,\v) |= |\v^{N_\pm}\v_*^{-N_\pm}  \v_*^{N_\pm}  \hat{f}_{\pm,0}(\u_*,\v_*)| \le 2^{N_\pm} \|\v^{N_\pm} f^\pm |_{\C} \|_\infty 
%\ee
If $\u_*>\u_0$, we have $\v_*=\v_0$ and $\v<\v_0 + k_\pm \|\mathcal U\|_\infty (\u-\u_0)$. Hence in this case, for sufficiently small $\delta>0$, 
\be\label{est2}
| \log (\frac{\v}{\v_*}) | \le \log (1+ k_\pm \|\mathcal U\|_\infty \v_0^{-1}\delta)\le k_\pm \|\mathcal U\|_\infty \v_0^{-1}\delta. 
\ee
We next estimate the integral term in \eqref{est_logf}. %\eqref{int_rep0}. 
Using $\mathfrak{\v}_\pm (s) \ge \v_*$ for all $\u_*\le s\le \u$, %$\v_*>\frac{\v}{2}$, 
and $\u-\u_* \le \delta$, 
\be\label{est3}
\begin{split}
&\left| \int_{\u_*}^\u 2 k_\pm  \left[ (2\frac{\pa_\v \Om}{\Om} - 2\e\frac{\pa_\v r}{r})  \mathcal U  \right] (s, \mathfrak{\v}_\pm (s)) ds \right|\\
&= \left| 4k_\pm  \int_{\u_*}^\u \frac{1}{\mathfrak{\v}_\pm} \left[  \mathfrak{\v}_\pm (\frac{  \pa_\v \Om}{\Om} - \e\frac{\pa_\v r}{r}) \mathcal U \right] (s, \mathfrak{\v}_\pm (s)) ds \right| \\
&\le 4k_\pm \| \mathcal U\|_\infty  \left\| \frac{\v\pa_\v \Om}{\Om} - \e\frac{\v\pa_\v r}{r} \right\|_\infty \frac{\u-\u_*}{\v_*}\\
& \le 4k_\pm \| \mathcal U\|_\infty  \left\| \frac{\v\pa_\v \Om}{\Om} - \e\frac{\v\pa_\v r}{r} \right\|_\infty  \frac{\delta}{\v_0}.
\end{split}
\ee
Therefore by \eqref{est_logf}, \eqref{est1}, \eqref{est2}, \eqref{est3}, we obtain 
\be\label{est4}
\begin{split}
 \|\log(\v^{N_\pm} f_{n+1}^\pm) \|_\infty& \le   \|\log (\v^{N_\pm} f^\pm) |_{\uC \cup \C} \|_\infty + 
 N_\pm k_\pm \| \mathcal U\|_\infty \frac{\delta}{\v_0} \\
 &\quad+ 4k_\pm \| \mathcal U\|_\infty  \left\| \frac{\v\pa_\v \Om}{\Om} - \e\frac{\v\pa_\v r}{r} \right\|_\infty  \frac{\delta}{\v_0} . 
 \end{split}
\ee

Next we show that $\pa_\v f_{n+1}^\pm$ is continuous. Since $\u_*\in C^1$ if $\u_*\neq \u_0$ by Lemma \ref{lem:Char}, the right-hand side of \eqref{int_rep} is $C^1$ and thus if $\u_*\neq \u_0$, $\pa_\v f_{n+1}^\pm$ 
is clearly continuous. Therefore, it suffices to show that $\pa_\v f_{n+1}^\pm (\u,\v)$ is continuous when $\u_*=\u_0$ and $\v_*=\v_0$. To this end, let $\u_*(\u,\v)=\u_0$ and $\v_*(\u,\v)=\v_0$, and 
take any $(\bar\u,\bar\v)\neq (\u,\v)$  in a small neighborhood of such $(\u,\v)$. Then we have $\u_*(\bar\u,\bar\v)\neq \u_0$ and hence $\pa_\v f_{n+1}^\pm (\bar\u,\bar\v) $ is continuous. In fact, 
for any $(\bar\u,\bar\v)$ with $\bar\u_*:= \u_*(\bar\u,\bar\v)\neq \u_0$ we can solve the equations for $\pa_\v \log f_{n+1}^\pm$:
\[
\begin{split}
\pa_\u (\pa_\v  \log f_{n+1}^\pm) + k_\pm \mathcal U \pa_\v (\pa_\v\log  f_{n+1}^\pm) + k_\pm \pa_\v \mathcal U \pa_\v \log f_{n+1}^\pm  \pm 4k_\pm \pa_\v \left( (\frac{\pa_\v \Om}{\Om} - \e\frac{\pa_\v r}{r})  \mathcal U\right)  =0
\end{split}
\] along the characteristics and 
 obtain the integral representation 
\be\label{int_rep1}
\begin{split}
%\pa_\v f^\pm (\bar\u,\bar\v) = \left(  \frac{\pa_\v \hat f^{\pm}_*(\bar\u,\bar\v)}{\hat f^{\pm}_*(\bar\u,\bar\v)}-  \int_{\bar\u_*}^{\bar\u} \left[  k_\pm \pa_\v \mathcal U \pa_\v \log f^\pm  \pm 2k_\pm \pa_\v \left( (2\frac{\pa_\v \Om}{\Om} - 2\e\frac{\pa_\v r}{r})  \mathcal U\right)\right] ds  \right)  f^\pm (\bar\u,\bar\v)
\pa_\v \log f_{n+1}^\pm (\bar\u,\bar\v) & = \pa_\v \log \hat f^{\pm}_*(\bar\u,\bar\v) \\
&-  \int_{\bar\u_*}^{\bar\u} \left[  k_\pm \pa_\v \mathcal U \pa_\v \log f_{n+1}^\pm  \pm 4k_\pm \pa_\v \left( (\frac{\pa_\v \Om}{\Om} - \e\frac{\pa_\v r}{r})  \mathcal U\right)\right]  (s, \mathfrak{\v}_\pm (s)) ds
\end{split}
\ee 
 where 
\[
\pa_\v\log \hat f^{\pm}_*(\bar\u,\bar\v) = \mathbf{1}_{\bar\u_*>\u_0} \pa_\v \log\hat f\big\vert_{\uC}^{\pm}(\bar\u_*,\bar\v_*) +  \delta_{\bar\u_*\u_0} \pa_\v \log\hat f^{\pm}(\bar\u_*,\bar\v_*), 
\]
where $\delta_{\bar\u_*\u_0}$ is the usual Kronecker delta.
Notice that the integral terms are all continuous due to the continuity of $\u_*$ in Lemma \ref{lem:Char} and therefore $f_{n+1}^\pm \in C^1(\overline{\mathcal D})$.
%  Now for $\pa_\v \hat f^{\pm}_*(\bar\u,\bar\v)$, due to the 1-compatibility condition \eqref{compatibility}, $\lim_{(\bar\u,\bar\v)\to(\u,\v)} \pa_\v \hat f^{\pm}_*(\bar\u,\bar\v) =  \pa_\v \hat f_{\pm, \text{in}}(\u_*,\v_*)=\pa_\v \hat f_{\pm, \text{0}}(\u_*,\v_*)  $ and the same holds for $\pa_\v \log \hat f_{\pm,*}$. This concludes the proof of $f_{n+1}^\pm \in C^1(\overline{\mathcal D})$. 

To estimate $\v\pa_\v\log f_{n+1}^\pm $, %we rewrite \eqref{int_rep1} as 
%\[
%\begin{split}
%\v^{N_\pm +1} \pa_\v f^\pm (\u,\v) &= \v^{N_\pm +1} \pa_\v \hat f^{\pm}_*(\u,\v) \frac{ f^\pm (\u,\v)}{\hat f^{\pm}_*(\u,\v)}\\
%&-  \v^{N_\pm } f^\pm (\u,\v) \left\{ \v \int_{\u_*}^{\u} \left[  k_\pm \pa_\v \mathcal U \pa_\v \log f^\pm  \pm 2k_\pm \pa_\v \left( (2\frac{\pa_\v \Om}{\Om} - 2\e\frac{\pa_\v r}{r})  \mathcal U\right)\right] ds  \right\}
%\end{split}
%\]
as done in \eqref{est1} and \eqref{est2}, we first observe that $\frac{\v}{\v_*}\le \frac{1 }{1-k_\pm \| \mathcal U\|_\infty \v_0^{-1}\delta}$. Then for sufficiently small $\delta>0$ 
we have 
\be\label{boundq}
\frac{\v}{\v_*}, \ \frac{\v^2}{\v_*^2}% \frac{1+k_\pm \| \mathcal U\|_\infty \kappa }{1-k_\pm \| \mathcal U\|_\infty \v_0^{-1}\delta}
\le 1+2k_\pm \| \mathcal U\|_\infty \frac{\delta}{\v_0}.
\ee
Based on the integral representation~\eqref{int_rep1} we proceed to estimate  
\be
%\left| \v^{N_\pm +1} \pa_\v \hat f^{\pm}_*(\u,\v) \frac{ f^\pm (\u,\v)}{\hat f^{\pm}_*(\u,\v)}\right| \le e^{M_1} 2^{N_\pm+1}\|\v^{N_\pm+1} \pa_\v f^\pm |_{\pa \mathcal D} \|_\infty 
|\v \pa_\v \log \hat f^{\pm}_*(\bar\u,\bar\v) |\le \left( 1+2k_\pm \| \mathcal U\|_\infty\frac{\delta}{\v_0}  \right) \|\v  \pa_\v \log \hat f^\pm |_{\uC \cup \C} \|_\infty ,
\ee
where we have used the first bound in~\eqref{boundq}.

For the second term in~\eqref{int_rep1}, 
\[
\begin{split}
& \left| \v \int_{\u_*}^{\u} \left[  k_\pm \pa_\v \mathcal U \pa_\v \log f_{n+1}^\pm  \pm 4k_\pm \pa_\v \left( (\frac{\pa_\v \Om}{\Om} - \e\frac{\pa_\v r}{r})  \mathcal U\right)\right]  (s, \mathfrak{\v}_\pm (s)) ds  \right| \\
& \le k_\pm  \left(  \|\v \pa_\v \mathcal U\|_\infty   \| \v \pa_\v \log f_{n+1}^\pm \|_\infty+4 \left\| \v^2 \pa_\v \left( (\frac{\pa_\v \Om}{\Om} - \e\frac{\pa_\v r}{r})  \mathcal U\right) \right\|_\infty   \right)\frac{\v (\u-\u_*)}{\v_*^2} \\
&  \le k_\pm  \left(   \|\v \pa_\v \mathcal U\|_\infty   \| \v \pa_\v \log f_{n+1}^\pm \|_\infty +4 \left\| \v^2 \pa_\v \left( (\frac{\pa_\v \Om}{\Om} - \e\frac{\pa_\v r}{r})  \mathcal U\right) \right\|_\infty   \right)
2\frac{\delta}{\v_0} \\
%&\qquad\cdot\left[1+2k_\pm \| \mathcal U\|_\infty\frac{\delta}{\v_0} \right]
\end{split}
\]
where we have used the upper bound $\frac{\v}{\v_*} \le 2$ as it follows from \eqref{boundq} for $\delta>0$ sufficiently small. 
Therefore, we deduce that 
\be\label{est5}
\begin{split}
&\left( 1-2  k_\pm \|\v \pa_\v \mathcal U\|_\infty  \frac{\delta}{\v_0}  \right) \| \v \pa_\v \log f_{n+1}^\pm \|_\infty \\
&\le \left( 1+2k_\pm \| \mathcal U\|_\infty \frac{\delta}{\v_0} \right)  \|\v  \pa_\v \log \hat f^\pm |_{\uC \cup \C} \|_\infty  
+ 8k_\pm \frac{\delta}{\v_0} \left\| \v^2 \pa_\v \left( (\frac{\pa_\v \Om}{\Om} - \e\frac{\pa_\v r}{r})  \mathcal U\right) \right\|_\infty . 
\end{split}
\ee
%This proves \eqref{est_fpm} for $j=1$. 

From \eqref{int_rep}, since $f_{n+1}^\pm\in W^{2,\infty}$, we just need to estimate $\v^2\pa_\v^2\log f^\pm_{n+1}$. 
After applying $\pa_{\v\v}$ to the equation~\eqref{eq:fn+1}, we have the 
following integral representation:
\be
\begin{split}
&\pa_\v^2\log f_{n+1}^\pm (\u,\v)= \pa_\v^2\log \hat f_{\pm,* }(\u,\v) \\
&-  \int_{\u_*}^{\u} \left[   k_\pm \pa_\v^2 \mathcal U \pa_\v \log f_{n+1}^\pm + 2k_\pm \pa_\v \mathcal U \pa_\v^2 \log f_{n+1}^\pm \pm 4k_\pm \pa_\v^2 \left( (\frac{\pa_\v \Om}{\Om} - \e\frac{\pa_\v r}{r})  \mathcal U\right)\right]  (s, \mathfrak{\v}_\pm (s)) ds.
\end{split}
\ee
Analogously to  the previous step, we obtain  
\be\label{est6}
\begin{split}
&\left(  1-2  k_\pm \|\v \pa_\v \mathcal U\|_\infty \frac{\delta}{\v_0}   \right) \| \v^2 \pa_\v^2 \log f_{n+1}^\pm \|_\infty 
\le \left( 1+2k_\pm \| \mathcal U\|_\infty\frac{\delta}{\v_0}  \right) \|\v^2  \pa_\v^2 \log \hat f^\pm |_{\uC \cup \C} \|_\infty  \\
&+2k_\pm\frac{\delta}{\v_0}  \|\v^2\pa_\v^2\mathcal U\|_\infty\| \v \pa_\v \log f_{n+1}^\pm \|_\infty +8k_\pm \frac{\delta}{\v_0}\left\| \v^3 \pa_\v^2 \left( (\frac{\pa_\v \Om}{\Om} - \e\frac{\pa_\v r}{r})  \mathcal U\right) \right\|_\infty . 
\end{split}
\ee
 We now collect \eqref{est4}, \eqref{est5}, \eqref{est6} and use \eqref{estimate_pre}, \eqref{est_U} as well as $\vertiii{[f^\pm_{n},\Om_{n},r_{n}]}\le 2A\vertiii{[\hat f^\pm, \hat\Om, \hat r]|_{\uC\cup\C}}$ 
 to  deduce  \eqref{bound:fn+1} for sufficiently small $\delta>0$. This completes the proof.  
%that there exist sufficiently small $\delta_0>0$ and $\kappa_0$ depending  only on $\vertiii{[\hat f^\pm, \hat\Om, \hat r]|_{\uC \cup \C}}$ such that  for all $\delta\in (0,\delta_0)$ and $\kappa\in[0,\kappa_0)$, \eqref{bound:fn+1} holds true. This completes the proof. 
%This proves \eqref{est_fpm} for $j=2$. 
\end{proof}

\begin{proof}[Proof of Lemma \ref{lem:Omrn+1} (Solving the metric part).] As in the previous proof, for the sake of notational convenience, throughout the proof, we use $\Om=\Om_n$, $r=r_n$, $f^\pm=f_n^\pm$, $m=m_n$ so that the variables without the indices refer to the ones from the $n$-th step. 

Since \eqref{recast1},  \eqref{recast2} are linear inhomogeneous ODEs, we directly integrate them to solve for $ r_{n+1}$ and $\Om_{n+1}$. 
As the existence is clear, we focus on the estimates.

We now integrate \eqref{recast1} along an ingoing null curve from $\u_0$ to $\u$ to obtain 
\be\label{(2.131)}
\pa_\v ( (r_{n+1})^2) (\u,\v) = \pa_\v (\hat r^2) (\u_0, \v)+%(\u_0-\kappa(\v-\v_0), \v) +
 \int_{\u_0}^\u \left[  - \frac{\Om^2}{2} + \frac{2\pi \Om^2  }{r^{2\e}  } (f^+  f^-)^{\frac{1}{k_+-k_-}} \right] d\tilde\u 
\ee with 
\be\label{(2.132)}
( r_{n+1})^2 (\u,\v) = \hat r^2(\u, \v_0)+ \int_{\v_0}^\v \pa_\v (( r_{n+1})^2) (\u,\tilde \v)d \tilde \v . 
\ee %where $\v^*=\v_0$ if $\u\ge \u_0$ and $\v^*= \v_0+ \frac{1}{\kappa} (\u_0-\u)$ if $\u_0-\kappa(\v-\v_0) <\u<\u_0$. 

We first estimate $\frac{1}{\v}\pa_\v(( r_{n+1})^2)$. To this end, we bound the integral term of \eqref{(2.131)} as
\be
\begin{split}
&\left| \frac{1}{\v} \int_{\u_0}^\u \left[  - \frac{\Om^2}{2} + \frac{2\pi \Om^2  }{r^{2\e}  } (f^+  f^-)^{\frac{1}{k_+-k_-}} \right] d\tilde\u\right| \\
&\qquad\le \frac{\delta}{\v_0} \left( \frac{ \|\Om\|^2_\infty}{2} + \frac{2\pi}{\v_0^\beta}  \left\| \v^{1+\beta} \frac{ \Om^2  }{r^{2\e}  } (f^+  f^-)^{\frac{1}{k_+-k_-}} \right\|_\infty  \right)
\end{split}
\ee
where we have used $\frac{\u-\u_0}{\v}\le \frac{\delta}{\v_0}$. Therefore, we have 
\be\label{est_r1}
\begin{split}
\left\| \frac{1}{\v}\pa_\v((r_{n+1})^2) \right\|_\infty & \le \left\| \frac{1}{\v}\pa_\v(\hat r^2)|_{\C} \right\|_\infty 
+\frac{\delta}{\v_0} \left( \frac{ \|\Om\|^2_\infty}{2} + \frac{2\pi}{\v_0^\beta}  \left\| \v^{1+\beta} \frac{ \Om^2  }{r^{2\e}  } (f^+  f^-)^{\frac{1}{k_+-k_-}} \right\|_\infty  \right). 
% \left(\frac{\delta}{\v_0}+\kappa \right) \|\Om\|^2_\infty \left( \frac{1}{2} + 2\pi  \left\| \frac{\v}{r}\right\|_\infty^{2\e} (\|\v^{N_+} f^+\|_\infty \|\v^{N_-} f^-\|_\infty )^{\frac{1}{k_+-k_-}} \right)
\end{split}
\ee
We also have
\be\label{est_r2}
\begin{split}
 \frac{1}{\v}\pa_\v((r_{n+1})^2) (\u,\v) & \ge  \frac{1}{\v} \pa_\v (\hat r^2) (\u_0, \v) \\
&-\frac{\delta}{\v_0}\left( \frac{ \|\Om\|^2_\infty}{2} + \frac{2\pi}{\v_0^\beta}  \left\| \v^{1+\beta} \frac{ \Om^2  }{r^{2\e}  } (f^+  f^-)^{\frac{1}{k_+-k_-}} \right\|_\infty  \right). 
\end{split}
\ee
Since $   \frac{1}{\v} \pa_\v (\hat r^2) (\u_0, \v) \ge\frac{1}{ \|\frac{\v}{\pa_\v(\hat r^2)}|_{\C}\|_\infty}$, for sufficiently small $\delta$ we obtain 
\be\label{est_r22}
\begin{split}
 &\frac{\v}{\pa_\v(( r_{n+1})^2) (\u,\v)}  \le
\frac{1}{ \frac{1}{ \|\frac{\v}{\pa_\v(\hat r^2)}|_{\C}\|_\infty}
-\frac{\delta}{\v_0}\left( \frac{ \|\Om\|^2_\infty}{2} + \frac{2\pi}{\v_0^\beta}  \left\| \v^{1+\beta} \frac{ \Om^2  }{r^{2\e}  } (f^+  f^-)^{\frac{1}{k_+-k_-}} \right\|_\infty  \right)}\\
&\le  \left\|\frac{\v}{\pa_\v(\hat r^2)}|_{\C}\right\|_\infty  +2\frac{\delta}{\v_0} \left\|\frac{\v}{\pa_\v(\hat r^2)}|_{\C}\right\|_\infty^2 \left( \frac{ \|\Om\|^2_\infty}{2} + \frac{2\pi}{\v_0^\beta}  \left\| \v^{1+\beta} \frac{ \Om^2  }{r^{2\e}  } (f^+  f^-)^{\frac{1}{k_+-k_-}} \right\|_\infty\right). 
\end{split}
\ee

On the other hand, from \eqref{(2.132)} and \eqref{est_r1}, the upper bound of $\frac{( r_{n+1})^2}{\v^2}$ is easily obtained: since $\v_0<\v$ and using \eqref{est_r1}, 
\be\label{est_r3}
\begin{split}
\frac{( r_{n+1})^2(\u,\v)}{\v^2}&=\frac{\hat r^2(\u, \v_0)}{\v^2} + \frac{1}{\v^2} \int_{\v_0}^\v \tilde\v  \frac{\pa_\v((r_{n+1})^2)  (\u,\tilde \v)}{\tilde\v}d \tilde \v  \\
%&\le \left\| \frac{\hat r^2}{\v^2}|_{\uC \cup \C} \right\|_\infty + \frac12 \left\| \frac{1}{\v}\pa_\v(\tilde r^2)  \right\|_\infty \\
&\le  \left\| \frac{\hat r^2}{\v^2}|_{\uC \cup \C} \right\|_\infty +  \frac12\left\| \frac{1}{\v}\pa_\v(\hat r^2)|_{\C} \right\|_\infty \\
&\quad+\frac12\frac{\delta}{\v_0} \left( \frac{ \|\Om\|^2_\infty}{2} + \frac{2\pi}{\v_0^\beta}  \left\| \v^{1+\beta} \frac{ \Om^2  }{r^{2\e}  } (f^+  f^-)^{\frac{1}{k_+-k_-}} \right\|_\infty  \right). 
\end{split}
\ee
Next we use \eqref{est_r2} to first obtain from \eqref{(2.132)}
\be\label{est_r4}
\begin{split}
\frac{( r_{n+1})^2(\u,\v)}{\v^2}&\ge \frac{\hat r^2(\u, \v_0)}{\v^2} + \frac{\v^2-(\v_0)^2}{2\v^2} \inf_{\uC \cup \C} \left[ \frac{1}{\v} \pa_\v (\hat r^2)\right]% \frac{1}{\v^2} \int_{\v_0}^\v \tilde\v  \frac{1}{\tilde\v} \pa_\v (\hat r^2) (\u_0, \v) d \tilde \v  
\\
& -  \frac12\frac{\delta}{\v_0}   \left( \frac{ \|\Om\|^2_\infty}{2} + \frac{2\pi}{\v_0^\beta}  \left\| \v^{1+\beta} \frac{ \Om^2  }{r^{2\e}  } (f^+  f^-)^{\frac{1}{k_+-k_-}} \right\|_\infty  \right)\\
&\ge \frac13\min\left\{ \frac{1}{\| \frac{\v^2}{\hat r^2}|_{\uC \cup \C}\|_\infty }, \frac{1}{\| \frac{\v}{\pa_\v(\hat r^2)}|_{\uC \cup \C}\|_\infty } \right\} 
%\mathbf{1}_{\frac{\v}{2}\le \v_0 <\v} \left( \frac14 \inf_{\pa\mathcal D} \left[ \frac{\hat r^2}{\v^2} \right]\right) + \mathbf{1}_{\v_0\le  \v_0 < \frac{\v}{2}} \left(\frac38 \inf_{\pa\mathcal D} \left[ \frac{1}{\v} \pa_\v (\hat r^2)\right] \right)% \frac{1}{\v^2} \int_{\v_0}^\v \tilde\v  \frac{1}{\tilde\v} \pa_\v (\hat r^2) (\u_0, \v) d \tilde \v  
 -  \frac12\frac{\delta}{\v_0}   \left( \frac{ \|\Om\|^2_\infty}{2} + \frac{2\pi}{\v_0^\beta}  \left\| \v^{1+\beta} \frac{ \Om^2  }{r^{2\e}  } (f^+  f^-)^{\frac{1}{k_+-k_-}} \right\|_\infty  \right),
\end{split}
\ee
where we have used the following: if $\frac{\v}{\sqrt3}\le \v_0$, then $\frac{\hat r^2(\u, \v_0)}{\v^2} \ge \frac{\hat r^2(\u, \v_0)}{3(\v_0)^2}\ge\frac13 \frac{1}{\| \frac{\v^2}{\hat r^2}|_{\uC \cup \C}\|_\infty }$, while if 
$\v_0 < \frac{\v}{\sqrt3}$, 
$\frac{\v^2-(\v_0)^2}{2\v^2} \inf_{\uC \cup \C} \left[ \frac{1}{\v} \pa_\v (\hat r^2)\right] \ge \frac13  \frac{1}{\| \frac{\v}{\pa_\v(\hat r^2)}|_{\uC \cup \C}\|_\infty } $. %Since $\frac{\hat r}{\v}\sim 1$ and $\pa_\v \hat r\sim 1$, from \eqref{est_r3} and \eqref{est_r4} we deduce that $\frac{\tilde r^2}{\v^2}\sim 1$ for sufficiently small $\delta,\kappa$. Together with $ \frac{1}{\v}\pa_\v(\tilde r^2) \sim 1$, we also have $\pa_\v \tilde r\sim 1$. 
As done in \eqref{est_r22}, we deduce 
\begin{align}
&\frac{\v^2}{( r_{n+1})^2(\u,\v)}\le 3\max \left\{ \left\| \frac{\v^2}{\hat r^2}|_{\uC \cup \C}\right\|_\infty,\left\| \frac{\v}{\pa_\v(\hat r^2)}|_{\uC \cup \C}\right\|_\infty  \right\}\label{est_r44}\\
&+ 9 \frac{\delta}{\v_0}  \max \left\{ \left\| \frac{\v^2}{\hat r^2}|_{\uC \cup \C}\right\|_\infty,\left\| \frac{\v}{\pa_\v(\hat r^2)}|_{\uC \cup \C}\right\|_\infty  \right\}^2 \left( \frac{ \|\Om\|^2_\infty}{2} + \frac{2\pi}{\v_0^\beta}  \left\| \v^{1+\beta} \frac{ \Om^2  }{r^{2\e}  } (f^+  f^-)^{\frac{1}{k_+-k_-}} \right\|_\infty  \right). \notag
\end{align}

We now differentiate \eqref{recast1} with respect to $\pa_\v$ and integrate:
\be\label{E:TWODER}
\begin{split}x
\pa_\v^2 (( r_{n+1})^2) (\u,\v) = \pa_\v^2 (\hat r^2) (\u_0, \v)+%(\u_0-\kappa(\v-\v_0), \v) +
 \int_{\u_0}^\u \pa_\v\left[  - \frac{\Om^2}{2} + \frac{2\pi \Om^2  }{r^{2\e}  } (f^+  f^-)^{\frac{1}{k_+-k_-}} \right] d\tilde\u . 
\end{split}
\ee
%(We obtain the same formula by direct differentiation of \eqref{(2.131)}.)  
Now the estimation of $\pa_\v^2((r_{n+1})^2)$ follows similarly. Using $\frac{\u-\u_0}{\v}\le \frac{\delta}{\v_0}$, 
\be
\begin{split}\label{est_r5}
&\|\pa_\v^2 (( r_{n+1})^2)\|_\infty  \le \|\pa_\v^2 (\hat r^2)|_{\C} \|_\infty \\
&\quad+ \frac{\delta}{\v_0} \left(  \frac{\|\Om\|^2_\infty}{\v_0^\theta} \|\v^{1+\theta}\pa_\v \log \Om\|_\infty  +  \frac{2\pi}{\v_0^{1+\beta}}  \left\| \v^{2+\beta} \pa_\v\left(  \frac{ \Om^2  }{r^{2\e}  } (f^+  f^-)^{\frac{1}{k_+-k_-}} \right)\right\|_\infty   \right). 
%&\quad+ 2\pi  \frac{\delta}{\v_0}  \|\Om\|^2_\infty  \left\| \frac{\v}{r}\right\|_\infty^{2\e} (\|\v^{N_+} f^+\|_\infty \|\v^{N_-} f^-\|_\infty )^{\frac{1}{k_+-k_-}} \\
%&\quad\qquad\qquad\qquad \cdot\left( \e\left\| \frac{\v}{r}\right\|_\infty^{2} \left\| \frac{\pa_\v(r^2)}{\v}\right\|_\infty + \frac{ \|\v \pa_\v \log f^+\|_\infty+ \|\v\pa_\v \log f^-\|_\infty}{k_+-k_-}  \right)
\end{split}
\ee
The third derivative $\v\pa_\v^3((r_{n+1})^2)$ can be estimated analogously: 
\be\label{est_r6}
\begin{split}
\|\v\pa_\v^3 ((r_{n+1})^2)\|_\infty  \le & \|\v\pa_\v^3 (\hat r^2)|_{\C} \|_\infty + \frac{\delta}{\v_0}  \|\Om\|^2_\infty \left( \frac{ \|\v^{2+\theta}\pa_\v^2 \log \Om\|_\infty}{\v_0^\theta} 
+ \frac{2  \|\v^{1+\theta}\pa_\v \log \Om\|^2_\infty}{\v_0^{2\theta}} \right)\\
&+  \frac{\delta}{\v_0}  \frac{2\pi}{\v_0^{1+\beta}}  \left\| \v^{3+\beta} \pa_\v^2\left(  \frac{ \Om^2  }{r^{2\e}  } (f^+  f^-)^{\frac{1}{k_+-k_-}} \right)\right\|_\infty . 
\end{split}
\ee

We proceed in the same way for $ \Om_{n+1}$.  By integrating \eqref{recast2} along an ingoing curve we obtain the expression for $\pa_\v \log \Om_{n+1} (\u,\v) $
\be\label{(2.134)}
\begin{split}
\pa_\v \log  \Om_{n+1} (\u,\v) &= \pa_\v (\log \hat \Om) (\u_0, \v) + \int_{\u_0}^\u \left[  \frac{ \Om^2}{2 r^2} \frac{m}{r} - (1+\e) \pi \frac{ \Om^2 }{r^{2+2\e}}(f^+  f^-)^{\frac{1}{k_+-k_-}}  \right] d\tilde\u
\end{split}
\ee with
\be\label{(2.135)}
 \Om_{n+1} (\u,\v) =\hat \Om(\u, \v_0) e^{\int_{\v_0}^\v  \pa_\v \log \Om_{n+1} (\u,\tilde\v) d\tilde \v }.
\ee 
%and 
%\be
%m(\u,\v) = m (\u,\v^*) + \int_{\v^*}^\v 2\pi (f^+f^-)^\frac{1}{k_+-k_-} \left[ \frac{1 }{r^{2\e}} \pa_\v r -  (1+\e)^2\frac{1 }{r^{4\e}} \left(\frac{f^+}{f^-}\right)^\frac12 \pau r\right] \,d\tilde\v 
%\ee where we have used \eqref{1.59}. 

To estimate $\v^{1+\theta}\pa_\v\log \Om_{n+1}$, we start with the integral term of \eqref{(2.134)}. By taking the sup norm, and using $\frac{\u-\u_0}{\v}\le \frac{\delta}{\v_0}$, we obtain 
\be\label{est_Om1}
\begin{split}
&\left| \v^{1+\theta} \int_{\u_0}^\u \left[  \frac{ \Om^2}{2 r^2} \frac{m}{r} - (1+\e) \pi \frac{ \Om^2 }{r^{2+2\e}}(f^+  f^-)^{\frac{1}{k_+-k_-}}  \right] d\tilde\u \right| \\
&\le \frac{\delta}{\v_0}  \left(  \frac12\frac{1}{\v^{1-\theta}} \| \Om\|^2_\infty \left\|\frac{\v}{r}\right\|^3_\infty \|m\|_\infty + \frac{(1+\e)\pi }{\v^{1-\theta+\beta}}  \left\| \v^{3+\beta} \frac{ \Om^2  }{r^{2+2\e}  } (f^+  f^-)^{\frac{1}{k_+-k_-}} \right\|_\infty  \right) \\
% \left\| \frac{\v}{r}\right\|_\infty^{2+2\e} (\|\v^{N_+} f^+\|_\infty \|\v^{N_-} f^-\|_\infty )^{\frac{1}{k_+-k_-}} \right]
\end{split}
\ee
and hence 
\be\label{est_Om2}
\begin{split}
&\|\v^{1+\theta} \pa_\v \log \Om_{n+1}\|_\infty\le \|\v^{1+\theta} \pa_\v \log \hat \Om|_{\C}\|_\infty\\
&+\frac{\delta}{\v_0}  \left(  \frac12\frac{1}{\v_0^{1-\theta}} \| \Om\|^2_\infty \left\|\frac{\v}{r}\right\|^3_\infty \|m\|_\infty + \frac{(1+\e)\pi }{\v_0^{1-\theta+\beta}}  \left\| \v^{3+\beta} \frac{ \Om^2  }{r^{2+2\e}  } (f^+  f^-)^{\frac{1}{k_+-k_-}} \right\|_\infty  \right) . 
%+\frac{\delta}{\v_0}  \frac{ \| \Om\|^2_\infty}{\v_0^{1-\theta}} \left[ \frac12 \left\|\frac{\v}{r}\right\|^3_\infty \|m\|_\infty + (1+\e)\pi   \left\| \frac{\v}{r}\right\|_\infty^{2+2\e} (\|\v^{N_+} f^+\|_\infty \|\v^{N_-} f^-\|_\infty )^{\frac{1}{k_+-k_-}} \right]
\end{split}
\ee
For $\Om_{n+1}$, we first estimate the integral of \eqref{(2.135)}. 
\[
\begin{split}
\left| \int_{\v_0}^\v  \pa_\v \log \Om_{n+1} (\u,\tilde\v) d\tilde \v\right|& =\left| \int_{\v_0}^\v \frac{1}{\tilde\v^{1+\theta}} \tilde\v^{1+\theta}\pa_\v \log \Om_{n+1} (\u,\tilde\v) d\tilde \v\right|  \le \|\v^{1+\theta} \pa_\v \log  \Om_{n+1}\|_\infty \frac{1}{\theta} \frac{1}{\v_0^\theta}. 
\end{split} 
\]
Therefore we deduce 
\be\label{est_Om0}
\begin{split}
&\|\log \Om_{n+1}\|_\infty \le \|\log \hat\Om|_{\uC \cup \C}\|_\infty +\frac{1}{\theta} \frac{1}{(\v_0)^\theta} \|\v^{1+\theta} \pa_\v \log  \Om_{n+1}|_{\C}\|_\infty  \\
&+ \frac{1}{\theta} \frac{1}{\v_0^\theta}\frac{\delta}{\v_0} 
\left(  \frac12\frac{1}{\v_0^{1-\theta}} \| \Om\|^2_\infty \left\|\frac{\v}{r}\right\|^3_\infty \|m\|_\infty + \frac{(1+\e)\pi }{\v_0^{1-\theta+\beta}}  \left\| \v^{3+\beta} \frac{ \Om^2  }{r^{2+2\e}  } (f^+  f^-)^{\frac{1}{k_+-k_-}} \right\|_\infty  \right) . 
% \frac{ \| \Om\|^2_\infty}{\v_0^{1-\theta}} \left[ \frac12 \left\|\frac{\v}{r}\right\|^3_\infty \|m\|_\infty + (1+\e)\pi   \left\| \frac{\v}{r}\right\|_\infty^{2+2\e} (\|\v^{N_+} f^+\|_\infty \|\v^{N_-} f^-\|_\infty )^{\frac{1}{k_+-k_-}} \right]
\end{split}
\ee
%and also we see $\tilde\Om \sim 1$ from \eqref{(2.135)}. 
To estimate $\pa_\v^2\log\Om_{n+1}$, we first observe that 
\be\label{E:TWODER2}
\begin{split}
\pa_\v^2 \log  \Om_{n+1} (\u,\v) &= \pa_\v^2 (\log \hat \Om) (\u_0, \v) 
+ \int_{\u_0}^\u \pa_\v\left[  \frac{ \Om^2}{2 r^2} \frac{m}{r} - (1+\e) \pi \frac{ \Om^2 }{r^{2+2\e}}(f^+  f^-)^{\frac{1}{k_+-k_-}}  \right] d\tilde\u,
\end{split}
\ee 
wherefrom we deduce 
\be\label{est_Om3}
\begin{split}
&\|\v^{2+\theta}\pa_\v^2 \log \Om_{n+1}\|_\infty \le \|\v^{2+\theta}\pa_\v^2 \log \hat\Om|_{\C}\|_\infty \\
& +  \frac{\delta}{\v_0} \frac12 \left(  \|\Om\|_\infty^2\left\|\frac{\v}{r}\right\|^3_\infty\frac{ \|\v^{1+\beta}\pa_\v m \|_\infty}{\v_0^{1-\theta+\beta}} + \left\| \v^{4}\pa_\v\left(\frac{\Om^2}{r^3}\right)  \right\|_\infty \frac{\|m\|_\infty}{\v_0^{1-\theta}}\right) \\ %\|\Om\|_\infty^2\left\|\frac{\v}{r}\right\|^3_\infty\left( \frac{ \|\v^{1+\beta}\pa_\v m \|_\infty}{2\v_0^{1-\theta+\beta}} + \frac{\|\v^{1+\theta}\pa_\v\log\Om\|_\infty\|m\|_\infty}{\v_0} + \frac{3\|m\|_\infty\|\pa_\v r\|_\infty}{2\v_0^{1-\theta}}\left\|\frac{\v}{r}\right\|_\infty\right)\\
& +  \frac{\delta}{\v_0}  \frac{(1+\e)\pi}{\v_0} \left\| \v^{4+\beta} \pa_\v\left(\frac{ \Om^2  }{r^{2+2\e}  } (f^+  f^-)^{\frac{1}{k_+-k_-}}\right) \right\|_\infty.  
%&+ \frac{\delta}{\v_0} \frac{ \|\Om\|^2_\infty}{\v_0} \|\v^{1+\theta}\pa_\v \log \Om\|_\infty \left(  \left\|\frac{\v}{r}\right\|^3_\infty \|m\|_\infty +  2(1+\e)\pi   \left\| \frac{\v}{r}\right\|_\infty^{2+2\e} (\|\v^{N_+} f^+\|_\infty \|\v^{N_-} f^-\|_\infty )^{\frac{1}{k_+-k_-}} \right)\\
%&+  \frac{\delta}{\v_0}  \frac{\|\Om\|^2_\infty}{2\v_0^{1-\theta}}\Big\{ \left\| \frac{\v}{r}\right\|_\infty^{3} \|\v\pa_\v m\|_\infty + \frac32\| m\|_\infty  \left\| \frac{\v}{r}\right\|_\infty^{5} \left\| \frac{\pa_\v(r^2)}{\v}\right\|_\infty \\
%&\quad\quad\qquad\qquad\qquad + 2\pi (1+\e)  \left\| \frac{\v}{r}\right\|_\infty^{2+2\e} (\|\v^{N_+} f^+\|_\infty \|\v^{N_-} f^-\|_\infty )^{\frac{1}{k_+-k_-}} \\
%&\quad\qquad\qquad\qquad\qquad \cdot\left( (1+\e) \left\| \frac{\v}{r}\right\|_\infty^{2} \left\| \frac{\pa_\v(r^2)}{\v}\right\|_\infty + \frac{ \|\v \pa_\v \log f^+\|_\infty+ \|\v\pa_\v \log f^-\|_\infty}{k_+-k_-}  \right)\Big\}
\end{split}
\ee
Similarly one can derive 
\be\label{est_Om4}
\begin{split}
&\|\v^{3+\theta}\pa_\v^3 \log  \Om_{n+1}\|_\infty \le \|\v^{3+\theta}\pa_\v^3 \log \hat \Om|_{\C}\|_\infty  \\
&+  \frac{\delta}{\v_0} \frac12 \left(  \|\Om\|_\infty^2\left\|\frac{\v}{r}\right\|^3_\infty\frac{ \|\v^{2+\beta}\pa_\v^2 m \|_\infty}{\v_0^{1-\theta+\beta}} + 2 \left\| \v^{4}\pa_\v\left(\frac{\Om^2}{r^3}\right)  \right\|_\infty \frac{ \|\v^{1+\beta}\pa_\v m \|_\infty}{\v_0^{1-\theta+\beta}}  +  \left\| \v^{5}\pa_\v^2\left(\frac{\Om^2}{r^3}\right)  \right\|_\infty \frac{\|m\|_\infty}{\v_0^{1-\theta}}\right) \\
&+    \frac{\delta}{\v_0}  \frac{(1+\e)\pi}{\v_0} \left\| \v^{5+\beta} \pa_\v^2\left(\frac{ \Om^2  }{r^{2+2\e}  } (f^+  f^-)^{\frac{1}{k_+-k_-}}\right) \right\|_\infty . 
\end{split}
\ee

 We now collect \eqref{est_r1}, \eqref{est_r22}, \eqref{est_r3}, \eqref{est_r44}, \eqref{est_r5}, \eqref{est_r6}, \eqref{est_Om2}, \eqref{est_Om0}, \eqref{est_Om3}, \eqref{est_Om4} and use \eqref{est_S1}, \eqref{est_S2}, \eqref{est_aux}, \eqref{est_m}, \eqref{est_m1}  as well as $\vertiii{[f^\pm_{n},\Om_{n},r_{n}]}\le 2A\vertiii{[\hat f^\pm, \hat\Om, \hat r]|_{\uC \cup \C}}$ to  deduce  \eqref{bound:fn+1} for sufficiently small $\delta>0$. 
This concludes the proof. 
%The boundedness of $\pa_\u r$ and $\pa_\u\Om$ follows from the integration of \eqref{(2.130)} and \eqref{2.132}. 
\end{proof}

%%%%%%%%%%%%%%%%%%%%%%%%%%%%%%%
%%%%%%%%%%%%%%%%%%%%%%%%%%%%%%%
%%%%%%%%%%%%%%%%%%%%%%%%%%%%%%%

\subsubsection{Convergence of the iteration scheme}

%%%%%%%%%%%%%%%%%%%%%%%%%%%%%%%
%%%%%%%%%%%%%%%%%%%%%%%%%%%%%%%
%%%%%%%%%%%%%%%%%%%%%%%%%%%%%%%

From Proposition \ref{prop:sol}, the solutions $[f^\pm_{n+1},\Om_{n+1},r_{n+1}]$ of \eqref{eq:fn+1}, \eqref{recast1}, \eqref{recast2} have the uniform bounds $\vertiii{[f^\pm_{n+1},\Om_{n+1},r_{n+1}]}\le 2A\vertiii{[\hat f^\pm, \hat\Om, \hat r]|_{\uC \cup \C}} $ for all $n\ge 0$. In this section, we show the convergence of the sequence of approximations $[f^\pm_{n+1},\Om_{n+1},r_{n+1}]$. We will estimate the difference between  $[f^\pm_{n+1},\Om_{n+1},r_{n+1}]$ and $[f^\pm_{n},\Om_{n},r_{n}]$.
Let 
\be\label{def:diff}
[\triangle f_{n+1}^\pm, \triangle \Om_{n+1},\triangle r_{n+1} ]:=[\log f^\pm_{n+1}-\log f^\pm_n, \log \Om_{n+1}-\log \Om_n, (r_{n+1})^2-(r_n)^2]
\ee
for $n\ge0$. Our next task is to prove the corresponding difference bounds. 

%%%%%%%%%%%%%%%%%%%%%%%%%%%%%%%
%%%%%%%%%%%%%%%%%%%%%%%%%%%%%%%

\begin{proposition}\label{prop:diff} 
Let $[f^\pm_{n+1},\Om_{n+1},r_{n+1}]$ be the solution \eqref{eq:fn+1}, \eqref{recast1}, \eqref{recast2} enjoying the uniform bounds $\vertiii{[f^\pm_{n+1},\Om_{n+1},r_{n+1}]}\le 2A\vertiii{[\hat f^\pm, \hat\Om, \hat r]|_{\uC \cup \C}} $ for all $n\ge 0$. Then the difference norm for  $[\triangle f_{n+1}^\pm, \triangle \Om_{n+1},\triangle r_{n+1} ]$ defined in \eqref{def:diff} satisfies the following recursive inequality
\be\label{recursive_ineq}
a_{n+1} \le C  \frac{\delta}{\v_0} a_n \ \  \text{ for  } \ \ n\ge 1
\ee
where 
\be
a_{n+1}=\|\triangle f_{n+1}^\pm\|_\infty+ \|\v^{-2}{\triangle r_{n+1}}\|_\infty+ \|\v^{-1}{\pa_\v\triangle r_{n+1}}\|_\infty+\|\triangle \Om_{n+1}\|_\infty + 
\|\v^{1+\theta}\pa_\v\triangle \Om_{n+1}\|_\infty
\ee
and $C$ does not depend on $n$ but only on $2A\vertiii{[\hat f^\pm, \hat\Om, \hat r]|_{\uC \cup \C}}$.  %but independent of $n$. 
 %satisfy \eqref{eq:f_d}-\eqref{eq:om_d} with the zero initial boundary data $[\triangle f_{n+1}_\pm, \triangle \Om_{n+1},\triangle r_{n+1} ] |_{\uC \cup \C} =0$. 
\end{proposition}

%%%%%%%%%%%%%%%%%%%%%%%%%%%%%%%

The proof relies on $C^0$ estimates of $[\triangle f_{n+1}^\pm]$ and $C^1$ estimates of $[\triangle \Om_{n+1},\triangle r_{n+1} ]$ in the spirit of Lemma \ref{lem:fn+1} and Lemma \ref{lem:Omrn+1}. Before giving the proof, we present some preliminary estimates. First, we note that $[\triangle f_{n+1}^\pm, \triangle \Om_{n+1},\triangle r_{n+1} ]$ satisfy
\begin{align}
&\pa_\u \triangle f_{n+1}^\pm + k_\pm \mathcal U_n \pa_\v \triangle f_{n+1}^\pm + k_\pm (\mathcal U_n-\mathcal U_{n-1})\pa_\v \log f^\pm_n  \label{eq:f_d}  \\
&\qquad\quad \  \pm 2k_\pm \left[  (2\pa_\v \log \Om_n  - \e\frac{\pa_\v (r_n)^2}{(r_n)^2})  \mathcal U_n -(2\pa_\v \log \Om_{n-1}  - \e\frac{\pa_\v (r_{n-1})^2}{(r_{n-1})^2})  \mathcal U_{n-1}  \right] =0 , \notag\\
& \pa_\u\pa_\v \triangle r_{n+1}= \frac{(\Om_{n-1})^2-(\Om_n)^2 }{2} + 2\pi\left[  \frac{(\Om_n)^2}{(r_n)^{2\e}} (f^+_nf^-_n) ^{\frac{1}{k_+-k_-}} -\frac{(\Om_{n-1})^2}{(r_{n-1})^{2\e}} (f^+_{n-1}f^-_{n-1}) ^{\frac{1}{k_+-k_-}}  \right], \label{eq:r_d} \\
&  \pa_\u \pa_\v \triangle \Om_{n+1} =\frac12\left[  \frac{ (\Om_n)^2}{ (r_n)^2} \frac{m_n}{r_n} - \frac{ (\Om_{n-1})^2}{ (r_{n-1})^2} \frac{m_{n-1}}{r_{n-1}}  \right]\notag \\
&\qquad\qquad- (1+\e) \pi \left[ \frac{ (\Om_n)^2 }{(r_n)^{2+2\e}}(f^+_n  f^-_n)^{\frac{1}{k_+-k_-}} -  \frac{ (\Om_{n-1})^2 }{(r_{n-1})^{2+2\e}}(f^+_{n-1}  f^-_{n-1})^{\frac{1}{k_+-k_-}} \right], \label{eq:om_d} 
\end{align} 
with $[\triangle f_{n+1}^\pm, \triangle \Om_{n+1},\triangle r_{n+1} ] |_{\uC \cup \C} =0$.

We start with the following inequality, which will allow us to compare the difference of two functions to the logarithm of their ratio. 

%%%%%%%%%%%%%%%%%%%%%%%%%%%%%%%
%%%%%%%%%%%%%%%%%%%%%%%%%%%%%%%

\begin{lemma} Let $L>0$ be given. For any $-1< x \le L$, the following holds 
\be\label{ineq1}
 |x|\le (1+L) |\log (1+ x )|  . 
\ee
\end{lemma}

%%%%%%%%%%%%%%%%%%%%%%%%%%%%%%%

\begin{proof} If $x\in (-1,0)$, we claim $-x \le -(1+L)\log (1+x) $. Letting $h(x) = (1+L)\log (1+x)-x $, we have $h(0)=0$ and $h'(x) = \frac{1+L}{1+x} - 1 >0 $ for $x\in (-1,0)$. Hence $h<0$ for $x\in (-1,0)$. Now let $0\le x \le L$. Then since $h'(x)\ge 0$ for $0\le x \le L$, $h\ge 0$ for $0\le x \le L$, which shows  \eqref{ineq1}. 
\end{proof}

%%%%%%%%%%%%%%%%%%%%%%%%%%%%%%%
%%%%%%%%%%%%%%%%%%%%%%%%%%%%%%%

%%%%%%%%%%%%%%%%%%%%%%%%%%%%%%%
%%%%%%%%%%%%%%%%%%%%%%%%%%%%%%%

\begin{lemma}\label{lem:diff_to_log} For any $b>0$ let
\be\label{def:ell}
\ell_\Om:= \max\left\{ \sup_{n} \frac{(\Om_{n})^2}{(\Om_{n-1})^2}, \  \sup_{n} \frac{(\Om_{n-1})^2}{(\Om_{n})^2}\right\}, \quad \ell^\pm=\max\left\{ \sup_n \left( \frac{f_n^\pm}{f_{n-1}^\pm}\right)^b, \ \sup_n \left(\frac{f_{n-1}^\pm}{f_{n}^\pm}\right)^b \right\},
\ee
and 
\be
\ell_R :=\max\left\{\sup_{n} \left\| \frac{\v^b}{(r_{n-1})^b} \frac{(\v^{-1}r_n)^b - (\v^{-1}r_{n-1})^b}{ (\v^{-1}r_n)^2 -(\v^{-1} r_{n-1})^2 } \right\|_\infty, \ \sup_{n} \left\| \frac{\v^b}{(r_{n})^b} \frac{(\v^{-1}r_n)^b - (\v^{-1}r_{n-1})^b}{ (\v^{-1}r_n)^2 -(\v^{-1} r_{n-1})^2 } \right\|_\infty  \right\}.
\ee
 Note that $\ell_\Om$, $\ell^\pm$, $\ell_R$ are finite by the uniform bound $\vertiii{[f^\pm_{n+1},\Om_{n+1},r_{n+1}]}\le 2A\vertiii{[\hat f^\pm, \hat\Om, \hat r]|_{\uC \cup \C}} $. 
 Then the following bounds hold:
\begin{align}
&\left\|  \left( \frac{f_n^\pm}{f_{n-1}^\pm}\right)^b - 1\right\|_\infty, \quad  \left\|  \left(\frac{f_{n-1}^\pm}{f_{n}^\pm} \right)^b- 1\right\|_\infty \le b \ell^\pm \| \triangle f_{n} \|_\infty, \label{est:frac1}\\
&\left\|  \frac{(\Om_{n})^2}{(\Om_{n-1})^2} - 1\right\|_\infty , \quad\left\|  \frac{(\Om_{n-1})^2}{(\Om_{n})^2}- 1\right\|_\infty \le 2 \ell_\Om\| \triangle \Om_{n} \|_\infty, \label{est:frac2} \\
&\left\| \frac{(r_n)^b}{(r_{n-1})^b} -1 \right\|_\infty, \quad \left\| \frac{(r_{n-1})^b}{(r_{n})^b} -1 \right\|_\infty  \le \ell_R  \left\| \v^{-2}\triangle r_{n} \right\|_\infty. \label{est:frac3}
\end{align}
\end{lemma}

%%%%%%%%%%%%%%%%%%%%%%%%%%%%%%%

\begin{proof}
\eqref{est:frac1} and \eqref{est:frac2} are direct consequences of \eqref{ineq1}, \eqref{def:ell} with $x+1=( \frac{f_n^\pm}{f_{n-1}^\pm})^b$, $ (\frac{f_{n-1}^\pm}{f_{n}^\pm} )^b$, $ \frac{(\Om_{n})^2}{(\Om_{n-1})^2}$, $ \frac{(\Om_{n-1})^2}{(\Om_{n})^2}$. \eqref{est:frac3} directly follows by writing 
\[
\frac{(r_n)^b}{(r_{n-1})^b} -1 = \frac{\v^b}{(r_{n-1})^b} \frac{(\v^{-1}r_n)^b - (\v^{-1}r_{n-1})^b}{ (\v^{-1}r_n)^2 -(\v^{-1} r_{n-1})^2 } \left(\v^{-2} [ (r_n)^2 - (r_{n-1})^2 ]\right).
\]
\end{proof}

%%%%%%%%%%%%%%%%%%%%%%%%%%%%%%%
%%%%%%%%%%%%%%%%%%%%%%%%%%%%%%%

We are now ready to prove Proposition \ref{prop:diff}. 

%%%%%%%%%%%%%%%%%%%%%%%%%%%%%%%

\begin{proof}[Proof of Proposition \ref{prop:diff}] We start with $\|\triangle f_{n+1}^\pm\|_\infty$.  We recall the notations from Lemma \ref{lem:fn+1}. 
By integrating \eqref{eq:f_d} along the characteristics and using the zero data, we have the representation of $\triangle f_{n+1}^\pm$: 
\be
\begin{split}
\triangle f_{n+1}^\pm (\u,\v)= k_\pm \int_{\u_*}^\u (S_1 +S_2) (s, \mathfrak{\v}_\pm (s)) ds
\end{split}
\ee
where 
\be
\begin{split}
S_1& = -  (\mathcal U_n-\mathcal U_{n-1})\pa_\v \log f^\pm_n , \\
S_2 &= \mp 2 \left[  (2\pa_\v \log \Om_n  - \e\frac{\pa_\v (r_n)^2}{(r_n)^2})  \mathcal U_n -(2\pa_\v \log \Om_{n-1}  - \e\frac{\pa_\v (r_{n-1})^2}{(r_{n-1})^2})  \mathcal U_{n-1}  \right] . 
\end{split}
\ee
Using $\mathfrak{\v}_\pm(s)\ge \v_*$ for all $\u_* \le s\le \u$ and $\u-\u_* \le \delta$, 
\be
\begin{split}
\left| \int_{\u_*}^\u S_1 (s, \mathfrak{\v}_\pm (s)) ds   \right| &= \left| \int_{\u_*}^\u \frac{1}{\mathfrak{\v}_\pm} (\mathcal U_n-\mathcal U_{n-1})  \mathfrak{\v}_\pm  \pa_\v \log f^\pm_n  (s, \mathfrak{\v}_\pm (s)) ds   \right| \\
& \le   \|\mathcal U_{n-1}\|_\infty \|\v \pa_\v \log f_n^\pm \|_\infty \left\| \frac{\mathcal U_n}{\mathcal U_{n-1}} -1 \right\|_\infty \frac{\delta}{\v_0}. 
\end{split}
\ee
For $\left\| \frac{\mathcal U_n}{\mathcal U_{n-1}} -1 \right\|_\infty$, we rewrite 
\be\label{diff_U0}
\begin{split}
\frac{ \mathcal U_n}{ \mathcal U_{n-1}}-1 &=\left(  \frac{\Om_n}{\Om_{n-1}}\right)^2 \left( \frac{f^+_n}{f^+_{n-1}}\right)^\frac12 \left( \frac{f^-_{n-1}}{f^-_{n}}\right)^\frac12 \left( \frac{(r_{n-1})^2}{(r_n)^2}\right)^\e -1 \\
&= \left[\left(  \frac{\Om_n}{\Om_{n-1}}\right)^2 -1  \right]\left( \frac{f^+_n}{f^+_{n-1}}\right)^\frac12 \left( \frac{f^-_{n-1}}{f^-_{n}}\right)^\frac12 \left( \frac{(r_{n-1})^2}{(r_n)^2}\right)^\e \\
&+ \left[ \left( \frac{f^+_n}{f^+_{n-1}}\right)^\frac12-1 \right] \left( \frac{f^-_{n-1}}{f^-_{n}}\right)^\frac12 \left( \frac{(r_{n-1})^2}{(r_n)^2}\right)^\e\\
&+ \left[ \left( \frac{f^-_{n-1}}{f^-_{n}}\right)^\frac12-1 \right] \left( \frac{(r_{n-1})^2}{(r_n)^2}\right)^\e
+ \left( \frac{(r_{n-1})^2}{(r_n)^2}\right)^\e -1 . 
\end{split}
\ee
Note that factors next to the rectangular  brackets are all uniformly bounded. 
By Lemma \ref{lem:diff_to_log}, we deduce that 
\be\label{diff_U}
\left\| \frac{\mathcal U_n}{\mathcal U_{n-1}} -1 \right\|_\infty \le c_1 (\|\triangle f_{n}\|_\infty + \|\triangle f_{n}\|_\infty+\|\v^{-2}\triangle r_{n}\|_\infty)
\ee
where $c_1>0$ is independent of $n$ and hence using the uniform bounds of the uniform bounds of $\vertiii{[f^\pm_{n-1},\Om_{n-1},r_{n-1}]}$ and $\vertiii{[f^\pm_{n},\Om_{n},r_{n}]}$, we have 
\be\label{diff_est2}
\left| \int_{\u_*}^\u S_1 (s, \mathfrak{\v}_\pm (s)) ds   \right|\le c_2 \frac{\delta}{\v_0}(\|\triangle f_{n}\|_\infty + \|\triangle \Om_{n}\|_\infty+\|\v^{-2}\triangle r_{n}\|_\infty)
\ee
where $c_2>0$ is independent of $n$. For $S_2$, we first rewrite 
\be
\begin{split}
\mp\frac{ S_2}{2}&= \left( 2\pa_\v \triangle \Om_{n}  - \e\left( \frac{\pa_\v \triangle r_{n}}{(r_n)^2} + \frac{\pa_\v (r_{n-1})^2}{(r_{n-1})^2} \left[ \frac{(r_{n-1})^2}{(r_n)^2}-1 \right]\right)\right)  \mathcal U_n  \\
&+(2\pa_\v \log \Om_{n-1}  - \e\frac{\pa_\v (r_{n-1})^2}{(r_{n-1})^2}) \mathcal U_{n-1} \left( \frac{\mathcal U_n}{\mathcal U_{n-1}} - 1 \right) . 
\end{split}
\ee
Now by using the uniform bounds on $\vertiii{[f^\pm_{n-1},\Om_{n-1},r_{n-1}]}$ and $\vertiii{[f^\pm_{n},\Om_{n},r_{n}]}$, 
Lemma \ref{lem:diff_to_log}, \eqref{diff_U}, and $\frac{\u-\u_*}{\v_*}\le  \frac{\delta}{\v_0}$, we deduce 
\be\label{diff_est3}
\begin{split}
\left| \int_{\u_*}^\u S_2 (s, \mathfrak{\v}_\pm (s)) ds \right| \le c_3 \frac{\delta}{\v_0}( \|\v^{1+\theta}\pa_\v \triangle \Om_{n}\|_\infty+\|\v^{-1}\pa_\v\triangle r_{n}\|_\infty+\|\triangle \Om_{n}\|_\infty + \|\triangle f_{n}\|_\infty+\|\v^{-2}\triangle r_{n}\|_\infty). 
\end{split}
\ee

We now estimate $ \|\v^{-2}{\triangle r_{n+1}}\|_\infty+ \|\v^{-1}{\pa_\v\triangle r_{n+1}}\|_\infty$. We recall the notations used in the proof of Lemma \ref{lem:Omrn+1}. 
We first integrate \eqref{eq:r_d} along an ingoing null curve from the initial point $\u_0=\u_0$ to $\u$ to obtain 
\be\label{eq:Rn+1}
\pa_\v \triangle r_{n+1}(\u,\v) = \int_{\u_0}^\u S_3 +S_4 \, d\tilde \u 
\ee
where 
\be
\begin{split}
S_3&=- \frac{( \Om_{n-1})^2}{2} \left[\frac{(\Om_n)^2}{(\Om_{n-1})^2}-1 \right] , % \frac{(\Om_{n-1})^2-(\Om_n)^2 }{2}
  \\
S_4&= 2\pi \frac{(\Om_{n-1})^2}{(r_{n-1})^{2\e}} (f^+_{n-1}f^-_{n-1}) ^{\frac{1}{k_+-k_-}}  \left(  \frac{(\Om_n)^2}{(\Om_{n-1})^2} \frac{(r_{n-1})^{2\e}}{(r_n)^{2\e}} (\frac{f^+_n}{f^+_{n-1}})^{\frac{1}{k_+-k_-}} (\frac{f^-_n}{f^-_{n-1}})^{\frac{1}{k_+-k_-}} -1\right). 
\end{split}
\ee
The last factor of $S_4$ can be rewritten via add and subtract trick as done for $\frac{ \mathcal U_n}{ \mathcal U_{n-1}}-1 $ in \eqref{diff_U0}. 
Then using $\frac{\u-\u_0}{\v} \le \frac{\delta}{\v_0}$, Lemma \ref{lem:diff_to_log}, the uniform bounds, we deduce that 
\be\label{diff_est4}
\|\v^{-1} \pa_\v \triangle r_{n+1}\|_\infty\le c_4  \frac{\delta}{\v_0}(\|\triangle f_{n}\|_\infty + \|\triangle \Om_{n}\|_\infty+\|\v^{-2}\triangle r_{n}\|_\infty). 
\ee
By integrating \eqref{eq:Rn+1} with respect to $\v$, we have $
\triangle r_{n+1} (\u,\v)=\int_{\v_0}^\v  \pa_\v \triangle r_{n+1}(\u,\tilde \v) d\tilde \v $. Then by writing 
$\v^{-2}\triangle r_{n+1} (\u,\v)=\v^{-2}\int_{\v_0}^\v \tilde\v \tilde\v^{-1} \pa_\v \triangle r_{n+1}(\u,\tilde \v) d\tilde \v $, we deduce that
\be\label{diff_est5}
\|\v^{-2}  \triangle r_{n+1}\|_\infty\le \frac{c_4}{2}  \frac{\delta}{\v_0}(\|\triangle f_{n}\|_\infty + \|\triangle \Om_{n}\|_\infty+\|\v^{-2}\triangle r_{n}\|_\infty). 
\ee

The estimates for $\|\triangle \Om_{n+1}\|_\infty + 
\|\v^{1+\theta}\pa_\v\triangle \Om_{n+1}\|_\infty$ can be derived in the same fashion. By integrating \eqref{eq:om_d},  
\be
\pa_\v \triangle \Om_{n+1}(\u,\v) = \int_{\u_0}^\u S_5 +S_6\, d\tilde \u, 
\ee
where 
\be
\begin{split}
S_5&=\frac12\left[  \frac{ (\Om_n)^2}{ (r_n)^2} \frac{m_n}{r_n} - \frac{ (\Om_{n-1})^2}{ (r_{n-1})^2} \frac{m_{n-1}}{r_{n-1}}  \right], \\
S_6&=- (1+\e) \pi \frac{ (\Om_{n-1})^2 }{(r_{n-1})^{2+2\e}}(f^+_{n-1}  f^-_{n-1})^{\frac{1}{k_+-k_-}}  \left[  \frac{(\Om_n)^2}{(\Om_{n-1})^2} \frac{(r_{n-1})^{2+2\e}}{(r_n)^{2+2\e}} (\frac{f^+_n}{f^+_{n-1}})^{\frac{1}{k_+-k_-}} (\frac{f^-_n}{f^-_{n-1}})^{\frac{1}{k_+-k_-}}  -  1\right].
\end{split}
\ee
The structure of $S_6$ is similar to $S_4$ in the previous case. Using the uniform bounds for $\vertiii{[f^\pm_{n-1},\Om_{n-1},r_{n-1}]}$ and $\vertiii{[f^\pm_{n},\Om_{n},r_{n}]}$  and Lemma \ref{lem:diff_to_log}, we have 
\be
\left| \v^{1+\theta} \int_{\u_0}^\u S_6 d\tilde \u  \right| \le c_5  \frac{\delta}{\v_0}(\|\triangle f_{n}\|_\infty + \|\triangle \Om_{n}\|_\infty+\|\v^{-2}\triangle r_{n}\|_\infty)
\ee
for some $c_5>0$ independent of $n$. For $S_5$, we rewrite it 
\be
2S_5 = \frac{ (\Om_n)^2}{ (r_n)^3} (m_n-m_{n-1}) +\left(  \frac{ (\Om_n)^2}{ (\Om_{n-1})^2}\frac{(r_{n-1})^3}{ (r_n)^3} - 1\right)  \frac{ (\Om_{n-1})^2}{ (r_{n-1})^3} m_{n-1}.
\ee
The second term is in the form where we can apply Lemma \ref{lem:diff_to_log}. For the first term, note that 
%\bcr
\be
\begin{split}
m_n-m_{n-1} = \int_{\v_0}^\v \left(2\pi (f_n^+f_n^-)^\frac{1}{k_+-k_-} \left[ \frac{1 }{(r_n)^{2\e}} \pa_\v r_n -  \frac{1}{2}\frac{ (\Om_n)^{-2} }{r_n} \left(\frac{f^-_n}{f_n^+}\right)^\frac12 \mathfrak s_n\right]  \right) \\
- \left(2\pi (f_{n-1}^+f_{n-1}^-)^\frac{1}{k_+-k_-} \left[ \frac{1 }{(r_{n-1})^{2\e}} \pa_\v r_{n-1} -  \frac{1}{2}\frac{(\Om_{n-1})^{-2}}{r_{n-1}} \left(\frac{f^-_{n-1}}{f_{n-1}^+}\right)^\frac12 \mathfrak s_{n-1}\right]  \right) d\tilde \v
\end{split}
\ee 
%\ec
where we also recall \eqref{mathfraks} for $\mathfrak s_n$. As done for $\frac{ \mathcal U_n}{ \mathcal U_{n-1}}-1 $ in \eqref{diff_U0}, it is evident that the integrand can be written as the sum of the terms that contain one of the forms in Lemma \ref{lem:diff_to_log} or $\pa_\v \triangle r_{n}$ or $\mathfrak s_n -\mathfrak s_{n-1}$ where 
\[
\mathfrak s_n -\mathfrak s_{n-1}  = \int_{\v_0}^\v \left[  \frac{(\Om_{n-1})^2-(\Om_n)^2}{2} + \frac{2\pi (\Om_n)^2  }{(r_n)^{2\e}  } (f_n^+  f_n^-)^{\frac{1}{k_+-k_-}} -  \frac{2\pi (\Om_{n-1})^2  }{(r_{n-1})^{2\e}  } (f_{n-1}^+  f_{n-1}^-)^{\frac{1}{k_+-k_-}}  \right] d\tilde\v . 
 \]
Following the strategy of the proof of \eqref{est_m} in Lemma \ref{lem:est_m}, using Lemma \ref{lem:diff_to_log} and the uniform bounds of $\vertiii{[f^\pm_{n-1},\Om_{n-1},r_{n-1}]}$ and $\vertiii{[f^\pm_{n},\Om_{n},r_{n}]}$, we deduce 
\be
\|m_n-m_{n-1}\|_\infty \le c_6  \frac{\delta}{\v_0}(\|\triangle f_{n}\|_\infty + \|\triangle \Om_{n}\|_\infty+\|\v^{-2}\triangle r_{n}\|_\infty+ \|\v^{-1}\pa_\v \triangle r_{n}\|_\infty). 
\ee
Thus the $S_5$ term gives the desired estimates and hence 
\be\label{diff_est7}
\|\v^{1+\theta}\pa_\v \triangle \Om_{n+1}\|_\infty \le c_7  \frac{\delta}{\v_0}(\|\triangle f_{n}\|_\infty + \|\triangle \Om_{n}\|_\infty+\|\v^{-2}\triangle r_{n}\|_\infty+ \|\v^{-1}\pa_\v \triangle r_{n}\|_\infty)
\ee
for some constant $c_7>0$ independent of $n$. The estimation of $\|\triangle \Om_{n+1}\|_\infty$ directly follows from $\triangle \Om_{n+1}(\u,\v) = \int_{\v^*}^\v \pa_\v \triangle \Om_{n+1} (\u,\tilde \v)d\tilde \v$ and \eqref{diff_est7}: 
\be\label{diff_est8}
\| \triangle \Om_{n+1}\|_\infty \le c_8  \frac{\delta}{\v_0}(\|\triangle f_{n}\|_\infty + \|\triangle \Om_{n}\|_\infty+\|\v^{-2}\triangle r_{n}\|_\infty+ \|\v^{-1}\pa_\v \triangle r_{n}\|_\infty)
\ee
for some constant $c_8>0$ independent of $n$. Collecting \eqref{diff_est2}, \eqref{diff_est3}, \eqref{diff_est4}, \eqref{diff_est5}, \eqref{diff_est7}, \eqref{diff_est8}, we obtain \eqref{recursive_ineq}. 
\end{proof}

%%%%%%%%%%%%%%%%%%%%%%%%%%%%%%%

We will now finish the proof of Theorem \ref{thm:LWP}. 

%%%%%%%%%%%%%%%%%%%%%%%%%%%%%%%
%%%%%%%%%%%%%%%%%%%%%%%%%%%%%%%

\begin{proof}[Proof of Theorem \ref{thm:LWP}.] 
{\em Convergence and uniqueness.} 
By Proposition \ref{prop:sol} and Proposition \ref{prop:diff}, the iterates 
$[f^\pm_{n},\Om_{n},r_{n}]_{n\in\mathbb N}$  satisfy the uniform bounds $\vertiii{[f^\pm_{n},\Om_{n},r_{n}]}\le 2A\vertiii{[\hat f^\pm, \hat\Om, \hat r]|_{\uC \cup \C}} $ for all $n\in\mathbb N$. Moreover 
$\|\triangle f_{n}^\pm\|_\infty+ \|\v^{-2}{\triangle r_{n}}\|_\infty+ \|\v^{-1}{\pa_\v\triangle r_{n}}\|_\infty+\|\triangle \Om_{n}\|_\infty + 
\|\v^{1+\theta}\pa_\v\triangle \Om_{n}\|_\infty\to 0$ as $n\to \infty$, which in turn implies the strong convergence in $C^1$ for $f^\pm_n$ and in $C^2$ for $[\Om_n, \frac{r_n}{\v}]$ by standard interpolation. 
Hence, as $n\to\infty$, $[\Om_{n},\frac{r_{n}}{\v}]$ converges to $[\Om,\frac{r}{\v}]$ strongly in $C^2$ and $f^\pm_n$ to $f^\pm$ in $C^1$. 
Using the strong convergence shown above, the fundamental theorem of calculus, and the formulas~\eqref{int_rep1},~\eqref{(2.131)}--\eqref{(2.132)},~\eqref{E:TWODER},~\eqref{(2.134)}--\eqref{(2.135)}, and~\eqref{E:TWODER2}, we may pass
to the limit as $n\to\infty$ to conclude that $\pau r(\u,\cdot), \pau\Om(\u,\cdot)\in W^{2,\infty}([\v_0,\infty))$ and $\pau f(\u,\cdot)\in W^{1,\infty}([\v_0,\infty))$.
In particular, using~\eqref{E:RPQ} and~\eqref{recast1}, upon passing to the limit we obtain a classical solution of~\eqref{E:RWAVE1}.
After passing to the limit in \eqref{eq:fn+1}-\eqref{recast2}, we see that $[f^\pm,\Om,\frac{r}{\v}]$ is the desired classical solution to \eqref{E:RWAVE1}--\eqref{E:OMEGAWAVE1} and \eqref{E:FPLUSEVOLUTION}--\eqref{E:FMINUSEVOLUTION} in $\mathcal D$. Note that in fact $\pa_{\u\u}r(\u,\cdot),\pa_{\u\u}\Om(\u,\cdot)\in W^{1,\infty}([\v_0,\infty))$, which follows easily from~\eqref{E:OMEGAWAVE1},~\eqref{E:RPQ} and the bootstrap argument (involving one application of the $\pav$-derivative) similarly to above.

Observe that $f^\pm \in W^{2,\infty}$ and $[\Omega,r]\in W^{3,\infty}$ by standard weak-* convergence arguments, since the iterates are uniformly bounded in the same spaces.
Uniqueness easily follows an adaptation of the difference estimates of Proposition \ref{prop:diff}. To show that the solution corresponds to the solution of the original problem~\eqref{E:RWAVE1}--\eqref{E:MOMENTUMNULL13}, 
it only remains to show that the constraint equations~\eqref{E:CONSTRAINTU1}--\eqref{E:CONSTRAINTV1} are satisfied on the interior of the domain $\D$. This is a standard argument, which follows from the observation 
that for the solutions of~\eqref{E:RWAVE1}--\eqref{E:OMEGAWAVE1} and~\eqref{E:CONTNULL13}--\eqref{E:MOMENTUMNULL13} necessarily $\pau\left(\pav \left(\Om^{-2}\pav r\right) + \pi r \Om^2 T^{\u\u}\right)=0$
and $\pa_\v\left(\pau\left(\Om^{-2}\pau r\right) + \pi r \Om^2 T^{\v\v}\right)=0$. Since the constraints are satisfied by the characteristic data, this gives the claim. Note that these expressions make sense given the above shown regularity.

\noindent
{\em Asymptotic flatness.}
To show the asymptotic flatness we must show that $\lim_{\v\to\infty}m(\u,\v)<\infty$ for all $\u\in[\u_0,\u_0+\delta]$.
By~\eqref{E:HM} this follows if we can show that $\rho\pav r$ and $\rho \Om^2 (u^{\u})^2\pau r$ are integrable on $[\v_0,\infty)$.
Note that $\|\pav r\|_{L^\infty([\v_0,\infty))}<\infty$ by the boundedness of our norms. Moreover, by integrating~\eqref{E:RPQ} with respect to $\v$ and using $\v \lesssim r\lesssim \v$ and the boundedness of $\Omega$,
we conclude that $\|\pau r\|_{\infty}\lesssim 1.$
From \eqref{2.99}
%we can express the fluid density $\rho$ in terms of $f^\pm$ and $r$. The decay of the fluid density $\rho$ is given by  %and \eqref{2.122} we also have 
we have
\be
\begin{split}
\rho=\frac{ (f^+  f^-)^{\frac{1}{k_+-k_-}}}{(1-\k) r^{2+2\e}} &=\frac1{1-\k} \v^{-\frac{N_++N_-}{k_+-k_-} -2-2\e} (\frac{\v}{r})^{2+2\e} (\v^{N_+}f^+ \v^{N_-} f^-)^{\frac{1}{k_+-k_-}}  \\
&\le \frac1{1-\k} \v^{-3-\beta} \left\|\frac{\v}{r}\right\|_\infty^{2+2\e}(\|\v^{N_+} f^+\|_\infty \|\v^{N_-} f^-\|_\infty )^{\frac{1}{k_+-k_-}},   \label{E:AFP}
%\frac{(C_{2,+}C_{2,-})^{\frac{1}{k_+-k_-}}}{C_{1,r}^{2+2\e}}\v^{-2-2\e} \v^{-\frac{N_++N_-}{k_+-k_-}}\\
%&<  \frac{(C_{2,+}C_{2,-})^{\frac{1}{k_+-k_-}}}{C_{1,r}^{2+2\e}}\v^{-3-2\e}
%C_9^{\frac{2}{k_+-k_-}} C_1^{2+2\e} \v^{-2-2\e} \v^{-\frac{k_++k_-}{k_+-k_-}}< C_{11} \v^{-3-2\e}
\end{split}
\ee
where we recall $\beta$ from~\eqref{beta}.
Therefore $\rho(\u,\cdot)\in L^1([\v_0,\infty))$ since  $\vertiii{[f^\pm, \Om, r]}<\infty$. 
To show that $\rho \Om^2 (u^{\u})^2$ is integrable, we observe that 
\[
\rho \Om^2 (u^{\u})^2 = \frac{ (f^+  f^-)^{\frac{1}{k_+-k_-}}}{(1-\k) r^{2+2\e}} \frac1{1+\e}\frac{r^{2\e}}{\Om^2} \left(\frac{f^-}{f^+}\right)^{\frac12}
= \frac{ (f^+)^{\frac1{k_+-k_-}-\frac12}  (f^-)^{\frac{1}{k_+-k_-}+\frac12}}{(1+\k) \Om^2 r^{2}} . 
\]
Using~\eqref{E:NORMALISATIONNULL} and~\eqref{2.100}, we also have the expression
\begin{align}
\Om^2 (u^{\u})^2 = \frac1{1+\e}\frac{r^{2\e}}{\Om^2} \left(\frac{f^-}{f^+}\right)^{\frac12}
\lesssim \v^{2\e+\frac{N_+-N_-}{2}}\left(\frac{\v^{N_-}f^-}{\v^{N_+}f^+}\right)^{\frac12} \lesssim1,
\end{align}
where we have used~\eqref{2.122},~\eqref{logto}, and the boundedness of our norms.
Therefore, from~\eqref{E:AFP} we conclude $\rho \Om^2 (u^{\u})^2\in  L^1([\v_0,\infty))$, and the 
spacetime is therefore asymptotically flat.
\end{proof}

%%%%%%%%%%%%%%%%%%%%%%%%%%%%%%%
%%%%%%%%%%%%%%%%%%%%%%%%%%%%%%%
%%%%%%%%%%%%%%%%%%%%%%%%%%%%%%%

\subsection{Proof of Theorem~\ref{T:MAINTHEOREM}: existence of naked singularities}\label{SS:NAKED}

Recall the discussion of naked singularities in the introduction and in Section~\ref{SS:NSINTRO}.
Following~\cite{RoSR2019}, a spacetime contains a naked singularity if
it corresponds to a maximal hyperbolic development
of suitably regular data,
and the future null-infinity is geodesically incomplete. 
The latter statement does not actually require the construction of future null-infinity as an idealised boundary attached to a suitable
spacetime compactification. Instead, we define it to mean that affine length of a sequence of maximal ingoing null geodesics 
initiated along a sequence of points (approaching infinity) along an asymptotically flat outgoing null-surface, and suitably normalised, is uniformly bounded
by some positive constant.

{\em Proof of Theorem~\ref{T:MAINTHEOREM}.}
For any $\k\in(0,\k_0]$ we consider the associated RLP spacetime $(\MRLP,\grlp)$ given in Definition~\ref{D:RLPST}. We use it to prescribe 
the data for the characteristic problem in the region $\D$ as described in Section~\ref{SS:LWPSTATEMENT}. The associated solution to 
the Einstein-Euler system exists in region $\D$ by Theorem~\ref{thm:LWP}.
Since $\mathcal U_{\pm}=1+O(\sqrt\k)>0$, and since the data are exactly self-similar on $\uC$ and on the finite segment $\{(\u_0,\v) \ \big|\,\ \v\in[\v_0,\v_0+A_0]\}\subset \C$, we conclude that the solution 
coincides with the self-similar RLP-solution in the region $\D_{A_0} = \{(\u,\v)\,\big| \, \u_0\le\u<0, \, \v_0\le\v<\v_0+A_0\}$, see Figure~\ref{F:CUTOFF2}. We now consider a new spacetime
$(\mathcal M_\k,g_\k)$ obtained by gluing together the solution in the region $\D$ and to the past of the ingoing null-segment $\uC$ inside $\MRLP$. 
Clearly, the new spacetime is identical to the RLP spacetime (and therefore smooth) in an open neighbourhood across $\uC$.
It therefore coincides with the exact self-similar RLP-spacetime $(\MRLP,\grlp)$ in the past of $\uC$.

The exterior region, viewed as a development of the characteristic problem with data prescribed along the semi-infinite rectangle with outgoing data prescribed on $\{\v\ge0\}$ and ingoing data on $\u\in[\u_0,0)$
is maximal, as the ingoing null-curve $\mathcal N$ is incomplete on approach to the singularity. In fact, since along backward null-cone $\mathcal N$ we have $\frac{R}{-\sqrt\k\tau}=y_{\mathcal N}$, we 
have 
\begin{align}
\lim_{\tau\to0^- \atop (\tau,R)\in\mathcal N} \rho(\tau, R) = \frac1{2\pi \tau^2} \Sigma(y_\mathcal N) = \infty,
\end{align}
since $ \Sigma(y_{\mathcal N})\neq0$, where we recall that the density $\Sigma$ is in fact strictly positive on $[0,\infty)$.
Therefore, by~\eqref{E:CURVATUREDENSITY} the Ricci scalar blows up as the observer approaches the scaling origin $\mathcal O$ along $\mathcal N$.

It remains to show that the future null-infinity is incomplete in the sense of Definition 1.1. from~\cite{RoSR2019}.
Consider now a sequence of points $(\u_0,\v_n)$ such that $\lim_{n\to\infty}\v_n=\infty$. For any $n\in\mathbb N$, consider the future oriented ingoing radial null-geodesic emanating from $(\u_0,\v_n)$.
Let the affine parameter\footnote{A geodesic $\gamma^\nu$ is parametrised by an affine parameter $\ell$ if 
$\ddot\gamma^{\nu}(\ell)+\Gamma^\nu_{\alpha\beta}\dot\gamma^\alpha(\ell)\dot\gamma^\beta(\ell)=0$.} be denoted by $\ell$. Then $\v(\ell)\equiv \v_n$, and the angular coordinates are also constant. The $\u$-component satisfies the ODE
\be\label{E:AFFINENS}
0 = \ddot \u(\ell) + \Gamma^\u_{\u\u} (\dot\u(\ell))^2 = \ddot \u(\ell) +  \pau\log(\Om^2)(\dot\u(\ell))^2, 
\ee
where we have used~\eqref{E:CHR3}. By our assumptions $\u(0)=\u_0$ and we normalise the tangent vector to be parallel to $\pau$ so that 
at $(\u_0,\v_n)$, we have  $-2 = g(\pav, \dot p(0)\pau) = - \dot p(0) \frac12\Om_n^2$, where $\Om_n: = \Om(\u_0,\v_n)$ for all $n\in\mathbb N$. Therefore
$\dot\u(0) = 4\Om_n^{-2}.$
%without loss assume that $\dot\u(0)=1$.
We may integrate~\eqref{E:AFFINENS} once to conclude that $\frac{d}{d\ell}(-\frac1{\dot \u}+\log(\Om^2))=0$ and therefore
\begin{align}
-\frac1{\dot \u}+\log(\Om^2) = -\frac{\Om_n^2}{4} + \log(\Om_n^2).
\end{align}
It then follows that 
\begin{align}
\dot\u(\ell) = \frac1{\frac{\Om_n^2}{4} -2\log\left(\frac{\Om_n}{\Om(\u(\ell),\v_n)}\right)}.
\end{align}
However, by the proof of Theorem~\ref{thm:LWP} there exists a constant $C$ which depends only on the data
$\vertiii{[f^\pm, \Om, r]|_{\uC \cup \C}}$ such that $\|\pau\log\Om\|_{L^\infty(\D)}\le C$. In particular,
by the mean value theorem 
$\log\left(\frac{\Om_n}{\Om(\u(\ell),\v_n)}\right) \le C |\u-\u_0| \le C\delta$, for any $\u\in[\u_0,0)$ and $n\in\mathbb N$. Note further that 
by~\eqref{E:OMEGAOUTGOING} and our bounds on the data along $\C$, $\Omega_n^2$ is bounded from below and above uniformly in $n$.
For $\delta\ll1$ sufficiently small, we conclude that $\dot\u$ remains positive and $\u$ reaches $0$ (i.e. the Cauchy horizon) in finite $\ell$-time, independent of $n$.
\prfe

\begin{remark}[The Cauchy horizon]
By construction, the Cauchy horizon coincides with the null-curve $\{\u=0\}$.
\end{remark}

\begin{remark}
By construction, for any $\k\in(0,\k_0]$ there exists in fact an infinite family of naked singularity solutions. This freedom
comes from the essentially arbitrary choice of the truncation in the region $\D$, modulo the size and decay limitations imposed
by the local existence theorem, Theorem~\ref{thm:LWP}. In a neighbourhood of the scaling origin $\mathcal O$, our solutions 
are however exactly self-similar.
\end{remark}

%%%%%%%%%%%%%%%%%%%%%%%%%%%%%%%
%%%%%%%%%%%%%%%%%%%%%%%%%%%%%%
%%%%%%%%%%%%%%%%%%%%%%%%%%%%%%%

\appendix

\section{Proof of local existence around the sonic point: combinatorial argument}\label{A:SONIC}

The main goal of the rest of this section is to show the convergence of the power series $\sum_{N=0}^\infty \RH_N (\dx)^N$ and  $\sum_{N=0}^\infty \WH_N (\dx)^N$. We will do so by induction on the coefficients $\RH_N$ and $\WH_N$. %Due to the complexity of $\mathcal S_N$ and $\mathcal V_N$ caused by the presence of $(\RH^{-\e})_\ell$.  it is convenient to introduce $N^{-\frac32}$ scale in the induction scheme, which is equivalent to $(-1)^{N-1}\binom{\frac12}{N}$. 
Before proceeding, we record some technical lemmas. The proofs of Lemmas~\ref{L:A1}--\ref{lem:comb} below are given in detail in~\cite{GHJS2021} and we therefore only state the lemmas without proof for reader's convenience.

\begin{lemma} \label{L:A1}
There exists a constant $c>0$ such that for all $N\in \mathbb N$, the following holds
\begin{align}
\sum_{\substack{\ell+m=N\\ \ell,m \geq 1}}  \frac{1}{\ell^{3} m^{3}} \leq  \frac{c}{N^3}, \label{2.51}\\\sum_{\substack{\ell+m=N\\ \ell,m \geq 1}}  \frac{1}{\ell^{3} m^{2}} \leq  \frac{c}{N^2} ,\label{2.54}\\
\sum_{\substack{\ell+m+n=N\\ \ell, m, n \geq 1}} \frac{1}{\ell^{3} m^{3} n^{3}} \leq  \frac{c}{N^3}, \label{2.52}\\
\sum_{\substack{\ell+m+n=N\\ \ell, m, n \geq 1}} \frac{1}{\ell^{3} m^{2} n^{3}} \leq  \frac{c}{N^2}.  \label{2.53}
\end{align}
\end{lemma}

\begin{lemma}  There exists a constant $c>0$ such that for all $N\geq 3$ and all $C\geq 2$, the following holds
\begin{align}\label{3.166}
\sum_{\ell =2}^{N-1} \frac{C^{\ell -1}}{\ell^q} \leq c \frac{C^{N-2}}{N^q}, \quad q=2, 3. 
%\sum_{\ell =2}^{N-2} \frac{C^{\ell -1}}{\ell^q} \leq c \frac{C^{N-2-\alpha}}{N^q}, \quad q=2, 3
\end{align}
\end{lemma}
%\begin{proof} Let $q=2$ or 3 be given. 
%\[
%\begin{split}
% \sum_{\ell=2}^{N-1} \frac{C^{\ell -1}}{\ell^q}& =\sum_{\ell=2}^{[\frac{N}{2}]} \frac{C^{\ell -1}}{\ell^q} +\sum_{\ell=[\frac{N}{2}]+1}^{N-1} \frac{C^{\ell -1}}{\ell^q} \\
% &\leq  C^{[\frac{N}{2}]-1}\sum_{\ell=2}^\infty \frac{1}{\ell^q} +\frac{1}{([\frac{N}{2}]+1)^q} \sum_{\ell=[\frac{N}{2}]+1}^{N-1} C^{\ell -1}\\
% &\lesssim  C^{[\frac{N}{2}]-1} +  \frac{1}{N^q} C^{[\frac{N}{2}]}\sum_{\ell'=0}^{N-[\frac{N}{2}] -2}C^{\ell'} \\
% &= C^{[\frac{N}{2}]-1} +  \frac{1}{N^q} C^{[\frac{N}{2}]} \frac{C^{ N-[\frac{N}{2}] -1}-1}{C-1}\\
% &\leq C^{[\frac{N}{2}]-1} +  \frac{1}{N^q} \frac{C^{ N-1} }{\frac{C}{2}}\\ 
% &\lesssim  \frac{C^{N-2}}{N^q}
%\end{split}
%\] where we have used ${C-1}\geq \frac{C}{2}$ in the fifth line and $N^q \lesssim C^{N-[\frac{N}{2}]-1}$ for all $N\geq3$ . 
%\end{proof}

For any $\alpha\in \mathbb R$, we let
\[
\binom{\alpha}{j} =\frac{(\alpha)_j}{j!}= \frac{\alpha (\alpha-1)\cdots (\alpha-j+1)}{j!} \ \  \text{for} \ \ j\in \mathbb N, \ \ \text{and} \ \  \binom{\alpha}{0}=1.
\]

\begin{lemma}\label{formula1} Recall the set  $\pi(n,m)$ defined in \eqref{pi}. 
\begin{enumerate}
\item  For each $n\in \mathbb N$, 
\begin{equation}
\sum_{m=1}^n\sum_{\substack{\pi(n,m) } } \frac{(-1)^m m! }{ \lambda_1 ! \dots \lambda_n!} \binom{\frac12}{1}^{\lambda_1} \cdots  \binom{\frac12}{n}^{\lambda_n}  = 2 (n+1)\binom{\frac12}{n+1}
\end{equation} holds. 
\item There exist universal constants $c_1,c_2>0$ such that 
\begin{equation}\label{12n}
 c_1 \frac{1}{n^\frac32}\leq (-1)^{n-1} \binom{\frac12}{n} \leq c_2 \frac{1}{n^\frac32}, \quad n\in \mathbb N . 
\end{equation}
\end{enumerate}
\end{lemma}

%\begin{proof} The first statement follows from Lemma 1.5.2 of [Krantz-Parks, 2002]. 
%
%For the second statement, \eqref{12n} is trivial for $n=1$. Let $n\geq2$. Then 
%\be\label{3.167}
%(-1)^{n-1} \binom{\frac12}{n} = \frac{ \frac{1}{2}\cdot \frac{1}{2} \cdots \frac{2n-3}{2}}{  n!} = \frac{(2n-2)! }{ 2^{2n-1} (n-1)! n! } = \frac{1}{2n-1}\frac{(2n)!}{2^{2n} (n!)^2}
%\ee
%To estimate the last fraction, we invoke Stirling's formula $ 
%n! \sim \sqrt{2\pi n}\left( \frac{n}{e}\right)^n, \ \ n \gg 1$. We will use the following version with upper and lower bounds valid for all $n$: 
%\be
%\sqrt{2\pi} n^{n+\frac{1}{2}} e^{-n} \leq n! \leq e n^{n+\frac12}e^{-n}, \quad n\in \mathbb N
%\ee
%Then we have
%\be
%\frac{ \sqrt{2\pi}\sqrt{2} }{e^2n^\frac12}=\frac{\sqrt{2\pi}  (2n)^{2n+\frac12}e^{-2n} }{2^{2n} e^2 (n^{n+\frac{1}{2}} e^{-n})^2 }  \leq \frac{(2n)!}{2^{2n} (n!)^2} \leq \frac{ e (2n)^{2n+\frac12}e^{-2n} }{ 2^{2n}2\pi (n^{n+\frac{1}{2}} e^{-n})^2 }= \frac{e \sqrt{2}}{ 2\pi n^\frac12}
%\ee
%Hence together with \eqref{3.167}, it implies \eqref{12n}. 
%\end{proof}

\begin{lemma}\label{lem:comb} Let $p>0$ be a given positive number. Let $(\lambda_1,\dots, \lambda_\ell ) \in \pi (\ell ,m)$ where $1\leq m\leq \ell $ and $\ell\geq2$ be given. 
\begin{enumerate}
\item If $1\leq m\leq \left[\tfrac{\sqrt\ell}{\sqrt3}\right]$, there exists a constant $c_3=c_3(p)>0$ such that 
\be\label{E:comb1}
 \prod_{n=1}^\ell \left(\frac{1}{n^{\lambda_n}}\right)^p  \leq \frac{c_3}{\ell^{p}}. 
\ee
\item There exist  $c_4=c_4(p)>0$ and $L_0=L_0(p)>1$ such that if $L\geq L_0$,  the following holds:  
\be\label{E:comb2}
\frac{1}{L^{m-1}}\prod_{n=1}^\ell \left(  \frac{1}{n^{\lambda_n}} \right)^p \leq \frac{c_4}{\ell^{p}},   \text{ for all }  1\leq m \leq \ell. 
\ee
\item Let $\ell \geq 3$. Then there exists  $c_5=c_5(p)>0$ such that if $L\geq L_0$,  the following holds:  
\be\label{E:comb3}
\frac{1}{L^{m-2}}\prod_{n=1}^\ell \left(  \frac{1}{n^{\lambda_n}} \right)^p \leq \frac{c_5}{\ell^{p}},   \text{ for all }  2\leq m \leq \ell. 
\ee
\end{enumerate}
\end{lemma}

By Lemmas~\ref{R0W0} and~\ref{R1W1} there exist  constants $0<m<M<1$ such that
%Let $0<M\leq 1$ be a fixed upper bound of $|R_0|, \ |W_0|, \  |R_1|, \  |\WH_1|$ such that 
\begin{align}\label{E:MBOUND}
& |\RH_0|, \ |\WH_0|, \  |\RH_1|, \  |\WH_1| \leq M, \\
& \RH_0>m,
\end{align}
for all $\k\in(0,\k_0]$, with $\k_0>0$ chosen sufficiently small as in Lemma~\ref{R0W0}.
%In fact $M<1$ for $x_\ast \in (\xmin,\xmax)$ and sufficiently small $\k$ due to Lemmas~\ref{R0W0} and~\ref{R1W1}. 

%%%%%%%%%%%%%%%%%%%%%%%%%%%%%%%%%
%%%%%%%%%%%%%%%%%%%%%%%%%%%%%%%%%

\begin{lemma} Let $x_\ast\in[\xc+\kappa,\xmax]$ %$x_\ast \in (\xmin,\xmax)$ 
and $\alpha\in (1,2)$. Assume that 
\begin{align}
|\RH_m| \leq \frac{C^{m-\alpha}}{m^3}, \quad 2\leq m\leq N-1, \label{assumptionR}\\
|\WH_m| \leq \frac{C^{m-\alpha}}{m^3}, \quad 2\leq m\leq N-1, \label{assumptionW}
\end{align} for some $C\geq 1$ and $N\geq3$. Then there exists a  constant $\hat C=\hat C(M)>0$ such that 
\begin{align}
|(\WH^2)_\ell| +|(\RH\WH)_\ell | +|(\RH^2)_\ell| &\le 
\begin{cases}
\hat C & \ \text{ if } \ \ell=0,1,\\
%D + D \frac{C^{\ell-\alpha}}{\ell^2} & \ \text{ if } \ \ell =2\\
\hat C \frac{C^{\ell-\alpha}}{\ell^3} & \ \text{ if } \ 2\le \ell \le N-1, 
\end{cases}\label{3.174} \\
%\e \frac{C^{\ell-\alpha}}{\ell ^2}, \quad 1\leq l\leq N-1\label{faaRl} \\
%|H_l|& \lesssim  \frac{C^{l-1}}{l^3}, \quad 1\leq l\leq N-1\label{indHl}
|H_\ell |& \le 
\begin{cases}
\hat C & \ \text{ if } \ \ell=0,1, \\
%D + D \frac{C^{\ell-\alpha}}{\ell^2} & \ \text{ if } \ \ell =2\\
\hat C (1+\k) \frac{C^{\ell-\alpha}}{\ell^3} & \ \text{ if } \ 2\le \ell \le N-1. 
\end{cases}\label{indHl}%\\
%|(R^{-\e})_\ell  |& \le
%\begin{cases}
%\e R_0^{-\e-1} & \ \text{ if } \ \ell=1 \\
%\e D \RH_0^{-\e-1} \frac{C^{\ell-1}}{(1+\ell)^2} & \ \text{ if } \ 2\le \ell \le N-1
%\end{cases} \label{faaRl} 
\end{align} 
\end{lemma}

%%%%%%%%%%%%%%%%%%%%%%%%%%%%%%%%%

\begin{proof}We first prove the bounds for $|(\WH^2)_\ell|$, $\ell\ge0$. 
The bounds $|(\WH^2)_0|\le M^2$ and $|(\WH^2)_1|\le 2M^2$ are obvious from~\eqref{E:MBOUND}.
Clearly 
\be\label{E:CASETWO}
|(\WH^2)_2|\le 2M|\WH_2|+M^2 \le 2M \frac{C^{2-\alpha}}{2^3}+M^2\le (2M +2^3 M^2) \frac{C^{2-\alpha}}{2^3}
\ee where we have used $C^{2-\alpha}\geq 1$. 
If $\ell\ge3$ we then have
\begin{align}
|(\WH^2)_\ell| & \le \sum_{m=0}^\ell|\WH_m||\WH_{\ell-m}| \notag \\
&\le 2|\WH_0||\WH_\ell| + 2|\WH_1||\WH_{\ell-1}| + \sum_{m=2}^{\ell-2}|\WH_m||\WH_{\ell-m}| \notag\\
& \le 2M \frac{C^{\ell-\alpha}}{\ell^3}  + 2M  \frac{C^{\ell-1-\alpha}}{(\ell-1)^3}
+ \sum_{m=2}^{\ell-2} \frac{C^{\ell-2\alpha}}{m^3(\ell-m)^3} \notag \\
& \le 2M C^{\ell-\alpha} \left(\frac1{\ell^3} + \frac1{(\ell-1)^3} + \frac{1}{2M}\sum_{m=2}^{\ell-2} \frac{1}{m^3(\ell-m)^3} \right) \notag \\
& \le 2M \tilde C \frac{C^{\ell-\alpha}}{\ell^3},
\end{align}
for some constant $\tilde C$. %where we have used \eqref{2.51}. 
It is now clear, that the estimates for $(\RH\WH)_\ell$ and $(\RH^2)_\ell$, $\ell\ge0$ follow in the same way, as the only estimates we have used are~\eqref{E:MBOUND} and the inductive assumptions~\eqref{assumptionR}--\eqref{assumptionR}, which both depend only on the index, and are symmetric with respect to $\RH$ and $\WH$. %And from~\eqref{E:CASETWO} it is clear that $D\ge M^2\ge M$ since $M\ge1$.
Recalling $H_\ell$ from \eqref{Hl}, 
%\[
%\begin{split}
%H_\ell& = (1-k)(W^2)_\ell  + \tfrac{4k}{k+1} \WH_\ell+ 4k ( R W)_\ell \\
%&=(\tfrac{4k}{k+1} + 2(1-k)W_0 )\WH_\ell+ 4k (\RH_0\WH_\ell +R_\ell W_0 )+ \sum_{\substack{m+n=\ell \\ m,n\geq 1}} (1-k)W_n W_m + 4k \RH_N W_m 
%\end{split}
%\] 
the bounds \eqref{indHl} now immediately follow from \eqref{3.174} and \eqref{assumptionW}. %and \eqref{2.51} we have 
%\[
%|H_l|\lesssim \frac{C^{l-1}}{l^3} +  \sum_{\substack{m+n=l\\ m,n\geq 1}} \frac{C^{m-1}}{m^3}\frac{C^{n-1}}{n^3} \lesssim \frac{C^{l-1}}{l^3}. 
%\]
\end{proof}

%%%%%%%%%%%%%%%%%%%%%%%%%%%%%%%%%
%%%%%%%%%%%%%%%%%%%%%%%%%%%%%%%%%

\begin{lemma} 
Let $x_\ast\in[\xc+\kappa,\xmax]$ %$x_\ast \in (\xmin,\xmax)$ 
and $\alpha\in (1,2)$. Assume that \eqref{assumptionR} and \eqref{assumptionW} for  $N\geq3$ and 
%\begin{align}
%|R_m| \leq \frac{C^{m-\alpha}}{m^3}, \quad 2\leq m\leq N-1 \label{assumptionR}\\
%|W_m| \leq \frac{C^{m-\alpha}}{m^3}, \quad 2\leq m\leq N-1 \label{assumptionW}
%\end{align} 
 some large enough $C>1$ satisfying 
 \be\label{largeC}
 C >\frac{4L_0}{c_1} 
 %\left( >\frac{L_0}{c_1 \RH_0} \right)
 \ee
 where $c_1$ and $L_0=L_0(\frac32)$ are universal constants in \eqref{12n} and Lemma \ref{lem:comb}. Then there exists a constant $\hat C=\hat C(\RH_0)$ such that 
\begin{align}
|(\RH^{-\e})_\ell  |& \le
\begin{cases}
 \hat C %R_0^{-\e-1} 
& \ \text{ if } \ \ell=1, \\
%D + D \frac{C^{\ell-\alpha}}{\ell^2} & \ \text{ if } \ \ell =2\\
 \hat C %c \RH_0^{-\e-1} 
\left(\frac{C^{\ell-\alpha}}{\ell^3} + \frac{C^{\ell-2}}{\ell^2}\right) & \ \text{ if } \ 2\le \ell \le N-1. 
\end{cases} \label{faaRl} 
\end{align} 
\end{lemma}

%%%%%%%%%%%%%%%%%%%%%%%%%%%%%%%%%

\begin{proof} %We next show  \eqref{faaRl}, the bounds of $|(\RH^{-\e})_\ell |$. 
When  $\ell=1$, 
$
|(\RH^{-\e})_1| = | - \e \RH_0^{-\e-1} \RH_1 | \leq  \RH_0^{-\e-1}. 
$
When $\ell=2$ it is easy to see from~\eqref{faaR} that
\begin{align}
(\RH^{-\e})_2 = -\e \RH_0^{-\e}\left(\RH_0^{-1}R_2 -(\e+1)\RH_0^{-2}\RH_1^2\right).
\end{align}
In particular, $\lv(\RH^{-\e})_2 \rv \le  c m^{-\e-2}\left(R_2 + \RH_1^2\right)\le c m^{-\e-2}\left(\frac{C^{2-\alpha}}{2^3} + M^2\right) 
\le \hat C \left(\frac{C^{2-\alpha}}{2^3} + \frac1{2^2}\right)$ for some universal constant $\hat C$ and the claim is thus clear for $\ell=2$.
% \eqref{assumptionR} and \eqref{assumptionW}
%\begin{align}
%|(\RH^{-\e})_l |& \lesssim \e \frac{C^{l-1}}{(l+l)^2}, \quad 1\leq l\leq N-1\label{faaRl} \\
%|H_l|& \lesssim  \frac{C^{l-1}}{l^3}, \quad 1\leq l\leq N-1\label{indHl}
%\end{align}
For $\ell\geq3$, we rewrite $(\RH^{-\e})_\ell$ in the form 
\begin{align}
(\RH^{-\e})_\ell = -\e \RH_0^{-\e-1} \RH_\ell + \RH_0^{-\e} \sum_{m=2}^\ell \frac{1}{\RH_0^m}\sum_{\pi(\ell,m)} (- \e)_m \frac{1}{\lambda_1 ! \dots \lambda_\ell !} {\RH_1}^{\lambda_1}\dots 
{\RH_\ell}^{\lambda_\ell}. \label{E:RETADECOMP}
\end{align}
Now clearly, by the inductive assumption
\be
\lv -\e \RH_0^{-\e-1} \RH_\ell  \rv \le \hat C \frac{C^{\ell-\alpha}}{\ell^2}
\ee
for a universal constant $\hat C>0$.
To bound the second summand on the right-hand side of~\eqref{E:RETADECOMP}, we use~\eqref{assumptionR} and  Lemma~\ref{formula1} first to conclude
\[
\begin{split}
\left| {\RH_1}^{\lambda_1}\dots 
{\RH_\ell}^{\lambda_\ell} \right|=\left| \prod_{n=1}^\ell R_n^{\lambda_n} \right| &\le \left(\tfrac{1}{1^3}\right)^{\lambda_1} \left(\tfrac{C^{2-\alpha}}{2^3}\right)^{\lambda_2}  \dots 
\left(\tfrac{C^{\ell-\alpha}}{\ell^3}\right)^{\lambda_\ell} \\
&= C^{(\alpha-1)\lambda_1 + \sum_{i=1}^\ell(i\lambda_i -\alpha \lambda_i)}\left[ \prod_{n=1}^\ell \left( \frac{1}{n^{\frac32}}\right)^{\lambda_n} \right] \left[\prod_{n=1}^\ell \left( \frac{1}{n^{\lambda_n}}\right)^\frac32 \right] \\
&\leq C^{(\alpha-1)m} C^{\ell -\alpha m} c_1^{-m}\left[ \prod_{n=1}^\ell \left( (-1)^{n-1} \binom{\frac12}{n}\right)^{\lambda_n} \right]
\left[\prod_{n=1}^\ell \left( \frac{1}{n^{\lambda_n}}\right)^\frac32\right]
\end{split}
\]
where we have used $(\alpha-1)\lambda_1\leq (\alpha-1)m$ in the third line since $\alpha>1$ and $ \lambda_1\leq m$.  Hence, using $|(- \e)_m| \leq  m! $, we observe that 
\be\label{3.176}
\begin{split}
&\left|  (- \e)_m \frac{1}{\lambda_1 ! \dots \lambda_\ell !} {\RH_1}^{\lambda_1}\dots 
{\RH_\ell}^{\lambda_\ell}\right| \\%\leq \e   \frac{m!}{\lambda_1 ! \dots \lambda_\ell !} 
%\left(\tfrac{1}{1^3}\right)^{\lambda_1} \left(\tfrac{C^{2-\alpha}}{2^3}\right)^{\lambda_2}  \dots \left(\tfrac{C^{\ell-\alpha}}{\ell^3}\right)^{\lambda_\ell} \\
%&= \e  C^{(\alpha-1)\lambda_1} \frac{m! }{\lambda_1 ! \dots \lambda_\ell !} \left(\tfrac{C^{1-\alpha}}{1^3}\right)^{\lambda_1} \left(\tfrac{C^{2-\alpha}}{2^3}\right)^{\lambda_2}  \dots \left(\tfrac{C^{\ell-\alpha}}{\ell^3}\right)^{\lambda_\ell} \\
%&\leq \e C^{(\alpha-1)m}C^{\sum_{i=1}^\ell({i\lambda_i} -\alpha \lambda_i)} c_1^{- \sum_{i=1}^\ell \lambda_i }  \frac{m!}{\lambda_1 ! \dots \lambda_\ell !} \left((-1)^0\binom{\frac12}{1}\right)^{\lambda_1}\dots 
%\left((-1)^{\ell-1}\binom{\frac12}{\ell}\right)^{\lambda_\ell}\\
&\quad\leq  C^{\ell-m} c_1^{- m } (-1)^{\ell}   \frac{(-1)^m m!}{\lambda_1 ! \dots \lambda_\ell !} \binom{\frac12}{1}^{\lambda_1}\dots 
\binom{\frac12}{\ell}^{\lambda_\ell}  \left[\prod_{n=1}^\ell \left( \frac{1}{n^{\lambda_n}}\right)^\frac32\right]. 
%&=\eC^{l-m} c_1^{- m } (-1)^l2 (l+1) \binom{\frac12}{l+1}\\
%&\leq \e\tfrac{C^{l-1}}{(l+1)^\frac12} 2c_2 C (C c_1)^{- m }
\end{split}
\ee  In turn by recalling \eqref{faaR} and using Lemma \ref{formula1} and Lemma \ref{lem:comb} with $p=\frac32$, we see that   
\be
\begin{split}
&\left| \RH_0^{-\e} \sum_{m=2}^\ell \frac{1}{\RH_0^m}\sum_{\pi(\ell,m)} (- \e)_m \frac{1}{\lambda_1 ! \dots \lambda_\ell !} {\RH_1}^{\lambda_1}\dots 
{\RH_\ell}^{\lambda_\ell}\right| \\
&\leq  \RH_0^{-\e} \frac{C^\ell (-1)^\ell }{(c_1C\RH_0)^2} \sum_{m=2}^\ell \sum_{\pi(\ell,m)} \left[ \frac{1}{(c_1C\RH_0)^{m-2}} \prod_{n=1}^\ell \left( \frac{1}{n^{\lambda_n}}\right)^\frac32\right] \frac{(-1)^m m!}{\lambda_1 ! \dots \lambda_\ell !} \binom{\frac12}{1}^{\lambda_1}\dots 
\binom{\frac12}{\ell}^{\lambda_\ell} \\
&\leq  \RH_0^{-\e} \frac{C^\ell }{(c_1C\RH_0)^2}\frac{c_4}{\ell^\frac32} (-1)^\ell 2 (\ell+1) \binom{\frac12}{\ell+1} \\
%&\leq \e \RH_0^{-\e} \tfrac{C^{l-1}}{(l+1)^\frac12} \sum_{m=1}^l \frac{2c_2 C}{(c_1C\RH_0)^m}  \\
&\leq  \RH_0^{-\e-1} \tfrac{2c_2c_4}{c_1^2}\tfrac{C^{\ell-2}}{\ell^\frac32(\ell+1)^\frac12} 
\end{split}
\ee
%\be
%\begin{split}
%|(\RH^{-\e }  )_l|& =\left| \RH_0^{-\e } \sum_{m=1}^l \frac{1}{\RH_0^m}\sum_{\pi(l,m)} (- \e)_m \frac{1}{\lambda_1 ! \dots \lambda_l !} {\RH_1}^{\lambda_1}\dots 
%{R_l}^{\lambda_l}\right| \\
%&\leq \e \RH_0^{-\e } \tfrac{C^{l-1}}{(l+1)^\frac12} \sum_{m=1}^l \frac{2c_2 C}{(c_1C\RH_0)^m}  \\
%&\leq \e \RH_0^{-\e } \tfrac{C^{l-1}}{(l+1)^\frac12} \tfrac{2c_2}{c_1\RH_0}
%\end{split}
%\ee 
where $C$ is large enough so that \eqref{largeC} holds. 
This proves \eqref{faaRl}. 
\end{proof}

%%%%%%%%%%%%%%%%%%%%%%%%%%%%%%%%%
%%%%%%%%%%%%%%%%%%%%%%%%%%%%%%%%%

\begin{remark}
In~\eqref{largeC} the constant $4$ in the numerator ensures that $C>\frac{L_0}{c_1 \RH_0}$ for all $x_\ast\in[\xc+\kappa,\xmax]$, %$x_\ast\in[\xmin,\xmax]$, 
since for $\k_0>0$ sufficiently small, there exists a constant $0<\delta<1$ such that $\frac1{\RH_0}<3+\delta$ for all $\k\in(0,\k_0]$. 
\end{remark}

%%%%%%%%%%%%%%%%%%%%%%%%%%%%%%%%%
%%%%%%%%%%%%%%%%%%%%%%%%%%%%%%%%%

\begin{lemma}\label{lem:SVbound} Let $x_\ast\in[\xc+\kappa,\xmax]$ %$x_\ast \in (\xmin,\xmax)$ 
and $\alpha\in (1,2)$. Then there exists a constant $C_\ast>0$ such that if $C>C_\ast$ and for any $N\geq3$, the following assumptions hold  
\begin{align}
|\RH_m| \leq \frac{C^{m-\alpha}}{m^3}, \quad 2\leq m\leq N-1, \label{assumptionRm}\\
|\WH_m| \leq \frac{C^{m-\alpha}}{m^3}, \quad 2\leq m\leq N-1, \label{assumptionWm}
\end{align}then we have 
\begin{align}
|\mathcal S_N| &\leq \beta  \frac{C^{N-\alpha}}{N^2}\left[ \frac{1}{C^{\alpha-1}}+ \frac{1}{C^{2-\alpha}} +\frac{1}{C}  \right] ,\label{SNbound}\\
|\mathcal V_N| &\leq \beta   \frac{C^{N-\alpha}}{N^2}\left[ \frac{1}{C^{\alpha-1}}+  \frac{1}{C^{2-\alpha}} + \frac{1}{CN} \right], \label{VNbound}
\end{align}
for some universal constant $\beta>0$. 
\end{lemma}

%%%%%%%%%%%%%%%%%%%%%%%%%%%%%%%%%

\begin{proof} 
%With \eqref{faaRl} and \eqref{indHl} in hand, we are now ready to prove the lemma. 
We start with \eqref{SNbound}. Recall \eqref{SN}. First we show 
\begin{align}
\left| \RH_1\RH_0^{-\e} \sum_{m=2}^N \frac{1}{\RH_0^m}\sum_{\pi(N,m)} (- \e)_m \frac{1}{\lambda_1 ! \dots \lambda_N !} {\RH_1}^{\lambda_1}\dots 
{\RH_N}^{\lambda_N}\right|\lesssim \frac{C^{N-2}}{N^2}. % \lesssim \e\frac{C^{N-\alpha}}{N^2} 
\label{SN1}
\end{align}
%where $\alpha<2$ has been used for the last bound. 
Note that $\lambda_N=0$ and thus \eqref{SN1} does not depend on $\RH_N$. As in \eqref{3.176}, using \eqref{assumptionRm} and Lemma \ref{formula1}, we have 
\[
\begin{split}
&\left|  (- \e)_m \frac{1}{\lambda_1 ! \dots \lambda_N!} {\RH_1}^{\lambda_1}\dots 
{\RH_N}^{\lambda_N}\right| \\%\leq \e   \frac{m!}{\lambda_1 ! \dots \lambda_\ell !} 
%\left(\tfrac{1}{1^3}\right)^{\lambda_1} \left(\tfrac{C^{2-\alpha}}{2^3}\right)^{\lambda_2}  \dots \left(\tfrac{C^{\ell-\alpha}}{\ell^3}\right)^{\lambda_\ell} \\
%&= \e  C^{(\alpha-1)\lambda_1} \frac{m! }{\lambda_1 ! \dots \lambda_\ell !} \left(\tfrac{C^{1-\alpha}}{1^3}\right)^{\lambda_1} \left(\tfrac{C^{2-\alpha}}{2^3}\right)^{\lambda_2}  \dots \left(\tfrac{C^{\ell-\alpha}}{\ell^3}\right)^{\lambda_\ell} \\
%&\leq \e C^{(\alpha-1)m}C^{\sum_{i=1}^\ell({i\lambda_i} -\alpha \lambda_i)} c_1^{- \sum_{i=1}^\ell \lambda_i }  \frac{m!}{\lambda_1 ! \dots \lambda_\ell !} \left((-1)^0\binom{\frac12}{1}\right)^{\lambda_1}\dots 
%\left((-1)^{\ell-1}\binom{\frac12}{\ell}\right)^{\lambda_\ell}\\
&\quad\leq C^{N-m} c_1^{- m } (-1)^{N}   \frac{(-1)^m m!}{\lambda_1 ! \dots \lambda_N !} \binom{\frac12}{1}^{\lambda_1}\dots 
\binom{\frac12}{N}^{\lambda_N}  \left[\prod_{n=1}^N \left( \frac{1}{n^{\lambda_n}}\right)^\frac32\right]. 
%&=\e C^{l-m} c_1^{- m } (-1)^l2 (l+1) \binom{\frac12}{l+1}\\
%&\leq \e \tfrac{C^{l-1}}{(l+1)^\frac12} 2c_2 C (C c_1)^{- m }
\end{split}
\]
%\[
%\begin{split}
%&\left|\sum_{\pi(N,m)} (- \e )_m \frac{1}{\lambda_1 ! \dots \lambda_N !} {\RH_1}^{\lambda_1}\dots 
%{\RH_N}^{\lambda_N}\right| \\
%&\quad\leq \e  %\frac{C^{N-1}}{(N+1)^2} 2c_2C (Cc_1)^{-m}
%C^{N-m} c_1^{- m } (-1)^{N} \sum_{\pi(N,m)}  \frac{(-1)^m m!}{\lambda_1 ! \dots \lambda_l !} \binom{\frac12}{1}^{\lambda_1}\dots 
%\binom{\frac12}{N}^{\lambda_N} 
%\end{split}
%\]
Hence by using \eqref{E:comb3} of Lemma \ref{lem:comb}, the left-hand side of \eqref{SN1} is bounded by 
\[
\begin{split}
&\text{LHS of }\eqref{SN1} \\
&\leq %\RH_1\RH_0^{-\frac{2k}{1-k}} \frac{2k}{1-k} \frac{C^{N-1}}{(N+1)^\frac12} \sum_{m=2}^N \frac{2c_2C}{(c_1C\RH_0)^m}\\
\RH_1\RH_0^{-\frac{2k}{1-k}}  C^N \sum_{m=2}^N \sum_{\pi(N,m)} \frac{(-1)^N}{(c_1C\RH_0)^m} \left[\prod_{n=1}^N \left( \frac{1}{n^{\lambda_n}}\right)^\frac32\right]  \frac{(-1)^m m!}{\lambda_1 ! \dots \lambda_l !} \binom{\frac12}{1}^{\lambda_1}\dots 
\binom{\frac12}{N}^{\lambda_N}  \\
& \leq \RH_1\RH_0^{-\frac{2k}{1-k}}   \frac{C^N}{(c_1C\RH_0)^2}\frac{c_5}{N^\frac32}(-1)^N 2(N+1) \binom{\frac12}{N+1} \\
&\leq  \RH_1\RH_0^{-\frac{2k}{1-k}-2}    \frac{2c_2c_5}{c_1^2} \frac{C^{N-2}}{N^\frac32 (N+1)^\frac12}
\end{split}
\]
which shows \eqref{SN1}. 
To estimate the second term on the right-hand side of~\eqref{SN}, we use \eqref{assumptionRm}, \eqref{assumptionWm} and \eqref{2.51} to obtain 
\be\label{SN2}
\begin{split}
&\left| x_\ast^2 \RH_1 \Big[   (1-\k) \sum_{\substack{\ell+m =N\\ 1\leq m\leq N-1}} \WH_\ell \WH_m+  \sum_{\substack{\ell+m =N\\ 1\leq m\leq N-1}} 4\k \RH_\ell \WH_m\Big]\right| \\
&\leq \left| x_\ast^2 \RH_1 \Big[ 2(1-\k)\WH_1\WH_{N-1} + 4\k(\RH_1W_{N-1}+\RH_{N-1}\WH_1) \Big] \right| \\
 &\quad+ \left| x_\ast^2 \RH_1 \Big[  (1-\k) \sum_{m=2}^{N-2}\WH_{N-m} \WH_m + \sum_{m=2}^{N-2} 4\k \RH_{N-m} \WH_m\Big]\right| \\
&\lesssim (1+\k) \frac{C^{N-1-\alpha}}{N^{3}} + (1+\k) \sum_{m=2}^{N-2} \frac{C^{N-2\alpha}}{(N-m)^3 m^3}\lesssim (1+\k)\frac{C^{N-1-\alpha}}{N^3}. 
\end{split}
\ee
Finally, we estimate the last term on the right-hand side of~\eqref{SN} - the expression $\tilde{\mathcal S}_N$ - given by~\eqref{SN0}.
%\[
%\begin{split}
%&\tilde{\mathcal S}_N = - \sum_{\substack{\ell+m=N \\ 1\leq m\leq N-2}} (m+1) \RH_{m+1} (R^{-\frac{2k}{1-k}})_\ell \\
%&+x_\ast^2 \left( \sum_{\substack{\ell+m=N\\ 1\leq m\leq  N-2}} (m+1) \RH_{m+1} H_\ell + 2\sum_{\substack{\ell+m=N-1 \\ m\leq N-2}} (m+1) \RH_{m+1}H_\ell +\sum_{\ell+m=N-2} (m+1) \RH_{m+1}H_\ell  \right) \\
%&- 2x_\ast^2(1-k)\left(  \sum_{\substack{\ell+m+n=N\\ 1\leq n\leq N-1}} \RH_\ell ( W + \tfrac{k}{k+1})_m ( R- W)_n  +( R ( W + \tfrac{k}{k+1}) ( R- W) )_{N-1}  \right)
%\end{split}
%\]
To treat the first line of \eqref{SN0}, by \eqref{assumptionRm} and \eqref{faaRl}, we obtain 
\be\label{SN3}
\begin{split}
\left| \sum_{\substack{\ell+m=N \\ 1\leq m\leq N-2}} (m+1) \RH_{m+1} (\RH^{-\e})_\ell \right|&\lesssim
 \sum_{\substack{\ell+m=N \\ 1\leq m\leq N-2}} \frac{C^{m+1-\alpha}}{(m+1)^2}\left(\frac{C^{\ell -\alpha}}{\ell^3} + \frac{C^{\ell -2}}{\ell^2} \right) \\
&\lesssim  C^{N+1-2\alpha}\sum_{\substack{\ell+m=N \\ 1\leq m\leq N-2}} \frac{1}{m^2\ell^3} + 
{C^{N-1-\alpha}}\sum_{m=1}^{[\frac{N}{2}]} \frac{1}{m^2}\frac{1}{(N-m)^2}\\
&\lesssim   \frac{C^{N-\alpha}}{N^2} \left(\frac1{C^{\alpha-1}}+ \frac1{C}\right).
\end{split}
\ee %where we have used \eqref{2.54}. 
For the first term of the second line of~\eqref{SN0}, by \eqref{assumptionRm},  \eqref{indHl}, \eqref{2.54}
\be\label{SN4}
\begin{split}
\left| x_\ast^2 \sum_{\substack{l+m=N\\ 1\leq m\leq  N-2}} (m+1) \RH_{m+1} H_l  \right| &\lesssim \sum_{\substack{\ell+m=N \\ 1\leq m\leq N-2}} \frac{C^{m+1-\alpha}}{(m+1)^2}\frac{C^{\ell-\alpha}}{(\ell+1)^3}\\
%&\lesssim C^{N+1-2\alpha} \sum_{m=1}^{[\frac{N}{2}]} \frac{1}{m^2}\frac{1}{(N-m)^2} \\
& \lesssim \frac{C^{N+1-2\alpha}}{N^2}. 
\end{split}
\ee
The other two terms of the second line can be estimated analogously. 
For the first term of the third line of~\eqref{SN0}, by \eqref{assumptionRm}, \eqref{assumptionWm}, \eqref{3.174}, \eqref{2.51} we obtain
\be\label{SN5}
\begin{split}
&\left| 2x_\ast^2(1-\k)\sum_{\substack{\ell+m+n=N\\ 1\leq n\leq N-1}} \RH_\ell ( \WH + \k)_m ( \RH- \WH)_n   \right| \\
&\lesssim \left| ( \RH(\WH+\k) )_{N-1} (\RH_1-\WH_1)  \right| + \left| ( \RH(\WH+\k) )_{1} (\RH_{N-1}-\WH_{N-1})  \right| \\
& \quad+  \left| \sum_{n=2}^{N-2}( \RH(\WH+\k) )_{N-n} (\RH-\WH)_{n}  \right| \\
%&\leq 2(D+k)M \frac{C^{N-1-\alpha}}{(N-1)^3} + 2(2M^2+kM) \frac{C^{N-1-\alpha}}{(N-1)^3}+   D\sum_{n=2}^{N-2}\frac{C^{N-n-\alpha}}{(N-n)^3}\frac{C^{n-\alpha}}{n^3}\\
&\lesssim  \frac{C^{N-1-\alpha}}{(N-1)^3} +  \frac{C^{N-1-\alpha}}{(N-1)^3}+   \sum_{n=2}^{N-2}\frac{C^{N-n-\alpha}}{(N-n)^3}\frac{C^{n-\alpha}}{n^3}\\
%&\lesssim C^{N+1-2\alpha} \sum_{m=1}^{[\frac{N}{2}]} \frac{1}{m^2}\frac{1}{(N-m)^2} \\
& \lesssim \frac{C^{N-1-\alpha}}{N^3} + \frac{C^{N-2\alpha}}{N^3}.
\end{split}
\ee
The bound on the remaining term in the third line of~\eqref{SN0} is entirely analogous. Collecting all the bounds \eqref{SN1}, \eqref{SN2}, \eqref{SN3}, \eqref{SN4}, \eqref{SN5}, we conclude 
\be
|\mathcal S_N|\lesssim \frac{C^{N-\alpha}}{N^2} \left[ \frac{1}{C^{\alpha-1}}+ \frac{1}{CN} +\frac{1}{C^\alpha N}+ \frac1C\right],
\ee 
which leads to \eqref{SNbound} since $1<\alpha<2$ and $C>1$. 

To prove \eqref{VNbound}, we first recall \eqref{VN}. Two first two terms on the right-hand side of~\eqref{VN} have the same structure as the first two terms in~\eqref{SN} and hence, by using the same uniform bound \eqref{E:MBOUND}, they can be bounded analogously to~\eqref{SN1} and~\eqref{SN2}. It remains to estimate $\tilde{\mathcal V}_N$ given in \eqref{VN0}. The first, second and sixth lines of \eqref{VN0} have the same formal structure as the first, second and third lines of \eqref{SN0} and hence we obtain the same bounds as in~\eqref{SN3},~\eqref{SN4}, and~\eqref{SN5} by using \eqref{assumptionWm} in place of \eqref{assumptionRm} when necessary. We focus on the third, fourth and fifth lines of \eqref{VN0}. For the first term of the third line, by using \eqref{faaRl} 
\be
\begin{split}
 \left\vert \sum_{\substack{\ell+m=N \\ 1\leq m\leq N}} (\RH^{-\e})_\ell (-1)^m \right\vert& \lesssim 
 \RH_0^{-\e} + \left| (\RH^{-\e})_1\right| + \sum_{\ell=2}^{N-1} \left| (\RH^{-\e})_\ell \right| \\
 & \lesssim 1 + \hat C  \sum_{\ell=2}^{N-1} \left(\frac{C^{\ell -\alpha}}{\ell^3} + \frac{C^{\ell -2}}{\ell^2} \right)\\
% &= 1 + \e  \sum_{\ell=2}^{[\frac{N}{2}]} \frac{C^{\ell -1}}{\ell^2} + \e  \sum_{\ell=[\frac{N}{2}]+1}^{N-1} \frac{C^{\ell -1}}{\ell^2} \\
% &\leq 1+ \e   C^{[\frac{N}{2}]-1}\sum_{\ell=2}^\infty \frac{1}{\ell^2} + \e  \frac{1}{([\frac{N}{2}]+1)^2} \sum_{\ell=[\frac{N}{2}]+1}^{N-1} C^{\ell -1}\\
% &\lesssim 1+ \e   C^{[\frac{N}{2}]-1} +  \e  \frac{C^{N-1}}{N^2} \\
 &\lesssim \frac{C^{N-1-\alpha}}{N^3} +   \frac{C^{N-3}}{N^2}
\end{split}
\ee where we have used $N^3\lesssim C^{N-1-\alpha}$ for all $N\geq 3$ and \eqref{3.166} with $q=2,3$ and $C\geq 2$.  
For the second term of the third line, we isolate $m=N$, $m=N-1$ cases, $\ell=0, 1$ cases, further $\ell=N-m$ and $\ell=N-m-1$, and use \eqref{E:MBOUND}, \eqref{faaRl}, \eqref{assumptionWm}   
\be\label{3.196}
\begin{split}
&\left|  3\sum_{\substack{\ell+m+n=N \\ 1\leq m\leq N}} \WH_n (\RH^{-\e})_\ell  (-1)^m \right|
\lesssim 1+ \k + \sum_{m=1}^{N-2} \sum_{\ell=0}^{N-m} \left| \WH_{N-m-\ell} (\RH^{-\e})_\ell \right| \\
&\lesssim 1+  \sum_{m=1}^{N-2} |\WH_{N-m}| +   \sum_{m=1}^{N-2} |\WH_{N-m-1}| + \hat C   \sum_{m=1}^{N-2} \sum_{\ell=2}^{N-m} \left| \WH_{N-m-\ell} \right| \left(\frac{C^{\ell-\alpha}}{\ell^3}+\frac{C^{\ell-2}}{\ell^2}\right) \\
&\lesssim 1+ \sum_{m=1}^{N-2} \frac{C^{N-m-\alpha}}{(N-m)^3} +  \sum_{m=1}^{N-3}
 \frac{C^{N-m-1-\alpha}}{(N-m-1)^3} \\
 &\quad +   \sum_{m=1}^{N-2}\left[ \frac{C^{N-m-\alpha}}{(N-m)^3} + \frac{C^{N-m-2}}{(N-m)^2}\right]
  + \sum_{m=1}^{N-2}\left[ \frac{C^{N-m-1-\alpha}}{(N-m-1)^3} + \frac{C^{N-m-3}}{(N-m-1)^2}\right] \\
 &\quad  + \sum_{m=1}^{N-2} \sum_{\ell=2}^{N-m-2} \frac{C^{N-m-\ell-\alpha}}{(N-m-\ell)^3} \left(\frac{C^{\ell-\alpha}}{\ell^3}+\frac{C^{\ell-2}}{\ell^2}\right).
% &=: S_{1L} + S_{2L}
 \end{split}
 \ee
 To bound the summations appearing in the third and the fourth line of~\eqref{3.196}, we apply \eqref{3.166} six times with $\ell = N-m$ 
 and additionally we use $ 1\lesssim \frac{C^{N-1-\alpha}}{N^3} $ to bound them by
 \be
 S_{1L} \lesssim  \frac{C^{N-1-\alpha}}{N^3} +   \frac{C^{N-3}}{N^2}.
 \ee 
 To bound the last line of~\eqref{3.196} we use \eqref{2.54} and  \eqref{3.166} to bound it by
 \be
 \begin{split}
& \sum_{m=1}^{N-2} C^{N-m-\alpha} \sum_{\ell=2}^{N-m-2} \frac{1}{(N-m-\ell)^3} \frac{1}{\ell^3}+\sum_{m=1}^{N-2} C^{N-m-2-\alpha} \sum_{\ell=2}^{N-m-2} \frac{1}{(N-m-\ell)^3} \frac{1}{\ell^2}\\
 & \lesssim    \sum_{m=1}^{N-2}  \frac{ C^{N-m-\alpha}}{(N-m)^3} + \sum_{m=1}^{N-2} \frac{ C^{N-m-2-\alpha}}{ (N-m)^2 }\lesssim \frac{C^{N-1-\alpha}}{N^3} +  \frac{ C^{N-3-\alpha}}{N^2}. 
\end{split}
\ee 
For the fourth line of \eqref{VN0} we only present the details for the first term as the other two terms are estimated analogously. We first isolate $m=N$ and $m=N-1$ and then use \eqref{indHl}, \eqref{3.166}, and \eqref{2.51} to obtain 
\be
\begin{split}
\left| x_\ast^2  \sum_{\substack{\ell+m=N \\ 1\leq m\leq N}} H_\ell (-1)^m \right|& \lesssim 1 + \sum_{\ell=2}^{N-1} |H_\ell |\lesssim 1+  \sum_{\ell=2}^{N-1} \frac{C^{\ell-\alpha}}{\ell^3} \lesssim  \frac{C^{N-1-\alpha}}{N^3}. 
\end{split}
\ee
For the fifth line of \eqref{VN0} we only present the detail for the first two terms, as the estimate for the remaining two terms in the fifth line of the right-hand side of~\eqref{VN0} is analogous and strictly easier. By~\eqref{E:MBOUND}, \eqref{assumptionWm}, \eqref{indHl} we have
\be
\begin{split}
\left| 3x_\ast^2 \sum_{\substack{\ell+n=N\\1\leq  n\leq N-1}} \WH_n H_\ell  \right|& \lesssim |\WH_{N-1} H_1| + |\WH_1 H_{N-1} | + \sum_{\ell=2}^{N-2} |\WH_{N-\ell} H_\ell| \\
&\lesssim  \frac{C^{N-1-\alpha}}{(N-1)^3} +  \sum_{\ell=2}^{N-2} \frac{C^{N-\ell-\alpha}}{(N-\ell)^3}\frac{C^{\ell -\alpha}}{\ell^3}\\
&\lesssim  \frac{C^{N-1-\alpha}}{(N-1)^3} + \frac{C^{N-2\alpha}}{N^3}. 
\end{split}
\ee
The second term of the  fifth line of \eqref{VN0} can be estimated in a similar way as in \eqref{3.196} by using \eqref{indHl} instead of \eqref{faaRl}: 
\be
\begin{split}
&\left|  3x_\ast^2 \sum_{\substack{\ell+m+n=N\\1\leq  m\leq N}} \WH_n H_\ell (-1)^m \right| \lesssim 1+ \sum_{m=1}^{N-2} \sum_{\ell=0}^{N-m}|\WH_{N-m-\ell} H_\ell| \\
&\lesssim 1+ \sum_{m=1}^{N-2} |\WH_{N-m}| + \sum_{m=1}^{N-3}  |\WH_{N-m-1}|+ \sum_{m=1}^{N-2} \sum_{\ell=2}^{N-m}|\WH_{N-m-\ell} H_\ell|  \\
&\lesssim 1+ \sum_{m=1}^{N-2}\frac{C^{N-m-\alpha}}{(N-m)^3}+ \sum_{m=1}^{N-3}\frac{C^{N-m-1-\alpha}}{(N-m-1)^3}
+ \sum_{m=1}^{N-2} \left[ \frac{C^{N-m-\alpha}}{(N-m)^3}+   \frac{C^{N-m-1-\alpha}}{(N-m-1)^3}\right] \\
&\quad+ \sum_{m=1}^{N-2} \sum_{\ell=2}^{N-m-2}\frac{C^{N-m-\ell-\alpha}}{(N-m-\ell)^3 } \frac{C^{\ell-\alpha}}{\ell^3} \\
&\lesssim \frac{C^{N-1-\alpha}}{N^3} + \frac{C^{N-1-2\alpha}}{N^3}  \\
\end{split}
\ee where we have used \eqref{3.166} and \eqref{2.51} as before. Combining all the estimates, we deduce the desired bound \eqref{VNbound}. 
\end{proof}

%%%%%%%%%%%%%%%%%%%%%%%%%%%%
%%%%%%%%%%%%%%%%%%%%%%%%%%%%

\begin{lemma}\label{lem:InductionFinal} Let $x_\ast\in[\xc+\kappa,\xmax]$ %$x_\ast \in (\xmin,\xmax)$ 
and $\alpha\in (1,2)$. Consider $(\RH_0,\WH_0)$, $(\RH_1,\WH_1)$ constructed in Lemma \ref{R0W0} and Lemma \ref{R1W1}, and let $(\RH_N,\WH_N)$ be given recursively by \eqref{recR} and \eqref{recW}. There exist a constant $C>1$ and $\k_0>0$ such that for all $0<\k < \k_0$ and all $x_\ast \in (\xmin,\xmax)$  
\begin{align}
|\RH_N|\leq \frac{C^{N-\alpha}}{N^3},\label{E:inductionR} \\
|\WH_N|\leq \frac{C^{N-\alpha}}{N^3},\label{E:inductionW}
\end{align}
for all $N\geq 2$.  
\end{lemma}

%%%%%%%%%%%%%%%%%%%%%%%%%%%%

\begin{proof} The proof is based on induction on $N$. When $N=2$, it is clear that there exists $C_0=C_0(\alpha) >1$ such that for any $C> C_0$ the bound 
\be
|\RH_2|,  \ |\WH_2| \leq \frac{C^{2-\alpha}}{2^3}
\ee
holds true for any $x_\ast\in[\xc+\kappa,\xmax]$. %$x_\ast \in  (\xmin,\xmax)$. 
Fix an $N\geq3$ and suppose the claim is true for all $2\leq m\leq N-1$. Then \eqref{assumptionRm} and \eqref{assumptionWm} are satisfied and therefore by Lemma \ref{lem:SVbound} we conclude that \eqref{SNbound} and \eqref{VNbound} hold for all $C>C_\ast$.  Together with \eqref{recR1} and \eqref{recW1}, those bounds lead to 
\begin{align*}
|\RH_N|&\leq %\frac{\beta_0}{N} \beta \left( 1+\frac{k}{N} \right)   \frac{C^{N-\alpha}}{N^2}\left[ \frac{1}{C^{\alpha-1}} +\frac{1}{CN} +k \right] =
\beta_0\beta  \left( 1+\frac{\k}{N} \right)\left[ \frac{1}{C^{\alpha-1}} +\frac{1}{CN} +\k \right] \frac{C^{N-\alpha}}{N^3} \\
&\leq \beta_0\beta \left( 1+\frac{\k}{3} \right)\left[ \frac{1}{C^{\alpha-1}} +\frac{1}{3C} +\k \right] \frac{C^{N-\alpha}}{N^3}  =: c_1 \frac{C^{N-\alpha}}{N^3}, \\
|\WH_N|&\leq \beta_0\beta  \left( 1+\frac{1}{N} \right)\left[ \frac{1}{C^{\alpha-1}} +\frac{1}{CN} +\k \right] \frac{C^{N-\alpha}}{N^3}\\
&\leq \frac43 \beta_0\beta \left[ \frac{1}{C^{\alpha-1}} +\frac{1}{3C} +\k \right] \frac{C^{N-\alpha}}{N^3}  =: c_2 \frac{C^{N-\alpha}}{N^3}. 
\end{align*}
It is now clear that since $\alpha>1$ we can choose $C>C_\ast,C_0$ sufficiently large and $\k<\k_0$ with $\k_0>0$ sufficiently small so that $c_1,c_2<1$  and hence \eqref{E:inductionR} and \eqref{E:inductionW} hold true. %Therefore, the induction claim follows for a sufficiently large $C$ and for all sufficiently small $\k$. 
\end{proof}

%%%%%%%%%%%%%%%%%%%%%%%%%%%%
%%%%%%%%%%%%%%%%%%%%%%%%%%%%
%%%%%%%%%%%%%%%%%%%%%%%%%%%%
%%%%%%%%%%%%%%%%%%%%%%%%%%%%

\section{Null-geodesic flow in the RLP-spacetime}\label{A:NNG}

%%%%%%%%%%%%%%%%%%%%%%%%%%%%
%%%%%%%%%%%%%%%%%%%%%%%%%%%%
%%%%%%%%%%%%%%%%%%%%%%%%%%%%
%%%%%%%%%%%%%%%%%%%%%%%%%%%%

%\[
%g = - e^{2\tilde\mu(Y)}\,d\tilde\tau^2 - \frac{4\sqrt\k}{1+\k} Y e^{2\tilde\mu(Y)}\,d\tt\,dR + \left(e^{2\tilde\l(Y)}-\frac{4\k}{(1+\k)^2} Y^2 e^{2\tilde\mu(Y)} \right)\,dR^2 + 
%R^2 \chi(Y)^2\, d\gamma
%\]

We shall carry out the analysis of nonradial null-geodesics (NNG) in the comoving coordinates $(\tau,R,\theta,\phi)$. Partial analysis of simple NNG-s in Schwarzschild coordinates
was already carried out in~\cite{OP1990} and a further analysis of non-spacelike geodesics was done in \cite{JoDw1992}. A related problem for the self-similar dust collapsing clouds was studied in detail in~\cite{OrSaZa}. 
%
%\subsection{Nonradial null-geodesics}

It is convenient to introduce 
\be
z:=  \frac{\sqrt{\k}\tau}{R} \left( = -\frac{1}{y} \right), \quad s:=- \log R, 
\ee
so that 
\begin{align}
z_{\mathcal N} := - \frac1{y_{\mathcal N}}, \ \ z_1:=-\frac1{y_1},
\end{align}
correspond to the boundary of the backward light cone $\mathcal N$ ($y_{\mathcal N}= Y_{\mathcal N}^{-1-\e}$) and the ``first" outgoing null-geodesic $\mathcal B_1$ ($y_1=-|Y_1|^{-1-\e}$) respectively.
Then $\pa_z \tau = \frac{1}{\sqrt\k}e^{-s}$ and $\pa_s \tau= - \frac{1}{\sqrt\k} z e^{-s}$ and hence $d\tau = \frac{e^{-s}}{\sqrt\k} (dz - z ds)$ and $dR = - e^{-s} ds$. 
It is straightforward to check that the homothetic Killing vector field $\xi= {\tau}{\pa_\tau} + R\pa_R$ takes the form $\xi= -\pa_s$ and the metric $g$ in these coordinates reads
\be\label{Eq: metric in z}
\begin{split}
g
%&=  -e^{2\mu} ( \frac{e^{-s}}{\sqrt\k} (dz - z ds))^2 + e^{2\l }(e^{-s} ds)^2 + e^{-2s}\chi^2 d\phi^2 \\
&= e^{-2s} \left[ - \frac{e^{2\mu}}{\k} dz^2 + 2 \frac{e^{2\mu} }{\k} zdzds + \left( e^{2\l } - \frac{e^{2\mu}}{\k}  z^2 \right) ds^2  + \ch^2 d\phi^2 \right].
\end{split}
\ee
%Moreover, the homothetic Killing vector field $\xi= {\tau}{\pa_\tau} + R\pa_R $ takes the simple form $\xi= -\pa_s$. 
We note that the metric is not regular at $z=0$. This corresponds to a harmless coordinate singularity which can be easily avoided by introducing a suitable change of variables, see Section~\ref{S:MAE}.
We shall nevertheless work with~\eqref{Eq: metric in z} to avoid further notational complications and formally limit our analysis to the regions of $\MRLP$ satisfying $\{z<0\}$ and $\{z>0\}$.

Let  $\gamma^\kappa$, $\kappa=z,s,\theta,\phi$ be a null-geodesic and we denote the associated tangent vector by $V^\kappa:=\dot\gamma^\kappa$, $\kappa=z,s,\theta,\phi$. 
We let $\ell$ denote the affine parameter.
Due to spherical symmetry, we have
\be\label{E:PIHALF}
\gamma^\theta(\ell) = \theta(\ell)=\frac\pi2. 
\ee
The remaining geodesic equations read
\be\label{Eq: null geo 0}
\frac{dV_\kappa}{d\ell} = \frac12 (\pa_\kappa g_{\alpha\beta}) V^\alpha V^\beta, \ \ \kappa=z,s,\phi,
\ee 
where $V_\kappa= g_{\kappa\nu} V^\nu$. 

%%%%%%%%%%%%%%%%%%%%%%%%%
%%%%%%%%%%%%%%%%%%%%%%%%%

\begin{lemma}[Geodesic flow]
Let $V^\kappa$ be the tangent to a null-geodesic as above. Then, there exist constants $C,L\in\mathbb R$ such that
\begin{align}
- \frac{e^{2\mu}}{\k} (V^z)^2 + 2 \frac{e^{2\mu} }{\k} zV^z V^s + \left( e^{2\l } - \frac{e^{2\mu}}{\k}  z^2 \right) (V^s)^2  + \ch^2 (V^\phi)^2 &=0, \label{Eq: null geo 1}\\
 V_\phi & =L, 
 %\ \text{ and } \ V^\phi = \frac{L}{\chi^2} e^{2s},
 \label{Eq: angular momentum}\\
V_s& = C. \label{E:CHAM}
\end{align} 
Moreover, the $V^s$-component satisfies the quadratic equation
\be\label{Eq: null geo 3}
e^{\l +\mu} \left(z^2 - \k e^{2\l - 2\mu} \right)\left( \frac{V^s}{e^{2s}} \right)^2 + 2 C\k e^{\l -\mu}\left( \frac{V^s}{e^{2s}} \right) - e^{-\l -\mu }\k C^2  +e^{-\l+\mu } z^2\frac{L^2}{\ch^2}=0. 
\ee
\end{lemma}

%%%%%%%%%%%%%%%%%%%%%%%%%
%%%%%%%%%%%%%%%%%%%%%%%%%

\begin{proof}
Equation~\eqref{Eq: null geo 1} is just the statement that the 4-vector $V^\alpha$ is null.
%\be\label{Eq: null geo 1}
%- \frac{e^{2\mu}}{\k} (V^z)^2 + 2 \frac{e^{2\mu} }{\k} zV^z V^s + \left( e^{2\l } - \frac{e^{2\mu}}{\k}  z^2 \right) (V^s)^2  + \chi^2 (V^\phi)^2 =0
%\ee
%We examine the geodesic equations \eqref{Eq: null geo 0}. 
We now let $\alpha=\phi$ in~\eqref{Eq: null geo 0}. Since $g_{\alpha\beta}$-terms do not depend on $\phi$, from \eqref{Eq: null geo 0} we immediately see that $\frac{dV_\phi}{d\ell}=0$, which implies~\eqref{Eq: angular momentum}.
% that $\pa_\phi$ is the rotational Killing vector field leading to the conservation of the angular momentum. 
Since $V_\phi = g_{\phi\phi} V^{\phi} $, we also obtain 
\be\label{Eq: angular momentum upper}
V^\phi = \frac{L}{\ch^2} e^{2s}.
\ee

Next let $\kappa=s$ in  \eqref{Eq: null geo 0}. Since $\pa_s g = - 2 g $ (cf. \eqref{Eq: metric in z}) we deduce that $\frac{dV_s}{d\ell}=0$ from \eqref{Eq: null geo 0} and~\eqref{E:CHAM} follows. 
%Let
%\[
%V_s = C
%\]
%where $-C$ represents the constant of the motion associated with the homothetic Killing vector field $-\pa_s$. 

Since $V_s= g_{ss}V^s + g_{sz}V^z$, we obtain 
\be\label{E:SIMPLEHELP}
 \left( e^{2\l } - \frac{e^{2\mu}}{\k}  z^2 \right)\frac{V^s}{e^{2s}}  +\frac{e^{2\mu}}{\k} z \frac{V^z}{e^{2s}}  =C,
\ee 
which in turn gives
\be\label{Eq: Vz Vs}
z\frac{V^z}{e^{2s}} = \frac{\k C }{ e^{2\mu}}+  \left(z^2 - \k e^{2\l - 2\mu} \right)\frac{V^s}{e^{2s}}. 
\ee
Using \eqref{Eq: angular momentum upper}, we may rewrite  \eqref{Eq: null geo 1} as 
\be\label{Eq: null geo 2}
- \frac{e^{2\mu}}{\k} \left(\frac{V^z}{e^{2s}}\right)
^2 + 2 \frac{e^{2\mu} }{\k} z\left( \frac{V^z}{e^{2s}} \right) \left( \frac{V^s}{e^{2s}}\right) + \left( e^{2\l } - \frac{e^{2\mu}}{\k}  z^2 \right) \left( \frac{V^s}{e^{2s}} \right)^2  + \frac{L^2}{ \ch^2} =0 . 
\ee
%We first assume $z=0$. Then from \eqref{Eq: Vz Vs} and \eqref{Eq: null geo 2} we deduce that 
%\be\label{Vs Vz 0}
%\frac{V^s}{e^{2s}} (z=0) = C e^{-2\l (0)}, \quad \frac{V^z}{e^{2s}} (z=0) =\pm C\sqrt\k e^{- \mu(0)-\l(0)} \sqrt{ 1 + e^{2\l(0)} \frac{L^2}{\chi^2(0) C^2}}
%\ee
%$\frac{V^s}{e^{2s}} (z=0) = C e^{-2\l (0)}$. From \eqref{Eq: null geo 2}, we deduce that $\frac{V^z}{e^{2s}} (z=0) =\pm \sqrt\k e^{-\mu(0)} \sqrt{C^2 e^{-2\l } + \frac{L^2}{\chi^2(0)}}$. 
Recall that we work in a patch where $z\neq 0$. Plugging~\eqref{Eq: Vz Vs} into \eqref{Eq: null geo 2} to replace $V^z$, we get 
\[
\begin{split}
&- \frac{e^{2\mu}}{\k} \left( \frac{\k C }{z e^{2\mu}}-+ \frac{1}{z} \left(z^2 - \k e^{2\l - 2\mu} \right) \frac{V^s}{e^{2s}} \right)
^2\\
& + 2 \frac{e^{2\mu} }{\k} z\left( \frac{\k C }{z e^{2\mu}} + \frac{1}{z} \left(z^2 - \k e^{2\l - 2\mu} \right) \frac{V^s}{e^{2s}}  \right) \left( \frac{V^s}{e^{2s}}\right) - \frac{e^{2\mu}}{\k}  \left(z^2 - \k e^{2\l - 2\mu} \right) %+ \left( e^{2\l } - \frac{e^{2\mu}}{\k}  z^2 \right)
 \left( \frac{V^s}{e^{2s}} \right)^2  
+ \frac{L^2}{ \ch^2} =0 
\end{split}
\] which is easily simplified to  give~\eqref{Eq: null geo 3}.
%\be\label{Eq: null geo 3}
%e^{\l +\mu} \left(z^2 - \k e^{2\l - 2\mu} \right)\left( \frac{V^s}{e^{2s}} \right)^2 + 2 C\k e^{\l -\mu}\left( \frac{V^s}{e^{2s}} \right) - e^{-\l -\mu }\k C^2  +e^{-\l+\mu } z^2\frac{L^2}{\chi^2}=0 
%\ee
\end{proof}

\begin{remark}
Relations~\eqref{Eq: angular momentum}--\eqref{E:CHAM} are the Hamiltonian constraints associated with the geodesic flow. Equation~\eqref{Eq: angular momentum} expresses
the conservation of angular momentum associated with the Killing vector field $\pa_\phi$, while~\eqref{E:CHAM} is the conservation law generated by the 
homothetic Killing vector field $-\pa_s = \tau\pa_\tau+R\pa_R$.
\end{remark}

%%%%%%%%%%%%%%%%%%
%%%%%%%%%%%%%%%%%%

\begin{lemma}\label{L:SIMPLENNGCHAR}
Assume that there exists a simple nonradial null-geodesic, i.e. a null-geodesic such that $L\neq0$
and $z(\ell)=\bar z\equiv \text{const}$ for all $\ell\in\mathbb R$. Then $\bar z$ is a critical point of the function
\begin{align}\label{E:MATHCALHDEF}
\mathcal H(z) : =  \left( e^{2\l(z) } - \frac{e^{2\mu(z)}}{\k}  z^2 \right) \ch(z)^{-2}.
\end{align}
\end{lemma}

%%%%%%%%%%%%%%%%%%
\begin{proof}
By our assumption, for any simple geodesic we have $V^z=0$.
By letting $\kappa=z$ in~\eqref{Eq: null geo 0}, we then obtain the formula
\begin{align}\label{E:SIMPLEHELP2}
\frac{d}{d\ell} \left(\frac{e^{2\mu}}{\k}e^{-2s}V^s\right) = \frac{e^{-2s}}{2} \pa_z\left( e^{2\l } - \frac{e^{2\mu}}{\k}  z^2 \right) (V^s)^2 +  \frac{e^{-2s}}{2} \pa_z(\ch^2)(V^\phi)^2.
\end{align}
Since $V^z=0$ in~\eqref{E:SIMPLEHELP}, we conclude that $V^se^{-2s}$ is $\ell$-independent and therefore the left-hand side of~\eqref{E:SIMPLEHELP2} vanishes. Letting $V^z=0$ 
in~\eqref{Eq: null geo 1} we have $(V^\phi)^2= -\left(e^{2\l}- \frac{e^{2\mu}}{\k}  z^2 \right) \frac{(V^s)^2}{\ch^2}$. We plug this back into~\eqref{E:SIMPLEHELP2} and the claim follows.
\end{proof}

%%%%%%%%%%%%%%%%%%
%%%%%%%%%%%%%%%%%%

\begin{lemma}[Monotonicity of $\mathcal H$]\label{L:MATHCALHMONOTONE}
The function $\mathcal H$ defined by~\eqref{E:MATHCALHDEF} is strictly monotone on $(-\infty,0)$. 
%In particular, by Lemma~\ref{L:SIMPLENNGCHAR} there
%do not exist simple NNG-s for any $z<0$. 
\end{lemma}

%%%%%%%%%%%%%%%%%%

\begin{proof}
Recalling the notation $x=\tr(y)$, $\ch(y)=\frac{x}{y}$, and $z=-\frac1y$, we see that $\mathcal H$, written as a function of $y$ can be written in the form
\begin{align}
\mathcal H(y) = x^{-2}\left(e^{2\l}y^2 - \frac1\k e^{2\mu}\right).
\end{align}
Therefore 
\begin{align}
&\mathcal H'(y) = -2 x^{-2}\frac{x'}{x} \left(e^{2\l}y^2 - \frac1\k e^{2\mu}\right) + x^{-2} \left(2e^{2\l}y^2\left(\l'+\frac1y\right)-\frac2\k e^{2\mu}\mu'\right) \notag\\
& = -2 x^{-2}\frac{\w +\k}{(1+\k)y} \left(e^{2\l}y^2 - \frac1\k e^{2\mu}\right) + 2x^{-2} \left(e^{2\l}y^2\left(-\frac1{1+\k}\left(\frac{ \Sigma'}{ \Sigma}+\frac{2}{y}\right) - 2\left(\frac{x'}{x}-\frac1y\right)+\frac1y\right) +\frac1\k e^{2\mu} \frac{\k}{1+\k} \frac{ \Sigma'}{ \Sigma}\right) \notag\\
& =  -2 x^{-2}\frac{\w +\k}{(1+\k)y} \left(e^{2\l}y^2 - \frac1\k e^{2\mu}\right) + \frac{2x^{-2}}{1+\k} \left(e^{2\l}y^2\left(-\frac{ \Sigma'}{ \Sigma}+\frac{1+3\k}{y} - 2\frac{\w +\k}{y} \right)+e^{2\mu} \frac{ \Sigma'}{ \Sigma} \right) \notag\\
& =  -2 x^{-2}\frac{\w +\k}{(1+\k)y} \left(e^{2\l}y^2 - \frac1\k e^{2\mu}\right) + \frac{2x^{-2}}{1+\k} \left(\frac{ \Sigma'}{ \Sigma}\left(e^{2\mu}-e^{2\l}y^2\right)+e^{2\l}y\left(1+\k  - 2\w  \right)\right),  \label{E:SIMPLEHELP3}
\end{align}
where we have used the relation $\frac{x'}{x}=\frac{\w +\k}{(1+\k)y}$ (see~\eqref{E:TILDERPRIME}), formula~\eqref{E:LAMBDASSNEW} for $\l'$, and~\eqref{E:MUFORMULASS} to substitute for $\mu'$.
We now use~\eqref{E:DSSEQN} and part (c) of Lemma~\ref{L:NONLOCALISLOCAL} to obtain  the formula
\begin{align}
\frac{ \Sigma'}{(1+\k) \Sigma}\left(e^{2\mu}-e^{2\l}y^2\right) = -2ye^{2\l}K %\frac{2ye^{2\l}H}{1+\k} 
= -\frac{2(\d-\w )}{(1+\k)}e^{2\l}y.
\end{align}
Using this in~\eqref{E:SIMPLEHELP3} we obtain
\begin{align}
\mathcal H'(y) &=-2 x^{-2}\frac{\w +\k}{(1+\k)y} \left(e^{2\l}y^2 - \frac1\k e^{2\mu}\right) + \frac{2x^{-2}}{1+\k} e^{2\l}y\left(1+\k  - 2\d \right) \notag\\
& =\frac{2 x^{-2}}{(1+\k)y} e^{2\l} \left(\frac1\k (\w +\k)e^{2\mu-2\l} + y^2 \left(1-\w -2\d\right)\right). \label{E:SIMPLEHELP4}
\end{align}
We now recall the formula
\begin{align}
e^{2\mu-2\l} = \frac{y^2}{x^2(\w +\k)^2}\left(\d^{-\e}+\k x^2(\w -1)^2 - 4\k \d \w  x^2\right),
\end{align}
see~\eqref{E:SONICKEY1}.
We plug it back into~\eqref{E:SIMPLEHELP4} to conclude that 
\begin{align}
\mathcal H'(y)& =\frac{2 x^{-2}y}{(1+\k)} e^{2\l} \left(\frac1\k \frac{\d^{-\e}}{x^2(\w +\k)} + \frac{(\w -1)^2}{\w +\k}-4\frac{\d\w }{\w +\k}+1-\w -2\d\right) \notag\\
& = \frac{2 x^{-2}y}{(1+\k)} e^{2\l}  \left(\frac{\d^{-\e}}{x^2}\left[\frac1\k \frac{1}{\w +\k} - \d^{1+\e}x^2\left(2+\frac{4\w }{\w +\k}\right) \right]  + \frac{(\w -1)^2}{\w +\k} + 1-\w \right) . \label{E:SIMPLEHELP5}
\end{align}
Recall that $\frac16\le \w \le 1$ for all $y\in(0,\infty)$. Moreover, by~\eqref{E:PIWX2} and the inequality $\w (y)=W(x)<1$ for all $y,x\ge0$ we know that  that the function $x\mapsto \R^{1+\e}x^2 W(x)$ is increasing.
It follows from~\eqref{E:PIWX2LB} that there exists an $\k$-independent constant $C$ such that $\R^{1+\e}x^2\le C$ for all $\k\in(0,\k_0]$, with $\k_0$ sufficiently small.
Since $ \frac{1}{\w +\k}\ge\frac1{1+\k}$, $2+\frac{4\w }{\w +\k} \le 2+\frac4{\frac13+\k}$, and $ \frac{(\w -1)^2}{\w +\k} + 1-\w >0$, we conclude from~\eqref{E:SIMPLEHELP5} that for $\k>0$ sufficiently small that $\mathcal H'$
is strictly positive for $y\in(0,\infty)$.
\end{proof}

%%%%%%%%%%%%%%%%%%
%%%%%%%%%%%%%%%%%%

\begin{lemma}[$L=0$: radial null-geodesics]
Let $V^\kappa$ as above. If $L=0$ the flow takes the form
\be\label{geodesicODE0}
\frac{ds}{dz}=\frac{V^s}{V^z} =\frac{1}{ z \pm \sqrt{\k} e^{\lambda -\mu}},
\ee
while $V^s$, $V^z$ are then recovered using \eqref{Eq: Vz Vs}. The behaviour of the solutions is described by Lemma~\ref{L:INCOMING}.
\end{lemma}

\begin{proof}
Equation~\eqref{geodesicODE0} is a simple consequence of~\eqref{Eq: null geo 2}. 
Equation \eqref{geodesicODE0} is equivalent to the radial null geodesic equation \eqref{E:RYEQN} expressed in $(Y,\log R)$-variables.
\end{proof}
%We consider three different cases (I) $L=0$ (radial case),  (II) $L\neq 0$ and $C=0$ (non-radial case with $C=0$),  (III) $L\neq 0$ and $C\neq 0$ (general non-radial case). 
%
%\
%
%\noindent {\em Case (I) $L=0$ (radial case).} 
%In this case, \eqref{Eq: null geo 2} yields that 
%\be\label{geodesicODE0}
%\frac{ds}{dz}=\frac{V^s}{V^z} =\frac{1}{ z \pm \sqrt{\k} e^{\lambda -\mu}}
%\ee
%and $V^z$ and $V^s$ can be recovered by further using \eqref{Eq: Vz Vs}. Equation \eqref{geodesicODE0} is equivalent to the radial null geodesic equation \eqref{E:RYEQN} written in terms of $Y$ and $\log R$.  
% 
%\
%

\begin{lemma}[$L\neq0$: nonradial null-geodesics]
Let $V^\kappa$ as above and $L\neq0$. Then any such geodesic that emanates from the union of the exterior and the interior region exists globally in its affine time and does not converge to the scaling origin.
%In particular, there do not exist simple NNG-s in the union of the exterior and the interior region.
\end{lemma}

\begin{proof}
{\em Step 1. We first consider the case where $C=0$.}
Then~\eqref{Eq: Vz Vs} gives the relation 
\be\label{geodesicODE1}
\frac{ds}{dz} =\frac{V^s}{V^z} = - \frac{z}{\k e^{2\l -2\mu} - z^2}.
\ee
In order for $V^s$ to satisfy \eqref{Eq: null geo 3} we must have $\k e^{2\l -2\mu} - z^2 > 0$ since $\frac{L^2}{\ch^2}>0$. In particular, the null geodesics are confined to the exterior region $z_{\mathcal N}< z< z_1$ where $z_{\mathcal N} := - |Y_{\mathcal N}|^{1+\e}<0$ and $z_1:= |Y_1|^{1+\e}>0$. Let $s_0\in \mathbb R$ and $z_0 \in ( z_{\mathcal N}, z_1)$ be given initial point in the exterior region. Then by integrating \eqref{geodesicODE1}, we obtain 
\be
s=s_0 - \int_{z_0}^{z(s)} \frac{z}{\k e^{2\l -2\mu} - z^2} dz . 
\ee
We recall that $\k e^{2\l -2\mu} - z^2|_{z= z_{\mathcal N}, z_1} =0$ and $\k e^{2\l -2\mu} - z^2 > 0$ for $z_{\mathcal N}< z< z_1$. Since $s= -\log R$, we deduce that $R\to \infty$ ($s \to -\infty$)  as $z(s)\to z_{\mathcal N}$ or $z(s)\to z_1$. Therefore a null geodesic in this case asymptotes to the backward light cone $z=z_{\mathcal N}$ in the past and to the Cauchy horizon $z=z_1$ in the future. 
In particular, there are no null geodesics emanating from or going into the origin.  

\noindent
{\em Step 2. We now consider the general case $C\neq0$.}
Assume without loss of generality $C>0$. In order for the solution of \eqref{Eq: null geo 3} to exist the discriminant must be nonnegative.  
%\[
%\left(  \frac{C\k e^{2\l -2\mu}}{z^2} \right)^2 - e^{2\l } \left( 1- \frac{\k}{z^2} e^{2\l - 2\mu }\right)  \left(- \frac{\k C^2}{z^2 e^{2\mu}} + \frac{L^2}{\ch^2} \right) \ge 0
%\] 
This amounts to requiring  
\be\label{Eq: effective potential}
a(z):= \frac{e^{2\mu} (z^2 - \k e^{2\l - 2\mu})}{\ch^2 }\le  \frac{\k C^2}{L^2}. 
\ee
Note that $\lim_{z\to -\infty} a(z) = +\infty$, which follows from $z= -\frac{1}{y}$,~\eqref{E:LAMBDABADATZERO}, and $e^{\mu}|_{y=0}>0$, while  $\lim_{z\to z^{\text{ms}}} a(z)= -\infty$ by Proposition \ref{P:SINGBEHAVIOUR}. Moreover, we also know $a(z_{\mathcal N})=0$, $a(z_1)=0$ and $a(z)<0$ for $z_{\mathcal N}< z< z_1$ from the analysis of radial null geodesics. Therefore, for given $\frac{\k C^2}{L^2}>0$, there exists $\bar z<z_{\mathcal N}$ such that $a(z) > \frac{\k C^2}{L^2}$ for all $z\in (-\infty, \bar z)$ and hence, the non-radial null geodesics cannot enter the part of the interior region $z<\bar z$. 

%We will consider the null geodesics satisfying \eqref{Eq: effective potential}. 
Under the assumption~\eqref{Eq: effective potential}, the solution to \eqref{Eq: null geo 3} is given by 
\be
\frac{V^s}{e^{2s}} = C\sqrt\k \frac{- \sqrt \k e^{\l -\mu} \pm |z| \sqrt{ 1 - \frac{e^{2\mu}(z^2- \k e^{2\l-2\mu})}{\ch^2} \frac{L^2}{\k C^2} } }{e^{\l+\mu}\left( z^2 - \k e^{2\l -2\mu}  \right)}.
\ee
From \eqref{Eq: Vz Vs} we also have (recall $z\neq0$)
\be
z\frac{V^z}{e^{2s}} = \pm \frac{C\sqrt\k |z|}{e^{\l+\mu}} \sqrt{ 1 - \frac{e^{2\mu}(z^2- \k e^{2\l-2\mu})}{\ch^2} \frac{L^2}{\k C^2} }. 
\ee
Let 
\be\label{E:QNNGDEF}
Q(z):=  \sqrt{ 1 - \frac{e^{2\mu}(z^2- \k e^{2\l-2\mu})}{\ch^2} \frac{L^2}{\k C^2} }.
\ee 
We then conclude that
\be\label{E:DSDZEQN}
\frac{ds}{dz} = \frac{V^s}{V^z}=  \frac{z\left[ - \sqrt \k e^{\l -\mu} \pm |z| Q(z) \right]}{\pm |z| Q(z)( z^2 - \k e^{2\l -2\mu})}. 
%=\begin{cases}
% \frac{ - \sqrt \k e^{\l -\mu} \mp  z Q(z) }{\mp  Q(z)( z^2 - \k e^{2\l -2\mu})},  & z<0 \\
%  \frac{ - \sqrt \k e^{\l -\mu} \pm z Q(z) }{\pm Q(z)( z^2 - \k e^{2\l -2\mu})},  & z>0
%\end{cases}
\ee
%and when $z=0$, $V^s$ and $V^z$ are given in \eqref{Vs Vz 0}. 
%
%\

\noindent
{\em Step 3. No simple NNG-s}. By definition, we have $\frac{dz}{ds}=0$ for simple NNG-s. Therefore by~\eqref{E:DSDZEQN} they must satisfy either $z^2-\k e^{2\l - 2\mu}=0$ or $Q(z)=0$. 
The first case corresponds to simple radial null geodesics, while zeros of $Q(z)=0$ possibly describe the simple non-radial geodesics. We know that $z^2 - \k e^{2\l-e\mu} <0 $ for all $z_{\mathcal N} < z<z_1$ and thus, $Q >0$ for all  $z_{\mathcal N} < z<z_1$. Therefore we deduce that there is no simple non-radial null geodesics in the exterior (i.e. in the causal past of $\mathcal B_1$).  In the interior region however $\bar z$ cannot correspond to a geodesic, since 
otherwise by Lemma~\ref{L:SIMPLENNGCHAR} we would have $\mathcal H'(\bar z)=0$, but this is impossible by Lemma~\ref{L:MATHCALHMONOTONE}.

\noindent
{\em Step 4. NNG-s cannot emanate from or go towards the scaling origin $\mathcal O$.} 

{\em Step 4.1. First let $z_1>z_0>0$ and $s_0\in \mathbb R$ be given in the exterior region above $\tau=0$ ($z>0$).} %We choose the affine parameter so that $V^z>0$ for $z>0$. 
Then the null geodesics satisfy  
\be\label{geodesicODE2}
\frac{ds}{dz} = \frac{ - \sqrt \k e^{\l -\mu} \pm z Q(z)}{\pm Q(z)( z^2 - \k e^{2\l -2\mu})}. 
\ee

We start with $(-)$ part. In this case, the right-hand side of \eqref{geodesicODE2} is negative and hence the null geodesics are outgoing. We claim that outgoing null geodesics $(-)$ starting from $z_0>0$ and $s_0\in \mathbb R$ meet $z=0$ (namely $\tau=0$) for some positive $R>0$ in the past. Observe that there exist $-\infty <M_1 < M_0<0$ such that  $M_1<Q(z) ( z^2 - \k e^{2\l -2\mu}) < M_0$ for all $z\in [0,z_0]$. Integrating \eqref{geodesicODE2}, we have
 \be
 s(0) = s(z_0) + \int_{z_0}^0 \frac{  \sqrt \k e^{\l -\mu} + z Q(z)}{ Q(z)( z^2 - \k e^{2\l -2\mu})}  dz . 
 \ee
The integral in the right-hand side is finite and using $s=-\log R$, the claim follows. On the other hand, in the future direction, since $z^2-\k e^{2\l -2\mu} \to 0$ as $z\to z_1^-$, $s \to -\infty$ as $z\to z_1^-$ and thus it asymptotes to the Cauchy horizon $z=z_1$. 

We move onto $(+)$ part, the ingoing case. As above, the null geodesics meet $\tau=0$ for some positive $R>0$ in the past as the corresponding integral stays finite. We claim that the null geodesics meets $\mathcal B_1$ at positive $R>0$ in the future. Integrating \eqref{geodesicODE2}, we have 
 \be
 s = s(z_0) + \int_{z_0}^{z(s)} \frac{ - \sqrt \k e^{\l -\mu} + z Q(z)}{ Q(z)( z^2 - \k e^{2\l -2\mu})}  dz 
 \ee
 for $z=z(s)< z_1$. To prove the claim, it suffices to show the integral is finite when $z(s)=z_1$. To this end, we rewrite the integrand as 
 \[
 \begin{split}
 \frac{ - \sqrt \k e^{\l -\mu} + z Q(z)}{ Q(z)( z^2 - \k e^{2\l -2\mu})} &= \frac{ - \sqrt \k e^{\l -\mu} + z + z (Q(z)-1)}{ Q(z)( z^2 - \k e^{2\l -2\mu})} \\
  &= \frac{1}{Q(z) (z+ \sqrt\k e^{\l-\mu })} + \frac{  z (Q(z)^2-1)}{ Q(z)( z^2 - \k e^{2\l -2\mu}) (Q(z) +1)} \\
  &=  \frac{1}{Q(z) (z+ \sqrt\k e^{\l-\mu })}  - \frac{z \frac{e^{2\mu} }{\ch^2} \frac{L^2}{\k C^2}}{Q(z) (Q(z)+1)}
 \end{split}
 \]
for $z_0<z<z_1$, where we have used~\eqref{E:QNNGDEF} in the last line. It is now clear that both terms are finite and thus the integral is indeed bounded, thereby proving the claim. 

{\em Step 4.2. Next we let $z_{\mathcal N}<z_0<0$ and $s_0\in \mathbb R$ be given in the exterior region below $\tau=0$ $(z<0)$.} The null geodesics satisfy  
\be\label{geodesicODE3}
\frac{ds}{dz} = \frac{ - \sqrt \k e^{\l -\mu} \mp z Q(z)}{\mp Q(z)( z^2 - \k e^{2\l -2\mu})}. 
\ee
We start with $(+)$ part. In this case, we claim that the null geodesics meet $\tau=0$ for positive $R>0$ in the future and asymptotes to the past light cone $z_{\mathcal N}$ in the past. The geodesic in the future satisfies 
\be
s(0) = s(z_0) +  \int_{z_0}^0 \frac{ - \sqrt \k e^{\l -\mu} + z Q(z)}{Q(z)( z^2 - \k e^{2\l -2\mu})} dz . 
\ee
Observe that the integrand is positive and uniformly bounded. Therefore, the claim for the future follows. Now in the past, the right-hand side of \eqref{geodesicODE3} is positive and the denominator goes to 0 as $z\to z_{\mathcal N}$. Therefore $s\to -\infty$ and $R\to \infty$ as $z\to z_{\mathcal N}$, and hence the null geodesics asymptote to $z=z_{\mathcal N}$.   

We move onto $(-)$ part. Since the denominator is uniformly bounded for $z\in [z_0,0]$, the integral is finite and hence we deduce that the null geodesics meet $\tau=0$ for some positive $R>0$ in the future. On the other hand, the null geodesics in the past satisfy 
\be
s = s(z_0) +  \int_{z_0}^{z(s)} \frac{ \sqrt \k e^{\l -\mu} + z Q(z)}{Q(z)( z^2 - \k e^{2\l -2\mu})} dz . 
\ee
We claim that they meet $z=z_{\mathcal N}$ for finite $R>0$. To this end, we rewrite the integrand as 
 \[
 \begin{split}
 \frac{ \sqrt \k e^{\l -\mu} + z Q(z)}{ Q(z)( z^2 - \k e^{2\l -2\mu})} &= \frac{ \sqrt \k e^{\l -\mu} + z + z (Q(z)-1)}{ Q(z)( z^2 - \k e^{2\l -2\mu})} =  \frac{1}{Q(z) (z- \sqrt\k e^{\l-\mu })}  - \frac{z \frac{e^{2\mu} }{\ch^2} \frac{L^2}{\k C^2}}{Q(z) (Q(z)+1)}
 \end{split}
 \]
 for $z_{\mathcal N}<z<z_0<0$, where we have used~\eqref{E:QNNGDEF} in the last line. It is now clear that both terms are finite and thus the integral is bounded, thereby proving the claim. 
 
{\em Step 5. NNG-s in the interior region.} 
Using the same strategy as above, one can show that nonradial geodesics ($L\neq0$) starting in the interior region will asymptote to the past to the boundary of the backward light cone $\mathcal N$, and in the future 
they will asymptote to the Cauchy horizon $\mathcal B_1$ to the future. In between, they will go through a turning point, which corresponds to the point $(\bar s, \bar z)$, where $\bar z<z_{\mathcal N}$ is the zero of $Q$
discussed above. We omit the details.
% In summary, we conclude that no non-radial null geodesics  in the exterior region can emanate from or go into the scaling origin. It also confirms that $\mathcal N$ is the right most ingoing null-geodesics connecting to the origin.
\end{proof}
%

%%%%%%%%%%%%%%%%%%%%%%%%%%%%
%%%%%%%%%%%%%%%%%%%%%%%%%%%%
%%%%%%%%%%%%%%%%%%%%%%%%%%%%

\section{Einstein-Euler system in the double-null gauge}\label{A:DOUBLENULL}

%%%%%%%%%%%%%%%%%%%%%%%%%%%%
%%%%%%%%%%%%%%%%%%%%%%%%%%%%
%%%%%%%%%%%%%%%%%%%%%%%%%%%%

\subsection{Proof of Lemma~\ref{L:DOUBLENULL}}

\begin{proof}[Proof of Lemma~\ref{L:DOUBLENULL}.]
Note that
\begin{align}\label{E:EMT2}
T_{\u\v} = \frac14 \Om^{4} T^{\u\v}, \ \ T_{\u\u} = \frac14 \Om^{4} T^{\v\v}, \ \ T_{\v\v} = \frac14 \Om^{4} T^{\u\u}.
\end{align}
We use~\eqref{E:TPP} to obtain $\gamma^{AB}T_{AB} =  2 \k \rho r^2=  \e \Om^2 r^2 T^{\u\v}$. 
Equations~\eqref{E:RWAVE1}--\eqref{E:CONSTRAINTU1} now follow directly from Section 5 of~\cite{DaRe2005}.
%%%%%%%%%%%%%%%%%%%%%%%%%%%%
%%%%%%%%%%%%%%%%%%%%%%%%%%%%
Formulas~\eqref{E:TPP}-\eqref{E:TAB} follow directly from~\eqref{E:EMTENSOR} and~\eqref{E:METRICNULL}.
The fluid evolution equations~\eqref{E:CONTNULL13}--\eqref{E:MOMENTUMNULL13} are a consequence of the Bianchi identities $\nabla_\mu T^{\mu\nu}=0$, $\nu = \u,\v,2,3$. 
The Christoffel symbols associated with the purely angular degrees of freedom $\Gamma^A_{BC}$, $A,B,C\in\{2,3\}$ may not vanish, but are not needed in the computations to follow, so we do not compute them.
All the remaining non-vanishing Christoffel symbols
are given by (Appendix A in~\cite{DaRe2005})
\begin{align}
& \Gamma^\u_{AB} = - g^{\u\v} r \pa_\v r \gamma_{AB}, \ \ 
\Gamma^\v_{AB} = - g^{\u\v} r \pa_\u r \gamma_{AB}, \label{E:CHR1}\\
& \Gamma^A_{B\v} = \Gamma^A_{\v B}= \frac{\pa_\v r}{r} \delta^A_B, \ \ 
\Gamma^A_{B\u} =  \Gamma^A_{\u B} = \frac{\pa_\u r}{r} \delta^A_B, \\
& \Gamma^\u_{\u\u} = \pau\log(\Om^2), \ \ 
\Gamma^\v_{\v\v} = \pav\log(\Om^2).\label{E:CHR3}
\end{align}
%We also have the formulas
%\begin{align}
%& T^{\u\u} = (1+\k)\rho (u^\u)^2, \ \ T^{\v\v} = (1+\k)\rho (u^\v)^2, \label{E:TPP}\\ 
%& T^{\u\v} = (1+\k) \rho u^\u u^\v - 2 \k \Om^{-2} \rho =  (1-\k)\rho \Om^{-2}  \label{E:TUV}\\
%& T^{AB} = (1+\k)\rho \underbrace{u^Au^B}_{=0} + \k \rho g^{AB} = \k \rho r^{-2} \gamma^{AB}, \label{E:TAB} \\
%& T^{\u A} = T^{\v A} = 0, \ \ A = 2,3 \label{E:TAB2}
%\end{align}

By~\eqref{E:TAB} we note that 
\begin{align}\label{E:TRACEGAMMA}
\gamma_{AB}T^{AB} = \k \rho r^{-2} \gamma_{AB}\gamma^{AB} = 2\k \rho r^{-2}. 
\end{align}
We also observe that by~\eqref{E:TRACEGAMMA} and~\eqref{E:TPP}
\be\label{E:TRACE2}
\Om^{-2}\gamma_{AB}T^{AB} = 2\k \rho\Om^{-2} r^{-2} = \frac{2\k}{1-\k} r^{-2} T^{\u\v} = \e r^{-2} T^{\u\v} ,
\ee
where we recall the notation
$
\e = \frac{2\k}{1-\k}.
$
To express the Bianchi identities in coordinates, we use the formula
\begin{align}\label{E:CDER}
\nabla_\delta T^{\alpha\beta} = \pa_\delta T^{\alpha\beta} + \Gamma^\alpha_{\gamma\delta}T^{\gamma\beta} 
+ \Gamma^{\beta}_{\gamma\delta}T^{\alpha\gamma}
\end{align}
and the formulas~\eqref{E:CHR1}--\eqref{E:CHR3}.
We start the with the $p$-equation and obtain from~\eqref{E:BIANCHI} 
\begin{align}\label{E:CNULL}
0 = \nabla_\mu T^{\mu \u} = \nabla_\u T^{\u\u} + \nabla_\v T^{\v\u} + \nabla_AT^{Ap}.
\end{align}
By~\eqref{E:CDER} and~\eqref{E:CHR1}--\eqref{E:CHR3} we have
\begin{align}
\nabla_\u T^{\u\u} & = \pau T^{\u\u} + 2\Gamma^\u_{\u\u}T^{\u\u} = \left(\pa_\u + 4\frac{\pau\Om}{\Om}\right)T^{\u\u}, \\
 \nabla_\v T^{\v\u} & = \pa_\v T^{\u\v} + \Gamma^{\v}_{\v\v}T^{\u\v} =  \left(\pa_\v + 2\frac{\pav\Om}{\Om}\right)T^{\u\v} ,\\
 \nabla_A T^{A\u} & = \pa_A T^{A\u} + \Gamma^A_{\gamma A}T^{\gamma \u} + \Gamma^\u_{\gamma A}T^{A\gamma} \notag \\
 & = \Gamma^A_{\u A}T^{\u\u} + \Gamma^A_{\v A}T^{\v \u} +\Gamma^\u_{AB}T^{AB} \notag \\
 & = 2\frac{\pau r}{r} T^{\u\u}+2\frac{\pav r}{r} T^{\u\v} -g^{\u\v} r \pav r \gamma_{AB}T^{AB} \notag\\
 &  = 2\frac{\pau r}{r} T^{\u\u}+2\frac{\pav r}{r} T^{\u\v} + 2\Om^{-2}r \pav r \gamma_{AB}T^{AB} \notag\\
 & = 2\frac{\pau r}{r} T^{\u\u}+(2+2\e)\frac{\pav r}{r} T^{\u\v} , 
\end{align}
where we have used~\eqref{E:METRICNULL} in the next-to-last line and~\eqref{E:TRACE2} in the last.
Similarly, we compute the $\v$-equation of the Bianchi identities~\eqref{E:BIANCHI}.
\begin{align}\label{E:MNULL}
0 = \nabla_\mu T^{\mu \v} = \nabla_\u T^{\u\v} + \nabla_\v T^{\v\v} + \nabla_A T^{A\v}.
\end{align}
By~\eqref{E:CDER} and~\eqref{E:CHR1}--\eqref{E:CHR3} we have just like above
\begin{align}
\nabla_\u T^{\u\v} & =  \left(\pa_\u + 2\frac{\pau\Om}{\Om}\right)T^{\u\v}, \notag\\
 \nabla_\v T^{\v\v}  & = \left(\pa_\v + 4\frac{\pav\Om}{\Om}\right)T^{\v\v}, \\
  \nabla_A T^{A\v} & =  \Gamma^A_{\u A}T^{\u\v} + \Gamma^A_{\v A}T^{\v \v} +\Gamma^\v_{AB}T^{AB} \notag\\
  & =  2\frac{\pau r}{r} T^{\u\v} + 2\frac{\pav r}{r} T^{\v\v}+ 2\e \frac{\pau r}{r}T^{\u\v} \notag \\
  & = \left(2+2\e\right)\frac{\pau r}{r}  T^{\u\v}+ 2\frac{\pav r}{r} T^{\v\v}. 
  \end{align}

Therefore, keeping in mind equations~\eqref{E:CNULL} and~\eqref{E:MNULL} can be rewritten in the form
%~\eqref{E:CONTNULL1}--\eqref{E:MOMENTUMNULL1}.
\begin{align}
\left(\pa_\u + \pau \log\left(\Om^4r^2\right)\right)T^{\u\u} + \left(\pa_\v + \pav \log\left(\Om^2r^{2+2\e}\right)\right)T^{\u\v}   & = 0,   \label{E:CONTNULL}\\
\left(\pau + \pau \log\left(\Om^2r^{2+2\e}\right)\right)T^{\u\v} + \left(\pa_\v + \pav \log\left(\Om^4r^2\right)\right)T^{\v\v} & = 0,
\label{E:MOMENTUMNULL}
\end{align}
which is equivalent to~\eqref{E:CONTNULL13}--\eqref{E:MOMENTUMNULL13}.
From~\eqref{E:TPP}--\eqref{E:TAB} and~\eqref{E:NORMALISATIONNULL} we have the additional relationship
\begin{align}
T^{\u\u} T^{\v\v} & = (1+\k)^2 \rho^2 (u^\u)^2 (u^\v)^2 = (1+\k)^2 \rho^2\Om^{-4} \notag \\
& = \left(\frac{1+\k}{1-\k}\right)^2 (T^{\u\v})^2 
= (1+\e)^2(T^{\u\v})^2 .
\end{align}
%In particular, we may replace $T^{\v\v}$ by $(1+\e)^2\frac{(T^{\u\v})^2}{T^{\u\u}}$ and obtain a closed system for the unknowns 
%$T^{\u\u}$ and $T^{\u\v}$, assuming the metric terms $r$ and $\Om$ are known. Moreover, from~\eqref{E:TPP} and~\eqref{E:TAB}, we may think of $T^{\u\u}$ as a kind of a  energy density and $T^{\u\v}$ as a weighted mass density.
\end{proof}

\subsection{Double-null description of the RLP-spacetime}\label{A:DN}

%%%%%%%%%%%%%%%%%%%%%%%%%%%%
%%%%%%%%%%%%%%%%%%%%%%%%%%%%

%In this section we explain how to change variables from the comoving frame to the double-null frame. This change of variables is technically more demanding than the change from Schwarzschild to double-null coordinates.
%It is important to note that in the region $\DRLPtilde\cap\{(\ttau,R)\,\big| \ttau>0\}$ we may switch back to comoving coordinates~\eqref{E:METRIC}. 
%The corresponding change of variables is given by
%\begin{align}
%\tau = \k^{\frac\eta2} R^{-\e}\ttau^{1+\e}, \ \ (\ttau,R)\in\DRLPtilde^+:= (\ttau,R)\in\DRLPtilde\cap \{\ttau>0\}.
%\end{align}
%The self-similar variable $Y$ defined in~\eqref{E:YDEF} relates to $(\tau,R)$ via
%\begin{align}
%Y = - \left(\frac{R}{\sqrt \k \tau}\right)^{-\frac{1-\k}{1+\k}}, \ \ \tau>0.
%\end{align}
% 
%Let us denote
%\begin{align}\label{E:RAD}
%U:=\pa_\tau r, \ \ v = e^{-\mu}\pa_\tau r
%\end{align}
%We refer to $v$ as the {\em radial velocity}, which is consistent with the observation that the 4-velocity vector field in the comoving frame corresponds to $e^{-\mu}\pa_\tau$. 
%
%Note that 
%\begin{align}
%d\u = \pa_\tau \u \,d\tau + \pa_R \u\,dR, \ \ d\v = \pa_\tau \v \,d\tau + \pa_R \v\,dR,
%\end{align}
%and therefore
%\begin{align}
%g = - \Om^2 \,d\u\,d\v = - \Om^2 \pa_\tau \u\pa_\tau \v \,(d\tau)^2 - \Om^2 \left(\pa_\tau \u\pa_R \v + \pa_R \u\pa_\tau \v\right) -
%\Om^2 \pa_R \u\pa_R \v \,(dR)^2.
%\end{align}

\begin{lemma}\label{L:RLPDN}
The region to the past of the curve $\mathcal B_1$ in the RLP spacetime $(\MRLP,\grlp)$ is well-defined in the double-null coordinates, with the normalisation conditions~\eqref{E:NORMIN}--\eqref{E:NORMOUT}.
\end{lemma}

%%%%%%%%%%%%%%%%%%%%%%%%

\begin{proof}

Recalling the expressions~\eqref{E:METRIC} and~\eqref{E:DNG}, it is easy to see that 
\begin{align}
\Om^2 \pa_\tau \u\pa_\tau \v  = e^{2\mu}, \ \ 
\Om^2 \pa_R \u\pa_R \v  = -e^{2\lambda}, \ \ 
\frac{\pa_\tau \u}{\pa_R \u} + \frac{\pa_\tau \v}{\pa_R \v}  = 0. \label{E:J3}
\end{align}
Using~\eqref{E:J3} we conclude that
$
\frac{\pa_\tau \u}{\pa_R \u} = -\frac{\pa_\tau \v}{\pa_R \v} = e^{\mu-\lambda}.
$
%Here $\u=\text{const}$ is the outgoing direction and therefore $\pa_R \u>0$, while the level sets of $\v$ correspond to ingoing null curves and therefore $\pa_R \v<0$.
Recall that the 4-velocity $u^\nu$ in the comoving frame takes the form $e^{-\mu}\pa_\tau$
and therefore
\begin{align}
 \mathcal V = e^{-\mu}\pa_\tau r   = u^\u\pau r + u^\v\pav r = \frac1{\Om^2 u^\v}\pau r + u^\v\pav r. 
\end{align}
It follows that $u^\v$ solves the quadratic equation $\pav r (u^\v)^2 - \mathcal V u^\v + \frac1{\Om^2}\pau r = 0$. Therefore
\begin{align}
u^\u &= \frac{\mathcal V -\sqrt{\mathcal V^2-\frac{4\pau r \pav r }{\Om^2}}}{2\pau r}= \frac{\mathcal V -\sqrt{\mathcal V ^2+1-\frac{2m}{r}}}{2\pau r}, \label{E:UTAUNULL1}\\
u^\v & =  \frac{\mathcal V +\sqrt{\mathcal V^2-\frac{4\pau r \pav r}{\Om^2}}}{2\pav r} = \frac{\mathcal V +\sqrt{\mathcal V^2+1-\frac{2m}{r}}}{2\pav r},\label{E:UTAUNULL2}
\end{align}
where the choice of signs in the solutions of the quadratic equations is consistent with the normalisation that $e^{-\mu}\pa_\tau=u^\u\pau + u^\v\pav$ is future pointing.

We now let $r\to\infty$ along the outgoing geodesic through $(\tau_0,R_0)$. By Lemma~\ref{L:INCOMING} we know that this curve asymptotes to 
$\mathcal B_1$ in the $(\ttau,R)$ plane, or equivalently, $y$ converges to $y_1$ in the $(\tau,R)$-plane.
Recalling~\eqref{E:VELOCITYSSCHANGE},~\eqref{E:MUFORMULASS}--\eqref{E:MUFORMULASS2}, and~\eqref{E:LITTLEWDEF}, we see that $\mathcal V(\tau,R) = \sqrt\k V(y)=\sqrt\k\tr(y) \frac{\w(y)-1}{1+\k} \Sigma(y)^{\frac{\k}{1+\k}} \to \bar C_1\sqrt\k$ as $y\to y_1$ 
for some constant $\bar C_1>0$. Recall here that the density $\Sigma(y)=\d(y)^{\frac{1-\k}{1+\k}}=\RY(Y)^{\frac{1-\k}{1+\k}}$ for $y<0$.
Observe that along the outgoing geodesic
\be\label{E:ASPECTFUNCTIONLIMIT}
 \lim_{\v\to\infty} \left(\frac{2m}{r}-1\right) = -1+C\k
\ee
for some $\k$-independent constant $C>0$ by an argument analogous to~\eqref{E:NOTAF}.
Using the normalisation condition~\eqref{E:NORMOUT}, we see that there exists a constant $\bar C_2>0$ such that 
\be\label{E:UVLIMIT}
\lim_{\v\to\infty} (u^\v)^2=\bar C_2>0.
\ee

Using the normalisation condition~\eqref{E:NORMOUT}, equation~\eqref{E:CONSTRAINTV1} along the outgoing null-geodesic through $(\ttau_0,R_0)$ takes the form
\begin{align}
\pav(\Om^{-2})=-2\pi (1+\k)(\frac12\v +r_\ast) \Om^{-2}\rho (u^\v)^{-2},
\end{align}
where we have used~\eqref{E:NORMALISATIONNULL} and~\eqref{E:NORMALISATIONNULL}. Equivalently, after multiplying by $\Omega^2$ and integrating, we obtain the formula
\begin{align}\label{E:CONSTRAINTVRLP}
\pav\log\Om = (1+\k)\pi (\frac12 \v + r_\ast) \rho (u^\v)^{-2}.
\end{align} 
Note that by~\eqref{E:SS1} $\rho(\tau,R)=\frac{\k}{R^2}y^2 \Sigma(y)$.
Since we know by Lemma~\ref{L:INCOMING}
that as $\v\to\infty$, $y\to y_1$, and thereby $\frac{r}{R}\to\ch(y_1)$  it follows by $q=2(r-r_\ast)$ that there exists a constant $C_1$ independent of $\k$ such that 
\[
\frac{\rho(\tau,R)}{\frac{1}{\v^2}} \to_{\v\to\infty} C_1\k. 
\]
We then use~\eqref{E:UVLIMIT} to conclude that that the leading order asymptotic behaviour
of the right-hand side of~\eqref{E:CONSTRAINTVRLP} is $\frac{C\k}{\v}$ for some $\k$-independent constant $C>0$. Therefore
$\Om(\u_0,\cdot)$ is well-defined and $\Om\sim_{\v\to\infty}\v^{C\k}$.
Moreover, from~\eqref{E:HAWKINGMASSDEF} and~\eqref{E:ASPECTFUNCTIONLIMIT}, we have
\begin{align}
\lim_{\v\to\infty} \frac{4\pau r\pav r}{\Om^2} = \lim_{\v\to\infty} \frac{2m}{r}-1 = -1+C\k.
\end{align}
It follows therefore, that along the outgoing geodesic $\pau r\approx_{\v\to\infty} \Om^2\approx_{\v\to\infty} \v^{2C\k}$.
We can similarly use the normalisation condition~\eqref{E:NORMIN} to determine $\Om$ along the ingoing geodesic through $(\ttau_0,R_0)$ in the past of $\mathcal B_1$.
Using~\eqref{E:RWAVE1}--\eqref{E:OMEGAWAVE1}, it is straightforward to show that the the solution exists everywhere in the past of $\mathcal B_1$ with the asymptotic
behaviour $\Om\asymp_{\v\to\infty} \v^{C\k}$, $\pau r \asymp_{\v\to\infty} \v^{2C\k}$, $\pav r\asymp_{\v\to\infty} 1$.
\end{proof}

\end{document}